\documentclass[letterpaper,reqno]{amsbook}

\usepackage[letterpaper,hmarginratio=1:1]{geometry}
\usepackage{amsmath, amsthm, amsbsy, thmtools, amssymb, bm, comment}
\usepackage{bbm}
\usepackage{tikz}
\usepackage{xcolor}
\usepackage[mathscr]{euscript}

\numberwithin{equation}{section}
\usepackage{chngcntr}
\counterwithout{section}{chapter}

\usepackage[
backend=bibtex,
style=alphabetic,
giveninits=true,
url = false,
doi = false,
isbn =false,
maxbibnames=99,
]{biblatex}
\renewbibmacro{in:}{}
\DeclareFieldFormat{title}{#1} 
\DeclareFieldFormat[article]{title}{#1}
\DeclareFieldFormat[inbook]{title}{#1}
\DeclareFieldFormat[incollection]{title}{#1}
\usepackage[hidelinks]{hyperref}
\usepackage{cleveref}

\newtheorem{thm}{Theorem}[section]
\newtheorem{prop}[thm]{Proposition}
\newtheorem{lem}[thm]{Lemma}
\newtheorem{cor}[thm]{Corollary}

\theoremstyle{remark}
\newtheorem{rem}[thm]{Remark}
\theoremstyle{definition}
\newtheorem{assumption}[thm]{Assumption} 
\newtheorem{definition}[thm]{Definition}

\DeclareMathOperator{\sgn}{sgn}
\DeclareMathOperator{\Real}{Re}
\DeclareMathOperator{\Imaginary}{Im}

\renewcommand{\hat}{\widehat}
\renewcommand{\tilde}{\widetilde}

\renewcommand{\Re}{\operatorname{Re}}
\renewcommand{\Im}{\operatorname{Im}}

\newcommand{\E}{\mathbb{E}}
\renewcommand{\P}{\mathbb{P}}
\newcommand{\bfH}{\boldsymbol{H}}
\newcommand{\bbH}{\mathbb{H}}
\newcommand{\R}{\boldsymbol{R}}
\newcommand{\T}{\boldsymbol{T}}

\newcommand{\bbK}{\mathbb{K}}
\newcommand{\bbL}{\mathbb{L}}
\newcommand{\G}{\boldsymbol{G}}
\newcommand{\bbN}{\mathbb{N}}
\newcommand{\bbR}{\mathbb{R}}
\newcommand{\bbC}{\mathbb{C}}

\newcommand{\bbV}{\mathbb{V}}
\newcommand{\bbD}{\mathbb{D}}

\newcommand{\iu}{\mathrm{i}}

\renewcommand{\phi}{\varphi}
\newcommand{\one}{\mathbbm{1}}
\newcommand{\eps}{\varepsilon}
\renewcommand{\epsilon}{\varepsilon}

\newcommand{\deloc}{\mathrm{deloc}}
\newcommand{\loc}{\mathrm{loc}}
\newcommand{\mob}{\mathrm{mob}}
\newcommand{\Zplus}{\mathbb{Z}_{\ge 1}}

\newcommand{\be}{\begin{equation}}
\newcommand{\ee}{\end{equation}}
\newcommand{\bex}{\begin{equation*}}
\newcommand{\eex}{\end{equation*}}
\newcommand{\TV}{\mathrm{TV}}

\renewcommand{\star}{{00}}

\newcommand{\unn}[2]{[\![#1,#2]\!]}

\title{Mobility Edge for L\'evy Matrices}
\author{Amol Aggarwal}
\author{Charles Bordenave}
\author{Patrick Lopatto}

\begin{document}

\begin{abstract}
\normalsize
L\'evy matrices are symmetric random matrices whose entry distributions lie in the domain of attraction of an $\alpha$-stable law. For $\alpha < 1$, predictions from the physics literature suggest that high-dimensional L\'{e}vy matrices should display the following phase transition at a point $E_{\mob}$. Eigenvectors corresponding to eigenvalues in $(-E_{\mob},E_{\mob})$ should be delocalized, while eigenvectors corresponding to eigenvalues outside of this interval should be localized. Further, $E_{\mob}$ is given by the (presumably unique) positive solution to $\lambda(E,\alpha) =1$, where $\lambda$ is an explicit function of $E$ and $\alpha$.

We prove the following results about high-dimensional L\'{e}vy matrices. 
\begin{enumerate} 
	\item If $\lambda(E,\alpha) > 1$ then eigenvectors with eigenvalues near $E$ are delocalized.
	\item If $E$ is in the connected components of the set $\big\{ x : \lambda(x,\alpha) < 1 \big\}$ containing $\pm \infty$, then eigenvectors with eigenvalues near $E$ are localized.
	\item For $\alpha$ sufficiently near $0$ or $1$, there is a unique positive solution $E = E_{\mob}$ to $\lambda(E,\alpha) = 1$, demonstrating the existence of a (unique) phase transition. 
	\begin{enumerate} 
		\item If $\alpha$ is close to $0$, then  $E_{\mob}$ scales approximately as $|\log \alpha|^{-2/\alpha}$. 
		\item If $\alpha$ is close to $1$, then $E_{\mob}$ scales as $(1-\alpha)^{-1}$. 
	\end{enumerate} 
\end{enumerate} 

Our proofs proceed through an analysis of the local weak limit of a L\'{e}vy matrix, given by a certain infinite-dimensional, heavy-tailed operator $\bm{T}$ on the Poisson weighted infinite tree. In the process, we show that the diagonal entries of this operator's resolvent $\bm{R} (E + \mathrm{i } \eta) = (\bm{T} - E - \mathrm{i} \eta)^{-1}$ remain positive almost surely as $\eta$ tends to zero on $\big\{ E : \lambda(E,\alpha) > 1 \big\}$, but vanish as $\eta$ tends to zero on the connected components of $\big\{ E : \lambda (E,\alpha) < 1 \big\}$ containing $\pm \infty$.

\end{abstract}
\maketitle
\tableofcontents

\chapter{Results and Preliminaries}

\section{Introduction}

\subsection{Background}

L\'evy matrices are symmetric random matrices whose entry distributions lie in the domain of attraction of an $\alpha$-stable law; for $\alpha<2$, their entries have infinite variance. 
They were introduced by Bouchaud and Cizeau in 1994, motivated by various questions about heavy-tailed phenomena in physics and mathematical finance \cite{cizeau1994theory}. The ubiquity of heavy-tailed random variables has made L\'{e}vy matrices widely applicable. Indeed, their study has since yielded insights into spin glasses with power-law interactions \cite{cizeau1993mean}, portfolio optimization \cite{galluccio1998rational,bouchaud1998taming}, option pricing \cite{bouchaud1997option}, neural networks \cite{martin2021predicting}, and a variety of statistical questions \cite{sornette2006critical,heiny2019eigenstructure, bun2017cleaning, laloux1999noise,laloux2000random,bouchaud2009financial}.

The original work \cite{cizeau1994theory} made several predictions about the spectra of L\'evy matrices. The first was that the empirical spectral distribution of a L\'evy matrix of large dimension should converge to a deterministic, heavy-tailed measure $\mu_\alpha$ (depending only on the parameter $\alpha$ and not the precise laws of the matrix entries). This was proved in \cite{arous2008spectrum}; subsequent works showed that  $\mu_\alpha$ admits a density with respect to Lebesgue measure on $\bbR$, which is symmetric and scales as $(\alpha/2) |x|^{-\alpha - 1}$, as $|x|$ tends to $\infty$ \cite{arous2008spectrum,belinschi2009spectral,bordenave2011spectrum}. One already observes from this a sharp contrast between matrices with finite variance entries, such as the Gaussian Orthogonal Ensemble\footnote{Recall that the Gaussian Orthogonal Ensemble is a $N\times N$ symmetric matrix whose upper triangular entries $\{g_{ij}\}_{i,j=1}^N$ are mutually independent Gaussian random variables with mean zero and variances $(1 + \one_{i=j})N^{-1}$.} (GOE), and L\'{e}vy matrices; the global law of the former is compactly supported, while that of the latter is not. 

These differences become more pronounced when one probes eigenvector and local eigenvalue statistics. The authors of \cite{cizeau1994theory} presented detailed predictions for these phenomena with respect to L\'{e}vy matrices, which were later corrected and completed in the more recent work \cite{tarquini2016level} using the replica method. We present only the latter predictions here. 

For all $\alpha \in [1,2)$, the eigenvectors should be completely delocalized,\footnote{A unit vector $\boldsymbol{u} \in \mathbb{R}^N$ is \emph{completely delocalized} if, for any $\varepsilon > 0$, we have $\| \boldsymbol{u} \|_\infty \le N^{-1/2+\eps}$ with high probability for large enough $N$. It is \emph{completely localized} if, for any $\varepsilon > 0$, there exists a constant $m = m(\varepsilon) \ge 1$ such that at least $1-\varepsilon$ of its $\ell^2$-mass is supported on at most $m$ of its entries.} and the local statistics of the eigenvalues should converge to those of the Gaussian Orthogonal Ensemble (GOE) throughout the bulk of the spectrum (around any finite real number). Moreover, for all $\alpha \in (0,1)$, there should exist a sharp phase transition at some real number $E_{\mob} = E_{\mathrm{mob}}(\alpha)$, such that the following holds. First, in the interval $(-E_{\mob}, E_{\mob})$, all eigenvectors should be completely delocalized and the eigenvalues should display GOE statistics. Second, outside of this interval (that is, on $(-\infty, -E_{\mob}) \cup (E_{\mob}, \infty)$), all eigenvectors should be completely localized and the eigenvalues should be distributed as a Poisson point process. The work \cite{tarquini2016level} also predicted a deterministic formula for $E_{\mathrm{mob}}$, given by the (presumably unique) real number $E > 0$ such that $\lambda(E,\alpha) = 1$, where $\lambda = \lambda(E,\alpha)$ is the largest solution of the quadratic equation
\begin{flalign} 
	\label{mobilityintroduction}
\lambda^2 - 2t_{\alpha} K_{\alpha} \Re \ell(E) \cdot \lambda + K_{\alpha}^2 (t_{\alpha}^2 - 1) |\ell (E)|^2 = 0.  
\end{flalign}

\noindent Here, $t_{\alpha}$, $K_{\alpha}$, and $\ell (E)$ are explicit  functions given by \Cref{lambdaEalpha} below (though $\ell(E)$ is a bit intricate). An equivalent, but more complicated, characterization was predicted earlier in \cite{cizeau1994theory}. 

For $\alpha \in (0,1)$, the predicted coexistence of localization and delocalization  differs strongly from the behavior of random matrices whose entries have finite variance. Indeed, the latter display complete eigenvector delocalization and GOE local eigenvalue statistics throughout the entire bulk of the spectrum. This fact, known as the \emph{Wigner--Dyson--Mehta conjecture}, was originally shown in \cite{erdos2010bulk,TV11} for certain classes of complex Hermitian matrices. Over the past decade, its conclusion was vastly generalized through the three-step strategy of \cite{localrelax2} and was proven to hold for wide classes of finite-variance random matrix ensembles; we refer to the surveys \cite{erdos2011survey,survey,erdos2017dynamical} for more information and references on this topic.

The phase transition indicated above is an example of \emph{Anderson localization}, a term describing localization--delocalization transitions triggered by sufficiently high disorder. The real number (or ``energy'') at which this transition occurs  is called a \emph{mobility edge}, denoted by $E_{\mathrm{mob}}$ above. 
Originally proposed in the 1950s to explain puzzling experimental results on doped semiconductors, Anderson localization is now recognized as a general feature of wave transport in disordered media and is one of the most influential ideas in modern condensed matter physics \cite{lee1985disordered,evers2008anderson,lagendijk2009fifty, abrahams201050,anderson1958absence}. 
However, our theoretical understanding of Anderson localization remains severely incomplete.
 
 This phenomenon has traditionally been studied through the tight-binding model \cite{abou1973selfconsistent, abou1974self, anderson1978local, aizenman2015random}, defined on the lattice $\mathbb Z^d$ by the random self-adjoint operator $H_\lambda = T + \lambda V$ on $\ell^2( \mathbb Z^d)$. Here, $T$ is the graph Laplacian of the lattice; $V$, called the \emph{potential}, is a diagonal operator with independent, random entries; and $\lambda \ge 0$ governs the strength of the disorder. In dimensions $d \ge 3$, a mobility edge is expected to occur for small values of $\lambda$, separating a delocalized regime from a localized one \cite{aizenman2015random}. The localized regime has been thoroughly investigated by this point, and it has been established that eigenfunctions are completely localized and Poisson local eigenvalue statistics arise whenever the energy is sufficiently far from zero (and also when $\lambda$ is sufficiently large) \cite{frohlich1983absence,minami1996local,aizenman1993localization,spencer1988localization,gol1977pure,ding2018localization,li2019anderson,carmona1987anderson}. However, for small $\lambda$, it remains open to establish the existence of a delocalized phase with GOE statistics, as well as a sharp transition between the localized and delocalized phases. 
 
 The tight-binding model can also be defined on an infinite regular tree, which allows for technical simplifications. In this context, the existence of delocalized states at small disorder was shown in \cite{klein1998extended,froese2007absolutely,aizenman2006stability}. More generally, \cite{aizenman2013resonant} proved a sufficient condition for the appearance of absolutely continuous spectrum (a signature of delocalization) on the regular tree, which the authors used to show the entire spectrum of the model is absolutely continuous if the disorder is sufficiently small (and if the potential has unbounded support and bounded density). This criterion is given in terms of a fractional moment bound for the model's resolvent, which is almost complementary to the Simon--Wolff criterion for pure point spectrum (a signature of localization) \cite{SCSRP,aizenman1993localization}. This is strongly suggestive of a sharp transition separating the localized and delocalized phases but is too inexplicit to pin it down. In this direction, \cite{bapst2014large} used this criterion  to provide upper and lower bounds on the endpoints for intervals of localization and delocalization; these bounds converge towards each other as the degree $K$ of the graph tends to $\infty$ but do not coincide for any finite $K$. In fact, no exact form for the mobility edge seems to be known for any tight-binding model, even in the most solvable cases, such as the Anderson model on the regular tree with Cauchy disorder \cite{miller1994weak}. 
 
 There are three other prominent random matrix ensembles whose spectra are believed to exhibit a coexistence of localized and delocalized regimes; they are band matrices \cite{bourgade2018survey}, adjacency matrices of sparse graphs \cite{RMM}, and L\'{e}vy matrices \cite{HTRM}. Analogously to the tight-binding model, band matrices are believed to exhibit a mobility edge in dimensions $d \ge 3$ when the band width is sufficiently large (but constant) \cite{RBSM,DSMRM}. Proving this prediction remains open, though recent years have seen progress towards the analysis of band matrices in dimensions $d \ge 7$ with a band width diverging in the dimension of the matrix \cite{DQDRBMHD,DQDRBM,BUQURBM}, and in dimension $d = 1$ \cite{bourgade2018random,bourgade2019random,yang2018random,URBMSM,URBM,STM}; in both cases, no mobility edge appears, as the entire spectrum is either all localized or all delocalized (depending on the band width). Random graphs should exhibit such coexistence when the average vertex degree is constant\footnote{When the average degree diverges at least logarithmically in the number of vertices, the behavior of the eigenvectors was determined in \cite{alt2021completely,alt2021delocalization}.} \cite{bauer2001random, bauer2001exactly,bauer2001core}, though this coexistence is of a different nature from in the other models of Anderson localization mentioned above. Indeed, in this regime, the localized and delocalized phases are not separated by a sharp transition but instead overlap, in that intervals of the spectral support admit some eigenvectors that are delocalized \cite{bordenave2011rank,coste2021emergence,bordenave2017mean,arras2021existence} and others that are localized \cite{chayes1986density,salez2015every}. The endpoints of these intervals should serve as analogs of the mobility edges for this model. In both of these models, there is neither a mathematical proof for the existence of these sharp phase transitions (mobility edges), nor an explicit prediction for their locations.
 
 L\'evy matrices are one of a small number of models predicted to exhibit a mobility edge that can be computed exactly. Partly for this reason, the predictions of \cite{tarquini2016level,cizeau1994theory} have generated substantial interest in L\'evy matrices among both physicists and mathematicians \cite{auffinger2009poisson,arous2008spectrum,benaych2014central,benaych2014central,belinschi2009spectral,biroli2007top,bordenave2011spectrum,bordenave2017delocalization,bordenave2013localization,burda2006random,soshnikov2004poisson,tarquini2016level,biroli2007extreme,aggarwal2021goe}.  The analogy between Anderson localization in heavy-tailed random matrices and the tight-binding model becomes especially transparent upon viewing the tail parameter $\alpha$ as parallel to the inverse disorder strength $\lambda^{-1}$, as reducing $\alpha$ makes the L\'{e}vy matrix ``more noisy.'' The equation \eqref{mobilityintroduction} for the mobility edge predicts for small $\alpha$ that the delocalized phase should be short (shrinking to a point as $\alpha$ tends to $0$), while for large $\alpha$ (close to $1$) that it should be long (extending across the real axis as $\alpha$ tends to $1$). This is similar to the prediction for the tight-binding model on the lattice that, for small $\lambda$ the delocalized region should nearly saturate the spectrum \cite{aizenman2015random}, while for large $\lambda$ it is empty.
 
 We next review rigorous results on the predictions of \cite{cizeau1994theory,tarquini2016level} regarding localization and delocalization. For any $\alpha \in (1, 2)$, the predictions on complete eigenvector delocalization and GOE local statistics throughout the spectrum for $\alpha \in (1, 2)$ were proven in \cite{aggarwal2021goe}. In the case $\alpha \in (0, 1)$, results are less complete. For every $\alpha \in \big( 0, \frac{2}{3} \big)$, \cite{bordenave2013localization} showed the existence of a regime consisting of sufficiently large energies where eigenvectors are partially localized. For almost every $\alpha \in (0, 1)$, \cite{bordenave2017delocalization} established the existence of a region consisting of sufficiently small energies where eigenvectors are partially delocalized. Subsequently, \cite{aggarwal2021goe} showed in this region that eigenvectors are completely delocalized and that local eigenvalue statistics converge to those of the GOE. Building on \cite{aggarwal2021goe}, the article \cite{aggarwal2021eigenvector} demonstrated that the statistics of eigenvector entries in this small-energy region are non-Gaussian, further distinguishing it from those in the finite-variance case (which are known to be typically Gaussian \cite{bourgade2013eigenvector,marcinek2020high}). As all of these results on L\'evy matrices only apply to sufficiently small or large energies, no work thus far has touched on the existence of the predicted sharp localization--delocalization transition, nor its exact characterization given by \eqref{mobilityintroduction}. 
 
 The purpose of the present paper is to address this point.

\subsection{Results and Proof Ideas}\label{s:ideas}

We now proceed to explain our results and the ideas of their proofs. We will be informal here, both with definitions and explanations, referring to \Cref{Results} below for a more detailed exposition. 

Let $\bm{H} = \{ h_{ij} \} $ denote an $N \times N$ L\'{e}vy matrix (see \Cref{matrixh} below) and, for any $z = E + \mathrm{i} \eta \in \mathbb{H}$ in the upper half plane, let $\bm{G} = \bm{G}(z) = \{ G_{ij} \} = (\bm{H} - z)^{-1}$ denote its resolvent. As is common in the context of random operators \cite{aizenman2015random}, we use $\Imaginary G_{jj} (E + \mathrm{i}\eta)$ (which is independent of $j \in [1, N]$, by symmetry) to distinguish localization from delocalization at a point $E \in \mathbb{R}$. More specifically, we equate\footnote{This equivalence can be made precise through the introduction of a certain inverse participation ratio $Q_I$ (\Cref{jwi}) expressible in terms of resolvent entries (\Cref{l:loccriteria}) that is indicative of eigenvector delocalization or localization if it remains bounded or becomes unbounded, respectively (\Cref{westimate}).} eigenvector localization and delocalization around $E$ with the statements that $\lim_{\eta \rightarrow 0} \Imaginary G_{jj} (E + \mathrm{i} \eta) = 0$ and $\lim_{\eta \rightarrow 0} \Imaginary G_{jj} (E + \mathrm{i} \eta) > 0$ in probability, respectively.

Recalling $\lambda = \lambda (\alpha, E)$ from \eqref{mobilityintroduction}, we prove the following results on the large L\'evy matrix $\bm{H}$.
\begin{enumerate}
\item \emph{Localization}: Let $E_{\loc}$ be the largest real number such that $\lambda(E_{\loc},\alpha) = 1$. For any $E \in \mathbb{R}$ with $|E| > E_{\max}$, eigenvectors of $\bm{H}$ with eigenvalue around $E$ are localized. 
\item \emph{Delocalization}: For any $E \in \mathbb{R}$ such that $\lambda (E,\alpha) > 1$, the eigenvectors of $\bm{H}$ with eigenvalue around $E$ are delocalized. Consequently, letting $E_{\deloc}$ be the smallest positive real number such that $\lambda(E_{\deloc},\alpha) = 1$, the eigenvectors of $\bm{H}$ are delocalized around any $E$ with $|E| < E_{\deloc}$.
\end{enumerate}
The previous two results reduce proving the existence of a unique mobility edge to the question of whether $E_{\deloc} = E_{\loc}$, that is, where there is a unique positive solution $E = E_{\mob}$ to $\lambda(E,\alpha) = 1$. This is checked numerically in \cite{tarquini2016level}. As for a rigorous proof, although this is a question about the purely deterministic equation \eqref{mobilityintroduction}, it is complicated by the intricacy of the function $\ell(E)$. Nonetheless, we prove such a statement if $\alpha$ is sufficiently close to $0$ or $1$, giving the following result.
\begin{enumerate}\item[(3)] \emph{Sharp Anderson transition}: There exists a constant $c>0$ such that $E_{\deloc} = E_{\loc}$ if $\alpha \in (0,c)\cup (1-c, 1)$. Denoting this common value by $E_{\mob}$, we have $E_{\mob} \sim |\log \alpha|^{-2/\alpha}$ for $\alpha$ near $0$, and $E_{\mob} \sim (1-\alpha)^{-1}$ for $\alpha$ near $1$.
\end{enumerate}
We elaborate more on these statements and their relation to previous works in \Cref{s:relationprevious} below.

Our proofs are based on analyzing a certain infinite-dimensional operator $\bm{T}$ that is associated with a rooted and randomly weighted tree (which is a variant of the Poisson Weighted Infinite Tree introduced in \cite{aldous1992asymptotics,aldous2004objective}). Each vertex of the tree has infinitely many children, its vertex set is given by $\mathbb{V} = \{0 \} \cup \bigcup_{k = 1}^{\infty} \mathbb{Z}_{\ge 1}^k$, and its root is chosen to be $0$. For each $v \in \mathbb{V}$, the set of edge weights connecting $v$ to its children are given by the entries of a Poisson point process on $\bbR$ with intensity measure formally given by $\frac{\alpha}{2} |x|^{-\alpha - 1 }\, dx$. Then $\bm{T}$ is set to be the adjacency matrix of this tree, and it can be interpreted as the local limit of the matrix $\bm{H}$ as $N$ tends to $\infty$ (upon viewing $\bm{H}$ as the adjacency matrix for a weighted complete graph). As such, information about $\bm{T}$ can be transferred to $\bm{H}$. Indeed, letting $\bm{R} = \bm{R} (z) = \{ R_{uv} \} = (\bm{T} - z)^{-1}$ denote the resolvent of $\bm{T}$ and $ R_{\star} (z) = R_{00}$ denote its diagonal entry at the root, it was shown in \cite{bordenave2011spectrum} that $G_{jj}$ converges to $R_\star$ in law. Hence, to show the above (de)localization statements for $\bm{H}$, it suffices to show them for $\bm{T}$. 

This provides two advantages,\footnote{Analogs of both can be realized for the original $\bm{H}$, at the expense of having to tracking an (effective in $N$) error.} which are related. The first is, since $\bm{T}$ is the adjacency matrix of a tree, off-diagonal entries of $\bm{R}$ admit a simplified product form. Indeed, for any vertex $v \in \mathbb{V}$, we have 
\begin{flalign}
	\label{r0v0} 
	R_{0v} = - R_{00} T_{0 0_+} R_{0_+ v}^{(0)}, \qquad \text{where $0_+$ is the child of $0$ on the path between $0$ and $v$},
\end{flalign}

\noindent where $R_{0_+ v}^{(0)}$ is the $(0_+, v)$ entry of $\big(\bm{T}^{(0)} - z \big)^{-1}$ (and $\bm{T}^{(0)}$ is obtained from $\bm{T}$ by setting its $(0,0_+)$ entry to $0$); see \Cref{rproduct} below. The second is that this tree admits a metric that can be useful for quantifying (de)localization phenomena (and has been in analogous contexts, such as the tight-binding model). 
 
 Given this, the proofs of our main results are composed of three main parts. 
\begin{enumerate}
	\item \emph{Fractional moment criterion}:  We establish a criterion for delocalization (and localization) through the growth (and decay, respectively) of certain fractional moment sums of the resolvent entries $R_{0v}$. This reduces the original question to an analysis of fractional moments for these resolvent entries, though the latter are not always explicit.
	\item \emph{Explicit formula for the mobility edge}: We show that this fractional moment criterion is equivalent to a more explicit one, dependent on whether the quantity $\lambda(E,\alpha)$ is larger or smaller than $1$. This reduces us to an analysis of the deterministic equation \eqref{mobilityintroduction}.
	\item \emph{Small and large $\alpha$ behavior}: We analyze this equation \eqref{mobilityintroduction} for $\alpha$ near $0$ and $1$. We first show $E_{\mob} \sim |\log \alpha|^{-2/\alpha}$ for $\alpha$ near $0$ and $E_{\mob} \sim (1 - \alpha)^{-1}$ for $\alpha$ near $1$; we then show uniqueness of a solution in $E$ to $\lambda(E,\alpha) = 1$ in these regimes. Together, these statements show the existence of a unique mobility edge for $\alpha$ near $0$ and $1$.
\end{enumerate}

\noindent To the best of our knowledge, the random operator $\boldsymbol{T}$ is the first for which a non-trivial mobility edge can be computed and rigorously demonstrated. The first, second, and third parts described above constitute \Cref{MomentFractional}, \Cref{EdgeExplicit}, and \Cref{Scaling} of this paper, respectively. We next elaborate on these parts in more detail. 

\subsubsection{Fractional Moment Criterion} 

\label{MomentFractionals}

For any integer $L \ge 1$, let $\mathbb{V} (L) \subset \mathbb{V}$ denote the set of vertices in $\mathbb{V}$ of distance $L$ from the root $0$, and for any fixed real number $s\in (\alpha, 1)$ and complex number $z \in \mathbb{H} = \{ z \in \mathbb{C}  : \Im (z) > 0 \}$, define (assuming they exist)
\begin{flalign}
	\label{sz0} 
\Phi_{L} (s; z) = \mathbb{E} \Bigg[ \displaystyle\sum_{v \in \mathbb{V}  (L)} \big| R_{0v} (z) \big|^s \Bigg]; \qquad \varphi (s; z) = \lim_{L\rightarrow \infty} \frac{1}{L} \log \Phi_{L} (s; z); 
\end{flalign}

\noindent Thus, $\Phi_L (s; z)$ describes a sum of fractional moments of resolvent entries associated with the $L$-th level of $\mathbb{V}$, and $\varphi$ denotes its exponential growth (or decay) rate. For any real number $E \in \mathbb{R}$, also define the limits of the above quantities as $s$ tends to $1$ and $z$ tends to $E$, given by
\begin{flalign*}
	\varphi (z) = \displaystyle\lim_{s \rightarrow 1} \varphi (s; z); \qquad \varphi(E) = \displaystyle\lim_{\eta \rightarrow 0} \varphi (E + \mathrm{i} \eta).
\end{flalign*}

The fractional moment criterion we show  is then that $\bm{T}$ is delocalized around some real number $E \in \mathbb{R}$ if $\varphi (E) > 0$, and that it is localized around $E$ if $\varphi (E) < 0$ (see \Cref{rimaginary0} and \Cref{p:imvanish} below). The latter statement on localization is a quick consequence of the Ward identity \eqref{sumrvweta}, so here we focus on the former statement about delocalization.

The origin of this delocalization criterion goes back to the work \cite{aizenman2013resonant}, where an analogous statement was shown for the tight-binding model on the regular graph. While the operator $\bm{T}$ seems quite distant from the tight-binding Hamiltonian, our starting observation is that a formulaic analogy between them arises upon expressing their diagonal resolvent entries through the Schur complement identity \eqref{qvv}. For $\bm{T}$, this gives for any $w \sim 0$ (writing $u \sim v$ if $u, v \in \mathbb{V}$ are adjacent) that
\begin{flalign}
	\label{r00zt1}
	R_{00} = -\big( z + T_{0w}^2 R_{ww}^{(0)} + K_{0,v} \big)^{-1}, \qquad \text{where $K_{0,w} = \displaystyle\sum_{\substack{w \sim 0 \\ u \ne w}} T_{0u}^2 R_{uu}^{(0)}$},
\end{flalign}

\noindent while the analogous form for the tight-binding model on the regular graph would be 
\begin{flalign}
	\label{r0020} 
	R_{00} = -\bigg( z + \sum_{u \sim 0} R_{uu}^{(0)} + K_0 \bigg)^{-1}, \qquad \text{for a random potential $K_0$ independent from $\bm{R}$}.
\end{flalign}

\noindent Viewing $K_{0,v}$ as parallel to the potential $K_0$ makes the two formulas appear similar (though the $\{ K_v \}$ are independent in the latter, while the $\{ K_{u,v} \}$ are seriously correlated in the former).

This in mind, we next verify the existence of the limit \eqref{sz0} defining $\varphi (s; z)$, by showing $\Phi_L (s; z)$ is approximately submultiplicative and supermultiplicative in $L$. This proceeds by first (repeatedly) using the resolvent expansion \eqref{r0v0}; applying \eqref{r00zt1}; and taking the expectation over $\{ T_{0u} \}$, conditional on the tree edge weights not incident to $0$. Central in implementing this in \cite{aizenman2013resonant} was the assumption that the potential $K_0$ has bounded density, which directly guarantees the uniform boundedness\footnote{Indeed, if $K$ has bounded density, then $\mathbb{E} \big[ |R_{00}|^s \big] = \mathbb{E} \big[ |A+K_0|^{-s} \big] < \infty$, upon setting $A = z + \sum_{u \sim 0} R_{uu}^{(0)}$.} of $\mathbb{E} \big[ |R_{00}|^s \big]$, conditional on the $R_{uu}^{(0)}$. Fortunately, in the L\'{e}vy setup, we can essentially identify $\Real K_{0,v}$ with an $\frac{\alpha}{2}$-stable law, which does have bounded density.

However, a complication that arises in the L\'{e}vy setting is that the edge weights $T_{0u}$ appearing in \eqref{r0v0} are heavy-tailed. Since $s > \alpha$, this makes $\mathbb{E} \big[ |T_{00_+}|^s \big]$ infinite; what guarantees finiteness of $\mathbb{E}\big[ |R_{0v}|^s \big]$ is the presence of $R_{00}$ in \eqref{r0v0}. Indeed, by \eqref{r0v0} and \eqref{r00zt1}, we have 
\begin{flalign}
	\label{r00v001}
	\mathbb{E} \big[ |R_{0v}|^s \big] = \mathbb{E} \Big[ |R_{00}|^s \cdot |T_{00_+}|^s \cdot \big| R_{0_+ v}^{(0)} \big|^s \Big] = \mathbb{E} \Bigg[ \displaystyle\frac{|T_{00_+}|^s}{\big| z + T_{00_+}^2 R_{0_+ 0_+}^{(0)} + K_{0,v} \big|^s}\cdot \big| R_{0_+ v}^{(0)} \big|\Bigg],
\end{flalign}

\noindent and so when its numerator is large then its denominator is as well.\footnote{As stated in \Cref{salpha} below, $\Phi_L (s; z)$ diverges if $s < \alpha$; this is due to the excess of many small $\{ T_{0v} \}$ (instead of the presence of a few large ones).} However, unlike in \cite{aizenman2013resonant}, this bound is not uniform conditional on the $R_{uu}^{(0)}$; it diverges as $R_{0_+ 0_+}^{(0)}$ decreases. Confronting this requires a precise understanding of the expectations on the right side of \eqref{r00v001} (accessed through various integral estimates in \Cref{G0vMoment}), and a recursive analysis of modified (see \Cref{xil} below) fractional moments weighted by an additional factor of $|R_{00}|^{\chi-s}$ (done in \Cref{EstimateMomentss}).

Next, we seek to show delocalization at $E$ if $\varphi(E) > 0$. In their setup, \cite{aizenman2013resonant} did this through two claims (whose statements we simplify considerably here). First, delocalization occurs if for each $L \ge 1$ there exists at least one $v \in \mathbb{V} (L)$ such that $|R_{0v}|$ is large; such vertices are called ``resonant.'' Second, if $\Imaginary R_{\star}$ is too small (delocalization does not occur), then the exponential divergence of $\Phi_L (s; z)$ for $(s, z) \approx (1, E)$ (equivalently, $\varphi (E) > 0$) can guarantee the existence of at least one resonant $v \in \mathbb{V} (L)$. The former claim is an eventual consequence of a deterministic, elementary linear algebraic inequality (\Cref{rsum} below). The proof of the latter is more involved and proceeds by first using a large deviations argument to show that the divergence of $\Phi_L (s; z)$ for $(s, z) \approx (1, E)$ implies that the number of resonant vertices is large in expectation. It then lower bounds the probability that there exists at least one resonant vertex using a second moment method. 

The proof of the first claim quickly adapts to the L\'{e}vy context, but that of the latter does not. A large devations argument still shows that the number of resonant vertices is large in expectation, and we implement this in \Cref{EstimateN}. However, the second moment method used to show this in probability is obstructed by the infinite and heavy-tailed nature of our tree. 

To overcome this, we introduce a more robust amplification scheme in \Cref{EstimateR}, described as follows. Using the expectation bound, we show for any $M \ge 1$ and $v_0 \in \mathbb{V}$ that there exists some $\mu \ge 0$ such that the below holds if $\varphi (E) > 0$. With probability at least $e^{-\mu M}$, there are at least $e^{(\mu + \delta) M}$ vertices $v \in \mathbb{V}$ at $M$ levels below $v_0$, such that $|R_{v_0 v}|$ is large; moreover, these events are essentially independent as $v_0$ ranges over some fixed level of the tree. Although this probability is exponentially small, it produces sufficiently many opportunities so as to ensure that the number of resonant vertices (of length $M$) likely grows (``amplifies'') as one descends down the tree. We show that ``chaining together'' $\frac{L}{M}$ of these length $M$ resonant vertices typically yields one of length $L$, which produces delocalization.

\subsubsection{Explicit Formula for the Mobility Edge}

\label{LambdaE}

We would next like to convert the inexplicit fractional moment criterion into an explicit one involving the quantity $\lambda(E,\alpha)$ from \eqref{mobilityintroduction}, to which end we will try evaluating the quantity $\Phi_L (s; z)$. 
Recalling \eqref{r0v0}, we have
\be\label{productexpand}
R_{0v} = -R_{00} T_{00_+} R_{0_+ v}^{(0)},
\ee

\noindent where $0_+ \sim 0$ lies on the path between $0$ and $v$. As in \cite{bapst2014large} (which addressed the tight-binding model on the regular graph), we repeatedly expand the fractional moments of $R_{0v}$ in $\Phi_L (s;z)$ using \eqref{productexpand}; apply \eqref{r00zt1} on each diagonal resolvent entry $R_{ww}^{(w_+)}$ in the expansion; and integrate out the randomness $\{ T_{u_- u} \}$ one level at a time. In \Cref{s:transfer}, we find that this expresses $\Phi_L(s;z)$ through the $L$-fold iteration of certain integral transfer operator $T = T_{s,\alpha,z}$, defined by
\begin{flalign}
	\label{t11} 
Tf(x) 
=
\frac{\alpha}{2} \int_{\mathbb{C}} \displaystyle\int_{\mathbb{R}}
f(y)
\left| \frac{1}{z + y} \right|^s
 |h|^{s-\alpha -1 }  p_z\left(x  + h^2 (z+y)^{-1}\right)
\, d h \, dy,
\end{flalign}
where $p_z$ is the density of $-z - R_\star(z)^{-1}$. Thus, the quantity $e^{\varphi (s;z)}$ defining the exponential growth rate of $\Phi_L (s;z)$ becomes the Perron--Frobenius eigenvalue of $T$. Unfortunately, its computation is impeded by the fact that $T$ is defined through the density of the complex random variable $R_{\star}$, which is in general inexplicit; its real and imaginary parts are correlated in an unknown way. 

A simplification arises if we assume in probability that
\begin{flalign}
	\label{eta0rlimit} 
	\displaystyle\lim_{\eta \rightarrow 0} \Im R_{\star} (E + \mathrm{i} \eta) = 0, \qquad \text{namely, localization occurs at $E$}.
\end{flalign}

\noindent In this case, we show that the 
(real) limit of the random variable $R_{\star} (E + \mathrm{i} \eta)$ as $\eta$ tends to zero is explicit, and given by the inverse of a stable law with specific parameters. We denote this limit by  $R_{\loc} = R_{\loc} (E)$ (\Cref{l:boundarybasics} below). Inserting \eqref{eta0rlimit} into \eqref{t11} and taking the associated limit as $z \in \mathbb{H}$ tends to $E \in \mathbb{R}$ yields an operator $T_{s,\alpha,E}$ that is now fully explicit. This operator appeared earlier in \cite{tarquini2016level}, where its Perron--Frobenius eigenvalue $\lambda (\alpha, s, E)$ was found exactly through a Fourier transform computation; we carefully reimplement this in \Cref{TEigenvalue}. We have under this notation that $\lambda (E, \alpha) = \lim_{s \rightarrow 1} \lambda(E, s, \alpha)$; so, assuming \eqref{eta0rlimit}, this would indicate (if one can interchange several limits) that $e^{\varphi (E)} = \lim_{s \rightarrow 1} \lambda (E, s, \alpha) = \lambda (E, \alpha)$.

This already implies that, if $\lambda (E,\alpha) > 1$, then delocalization holds at $E$. Indeed, otherwise \eqref{eta0rlimit} would hold (at least along some sequence $\{ \eta_j \}$ tending to $0$, which suffices), indicating by the above that $\varphi (E) = \log \lambda (E, \alpha) > 0$. By the fractional moment criterion, this implies that delocalization holds at $E$, which is a contradiction.

An analogous route does not show localization, to which end we introduce a bootstrap argument involving two additional statements: (i) an estimate showing localization at $E$ whenever $|E|$ sufficiently large, and (ii) a bound showing uniform continuity of $\varphi (s; E) = \lim_{\eta \rightarrow 0} \varphi (s; E + \mathrm{i} \eta)$ on the domain $\big\{ E : \varphi (s; E) < -\delta \big\}$, for any fixed $\delta < 0$. The first implicitly follows from the estimates discussed in \Cref{MomentFractionals} (this can be seen on a high level by observing from \eqref{r00zt1} that $|z|$ large indicates $R_{00}$ should be small, which by repeated use of \eqref{productexpand} suggests that $R_{0v}$ is small). The second is proven in \Cref{s:lyapunovcontinuity} through a series of restricted moment estimates and resolvent bounds; we refer to the beginning of that section for a more detailed heuristic underlying its proof.

Given these, we proceed at follows. Fixing $E > E_{\max}$ (the case $E < -E_{\max}$ is similar), we seek to show localization at $E$. To do this, observe by (i) that there exists $E_0 > E_{\max}$ sufficiently large so that localization \eqref{eta0rlimit} holds at $E_0$. Hence, $\varphi (E_0) = \log \lambda (E_0, \alpha)$; since $\lambda (x, \alpha) < 1$ for $x > E_{\max}$, it follows from (ii) that there exists a real number $E_1 \in (E, E_0)$ (uniformly bounded away from $E_0$) such that $\varphi (E_1, \alpha) < 0$. By the fractional moment criterion, this verifies \eqref{eta0rlimit} as $E_1$, which yields $\varphi (E_1) = \lambda (E_1, \alpha)$. Repeating this procedure (uniformly decreasing the $E_j$, and applying (i), (ii), and the fractional moment criterion) shows that \eqref{eta0rlimit} holds at $E$, yielding localization there.

\subsubsection{Small and Large $\alpha$ Behavior}

The main difficulty in studying the equation \eqref{mobilityintroduction} lies in the analysis of the $\ell(E)$ term, which is defined (see \eqref{tlrk}) as a certain oscillatory integral. So, we first interpret $\ell(E)$ probabilistically (see \Cref{realimaginaryl}), by expressing
\begin{flalign}
\label{ralphae} 
\Real \ell (E)  = \pi^{-1} \cdot \Gamma (\alpha) \cdot \cos \Big( \displaystyle\frac{\pi \alpha}{2} \Big) \cdot \mathbb{E} \Big[ \big| R_{\loc} (E) \big|^{\alpha} \Big],
\end{flalign}

\noindent where we recall from \Cref{LambdaE} that $R_{\loc} (E)$ is an explicit, real random variable given by the inverse of a stable law law; a similar expression holds for $\Im \ell (E)$. More specifically, we may write \begin{flalign}
	\label{re0} 
	R_{\loc} (E) = - \big( E + a(E)^{2/\alpha} \cdot S - b(E)^{2/\alpha} \cdot T \big)^{-1},
\end{flalign} 

\noindent where $S$ and $T$ are independent, normalized, positive $\frac{\alpha}{2}$-stable laws, and $a = a(E)$ and $b = b(E)$ are certain real numbers satisfying the pair of coupled fixed-point equations (here, $x_+ = \max \{ x, 0 \}$ and $x_- = \max \{ -x, 0 \}$)
\begin{equation}
	\label{ab0} 
a = \E \big[ ( E + a^{2/\alpha} S - b^{2/\alpha} T )_-^{-\alpha/2}  \big]; \qquad
b = \E \big[ ( E + a^{2/\alpha} S - b^{2/\alpha} T  )_+^{-\alpha/2}  \big].
\end{equation} 

The problem of studying solutions to $\lambda(E,\alpha) = 1$ then effectively reduces to understanding the system \eqref{ab0}, which in general seems intricate. 
However, assuming that $\alpha$ is sufficiently close to $0$ or $1$ allows for a fairly complete analysis.

If $\alpha$ is close to $1$, then we can restrict to $E$ large. Indeed, from \eqref{mobilityintroduction} (and the exact expressions for $K_{\alpha}$, $t_{\alpha}$, and $t_1$ from \eqref{tlrk}), it is quickly seen that $\lambda (\alpha, E) \sim 1$ imposes\footnote{More broadly, it imposes an asymmetric scaling between the real and imaginary parts of $\ell (E)$, namely, $\Re \ell (E) \sim (1-\alpha)^2$ and $\Imaginary \ell (E) \sim 1-\alpha$, which eventually must be taken into account when implementing the analysis in full.} $\Real \ell(E) \sim (1-\alpha)^2$. Together with \eqref{ralphae}, this indicates that $\mathbb{E} \big[ |R_{\loc} (E)|^{\alpha} \big] \sim 1-\alpha$; in particular, it must be small, so \eqref{re0} suggests that $E \gg 1$. If $E$ is large, the equations \eqref{ab0} simplify substantially. For example, we will approximate $E + \alpha^{2/\alpha} S - b^{2/\alpha} T \approx E$, which directly from \eqref{ab0} implies $b \approx E^{-\alpha/2}$. In general, we obtain precise asymptotics over how $a(E)$ and $b(E)$ scale with $E$, which we use to show that any solution $E$ to $\lambda(E,\alpha)=1$ scales as $E \sim (1-\alpha)^{-1}$; this is done in \Cref{s:alphanear1}.

The proof of uniqueness for such a solution proceeds by showing that $\lambda(E,\alpha)$ is strictly decreasing in the region $E \sim (1-\alpha)^{-1}$. This is done in \Cref{s:alphanear1uniqueness}, by pinpointing the scaling behavior (and thus realizing the signs) for the derivatives of all functions involved in the computation of $\lambda$. These include $a(E)$, $b(E)$, expectations of the form $\mathbb{E} \big[ (E+x S - yT)^{\gamma} \big]$, and $\lambda(E, \alpha)$.

If $\alpha$ is close to $0$, then from \eqref{ralphae} and  \eqref{mobilityintroduction}, it can be seen that $\lambda (\alpha, E) \sim 1$ imposes $\mathbb{E} \big[ |R_{\loc}(E)|^{\alpha} \big] \approx 2 - c_{\star} \alpha$, for some explicit constant $c_{\star} > 0$ (\Cref{lambdaalpha0} below). To analyze the equations \eqref{ab0} we then make use of (and quantify; see \Cref{w0walphavariation} below) the fact \cite{SDCS} that, for small $\alpha$, the random variables $W = \frac{\alpha}{2} \log S$ and $V = \frac{\alpha}{2} \log T$ are approximately Gumbel laws. It thus becomes convenient to parameterize $E = u^{2/\alpha}$, in which case \eqref{ab0} becomes
\begin{flalign}
	\label{ab02} 
	a = \mathbb{E} \big[ (u^{2/\alpha} + (a e^W)^{2/\alpha} - (be^V)^{2/\alpha})_-^{-\alpha/2} \big]; \qquad b = \mathbb{E} \big[ (u^{2/\alpha} + (ae^W)^{2/\alpha} - (be^V)^{2/\alpha})_+^{-\alpha/2} \big].
\end{flalign}

\noindent Approximating $(W, V)$ by independent Gumbel random variables, the expectations in \eqref{ab02} become explicit, especially upon further approximating $(w^{2/\alpha} + x^{2/\alpha} - y^{2/\alpha})^{\alpha/2} \approx \pm \max \{ w, x, y \}$. 

This enables one to nearly solve to the equations \eqref{ab02} for $(a, b)$ in terms of $u$, and thus to show that any solution $u$ to $\lambda (u^{2/\alpha}, \alpha) = 1$ (or $\mathbb{E}\big[ |R_{\loc} (u^{2/\alpha})|^{\alpha} \big] \approx 2 - c_{\star} \alpha$) must scale as $u \sim |\log \alpha|^{-1}$; we implement this in \Cref{Alpha0Scaling}. The proof of uniqueness for such a solution proceeds by showing that $\lambda (u^{2/\alpha},\alpha)$ is strictly decreasing in the region where $u \sim |\log \alpha|^{-1}$. This is done in \Cref{E0Unique}, again by analyzing how the derivatives of all functions involved in the computation of $\lambda (E, \alpha)$ scale.

\subsection{Notation}\label{s:notation}

We let $\mathbb{H} = \{ z \in \mathbb{C} : \Imaginary z > 0 \}$ and $\bbR_+ = \{ r \in \bbR : r \ge 0 \}$. The notation $\unn{a}{b}$ for $a,b \in \mathbb{Z}$ denotes the set $\{ k \in \mathbb{Z} : a \le k \le b\}$. We define the function $z \mapsto z^a$ for $a \in (0,1)$ and $z\in \bbC \setminus (-\infty, 0)$ by $z^a = r^a \exp( \iu a \theta)$ if $z = r \exp(\iu \theta)$ for $\theta \in ( -\pi, \pi]$. The function $\log$ always denotes the natural logarithm. Given $x \in \bbR$, we set $(x)_+ = \max(x,0)$ and $(x)_- = \max(-x,0)$. For any two functions $f, g \colon \mathbb{R} \rightarrow \mathbb{C}$, we let $f * g : \mathbb{R} \rightarrow \mathbb{C}$ denote their convolution, given by $(f \ast g) (x) = \int_{-\infty}^{\infty} f(y) g(x-y) \, dy$. Moreover, for a Banach space $X$; subset $S \subseteq X$; and complex number $z \in \mathbb{C}$, we let $z \cdot S = \{ zs : s \in S \}$. Given an event $\mathscr E$ in some probability space, we let $\mathscr E^c$ denote its complement.

We will frequently define constants that depend on some number of parameters. These may be introduced either as $C = C(x_1, \dots, x_n)$ or simply as $C(x_1, \dots, x_n)$, for parameters $x_1, \dots, x_n$, and subsequently referred to as $C$ (suppressing the dependence on the parameters in the notation). Such constants may change line to line without being renamed; in all such cases the new constant depends on the same set of parameters as the previous one. We usually write $C>1$ for large constants, and $c>0$ for small constants.

We will often fix constants $\alpha \in (0,1)$ and $s \in (\alpha, 1)$, and sometimes omit notating the dependence on these parameters from our definitions of certain constants for brevity. Similarly, we will often fix $\eps$ and $B$ such that $\eps \in (0,1)$ and $B > 1$ and make claims for all points $z\in \mathbb{H}$ such that $\eps \le |\Re z| \le B$  and $\Im z \in (0,1)$ (which may also be subject to further conditions). We may similarly omit $\eps$ and $B$ from our notation. In statements about any of the parameters $\alpha$, $s$, $\eps$, and $B$, all constants are understood to depend on these parameters, unless explicitly stated otherwise. 

\subsection{Acknowledgements} 

The authors are immensely grateful to Alice Guionnet for extremely valuable and extensive discussions related to this work. They also thank Giulio Biroli, Marco Tarzia, and Horng-Tzer Yau for helpful conversations on this topic. Amol Aggarwal was partially supported by a Clay Research Fellowship, by the NSF grant DMS-1926686, and by the IAS School of Mathematics.
Patrick Lopatto was partially supported by the NSF postdoctoral fellowship DMS-2202891. Amol Aggarwal and Patrick Lopatto also wish to acknowledge the NSF grant DMS-1928930, which supported their participation in the Fall 2021 semester program at MSRI in Berkeley, California titled, ``Universality and Integrability in Random Matrix Theory and Interacting Particle Systems.''

\section{Results}

\label{Results}

	\subsection{L\'{e}vy Matrices and Their Local Limits} 
	
	\label{Stable} 
	
	In this section we recall L\'{e}vy matrices and a heavy tailed operator on an infinite tree that can be viewed as their local limit. 
We begin by recalling a class of heavy-tailed probability distributions defined in \cite[Section 1]{bordenave2011spectrum}.
It contains symmetric power laws and distributions close to such laws.

\begin{definition}\label{LaDef}
For every $\alpha \in (0,1)$, we define $\bbL_{\alpha}$ as the set of probability measures $\mu$ on $\bbR$ that are absolutely continuous with respect to Lebesgue measure and satisfy the following conditions.
\begin{enumerate}
\item 
The function $L: \bbR_+ \rightarrow \bbR$ given by 
\bex
L(t)  = t^{\alpha}\mu\big( (-\infty, t) \cup (t, \infty)\big)
\eex
has slow variation at $\infty$; this means that for any $x > 0$,
\bex
\lim_{t \rightarrow \infty} \frac{L(xt)}{L(t)} = 1. 
\eex
\item We have 
\bex
\lim_{t \rightarrow \infty} \frac{\mu\big (t, \infty) \big) }{\mu\big( (-\infty, t) \cup (t, \infty) \big)} = \frac{1}{2}.
\eex

\end{enumerate}
\end{definition}

We next define the class of random matrices considered in this work. 
	\begin{definition} 
		
		\label{matrixh}
		
		Fix an integer $N \ge 1$, and let $\{ h_{ij} \}_{1 \le i \le j \le N}$ denote mutually independent random variables that each have law $N^{-1 / \alpha} X$, where $X$ is a random variable whose law is contained in $\bbL_\alpha$. Set $h_{ij} = h_{ji}$ for each $1 \le j < i \le N$, and define the $N \times N$ random, real symmetric matrix $\boldsymbol{H} = \{ h_{ij} \}_{1 \le i, j \le N}$, which we call a \emph{L\'{e}vy matrix} (with parameter $\alpha$). 
		
	\end{definition} 

	It will further be useful to introduce notation for the resolvent of a L\'{e}vy matrix.
	
	\begin{definition} 
		
		\label{g} 
		
		Adopt the notation of \Cref{matrixh}, and fix a complex number $z \in \mathbb{H}$. The \emph{resolvent} of $\boldsymbol{H}$, denoted by $\boldsymbol{G} = \boldsymbol{G} (z) = \{ G_{ij} \}_{1 \le i, j \le N} = \big\{ G_{ij} (z) \big\}_{i, j}$, is defined by $\boldsymbol{G} = (\boldsymbol{H} - z)^{-1}$.	

	\end{definition}

	We next recall from \cite[Section 2.3]{bordenave2011spectrum} a description of the local limit of a L\'{e}vy matrix $\boldsymbol{H}$ as a heavy-tailed operator on an infinite tree $\mathbb T$ (which may be viewed as a variant of the Poisson Weighted Infinite Tree introduced in \cite{aldous1992asymptotics,aldous2004objective}). This tree has vertex set $\mathbb{V} = \bigcup_{k = 0}^{\infty} \mathbb{Z}_{\ge 1}^k$, with root $\mathbb{Z}^0$ denoted by $0$, so that the children of any $v \in \mathbb{Z}_{\ge 1}^k \subset \mathbb{V}$ are $(v, 1), (v, 2), \ldots \in \mathbb{Z}_{\ge 1}^{k + 1}$. 
	 For each vertex $v \in \mathbb{V}$, let $\boldsymbol{\xi}_{v} = (\xi_{(v, 1)}, \xi_{(v, 2)}, \ldots )$ denote a Poisson point process on $\mathbb{R}$ with intensity measure $\frac{1}{2} dx$ (where $dx$ is the Lebesgue measure on $\mathbb{R}$), ordered so that $|\xi_{(v, 1)}| \ge |\xi_{(v, 2)}| \ge \cdots $. Let $\mathcal{F} \subset L^2 (\mathbb{V})$ denote the (dense) set of vectors with finite support. For any vertex $v \in \mathbb{V}$, let $\delta_{v} \in \mathcal{F}$ denote the unit vector supported on $v$. Then define the operator $\boldsymbol{T} : \mathcal{F} \rightarrow L^2 (\mathbb{V})$ by setting
	\begin{flalign} 
		\label{t} 
		\begin{aligned}
		\langle \delta_{v}, \boldsymbol{T} \delta_{w} \rangle & = \sgn \xi_{w} \cdot |\xi_{w}|^{-1 / \alpha}, \qquad \text{if $w = (v, j)$ for some $j \in \mathbb{Z}_{\ge 1}$}; \\
		\langle \delta_{v}, \boldsymbol{T} \delta_{w}  \rangle& = \sgn \xi_{v} \cdot |\xi_{v}|^{-1 / \alpha}, \qquad \text{if $v = (w, j)$, for some $j \in \mathbb{Z}_{\ge 1}$}; \\
		\langle \delta_{v}, \boldsymbol{T} \delta_{w} \rangle & = 0, \qquad \qquad \qquad \qquad \text{otherwise},
		\end{aligned}
	\end{flalign} 
	
	\noindent for any vertices $v, w \in \mathbb{V}$, and extending it by linearity to all of $\mathcal{F}$. We abbreviate the matrix entry $T_{v w} = \langle \delta_{v}, \boldsymbol{T} \delta_{w} \rangle$ for any vertices $v, w \in \mathbb{V}$. 
	We further identify $\boldsymbol{T}$ with its closure, which is self-adjoint by \cite[Proposition A.2]{bordenave2011spectrum}. 
	
	\begin{rem}
		
	\label{tuvalpha} 
	
	For any $v \in \mathbb{V}$, let $\mathbb{D} (v)$ denote the set of children of $v$. Then, the sets of random variables $\big\{ |T_{vw}| \big\}_{w \in \mathbb{D}(v)}$ and $\big\{ |T_{vw}^2| \big\}_{w \in \mathbb{D} (v)}$ form Poisson point processes on $\mathbb{R}_{> 0}$ with intensity measures $\alpha x^{-\alpha-1} dx$ and $\frac{\alpha}{2} x^{-\alpha/2 - 1} dx$, respectively.	
	
	\end{rem} 
	
	The following definition provides notation for the resolvent 
	of $\boldsymbol{T}$. 
	
	\begin{definition}
		
	\label{r} 
	
	For any complex number $z \in \mathbb{H}$, the \emph{resolvent} of $\boldsymbol{T}$, denoted by $\boldsymbol{R}: L^2 (\mathbb{V}) \rightarrow L^2 (\mathbb{V})$, is defined by $\boldsymbol{R} = \boldsymbol{R} (z) = (\boldsymbol{T} - z)^{-1}$. For for any vertices $v, w \in \mathbb{V}$, we denote the $(v, w)$-entry of $\boldsymbol{R}$ by $R_{v w} = R_{v w} (z) = \langle \delta_{v}, \boldsymbol{R} \delta_{w} \rangle$. 
	
	\end{definition} 

	The following lemma, which follows from results in \cite{bordenave2011spectrum}, implies that the (diagonal) entries of $\boldsymbol{G}$ converge to those of $\boldsymbol{R}$ as $N$ tends to $\infty$. 
	
	\begin{lem}[{\cite[Theorem 2.2, Theorem 2.3]{bordenave2011spectrum}}] 

	\label{gr} 
	
	Adopt the notation of \Cref{g} and \Cref{r}. For any integer $i \in \unn{1}{N}$, we have the convergence in law $\lim_{N \rightarrow \infty} G_{ii} (z) = R_{\star} (z)$. 
	
	\end{lem}

We close this section with some notation on stable laws. Fix real numbers $\alpha \in (0,1)$, $\sigma \ge 0$, and $\beta \in [-1, 1]$. We say that a random variable $X$ is an \emph{$\alpha$-stable law} with \emph{scaling parameter $\sigma$} and \emph{skewness parameter $\beta$}, or an \emph{$(\alpha; \sigma; \beta)$-stable law}, if for any $t \in \mathbb{R}$ we have 
	\begin{flalign}
		\label{xtsigma}
		\mathbb{E} \big[ e^{\mathrm{i} tX} \big] = \exp \bigg(- \displaystyle\frac{\pi}{2 \sin (\pi \alpha / 2) \Gamma (\alpha)} \cdot \sigma^{\alpha} |t|^{\alpha} \big( 1 - \mathrm{i} \beta \sgn (t) u_{\alpha} \big) \bigg), \quad \text{where} \quad u_{\alpha} = \tan ( \pi \alpha/2 ).
	\end{flalign}
		
	\noindent The constant $\pi \big(2 \sin (\pi \alpha / 2) \Gamma (\alpha)\big)^{-1}$ in \eqref{xtsigma} is chosen so that when $\sigma = 1$ and $\beta=0$, the associated law is in the class $\mathbb{L}_\alpha$ defined in \Cref{LaDef}. See \cite[Theorem 1.2]{nolan2020univariate} for details.
	
We further say that $X$ is \emph{centered} if $\beta = 0$, and that it is \emph{nonnegative} if $\beta = 1$. If $X$ is a nonnegative $\alpha$-stable law with scaling parameter $\sigma$, then $Y \ge 0$ almost surely, and for any $t \ge 0$ we have 
\begin{equation}
		\label{ytsigma} 
		\mathbb{E} [e^{-tX}] = \exp \big( - \Gamma (1 - \alpha) \cdot \sigma^{\alpha} |t|^{\alpha} \big). 
\end{equation}

	We adopt the convention that any reference to a nonnegative $\alpha$-stable law without further qualification always specifies one with scaling parameter $\sigma =1$.

	\begin{rem} 
		
		\label{sumalpha} 
		
		 Any $\alpha$-stable law $X$ is the difference between two nonnegative stable laws. Indeed, if the scaling and skewness parameters of $X$ are $\sigma = \sigma (X) \ge 0$ and $\beta = \beta(X) \in [-1, 1]$, respectively, then by \eqref{xtsigma} we have $X = Y - Z$, where $Y$ and $Z$ are nonnegative $\alpha$-stable laws with scaling parameters 
		 \begin{flalign*}
		 	\sigma (Y) = \sigma \cdot \bigg( \displaystyle\frac{\beta + 1}{2} \bigg)^{1/\alpha}, \quad \text{and} \quad \sigma (Z) = \sigma \cdot \bigg( \displaystyle\frac{1 - \beta}{2} \bigg)^{1 / \alpha}, \quad \text{respectively}.
		 \end{flalign*}
		
	\end{rem}

\subsection{Results for the Poisson Weighted Infinite Tree}	\label{s:PWITresults}

We begin with some preliminary definitions.
\begin{definition}
Fix $\alpha \in (0,1)$ and recall $R_\star$ from \Cref{r}.
We define $y(z) = y_\alpha(z)$ for $z\in \bbH$ by 
\be\label{yrepresentation}
y(z) = \E \left[ \left(- \iu R_\star(z)   \right)^{\alpha/2}  \right].
\ee
\end{definition}
\begin{rem}
Alternatively, $y(z)$ can be characterized through a deterministic equation that does not reference the operator $\R$. Set $\bbK = \{ z \in \bbC: \Re z > 0\}$. For any $z\in \bbH$, we define $\phi_{\alpha, z} \colon \bbK \rightarrow \bbC$ by
\bex
\phi_{\alpha,z}(x) = \frac{1}{\Gamma(\alpha/2)} 
\int_0^\infty t^{\alpha/2 - 1} \exp(\iu tz) \exp\left(  
-\Gamma(1 - \alpha/2) t^{\alpha/2} x
\right)\, dt.
\eex
By \cite[Theorem 1.4]{arous2008spectrum} and \cite[Proposition 3.1]{bordenave2013localization}, 
$y(z)$ is the unique solution to the equation
$ y(z) = \phi_{\alpha,z}( y(z))$ for all $z \in \bbH$. 
\end{rem}
 
It follows from the proof of \cite[Theorem 1.4]{belinschi2009spectral} that $y(z)$ extends continuously to a function on $\overline{\bbH}$, so that $y(E)$ is defined for every $E\in \bbR$. For the convenience of the reader, we provide a proof of the previous statement in \Cref{s:boundaryvalues}, where it is stated as \Cref{l:boundaryvalues}.
\begin{definition}
Because the complex numbers $\iu^{\alpha/2}$ and $(-\iu)^{\alpha/2}$ are linearly independent over $\bbR$, for every $E\in \bbR$ there exists a  unique pair $\big(a(E), b(E)\big) \in \bbR^2$ such that 
\be\label{opaque}
y(E) = (-\iu)^{\alpha/2} a(E) + (\iu)^{\alpha/2} b(E).
\ee
Further, we have
\bex
\big (-iR_\star(z)\big)^{\alpha/2} \in \big\{ \mathrm{i}^{\alpha/2} \cdot a + (-\mathrm{i})^{\alpha/2} \cdot b : a, b \ge 0 \big\}
\eex
because $R_\star(z) \in \bbH$, so  $a(E), b(E) \ge0$. (Recall the definition of $z \mapsto z^{\alpha/2}$ given in \Cref{s:notation}.)
\end{definition}
\begin{definition}
	\label{pkappae} 
	
Let $S_1$ and $S_2$ be independent, nonnegative $\alpha/2$-stable random variables. Set 
\bex
\varkappa_{\loc}(E) = a(E)^{2/\alpha} S_1 - b(E)^{2/\alpha} S_2.
\eex
Then $\varkappa_\loc(E)$ is an $\alpha/2$-stable law whose scaling and skewness parameters are determined through \Cref{sumalpha}. 
We denote the density of $\varkappa_\loc(E)$ with respect to the Lebesgue measure on $\bbR$ by $p_E(x)$. We let $\hat p_E(\xi)$ denote the Fourier transform of the density $p_E(x)$.\footnote{
We define the Fourier transform so that $\hat p_E(\xi)$ is equal to the characteristic function of $\varkappa_{\loc}(E)$. See \Cref{s:fouriertransform} below for details.}
\end{definition} 

 While these quantities depend on $\alpha$, we suppress this dependence in our notation.
 
 \begin{rem}\label{r:abreason}
The definition of $a(E)$ and $b(E)$ through \eqref{opaque} can be understood in the following way.
Suppose that $R_\star(E) = \lim_{\eta \rightarrow 0} R_\star(E + \iu \eta)$ exists and satisfies $\Im R_\star(E) = 0$.
Then we have
\bex
a(E) = \E\left[ \big( R_\star(E) \big)^{\alpha/2}_+   \right], \qquad b(E) = \E\left[ \big( R_\star(E) \big)^{\alpha/2}_-  \right].
\eex
The assumption $\Im R_\star(E) = 0$ essentially implies localization, as noted below in \Cref{l:loccriteria}. So one may think of the definitions of $a(E)$ and $b(E)$ as being the boundary values of fractional moments at $E$, assuming localization at $E$.

The definition of $\varkappa_{\loc}$ arises in a related way. Given $z\in \bbH$, let $(\xi_j)_{j \ge 1}$ be a Poisson point process on $\mathbb{R}_{> 0}$ with intensity measure $\frac{\alpha}{2} x^{-\alpha / 2 - 1} dx$, and let $(R_j(z))_{j \ge 1}$ be mutually independent random variables with law $R_{\star} (z)$. The quantity
\begin{flalign*}
\displaystyle\sum_{j = 1}^{\infty} \xi_j R_j(z),
		\end{flalign*}
	known as the \emph{self energy} at $z$, plays a key role both in our work and in the physics literature \cite{tarquini2016level,cizeau1994theory}. If we consider the boundary value of the self energy at a point $E \in \bbR$, then the localization assumption $\Im R_\star(E) = 0$  and the Poisson thinning property (see \Cref{sumaxi} below) imply that the previous sum is equal in distribution to
\bex
\E\left[ \big( R_\star(E) \big)^{\alpha/2}_+   \right]^{2/\alpha} S_1 - 
\E\left[ \big( R_\star(E) \big)^{\alpha/2}_+   \right]^{2/\alpha} S_2,
\eex
with $S_1$ and $S_2$ as above, yielding the definition of $\varkappa_{\loc}$.
\end{rem}

We next define a quantity which plays a key role in our main results. 

\begin{definition} \label{lambdaEalpha}

	For any $x \in \mathbb{R}$ and $y \in (0, 1)$, define 
	\begin{flalign}
		\label{tlrk}
		\ell (x) = \displaystyle\frac{1}{\pi} \displaystyle\int_0^{\infty} e^{\mathrm{i} x \xi} |\xi|^{\alpha-1} \widehat{p}_E (\xi)\, d \xi; \quad K_{\alpha} = \displaystyle\frac{\alpha}{2} \cdot \Gamma \bigg( \displaystyle\frac{1-\alpha}{2} \bigg)^2 ;
	 \quad t_y = \sin \bigg( \displaystyle\frac{\pi y}{2} \bigg).
	\end{flalign}

	\noindent For any $E\in \bbR$ and $\alpha \in (0,1)$, we  let $\lambda(E, \alpha)$ denote the unique positive solution to the quadratic equation 
	\begin{flalign}\label{mobilityquadratic}
		\lambda(E, \alpha)^2 - 2 t_{\alpha} K_{\alpha} \Real \ell (E) \cdot \lambda(E, \alpha) + K_{\alpha}^2 (t_{\alpha}^2 - 1) \big| \ell (E) \big|^2 = 0,
	\end{flalign}

	\noindent which admits only one positive solution since its constant term is negative (as $t_{\alpha} < 1$ for $\alpha  < 1$).
	
	\end{definition}

We now state a proposition that provides important context for our main theorems. This proposition and the theorem that follows it are proved in \Cref{s:proofmain1}, assuming several preliminary lemmas that are proved in the remainder of the paper.
\begin{prop}\label{p:213} 
Fix $\alpha \in (0,1)$.
\begin{enumerate}
\item There exists a  constant $c(\alpha) > 0$ such that $\lambda(E, \alpha) > 1$ for $|E| < c$.
\item There exists a constant $C(\alpha) > 1$ such that  $\lambda(E, \alpha) < 1$ for $|E| >C$.
\end{enumerate}
\end{prop}

We now state our first main theorem. It connects $\lambda(E,\alpha)$ to the behavior of $R_\star(E + \iu \eta)$ as $\eta$ tends to zero. The conclusions of the first and second parts should be understood as signatures of delocalization and localization, respectively; this connection will be made precise below in \Cref{l:loccriteria}. The theorem essentially states that one has delocalization for energies $E$ such that $\lambda > 1$, and localization for $E$ in the connected components of $\lambda < 1$ containing $\infty$ and $-\infty$. We start with a definition of strict lower boundedness in probability. 

\begin{definition}\label{def:lbprob}
Consider a family of real random variables $(X_\eta)_{\eta \geq 0}$ and $c \in \mathbb{R}$, we write 
\emph{$
\liminf_{\eta \to 0} X_\eta > c
$ in probability}
if for all $\omega \in (0,1)$, there exists $\delta > 0$ such that 
\begin{flalign} \label{2141}
\liminf_{\eta \to 0} \mathbb{P} ( X_\eta > c + \delta ) \geq 1 - \omega.
\end{flalign}
\end{definition} 

We now state the theorem.

\begin{thm}\label{t:main1} Fix $\alpha \in(0,1)$.
\begin{enumerate}
\item For all $E \in \bbR$ such that $\lambda(E, \alpha) > 1$, $
 \liminf_{\eta \rightarrow 0} \Imaginary R_{\star} (E + \mathrm{i} \eta) > 0$ in probability.
\item Set
\be\label{eloc}
E_{\mathrm{loc}}(\alpha) = \sup \{ E \in \bbR_+ : \lambda(E,\alpha) =1 \}.
\ee
For all $E \in \bbR$ such that $|E| > E_{\mathrm{loc}}$, we have $\lambda(E, \alpha) < 1$ and $\lim_{\eta \rightarrow 0 }\Im R_\star(E + \iu \eta) = 0$ in probability.
\end{enumerate}
\end{thm}

We remark that \Cref{t:main1} makes no claim about 
localization or delocalization in any connected components of $\lambda < 1$ avoiding $\infty$ and $-\infty$. Conjecturally, such regions do not exist, and the following two theorems asserts thi for $\alpha$ sufficiently close to $1$ or to $0$, respectively. For such $\alpha$, it implies that delocalization holds for all $E\in \bbR$ such that $|E| < E_\mathrm{loc}$, and establishes the existence of exactly two sharp phase transitions between regions of delocalization to localization at $\pm E_\mathrm{loc}$. Further, it pinpoints  the scalings of these transitions (as $(1-\alpha)^{-1}$ for $\alpha$ near $1$, and as $|\log \alpha|^{-2/\alpha}$ for $\alpha$ near $0$). The first part of the theorem below is proved in \Cref{s:proveuniqueness}, and the second is proved in \Cref{Alpha0Unique}.

\begin{thm}\label{t:main2}
There exists a constant $c>0$ such that the following statements hold.
\begin{enumerate}
\item For every $\alpha \in (1-c, 1)$, the set \be\label{zalpha}
\mathcal Z_\alpha = \{E\in \bbR : E \ge 0, \, \lambda(E, \alpha) = 1\}
\ee consists of a single point $E_{\mob} = E_\mob(\alpha)$. It satisfies
\bex
c(1-\alpha)^{-1} < E_\mob < c^{-1}(1-\alpha)^{-1}.
\eex
\item For every $\alpha \in (0, c)$, the set $\mathcal{Z}_{\alpha}$  consists of a single point $E_{\mob} = E_{\mob} (\alpha)$. It satisfies 
	\begin{flalign}
		\label{ealpha0} 
		\bigg( \displaystyle\frac{1}{|\log \alpha|} - \displaystyle\frac{C \log |\log \alpha|}{|\log \alpha|^2} \bigg)^{2/\alpha} \le E_{\mob} \le \bigg( \displaystyle\frac{1}{|\log \alpha|} + \displaystyle\frac{C \log |\log \alpha|}{|\log \alpha|^2} \bigg)^{2/\alpha}.
	\end{flalign}
\end{enumerate}
\end{thm}

\subsection{Results for L\'evy Matrices}

	For any $N \times N$ symmetric matrix $\boldsymbol{M}$, let $\Lambda = \Lambda (\boldsymbol{M}) = (\lambda_1, \lambda_2, \ldots , \lambda)$ denote the set of eigenvalues of $\boldsymbol{M}$, and let $\boldsymbol{u}_i = \big( u_i (1), u_i (2), \ldots , u_i (N) \big) \in \mathbb{R}^N$ denote the eigenvector of $\boldsymbol{M}$ with eigenvalue $\lambda_i$ and $\ell^2$ norm equal to one. For any interval $I \subset \mathbb{R}$, let $\Lambda_I = \Lambda_I (\boldsymbol{M}) = \Lambda \cap I$. Following \cite{bordenave2013localization}, we introduce the below measure of eigenvector localization.
	
	\begin{definition}[{\cite[Section 1]{bordenave2013localization}}]
	
	\label{jwi}	
	
	For any interval $I \subset \mathbb{R}$ and index $j \in \unn{1}{N}$, we define $P_I (j) = P_I (j; \boldsymbol{M})$ and $Q_I = Q_I (\boldsymbol{M})$ by
	\begin{flalign*}
		P_I (j) = \displaystyle\frac{N}{|\Lambda_I|} \displaystyle\sum_{\lambda_i \in \Lambda_I} \big| u_i (j) \big|^2; \qquad Q_I = \displaystyle\frac{1}{N} \displaystyle\sum_{j = 1}^N P_I (j)^2.
	\end{flalign*}

	\end{definition} 

	\begin{rem} 
		
		\label{westimate}
	 
		Let us describe the asymptotic behavior of $Q_I$ in two regimes. In the ``maximally delocalized regime,'' we have $u_i (j) = N^{-1/2}$ for each $i, j \in \unn{1}{ N}$, in which case $P_I (j) = 1$ for all $j$, and $Q_I = 1$. In the ``maximally localized regime,'' we have $u_i (j) = \one_{i = j}$. In this case, $P_I (j) = 0$ if $\lambda_j \notin I$ and $P_I (j) = N |\Lambda_I|^{-1}$ otherwise,  so $Q_I = N |\Lambda_I|^{-1} $. In the latter case, for a fixed interval $I\subset \bbR$, we have $|\Lambda_I| < C  |I| N$ for some $C > 0$ (independent of $I$) by \cite[Theorem 1.1]{arous2008spectrum}.\footnote{We let $|I|$ denote the length of $I$.} Then $Q_{I} > C^{-1} | I|^{-1}$. Thus we may distinguish delocalization from localization, for eigenvectors corresponding to eigenvalues in $I$, by examining whether $Q_I$ remains bounded or grows as $|I|$ tends to zero, respectively. 	 
	\end{rem} 

We now introduce the notions of eigenvector localization and delocalization that we adopt in this paper.

\begin{definition}\label{d:loc}
Fix $E \in \bbR$ and a L\'evy  matrix $\bfH = \bfH_N$. We say that \emph{delocalization holds at $E$} for $\bfH$ if there exist $C(E)>1$ and $c(E) >0$ such that for all  $\eps \in (0, c)$, 
		\begin{flalign*}
			\displaystyle\lim_{N \rightarrow \infty} \mathbb{P} \big[ Q_{I(\eps)} \left( \boldsymbol{H} \right) \le C \big] = 1
		\end{flalign*}
for the interval $I(\eps) = [E- \eps, E + \eps]$. We say that \emph{localization holds at $E$} for $\bfH$ if for every $D > 0$, there exists $c(D)>0$ such that for all 
$\eps \in (0, c)$,
		\begin{flalign*}
			\displaystyle\lim_{N \rightarrow \infty} \mathbb{P} \big[ Q_{I(\eps)} \left( \boldsymbol{H} \right) \ge D \big] = 1.
		\end{flalign*}
\end{definition}


The next lemma asserts that the behavior of  $\Imaginary R_{\star} (E + \iu \eta)$ as $\eta$ tends to $0$ controls the behavior of the eigenvectors $\boldsymbol{H}$ with eigenvalues near $E$. Indeed, $\Imaginary R_{\star} (z)$ tending to $0$ in probability will imply localization for these eigenvectors, in the sense of \Cref{d:loc}. Similarly,  $\Imaginary R_{\star} (z)$ being positive with positive probability as $\eta$ tends to zero implies delocalization for these eigenvectors. We now state these two claims as a lemma; it is proved in 
\Cref{a:conclusion} below.

\begin{lem}\label{l:loccriteria}
Fix $\alpha \in (0,1)$ and $E \in \bbR$, and let $\bfH$ be a L\'evy matrix with parameter $\alpha$. 
\begin{enumerate}
\item If $\lim_{\eta \rightarrow 0 }\Im R_\star(E + \iu \eta) = 0$ in probability, then localization holds at $E$ for $\bfH$.
\item If there exists a real number $c > 0$ such that $\liminf_{\eta \rightarrow 0} \mathbb{P} \big[ \Imaginary R_{\star} (E + \mathrm{i} \eta) > c \big] > c$, then delocalization holds at $E$ for $\bfH$.
\end{enumerate}
\end{lem}

The following corollary asserts that the result of the previous section, which concerned the operator $\R$ on $\mathbb{T}$, also hold for L\'evy matrices. It is a straightforward consequence of \Cref{t:main1}, \Cref{t:main2}, and \Cref{l:loccriteria}. We prove it in \Cref{s:proofmain1}.

\begin{cor}\label{c:main1}
Fix $\alpha \in(0,1)$, and let $\bfH$ be a L\'evy matrix with parameter $\alpha$. 
\begin{enumerate}
\item For all $E \in \bbR$ such that $\lambda(E, \alpha) > 1$, delocalization holds at $E$ for $\bfH$.
\item Recall the quantity $E_\loc$ defined in \eqref{eloc}.
For all $E\in \bbR$ such that $|E| > E_{\mathrm{loc}}$, 
localization holds at $E$ for $\bfH$.
\item There exists $c>0$ such that if $\alpha\in(0,c)$ or $\alpha \in (1-c, 1)$, delocalization holds at $E$ when $|E| < E_\loc$, and localization holds at $E$ when $|E| > E_\loc$.
\end{enumerate}
\end{cor}

The focus of this paper is on the exact computation of the mobility edge for Lévy matrices. We might expect that a quantitative version of Corollary \ref{c:main1} (at least in the delocalized regime) might be obtained by combining the results of the present paper with previous works, notably \cite{bordenave2013localization, bordenave2017delocalization, aggarwal2021goe} but will not pursue this here.

%
%
%
%
%
 
  \subsection{Relation to Previous Work} \label{s:relationprevious}
 
 We now discuss the relation between our results and previous work on L\'evy matrices in the physics and mathematics literature. We first review the physics predictions from \cite{tarquini2016level, cizeau1994theory}. 
 
 These works posited that the set $\mathcal Z_\alpha$ (from \Cref{t:main2}) consists of a single point for each $\alpha \in (0, 1)$, that is, $\mathcal Z_\alpha =  \{E_{\mathrm{mob}}\}$ for some $E_{\mathrm{mob}} > 0$. Further, \cite{cizeau1994theory,tarquini2016level} predicted that eigenvector delocalization holds for $|E| < E_{\mathrm{mob}}$ and eigenvector localization holds for $|E| > E_{\mathrm{mob}}$. Regarding eigenvalues, they predicted that the local eigenvalue statistics are distributed as those of the GOE for energies $|E| < E_{\mathrm{mob}}$, and that these statistics are Poisson for $|E| > E_{\mathrm{mob}}$. The authors of \cite{tarquini2016level} also predicted that $E_\mob$ scales as $(1- \alpha)^{-1}$ for $\alpha$ near $1$. They gave no analytical prediction on how $E_{\mob}$ scales for $\alpha$ near $0$, but the plot \cite[Figure 1]{tarquini2016level} indicates that $E_{\mob}$ is very small  for $\alpha < \frac{1}{4}$ (which is in some sense confirmed by the super exponential decay in $1/\alpha$ of $E_\mob$ in Theorem \ref{t:main2}).

The predictions of \cite{tarquini2016level, cizeau1994theory} were subsequently investigated in the mathematically rigorous works \cite{bordenave2013localization, bordenave2017delocalization, aggarwal2021goe}. 
For all $\alpha \in (0, 2/3)$, \cite{bordenave2013localization} showed that there exists a constant $C(\alpha) > 1$ such that if $\bfH$ is a $N\times N$ L\'evy matrix with $\alpha$-stable entries and resolvent $\G$, then the fractional moment bound 
\bex
\E \left[ \big( \Im G_{11}(E + \iu \eta) \big)^{\alpha/2}  \right] = O(\eta^{\alpha/2 - \delta})
\eex
holds for $|E| > C(\alpha)$ and $\eta \gg N^{-(2+\alpha)/(4\alpha + 12)}$. Using this bound, \cite{bordenave2017delocalization} showed  the following localization result for the same range of $\alpha$. For every $\eps \in (0, \alpha/2)$, there exists $C(\alpha, \eps)>1 $ such that for any compact interval $K \subset  [-C,C]^c$, there exists a constant $ c_1(K) > 0$ such that the following bound holds. For any interval $I\subset K$ such that $|I| \ge N^{-\frac{\alpha}{2 + 3 \alpha}} (\log n)^2$, we likely have $Q_I \ge c_1 | I |^{ - \frac{2\eps}{2 - \alpha}}$. Further, for almost all $\alpha \in(0,1)$, \cite{bordenave2017delocalization} proved that if $\bfH$ is a L\'evy matrix with $\alpha$-stable entries, there exist constants $C(\alpha)>1$ and $c(\alpha) > 0$ such that for every interval $I \subset [-c, c]$ satisfying $|I| \ge n^{-\alpha/(4+\alpha)} (\log N)^{C}$, we likely have $Q_I \le C$. Finally, \cite{aggarwal2021goe} proved that eigenvectors of $\bm{H}$ with eigenvalues $\lambda_k \in [ -c,c]$ are completely delocalized (for a slightly more general class of matrices), and also confirmed the prediction of GOE eigenvalue statistics around energies $E \in [-c,c]$.

Our results pinpoint the exact mobility edge serving as the sharp transition between localization and delocalization, at the expense of not pursuing the effective or optimal (de)localization bounds showcased in the previous works. Indeed, \Cref{c:main1} verifies the (de)localization (in terms of the inverse participation ratio $Q_I$) aspects of the phase diagram predicted in \cite{tarquini2016level,cizeau1994theory}, conditional on the purely deterministic statement that $\mathcal{Z}_{\alpha}$ consists of a unique point (equivalently, that \eqref{mobilityquadratic} admits a unique solution in $E$ for fixed $\alpha\in(0,1)$). The third part of \Cref{c:main1} shows this uniqueness if $\alpha \in (0,c) \cup (1-c, 1)$ for some sufficiently small constant $c > 0$. Further, our \Cref{t:main2} shows that the unique point $E_{\loc} (\alpha) \in \mathcal{Z}_{\alpha}$ scales as $(1-\alpha)^{-1}$ for $\alpha$ close to $1$, as predicted in \cite{tarquini2016level}. For $\alpha$ near $0$, we show that $E_{\loc} (\alpha)$ scales about as $|\log \alpha|^{-2/\alpha}$, which had not been predicted earlier.

While we are unable to show that $\mathcal{Z}_{\alpha}$ reduces to a single point for $\alpha \in (c, 1-c)$, our results still improve on several aspects of previous works in this regime. First, the set of energies $E$ not covered by the first two parts of \Cref{c:main1} presumably reduces to the two mobility edges $\{ -E_{\loc} (\alpha), E_{\loc} (\alpha) \}$ (this is evident numerically \cite{tarquini2016level} and equivalent to the prediction that $\mathcal{Z}_{\alpha}$ consists of a single point). In contrast, there is a considerable gap between the localized phase shown in \cite{bordenave2013localization, bordenave2017delocalization} and the delocalized phase shown in \cite{aggarwal2021goe,bordenave2017delocalization}. Next, the second part of \Cref{c:main1} provides an explicit localized phase for all $\alpha \in (0, 1)$, while the previous localization results of \cite{bordenave2013localization, bordenave2017delocalization} apply for $\alpha \in (0, 2/3)$ only; the first part of \Cref{c:main1} also provides an explicit delocalized phase for all $\alpha \in (0, 1)$, slightly improving the range of $\alpha$ in the results of \cite{aggarwal2021goe,bordenave2017delocalization} (which missed a countable set of $\alpha$ parameters). These hold for all matrices satisfying \Cref{matrixh}, which significantly expands the class of matrices considered in \cite{bordenave2013localization, bordenave2017delocalization, aggarwal2021goe}.

\label{Estimates00}

\section{Proof Outline}\label{s:outline}

In \Cref{Tn}, we begin by characterizing the behavior of $R_\star(z)$ in this limit in terms of the fractional moments of resolvent entries. In \Cref{s:bootstrap}, we relate the fractional moments of the resolvent entries to $\lambda(E,\alpha)$, and state certain continuity properties for these quantities. We conclude in \Cref{s:proofmain1} with the proof of \Cref{t:main1}. 
As stated in our discussion of notation in \Cref{s:notation}, we will omit writing the dependence of all constants on $\alpha \in (0,1)$.

	\subsection{Localization and Delocalization Criteria} 
	
	\label{Tn}

	We first characterize the behavior of $\Imaginary R_{\star} (z)$ as $\eta$ tends to zero in terms of the growth of fractional moments of the off-diagonal resolvent entries (from \Cref{Stable}); these moments are given by the below definition. Throughout this section, we fix  $\alpha \in (0,1)$.

	
	\begin{definition} 
		
		\label{moment1} 
	
	 Recalling that $0$ is the root of $\mathbb{T}$, for any integer $L \ge 1$ define
	\begin{flalign}
		\label{sumsz1} 
		\Phi_{L} (s; z) = \mathbb{E} \Bigg[ \displaystyle\sum_{v \in \mathbb{V}  (L)} \big| R_{0v} (z) \big|^s \Bigg]; \qquad \varphi_L (s; z) = L^{-1} \log \Phi_L (s; z).
	\end{flalign}
	
	\noindent where $\mathbb{V}  (L) = \mathbb{Z}_{\geq 1}^L$ is the set of vertices of $\mathbb T$ at distance $L$ from the root $0$. When they exist, we also define the limits  
	\begin{flalign}
		\label{szl} 
		\varphi (s; z) = \displaystyle\lim_{L \rightarrow \infty} \varphi_L (s; z); \qquad \varphi (s; E) = \displaystyle\lim_{\eta \rightarrow 0} \varphi (s; E + \mathrm{i} \eta).
	\end{flalign}
	
	\end{definition}

	 The following theorem, whose statement resembles that of \cite[Theorem 3.2]{aizenman2013resonant}, indicates that the first limit in \eqref{szl} exists and analyzes its properties. We will establish it in \Cref{MultiplicativeR} below.

	 \begin{thm} 
	 	
	 	\label{limitr0j} 
	 	Fix $s \in (\alpha,1)$, $\eps \in (0,1)$, $B\ge 1$, and $z = E + \iu \eta \in \bbH$ such that $\eps \le |E| \le B$ and $\eta \in (0,1)$. 
	 	The following statements hold. 
	 	
	 	\begin{enumerate} 
	 	
	 		\item The limit $\varphi (s; z) = \lim_{L \rightarrow \infty} \varphi_L (s; z)$ exists and is finite. Moreover, there exists a constant $C( s,\eps,B) > 1$ such that
	 		\begin{flalign}
	 			\label{limitr0j2} 
	 			\big| \varphi_L (s; z) - \varphi (s; z) \big| \le \displaystyle\frac{C}{L}.
	 		\end{flalign}
	 		
	 		\item The function $\varphi (s; z)$ is (weakly) convex and nonincreasing in $s \in (\alpha, 1)$. 
	 		
	 		\item Fix $E\in \bbR$. If the limit $\varphi_L (s; E) = \lim_{\eta \rightarrow 0} \varphi_L (s; E + \mathrm{i} \eta)$ exists, then the limit $\varphi (s; E) = \lim_{\eta \rightarrow 0} \varphi (s; E + \mathrm{i} \eta)$ does also. In this case, the limits in $L$ and $\eta$ in the second equality of \eqref{szl} commute, namely, $\varphi (s; E) = \lim_{L\rightarrow\infty} \varphi_L (s; E )$.
	 		
	 		\item There exist constants $c = c(\varepsilon, B)> 0$ and $C = C(\varepsilon, B) > 1$ (independent of $s$) such that $\varphi (s; z) > -C - c \log (s-\alpha)$. In particular, $\lim_{s \rightarrow \alpha} \varphi (s; z) = \infty$. 
				 		
	 		\item There exist constants $C = C ( s) > 1$ and $c = c ( s) > 1$ (independent of $\varepsilon$ and $B$) such that for $|E| \ge 2$ we have $\varphi (s; z) < C - c \log |E|$. In particular, $\lim_{|E| \rightarrow \infty} \varphi (s; z) = -\infty$.
	 		
	 	\end{enumerate} 
	 	
	 \end{thm} 
 
 \begin{rem} 
 	
 	\label{salpha} 

We remark that, unlike in \cite{aizenman2013resonant}, it is necessary to take $s \in (\alpha, 1)$ in \Cref{limitr0j}, as for $s \in (0, \alpha)$ it can be shown that $\varphi_L (s; z)$ diverges (which is suggested by the fourth part of \Cref{limitr0j}). 

\end{rem} 

\begin{definition}
We set
\bex
\varphi (1; z) = \lim_{s \rightarrow 1} \varphi (s; z).
\eex
The existence of the limit follows from the second part of \Cref{limitr0j}.
\end{definition}
 		
 		Now we can state the following theorem, which resembles \cite[Theorem 2.5]{aizenman2013resonant} and essentially states delocalization for $\boldsymbol{T}$ around $E \in \mathbb{R}$ (allowing $E=0$) such that $\varphi (1; E) > 0$. It is proved in \Cref{RLarge}.
	
	\begin{thm} 
	
	\label{rimaginary0} 
	
	Fix $E \in \mathbb{R}$ and a decreasing sequence $\{\eta_j\}_{j=1}^\infty$ such that $\lim_{j\rightarrow\infty} \eta_j =0$. Suppose that
	\be
	\displaystyle\lim_{s \rightarrow 1} \bigg( \displaystyle\liminf_{j \rightarrow \infty} \varphi_s (E + \mathrm{i} \eta_j) \bigg) > 0.
	 \ee 
	 Then there exists a real number $c  > 0$ 
	 such that 
		 \bex
		\liminf_{\eta  \rightarrow 0} \mathbb{P} \big[ \Imaginary R_{\star} (E + \mathrm{i} \eta) > c \big] > c.
		\eex

	\end{thm}

 	In the complementary case, where $\varphi (1; E) < 0$, it can quickly be proven  that $\Imaginary R_{\star} (E + \mathrm{i} \eta)$ converges to $0$ in probability as $\eta$ tends to $0$. The following proposition is proved in \Cref{s:continuitypreliminary}.

	\begin{prop} \label{p:imvanish}
	Fix $E \in \mathbb{R}$ and $s\in(\alpha, 1)$.
	\begin{enumerate}
	 \item
	 Let $\{\eta_j\}_{j=1}^\infty$ be a decreasing sequence such that $\lim_{j\rightarrow\infty} \eta_j =0$, and suppose that \bex\limsup_{j \rightarrow \infty} \varphi (s; E + \mathrm{i} \eta_j) < 0.\eex
	  Then $\lim_{j\rightarrow \infty} \Imaginary R_{\star} (E + \mathrm{i} \eta_j) = 0$ in probability.
	\item
	 Suppose that $\limsup_{\eta \rightarrow 0} \varphi (s; E + \mathrm{i} \eta) < 0$. Then $\lim_{\eta \rightarrow 0} \Imaginary R_{\star} (E + \mathrm{i} \eta) = 0$ in probability.
	 \end{enumerate}
	\end{prop} 	
	
	Next, we state a compactness lemma. It is proved in \Cref{ResolventDensity}.
	
	\begin{lem}\label{l:compactness}
	Fix $E\in \bbR$. Then the closure of the set ${\big\{ R_\star(E + \iu \eta) \big\}_{\eta \in (0,1)}}$ in the weak topology is sequentially compact.
	\end{lem}

Finally, we state a lemma that upgrades the delocalization statement in \Cref{rimaginary0} to the one in \Cref{t:main1}. It is proved in \Cref{ResolventDensity}.
\begin{lem}\label{l:upgrade} Fix $E \in \bbR$, 
 and suppose there exists $c> 0$ such that 
	 \bex
		 \liminf_{\eta \rightarrow 0} \mathbb{P} \big[ \Imaginary R_{\star} (E + \mathrm{i} \eta) > c \big] > c.
		\eex
Then $
 \liminf_{\eta \rightarrow 0}  \Imaginary R_{\star} (E + \mathrm{i} \eta) > 0$ in probability.
\end{lem}

	\subsection{Bootstrap for Large Energies}\label{s:bootstrap}
	
	Our proof of the second part of \Cref{t:main1} will be accomplished through a bootstrap argument. Localization is established for sufficiently large $|E|$ (using the fifth part of \Cref{limitr0j} and \Cref{p:imvanish}), and then inductively shown hold for smaller values of $|E|$, until $E_\loc$ is reached. We now collect some lemmas that facilitate the inductive argument. For technical reasons, we cannot work with $\phi(1,z)$ directly during this process; we instead argue using $\phi(s,z)$ and then take $s$ to $1$ at the end. We first state a continuity lemma for $ \varphi (s; z)$. It is proved in \Cref{s:continuityproof}.

	\begin{lem}\label{l:phicontinuity}
	Fix $\alpha \in (0,1)$; $s \in (\alpha,1)$; 
	$\omega, \kappa >0$; and a compact interval $I \subset \bbR \setminus \{0\}$. There exists a constant $\delta(s,\omega,\kappa, I) >0$ such that the following holds. For any $E_1, E_2 \in I$ such that $|E_1 - E_2 | \le \delta$ and 
\bex
\limsup_{\eta \rightarrow 0} \varphi (s; E_1 + \mathrm{i} \eta) < - \kappa,
\eex
we have	
\bex
\limsup_{\eta \rightarrow 0} \varphi (s; E_2 + \mathrm{i} \eta) 
< 
-\kappa + \omega.
\eex
	

	\end{lem}
	
	We next introduce a quantity $\lambda(E,s,\alpha)$, which we will we see is closely connected to $\phi(s;z)$.

	\begin{definition} \label{lambdaEsalpha}
Fix  $\alpha \in (0,1)$ and $s \in (0,1)$. Recall the quantities $\ell(x)$ and $t_y$ defined in \Cref{lambdaEsalpha} and set
	\begin{flalign}
		 K_{\alpha, s} = \displaystyle\frac{\alpha}{2} \cdot \Gamma \bigg( \displaystyle\frac{s-\alpha}{2} \bigg) \cdot \Gamma \bigg( 1 - \displaystyle\frac{\alpha+s}{2} \bigg).
	\end{flalign}

	\noindent Then, let $\lambda(E, s,  \alpha)$ denote the unique positive solution to the quadratic equation 
	\begin{flalign*} 
		\lambda(E, s,  \alpha)^2 - 2 t_{\alpha} K_{\alpha,s} \Real \ell (E) \cdot \lambda(E, s,  \alpha) + K_{\alpha,s}^2 (t_{\alpha}^2 - t_s^2) \big| \ell (E) \big|^2 = 0,
	\end{flalign*}

	\noindent which admits only one positive solution since its constant term is negative (as $t_{\alpha} < t_s$ for $\alpha < s < 1$).
	
	\end{definition}
	
	The next lemma states basic properties of $\lambda(E,\alpha)$ and $\lambda(E,s,\alpha)$. It is proved in \Cref{s:boundaryvalues}.
	\begin{lem}\label{l:lambdalemma}
	For all $\alpha \in (0,1)$, $s\in (\alpha, 1)$, and $E \in \bbR$, the following claims hold.
	\begin{enumerate}
	\item We have $\lambda(E,s,\alpha) > 0$ and $\lambda(E,\alpha) >0$.
	\item The limit $\lambda(E, \alpha) = \lim_{s \rightarrow 1} \lambda(E,s,\alpha)$ holds.
	\item The functions $(E,s) \mapsto \lambda(E, s, \alpha)$ and $E \mapsto \lambda(E, \alpha)$ 
	are continuous.
	\item We have $\lambda(0, \alpha) > 1$.
	
	\end{enumerate}
	
	\end{lem}
	
	Finally, we relate $\lambda(E,s,\alpha)$ to $\phi(s;z)$. The following lemma is proved in \Cref{s:provebootstrap}.
	
	\begin{lem}\label{l:bootstrap}
Fix $\alpha \in (0,1)$, $s\in (\alpha,1)$, and $E\in \bbR \setminus \{ 0\}$.
Let $\{\eta_j\}_{j=1}^\infty$ be a decreasing sequence such that $\lim_{j\rightarrow\infty} \eta_j =0$. Suppose 
that
$\lim_{j \rightarrow\infty} \Im  R_\star(E + \iu \eta_j) = 0$ in probability.
Then
\bex
\lim_{j \rightarrow \infty} \phi(s; E+ \iu \eta_j) = \log \lambda(E,s,\alpha).
\eex
\end{lem}

\subsection{Conclusion} \label{s:proofmain1}

We now prove \Cref{p:213}, \Cref{t:main1}, and \Cref{c:main1}, assuming the previously stated preliminary results. For clarity, we separate the proofs of the first and second parts of \Cref{t:main1}.

\begin{proof}[Proof of \Cref{p:213}]
The first part is an immediate consequence of the third and fourth parts of \Cref{l:lambdalemma}. For the second part, fix $s_0 = (\alpha + 1)/2$. By the fifth part of \Cref{limitr0j}, there exists $E_0 = E_0(\alpha)> 1 $ such that 
\be\label{pt1}
\phi(s_0; E + \iu \eta) < -1\quad\text{for}\quad |E| \ge E_0,\,\eta \in (0,1).
\ee
Since $\phi(s;z)$ is nonincreasing in $s$ by the second part of \Cref{limitr0j}, the previous line implies that 
\be\label{pt2}
\phi(s; E + \iu \eta) < -1\quad\text{for}\quad |E| \ge E_0,\,\eta \in (0,1)\, ,  s\in (s_0, 1).
\ee
By the second part of \Cref{p:imvanish}, the bound \eqref{pt1}  implies that for any $E\in \bbR$ such that $|E| \ge E_0$, we have 
\be\label{pt3}
\limsup_{\eta \rightarrow 0} \Im R_{00} (E + \iu \eta) = 0
\ee
in probability.
Now fix $E$ such that $|E| \ge E_0$ and any $s \in (s_0 ,1)$. By \Cref{l:bootstrap} and \eqref{pt3},
\be\label{pt4}
\lim_{\eta \rightarrow 0} \varphi (s; E + \iu \eta) = \log \lambda (E, s, \alpha).
\ee
Then by \eqref{pt2} and \eqref{pt4}, it follows for any $s \in (s_0, 1)$ and $|E| \ge E_0$ that 
\be\label{usemelater}
\log \lambda (E, s, \alpha) \le -1.
\ee
Using the second part of \Cref{l:lambdalemma} and \eqref{usemelater}, we have $\lambda(E,\alpha) \le  e^{-1}$ for $|E|  \ge E_0$. This completes the proof of the second part of the proposition.
\end{proof}

\begin{proof}[Proof of \Cref{t:main1}(1)]

We may suppose $E\neq 0$, since in this case the result is immediate from \cite[Lemma 4.3(b)]{bordenave2011spectrum}. Suppose, for the sake of contradiction, that there exists some $E_0\in \bbR$ such that $\lambda(E_0, \alpha) > 1$ and
\begin{flalign*} 
 \liminf_{\eta \rightarrow 0}  \Imaginary R_{\star} (E_0 + \mathrm{i} \eta) = 0
\end{flalign*}
in probability. 
Then there exists a decreasing sequence $\{\eta_j\}_{j=1}^\infty$ such that $\lim_{j\rightarrow\infty} \eta_j =0$ and
\be\label{contradict2}
\lim_{j \rightarrow\infty} \Im  R_\star(E_0 + \iu \eta_j) = 0
\ee
in probability. 
Then using \Cref{l:bootstrap}, we have
\bex
\lim_{j \rightarrow \infty} \phi(s; E_0+ \iu \eta_j) = \log \lambda(E_0,s,\alpha)
\eex
for all $s \in (\alpha, 1)$. By the second and third parts of \Cref{l:lambdalemma}, and the assumption that $\lambda(E_0, \alpha) > 1$, there exists $\nu>0$ such that 
\bex
\lim_{s \rightarrow 1}\left( \lim_{j \rightarrow\infty} \  \phi(s; E_0+ \iu \eta_j)\right) \ge  \nu.
\eex
Then by \Cref{rimaginary0}, the previous inequality implies that there exists $c>0$ such that 
		\bex
		\liminf_{j \rightarrow \infty} \mathbb{P} \big[ \Imaginary R_{\star} (E + \mathrm{i} \eta_j) > c \big] > c.
		\eex
This contradicts the assumption that 
\eqref{contradict2} holds. We conclude that there exists  $c>0$ such that 
\bex
\liminf_{\eta \rightarrow 0} \mathbb{P} \big[ \Imaginary R_{\star} (E + \mathrm{i} \eta) > c \big] > c.
\eex
Using \Cref{l:upgrade}, the previous line implies the first part of \Cref{t:main1}.
\end{proof}

\begin{proof}[Proof of \Cref{t:main1}(2)]
We consider only positive $E$, since negative $E$ can be handled by  symmetric reasoning. For any $s \in (\alpha ,1)$, set
\bex
E_{\mathrm{loc},s } = \inf \{ E \in \bbR_+ : \lambda(E,s,\alpha) < 1 \text{ for all } x \ge E\}.
\eex
First observe that $E_{\loc, s}$ is finite. Indeed, by \eqref{usemelater}, there exists $E_0 = E_0(\alpha) > 1$ 
such that for any $s\in (\alpha ,1)$ and any $E\in \bbR$ such that $|E| \ge E_0$ we have $\lambda (E ,s , \alpha) < 1$, which gives $E_{\mathrm{loc},s } \le E_0$. Further, for $|E| \ge E_0$, the second part of Theorem 2.14 follows from \eqref{pt3}, so it remains to prove the claim for $E < E_0$.

To this end,  fix $s \in (\alpha , 1)$ and $E' \in ( E_{\loc,s}, E_0)$. 
By the definition of $E_{\loc,s}$, we have $\lambda(E', s, \alpha) < 1$.
Set
\bex
\kappa' = \kappa'(E',s) = \inf_{E \in [E', E_0]} \big( 1 - \lambda(E,s, \alpha) \big),
\qquad \kappa = \kappa(E',s) = \frac{\min ( \kappa', 1 )}{2}.
\eex
Since $E \mapsto \lambda(E,s, \alpha)$ is continuous (by \Cref{l:lambdalemma}) and strictly less than $1$ on the compact interval $[E', E_0]$, we have $\kappa' > 0$.

With $I = [E', E_0]$ and $\omega$ equal to $\kappa$, let $\delta$ be the constant given by \Cref{l:phicontinuity}. 
Let $\{E_i\}_{i=1}^M$ be a set of real numbers such that $E_M = E'$ and 
\bex
E_{i+1} < E_i, \qquad | E_{i+1} - E_i | < \delta
\eex
for all $i \in \unn{0}{M-1}$. 
We claim that for every $i\in \unn{0}{M}$, we have 
\be\label{bob}
\limsup_{\eta \rightarrow 0} \exp\big(\phi(s; E_i + \iu \eta)\big) \le  \lambda(E_i, s , \alpha).
\ee
We proceed by induction on $i$. The base case $i=0$, at $E_0$, follows from \eqref{pt4}. Next, assuming the induction hypothesis holds for some $i\in \unn{0}{M-1}$, we will show it holds for $i+1$. Using the induction hypothesis at $i$, and the definition of $\kappa$, we have 
\bex
\limsup_{\eta \rightarrow 0} \exp\big(\phi(s; E_i + \iu \eta)\big)  \le  \lambda(E_i, s , \alpha) \le 1 - 2\kappa.
\eex
By \Cref{l:phicontinuity} and the definition of $\delta$, 
\begin{align}\label{p19}
\limsup_{\eta \rightarrow 0} \varphi (s; E_{i+1} + \mathrm{i} \eta)
&< \limsup_{\eta \rightarrow 0} \varphi (s; E_i + \mathrm{i} \eta) + \kappa\le  1 - \kappa.
\end{align}
Using \Cref{p:imvanish}, the previous equation implies
\be\label{jcvg}
\lim_{\eta \rightarrow 0} \Im R_\star(E_{i+1} + \iu \eta) = 0.
\ee
Let $\{\eta_j\}_{j=1}^\infty$ be a sequence such that
\be\label{otter}
\lim_{j \rightarrow \infty}
\exp \left( \phi(s; E_{i+1} + \iu \eta_j) \right)= \limsup_{\eta\rightarrow 0} \exp \left( \phi(s; E_{i+1} + \iu \eta) \right).
\ee
Then using \eqref{jcvg} and \Cref{l:bootstrap}, equation \eqref{otter} implies that 
\bex
\limsup_{\eta \rightarrow 0} \exp \big( \phi(s; E_{i+1} + \iu \eta) \big) = \lim_{j \rightarrow \infty}
\exp \big( \phi(s; E_{i+1} + \iu \eta_j) \big) = \lambda(E_{i+1}, s, \alpha).   
\eex 
This completes the induction step and shows that \eqref{bob} holds for all $i\in \unn{0}{M}$. In particular, taking $i=M$, we have 
\be
\limsup_{\eta \rightarrow 0} \exp\big(\phi(s; E' + \iu \eta)\big) \le  \lambda(E', s , \alpha).
\ee
Using \eqref{p:imvanish}, this implies $ \lim_{\eta \rightarrow 0} \Im R_\star(E' + \iu \eta) = 0$.

Since $s \in (\alpha ,1)$ and $E'\in (E_{\loc,s}, E_0)$ were arbitrary, we conclude that for every $s\in (\alpha ,1)$, the localization statement 
\be\label{myloc}
\lim_{\eta \rightarrow 0} \Im R_\star(E + \iu \eta) = 0
\ee
holds in probability for every $E$ such that $|E| > E_{\loc,s}$. 

By the second part of \Cref{p:213} and the continuity of $E\mapsto \lambda(E ,\alpha)$ (given by \Cref{l:lambdalemma}), we have
\bex
E_{\mathrm{loc} } = \inf \{ E \in \bbR_+ : \lambda(E,\alpha) < 1 \text{ for all } x \ge E\}.
\eex
Then the continuity of $\lambda(E, s, \alpha)$ and the equation $\lambda(E,\alpha) = \lim_{s\rightarrow 1} \lambda(E,s,\alpha)$ (provided by the second part of \Cref{l:lambdalemma}) together show that $
\lim_{s\rightarrow 1} E_{\mathrm{loc},s } = E_{\mathrm{loc} }.$
Hence, for any $E$ such that $|E| > E_{\mathrm{loc} }$, there exists some $s\in (\alpha,1)$ such that $|E| > E_{\mathrm{loc},s }$, and by what was shown previously the localization claim \eqref{myloc} holds at $E$. This completes the proof of the second part of \Cref{t:main1}.
\end{proof}


\begin{proof}[Proof of \Cref{c:main1}]
The first part of the corollary follows from the first part of \Cref{t:main1} and the second part of \Cref{l:loccriteria}. The second part of the corollary follows form the second part of \Cref{t:main1} and the first part of \Cref{l:loccriteria}. 
The final part of the corollary follows from \Cref{t:main2}, the first two parts of the corollary, \Cref{p:213}, and the continuity of $\lambda(E,\alpha)$ in $E$ provided by the third part of \Cref{l:lambdalemma}.
\end{proof} 
	
	\section{Miscellaneous Preliminaries} 
	
	\label{Estimates0} 
	
	In this section we collect various miscellaneous results that will be used in what follows. We begin in \Cref{s:fouriertransform} with our conventions for the Fourier transform. In \Cref{ProofSum} we state properties of stable laws and Poisson point processes, and then provide properties of random trees in \Cref{EstimateTreeVertex}. In \Cref{EquationsResolvent} we recall various resolvent identities, and in \Cref{ResolventDensity} and \Cref{ProofSigma} we establish properties of the diagonal resolvent entry $R_{\star} (z)$ from \Cref{r}. Finally, in \Cref{s:rebdd} we show that any real boundary value of $R_\star(z)$ must be equal to $R_\loc(E)$.
	
		\subsection{Fourier Transform}\label{s:fouriertransform}
Our convention for the Fourier transform of a function $f\in L^1(\bbR)$ is 
\be\label{fourier}
\hat f(w) = \int_{-\infty}^\infty \exp(\iu wx) f(x) \, dx.
\ee
This definition is chosen to match the definition of the characteristic function of a random variable $X$ as $\E \left[ \exp( \iu t X)  \right]$. The inversion formula then reads
\be
f(x) =\frac{1}{2\pi} \int_{-\infty}^\infty
\exp(-\iu w x ) \hat f(w)\, dw,
\ee
which holds whenever $f\in L^1(\bbR)$ and $\hat f\in L^1(\bbR)$.
For notational convenience, we sometimes write $\mathcal F [f](w)$ in place of $\hat f(w)$, and $\mathcal F^{-1} [f](x)$ for the inverse transform $(2\pi)^{-1}\int_{\bbR} \exp(-iw x) f(w) \, dw$.

Recall that the convolution of two functions was defined in \Cref{s:notation}. For $f,g \in L^2(\bbR)$, we have under our convention the formulas
\bex
\mathcal F [f \ast g] = \mathcal F [f] \cdot \mathcal F [g], \qquad  \mathcal F [fg] = \frac{1}{2 \pi} \big( \mathcal F[f] \ast \mathcal F[g] \big).
\eex

	\subsection{Properties of Stable Laws and Poisson Point Processes}
	
	\label{ProofSum}
	
	In this section we provide several facts about stable laws and Poisson point processes. We begin with the following lemma stating that a stable law can be represented as a random linear combination of entries sampled from a Poisson point process with the heavy-tailed intensity $\frac{\alpha}{2} x^{-\alpha / 2 - 1} dx$.
	
	\begin{lem} 
		
		\label{sumaxi}
		
		Let $\boldsymbol{a} = (a_j)_{j \ge 1}$ denote a family of mutually independent, identically distributed, real random variables with law $a$. For each $j \ge 1$, set $b_j = \max \{ a_j, 0 \}$ and $c_j = \min \{ a_j, 0 \}$. Set
		\begin{flalign} 
			\label{bc} 
			B =  \mathbb{E} \big[ (b_i)^{\alpha / 2} \big]; \qquad C = \mathbb{E} \big[  (-c_i)^{\alpha / 2} \big].
		\end{flalign}
		
		\noindent If $\boldsymbol{\zeta} = \{ \zeta_j \}_{j \ge 1}$ is a Poisson point process with intensity measure $\frac{\alpha}{2} x^{-\alpha / 2 - 1} dx$, then the random variable $Z = \sum_{j = 1}^{\infty} a_j \zeta_j$ is an $\big( \frac{\alpha}{2}; \sigma; \beta)$-stable law, where 
		\begin{flalign*} 
			\sigma = (B + C)^{2 / \alpha} = \mathbb{E} \big[ |a|^{\alpha / 2} \big]^{2 / \alpha}; \qquad  \beta = \displaystyle\frac{B - C}{B + C}.
		\end{flalign*} 
		
	\end{lem}
	
	\begin{proof}
		
		Let $\mathcal{J} = \big\{ j \in \mathbb{Z}_{\ge 1} : a_j = b_j \}$. Then $Z = X - Y$, where $X = \sum_{j \in \mathcal{J}} b_j \zeta_j$ and $Y = -\sum_{k \notin \mathcal{J}} c_k \zeta_k$. Observe that $X$ and $Y$ are independent nonnegative random variables, as $(\zeta_j)_{j \in \mathcal{J}}$ is independent from $(\zeta_{j'})_{j' \notin \mathcal{J}}$ from the Poisson thinning property. Moreover, by the L\'{e}vy--Khintchine representation, we have for any $t \in \mathbb{R}_{\ge 0}$ that  
		\begin{flalign*}
			 \mathbb{E} [e^{-tX}] = \mathbb{E} \Bigg[ \exp \bigg( -t \displaystyle\sum_{j \in \mathcal{J}} a_j \zeta_j \bigg) \Bigg] & = \exp \Bigg( \frac{\alpha}{2} \displaystyle\int_0^{\infty} \mathbb{E} [e^{-xta} - 1]  \cdot  x^{-\alpha/2-1} dx \Bigg) \\ 
			& = \exp \bigg( - \Gamma \Big(1 - \displaystyle\frac{\alpha}{2} \Big) \cdot B |t|^{\alpha / 2} \bigg), 
		\end{flalign*} 
	
		\noindent and similarly
		\begin{flalign*}
			& \mathbb{E} [e^{-tY}] = \mathbb{E} \Bigg[ \exp \bigg( -t \displaystyle\sum_{j \notin \mathcal{J}} a_j \zeta_j \bigg) \Bigg] = \exp \bigg( - \Gamma \Big(1 - \displaystyle\frac{\alpha}{2} \Big) \cdot C |t|^{\alpha / 2} \bigg).
		\end{flalign*}
		
		\noindent  Thus, $X$ and $Y$ are independent, nonnegative $\frac{\alpha}{2}$-stable laws with scaling parameters $B^{2 / \alpha}$ and $C^{2 / \alpha}$, respectively. Since $Z = X - Y$, we deduce the lemma from \Cref{sumalpha}.
	\end{proof} 

	The next lemma bounds the denstiy of a random linear combination of entries sampled from the restriction of the Poisson point process with intensity $\frac{\alpha}{2} x^{-\alpha / 2 - 1} dx$ to the complement of some interval $[u, v]$.

	\begin{lem} 
		
	\label{zuv} 

	For any positive real numbers $u < v$ and $\kappa \in (0, 1/4)$, there exists a constant $c = c(u, v, \kappa) \in (0, 1)$ such that the following holds. Let $\boldsymbol{a} = (a_1, a_2, \ldots )$ denote a family of mutually independent, identically distributed, real random variables with law $a$. Defining $B$ and $C$ as in \eqref{bc}, assume that $B, C \in (\kappa, \kappa^{-1})$. Let $\boldsymbol{\zeta} = (\zeta_1, \zeta_2, \ldots )$ denote a Poisson point process on $\mathbb{R}_{> 0}$ with intensity proportional to $\frac{\alpha}{2} x^{-\alpha / 2 - 1} \cdot \one_{x \notin [u, v]} \cdot dx$, and denote $Z = \sum_{j = 1}^{\infty} a_j \zeta_j$. The following two statements hold for any interval $J \subset \mathbb{R}$.
	
	\begin{enumerate} 
		
		\item We have 
		\begin{flalign}
			\label{estimatezj1} 
		\mathbb{P} [Z \in J] \le c^{-1}  \displaystyle\int_J \displaystyle\frac{dx}{\big( |x| + 1 \big)^{\alpha/2 + 1}}.
		\end{flalign}
	
		\item Further assume for any interval $I \subset \mathbb{R}$ that 
		\begin{flalign}
			\label{ainterval} 
		\mathbb{P} [a \in I] \ge \kappa \displaystyle\int_I \displaystyle\frac{dx}{ \big( |x| + 1 \big)^{\alpha/2 +1}}.
		\end{flalign}

	\noindent Then, we have 
	\begin{flalign}
		\label{estimatezj2} 
		 \mathbb{P} [Z \in J] \ge c  \displaystyle\int_J \displaystyle\frac{dx}{\big( |x| + 1 \big)^{\alpha/2 + 1}}.
	\end{flalign}

	\end{enumerate} 

	\end{lem} 

	\begin{proof} 
				Set $J = [x_0, y_0]$. Let $\boldsymbol{\zeta}' = (\zeta_1', \zeta_2', \ldots )$ denote a Poisson point process on $\mathbb{R}_{> 0}$ with intensity $ \frac{\alpha}{2} x^{-\alpha / 2 - 1} dx$, and set $Z' = \sum_{j = 1}^{\infty} a_j \zeta_j'$. Then, there exists a constant $c_1 = c_1 (u, v) > 0$ such that none of the $\zeta_i'$ are in $[u, v]$ with probability at least $c_1$. In particular, there exists a coupling of $\boldsymbol{\zeta}$ and $\boldsymbol{\zeta}'$ such that $\mathbb{P} [\boldsymbol{\zeta} = \boldsymbol{\zeta}'] \ge c_1$. Since the law of $Z$ is that of $Z'$, conditional on the event $\{ \boldsymbol{\zeta} = \boldsymbol{\zeta}' \}$, this together with the fact that \Cref{sumaxi} implies the existence of a constant $c_2 = c_2 (\kappa) \in (0, 1)$ such that 
		\begin{flalign*}
			\mathbb{P} [Z' \in J] \le c_2^{-1} \displaystyle\int_J \displaystyle\frac{dx}{\big( |x| + 1 \big)^{\alpha/2 + 1}},
		\end{flalign*}
		
		\noindent yields \eqref{estimatezj1}.
		
		To establish \eqref{estimatezj2}, observe there exists a constant $c_3 = c_3 (u) > 0$ such that 
		\begin{flalign*} 
			\mathbb{P} [\mathscr{E}] \ge c_3, \qquad \text{where} \qquad \mathscr{E} = \bigg\{ \displaystyle\frac{u}{2} \le \zeta_1 \le u \bigg\} \cap \bigg\{ \zeta_2 < \displaystyle\frac{u}{2} \bigg\}.
		\end{flalign*}  
	
		\noindent Denoting $Z_0 = \sum_{j = 2}^{\infty} a_j \zeta_j$, we have that $a_1 \zeta_1$ and $Z_0$ are independent conditional on the event $\mathscr{E}$. Moreover, on $\mathscr{E}$, the random variable $\zeta_1$ has density proportional to $ x^{-\alpha / 2 - 1} \cdot \one_{x \in [u/2, u]} \cdot dx$. Now, by the upper bound \eqref{estimatezj1}, there exists a constant $C_1 = C_1 (u, v, \kappa) > 1$ such that $\mathbb{P} \big[ |Z_0| \le C_1 \big| \mathscr{E} \big] \ge \frac{1}{2}$. This, together with the law of $\zeta_1$ and the independence of $(\zeta_1, Z_0)$, implies
		\begin{flalign*}
			\mathbb{P} [Z \in J] & \ge \mathbb{P} [\mathscr{E}] \cdot \mathbb{P} \big[ a_1 \zeta_1 + Z_0 \in [x_0, y_0] \big| \mathscr{E} \big] \\
			& \ge c_3 \cdot \mathbb{P} \big[ a_1 \zeta_1 \in [x_0 - Z_0, y_0 - Z_0] \big| \mathscr{E} \big] \\
			& \ge \displaystyle\frac{c_3}{4} \cdot \mathbb{P} \Big[ a_1 \zeta_1 \in [x_0 - Z_0, y_0 - Z_0] \Big| \mathscr{E} \cap \big\{ |Z_0| \le C_1 \big\} \cap \big\{ |a_1| \ge \kappa \big\} \Big] \ge c_4 \displaystyle\int_J\displaystyle\frac{dx}{\big( |x| + 1 \big)^{\alpha/2 + 1}}
		\end{flalign*}
	
		\noindent for some constant $c_4 = c_4 (u, v, \kappa) > 0$; here, to deduce the last inequality, we applied \eqref{ainterval}. This confirms the lower bound \eqref{estimatezj2}.
	\end{proof}

	We next state a version of Campbell's theorem for Poisson point processes. To that end, let $\mathcal{M} (\mathbb{R}_{> 0})$ denote the space of locally finite measures on $\mathbb{R}_{> 0}$, and (following \cite{daley2008introduction}) let $\mathcal N_{\mathbb{R}_{> 0}}^{\#*}$ denote the set of all simple counting measures; these are boundedly finite integer-value measures $N$ on $\mathbb{R}_{> 0}$ such that $N(\{x\}) =0$ or $1$ for each $x \in \mathbb{R}_{> 0}$. The latter set is endowed with the $\sigma$-algebra $\mathcal B (\mathcal N_{\mathbb{R}_{> 0}}^{\#*})$, the smallest $\sigma$-algebra such that for $A \in \mathcal B(\mathbb{R}_{> 0})$, the map $\mathcal N_{\mathbb{R}_{> 0}}^{\#*} \rightarrow \bbR$ given by $N \mapsto N(A)$ is measurable. Here $\mathcal B(\mathbb{R}_{> 0})$ is the Borel $\sigma$-algebra on $\mathbb{R}_{> 0}$
	
	For any locally finite sequence of points $\Xi = (\xi_1, \xi_2, \ldots ) \subset \mathbb{R}_{> 0}$, we frequently also use $\Xi$ to denote the measure $\sum_{\xi \in \Xi} \delta_{\xi} \in \mathcal N_{\mathbb{R}_{> 0}}^{\#*}$. For any $x \in \mathbb{R}$, we further set $\Xi - \delta_x \in \mathcal N_{\mathbb{R}_{> 0}}^{\#*}$ equal to $\sum_{\xi \in \Xi \setminus x} \delta_{\xi}$ if $x \in \Xi$ and $\sum_{\xi \in \Xi} \delta_{\xi}$ if $x \notin \Xi$.

	\begin{lem}[{\cite[Lemma 13.1.II, Definition 13.1.I (b), and Exercise 13.1.1 (a)]{daley2008introduction}}] 
		
	\label{fidentityxi} 
	
	Let $\Xi$ denote a Poisson point process on $\mathbb{R}_{> 0}$ with intensity measure $\mu \in \mathcal{M} (\mathbb{R}_{> 0})$. For any nonnegative, measurable function $f : \mathbb{R}_{> 0} \times \mathcal N_{\mathbb{R}_{> 0}}^{\#*} \rightarrow \mathbb{R}_{\ge 0}$, we have  
	\begin{flalign*}
		\mathbb{E} \Bigg[ \displaystyle\sum_{j = 1}^{\infty} f(\xi_j, \Xi \setminus \xi_j) \Bigg] =  \displaystyle\int_0^{\infty} \mathbb{E} \big[ f(x, \Xi - \delta_x) \big] d \mu (x).
	\end{flalign*}

	\noindent In particular, if $\mu(dx) =  c x^{-\alpha-1} \cdot \one_{x \in I} \cdot dx$ for some constant $c > 0$ and interval $I \subset \mathbb{R}$, then
	\begin{flalign}
		\label{fxi} 
		\mathbb{E} \Bigg[ \displaystyle\sum_{j = 1}^{\infty} f(\xi_j, \Xi \setminus \xi_j) \Bigg] = c \displaystyle\int_I \mathbb{E} \big[ f(x, \Xi - \delta_x) \big] \cdot x^{-\alpha-1} dx.
	\end{flalign}

	\end{lem} 

\begin{rem}
We will often apply \Cref{fidentityxi} for functions $f$ containing an infinite sum built from the random variables $\xi_j$. We define such $f$ to be $0$ on the event where the sum does not converge (observing the fact that this event is measurable).
\end{rem}

	\subsection{Properties of Random Trees}
	
	\label{EstimateTreeVertex} 
	
	In this section we provide properties of various random trees. We begin by recalling the notion of unimodularity introduced in \cite{benjamini2001recurrence,aldous2004objective,aldous2007processes}. 
	
	 We consider edge-weighted graphs; these are graphs together a map from its set of edges to $\bbR$.  A (possibly infinite) graph $G$ is \emph{locally finite} if, for any compact $K \subset \bbR$, each of its vertices has a finite number of adjacent edges with weight in $K$; we let $V = V(G)$ denote the vertex set of $G$. Let $\mathcal{G}_*$ denote the set of isomorphism classes $(G, u_0)$ of \emph{rooted weighted graphs}, namely of a (connected) locally finite weighted graph $G$ with a distinguished vertex $u_0 \in V(G)$. Further let $\mathcal{G}_{**}$ denote the set of isomorphism classes $(G, u, v)$ of \emph{doubly rooted graphs}, that is, of a (connected) locally finite weighted graph $G$ with an ordered pair of distinguished vertices $(u, v)$ of $G$. The spaces $\mathcal{G}_*$ and $\mathcal{G}_{**}$ both admit a complete, separable metric; we refer to \cite[Section 2]{benjamini2015unimodular} for definitions.

	\begin{definition} 
		
		\label{definitionunimodular} 
		
		A measure $\mu$ on $\mathcal{G}_{*}$ is \emph{unimodular} if, for any Borel measurable function $f : \mathcal{G}_{**} \rightarrow \mathbb{R}_{\ge 0} \cup \{ \infty \}$, we have
		\begin{flalign}
			\label{sumfuv} 
			\mathbb{E}_{\mu (G, u_0)} \Bigg[ \displaystyle\sum_{v \in V(G)} f(G, u_0, v) \Bigg] = \mathbb{E}_{\mu (G, u_0)} \Bigg[ \displaystyle\sum_{v \in V(G)} f(G, v, u_0) \Bigg].
		\end{flalign}
	
		\noindent The equality \eqref{sumfuv} is sometimes referred to as the \emph{Mass Transport Principle}.
	\end{definition} 

	\begin{rem} 
		
		\label{tunimodular} 
		
		As observed in \cite[Section 3.2]{benjamini2001recurrence}, the local weak limit of any measure on finite trees is unimodular. Thus, the random tree $\mathbb{T}$ is unimodular.
		
	\end{rem}

	Next, we provide an estimate for the number of vertices on the $k$-th level of a certain random tree (namely, a Galton--Watson tree). Given a real number $\lambda > 0$, we recall that the \emph{Galton--Watson tree} with parameter $\lambda$ is a random rooted tree $\mathcal{T}$ with vertex set $\mathcal{V} = \bigcup_{k = 0}^{\infty} \mathcal{V} (k)$, where the $\mathcal{V} (k)$ are described inductively as follows. The set $\mathcal{Z} (0)$ consists of a single vertex $\{ 0 \}$, given by the root of $\mathcal{T}$. For each integer $k \ge 0$ and vertex $v \in \mathcal{V} (k)$, let $n_v \ge 0$ denote a Poisson random variable with parameter $\lambda$. Then, $v$ has $n_v$ children $\mathcal{D} (v) = \big\{ (v, 1), (v, 2), \ldots , (v, n_v) \big\}$, and we set $\mathcal{V} (k+1) = \bigcup_{v \in \mathcal{V} (k)} \mathcal{D} (v)$. We refer to $\mathcal{V} (k)$ as the \emph{$k$-th level} of $\mathcal{T}$. 
	
	The below lemma bounds the number of vertices in the $k$-th level of a Galton--Watson tree.
	
	\begin{lem} 
		
		\label{treenumber} 
		
		  We have $\mathbb{E} \big[ \big| \mathcal{V} (k) \big| \big] = \lambda^k$. Moreover, if $\lambda \ge 2$ then for any real number $B \ge 1$, we have 
		\begin{flalign}
			\label{vnk} 
			  \mathbb{P} \Big[ \big| \mathcal{V} (k) \big| \ge B \cdot (2\lambda)^k \Big] \le 3 e^{- B/2}.
		\end{flalign}
		
	\end{lem} 
	
	\begin{proof}
		
		Observe for any integer $\ell \ge 1$ that  
		\begin{flalign}
			\label{vum}
			\big| \mathcal{V} (k) \big| = \displaystyle\sum_{w \in \mathcal{V} (k - 1)} n_w,
		\end{flalign}
		
		\noindent where $n_w$ denotes the number of children of any vertex $w \in \mathbb{V}$. Since $n_w$ is a Poisson random variable with mean $\lambda$, this together with \eqref{vum} implies that $\mathbb{E} \big[ \big| \mathcal{V} (k) \big| \big] = \lambda \cdot \mathbb{E} \big[ \big| \mathcal{V} (k-1) \big| \big]$. Hence, $\mathbb{E} \big[ \big| \mathcal{V} (k) \big| \big] = \lambda^k$ by induction on $k$.
		
		To show \eqref{vnk}, define $f\colon \mathbb{R}_{> 0} \rightarrow \mathbb{R}$ by setting $f(x) = f_{\lambda} (x) = e^{\lambda (x - 1)}$ for each $x > 0$. Then, since $n_w$ is a Poisson random variable with mean $\lambda$, we have $\mathbb{E} [s^{n_w}] = e^{\lambda(s-1)} = f(s)$, for each $s > 0$. Denoting $g_{\ell} (s) = \mathbb{E} [s^{|\mathcal{V} (\ell)|}]$, \eqref{vum} then implies for each integer $\ell \ge 1$ and real number $s > 0$ that
		\begin{flalign}
			\label{vn1} 
			g_{\ell} (s) = \mathbb{E} \Bigg[ \displaystyle\prod_{w \in \mathcal{V} (\ell - 1)} s^{n_w} \Bigg] = \mathbb{E} \Bigg[ \displaystyle\prod_{w \in \mathcal{V} (\ell - 1)} \mathbb{E} [s^{n_w}] \Bigg] = \mathbb{E} \big[  (e^{\lambda (s - 1)})^{|\mathcal{V} (\ell - 1)|} \big] =  g_{\ell - 1} \big( f(s) \big),
		\end{flalign}
		
		\noindent and so $g_{\ell} (s) = f^{(\ell)} (s)$, where $f^{(\ell)}$ denotes the $\ell$-fold composition of $f$.
		
		Since $f(x + 1) - 1 \le \lambda x + (\lambda x)^2$ for any $0 \le x \le \lambda^{-1}$, we have by induction on $\ell$ that $f_{\ell} (1 + s) - 1 \le \lambda^{\ell} s + 2^{\ell} (\lambda^{\ell} s)^2$, if $0 \le s \le (2^{\ell + 1} \lambda^{\ell})^{-1}$, as then it is quickly verified that 
		\begin{flalign*} 
			\lambda \cdot (\lambda^{\ell} s + 2^{\ell} \lambda^{2 \ell} s^2) + \lambda^2 \cdot (\lambda^{\ell} s + 2^{\ell} \lambda^{2 \ell} s^2)^2 \le \lambda^{\ell + 1} s + 2^{\ell + 1} (\lambda^{\ell + 1} s)^2.
		\end{flalign*} 
		
		\noindent Applying this for $\ell = k - 1$ and $s = 2^{-k} \lambda^{-k}$ yields
		\begin{flalign*}
			\mathbb{E} \big[ ( 1 + 2^{-k} \lambda^{-k})^{|\mathcal{V} (k)|} \big] \le f(1 + \lambda^{k-1} s + 2^{k-1} \lambda^{2k-2} s^2) \le f \bigg( 1 + \displaystyle\frac{3}{4 \lambda} \bigg) < 3.
		\end{flalign*}
		
		\noindent Hence, in \eqref{vnk} follows from a Markov estimate.
	\end{proof}
	
	Finally, we introduce the following notation for the tree $\mathbb{T}$ (defined in \Cref{Stable}). The \emph{length} of any vertex $v \in \mathbb{Z}_{\ge 1}^k \subset \mathbb{V}$ is $\ell (v) = k$, and we let $\mathbb{V} (\ell) = \mathbb{Z}_{\ge 1}^\ell$ denote the set of vertices of length $\ell$. We write $v \sim w$ if there is an edge between $v, w \in \mathbb{V}$. The \emph{parent} of any vertex $v \in \mathbb{V}$ that is not the root\footnote{If $v = 0$, then its parent $0_-$ is defined to be the empty set.} (namely, $v \ne 0$) is the unique vertex $v_- \in \mathbb{V}$ such that $v_- \sim v$ and $\ell (v_-) = \ell (v) - 1$. A child of $v$ is any vertex $w \in \mathbb{V}$ whose parent is $v$. We let $\mathbb{D} (v) \subset \mathbb{V}$ denote the set of children of $v$. 
		
	A \emph{path} on $\mathbb{T}$ is a sequence of vertices $\mathfrak{p} = (v_0, v_1, \ldots , v_k)$ such that $v_i$ is the parent of $v_{i + 1}$ for each $i \in [0, k - 1]$. The \emph{length} of this path is $\ell (\mathfrak{p}) = k$, and its \emph{boundary vertices} are $v_0$ and $v_k$, with $v_0$ and $v_k$ being the \emph{starting} and \emph{ending} vertices of $\mathfrak{p}$, respectively. If $\mathfrak{p} \subset \mathbb{V}$ is a path in $\mathbb{T}$, and $v \in \mathfrak{p}$ is not the ending vertex, then let $v_+  = v_+ (\mathfrak{p}) \in \mathfrak{p}$ denote the child of $v$ in $\mathfrak{p}$. We write $v \preceq w$ (equivalently, $w \succeq v$) if there exists a path (possibly of length $0$) with starting and ending vertices $v$ and $w$, respectively; we further write $v \prec w$ (equivalently, $w \succ v$) if $v \preceq w$ and $v \ne w$. For any integer $k \ge 0$ and vertex $v \in \mathbb{V}$ with $\ell (v) = \ell$, we additionally let $\mathbb{D}_k (v) = \big\{ u \in \mathbb{V} (\ell + k) : v \preceq u \big\}$. Observe in particular that $\mathbb{D}_1 (v) = \mathbb{D} (v)$ and $\mathbb D_\ell(0) = \mathbb{V}(\ell)$.
	
	We will sometimes identify $\mathbb{T}$ with the edge-weighted graph whose weights are given by the operator $\T$, as in \eqref{t}.
	For any vertex $w \in \mathbb{V}$, let $\mathbb{T}_- (w)$ denote the set of all edge weights and vertices in connected component containing the root $0$ of $\mathbb{T} \setminus \{ w \}$. That is, it constitutes the subgraph of $\mathbb{T}$ that lies ``above'' $w$, including the weight on the edge $(w_-, w)$. 
	Let $\mathbb{P}_{\mathbb{T}_- (w)}$ and $\mathbb{E}_{\mathbb{T}_- (w)}$ denote the probability measure and expectation conditional on $\mathbb{T}_- (w)$, respectively. 
	We will often abbreviate $\mathbb{P}_w = \mathbb{P}_{\mathbb{T}_- (w)}$ and $\mathbb{E}_w = \mathbb{E}_{\mathbb{T}_- (w)}$. 

	\subsection{Identities for Resolvent Entries} 
	
	\label{EquationsResolvent}
	
	In this section we collect several known identities for entries of the resolvents of adjacency operators on trees. We begin with some notation; in what follows, we recall the notation from \Cref{Stable}.
	
	For any subset of vertices $\mathcal{U} \subseteq \mathbb{V}$, we let $\boldsymbol{T}^{(\mathcal{U})}$ denote the operator obtained from $\boldsymbol{T}$ by setting all entries associated with at least one vertex in $\mathcal{U}$ to $0$. More specifically, we let $\boldsymbol{T}^{(\mathcal{U})} : L^2 (\mathbb{V}) \rightarrow L^2 (\mathbb{V})$ denote the self-adjoint operator defined by first setting, for any vertices $v, w \in \mathbb{V}$,
		\begin{flalign*}
			T_{vw}^{(\mathcal{U})} = \big\langle \delta_v, \boldsymbol{T}^{(\mathcal{U})} \delta_w \big\rangle = T_{vw}, \quad \text{if $v, w \notin \mathcal{U}$}; \qquad  T_{vw}^{(\mathcal{U})} = \big\langle \delta_v, \boldsymbol{T}^{(\mathcal{U})} \delta_w \big\rangle = 0, \qquad \text{otherwise};
		\end{flalign*}
	
		\noindent then extending to $\mathcal{F}$ (finitely supported vectors of $L^2 (\mathbb V)$) by linearity; and next by extending to $L^2 (\mathbb{V})$ by density. For any complex number $z \in \mathbb{H}$, we further denote the associated resolvent operator $\boldsymbol{R}^{(\mathcal{U})} : L^2 (\mathbb{V}) \rightarrow L^2 (\mathbb{V})$ and its entries by 
		\begin{flalign*} 
			\boldsymbol{R}^{(\mathcal{U})} = \boldsymbol{R}^{(\mathcal{U})} (z) = \big( \boldsymbol{T}^{(\mathcal{U})} - z \big)^{-1}; \qquad R_{vw}^{(\mathcal{U})} = R_{vw}^{(\mathcal{U})} (z) = \big\langle \delta_v, \boldsymbol{R}^{(\mathcal{U})} \delta_w \big\rangle, \qquad \text{for any $v, w \in \mathbb{V}$}.
		\end{flalign*}

		The following lemma provides identities and estimates on the entries of $\boldsymbol{R}^{(\mathcal{U})}$. The first statement \eqref{qvv} below is given by \cite[Proposition 2.1]{klein1998extended}; the second follows from the first; and the third is given by \cite[Equation (3.37)]{aizenman2013resonant}. Alternatively, in the context of finite-dimensional operators, the first, second, and third statements below are given by \cite[Equation (8.8)]{erdos2017dynamical}, \cite[Equation (8.34)]{erdos2017dynamical}, and \cite[Equation (8.3)]{erdos2017dynamical}, respectively.

		\begin{lem}[{\cite{klein1998extended,aizenman2013resonant}}]
		
		\label{q12} 
		
		Fix $z = E + \mathrm{i} \eta \in \mathbb{H}$ and $\mathcal{U} \subset \mathbb{V}$.
		
		\begin{enumerate}
			\item For any vertex $v \in \mathbb{V}$, we have the \emph{Schur complement identity}
			\begin{flalign}
				\label{qvv}
				R_{vv}^{(\mathcal{U})} = R_{vv}^{(\mathcal{U})} (z) = - \Bigg( z + \displaystyle\sum_{w \sim v} \big( T_{vw}^{(\mathcal{U})} \big)^2 \cdot R_{ww}^{(\mathcal{U}, v)} \Bigg)^{-1}.
			\end{flalign}
				
			\noindent Moreover, $R_{ww}^{(v)} (z)$ has the same law as $R_{00} (z)$, for any $w \in \mathbb{D}(v)$.
			
			\item For any vertices $v, w \in \mathbb{V}$, we deterministically have 
			\begin{flalign*} 
				\big| R_{vw}^{(\mathcal{U})} (z) \big| \le \eta^{-1}.
			\end{flalign*} 
		
			\item For any vertex $v \in \mathbb{V}$, we have the \emph{Ward identity} 
			\begin{flalign}
				\label{sumrvweta} 
				\displaystyle\sum_{w \in \bbV \setminus \mathcal{U}} \big| R_{vw}^{(\mathcal{U})} \big|^2 = \eta^{-1} \cdot \Imaginary R_{vv}^{(\mathcal{U})}.
			\end{flalign}
		
		\end{enumerate}
		
		\end{lem}

		A quick consequence of the Schur complement identity (and \Cref{tuvalpha}) is the following result from \cite{bordenave2011spectrum}, which states that the entries of $\boldsymbol{R}$ satisfy a recursion in law, sometimes referred to as a \emph{recursive distributional equation}.

		\begin{prop}[{\cite[Theorem 4.1]{bordenave2011spectrum}}]
			
			\label{rdistribution} 
			
			Fix $z \in \mathbb{H}$, and let $R_1, R_2, \ldots $ denote mutually independent random variables with law $R_{\star} (z)$. Further let $(\xi_1, \xi_2, \ldots )$ denote a Poisson point process on $\mathbb{R}_{> 0}$ with intensity measure $\frac{\alpha}{2} x^{-\alpha / 2 - 1} dx$. Then, $R_{\star} (z)$ has the same law as 
			\begin{flalign*}
				-\Bigg( z + \displaystyle\sum_{j = 1}^{\infty} \xi_j R_j \Bigg)^{-1}.
			\end{flalign*}
			
		\end{prop}

	Next we have the following lemma, which expresses $R_{uv}$ as a product over diagonal resolvent entries. It essentially appears as \cite[Equation (3.4)]{aizenman2013resonant} (where there it is assumed that all nonzero entries of the adjacency matrix $\boldsymbol{T}$ are equal to $1$).
	
	\begin{lem}[{\cite[Equation (3.4)]{aizenman2013resonant}}]
		
		\label{rproduct} 
		
		Fix $v, w \in \mathbb{V}$ with $v \preceq w$ that are connected by a path $\mathfrak{p}$ (with $v$ as the starting vertex and $w$ as the ending vertex) consisting of $m+1$ vertices. For any subset $\mathcal{U} \subset \mathbb{V}$, we have
		\begin{flalign*}
			R_{vw}^{(\mathcal{U})} = (-1)^m \cdot R_{vv}^{(\mathcal{U})} \cdot \displaystyle\prod_{v \prec u \preceq w} T_{u u_-} R_{uu}^{(\mathcal{U}, u_-)} = (-1)^m \cdot R_{ww}^{(\mathcal{U})} \cdot \displaystyle\prod_{v \preceq u \prec w} T_{u u_+} R_{uu}^{(\mathcal{U}, u_+)}.
		\end{flalign*}
	
	\end{lem}

	The following corollary of \Cref{q12} and \Cref{rproduct}, which lower bounds $\Imaginary R_{00}$ in terms of other entries of $\boldsymbol{R}$, appears in \cite{aizenman2013resonant}. We include its quick proof.

	\begin{cor}[{\cite[Equation (4.3)]{aizenman2013resonant}}] 
		
		\label{rsum} 
		
		For any integer $k \ge 0$, we have
		\begin{flalign}
			\label{r00sum} 
			\Imaginary R_{00} \ge \displaystyle\sum_{v \in \mathbb{V} (k)} |R_{0v}|^2 \displaystyle\sum_{u \in \mathbb{D} (v)} |T_{vu}|^2 \Imaginary R_{uu}^{(v)}.
		\end{flalign}

	\end{cor} 
	
	\begin{proof} 
		
		By \eqref{qvv}, we have
		\begin{flalign}
			\label{r00imaginary}
			\begin{aligned} 
			\Imaginary R_{00} & = \Bigg( \eta + \displaystyle\sum_{v \sim 0} |T_{0v}|^2 \Imaginary R_{vv}^{(0)} \Bigg) \cdot \Bigg| z + \displaystyle\sum_{v \sim 0} |T_{0v}|^2 R_{vv}^{(0)} \Bigg|^{-2} \ge \Bigg( \displaystyle\sum_{v \sim 0} |T_{0v}|^2 \Imaginary R_{vv}^{(0)} \Bigg) \cdot |R_{00}|^2.
			\end{aligned} 
		\end{flalign}
		
		\noindent  Applying \eqref{r00imaginary} $k+1$ times, we find 
		\begin{flalign*}
			\Imaginary R_{00} \ge \displaystyle\sum_{u \in \mathbb{V} (k+1)} \Bigg( |T_{u_- u}|^2 \Imaginary R_{uu}^{(u_-)} \cdot \displaystyle\prod_{0 \preceq w \prec u} |T_{ww_-}|^2 \big| R_{ww}^{(w_-)} \big|^2 \Bigg),
		\end{flalign*} 
	
		\noindent where we have set $T_{ww_-} = 1$ for $w = 0$. Together with \Cref{rproduct}, this gives 
		\begin{flalign*}
			\Imaginary R_{00} \ge \displaystyle\sum_{u \in \mathbb{V} (k+1)} |R_{0u_-}|^2 \cdot |T_{u_-} T_u|^2 \cdot \Imaginary R_{uu}^{(u_-)},
		\end{flalign*} 
	
		\noindent which yields the corollary upon setting $v = u_-$.
	\end{proof}

	\subsection{Laws of the Resolvent Entries} 
	
	\label{ResolventDensity}
	
	In this section we analyze various properties of the laws of the diagonal entries of $\boldsymbol{R}$. Throughout this section, we fix a complex number $z = E + \mathrm{i} \eta \in \mathbb{H}$. We assume that $\eta \in (0, 1)$ and that $\varepsilon \le |E| \le B$ for some fixed real numbers $\varepsilon \in (0, 1)$ and $B \ge 1$. In what follows, various constants may implicitly depend on $\varepsilon$ and $B$, even when not stated. Denote
	 \begin{flalign}
	 	\label{kappatheta1} 
	 	\varkappa_0 = \varkappa_0 (z) = - E - \Real \bigg( \displaystyle\frac{1}{R_{00} (z)} \bigg); \qquad \vartheta_0 = \vartheta_0 (z) = -\eta - \Imaginary \bigg( \displaystyle\frac{1}{R_{00} (z)} \bigg),
	 \end{flalign}
 
 	\noindent so that by \eqref{qvv} we have
 	\begin{flalign}
 		\label{kappa0theta0sum} 
 		\varkappa_0 = \displaystyle\sum_{v \sim 0} T_{0v}^2 \Real R_{vv}^{(0)}; \qquad \vartheta_0 = \displaystyle\sum_{v \sim 0} T_{0v}^2 \Imaginary R_{vv}^{(0)}.
 	\end{flalign}
	
	The following lemma then states that $(\varkappa_0, \vartheta_0)$ are (possibly correlated) $\frac{\alpha}{2}$-stable laws; we recall the notation on such variables from \Cref{Stable}.

	\begin{lem} 
		
		\label{rrealimaginary}
		
		For any complex number $z = E + \mathrm{i} \eta \in \mathbb{H}$, the following two statements hold.
		
		\begin{enumerate}
			\item We have  $\varkappa_0$ is an $\frac{\alpha}{2}$-stable law with scaling parameter
			\begin{flalign*} 
				\sigma (\varkappa_0) = \mathbb{E} \big[  |\Real R_{\star}|^{\alpha / 2} \big]^{2 / \alpha},
			\end{flalign*} 
		
			\noindent and skewness parameter
			\begin{flalign*}
				 \beta (\varkappa_0) = \displaystyle\frac{\mathbb{E} \big[ \max \{ \Real R_{\star}, 0 \}^{\alpha / 2} - \max \{ -\Real R_{\star}, 0 \}^{\alpha / 2} \big]}{\mathbb{E} \big[ |\Real R_{\star}|^{\alpha / 2} \big]}.
			\end{flalign*} 
		
			\item We have  $\vartheta_0$ is a nonnegative $\frac{\alpha}{2}$-stable law with scaling parameter $\sigma (\vartheta_0) = \mathbb{E} \big[ (\Imaginary R_{\star})^{\alpha / 2} \big]^{2 / \alpha}$. 	
		\end{enumerate} 	
	\end{lem}

	\begin{proof}
		
		By \Cref{rdistribution} and \eqref{kappatheta1}, we have that 
		\begin{flalign*}
			\varkappa_0 + \mathrm{i} \vartheta_0 = - z - \displaystyle\frac{1}{R_{\star}} = \displaystyle\sum_{j = 1}^{\infty} \xi_j R_j,
		\end{flalign*}
	
		\noindent where $(\xi_j)_{j \ge 1}$ is a Poisson point process on $\mathbb{R}_{> 0}$ with intensity measure $\frac{\alpha}{2} x^{-\alpha / 2 - 1} dx$, and $(R_j)_{j \ge 1}$ are mutually independent random variables with law $R_{\star} (z)$. In particular,
		\begin{flalign*} 
			\varkappa_0 = \Real \Bigg( \displaystyle\sum_{j = 1}^{\infty} \xi_j R_j \Bigg) = \displaystyle\sum_{j = 1}^{\infty} \xi_j \Real R_j; \qquad \vartheta_0 = \Imaginary \Bigg( \displaystyle\sum_{j = 1}^{\infty} \xi_j R_j \Bigg) = \displaystyle\sum_{j = 1}^{\infty} \xi_j \Imaginary R_j,
		\end{flalign*} 
	
		\noindent and so both statements of the lemma follows from \Cref{sumaxi}.
	\end{proof}

	The following lemma estimates the densities of $\varkappa_0 (z)$ and $\vartheta_0 (z)$. Its first part bounds the density of the former from above and below on any interval (with constants dependent on $\varepsilon$ and $B$); its second part upper bounds the both densities on intervals disjoint from $[-1, 1]$ (with constants independent of $\varepsilon$ and $B$). We establish it in \Cref{ProofSigma} below.

	\begin{lem} 
		
		\label{ar}
		
		There exist constants $C_1 = C_1 (\varepsilon, B) > 1$; $C_2 > 1$ (independent of $\varepsilon$ and $B$); and $c_1 = c_1 (\varepsilon, B) > 0$ such that the following two statements hold. 
		
		\begin{enumerate}

			\item For any interval $I \subset \mathbb{R}$, we have 
			\begin{flalign}
				\label{q00delta1}
				c_1 \displaystyle\int_I \displaystyle\frac{dx}{\big( |x| + 1 \big)^{\alpha / 2 + 1}} \le \mathbb{P} \big[ \varkappa_0 (z) \in I \big] \le C_1 \displaystyle\int_I \displaystyle\frac{dx}{\big( |x| + 1 \big)^{\alpha / 2 + 1}}.
			\end{flalign}

			\item For any interval $I \subset \mathbb{R}$ disjoint from $[-1, 1]$, we have   
			\begin{flalign}
				\label{q00delta2} 
				\mathbb{P} \big[ \varkappa_0 (z) \in I \big] \le C_2 \displaystyle\int_I \displaystyle\frac{dx}{\big( |x| + 1 \big)^{\alpha / 2 + 1}}; \qquad \mathbb{P} \big[ \vartheta_0 (z) \in I \big] \le C_2 \displaystyle\int_I \displaystyle\frac{dx}{\big( |x| + 1 \big)^{\alpha / 2 + 1}}.
			\end{flalign}
		
		\end{enumerate} 
		
	\end{lem}

\noindent	The reason we stipulate that $|E| \ge \varepsilon$ is that the first part of \Cref{ar} becomes false if $E = 0$. Indeed, in this case, it is known from \cite[Lemma 4.3(b)]{bordenave2011spectrum} that $R_{\star} (\mathrm{i} \eta)$ is almost surely purely imaginary, which implies that \eqref{q00delta1} does not hold.

	The following corollary estimates fractional moments of the diagonal resolvent entries $R_{00}$.
	\begin{cor}
	
	\label{expectationqchi} 
	
	Adopt the notation and assumptions of \Cref{ar}, and let $\delta, \chi \in (0, 1)$ be real numbers. There exists a constant $C = C(\chi, \varepsilon, B) > 1$ such that
	\begin{flalign}
	\label{qchi}
		\mathbb{E} \big[ |R_{00}|^{\chi} \big] \le C; \qquad \mathbb{E} \big[ \one_{|R_{00}| > \delta^{-1}} \cdot |R_{00}|^{\chi} \big] < C \delta^{1-\chi}.
	\end{flalign}

	\end{cor} 

	\begin{proof}
		
	Observe, by \eqref{kappatheta1} and \Cref{ar}, that there exists a constant $C_1 > 1$ such that for any $\delta > 0$ we have
	\begin{flalign*}
	\mathbb{E} \big[ \one_{|R_{00}| > \delta^{-1}} \cdot |R_{00}|^{\chi} \big] \le \mathbb{E} \big[ \one_{|E + \varkappa_0| < \delta} \cdot |E + \varkappa_0|^{-\chi} \big] \le C_1 \displaystyle\int_{E - \delta}^{E + \delta} \displaystyle\frac{dx}{(x + E)^{\chi} \big( |x| + 1 \big)^{\alpha/2 + 1}}.
	\end{flalign*} 

	\noindent Thus, the first and second bounds in \eqref{qchi} follow from the fact that there exists a constant $C_2 = C_2 (\chi) > 1$ such that 
	\begin{flalign*}
		& \displaystyle\int_{-\infty}^{\infty} \displaystyle\frac{dx}{(x+E)^{\chi} \big( |x| + 1 \big)^{\alpha/2 + 1}} \le C_2 \qquad \text{and} \qquad \displaystyle\int_{E - \delta}^{E + \delta} \displaystyle\frac{dx}{(x + E)^{\chi} \big( |x| + 1\big)^{\alpha/2 + 1}} \le C_2 \delta^{1-\chi},
	\end{flalign*} 
	
	\noindent respectively. 	
	\end{proof}

\begin{proof}[Proof of \Cref{l:compactness}]
We may suppose that $E\neq 0$, since the $E=0$ case follows from \cite[Lemma 4.3(b)]{bordenave2011spectrum}.
By \eqref{qchi} and Markov's inequality, the set ${\big\{ R_\star(E + \iu \eta) \big\}_{\eta \in (0,1)}}$ is tight. Then the conclusion follows from Prokhorov's theorem.
\end{proof}

\begin{proof}[Proof of \Cref{l:upgrade}]
The assumption implies that $\liminf_{\eta \rightarrow 0} \sigma\big(\vartheta_0(E + \iu \eta)\big)  > 0$, which  (together with \Cref{rrealimaginary}) implies the conclusion.
\end{proof}

	\subsection{Proof of \Cref{ar}}
	
	\label{ProofSigma}
	
	In this section we establish \Cref{ar}. As in the statement of that lemma, here we assume that $z = E + \mathrm{i} \eta \in \mathbb{H}$ with $\varepsilon \le |E| \le B$ and $\eta \in (0, 1)$; various constants might or might not depend on $\varepsilon$ and $B$, and we will specify which do. Throughout, we will make use of equalities (from \eqref{qvv}) 
	\begin{flalign}
		\label{q00equation} 
		\displaystyle\frac{\Imaginary R_{00}}{|R_{00}|^2} = \eta + \displaystyle\sum_{v \sim 0} T_{0v} \Imaginary R_{vv}^{(0)} = \eta + \vartheta_0; \qquad \displaystyle\frac{\Real R_{00}}{|R_{00}|^2} = - E - \displaystyle\sum_{v \sim 0} T_{0v}^2 \Real R_{vv}^{(0)} = -E - \varkappa_0. 
	\end{flalign}

	We begin with the following lemma bounding from above the quantities $\sigma (\varkappa_0)$ and $\sigma (\theta_0)$ from \Cref{rrealimaginary}, and showing that $|R_{00}|$ is unlikely to be very small. In what follows, we recall the parameters $\sigma (\varkappa_0)$ and $\sigma (\vartheta_0)$ from \Cref{rrealimaginary}.
	
	\begin{lem} 
	
	\label{q0estimate1} 
	
	There exists a constant $C > 1$ such that the following two statements hold, for any real number $t \ge 1$ and complex number $z = E + \mathrm{i} \eta \in \mathbb{H}$ with $\eta \in (0, 1)$.
	
	\begin{enumerate} 
		
		\item Adopting the notation of \Cref{ar}, we have $\sigma (\vartheta_0) < C$ and $\sigma (\varkappa_0) < C$.
		
		\item We have 
	\begin{flalign}
		\label{r00estimate} 
		\mathbb{P} \bigg[ \big| R_{00} (z) \big| \ge \displaystyle\frac{1}{|E| + t} \bigg] \ge 1 - Ct^{-\alpha / 2}.
	\end{flalign}

	\end{enumerate} 

	\end{lem} 

	\begin{proof}
		
		We begin with the proof of the first statement. To that end, observe from \eqref{q00equation} that 
		\begin{flalign*}
			\vartheta_0 = \displaystyle\sum_{v \sim 0} T_{0v}^2 \Imaginary R_{vv}^{(0)} \le \displaystyle\frac{\Imaginary R_{00}}{|R_{00}|^2} \le \displaystyle\frac{1}{\Imaginary R_{00}},
		\end{flalign*}
		
		\noindent and so 
		\begin{flalign}
			\label{theta0estimate} 
			\sigma (\vartheta_0) = \mathbb{E} \big[ (\Imaginary R_{00})^{\alpha / 2} \big]^{2 / \alpha} \le \mathbb{E} \big[ (\vartheta_0)^{-\alpha / 2} \big]^{2 / \alpha},
		\end{flalign}
		where the equality is justified by the second part of \Cref{rrealimaginary}. Recalling from \Cref{ar} that $\vartheta_0$ is a nonnegative $\frac{\alpha}{2}$-stable law with scaling parameter $\sigma (\vartheta_0)$, the bound \eqref{theta0estimate}  implies the existence of a constant $C_1 > 1$ such that $\sigma (\vartheta_0) \le C_1$ (indeed, recall that a non-negative $\alpha/2$-stable law has all its negative moments finite). The proof of the upper bound on $\sigma (\varkappa_0)$ is entirely analogous (using the second equality in \eqref{q00equation}) and is therefore omitted.
		
		To establish the second statement of the lemma, we may assume that $t > 2$, for otherwise $1 - C t^{-\alpha / 2} < 0$ for any $C > 2$. Then,
		\begin{flalign*}
			\mathbb{P} \bigg[ \big| R_{00} (z) \big| \ge \displaystyle\frac{1}{|E| + t} \bigg] = \mathbb{P} \big[ |z + \varkappa_0 + \mathrm{i} \theta_0| \le |E| + t \big] \ge \mathbb{P} \bigg[ |\varkappa_0| + |\theta_0| \le \displaystyle\frac{t}{2} \bigg]
		\end{flalign*}
		
		\noindent where the first statement follows from \eqref{kappatheta1} and the second from the facts that $t \ge 2$ and $\eta \le 1$. This, together with the first statement of the lemma, implies \eqref{r00estimate}. 
	\end{proof}

	We next require the following lemma bounding from below $\sigma (\varkappa_0)$, and estimating $\beta (\varkappa_0)$; we recall these quantities were defined in \Cref{rrealimaginary}.
	
	\begin{lem} 
		
		\label{sigmabetaestimate} 
		
		There exists a constant $c = c (\varepsilon, B) > 0$ such that  
		\begin{flalign}
			\label{sigmathetakappa} 
			 \sigma (\varkappa_0) \ge c; \qquad c - 1 \le \beta (\varkappa_0) \le 1 - c.
		\end{flalign}
	\end{lem}
	
	\begin{proof}
		
		We first establish the lower bound on $\sigma (\varkappa_0)$. To that end, observe that the upper bound on $\sigma (\varkappa_0)$ from \Cref{q0estimate1} yields a constant $c_1 > 0$ such that 
		\begin{flalign}
			\label{kappa0estimate} 
			\mathbb{P} \bigg[ |\varkappa_0| \le \displaystyle\frac{\varepsilon}{2} \bigg] \ge c_1  \varepsilon.
		\end{flalign}
		
		\noindent Moreover, \eqref{r00estimate} yields a constant $c_2 = c_2 (\varepsilon, B) > 0$ such that for any $\varepsilon \in (0, 1)$ we have 
		\begin{flalign}
			\label{q00estimate}
			\mathbb{P} \bigg[ |R_{00}| \ge \displaystyle\frac{c_2 \varepsilon^{2 / \alpha}}{|E| + 1} \bigg] \ge 1 - \displaystyle\frac{c_1 \varepsilon}{2}.
		\end{flalign}
		
		\noindent Combining the second equality in \eqref{q00equation}, \eqref{kappa0estimate}, \eqref{q00estimate}, and the fact that $\varepsilon \le |E| \le B$ yields 
		\begin{flalign}
			\mathbb{P} \Bigg[ |\Real R_{00}| \ge \displaystyle\frac{\varepsilon}{2} \cdot \bigg( \displaystyle\frac{c_2 \varepsilon^{2 / \alpha}}{B + 1} \bigg)^2 \Bigg] \ge \displaystyle\frac{c_1 \varepsilon}{2}.
		\end{flalign}
		
		\noindent It follows that that there exists a constant $c_3 = c_3 (\varepsilon, B) > 0$ such that $\mathbb{E} \big[ |\Real R_{00}|^{\alpha / 2} \big] > c_3$, which verifies the lower bound on $\sigma (\varkappa_0) = \mathbb{E} \big[ |\Real R_{00}|^{\alpha / 2} \big]^{2 / \alpha}$ from \eqref{sigmathetakappa}. 
		
		It remains to bound $\beta (\varkappa_0)$; let us assume in what follows that $E > 0$, for the proof when $E < 0$ is entirely analogous. To that end, observe that the estimates $\sigma (\vartheta_0) \le C$ and $c \le \sigma (\varkappa_0) \le C$ from \Cref{q0estimate1} and \eqref{sigmathetakappa} together imply the existence of a constant $c_4 = c_4 (\varepsilon, B) > 0$ such that $\mathbb{P} [\Real R_{00} < -c_4] > c_4$. Since $\sigma (\varkappa_0) < C$ and 
		\begin{flalign}
			\label{betakappa01} 
			\beta (\varkappa_0) = 1  - \displaystyle\frac{2}{\sigma (\varkappa_0)^{\alpha / 2}} \cdot \mathbb{E} \big[ \max \{ -\Real R_{00}, 0 \}^{\alpha / 2} \big],
		\end{flalign} 
		
		\noindent this yields a constant $c_5 = c_5 (\varepsilon, B) > 0$ such that 
		\begin{flalign} 
			\label{beta1} 
			\beta (\varkappa_0) \le 1 - c_5,
		\end{flalign}
		
		\noindent verifying the upper bound on $\beta (\varkappa_0)$ from \eqref{sigmathetakappa}. 
		
		Next, by \eqref{beta1}, there exists a constant $c_6 = c_6 (\varepsilon, B) > 0$ such that $\mathbb{P} [\varkappa_0 < -2B] \ge c_6$. This, with the second equality of \eqref{q00equation}, \Cref{q0estimate1}, and the fact that $|E| \le B$, gives a constant $c_7 = c_7 (\varepsilon, B) > 0$ such that $\mathbb{P} [\Real R_{00} > c_7] > c_7$. Again using \eqref{betakappa01} and the bound $\sigma (\varkappa_0) \le C$, this yields a constant $c_8 = c_8 (\varepsilon, B) > 0$ so that $\beta (\varkappa_0) \ge c_8 - 1$, confirming the last bound in \eqref{sigmathetakappa}.	
	\end{proof}

	Now we can establish \Cref{ar}.

		\begin{proof}[Proof of \Cref{ar}]
		
		By \Cref{rrealimaginary}, $\varkappa_0 (z)$ is an $\frac{\alpha}{2}$-stable law with scaling parameter $\sigma (\varkappa_0)$ and skewness parameter $\beta (\varkappa_0)$, and $\vartheta_0 (z)$ is a nonnegative $\frac{\alpha}{2}$-stable law with scaling parameter $\sigma (\vartheta_0)$. By \Cref{q0estimate1} and \Cref{sigmabetaestimate}, there exist constants $C > 1$ and $c = c (\varepsilon, B) > 1$ such that $\sigma (\vartheta_0) \le C$, $c \le \sigma (\varkappa_0) \le C$, and $c - 1 \le \beta (\varkappa_0) \le 1 - c$. The upper bounds $\sigma (\varkappa_0) \le C$ and $\sigma (\vartheta_0) \le C$ imply the second statement \eqref{q00delta2} of the lemma. The upper bound in the first statement \eqref{q00delta1} follows from the bounds $c \le \sigma (\varkappa_0) \le C$; the lower bound there follows from this together with the estimates $c - 1 \le \beta (\varkappa_0) \le 1 - c$.
	\end{proof}

\subsection{Real Boundary Values of the Resolvent}\label{s:rebdd}

We begin with the following definition.

\begin{definition}
	\label{locre} 
Fix $\alpha \in (0,1)$ and $E \in \bbR$, and let $R_\loc(E)$ be the real random variable given by
\bex
R_\loc(E) = - \frac{1}{E + \varkappa_\loc(E)}, 
\eex
where we recall that $\varkappa_\loc$ was defined in \Cref{pkappae}.
\end{definition}
		
The following lemma asserts that any real boundary value of  $R_\star(z)$ is equal to $R_\loc(E)$. We recall the quantities $a(E)$ and $b(E)$ from \eqref{opaque}.

\begin{lem}\label{l:boundarybasics}
Fix $E \in \bbR$, and let $R(E)$ be any limit point (under the weak topology) of the sequence $\big\{ R_\star(E + \iu \eta) \big\}_{\eta > 0}$ as $\eta$ tends to $0$. If $\Im R(E) = 0$, then
\be\label{abrepresentation}
a(E) = \E\left[ \big( R(E) \big)^{\alpha/2}_+   \right], \qquad b(E) = \E\left[ \big( R(E) \big)^{\alpha/2}_-  \right],
\ee
and, abbreviating $a = a(E)$ and $b = b(E)$, 
\begin{equation}\label{fixedpoint}
a = \E \left[ \left( E + a^{2/\alpha} S_1 - b^{2/\alpha} S_2  \right)_-^{-\alpha/2}  \right],\qquad
b = \E \left[ \left( E + a^{2/\alpha} S_1 - b^{2/\alpha} S_2  \right)_+^{-\alpha/2}  \right],
\end{equation} 
where $S_1$ and $S_2$ are independent, nonnegative $\alpha/2$-stable laws and we have used the shorthand notation $(x)_{\pm}^{-\alpha /2} := ((x^{-1} )_{\pm} )^{\alpha /2}$. Further, we have $R(E) = R_\loc(E)$ in distribution.
\end{lem}
\begin{proof}
We prove \eqref{abrepresentation} first.
We may suppose that $E \neq 0$, since otherwise $\Im R(E)$ is almost surely positive and the claim is vacuously true; see \cite[Lemma 4.3(b)]{bordenave2011spectrum}.
Let $\{\eta_k\}_{k=1}^\infty$ be a decreasing sequence such that \bex
\lim_{k\rightarrow \infty} R_\star(E + \iu \eta_k) = R(E).
\eex Set $z_k = E + \iu \eta_k$. 
To show \eqref{abrepresentation}, note that 
\be\label{y1}
y(E) = \lim_{k \rightarrow \infty} \E \left[ \big(- \iu R_\star(z_k)   \big)^{\alpha/2}  \right],
\ee
since $y(z)$ extends continuously to $\overline{\bbH}$.
We have
\bex
R_\star(z_k) = - \frac{1}{z + \varkappa_0(z_k) + \iu \vartheta_0(z_k)},
\eex
and using \Cref{ar}, we find that the sequence $\big( \Re R_\star(z_k) \big)_+^{\alpha/2}$ is uniformly integrable. 
Then 
\bex
\lim_{k \rightarrow \infty} \E\left[ \big( \Re R_\star(z_k) \big)_+^{\alpha/2} \right]^{2/\alpha}
= \E\left[ \big( \Re R(E) \big)_+^{\alpha/2} \right]^{2/\alpha},
\eex
and similarly for the limits of the $\alpha/2$-moments of $\big( \Re R_\star(z_k) \big)_-$, $\big( \Im R_\star(z_k) \big)_+$, and $\big( \Im R_\star(z_k) \big)_-$. By assumption, the latter two limits vanish. Using \eqref{y1}, this gives
\bex
y(E) = 
(-\iu)^{\alpha/2} \E\left[ \big( \Re R(E) \big)_+^{\alpha/2} \right]^{2/\alpha} + (\iu)^{\alpha/2} \E\left[ \big( \Re R(E) \big)_-^{\alpha/2} \right]^{2/\alpha}.
\eex
Because the complex numbers $\iu^{\alpha/2}$ and $(-\iu)^{\alpha/2}$ are linearly independent over $\bbR$, the previous equation implies that \eqref{abrepresentation} holds by the definitions of $a(E)$ and $b(E)$ in \eqref{opaque}.

We recall from \Cref{rdistribution}
that, in distribution,
\be\label{rde2}
R_\star(z_k) = 
 - \left( z_k + \sum_{j=1}^\infty \xi_j R_j(z_k) \right)^{-1}
\ee
for $k \ge 1$.
By the Poisson thinning property (\Cref{sumaxi}),
\begin{align}\label{pthinning}
 \sum_{j=1}^\infty \xi_j R_j(z_k)
&= 
\E\left[ \big( \Re R_\star(z_k) \big)_+^{\alpha/2} \right]^{2/\alpha} S_1
- \E\left[ \big( \Re R_\star(z_k) \big)_-^{\alpha/2} \right]^{2/\alpha} S_2\\
&+ \iu \cdot \E\left[ \big( \Im R_\star(z_k) \big)_+^{\alpha/2} \right]^{2/\alpha} S_3
- \iu \cdot \E\left[ \big( \Im R_\star(z_k) \big)_-^{\alpha/2} \right]^{2/\alpha} S_4\notag
\end{align}
for $k \ge 1$, where $\{S_i\}_{i=1}^4$ are nonnegative $\alpha/2$-stable laws, the pair $(S_1, S_2)$ is independent, and the pair $(S_3, S_4)$ is independent.

Since $\Im R(E) = 0$, we have 
\begin{align}
\lim_{k\rightarrow \infty} \sum_{j=1}^\infty \xi_j R_j(z_k)
&= 
\E\left[ \big( \Re R(E) \big)_+^{\alpha/2} \right]^{2/\alpha} S_1
- \E\left[ \big( \Re R(E) \big)_-^{\alpha/2} \right]^{2/\alpha} S_2.
\end{align}
Then taking the limit as $k$ tends to infinity in \eqref{rde2}, we obtain
\begin{align}\label{fixedpenultimate}
R(E) &= 
-\left(
E 
+
\E\left[ \big( \Re R(E) \big)_+^{\alpha/2} \right]^{2/\alpha} S_1
- \E\left[ \big( \Re R(E) \big)_-^{\alpha/2} \right]^{2/\alpha} S_2
\right)^{-1}\\
&= -\big(
E 
+
a(E) S_1
- b(E) S_2
\big)^{-1}.\notag
\end{align}
We then obtain \eqref{fixedpoint} from \eqref{fixedpenultimate} after separating into positive and negative and taking $\alpha/2$ moments. Since \eqref{fixedpenultimate} also holds for $R_\loc$, it further shows that $R(E) = R_\loc(E)$ in distribution. This completes the proof.
\end{proof}

\newpage 

	\chapter{Fractional Moment Criterion} 
	
	\label{MomentFractional}	

	\section{Integral and Poisson Point Process Estimates}
	
	\label{G0vMoment}

	In this section we provide various estimates for heavy-tailed integrals and functionals of Poisson point processes that will be used in the proof of \Cref{limitr0j} in \Cref{EstimateMomentss} below. To briefly indicate how such quantities arise, let us explain how one might estimate $\Phi_L (s; z)$ at $L = 1$, namely, $\Phi_1 (s; z) = \mathbb{E} \big[ \sum_{v \sim 0} |R_{0v}|^s \big]$ (recall \Cref{moment1}). To that end, we first apply \Cref{rproduct} to express the off-diagonal resolvent entry $R_{0v}$ as a product of diagonal ones. For $v \sim 0$, this yields
	\begin{flalign*}
		|R_{0v}|^s = |R_{00}|^s \cdot |T_{0v}|^s \cdot \big|R_{vv}^{(0)} \big|^s = \displaystyle\frac{|T_{0v}|^s}{\big|z + T_{0v}^2 R_{vv}^{(0)} + K_v \big|^s} \cdot \big| R_{vv}^{(0)} \big|^s, \quad \text{where $K_v = \displaystyle\sum_{\substack{w \sim 0 \\ w \ne v}} T_{0w}^2 R_{ww}^{(0)}$},
	\end{flalign*}

	\noindent where in the last equality we applied the Schur complement identity \eqref{qvv}. Observe that the $K_v$ all have the same law $K$, which by \Cref{tuvalpha} and \Cref{sumaxi} is $\frac{\alpha}{2}$-stable. Hence, 
	\begin{flalign*}
		\Phi_1 (s; z) = \mathbb{E} \Bigg[ \displaystyle\sum_{v \sim 0} \displaystyle\frac{|T_{0v}|^s}{\big| z + T_{0v}^2 R_{vv}^{(0)} + K_v \big|^s} \cdot \big| R_{vv}^{(0)} \big|^s\Bigg].
	\end{flalign*}

	\noindent So, to bound $\Phi_1 (s; z)$ (and more generally $\Phi_L (s; z)$) we estimate more general sums of the form 
	\begin{flalign}
		\label{expectationtjrjkj}
		\mathbb{E} \Bigg[ \displaystyle\sum_{j=1}^{\infty} \displaystyle\frac{|T_j|^s}{|z + T_j^2 R_j + K_j|^{\chi}} \cdot |R|^s \Bigg],
	\end{flalign}

	\noindent where the $(T_j)_{j \ge 1}$ form a Poisson point process with intensity measure $\alpha x^{-\alpha-1} dx$, and the $R_j$ and $K_j$ are random variables. By the Campbell theorem (\Cref{fidentityxi}), the above expectation is given by an explicit heavy-tailed integral. We begin by bounding this integral in \Cref{EstimateIntegral}. We then use these integral estimates to bound quantities of the type \eqref{expectationtjrjkj} in \Cref{EstimateProcess}, conditional on an estimate for inverse moments of the $K_j$, which we prove in \Cref{ProofV}.

	\subsection{Integral Estimates} 
	
	\label{EstimateIntegral} 
	
	In this section we estimate certain integrals that arise from applying the Campbell theorem, \Cref{fidentityxi}, to evaluate expectations of the form \eqref{expectationtjrjkj}. Throughout this section, we fix complex numbers $z, R \in \mathbb{C}$ and real numbers $s \in (\alpha, 1)$ and $\chi \in \big( \frac{s-\alpha}{2}, 1 \big)$. We assume that there exists a constant $\varpi > 0$ such that 
	\begin{flalign}
		\label{chisalphaomega} 
		\displaystyle\frac{s-\alpha}{2} + \varpi \le \chi \le 1 - \varpi.
	\end{flalign}
	
	\noindent In what follows, various constants may depend on $\varpi$ (and $\alpha$), even when not written explicitly, but not on $z$, $R$, $s$, or $\chi$. Denoting
	\begin{flalign}
		\label{zgamma}
		\begin{aligned}
		\gamma = |R|^{-1/2} & \cdot \big(  |z| + 1 \big)^{1/2}; \qquad \mathfrak{Y} = \big( |z| + 1 \big)^{(s-\alpha)/2 - \chi} \cdot |R|^{(\alpha-s)/2}; \\
		& \mathfrak{z} (t) = \displaystyle\frac{t^{s-\alpha-1}}{\big( |z + t^2 R| + 1 \big)^{\chi}}; \qquad \mathfrak{Z} = \displaystyle\int_0^{\infty} \mathfrak{z} (t)\, dt,
		\end{aligned} 
	\end{flalign}

	\noindent we have the following estimate for $\mathfrak{Z}$.
	
	\begin{lem} 
		
		\label{integral3} 
		
		Under \eqref{chisalphaomega} and \eqref{zgamma}, there exists a constant $C > 1$ such that 
		\begin{flalign*} 
			C^{-1} (s-\alpha)^{-1} \cdot \mathfrak{Y} \le \mathfrak{Z} \le C (s-\alpha)^{-1} \cdot \mathfrak{Y}.
		\end{flalign*} 
	
	\end{lem} 

	To show the above lemma, for any $\delta > 0$, we define the quantities 
	\begin{flalign*} 
		\mathfrak{Z}_- (\delta) = \displaystyle\int_0^{\delta \gamma} \mathfrak{z} (t)\, dt; \qquad \mathfrak{Z}_+ (\delta) = \displaystyle\int_{\gamma / \delta}^{\infty} \mathfrak{z} (t)\, dt.
	\end{flalign*} 
	
	\noindent Then, \Cref{integral3} will quickly follow from the following two lemmas bounding $\mathfrak{Z}_- (\delta)$ and $\mathfrak{Z}_+ (\delta)$. 
	
	\begin{lem} 
		
		\label{integral} 

	Under \eqref{chisalphaomega} and \eqref{zgamma}, there exists a constant $C > 1$ such that  
	\begin{flalign}
		\label{yz} 
		\mathfrak{Z}_- (\delta) \le C (s-\alpha)^{-1} \cdot \delta^{s-\alpha} \cdot \mathfrak{Y}
	\end{flalign}
	for all $\delta >0$. 
	Further, if $\delta \in (0,1]$, we have
		\begin{flalign}
		\label{yzlower} 
		 C^{-1} (s-\alpha)^{-1} \cdot \delta^{s-\alpha} \cdot \mathfrak{Y} \le \mathfrak{Z}_- (\delta).
	\end{flalign} 
\end{lem}

\begin{lem} 
	
	\label{integral2} 
	
	Under \eqref{chisalphaomega} and \eqref{zgamma}, for all $\delta \in (0,1]$, there exists a constant $C > 1$ such that   
	\begin{flalign}
		\label{yz2} 
		C^{-1} \delta^{2\chi - s + \alpha} \cdot \mathfrak{Y} \le \mathfrak{Z}_+ (\delta) \le C \delta^{2\chi - s + \alpha} \cdot \mathfrak{Y}.
	\end{flalign} 
\end{lem}

\begin{proof}[Proof of \Cref{integral3}]
	This follows from summing \Cref{integral} and \Cref{integral2} at $\delta = 1$. 
\end{proof}

Now we show \Cref{integral} and \Cref{integral2}.

\begin{proof}[Proof of \Cref{integral}]
	
	Begin by supposing that $\delta \le 1$. We first establish the lower bound in \eqref{yzlower}. Changing variables $v = \gamma^{-1} t$, observe that
	\begin{flalign*}
		\mathfrak{Z}_- (\delta) \ge \displaystyle\int_0^{\delta \gamma} \displaystyle\frac{t^{s-\alpha-1} dt}{\big( |z| + t^2 |R| + 1 \big)^{\chi}} & = \gamma^{s-\alpha} \cdot \big( |z| + 1 \big)^{-\chi} \displaystyle\int_0^{\delta} \displaystyle\frac{v^{s-\alpha-1} dv}{(v^2+1)^{\chi}} \\
		& \ge \displaystyle\frac{1}{2} \cdot R^{(\alpha-s)/2} \cdot \big( |z| + 1 \big)^{(s-\alpha)/2 - \chi} \displaystyle\int_0^{\delta} v^{s-\alpha-1} dt = \displaystyle\frac{\delta^{s-\alpha}}{2 (s-\alpha)} \cdot \mathfrak{Y},
	\end{flalign*} 
	
	\noindent where in the third statement we used the definition of $\gamma$ and the fact that $v^2 + 1 \le 2$ for $v \le \delta$. This confirms the lower bound in the \eqref{yz}.
	
	To establish the upper bound, observe by again changing variables $v = \gamma^{-1} t$ that 
	\begin{flalign*}
		\mathfrak{Z}_- (\delta) & = \gamma^{s-\alpha} \displaystyle\int_{\delta/2}^{\delta} \displaystyle\frac{v^{s-\alpha-1} dv}{\Big( \big| z + (|z| + 1) v^2 \big| + 1 \Big)^{\chi}} + \gamma^{s-\alpha} \displaystyle\int_0^{\delta/2} \displaystyle\frac{v^{s-\alpha-1} dv}{\Big( \big| z + (|z|+1) v^2 \big| + 1 \Big)^{\chi}} \\
		& \le 2 \gamma^{s-\alpha} \cdot \big( |z| + 1 \big)^{-\chi} \cdot \delta^{s-\alpha-1} \displaystyle\int_{\delta/2}^{\delta} \Bigg( \bigg| \displaystyle\frac{z}{|z|+1}+ v^2 \bigg| + \displaystyle\frac{1}{|z|+1} \Bigg)^{-\chi} dv \\ & 
		\qquad + 4 \gamma^{s-\alpha} \cdot \big( |z| + 1 \big)^{-\chi} \displaystyle\int_0^{\delta/2} v^{s-\alpha-1} dv,
	\end{flalign*}

	\noindent where the last statement follows from the fact that $\big| z + (|z|+1)v^2 \big| + 1 \ge (|z|+1)/4$ for $v \le 1/2$. Thus, 
	\begin{flalign}
		\label{zdelta1} 
		\begin{aligned}
			\mathfrak{Z}_- (\delta) & \le 2 R^{(\alpha-s)/2} \cdot \big( |z| + 1 \big)^{(s-\alpha)/2-\chi} \cdot \delta^{s-\alpha-1} \displaystyle\int_0^{\delta/2} \Bigg( \bigg| \displaystyle\frac{z}{|z|+1}+ v^2 \bigg| + \displaystyle\frac{1}{|z|+1} \Bigg)^{-\chi} dv \\
			& \qquad + \displaystyle\frac{4\delta^{s-\alpha}}{s-\alpha} \cdot R^{(\alpha-s)/2} \cdot \big( |z| + 1 \big)^{(s-\alpha)/2-\chi} \\
			& \le \displaystyle\frac{4 \delta^{s-\alpha}}{s-\alpha} \cdot \mathfrak{Y} \cdot \Bigg( 1 + \delta^{-1} \displaystyle\int_0^{\delta/2} \bigg( \Big| \displaystyle\frac{z}{|z|+1} + v^2 \Big| + \displaystyle\frac{1}{|z|+1} \bigg)^{-\chi} dv \Bigg).
		\end{aligned} 
	\end{flalign}
	
	\noindent If $|z| \le 2$, then
	\begin{flalign*}
		\delta^{-1} \displaystyle\int_0^{\delta/2} \Bigg( \bigg| \displaystyle\frac{z}{|z|+1}+ v^2 \bigg| + \displaystyle\frac{1}{|z|+1} \Bigg)^{-\chi} dv \le \big( |z| + 1 \big)^{\chi} \le 3,
	\end{flalign*}
	
	\noindent and inserting this into \eqref{zdelta1} yields the upper bound in \eqref{yz}. Thus, assume instead that $|z| \ge 2$. Denoting $\kappa = z (|z|+1)^{-1}$, we have $|\kappa| \ge \frac{2}{3}$, in which case 
	\begin{flalign*}
		\delta^{-1} \displaystyle\int_0^{\delta/2} \Bigg( \bigg| \displaystyle\frac{z}{|z|+1}+ v^2 \bigg| + \displaystyle\frac{1}{|z|+1} \Bigg)^{-\chi} dv \le \delta^{-1} \displaystyle\int_0^{\delta/2} \displaystyle\frac{dv}{|v^2 + \kappa|^{\chi}} \le \delta^{-1} \displaystyle\int_0^{\delta/2} \bigg| v^2 - \displaystyle\frac{2}{3} \bigg|^{-\chi} dv \le C_1,
	\end{flalign*}
	
	\noindent for some constant $C_1 > 1$ (recall from \eqref{chisalphaomega} that $\chi \le 1 - \varpi$). Inserting this into \eqref{zdelta1} yields the upper bound in \eqref{yz}.	
	
	We now suppose that $\delta > 1$. In light of the previous calculations, it suffices to bound the integral
	\begin{align*}
	\gamma^{s-\alpha} \displaystyle\int_{1}^{\delta} \displaystyle\frac{v^{s-\alpha-1} dv}{\Big( \big| z + (|z| + 1) v^2 \big| + 1 \Big)^{\chi}}&
	\le \gamma^{s-\alpha} \delta^{s - \alpha - 1} \displaystyle\int_{1}^{\delta} \displaystyle\frac{ dv}{\Big( \big| z + (|z| + 1) v^2 \big| + 1 \Big)^{\chi}}\\
	&\le \gamma^{s-\alpha} \delta^{s-\alpha-1} ( |z| + 1)^{-\chi}\int_1^\delta (v^2 - 1)^{-\chi} \, dv\\ 
	& \le C \gamma^{s-\alpha} \delta^{s-\alpha} ( |z| + 1)^{-\chi} (1-\chi)^{-1},
	\end{align*}
which completes the proof.
\end{proof}

\begin{proof}[Proof of \Cref{integral2}] 
	
	We first establish the lower bound in \eqref{yz2}. As in the proof of \Cref{integral}, we change variables $v = \gamma^{-1} z$ to deduce
	\begin{flalign*}
		\mathfrak{Z}_+ (\delta) \ge \displaystyle\int_{\gamma / \delta}^{\infty} \displaystyle\frac{t^{s-\alpha-1} dt}{\big( |z| + t^2 |R| + 1 \big)^{\chi}} & = \gamma^{s-\alpha} \cdot \big( |z| + 1 \big)^{-\chi} \displaystyle\int_{1/\delta}^{\infty} \displaystyle\frac{v^{s-\alpha-1} dv}{(v^2 + 1)^{\chi}} \\
		& \ge \displaystyle\frac{1}{2} \cdot R^{(\alpha-s)/2} \cdot \big( |z| + 1 \big)^{(s-\alpha)/2 - \chi} \displaystyle\int_{1/\delta}^{\infty} v^{s-\alpha-2\chi-1} dv \\
		& \ge \displaystyle\frac{\delta^{s-\alpha-2\chi}}{2 (2\chi - s + \alpha)} \cdot \mathfrak{Y},
	\end{flalign*} 

	\noindent which verifies the lower bound in \eqref{yz2} (upon recalling from \eqref{chisalphaomega} that $\chi \ge \frac{s-\alpha}{2} + \varpi$). 
	
	To establish the upper bound, observe (again setting $v = \gamma^{-1} z$) that 
	\begin{flalign*} 
		\mathfrak{Z}_+ (\delta) & = \gamma^{s-\alpha} \displaystyle\int_{1/\delta}^{2/\delta} \displaystyle\frac{v^{s-\alpha-1} dv}{\Big( \big| z + (|z| + 1) v^2 \big| + 1 \Big)^{\chi}} + \gamma^{s-\alpha} \displaystyle\int_{2/\delta}^{\infty} \displaystyle\frac{v^{s-\alpha-1} dv}{\Big( \big| z + (|z| + 1) v^2 \big| + 1 \Big)^{\chi}} \\
		& \le  \gamma^{s-\alpha} \cdot \big( |z| + 1 \big)^{-\chi} \cdot \delta^{\alpha-s+1} \displaystyle\int_{1/\delta}^{2/\delta} \Bigg( \bigg| \displaystyle\frac{z}{|z| + 1} + v^2 \bigg| + \displaystyle\frac{1}{|z|+1} \Bigg)^{-\chi} dv \\
		& \qquad + 2 \gamma^{s-\alpha} \cdot \big( |z| + 1 \big)^{-\chi}\displaystyle\int_{2/\delta}^{\infty} v^{s-\alpha-2\chi-1} dv,
	\end{flalign*} 
	
	\noindent where in the second statement we used the fact that $\big| z + (|z| + 1) v^2 \big| \ge (|z|+1) v^2/2$ for $v \ge 2$. Thus, again recalling from \eqref{chisalphaomega} that $\chi > \frac{s-\alpha}{2} + \varpi$, there exists a constant $C_1 > 1$ such that 
	\begin{flalign}
		\label{zdelta2} 
		\begin{aligned} 
		\mathfrak{Z}_+ (\delta) \ge C_1 \delta^{2\chi - s + \alpha} \cdot \mathfrak{Y} \cdot \Bigg( 1 + \delta^{1-2\chi} \displaystyle\int_{1/\delta}^{2/\delta} \bigg( \Big| \displaystyle\frac{z}{|z| + 1} + v^2 \Big| + \displaystyle\frac{1}{|z| + 1} \bigg)^{-\chi} dv \Bigg).
		\end{aligned} 
	\end{flalign} 

	\noindent Now, observe that 
	\begin{flalign*}
		\delta^{1-2\chi} \displaystyle\int_{1/\delta}^{2/\delta} \Bigg( \bigg| \displaystyle\frac{z}{|z| + 1} + v^2 \bigg| + \displaystyle\frac{1}{|z| + 1} \Bigg)^{-\chi} dv \le \delta^{1-2\chi} \displaystyle\int_{1/\delta}^{2/\delta} (v^2 - 1)^{-\chi} dv \le \displaystyle\frac{1}{1-\chi},
	\end{flalign*}

	\noindent and inserting this into \eqref{zdelta2} yields the upper bound in \eqref{yz2}.  
\end{proof}

\subsection{Poisson Point Process Estimates}

\label{EstimateProcess}

In this section we establish estimates for quantities of the form \eqref{expectationtjrjkj}. We begin by imposing the following restrictions on the random variables $R_j$ and $K_j$ there. Throughout, $z = E + \mathrm{i} \eta \in \mathbb{H}$ is a complex number, and $s \in (\alpha, 1)$ and $\chi \in \big( \frac{s-\alpha}{2}, 1 \big)$ are real numbers satisfying \eqref{chisalphaomega} for some $\varpi > 0$.

\begin{assumption}
	
	\label{sqk}

Let $\mathcal{T} = (T_1, T_2, \ldots )$ be a Poisson point process with intensity $\alpha x^{-\alpha-1} dx$, and let $\mathcal{R} = (R_1, R_2, \ldots )$ and $\mathcal{S} = (S_1, S_2, \ldots )$ denote sequences of random variables independent from $\mathcal{T}$, with the $R_j \in \overline{\mathbb{H}}$ and the $S_j \in \mathbb{R}_{\ge 0}$. Assume that the pairs $(R_1, S_1), (R_2, S_2), \ldots $ are identically distributed, each with some law $(R, S) \in \overline{\mathbb{H}} \times \mathbb{R}_{\ge 0}$. In particular, $\mathcal P = \sum_i \delta_{T_i,R_i,S_i}$ is a Poisson point process on $\bbR_{\ge 0} \times \overline{\mathbb{H}} \times \mathbb{R}_{\ge 0}$. Further let $\mathcal{K} = (K_1, K_2, \ldots )$ denote a sequence of random variables with all $K_j \in \overline{\mathbb{H}}$ and $K_j = f (\mathcal P_j, R_j, S_j, U_j)$ where $\mathcal{P}_j = \sum_{i \ne  j } \delta_{T_i,R_i,S_i}$, $(U_1, U_2, \ldots)$ are iid uniform $[0,1]$ random variables and $f$ is a measurable function.

Assume there exist constants $A_1, A_2 > 1$ satisfying the below two conditions; conditional on $(T_j, R_j, S_j)$ they essentially state that the law of $\Real K_j$ has a bounded density and that $|K_j|$ has an $\frac{\alpha}{2}$-heavy tail. Here, $\mathbb{P}^j$ denotes the probability measure for $K_j$ conditional on $(T_j, R_j, S_j)$.
	
	\begin{enumerate} 
		\item \label{k11} For any interval $I \subset \mathbb{R}$, we have 
		\begin{flalign*} 
			A_1^{-1} \displaystyle\int_I \displaystyle\frac{dx}{\big( |x| + 1 \big)^{\alpha / 2 + 1}} \le \mathbb{P}^j [ \Real K_j \in I] \le A_1 \displaystyle\int_I \displaystyle\frac{dx}{ \big( |x| + 1 \big)^{\alpha / 2 + 1}}.
		\end{flalign*} 
		
		\item \label{k12} For any interval $I \subset \mathbb{R} \setminus  [-1, 1]$, we have 
		\begin{flalign*} 
			\mathbb{P}^j [\Real K_j \in I] \le A_2 \displaystyle\int_I \displaystyle\frac{dx}{\big( |x| + 1 \big)^{\alpha / 2 + 1}}; \qquad \mathbb{P}^j [\Imaginary K_j \in I] \le A_2 \displaystyle\int_I \displaystyle\frac{1}{\big( |x| + 1 \big)^{\alpha / 2 + 1}}.
		\end{flalign*}
		
	\end{enumerate}

\end{assumption}

Under the above notation, define for any $\delta > 0 $ the expectations 
\begin{flalign*} 
& \mathfrak{I} = \displaystyle\sum_{j=1}^{\infty} \displaystyle\frac{|T_j|^s S_j}{|z + T_j^2 R_j + K_j|^{\chi}}; \qquad 
\mathfrak{I} (\delta)  = \displaystyle\sum_{j=1}^{\infty} \displaystyle\frac{|T_j|^s S_j}{|z + T_j^2 R_j + K_j|^{\chi}} \cdot \one_{|T_j| < \delta}.
\end{flalign*} 

\noindent The following lemma provides upper and lower bounds for the quantities $\mathfrak{I}$ and $\mathfrak{I} (\delta)$. Below, constants might implicitly depend on $\varpi$ and $A_2$ (and $\alpha$), even when not explicitly stated, but we will clarify when they depend on $A_1$ or $s$. In particular, observe that the third statement below removes the dependence on $A_1$ of the constant $C$ from the first part, assuming that $(R_j, S_j, K_j)$ are identically distributed and that $|\Re z| \ge 2$.

\begin{lem} 
	
	\label{expectationsum1} 
	
	The following three statements hold. 
	
	\begin{enumerate} 
		
		\item There exist constants $c > 0$ and $C (A_1) > 1$ such that 
		\begin{flalign*} 
		c (s-\alpha)^{-1} & \cdot \big( |z| + 1 \big)^{(s-\alpha)/2 - \chi} \cdot \mathbb{E} \big[ |R|^{(\alpha-s)/2} S \big] \\
		& \le \mathbb{E} [\mathfrak{I}] \le C (s-\alpha)^{-1} \cdot \big( |z| + 1 \big)^{(s-\alpha)/2 - \chi} \cdot \mathbb{E} \big[ |R|^{(\alpha-s)/2} S \big].
	\end{flalign*} 

	\item  There exists a constant $C > 1$ such that  
	\begin{flalign*} 
		 \mathbb{E} \big[ \mathfrak{I} (\delta) \big] \le C (s-\alpha)^{-1} \cdot \delta^{s-\alpha} \cdot \big( |z| + 1 \big)^{-\chi} \cdot \mathbb{E} [S].
	\end{flalign*} 
	
	\item Assume that $|\Real z| \ge 2$ and that the $(R_j, S_j, K_j)$ are all identically distributed and independent from $\mathcal{T}$. Then there exists some constant $C > 1$ (independent of $A_1$) such that 
	\begin{flalign*}
		\mathbb{E} [\mathfrak{I}] \le C (s-\alpha)^{-1} \cdot \big( |z| + 1 \big)^{(s-\alpha)/2 - \chi} \cdot \mathbb{E} \big[ |R|^{(\alpha-s)/2} S \big].
	\end{flalign*}

	\end{enumerate}

\end{lem}

To establish \Cref{expectationsum1}, we require the following lemma bounding negative moments of continuous random variables, which will be established in \Cref{ProofV} below.

\begin{lem}
	
	\label{zxrkestimate} 
	
	There exist constants $c > 0$; $C_1 = C_1 (A_1) > 1$; and $C_2 > 1$ such that the following holds. Fix real numbers $x > 0$ and $\delta \in (0,1]$, a complex number $Q \in \mathbb{H}$ with $|Q| = 1$, and a complex random variable $K$ satisfying the two assumptions in \Cref{sqk}. Setting $V = V(K) = |x Q + K|$, the below statements hold. 
	
	\begin{enumerate} 
		
		\item We have $c (x + 1)^{-\chi} \le \mathbb{E} [ V^{-\chi} ] \le C_1 (x + 1)^{- \chi}$.
		\item If $x |\Real Q| \ge 2$, then $\mathbb{E} [ V^{-\chi} ] \le C_2 (x+1)^{-\chi}$. 
		\item We have $\mathbb{E} \big[ V^{-\chi} \cdot \one_{V \le \delta} \big] \le C_1 \delta^{1 - \chi} \cdot (x + 1)^{-\alpha / 2 - 1}$. 
		
		\item For any finite union of intervals $J \subset \mathbb{R}$ with total length at most $\delta$, we have 
		\begin{flalign*} 
			\mathbb{E} \big[ V^{-\chi} \cdot \one_{\Real K \in J} \big] \le C_1 \delta^{1 - \chi} \cdot (x + 1)^{-\alpha/2 - 1} + C \delta \cdot (x + 1)^{-\chi}.
		\end{flalign*}

	\end{enumerate} 
	
\end{lem}

\begin{proof}[Proof of \Cref{expectationsum1}] 

	Throughout this proof, let $\mathbb{E}^j$ denote the expectation for $K_j$, conditional on $(T_j, R_j, S_j)$. In view of the first part of \Cref{zxrkestimate}, there exist constants $c_1 > 0$ and $C_1 = C_1 (A_1) > 1$ such that 
	\begin{flalign}
	\label{ztjrj1} 
		c_1 \cdot \displaystyle\frac{1}{\big( |z + T_j^2 R_j| + 1 \big)^{\chi}} \le \mathbb{E}^j \Bigg[ \displaystyle\frac{1}{|z + T_j^2 R_j + K_j|^{\chi}} \Bigg] \le C_1 \cdot \displaystyle\frac{1}{\big( |z + T_j^2 R_j| + 1 \big)^{\chi}}.
	\end{flalign}
	
	\noindent Further observe from \Cref{fidentityxi} that
	\begin{flalign}
	\label{ztjrj2}
	\begin{aligned}
		\mathbb{E} \Bigg[ \displaystyle\sum_{j=1}^{\infty} \displaystyle\frac{|T_j|^s S_j}{\big( |z + T_j^2 R_j| + 1 \big)^{\chi}} \Bigg] &  = \alpha \displaystyle\int_0^{\infty} \mathbb{E} \Bigg[ \displaystyle\frac{t^s S}{\big( |z + t^2 R| + 1 \big)^{\chi}} \Bigg] \cdot t^{-\alpha-1} dt; \\
		\mathbb{E} \Bigg[ \displaystyle\sum_{j=1}^{\infty} \displaystyle\frac{|T_j|^s S_j}{\big( |z+T_j^2 R_j| + 1 \big)^{\chi}} \cdot \one_{|T_j| \le \delta} \Bigg] & = \alpha \displaystyle\int_0^{\delta} \mathbb{E} \Bigg[ \displaystyle\frac{t^s S}{\big( |z + t^2 R| + 1 \big)^{\chi}}\Bigg] \cdot t^{-\alpha-1} dt.	
	\end{aligned}
	\end{flalign}

	\noindent Moreover, for any deterministic $R_0 \in \mathbb{C}$, there exists a constant $C_2 > 1$ such that 
	\begin{flalign}
	\label{ztjrj3}
	\begin{aligned}
		C_2^{-1} & (s-\alpha)^{-1} \cdot \big( |z| + 1 \big)^{(s-\alpha)/2 - \chi} \cdot |R_0|^{(\alpha-s)/2} \\
		& \le \alpha \displaystyle\int_0^{\infty}  \displaystyle\frac{t^{s-\alpha-1} dt}{\big( |z + t^2 R_0| + 1 \big)^{\chi}} \le C_2 (s-\alpha)^{-1} \cdot \big( |z| + 1 \big)^{(s-\alpha)/2 - \chi} \cdot |R_0|^{(\alpha-s)/2},
		\end{aligned}
	\end{flalign}
	
	\noindent and 
	\begin{flalign}
	\label{ztjrj4}
		\alpha \displaystyle\int_0^{\delta}  \displaystyle\frac{t^{s-\alpha-1} dt}{\big( |z + t^2 R_0| + 1 \big)^{\chi}} \le C_2 (s-\alpha)^{-1} \cdot \delta^{s-\alpha} \cdot \big( |z| + 1 \big)^{-\chi}.
	\end{flalign}
	
	\noindent where the first bound follows from \Cref{integral3} and the second from \Cref{integral} (using $\delta \gamma^{-1}$ there instead of $\delta$). Combining \eqref{ztjrj1}, \eqref{ztjrj2}, \eqref{ztjrj3}, and \eqref{ztjrj4} yields the first two statements of the lemma.
	
	To establish the third, let the joint law of each $(R_j, S_j, K_j)$ be $(R, S, K)$. Then, applying \Cref{fidentityxi} yields
	\begin{flalign}
	\label{estimate2i} 
		\mathbb{E} [\mathfrak{I}] = \alpha \displaystyle\int_0^{\infty} \mathbb{E} \Bigg[ \displaystyle\frac{t^s S}{|z + t^2 R + K|^{\chi}} \Bigg] \cdot t^{-\alpha-1} dt.
	\end{flalign}
	
	\noindent By changing variables $t = R^{-1/2} \cdot |z+K|^{1/2} \cdot v$, we deduce that 
	\begin{flalign*}
		\displaystyle\int_0^{\infty} \displaystyle\frac{t^{s-\alpha-1} dt}{|z + t^2 R + K|^{\chi}} & \le R^{(\alpha-s)/2} \cdot |z + K|^{(s-\alpha)/2} \displaystyle\int_0^{\infty} \displaystyle\frac{v^{s-\alpha-1} dv}{|v^2 - 1|^{\chi}} \\
		& \le C_3 (s-\alpha)^{-1} \cdot R^{(\alpha-s)/2} \cdot |z+K|^{(s-\alpha)/2 - \chi},
	\end{flalign*}
	
	\noindent for some constant $C_3 > 1$. Inserting this into \eqref{estimate2i} gives
	\begin{flalign*}
		\mathbb{E} [\mathfrak{I}] \le C_3 \cdot \mathbb{E} \Big[ |R|^{(\alpha-s)/2} S \cdot \mathbb{E}^j \big[ |z + K|^{(s-\alpha)/2 - \chi} \big] \Big],
	\end{flalign*}
	
	\noindent which together with the second statement of \Cref{zxrkestimate} yields the third statement of the lemma.
\end{proof}

\subsection{Proof of \Cref{zxrkestimate}}

\label{ProofV}

In this section we establish \Cref{zxrkestimate}. 

\begin{proof}[Proof of \Cref{zxrkestimate}]
	
	Observe from the second part of \Cref{sqk} that there exists some constant $C_1 > 1$ such that $\mathbb{P} \big[ |K| \le C_1 \big] \ge \frac{1}{2}$. Restricting to this event, we find 
	\begin{flalign*}
	\mathbb{E} \big[ |xQ + K|^{-\chi} \big] \ge \displaystyle\frac{1}{2} \cdot (x + C_1)^{-\chi} \ge \displaystyle\frac{1}{2C_1} \cdot (x + 1)^{-\chi},
	\end{flalign*}
	
	\noindent thereby establishing the lower bound in the first statement of the lemma.
	
	The remaining parts of the lemma constitute upper bounds. Since $|Q| = 1$ and $Q, K \in \overline{\mathbb{H}}$, we have $|xQ + K| \ge |xQ + \Real K|$, and so \Cref{sqk} implies for any interval $I \subseteq \mathbb{R}_{\ge 0}$ that 
	\begin{flalign}
		\label{vichi}
		\begin{aligned}
		\mathbb{E} \big[ V^{-\chi} \cdot \one_{V \in I} \big] & \le A_1 \displaystyle\int_{-\infty}^{\infty} \displaystyle\frac{\one_{V (w) \in I}}{\big( |w| + 1 \big)^{\alpha / 2 + 1}} \cdot \displaystyle\frac{dw}{|xQ + w|^{\chi}}; \\
		\mathbb{E} [V^{-\chi}] & \le A_2 \displaystyle\int_{-\infty}^{\infty} \displaystyle\frac{\one_{|w| \ge 1}}{\big( |w| + 1 \big)^{\alpha/2 + 1}} \cdot \displaystyle\frac{dw}{|xQ + w|^{\chi}} + \displaystyle\max_{w \in [-1, 1]} |xQ + w|^{-\chi}.
		\end{aligned}
	\end{flalign}
	
	Now let us establish the first part of the lemma. We first take $I = [0, \infty)$ in the first statement of \eqref{vichi}; then, there exists a constant $C_2 = C_2 (A_1) > 1$ such that
	\begin{flalign}
		\label{vchi1} 
		\begin{aligned} 
			\mathbb{E} [V^{-\chi}] & \le 2 A_1 \Bigg( \displaystyle\int_{-\infty}^{\infty} \displaystyle\frac{\one_{|w| < 1}}{\big( |w| + 1 \big)^{\alpha/2 + 1}} \cdot \displaystyle\frac{dw}{|xQ + w|^{\chi}} + \displaystyle\int_{-\infty}^{\infty} \displaystyle\frac{\one_{1 \le |w| < x_0/2}}{\big( |w| + 1 \big)^{\alpha/2 + 1}} \cdot \displaystyle\frac{dw}{|xQ + w|^{\chi}} \\
			& \qquad \quad + \displaystyle\int_{-\infty}^{\infty} \displaystyle\frac{\one_{x_0/2 \le |w| < 2x_0}}{\big( |w| + 1 \big)^{\alpha/2 + 1}} \cdot \displaystyle\frac{dw}{|xQ + w|^{\chi}} + \displaystyle\int_{-\infty}^{\infty} \displaystyle\frac{\one_{|w| \ge 2x_0}}{\big( |w| + 1 \big)^{\alpha/2 + 1}} \cdot \displaystyle\frac{dw}{|xQ + w|^{\chi}} \Bigg) \\
			& \le C_2 \Bigg( \displaystyle\int_{-1}^1 \displaystyle\frac{dw}{|xQ+w|^{\chi}} + (x + 1)^{-\chi} \displaystyle\int_1^{x_0/2} \displaystyle\frac{dw}{(w + 1)^{\alpha/2 + 1}} \\
			& \qquad \quad + (x + 1)^{-\alpha/2 - 1} \displaystyle\int_{-\infty}^{\infty} \one_{x_0/2 \le |w| < 2x_0} \cdot \displaystyle\frac{dw}{|xQ+w|^{\chi}} + \displaystyle\int_{2x_0}^{\infty} \displaystyle\frac{dx}{w^{\alpha/2 + \chi + 1}}\Bigg),
		\end{aligned} 
	\end{flalign} 
	
	\noindent where $x_0 = x + \frac{1}{2}$. Here, for the first integral we used the fact that $|w| + 1 \ge 1$; for the second we used the fact that $|xQ + w| \ge \frac{x_0}{4} \ge \frac{x+1}{8}$ whenever $1 \le |w| \le \frac{x_0}{2}$; for the third we used the fact that $|w| + 1 \ge \frac{x_0}{2}$ on its support; and for the fourth we used the fact that $|w| \ge 2x_0 \ge 2x$ implies $|x+w| \ge \frac{|w|}{2}$. Since there exists a constant $C_3 > 1$ such that 
	\begin{flalign*}
		& \displaystyle\int_{-1}^1 \displaystyle\frac{dw}{|xQ + w|^{\chi}} \le C_3 (x + 1)^{-\chi}; \qquad \qquad \qquad \qquad \quad \displaystyle\int_1^{x_0/2} \displaystyle\frac{dw}{(w + 1)^{\alpha/2 + 1}} \le C_3; \\
		& \displaystyle\int_{-\infty}^{\infty} \one_{x_0/2 \le |w| < 2x_0} \cdot \displaystyle\frac{dw}{|xQ + w|^{\chi}} \le C_3 (x + 1)^{1 - \chi}; \qquad  \displaystyle\int_{2x_0}^{\infty} w^{-\alpha/2 - \chi - 1} \le C_3 (x + 1)^{-\alpha/2 - \chi},
	\end{flalign*}
	
	\noindent inserting these bounds into \eqref{vchi1} yields the first statement of the lemma. The proof of the second statement is entirely analogous, using the second part of \eqref{vichi} and the fact that $|xQ + w| \ge \frac{1}{8} \cdot \big( |x| + 1 \big)$ for $|x \Real Q| \ge 2$ and $|w| \le 1$ (to  bound the second term on its right side). 
	
	We next establish the third statement of the lemma, to which end we take $I = [0, \delta]$. Then, since $|Q| = 1$ and $\delta \le 1$, we have that $|w| + 1 \ge \frac{1}{2} (x + 1)$ whenever $V(w) = |xQ + w| \le \delta$. Hence,
	\begin{flalign*}
		\mathbb{E} \big[ V^{-\chi} \cdot \one_{V \le \delta} \big] & \le 4 A_1 \cdot (x + 1)^{-\alpha / 2 - 1} \displaystyle\int_{-\infty}^{\infty} \one_{V(w) \le \delta} \cdot \displaystyle\frac{dw}{|x \Real Q + w|^{\chi}} \\
		& \le 8 A_1 \cdot (x + 1)^{-\alpha / 2 - 1} \displaystyle\int_0^{\delta} u^{-\chi} du = \displaystyle\frac{8 A_1}{1 - \chi} \cdot (x + 1)^{-\alpha / 2 - 1} \cdot \delta^{1 - \chi},
	\end{flalign*}
	
	\noindent where to deduce the second inequality we changed variables $u = x \Real Q + w$. This establishes the third bound in the lemma. 
	
	To establish the fourth, observe from \eqref{vichi} that 
	\begin{flalign*}
		\mathbb{E} \big[ V^{-\chi} \cdot \one_{\Real K \in J} \big] \le A_1 \displaystyle\int_J \displaystyle\frac{1}{\big( |w| + 1 \big)^{\alpha/2 + 1}} \cdot \displaystyle\frac{dw}{|xQ + w|^{\chi}}.
	\end{flalign*}
	
	\noindent Since $|Q| = 1$, we either have $|xQ + w| \ge \frac{1}{2} (x + 1)$ or $|w| + 1 \ge \frac{1}{2} (x + 1)$. It follows for some constant $C_4 = C_4 (A_1) > 1$ that 
	\begin{flalign*} 
		\mathbb{E} \big[ V^{-\chi} \cdot \one_{\Real K \in J} \big] & \le 2 A_1 \int_{-\delta}^{\delta} \displaystyle\frac{dw}{(x + 1)^{\chi}} + 2 A_1 \cdot (x+1)^{-\alpha/2 - 1} \displaystyle\int_{-\delta}^{\delta} w^{-\chi} dw \\
		& \le C_4 \delta \cdot (x + 1)^{-\chi} + C_4 \delta^{1-\chi} \cdot (x + 1)^{-\alpha/2-1},
	\end{flalign*} 
	
	\noindent where we used the fact that $J$ has length at most $\delta$ (enabling us to replace it by $[-\delta, \delta]$ to upper bound the integrals). This verifies the fourth statement of the lemma.
\end{proof}

	\section{Fractional Moment Estimates}
	
	\label{EstimateMomentss}
	
	In this section we establish \Cref{limitr0j}. We begin in \Cref{EstimateR0} by establishing recursive estimates for a variant of the moment $\Phi_{\ell} (s; z)$ (from \Cref{moment1}), which we refer to by $\Xi_{\ell} (s; z)$ (see \Cref{xil} below). In \Cref{MultiplicativeR} establish \Cref{limitr0j} by showing that the $\Phi_{\ell} (s; z)$ are approximately multiplicative in $\ell$. Throughout this section, we fix a complex number $z = E + \mathrm{i} \eta \in \mathbb{H}$ with $\eta \in (0, 1)$ and $\varepsilon \le |E| \le B$; we will frequently abbreviate $\Phi_{\ell} (s) = \Phi_{\ell} (s; z)$. We further fix $s \in (\alpha, 1)$ and $\chi \in \big( \frac{s-\alpha}{2}, 1 \big)$ satisfying \eqref{chisalphaomega} for some $\varpi > 0$. In what follows, constants might implicitly depend on $\varpi$ (and $\alpha$), and we will mention how they depend on $\varepsilon$, $B$, and $s$.

\subsection{Resolvent Estimates}

\label{EstimateR0} 

In this section we establish estimates for the following variant $\Xi_{\ell}$ of the fractional moment $\Phi_{\ell}$.

\begin{definition}

\label{xil} 

For any integer $\ell > 0$, define $\Xi_\ell (\chi) = \Xi_\ell (\chi; s; z)$ and $\Xi_{-\ell} (\chi) = \Xi_{-\ell} (\chi; s; z)$ by 
\begin{flalign*}
	 \Xi_\ell (\chi) = \mathbb{E} \Bigg[ \displaystyle\sum_{v \in \mathbb{V}(\ell)} |R_{00}|^{\chi} \cdot |T_{00_+}|^s \cdot \big| R_{0_+ v}^{(0)} \big|^s \Bigg]; \qquad \Xi_{-\ell} (\chi) = \mathbb{E} \Bigg[ \displaystyle\sum_{v \in \mathbb{V} (\ell)} |R_{vv}|^{\chi} \cdot |T_{v_- v}|^s \cdot \big| R_{0v_-}^{(v)} \big|^s \Bigg],
\end{flalign*}

\noindent where $0_+$ is the child of $0$ on the path $\mathfrak{p} (0, v) \subset \mathbb{V}$ from $0$ to $v$. At $\ell = 0$, we set $\Xi_0 (\chi) = \Xi_0 (\chi; s; z) = \mathbb{E} \big[ |R_{00}|^{\chi} \big]$. 

\end{definition} 

\begin{rem}

\label{xir}

By \Cref{rproduct}, we have that 
\begin{flalign}
	\label{rproduct2} 
| R_{00} | \cdot |T_{00_+}| \cdot \big| R_{0_+ v}^{(0)} \big| = | R_{0v}| = |R_{vv}| \cdot |T_{v_- v}| \cdot \big| R_{0v_-}^{(v)} \big|,
\end{flalign} 

\noindent so $\Xi_{\ell} (s) = \Xi_{-\ell} (s) = \Phi_{\ell} (s)$.

\end{rem} 


	The following lemma shows that $\Xi_{\ell} (\chi) = \Xi_{-\ell} (\chi)$ holds for general $\chi$. 
	
	\begin{lem} 
		
		\label{xichi} 
		
		For any integer $\ell \ge 0$, we have $\Xi_{\ell} (\chi) = \Xi_{-\ell} (\chi)$. 

	\end{lem} 
	
	\begin{proof}
		
		We may assume that $\ell \ge 1$. For any $u, v \in \mathbb{V}$, set
		\begin{flalign*}
			f(u, v) = |R_{uu}|^{\chi} \cdot |T_{uw}|^s \cdot \big| R_{w v}^{(u)} \big|^s \cdot \one_{d(v, u) = \ell},
		\end{flalign*}
	
		\noindent where $d(u, v)$ denotes the distance between $u$ and $v$ in $\mathbb{T}$, and $w \sim u$ is the vertex adjacent to $u$ in the path $\mathfrak{p} (u, v) \subset \mathbb{V}$ between $u$ and $v$. Then, the unimodularity of $\mathbb{T}$ (recall \Cref{tunimodular}) gives 
		\begin{flalign*}
			\Xi_{\ell} (\chi) = \mathbb{E} \Bigg[ \displaystyle\sum_{v \in \mathbb{T}} f(0, v) \Bigg] = \mathbb{E} \Bigg[ \displaystyle\sum_{v \in \mathbb{T}} f(v, 0) \Bigg] = \Xi_{-\ell} (\chi),
		\end{flalign*}
		
		\noindent which implies the lemma.
	\end{proof}	

	Next, we have the following two lemmas showing that $\Xi_{\ell} (\chi)$ and $\Xi_{\ell} \big( \frac{\alpha+s}{2} \big)$ are comparable.

\begin{lem} 
	
	\label{chilchil1}
	
	There exists a constant $C_1 = C_1 (\varepsilon, B) > 1$ such that for any integer $\ell \ge 1$ we have 
	\begin{flalign*} 
		C_1^{-1} (s-\alpha)^{-1} \cdot \Xi_{\ell-1} \bigg( \displaystyle\frac{\alpha+s}{2} \bigg) \le \Xi_\ell (\chi) \le C_1 (s-\alpha)^{-1} \cdot \Xi_{\ell-1} \bigg( \displaystyle\frac{\alpha+s}{2} \bigg). 
	\end{flalign*} 

	\noindent Additionally, if $|E| \ge 2$, then there exists a constant $C_2 > 1$ (independent of $\varepsilon$ and $B$) such that for any integer $\ell \ge 1$ we have 
	\begin{flalign*}
		 \Xi_{\ell} (\chi) \le C_2 (s-\alpha)^{-1} \cdot E^{(s-\alpha)/2 - \chi} \cdot \Xi_{\ell-1} \bigg( \displaystyle\frac{\alpha+s}{2} \bigg).
	\end{flalign*}
\end{lem} 

\begin{proof} 
	
Applying the Schur complement identity \eqref{qvv} (and setting $u = 0_+$), we deduce
\begin{flalign}
	\label{sestimate}
	\Xi_\ell (\chi) & = \mathbb{E} \Bigg[ \displaystyle\sum_{u \in \mathbb{V} (1)} \displaystyle\frac{|T_{0u}|^s}{\big| z + T_{0u}^2 R_{uu}^{(0)} + K_u \big|^{\chi}} \cdot \displaystyle\sum_{v \in \mathbb{D}_{\ell - 1} (u)} \big| R_{uv}^{(0)} \big|^s \Bigg], \quad \text{where} \quad  K_u = \displaystyle\sum_{\substack{w \in \mathbb{V} (1) \\ w \ne u}} T_{0w}^2 R_{ww}^{(0)}.
\end{flalign}

\noindent Observe from \Cref{tuvalpha} and \Cref{sumaxi} that $K_u$ is a stable random variable with some law $K$, and that it is independent from $T_{0u}$, $R_{uu}^{(0)}$, and $R_{uv}^{(0)}$. Applying the first part of \Cref{expectationsum1}, with the $R_j$ there equal to $R_{uu}^{(0)}$ here and the $S_j$ there equal to the inner sum $\sum_{v \in \mathbb{D}_{\ell-1} (u)} \big| R_{0v}^{(u)} \big|^s$ on the right side of \eqref{sestimate} (and using \Cref{tuvalpha}, \Cref{sumaxi}, the first statement of \Cref{q0estimate1}, and \eqref{sigmathetakappa} to verify \Cref{sqk} for $K_u$), yields a constant $C = C (\varepsilon, B) > 1$ such that 
	\begin{flalign*} 
		C^{-1} \cdot \Xi_{\ell-1} \bigg( \displaystyle\frac{\alpha+s}{2} \bigg) & = C^{-1} \cdot \mathbb{E} \Bigg[ \displaystyle\sum_{v \in \mathbb{D}_{\ell-1} (u)} \big| R_{uu}^{(0)} \big|^{(\alpha+s)/2} \cdot |T_{uu_+}|^s \cdot \big| R_{u_+ v}^{(u)} \big|^s \Bigg] \\
		& = C^{-1} \cdot \mathbb{E} \Bigg[ \displaystyle\sum_{v \in \mathbb{D}_{\ell-1} (u)} \big| R_{uu}^{(0)} \big|^{(\alpha-s)/2} \cdot \big| R_{uv}^{(0)} \big|^s \Bigg] \\
		& \le (s-\alpha) \cdot \Xi_\ell (\chi) \\
		& \le C \cdot \mathbb{E} \Bigg[ \displaystyle\sum_{v \in \mathbb{D}_{\ell - 1} (u)} \big| R_{uu}^{(0)} \big|^{(\alpha-s)/2} \cdot \big| R_{uv}^{(0)} \big|^s \Bigg] \\
		& = C \cdot \mathbb{E} \Bigg[ \displaystyle\sum_{v \in \mathbb{D}_{\ell-1} (u)} \big| R_{uu}^{(0)} \big|^{(\alpha+s)/2} \cdot |T_{uu_+}|^s \cdot \big| R_{u_+ v}^{(u)} \big|^s \Bigg] = C \cdot \Xi_{\ell-1} \bigg( \displaystyle\frac{\alpha+s}{2} \bigg).
	\end{flalign*}

	\noindent Here, to obtain the first and sixth statements we used the fact that $\big( R_{u_+ v}^{(u)}, T_{uu_+} \big)$ and $(R_{0w}, T_{00_+})$ have the same law, for $u \in \mathbb{V} (1)$, $v \in \mathbb{D}_{\ell-1} (u)$, and $w \in \mathbb{V} (\ell-1)$ suffix of $v$ (that is $v = (u,w)$); to obtain the second and fifth, we used \Cref{rproduct} (see also \eqref{rproduct2}). This establishes the first statement of the lemma.
	
	The proof of the second is entirely analogous and thus omitted, obtained by replacing the use of the first part of \Cref{expectationsum1} to estimate \eqref{sestimate} with the third part of \Cref{expectationsum1} (using \Cref{tuvalpha}, \Cref{sumaxi}, and the first statement of \Cref{q0estimate1} to verify \Cref{sqk} for $K_u$, with a constant $A_2$ that is independent of $\varepsilon$ and $B$). 
	\end{proof}

	\begin{lem} 
		
		\label{xilchi2}
		
		There exists a constant $C = C(s, \varepsilon, B) > 1$ such that for any integer $\ell > 0$ we have 		
		\begin{flalign*} 
			C^{-1} \cdot \Xi_{-\ell} \bigg( \displaystyle\frac{\alpha+s}{2} \bigg) \le \Xi_{-\ell - 1} (\chi) \le C \cdot \Xi_{-\ell} \bigg( \displaystyle\frac{\alpha+s}{2} \bigg). 
		\end{flalign*} 
	\end{lem} 
	
	\begin{proof} 
		
		This follows from \Cref{xichi} and \Cref{chilchil1}, but let us also provide an alternative proof, similar to that of \Cref{chilchil1}. Again applying the Schur complement identity \eqref{qvv} yields
		\begin{flalign*}
			\Xi_{-\ell} (\chi) = \mathbb{E} \Bigg[ \displaystyle\sum_{u \in \mathbb{V} (\ell - 1)} \displaystyle\sum_{v \in \mathbb{D} (u)} \displaystyle\frac{|T_{uv}|^s}{\big| z + T_{uv}^2 R_{uu}^{(v)} + K_v \big|^{\chi}} \cdot \big| R_{0u}^{(v)} \big|^s \Bigg], \quad \text{where} \quad K_v = \displaystyle\sum_{w \in \mathbb{D} (v)} T_{vw}^2 R_{ww}^{(v)}.
		\end{flalign*} 
	
		\noindent In particular, by \Cref{tuvalpha} and \Cref{sumaxi}, $K_v$ is an $\frac{\alpha}{2}$-stable law $K$. Since $K_v$ is also independent from $T_{uv}$ and $R_{uu}^{(v)}$, we have 
		 \begin{flalign}
		 	\label{xi1}
		 	\Xi_{-\ell} (\chi) = \mathbb{E} \Bigg[ \displaystyle\sum_{u \in \mathbb{V} (\ell - 1)} |T_{u_- u}|^s \cdot \big| R_{0 u_-}^{(u)} \big|^s \cdot \mathbb{E}_u \bigg[ \displaystyle\sum_{v \in \mathbb{D}(u)} \displaystyle\frac{|T_{uv}|^s}{\big| z + T_{uv}^2 R_{uu}^{(v)} + K \big|^{\chi}} \cdot \big| R_{uu}^{(v)} \big|^s  \bigg] \Bigg],
		 \end{flalign}
		
		\noindent where we used \Cref{rproduct} and recalled that $\mathbb{E}_u$ denotes the expectation conditional on $\mathbb{T}_- (u)$. Also observe by \eqref{qvv} that 
		\begin{flalign*}
			R_{uu}^{(v)} = - \big( z + T_{u_- u}^2 R_{u_- u_-}^{(u)} + g_v \big)^{-1}, \quad \text{where} \quad g_v = \displaystyle\sum_{\substack{w \in \mathbb{D} (u) \\ w \ne v}} T_{uw}^2 R_{ww}^{(u)}.
		\end{flalign*}
	
		 Now let $(R_w)_{w \sim u}$ denote a family of independent, identically distributed random variables, each with law $R_{00}$. Then the first part of \Cref{q12} implies, for any $u \in \mathbb{V} (\ell - 1)$, the law of the sequence of random variables $\big( R_{ww}^{(u)} \big)$ is the same as that of $(R_w)$. Hence, exchanging the expectation with the sum in \eqref{xi1} twice, we obtain 
		\begin{flalign}
			\label{xil2} 
			\begin{aligned}
			\Xi_{-\ell} (\chi) & = \mathbb{E} \Bigg[ \displaystyle\sum_{u \in \mathbb{V} (\ell-1)} |T_{u_- u}|^s \cdot \big| R_{0u_-}^{(u)} \big|^s \cdot \displaystyle\sum_{v \in \mathbb{D} (u)} \mathbb{E}_u \bigg[  \displaystyle\frac{|T_{uv}|^s}{\big| z + T_{uv}^2 R_{uu}^{(v)} + K \big|^{\chi} } \cdot \big| R_{uu}^{(v)} \big|^s \bigg] \Bigg] \\
			& = \mathbb{E} \Bigg[ \displaystyle\sum_{u \in \mathbb{V} (\ell-1)} |T_{u_- u}|^s \cdot \big| R_{0u_-}^{(u)} \big|^s \cdot \displaystyle\sum_{v \in \mathbb{D} (u)} \mathbb{E}_u \bigg[  \displaystyle\frac{|T_{uv}|^s}{\big| z + T_{uv}^2 \widetilde{R}_{uu}^{(v)} + K \big|^{\chi} } \cdot \big| \widetilde{R}_{uu}^{(v)} \big|^s \bigg] \Bigg] \\ 
			& = \mathbb{E} \Bigg[ \displaystyle\sum_{u \in \mathbb{V} (\ell-1)} |T_{u_- u}|^s \cdot \big| R_{0u_-}^{(u)} \big|^s \cdot \mathbb{E}_u \bigg[ \displaystyle\sum_{v \in \mathbb{D} (u)}   \displaystyle\frac{|T_{uv}|^s}{\big| z + T_{uv}^2 \widetilde{R}_{uu}^{(v)} + K \big|^{\chi} } \cdot \big| \widetilde{R}_{uu}^{(v)} \big|^s \bigg] \Bigg],
			\end{aligned} 
		\end{flalign} 
	
		\noindent where
		\begin{flalign*}
			\widetilde{R}_{uu}^{(v)} = - \big( z + T_{u_- u}^2 R_{u_- u_-}^{(u)} + \widetilde{g}_v \big)^{-1}, \quad \text{where} \quad \widetilde{g}_v = \displaystyle\sum_{\substack{w \in \mathbb{D} (u) \\ w \ne v}} T_{vw}^2 R_w.
		\end{flalign*}
	
		\noindent We next us apply \eqref{fxi} with 
		\begin{flalign*}
			f(x, \nu) = \displaystyle\frac{|x|^s}{\big|z + x^2 \widetilde{R} (\nu) + K \big|^{\chi}} \cdot \big| \widetilde{R} (\nu) \big|^s, 
		\end{flalign*} 
	
		\noindent where $\nu = (\nu_1, \nu_2, \ldots ) \in \mathbb{R}_{> 0}$ is a point process, and
		\begin{flalign*} 
			\widetilde{R} (\nu) = - \big( z + T_{u_- u}^2 \widetilde{R}_{u_- u_-}^{(u)} + \widetilde{g} (\nu) \big)^{-1}, \quad \text{and} \quad \widetilde{g} (\nu) = \displaystyle\sum_{j = 1}^{\infty} \nu_j^2 R_j,
		\end{flalign*}
	
		\noindent where $R_1, R_2, \ldots $ are independent, identically distributed random variables with law $R_{00}$. Observe in particular from the Schur complement identity \eqref{qvv} that $\widetilde{R} (\nu)$ has the same law as $R_{uu}$, conditional on $(T_{uw})_{w \in \mathbb{D}(u)} = (\nu_1, \nu_2, \ldots)$. Inserting \eqref{fxi} into \eqref{xil2}, we obtain
		\begin{flalign}
			\label{xil3}
			\begin{aligned}
			\Xi_{-\ell} (\chi) & = \alpha \cdot \mathbb{E} \Bigg[ \displaystyle\sum_{u \in \mathbb{V} (\ell - 1)} |T_{u_- u}|^s \cdot \big| R_{0u_-}^{(u)} \big|^s \cdot \displaystyle\int_0^{\infty} \mathbb{E}_u \bigg[ \displaystyle\frac{x^s}{\big| z + x^2 \widetilde{R} (\nu) + K \big|^{\chi}} \cdot \big| \widetilde{R} (\nu) \big|^s \bigg] \cdot x^{-\alpha-1} dx \Bigg] \\
			& = \alpha \cdot \mathbb{E} \Bigg[ \displaystyle\sum_{u \in \mathbb{V} (\ell - 1)} |T_{u_- u}|^s \cdot \big|R_{0u_-}^{(u)} \big|^s \cdot  \displaystyle\int_0^{\infty} \mathbb{E}_u \Bigg[ \displaystyle\frac{x^s}{|z + x^2 R_{uu} + K|^{\chi}} \cdot |R_{uu}|^s \Bigg] \cdot x^{-\alpha-1} dx \Bigg],
			\end{aligned} 
		\end{flalign}
	
		\noindent where in the first expectation $\nu$ is a Poisson point process with intensity measure $\alpha x^{-\alpha-1} dx$. By \Cref{expectationsum1}, there exists a constant $C (s, \varepsilon, B) > 1$ for which  
		\begin{flalign*}
			C^{-1} \cdot \mathbb{E}_u \big[ |R_{uu}|^{(\alpha+s)/2} \big] & \le \displaystyle\int_0^{\infty} \mathbb{E}_u \Bigg[ \displaystyle\frac{x^s}{ |z + x^2 R_{uu} + K|^{\chi}} \cdot |R_{uu}|^s \Bigg] \cdot x^{-\alpha - 1} dx \le C \cdot \mathbb{E}_u \big[ |R_{uu}|^{(\alpha+s)/2} \big].
		\end{flalign*}
	
		\noindent This, with \eqref{xil3}, then yields the lemma. 
	\end{proof}

	\begin{cor}
		
		\label{estimatemoment1}
		
		There exist constants $C_1 = C_1 (s, \varepsilon, B) > 1$; $C_2 = C_2 (s) > 1$ (independent of $\varepsilon$ and $B)$; $c_1 = c_1 (\varepsilon, B)$ (independent of $s)$; and $c_2 > 0$ such that the following hold for any integer $\ell \ge 2$.
		
		\begin{enumerate}
			\item  We have $C_1^{-1} \cdot \Phi_{\ell-1} (s) \le \Phi_{\ell} (s) \le C_1 \cdot \Phi_{\ell-1} (s)$.
			\item We have $C_1^{-1} c_1^{\ell} \cdot (s-\alpha)^{-\ell} \le \Phi_{\ell} (s) \le C_1 c_1^{-\ell} \cdot (s-\alpha)^{-\ell}$.
			\item If $|E| \ge 2$, then  $\Phi_{\ell} (s) \le C_2 c_2^{-\ell} \cdot (s- \alpha)^{-\ell} \left(|E|^{(s-\alpha)/2 - \chi}\right)^{\ell}$. 
		\end{enumerate}
	\end{cor}

	\begin{proof}
		
		By \Cref{xir} and three applications of the first statement of \Cref{chilchil1} (twice at $\chi = s$ and once at $\chi = \frac{\alpha+s}{2}$), we deduce the existence of a constant $C_1 = C_1 (s, \varepsilon, B) > 1$ such that 
		\begin{flalign*}
			\Phi_{\ell} (s) =\Xi_{\ell} (s) \le C_1 \cdot \Xi_{\ell-1} \bigg( \displaystyle\frac{\alpha+s}{2} \bigg) \le C_1^2 \cdot \Xi_{\ell-2} \bigg( \displaystyle\frac{\alpha+s}{2} \bigg) \le C_1^3 \cdot \Xi_{\ell-1} (s) = C_1^3 \cdot \Phi_{\ell-1} (s),
		\end{flalign*}
	
		\noindent and so $\Phi_{\ell} (s) \le C_1^3 \cdot \Phi_{\ell-1} (s)$; entirely analogous reasoning yields $\Phi_{\ell} (s) \ge C_1^{-3} \cdot \Phi_{\ell-1} (s)$. This establishes the first statement of the lemma. To establish the second, observe from $\ell$ applications of the first part of \Cref{chilchil1} that there exists a constant $c = c(\varepsilon, B) > 0$ such that
		\begin{flalign*}
			c^{\ell} \cdot (s-\alpha)^{-\ell} \cdot \Xi_0 \bigg( \displaystyle\frac{\alpha+s}{2} \bigg) \le \Phi_{\ell} (s) \le c^{-\ell} \cdot (s-\alpha)^{-\ell} \cdot \Xi_0 \bigg( \displaystyle\frac{\alpha+s}{2}\bigg).
		\end{flalign*} 
	
		\noindent This, together with the fact that $\Xi_0 (\chi)$ is bounded for any $\chi \in (0, 1)$ (by \Cref{expectationqchi}), yields the second statement of the lemma. The proof of the third is very similar, obtained by repeatedly applying the second part of \Cref{chilchil1} (instead of the first). 
	\end{proof}

	\subsection{Approximate Multiplicativity of Resolvent Moments} 
	
	\label{MultiplicativeR} 
	
	In this section we establish \Cref{limitr0j}, to which end we will first show that the resolvent moments $\Phi_{\ell} (s)$ (from \Cref{moment1}) are approximately multiplicative through the following proposition. Throughout this section, all constants may implicitly depend on $s$, $\varepsilon$, and $B$, even when not explicitly stated.
	
	\begin{prop} 
		
	\label{smultiplicative} 
	
	There is a constant $C = C(s, \varepsilon, B) > 1$ such that, for any integers $\ell, m \ge 1$, we have
	\begin{flalign*} 
		C^{-1} \cdot \Phi_{\ell-1} (s) \cdot \Phi_{m-1} (s) \le \Phi_{\ell + m} (s) \le C \cdot \Phi_{\ell-1} (s) \cdot \Phi_{m-1} (s).
	\end{flalign*} 
	
	\end{prop} 

	The above proposition will quickly follow from the following two lemmas.

	\begin{lem}
	
	\label{lm1} 
	
	There is a constant $C = C(s, \varepsilon, B) > 1$ such that, for any integers $\ell, m \ge 1$, we have 
	\begin{flalign*}
		C^{-1} & \cdot \Phi_{\ell-1} (s) \cdot \mathbb{E} \Bigg[ \displaystyle\sum_{u \in \mathbb{V} (m)} \big| R_{0u_-}^{(u)} \big|^s \cdot |T_{u_- u}|^s \cdot \Big( \big| z + T_{u_- u}^2 R_{u_- u_-}^{(u)} \big| + 1 \Big)^{-s} \Bigg] \\
		& \le \Phi_{\ell + m} (s) \le C \cdot \Phi_{\ell-1} (s) \cdot \mathbb{E} \Bigg[ \displaystyle\sum_{u \in \mathbb{V} (m)} \big| R_{0u_-}^{(u)} \big|^s \cdot |T_{u_- u}|^s \cdot \Big( \big| z + T_{u_- u}^2 R_{u_- u_-}^{(u)} \big| + 1 \Big)^{-s} \Bigg].
	\end{flalign*} 
	\end{lem} 

	\begin{lem}
	
	\label{m2} 
	
	There is a constant $C = C(s, \varepsilon, B) > 1$ such that, for any integer $m \ge 1$, we have 
	\begin{flalign*} 
		C^{-1} \cdot \Phi_{m-1} (s) \le \mathbb{E} \Bigg[ \displaystyle\sum_{u \in \mathbb{V} (m)} \big| R_{0u_-}^{(u)} \big|^s \cdot |T_{u_- u}|^s \cdot \Big( \big| z + T_{u_- u}^2 R_{u_- u_-}^{(u)} \big| + 1 \Big)^{-s} \Bigg] \le C \cdot \Phi_{m-1} (s).
	\end{flalign*}
	\end{lem}

	\begin{proof}[Proof of \Cref{smultiplicative}]
		
		This follows from \Cref{lm1} and \Cref{m2}.
	\end{proof} 

	Now let us establish \Cref{lm1} and \Cref{m2}.

	\begin{proof}[Proof of \Cref{lm1}] 

	We only establish the upper bound, as the proof of the lower bound is entirely analogous. To that end, by \Cref{rproduct}, we have 
	\begin{flalign*} 
		\Phi_{\ell + m} (s) & = \mathbb{E} \Bigg[ \displaystyle\sum_{u \in \mathbb{V} (m)} \displaystyle\sum_{v \in \mathbb{D}_{\ell} (u)} |R_{0v}|^s \Bigg] \\
		& = \mathbb{E} \Bigg[ \displaystyle\sum_{u \in \mathbb{V}(m)} \big| R_{0u_-}^{(u)} \big|^s \cdot |T_{u_- u}|^s \cdot | R_{uu}|^s \displaystyle\sum_{w \in \mathbb{D} (u)} |T_{uw}|^s \displaystyle\sum_{v \in \mathbb{D}_{\ell-1} (w)} \big| R_{wv}^{(u)} \big|^s \Bigg],
	\end{flalign*} 

	\noindent from which it follows (by the Schur complement identity \eqref{qvv}) that 
	\begin{flalign}
		\label{slm} 
		\begin{aligned} 
		\Phi_{\ell+m} (s) = \mathbb{E} \Bigg[ \displaystyle\sum_{u \in \mathbb{V} (m)} \big| R_{0u_-}^{(u)} \big|^s \cdot |T_{u_- u}|^s \cdot \mathbb{E}_u \bigg[ & \displaystyle\sum_{w \in \mathbb{D} (u)} \displaystyle\frac{|T_{uw}|^s}{\big| z + T_{u_- u}^2 R_{u_- u_-}^{(u)} + T_{uw}^2 R_{ww}^{(u)} + K_{u, w} \big|^{\chi}} \\
		&  \times \displaystyle\sum_{v \in \mathbb{D}_{\ell - 1} (w)} \big| R_{wv}^{(u)} \big|^s \bigg] \Bigg], 
		\end{aligned} 
	\end{flalign} 

	\noindent where we recall that  $\mathbb{E}_u$ denotes the expectation conditional on $\mathbb{T}_- (u)$ and we have denoted
	\begin{flalign*} 
		K_{u, w} = \displaystyle\sum_{\substack{v' \in \mathbb{D} (u) \\ v' \ne w, u_- }} T_{uv'}^2 R_{v' v'}^{(u)}. 
	\end{flalign*} 

	\noindent To estimate the inner expectation in \eqref{slm}, we proceed as in the proof of \Cref{xilchi2}. Specifically, observe that $K_{u, w}$ is an $\frac{\alpha}{2}$-stable random variable (by \Cref{tuvalpha} and \Cref{sumaxi}) with some law $K$ that is independent from $T_{u_- u}$, $T_{uw}$, $R_{u_- u_-}^{(u)}$, $R_{ww}^{(u)}$, and $R_{wv}^{(u)}$. We then apply \Cref{expectationsum1}, with the $z$ there equal to $z + T_{u_- u}^2 R_{u_- u_-}^{(u)}$ here; the $S_j$ there equal to $\sum_{v \in \mathbb{D}_{\ell-1} (w)} \big| R_{wv}^{(u)} \big|^s$; and the $K_j$ there equal to $K$ here (whose conditions are verified by \Cref{tuvalpha}, \Cref{zuv}, and \Cref{ar}). This yields a constant $C_1 > 1$ such that 
	\begin{flalign}
		\label{expectationtu} 
		\begin{aligned} 
		\mathbb{E}_u & \Bigg[ \displaystyle\sum_{w \in \mathbb{D} (u)} \displaystyle\frac{|T_{uw}|^s}{\big| z + T_{u_- u}^2 R_{u_- u_-}^{(u)} + T_{uw}^2 R_{ww}^{(u)} + K_{u, w}\big|^s} \displaystyle\sum_{v \in \mathbb{D}_{\ell-1} (w)} \big| R_{wv}^{(u)} \big|^s \Bigg] \\ 
		  & \le \displaystyle\frac{C_1}{\Big( \big| z + T_{u_- u}^2 R_{u_- u_-}^{(u)} \big| + 1 \Big)^s} \cdot \mathbb{E}_u \Bigg[  \big| R_{u_+ u_+}^{(u)} \big|^{(\alpha-s)/2} \displaystyle\sum_{v \in \mathbb{D}_{\ell-1} (u_+)} \big| R_{u_+ v}^{(u)} \big|^s \Bigg] \\
		  & = \displaystyle\frac{C_1}{\Big( \big| z + T_{u_- u}^2 R_{u_- u_-}^{(u)} \big| + 1 \Big)^s} \cdot \mathbb{E} \Bigg[ | R_{00}|^{(\alpha-s)/2} \displaystyle\sum_{v \in \mathbb{V} (\ell-1)} | R_{0v}|^s \Bigg],
		  \end{aligned} 
	\end{flalign}

	\noindent where in the second statement we  used the fact that $\big( R_{u_+ u_+}^{(u)}, R_{u_+ v}^{(u)} \big)$ for $w \in \mathbb{D} (u)$ and $v \in \mathbb{D}_{\ell-1} (u_+)$ has the same law as $(R_{00}, R_{0v})$ for $v \in \mathbb{V} (\ell-1)$. Hence,
	\begin{flalign}
		\label{expectationtu2} 
		\begin{aligned} 
	\mathbb{E}_u & \Bigg[ \displaystyle\sum_{w \in \mathbb{D} (u)} \displaystyle\frac{|T_{uw}|^s}{\big| z + T_{u_- u}^2 R_{u_- u_-}^{(u)} + T_{uw}^2 R_{ww}^{(u)} + K_{u, w} \big|^s} \displaystyle\sum_{v \in \mathbb{D}_{\ell-1} (w)} \big| R_{wv}^{(u)} \big|^s \Bigg] \\ 
		  & \le \displaystyle\frac{C_1}{\Big( \big| z + T_{u_- u}^2 R_{u_- u_-}^{(u)} \big| + 1 \Big)^s} \cdot \mathbb{E} \Bigg[ \displaystyle\sum_{v \in \mathbb{V} (\ell-1)}| R_{00}|^{(\alpha+s)/2} \cdot |T_{00_+}|^s \cdot | R_{0_+ v}|^s \Bigg] \\
		  & = \displaystyle\frac{C_1}{\Big( \big| z + T_{u_- u}^2 R_{u_- u_-}^{(u)} \big| + 1 \Big)^s} \cdot \Xi_{\ell-1} \bigg( \displaystyle\frac{\alpha+s}{2} \bigg),
		 \end{aligned} 
	\end{flalign}

	\noindent where in the first statement we applied \Cref{rproduct}. Inserting \eqref{expectationtu} and \eqref{expectationtu2} into \eqref{slm} yields
	\begin{flalign*}
		\Xi_{\ell+m} (s) \le C_1 \cdot \Xi_{\ell-1} \bigg( \displaystyle\frac{\alpha+s}{2} \bigg) \cdot \mathbb{E} \Bigg[ \displaystyle\sum_{u \in \mathbb{V} (m)} \big| R_{0u_-}^{(u)} \big|^s \cdot |T_{u_- u}|^s \cdot \Big( \big| z + T_{u_- u}^2 R_{u_- u_-}^{(u)} \big| + 1 \Big)^{-s} \Bigg],
	\end{flalign*} 

	\noindent and so the upper bound in the lemma follows from \Cref{chilchil1} and \Cref{estimatemoment1}; as mentioned previously, the proof of the lower bound is entirely analogous and is thus omitted.
	\end{proof} 

	\begin{proof}[Proof of \Cref{m2}] 

	By \Cref{rproduct}, we have 
	\begin{flalign}
		\label{expectationsumvm} 
		\begin{aligned} 
		\mathbb{E} & \Bigg[ \displaystyle\sum_{u \in \mathbb{V} (m)} \big| R_{0u_-}^{(u)} \big|^s \cdot |T_{u_- u}|^s \cdot \Big( \big| z + T_{u_- u}^2 R_{u_- u_-}^{(u)} \big| + 1 \Big)^{-s} \Bigg] \\
		& = \mathbb{E} \Bigg[ \displaystyle\sum_{w \in \mathbb{V} (m-1)} \displaystyle\sum_{u \in \mathbb{D} (w)} \big| R_{0w}^{(u)} \big|^s \cdot |T_{w u}|^s \cdot \Big( \big| z + T_{w u}^2 R_{ww}^{(u)} \big| + 1 \Big)^{-s} \Bigg] \\
		& = \mathbb{E} \Bigg[ \displaystyle\sum_{w \in \mathbb{V} (m-1)} \big| R_{0w_-}^{(w)} \big|^s \cdot |T_{w_- w}|^s \cdot \mathbb{E}_w \bigg[ \displaystyle\sum_{u \in \mathbb{D} (w)} \displaystyle\frac{|T_{wu}|^s}{\Big( \big| z + T_{wu}^2 R_{ww}^{(u)} \big| + 1 \Big)^s} \cdot \big| R_{ww}^{(u)} \big|^s\bigg] \Bigg].
		\end{aligned} 
	\end{flalign}

	Now, observe from \eqref{qvv} that 
	\begin{flalign*}
		R_{ww}^{(u)} = - \big( z + T_{w_- w}^2 R_{w_- w_-}^{(w)} + g_u \big)^{-1}, \quad \text{where} \quad g_u = \displaystyle\sum_{\substack{v \in \mathbb{D} (w) \\ v \ne u}} T_{wv}^2 R_{vv}^{(w)},
	\end{flalign*}

	\noindent and let $(R_v)_{v \sim u}$ denote a family of independent, identically distributed random variables, each with law $R_{00}$. Then, for any $w \in \mathbb{V} (m-1)$, the law of the $\big( R_{vv}^{(w)} \big)$ is the same as that of the $(R_v)$. So, 
	\begin{flalign}
		\label{expectationtw1} 
		\begin{aligned} 
		\mathbb{E}_w & \Bigg[ \displaystyle\sum_{u \in \mathbb{D}(w)} \displaystyle\frac{|T_{wu}|^s}{\Big( \big| z + T_{wu}^2 R_{ww}^{(u)} \big| + 1 \Big)^s} \cdot \big| R_{ww}^{(u)} \big|^s \Bigg]  = \mathbb{E}_w \Bigg[ \displaystyle\sum_{u \in \mathbb{D}(w)} \displaystyle\frac{|T_{wu}|^s}{\Big( \big| z + T_{wu}^2 \widetilde{R}_{ww}^{(u)} \big| + 1 \Big)^s} \cdot \big| \widetilde{R}_{ww}^{(u)} \big|^s \Bigg],
		\end{aligned} 
		\end{flalign} 

	\noindent where 
	\begin{flalign*}
		\widetilde{R}_{ww}^{(u)} = - \big( z + T_{w_- w}^2 R_{w_- w_-}^{(w)} + \widetilde{g}_u \big)^{-1}, \quad \text{and} \quad \widetilde{g}_u = \displaystyle\sum_{\substack{v \in \mathbb{D} (w) \\ v \ne u}} T_{wv}^2 R_v.
	\end{flalign*}
	
	\noindent Again as in the proof of \Cref{xilchi2}, for any point process $\nu = (\nu_1, \nu_2, \ldots )$ on $\mathbb{R}_{> 0}$, let
	\begin{flalign*} 
		\widetilde{R} (\nu) = - \big( z + T_{w_- w}^2 R_{w_- w_-}^{(w)} + \widetilde{g} (\nu) \big)^{-1}, \quad \text{where} \quad \widetilde{g} (\nu) = \sum_{j = 1}^{\infty} \nu_j^2 R_j,
	\end{flalign*} 

	\noindent and the $(R_j)$ are mutually independent random variables with law $R_{00}$. Then, it follows from applying \eqref{fxi} that
	\begin{flalign}
		\label{expectationtw} 
		\begin{aligned} 
		\mathbb{E}_w \Bigg[  \displaystyle\sum_{u \in \mathbb{D}(w)} \displaystyle\frac{|T_{wu}|^s}{\big| z + T_{wu}^2 \widetilde{R}_{ww}^{(u)} \big|^s} \cdot \big| \widetilde{R}_{ww}^{(u)} \big|^s \Bigg] & = \alpha \displaystyle\int_0^{\infty} \mathbb{E}_w \Bigg[ \displaystyle\frac{x^s}{\Big( \big| z + x^2 \widetilde{R} (\nu) \big| + 1 \Big)^s} \cdot \big| \widetilde{R} (\nu) \big|^s \Bigg] \cdot x^{-\alpha-1} dx \\
		& = \alpha \displaystyle\int_0^{\infty} \mathbb{E}_w \Bigg[ \displaystyle\frac{x^s}{\big( | z + x^2 R_{ww}| + 1 \big)^s} \cdot |R_{ww}|^s \Bigg] \cdot x^{-\alpha-1} dx.
		\end{aligned} 
	\end{flalign}

	\noindent where in the expectation $\nu$ is a Poisson point process with intensity $\alpha x^{-\alpha-1} dx$, and we have used the fact that $\widetilde{R} (\nu)$ has the same law as $R_{ww}$. 
	
	Applying \Cref{integral} in \eqref{expectationtw}, using \Cref{rproduct}, and inserting into \eqref{expectationtw1} and \eqref{expectationsumvm} yields a constant $C > 1$ such that 
	\begin{flalign*}
		& C^{-1} \cdot \Xi_{1-\ell} \bigg( \displaystyle\frac{\alpha+s}{2} \bigg) \\
		& \qquad  = C^{-1} \cdot \mathbb{E} \Bigg[ \displaystyle\sum_{w \in \mathbb{V} (m-1)} \big| R_{0w_-}^{(w)} \big|^s \cdot |T_{w_- w}|^s \cdot |R_{ww}|^{(\alpha+s)/2} \Bigg] \\
		& \qquad \le \mathbb{E} \Bigg[ \displaystyle\sum_{w \in \mathbb{V}(m-1)} \big| R_{0w_-}^{(w)} \big|^s \cdot |T_{w_- w}|^s \cdot \mathbb{E}_w \bigg[ \displaystyle\sum_{u \in \mathbb{D} (w)} \displaystyle\frac{|T_{wu}|^s}{\Big( \big| z + T_{wu}^2 R_{ww}^{(u)} \big| + 1 \Big)^s} \cdot \big| R_{ww}^{(u)} \big|^s \bigg] \Bigg] \\
		& \qquad \le C \cdot \mathbb{E} \Bigg[ \displaystyle\sum_{w \in \mathbb{V} (m-1)} \big| R_{0w_-}^{(w)} \big|^s \cdot |T_{w_- w}|^s \cdot |R_{ww}|^{(\alpha+s)/2} \Bigg] = C \cdot \Xi_{1-\ell} \bigg( \displaystyle\frac{\alpha+s}{2} \bigg).
	\end{flalign*}
	
	\noindent This, together with \Cref{xilchi2}, \Cref{xichi}, and \Cref{estimatemoment1}, yields the lemma. 
	\end{proof}
	
	Now we can establish \Cref{limitr0j}.   
	
	\begin{proof}[Proof of \Cref{limitr0j}]
		
		The first statement of the theorem follows from \Cref{estimatemoment1}, \Cref{smultiplicative}, and the P\'{o}lya--Szeg\"{o} lemma  sequences that are both approximately subadditive and superadditive \cite[Lemma 1.9.1]{steele1997probability}.
		To establish the second, observe for any integer $L \ge 1$ that the function $\varphi_L (s; z)$ is convex in $s$, since $\mathbb{E} \big[ |R_{0v}|^s \big]$ is for any $v \in \mathbb{V} (L)$ (by Young's inequality for products). Thus, $\varphi (s; z)$ is convex in $s$ as well, as it is a limit of convex functions. Further, for $\alpha < s < s' < 1$ we have $\mathbb{E} \big[ |R_{0v}|^s \big] \le \eta^{s' - s} \cdot \mathbb{E} \big[ |R_{0v}|^{s'} \big]$ for each $v \in \mathbb{V} (L)$ (by the second part of \Cref{q12}). Taking logarithms, dividing by $L$, and letting $L$ tend to $\infty$, we deduce that $\varphi (s; z)$ is nonincreasing in $s$; this verifies the second statement of the theorem. 
		The third follows from the uniformity of the constant $C$ from \Cref{smultiplicative} in $\Imaginary z$. 
		The fourth and fifth parts of the theorem follow from the second and third parts of \Cref{estimatemoment1}, respectively.
	\end{proof}

	\section{Restricted Moment Estimates}
	 
	\label{MomentEvent}

	In this section we estimate fractional moments $|R_{0v}|^s$ of the resolvent, upon restricting to certain events. In \Cref{EventR} we define these events and establish the ``restricted fractional moment bounds,'' conditional on several estimates that will be shown in \Cref{ProofB}, \Cref{ProofG}, and \Cref{ProofRY}. Throughout this section, we fix a real number $s \in (\alpha, 1)$ and a complex number $z = E + \mathrm{i} \eta \in \mathbb{H}$ with $\eta \in (0, 1)$ and $\varepsilon \le |E| \le B$. Below, constants may implicitly depend on $s$, $\varepsilon$, and $B$, even when not explicitly stated.
	
	\subsection{Events and Restricted Moments}
	
	\label{EventR}

	In this section we define events on which certain tree weights and resolvent entries are bounded, and we analyze fractional moments of off-diagonal resolvent entries restricted to such events. We begin with defining these events. The first considers the event on which exactly one descendant $w$ of a vertex $v$ has $|T_{vw}| \in [1, 2]$; this will later be useful in specifying paths with large resolvent entries. The second defines events on which the tree weights $|T_{uu_+}|$ are bounded from below, and on which the diagonal resolvent entries are bounded from above; in view of \Cref{rproduct} and \eqref{qvv}, this will also later be useful in locating large off-diagonal resolvent entries.

	\begin{definition} 
		
		\label{eventg} 
		
		For any vertex $v \in \mathbb{V}$, we define the vertex set $\mathcal{Y} (v) \subset \mathbb{V}$ and the event $\mathscr{G}_0 (v)$ by
		\begin{flalign*}
			\mathcal{Y} (v) = \big\{ w \in \mathbb{D} (v) : |T_{vw}| \in [1, 2] \big\}; \qquad \mathscr{G}_0 (v) = \Big\{ \big| \mathcal{Y} (v) \big| = 1 \Big\}.
		\end{flalign*} 
		
		\noindent On the event $\mathscr{G}_0 (v)$, let $\mathfrak{c} = \mathfrak{c} (v)$ denote the unique child of $v$ such that $|T_{v\mathfrak{c}}| \in [1, 2]$. 
		
	\end{definition} 
	
	\begin{definition}
		
		\label{br0vw} 
		
		Fix $z \in \mathbb{H}$; let $v, w \in \mathbb{V}$ be vertices with $v \prec w$; and let $\omega \in (0,1)$ and $\Omega > 1$ be real numbers. For any vertex $u \in \mathbb{V}$ with $v \preceq u \prec w$, define the events $\mathscr{D} (u; v, w) = \mathscr{D} (u; v, w; \omega)$ and $\mathscr{D} (v, w) = \mathscr{D} (v, w; \omega)$ by 
		\begin{flalign*}
			\mathscr{D} (u; v, w) = \big\{   |T_{uu_+}| \ge \omega \big\}; \qquad \mathscr{D} (v, w) = \bigcap_{v \preceq u \prec w} \mathscr{D} (u; v, w).
		\end{flalign*}		
		
		\noindent Next, if $\mathscr{G}_0 (w)$ holds, then denote $w_+ = \mathfrak{c} (w)$. Then, define the events $\mathscr{B}_0 (v, w) = \mathscr{B}_0 (v, w; \omega; \Omega) = \mathscr{B}_0 (v, w; \omega; \Omega; z)$; $\mathscr{B}_1 (v, w) = \mathscr{B}_1 (v, w; \omega; \Omega) = \mathscr{B}_1 (v, w; \omega; \Omega; z)$; and $\mathscr{B} (v, w) = \mathscr{B} (v, w; \omega; \Omega) = \mathscr{B} (v, w; \omega; \Omega; z)$ by
		\begin{flalign*}
			& \mathscr{B}_0 (v, w) = \mathscr{G}_0 (w) \cap \Big\{ \big| R_{ww}^{(w_+)} (z) \big| \le \Omega \Big\} \cap \mathscr{D} (v, w; \omega); \\
			& \mathscr{B}_1 (v, w) = \mathscr{G}_0 (w) \cap \Big\{ \big| R_{vv}^{(v_-, w_+)} (z) \big| \le \Omega \Big\} \cap \mathscr{D} (v, w; \omega),
		\end{flalign*} 
		
		\noindent and $\mathscr{B} (v, w) = \mathscr{B}_0 (v, w) \cap \mathscr{B}_1 (v, w)$. Here we used the convention that $0- = \emptyset$ so that $R_{00}^{(0_-, w_+)} (z) = R_{00}^{(w_+)} (z)$ for example.
		
	\end{definition} 
	
	We next have the following proposition, which indicates that restricting to the events $\mathscr{B} (v, w)$ does not substantially decrease the sum of fractional moments $|R_{vw}|^s$, if $\omega$ and $\Omega$ are sufficiently small and large, respectively. 
	
	\begin{prop} 
		
		\label{expectationr0vsd} 
		
		For any real number $\delta \in (0, 1)$, there exist constants $c > 0$ (independent of $\delta$), $\omega = \omega (\delta) \in (0, 1)$, and $\Omega = \Omega (\delta) > 1$ such that the following holds. For any integer $L \ge 1$ and vertex $v \in \mathbb{V}$, we have
		\begin{flalign*}
			\mathbb{E} \Bigg[ \displaystyle\sum_{w \in \mathbb{D}_L (v)} \big| R_{vw}^{(v_-, w_+)} (z) \big|^s \cdot \one_{\mathscr{B} (v, w; \omega; \Omega)} \Bigg] \ge c \cdot \exp \Big( L \cdot \big( \varphi (s; z) - \delta \big) \Big).   
		\end{flalign*}
		
	\end{prop}

	We will deduce \Cref{expectationr0vsd} as a consequence of the following two propositions. The former will be established in \Cref{ProofB} below and the latter in \Cref{ProofG} below. 
	
	\begin{prop} 
		
		\label{expectationbr}
		
		For any real number $\delta > 0$, there exist constants $c > 0$ and $\Omega = \Omega (\delta) > 1$ such that, for any $\omega \in \big( 0, \frac{1}{2} \big)$, we have
		\begin{flalign*} 
			\mathbb{E} \Bigg[ \displaystyle\sum_{w \in \mathbb{D}_{\ell} (v)} \big| R_{vw}^{(v_-, w_+)} \big|^s \cdot \one_{\mathscr{B} (v, w)} \Bigg] \ge (1 - \delta) \cdot \mathbb{E} \Bigg[ \displaystyle\sum_{w \in \mathbb{D}_{\ell} (v)} \big| R_{vw}^{(v_-, w_+)} \big|^s \cdot \one_{\mathscr{D} (v, w)} \cdot \one_{\mathscr{G}_0 (w)} \Bigg].
		\end{flalign*} 
	\end{prop}

	\begin{prop} 
		
		\label{dexpectationg} 
		
		For any real number $\delta > 0$, there exist constants $c > 0$ and $ \omega = \omega (\delta) \in (0, 1)$ such that 
		\begin{flalign*} 
			\mathbb{E} \Bigg[ \displaystyle\sum_{w \in \mathbb{D}_{\ell} (v)} \big| R_{vw}^{(v_-, w_+)} \big|^s \cdot \one_{\mathscr{D} (v,w)} \cdot \one_{\mathscr{G}_0 (w)} \Bigg] \ge c (1 - \delta)^{\ell} \cdot \Phi_{\ell} (s).
		\end{flalign*} 
	\end{prop}

		\begin{proof}[Proof of \Cref{expectationr0vsd}]
		
		This follows from \Cref{expectationbr}, \Cref{dexpectationg}, and \eqref{limitr0j2}.
	\end{proof}

	\subsection{Proof of \Cref{expectationbr}}

	\label{ProofB}

	In this section we establish \Cref{expectationbr}, which will follow from the two following two lemmas.

	\begin{lem} 
	
	\label{ruw} 

	For any real number $\delta > 0$, there exists a constants $\Omega = \Omega (\delta) > 1$ such the following two statements hold whenever $\omega \in \big(0, \frac{1}{2} \big)$.
	
	\begin{enumerate} 
		
		\item We have  
	\begin{flalign}
		\label{rvw10}
		\mathbb{E} \Bigg[ \displaystyle\sum_{w \in \mathbb{D}_{\ell} (v)} \big| R_{vw}^{(v_-, w_+)} \big|^s \cdot \one_{\mathscr{B}_0 (v, w; \omega; \Omega)} \Bigg] \ge (1 - \delta) \cdot \mathbb{E} \Bigg[ \displaystyle\sum_{w \in \mathbb{D}_{\ell} (v)}  \big| R_{vw}^{(v_-, w_+)} \big|^s \cdot \one_{\mathscr{D}(v,w; \omega)} \cdot \one_{\mathscr{G}_0 (w)} \Bigg].
	\end{flalign} 

		\item We have   
		\begin{flalign}
			\label{rvw2} 
			\begin{aligned}
			\mathbb{E} \Bigg[ \displaystyle\sum_{w \in \mathbb{D}_{\ell} (v)} \big| R_{vw}^{(v_-)} \big|^s \cdot \big| & R_{ww}^{(v_-)} \big|^{(\alpha-s)/2} \cdot \one_{|R_{ww}^{(v_-)}| \le 1/\Omega} \cdot \one_{\mathscr{D} (v, w; \omega)} \Bigg] \\
			& \ge (1 - \delta) \cdot \mathbb{E} \Bigg[ \displaystyle\sum_{w \in \mathbb{D}_{\ell} (v)} \big| R_{vw}^{(v_-)} \big|^s \cdot \big| R_{ww}^{(v_-)} \big|^{(\alpha-s)/2} \cdot \one_{\mathscr{D} (v, w; \omega)} \Bigg].
			\end{aligned}
		\end{flalign}
		
	\end{enumerate} 

	\end{lem}

	\begin{lem} 
		
	\label{expectationrvwdelta}
	
	For any $\delta > 0$, there exists a constant $\Omega = \Omega (\delta) > 1$ for which the following two statements hold.
	\begin{enumerate} 
		
		\item We have 
	\begin{flalign}
		\label{rv3} 
		\mathbb{E} \Bigg[ \displaystyle\sum_{w \in \mathbb{D}_{\ell} (v)} \big| R_{vw}^{(v_-, w_+)} \big|^s \cdot \one_{\mathscr{B}_1 (v, w; \omega; \Omega)} \Bigg] \ge (1 - \delta) \cdot \mathbb{E} \Bigg[ \displaystyle\sum_{w \in \mathbb{D}_{\ell} (v)}  \big| R_{vw}^{(v_-, w_+)} \big|^s \cdot \one_{\mathscr{D} (v, w; \omega)} \cdot \one_{\mathscr{G}_0 (w)} \Bigg].
	\end{flalign} 

	\item We have 
		\begin{flalign}
			\label{rv4} 
			\mathbb{E} \Bigg[ \displaystyle\sum_{w \in \mathbb{V} (\ell)} |R_{0w}|^s \cdot \one_{\mathscr{D} (0, w; \omega)} \cdot \one_{|R_{00}| \le \Omega} \Bigg] \ge (1 - \delta) \cdot \mathbb{E} \Bigg[ \displaystyle\sum_{w \in \mathbb{V} (\ell)} |R_{0w}|^s \cdot \one_{\mathscr{D} (0, w; \omega)} \Bigg].
		\end{flalign} 
	\end{enumerate} 
	\end{lem}

	\begin{proof}[Proof of \Cref{expectationbr}]
	For any events $\mathscr A, \mathscr B, \mathscr C$, if $\mathscr B \cap \mathscr C \subset \mathscr A$, then $\one_{\mathscr A} - \one_{\mathscr B \cap \mathscr C } \leq  \one_{\mathscr A}  - \one_{\mathscr B}  + \one_{\mathscr A}  - \one_{\mathscr C}$.   Since $\mathscr{B} (v, w) = \mathscr{B}_0 (v, w) \cap \mathscr{B}_1 (v, w)$,  we  have
	\begin{align*}
	\mathbb{E}& \Bigg[ \displaystyle\sum_{w \in \mathbb{D}_{\ell} (v)}  \big| R_{vw}^{(v_-, w_+)} \big|^s \cdot 
	(
	\one_{\mathscr{D} (v, w; \omega)\cap \mathscr{G}_0 (w)}- \one_{ \mathscr{B} (v, w)}
	) \Bigg] \\
	&\leq \mathbb{E} \Bigg[ \displaystyle\sum_{w \in \mathbb{D}_{\ell} (v)}  \big| R_{vw}^{(v_-, w_+)} \big|^s \cdot  (\one_{\mathscr{D} (v, w; \omega)\cap \mathscr{G}_0 (w)}  -\one_{ \mathscr{B}_0 (v, w)})  \Bigg]\\
	&+ \mathbb{E} \Bigg[ \displaystyle\sum_{w \in \mathbb{D}_{\ell} (v)}  \big| R_{vw}^{(v_-, w_+)} \big|^s \cdot  (\one_{\mathscr{D} (v, w; \omega)\cap \mathscr{G}_0 (w)}  -\one_{ \mathscr{B}_1 (v, w)})   \Bigg].
	\end{align*}The proposition next follows from \eqref{rvw10} and \eqref{rv3}. 
	\end{proof}

	Now we establish \Cref{ruw} and \Cref{expectationrvwdelta}.
	
	\begin{proof}[Proof of \Cref{ruw}]
		
		The proofs of \eqref{rvw10} and \eqref{rvw2} are very similar, so we only establish the first. Applying \Cref{rproduct} and \eqref{qvv} yields
		\begin{flalign}
			\label{rvwb0} 
			\begin{aligned}	
				\mathbb{E} & \Bigg[ \displaystyle\sum_{w \in \mathbb{D}_{\ell} (v)} \big| R_{vw}^{(v_-, w_+)} \big|^s \cdot \one_{\mathscr{B}_0 (v, w)} \Bigg] \\
				& = \mathbb{E} \Bigg[ \displaystyle\sum_{u \in \mathbb{D}_{\ell-1} (v)} \displaystyle\sum_{w \in \mathbb{D}(u)} \big| R_{vu}^{(v_-, w)} \big|^s \cdot |T_{uw}|^s \cdot \big| R_{ww}^{(v_-, w_+)} \big|^s \cdot \one_{\mathscr{B}_0 (v, w)} \Bigg] \\
				& = \mathbb{E} \Bigg[ \displaystyle\sum_{u \in \mathbb{D}_{\ell-1} (v)} \displaystyle\sum_{w \in \mathbb{D} (u)} \displaystyle\frac{|T_{uw}|^s}{\big| z + T_{uw}^2 R_{uu}^{(v_-, w)} + K_u \big|^s } \cdot \big| R_{vu}^{(v_-, w)} \big|^s\cdot \one_{\mathscr{D} (v, w)} \cdot \one_{\mathscr{G}_0 (w)} \cdot \one_{|R_{ww}^{(w_+)}| \le \Omega} \Bigg],
			\end{aligned} 
		\end{flalign}
		
		\noindent where $K_u = K_{u; v, w}$ is defined by
		\begin{flalign*}
			K_u = \displaystyle\sum_{\substack{u' \in \mathbb{D} (w) \\ u' \ne \mathfrak{c} (w)}} T_{wu'}^2 R_{u'u'}^{(w)}.
		\end{flalign*} 
		
		\noindent Observe that $K_{u; v, w}$ is independent from the $T_{uw}$, $R_{uu}^{(v_-, w)}$, $R_{vu}^{(v_-, w)}$, and $\mathscr{D} (v, w)$. Moreover, conditional on $\mathscr{G}_0 (w)$, the random variable $K_{u; v, w}$ satisfies \Cref{sqk} (by \Cref{tuvalpha}, \Cref{zuv}, and \Cref{ar}). Additionally, again by \eqref{qvv} we have the equivalence of events
		\begin{flalign*}
			\Big\{ \big| R_{ww}^{(w_+)} \big| \ge \Omega \Big\} = \big\{ |Q_u + K_u| \le \Omega^{-1} \big\}, \qquad \text{where} \qquad Q_u = z + T_{uw}^2 R_{uu}^{(w)}.
		\end{flalign*} 
		
		Thus, applying the lower bound in the first part of \Cref{zxrkestimate}, and the fourth part of the same lemma (with the $K_u$ there equal to $K$ here and the $J$ there equal to $[-\Omega^{-1} - \Real Q_u, \Omega^{-1} - \Real Q_u]$), and using the independence between $K_u$ and $Q_u$, yields a constant $C > 1$ such that 
			\begin{flalign*}
			\mathbb{E}_w \Big[ \big| z + T_{uw}^2 & R_{uu}^{(v_-, w)} + K_u \big|^{-s} \cdot \one_{\mathscr{G}_0 (w)} \cdot \one_{|R_{ww}^{(w_+)}| > \Omega} \Big] \\
			& \le C \Omega^{s-1} \cdot \mathbb{E}_w \Big[ \big| z + T_{uw}^2 R_{uu}^{(v_-, w_+)} + K_u \big|^{-s} \cdot \one_{\mathscr{G}_0 (w)} \Big],
		\end{flalign*}
	which yields
		\begin{flalign*}
			\mathbb{E}_w \Big[ \big| z + T_{uw}^2 & R_{uu}^{(v_-, w)} + K_u \big|^{-s} \cdot \one_{\mathscr{G}_0 (w)} \cdot \one_{|R_{ww}^{(w_+)}| \le \Omega} \Big] \\
			& \ge (1 - C \Omega^{s-1}) \cdot \mathbb{E}_w \Big[ \big| z + T_{uw}^2 R_{uu}^{(v_-, w_+)} + K_u \big|^{-s} \cdot \one_{\mathscr{G}_0 (w)} \Big].
		\end{flalign*}
		
		\noindent Inserting this into \eqref{rvwb0} and then applying the same reasoning as in \eqref{rvwb0} yields 
		\begin{flalign*}
			\mathbb{E} \Bigg[ & \displaystyle\sum_{w \in \mathbb{D}_{\ell} (v)} \big| R_{vw}^{(v_-, w_+)} \big|^s \cdot \one_{\mathscr{B}_0 (v, w)}\Bigg] \\
			& \ge (1 - C \Omega^{s-1}) \cdot \mathbb{E} \Bigg[ \displaystyle\sum_{u \in \mathbb{D}_{\ell-1} (v)} \displaystyle\sum_{w \in \mathbb{D} (u)} \displaystyle\frac{|T_{uw}|^2}{\big| z + T_{uw}^2 R_{uu}^{(v_-, w)} + K_u \big|^s } \cdot \big| R_{vu}^{(v_-, w)} \big|^s \cdot \one_{\mathscr{D} (v, w)} \cdot \one_{\mathscr{G}_0 (w)} \Bigg] \\
			& = (1 - C \Omega^{s-1}) \cdot \mathbb{E} \Bigg[ \displaystyle\sum_{w \in \mathbb{D}_{\ell} (v)} \big| R_{vw}^{(v_-, w_+)} \big|^s \cdot \one_{\mathscr{D} (v, w)} \cdot \one_{\mathscr{G}_0 (w)} \Bigg],
		\end{flalign*}
		
		\noindent from which we deduce \eqref{rvw10} by taking $\Omega = \Omega (\delta)$ sufficiently large so that $C \Omega^{s-1} \le \delta$. 
		
		As mentioned previously, the proof of \eqref{rvw2} is very similar (by observing that the random variable $K_{u; v, w}' = \sum_{u' \in \mathbb{D}(w)} T_{wu'}^2 R_{u' u'}^{(w)}$ also satisfies \Cref{sqk}, and then applying the first and third parts of \Cref{zxrkestimate}).
	\end{proof} 

	\begin{proof}[Proof of \Cref{expectationrvwdelta}]
		
		The proofs of \eqref{rv3} and \eqref{rv4} are very similar, and so we only establish the former. To that end, we have by \Cref{rproduct} and \eqref{qvv} that
		\begin{flalign}
			\label{b1sum2} 
			\begin{aligned}
			\mathbb{E} & \Bigg[ \displaystyle\sum_{w \in \mathbb{D}_{\ell} (v)} \big| R_{vw}^{(v_-, w_+)} \big|^s \cdot \one_{\mathscr{B}_1 (v, w)} \Bigg] \\
			& = \mathbb{E} \Bigg[ \displaystyle\sum_{u \in \mathbb{D} (v)} \displaystyle\sum_{w \in \mathbb{D}_{\ell-1} (u)} \big| R_{vv}^{(v_-, w_+)} \big|^s \cdot |T_{vu}|^s \cdot \big| R_{uw}^{(v, w_+)} \big|^s \cdot \one_{\mathscr{B}_1 (v, w)} \Bigg] \\
			& = \mathbb{E} \Bigg[ \displaystyle\sum_{u \in \mathbb{D}(v)} \displaystyle\frac{|T_{vu}|^s}{\big| z + T_{vu}^2 R_{uu}^{(v, w_+)} + K_{u; v, w} \big|^s} \cdot \one_{|R_{vv}^{(v_-, w_+)}| \le \Omega}  \displaystyle\sum_{w \in \mathbb{D}_{\ell-1} (u)} \big| R_{uw}^{(v, w_+)} \big|^s \cdot \one_{\mathscr{D} (v, w)} \cdot \one_{\mathscr{G}_0 (w)} \Bigg],
			\end{aligned} 
		\end{flalign}
	
		\noindent where 
		\begin{flalign*}
			K_{u; v, w} = \displaystyle\sum_{\substack{v' \in \mathbb{D} (v) \\ v' \ne u}} T_{vv'}^2 R_{v'v'}^{(v)}.
		\end{flalign*}
	
		\noindent Observe that $K_{u; v, w}$ is independent from $T_{vu}$, $R_{uu}^{(v, w_+)}$, $R_{uw}^{(v, w_+)}$, $\mathscr{D} (v, w)$, and $\mathscr{G}_0 (w)$ Therefore, applying the first and third parts of \Cref{zxrkestimate} (with the $K$ there equal to $K_{u; v, w}$ here) yields a constant $C_1 > 1$ such that 
		\begin{flalign*}
			 \mathbb{E} & \Bigg[ \displaystyle\sum_{u \in \mathbb{V}} \displaystyle\frac{|T_{vu}|^s}{\big| z + T_{vu}^2 R_{uu}^{(v, w_+)} + K_{u; v, w} \big|^s} \cdot \one_{|R_{vv}^{(v_-, w_+)}| \le \Omega} \displaystyle\sum_{w \in \mathbb{D}_{\ell-1} (u)} \big| R_{uw}^{(v, w_+)} \big|^s \cdot \one_{\mathscr{D} (v, w)} \cdot \one_{\mathscr{G}_0 (w)} \Bigg] \\
			& \ge (1 - C \Omega^{s - 1}) \cdot \Bigg[ \displaystyle\sum_{u \in \mathbb{V}} \displaystyle\frac{|T_{vu}|^s}{\big| z + T_{vu}^2 R_{uu}^{(v, w_+)} + K_{u; v, w} \big|^s}  \displaystyle\sum_{w \in \mathbb{D}_{\ell-1} (u)} \big| R_{uw}^{(v, w_+)} \big|^s \cdot \one_{\mathscr{D} (v, w)} \cdot \one_{\mathscr{G}_0 (w)} \Bigg] \\
			& = (1 - C \Omega^{s -1}) \cdot \mathbb{E} \Bigg[ \displaystyle\sum_{w \in \mathbb{D}_{\ell} (v)} \big| R_{vw}^{(v_-, w_+)} \big|^s \cdot \one_{\mathscr{D} (v, w)} \cdot \one_{\mathscr{G}_0 (w)} \Bigg],
		\end{flalign*}
		
		\noindent where the last equality is obtained through the same reasoning as in \eqref{b1sum2}. Then \eqref{rv3} follows from selecting $\Omega$ sufficiently large so that $C \Omega^{1-\chi} < \delta$. 	 
	\end{proof}

	\subsection{Proof of \Cref{dexpectationg}}
	
	\label{ProofG}
	
	In this section we establish \Cref{dexpectationg}, which will be a consequence of the following two lemmas. The first will be established in this section and the latter in \Cref{ProofRY} below.
	
	\begin{lem} 
		The following two claims hold.
		\label{ryestimate} 
		\begin{enumerate}
		\item
		There exists a constant $c > 0$ such that for any sufficiently small $\omega \in \big( 0, \frac{1}{2} \big)$ we have
		\begin{flalign*}
			\mathbb{E} \Bigg[ \displaystyle\sum_{w \in \mathbb{D}_{\ell} (v)} \big| R_{vw}^{(v_-, w_+)} \big|^s \cdot \one_{\mathscr{D} (v, w; \omega)} \cdot \one_{\mathscr{G}_0 (w)} \Bigg] \ge c \cdot \mathbb{E} \Bigg[ \displaystyle\sum_{u \in \mathbb{D}_{\ell-1} (v)} \big| R_{vu}^{(v_-)} \big|^s \cdot \one_{\mathscr{D} (v, u; \omega)} \Bigg]. 
		\end{flalign*}
		\item There exists a constant $C>0$ such that 
		\be
		\E \left[\sum_{w \in \mathbb D_\ell (v)} |R_{vw}^{(v_-, w_+)}|^s \cdot \one_{\mathscr G_0 (w)}\right] \le C \cdot \E\left[ \sum_{u \in \mathbb D_{\ell-1} (v)} | R_{vu}^{(v_-)}|^s \right].
	\ee
		\end{enumerate}
	\end{lem} 

		\begin{lem} 
		
		\label{ry0estimate}
		
		For any $\delta > 0$, there exists a constant $\omega = \omega (\delta) > 0$ such that 
		\begin{flalign*} 
			\mathbb{E} \Bigg[ \displaystyle\sum_{w \in \mathbb{D}_{\ell (v)}} \big| R_{vw}^{(v_-)} \big|^s \cdot \one_{\mathscr{D} (v, w; \omega)} \Bigg] \ge (1 - \delta)^{\ell} \cdot \Phi_{\ell} (s).
		\end{flalign*} 
		Moreover, there exists constants $C > 1$ and $c \in (0, 1)$ such that $\omega \ge c \cdot \delta^C$
	\end{lem} 
	
	\begin{proof}[Proof of \Cref{dexpectationg}]
		
		This follows from \Cref{ryestimate} and \Cref{ry0estimate}. 
	\end{proof} 

 	Now let us establish \Cref{ryestimate}. 	
	
	\begin{proof}[Proof of \Cref{ryestimate} (Outline)]
		
		Since the joint law of $\big( R_{vw}^{(v_-, w_+)}, \mathscr{D} (v, w), \mathscr{G}_0 (w) \big)$ for $w \in \mathbb{D}_{\ell} (v)$ is the same as that of $\big( R_{0w}^{(w_+)}, \mathscr{D} (0, w), \mathscr{G}_0 (w) \big)$ for $w \in \mathbb{V} (\ell)$, it suffices to address the case when $v = 0$. This proof will proceed similarly to that of \Cref{xilchi2}, so we only outline it. Further, we only prove the first part of the lemma, since the proof of the second part is similar.
		
		Following \eqref{xi1} and \eqref{xil2} there, we find that 
		\begin{flalign}
			\label{expectationwv1} 
			\begin{aligned}
			\mathbb{E} \Bigg[ \displaystyle\sum_{w \in \mathbb{V} (\ell)} \big| R_{0w}^{(w_+)} \big|^s \cdot \one_{\mathscr{D} (0, w)} \cdot \one_{\mathscr{G}_0 (w)} \Bigg] &  = \mathbb{E} \Bigg[ \displaystyle\sum_{u \in \mathbb{V} (\ell - 1)} \big| R_{0u_-}^{(u)} \big|^s \cdot |T_{u_- u}|^s \cdot \one_{\mathscr{D} (0, u)} \\
			& \quad \times \mathbb{E}_w \bigg[ \displaystyle\sum_{w \in \mathbb{D} (u)} \displaystyle\frac{|T_{uw}|^s \cdot \big| \widetilde{R}_{uu}^{(w)} \big|^s}{\big| z + T_{uw}^2 \widetilde{R}_{uu}^{(w)} + K \big|^s} \cdot \one_{|T_{uw}| \ge \omega} \cdot \one_{\mathscr{G}_0 (w)} \bigg] \Bigg],
			\end{aligned} 
		\end{flalign}
	
		\noindent where
		\begin{flalign*}
			\widetilde{R}_{uu}^{(w)} = - \big( z + T_{u_- u}^2 R_{u_- u_-}^{(u)} + \widetilde{g}_w \big)^{-1}; \qquad K = \displaystyle\sum_{\substack{v \in \mathbb{D} (w) \\ v \ne \mathfrak{c} (w)}} T_{wv}^2 R_v; \qquad \widetilde{g}_w = \displaystyle\sum_{\substack{v \in \mathbb{D} (u) \\ v \ne w}} T_{wv}^2 R_v',
		\end{flalign*} 
	
		\noindent and $(R_v)_{v \in \mathbb{D} (w)}$ and $(R_v')_{v \in \mathbb{D} (u)}$ are mutually independent, identically distributed random variable with law $R_{00}$. Observe in particular that $\widetilde{R}_{uu}^{(w)}$ has the same law as $R_{uu}$ (conditional on $\mathbb{T}_- (w)$).
	
		Now, observe (from \Cref{tuvalpha}) that there exists a constant $c_1 > 0$ for which $\mathbb{P}\big[ \mathscr{G}_0 (w) \big] \ge c_0$. Since, conditional on $\mathscr{G}_0 (w)$, the random variable $K$ satisfies \Cref{sqk} by \Cref{zuv} (and \Cref{tuvalpha} and \Cref{ar}), it follows from the first statement of \Cref{zxrkestimate} that there exists a constant $c_2 > 0$ such that 
		\begin{flalign}
			\label{sumtw} 
			\begin{aligned}
			\mathbb{E}_w & \Bigg[ \displaystyle\sum_{w \in \mathbb{D} (u)} \displaystyle\frac{|T_{uw}|^s \cdot \big| \widetilde{R}_{uu}^{(w)} \big|^s}{\big| z + T_{uw}^2 \widetilde{R}_{uu}^{(w)} + K \big|^s} \cdot \one_{|T_{uw}| \ge \omega} \cdot \one_{\mathscr{G}_0 (w)} \Bigg] \\
			& \ge c_2 \cdot \mathbb{E}_w \Bigg[ \displaystyle\sum_{w \in \mathbb{D} (u)} \displaystyle\frac{|T_{uw}|^s}{\Big( \big| z + T_{uw}^2 \widetilde{R}_{uu}^{(w)} \big| + 1 \Big)^s } \cdot \big| \widetilde{R}_{uu}^{(w)} \big|^s \cdot \one_{|T_{uw}| \ge \omega} \Bigg].
			\end{aligned} 
		\end{flalign}	
	
		\noindent Since $\widetilde{R}_{uu}^{(w)}$ does not depend on $T_{uw}$, we may proceed as in \eqref{xil3} (through applying \eqref{fxi}) to deduce 
		\begin{flalign}
			\label{integralsumomega1}
			\begin{aligned} 
			 \mathbb{E}_w \Bigg[ \displaystyle\sum_{w \in \mathbb{D} (u)} & \displaystyle\frac{|T_{uw}|^s}{\Big( \big| z + T_{uw}^2 \widetilde{R}_{uu}^{(w)} \big| + 1 \Big)^s } \cdot \big| \widetilde{R}_{uu}^{(w)} \big|^s \cdot \one_{|T_{uw}| \ge \omega} \Bigg] \\
			 & = \alpha \displaystyle\int_{\omega}^{\infty} \mathbb{E}_w \Bigg[ \displaystyle\frac{x^2}{\big( |z + x^2 R_{uu}| + 1 \big)^s} \cdot |R_{uu}|^s \Bigg] \cdot x^{-\alpha-1} dx.
			 \end{aligned}
		\end{flalign}
		
		Next, let $\varsigma \in (0, 1)$ be a real number (to be determined later). By \Cref{integral} and \Cref{integral3}, there exists a constant $C > 1$ such that 
		\begin{flalign}
			\label{integralomega1}
			\begin{aligned}  
			\displaystyle\int_{\omega}^{\infty} \mathbb{E}_w & \Bigg[ \displaystyle\frac{x^s}{|z + x^2 R_{uu}|^s} \cdot |R_{uu}|^s \Bigg] \cdot x^{-\alpha-1} dx \ge \one_{|R_{uu}| \le 1/\varsigma} \cdot \big(1 - C (\omega \varsigma^{-1})^{s-\alpha} \big) \cdot |R_{uu}|^{(\alpha+s)/2}.
			\end{aligned} 
	\end{flalign}
	
		\noindent By \eqref{expectationwv1}, \eqref{integralomega1}, \eqref{integralsumomega1}, \eqref{sumtw}, \Cref{rproduct}, and \eqref{rvw2}, we deduce 
		\begin{flalign*}
			\mathbb{E} & \Bigg[ \displaystyle\sum_{w \in \mathbb{V} (\ell)} \big| R_{0w}^{(w_+)} \big|^s \cdot \one_{\mathscr{D} (0, w)} \cdot \one_{\mathscr{G}_0 (w)} \Bigg] \\
			& \ge c_2 \big( 1 - 2C (\omega \varsigma^{-1})^{s-\alpha} \big) \cdot \mathbb{E} \Bigg[ \displaystyle\sum_{u \in \mathbb{V} (\ell-1)} \big| R_{0u_-}^{(u)} \big|^s \cdot |T_{u_- u}|^s \cdot |R_{uu}|^{(\alpha+s)/2} \cdot \one_{\mathscr{D} (0, u)} \cdot \one_{|R_{uu}| \le 1/\varsigma,} \Bigg] \\
			& \ge \displaystyle\frac{c_2}{2} \cdot \mathbb{E} \Bigg[ \displaystyle\sum_{u \in \mathbb{V} (\ell-1)} |R_{0u}|^s \cdot |R_{uu}|^{(\alpha-s)/2} \cdot \one_{\mathscr{D} (0,u)} \cdot \one_{|R_{uu}| \le 1/\varsigma}\Bigg] \\
			& \ge \displaystyle\frac{c_2}{2} \cdot \varsigma^{(\alpha-s)/2} \cdot \mathbb{E} \Bigg[ \displaystyle\sum_{u \in \mathbb{V} (\ell-1)} |R_{0u}|^s \cdot \one_{\mathscr{D} (0, u)} \Bigg], 
		\end{flalign*}
	
		\noindent where the second inequality holds by fixing $\omega$ sufficiently small with respect to $\varsigma$, and the third holds by fixing $\varsigma$ sufficiently smal (due to the third by \eqref{rvw2}). This yields the lemma.
	\end{proof}

	\subsection{Proof of \Cref{ry0estimate}} 
	
	\label{ProofRY} 
	
	In this section we establish \Cref{ry0estimate}, which will be a quick consequence of the following two lemmas. 
	
	\begin{lem}
	
	\label{ry0estimate2} 
	
	For any real number $\delta > 0$, there exist constants $c > 0$ (independent of $\delta$) and $\omega = \omega (\delta) > 0$ such that for any integer $k \in \unn{0}{\ell}$ we have  
	\begin{flalign*} 
		\mathbb{E} \Bigg[ \displaystyle\sum_{u \in \mathbb{V} (k)} \displaystyle\sum_{w \in \mathbb{D}_{\ell-k} (u)} |R_{0w}|^s \cdot \one_{\mathscr{D} (u,w)} \Bigg] \ge c \cdot \mathbb{E} \Bigg[ \displaystyle\sum_{u \in \mathbb{V} (k-1)} | R_{0u}|^s \Bigg] \cdot \mathbb{E} \Bigg[ \displaystyle\sum_{u \in \mathbb{V} (\ell-k-1)} |R_{0u}|^s \cdot \one_{\mathscr{D} (0, u)}\Bigg].
	\end{flalign*} 
	\end{lem} 
	
	\begin{lem}
		
		\label{ry0estimate3}
		
		For any real number $\delta > 0$, there exist constants $C > 1$ (independent of $\delta$) and $\omega = \omega (\delta) > 0$ such that for any integer $k \in \unn{0}{\ell}$ we have  
		\begin{flalign*} 
			\mathbb{E} \Bigg[ \displaystyle\sum_{u \in \mathbb{V} (k)} & \displaystyle\sum_{w \in \mathbb{D}_{\ell-k} (u)} |R_{0w}|^s \cdot \one_{|T_{uu_+}| < \omega} \cdot \one_{\mathscr{D} (u_+,w)} \Bigg] \\
			& \le C \delta \cdot \mathbb{E} \Bigg[ \displaystyle\sum_{u \in \mathbb{V} (k-1)} | R_{0u}|^s \Bigg] \cdot \mathbb{E}\Bigg[ \displaystyle\sum_{u \in \mathbb{V} (\ell - k- 1)}  | R_{0u}|^s \cdot \one_{\mathscr{D} (0, u)} \Bigg],
		\end{flalign*} 
	
		\noindent where $u_+ \in \mathbb{D} (u)$ denotes the unique child of $u$ in the path $\mathfrak{p} (u, w)$ between $u$ and $w$. Moreover, there exists a constant $c \in (0, 1)$ such that $\omega \ge c \cdot \delta^C$.
	\end{lem}

	\begin{proof}[Proof of \Cref{ry0estimate}]
		
		Let $\delta >0$ be a parameter. Combining \Cref{ry0estimate2} and \Cref{ry0estimate3}, 
		we deduce that there exists a constant $\omega = \omega (\delta) > 0$ such that for any $k \in \unn{0}{\ell}$ we have
		\bex
		\E \left[ \sum_{u \in \bbV(k)} \sum_{w \in {\mathbb D}_{\ell-k} (u)} |R_{0w}|^s \cdot \one_{|T_{uu_+}| < \omega} \cdot \one_{\mathscr{D} (u_+, w)} \right] 
\le c^{-1} C \delta \cdot  \E\left[ \sum_{u \in \bbV(k)} \sum_{w \in {\mathbb D}_{\ell-k} (u)} |R_{0w}|^s \cdot \one_{\mathscr{D}(u,w)} \right],
\eex
where $c$ and $C$ are from \Cref{ry0estimate2} and \Cref{ry0estimate3}, respectively. Since $\big( R_{vw}^{(v_-)} \big)_{w \in \mathbb{D}_m (v)}$ has the same law as $(R_{0w})_{w \in \mathbb{V} (m)}$ for any integer $m \ge 0$, this is equivalent to 
	\begin{flalign*}
\E \Bigg[ \sum_{u \in \mathbb{D}_k (v)} \sum_{w \in {\mathbb D}_{\ell-k} (u)} & | R_{vw}^{(v_-)} |^s \cdot \one_{|T_{uu_+}| < \omega} \cdot \one_{\mathscr{D} (u_+, w)} \Bigg] \\
& \le c^{-1} C \delta \cdot \E\left[ \sum_{u \in \mathbb{D}_k (v)} \sum_{w \in {\mathbb D}_{\ell-k} (u)} | R_{vw}^{(v_-)} |^s \cdot \one_{\mathscr{D}(u,w)} \right],
\end{flalign*}
for any $v \in \bbV$.
		Observe from the definition of $\mathscr D(u, v)$ that 
\bex
\one_{\mathscr{D} (u_+, v)} - \one_{\mathscr{D}(u, w)} = \one_{\mathscr{D}(u_+, w)} \cdot \one_{|T_{uu_+}| < \omega}.\eex

\noindent Thus, adding 
\bex
\E\left[ \sum_{u \in \mathbb{D}_k (v)} \sum_{w \in \mathbb D_{\ell-k} (u)} | R_{vw}^{(v_-)}|^s \cdot \one_{\mathscr{D}(u, w)} \right]
\eex
to both sides of the previous bound gives
\bex
\E \left[ \sum_{u \in \mathbb{D}_k (v)} \sum_{w \in \mathbb D_{\ell-k} (u)} | R_{vw}^{(v_-)} |^s \cdot \one_{\mathscr{D} (u_+, w)} \right] 
\le (1 + c^{-1} C\cdot\delta) \cdot \E\left[ \sum_{u \in \mathbb{D}_k (v)} \sum_{w \in {\mathbb D}_{\ell-k} (u)} | R_{vw}^{(v_-)} |^s \cdot \one_{\mathscr{D}(u,w)} \right].\eex

\noindent This inequality implies that 
\begin{align*}
\E & \left[ \sum_{u \in \mathbb{D}_k (v)} \sum_{w \in {\mathbb D}_{\ell-k} (u)} | R_{vw}^{(v_-)} |^s \cdot \one_{\mathscr{D}(u,w)} \right] \\
&\ge (1 + c^{-1} C\cdot\delta)^{-1} \cdot \E\left[ \sum_{u_+ \in {\mathbb D}_{k+1} (v)} \sum_{w \in {\mathbb D}_{\ell-k-1} (u_+)} | R_{vw}^{(v_-)} |^s \cdot \one_{\mathscr{D}(u_+, w)} \right]\\
&\ge (1 - c^{-1} C\cdot\delta) \cdot \E\left[ \sum_{u \in {\mathbb D}_{k+1} (v)} \sum_{w \in {\mathbb D}_{\ell-k-1} (u)} | R_{vw}^{(v_-)} |^s \cdot \one_{\mathscr{D}(u, w)} \right],\end{align*}

\noindent where in the first inequality we observed that the sum over $u \in \mathbb D_k (v)$ is equivalent to the sum over $u_+ \in \mathbb D_{k+1} (v)$, and in the second inequality we used the fact that $(1+a)^{-1} \ge 1-a$, and we also replaced $u_+$ by $u$. Hence,		
		\begin{flalign*} 
			\mathbb{E} \Bigg[ \displaystyle\sum_{u \in \mathbb{D}_k (v)} \displaystyle\sum_{w \in \mathbb{D}_{\ell-k} (v)} \big| R_{vw}^{(v_-)} \big|^s \cdot \one_{\mathscr{D} (u,w)} \Bigg] \ge (1 - c^{-1} C \cdot \delta) \cdot \mathbb{E}\Bigg[ \displaystyle\sum_{u \in \mathbb{D}_{k+1} (v)} \displaystyle\sum_{w \in \mathbb{D}_{\ell-k-1} (v)} \big| R_{vw}^{(v_-)} \big|^s \cdot \one_{\mathscr{D} (u,w)} \Bigg].
		\end{flalign*} 
	
		\noindent Applying this bound $\ell$ times (to increase the parameter $k$ from $0$ to $\ell$), we deduce that there exists a constant such that
		\begin{flalign*}
			\mathbb{E} \Bigg[ \displaystyle\sum_{w \in \mathbb{D}_{\ell} (v)} \big| R_{vw}^{(v_-)} \big|^s \cdot \one_{\mathscr{D} (v, w)} \Bigg] \ge (1 - c^{-1} C \cdot \delta)^{\ell} \cdot \mathbb{E}\Bigg[ \displaystyle\sum_{w \in \mathbb{D}_{\ell} (v)} \big| R_{vw}^{(v_-)} \big|^s \Bigg] = (1 - c^{-1} C \cdot  \delta)^{\ell} \cdot \Phi_{\ell} (s),
		\end{flalign*} 
	
		\noindent where for the last equality we again used the fact that $R_{vw}^{(v_-)}$ for $w \in \mathbb{D}_{\ell} (v)$ and $R_{0w}$ for $w \in \mathbb{V} (\ell)$ have the same law; this establishes the lemma.
	\end{proof} 

	Next we establish \Cref{ry0estimate2} and \Cref{ry0estimate3}.

	\begin{proof}[Proof of \Cref{ry0estimate2}]

	The proof of this lemma is similar to that of \Cref{smultiplicative}. In particular, observe from \Cref{rproduct} and \eqref{qvv} that 
	\begin{flalign}
		\label{expectationwt0}
		\begin{aligned}
		 \mathbb{E} \Bigg[ & \displaystyle\sum_{u \in \mathbb{V} (k)} \displaystyle\sum_{w \in \mathbb{D}_{\ell-k} (u)} |R_{0w}|^s \cdot \one_{\mathscr{D} (u, w)} \Bigg] \\
		& = \mathbb{E} \Bigg[ \displaystyle\sum_{u \in \mathbb{V}(k)} \big| R_{0u_-}^{(u)} \big|^s \cdot |T_{u_- u}|^s \cdot | R_{uu}|^s \displaystyle\sum_{v \in \mathbb{D} (u)} |T_{uv}|^s \displaystyle\sum_{w \in \mathbb{D}_{\ell-k-1} (v)} \big| R_{vw}^{(u)} \big|^s \cdot \one_{\mathscr{D} (u, w)} \Bigg] \\
		& = \mathbb{E} \Bigg[ \displaystyle\sum_{u \in \mathbb{V} (k)} \big| R_{0u_-}^{(u)} \big|^s \cdot |T_{u_- u}|^s \cdot \mathbb{E}_u \bigg[ \displaystyle\sum_{v \in \mathbb{D} (u)} \displaystyle\frac{|T_{uv}|^s \cdot \one_{|T_{uv}| \ge \omega}}{\big| z + T_{u_- u}^2 R_{u_- u_-}^{(u)} + T_{uv}^2 R_{vv}^{(u)} + K_{u, v} \big|^s} \\
			& \qquad \qquad \qquad \qquad \qquad \qquad \qquad \qquad \qquad \times \displaystyle\sum_{w \in \mathbb{D}_{\ell - k-1} (v)} \big| R_{vw}^{(u)} \big|^s \cdot \one_{\mathscr{D} (v, w)} \bigg] \Bigg], 
			\end{aligned} 
	\end{flalign} 
	
	\noindent where 
	\begin{flalign}
	\label{kuv1}
		K_{u, v} = \displaystyle\sum_{\substack{v' \in \mathbb{D} (u) \\ v' \ne v, u_-}} T_{uv'}^2 R_{v' v'}^{(u)}. 
	\end{flalign} 
	 
	\noindent To estimate the inner expectation in \eqref{slm}; fix a real number $\varsigma \in (0, 1)$ (to be determined later); and observe from \Cref{sumaxi} (and \Cref{tuvalpha} and \Cref{ar}) that $K_{u, v}$ is an $\frac{\alpha}{2}$-stable law satisfying \Cref{sqk}. Moreover, it is independent from the random variables $T_{u_- u}$, $T_{uv}$, $R_{u_- u_-}^{(u)}$, $R_{vv}^{(u)}$, and $R_{vw}^{(u)}$, as well as the event $\mathscr{D} (v, w)$. Thus, (the $\chi = s$ case of) the first and second parts of \Cref{expectationsum1} together yield a constant $C_1 > 1$ such that 
	\begin{flalign*} 
		& \mathbb{E}_u \Bigg[ \displaystyle\sum_{v \in \mathbb{D} (u)} \displaystyle\frac{|T_{uv}|^s}{\big| z + T_{u_- u}^2 R_{u_- u_-}^{(u)} + T_{uv}^2 R_{vv}^{(u)} + K_{u, v}\big|^s} \cdot \one_{|T_{uv}| \ge \omega}  \displaystyle\sum_{w \in \mathbb{D}_{\ell-k-1} (v)} \big| R_{vw}^{(u)} \big|^s \cdot \one_{\mathscr{D} (v, w)} \Bigg] \\
		& \quad \ge \Big( \big| z + T_{u_- u}^2 R_{u_- u_-}^{(u)} \big| + 1 \Big)^{-s} \cdot \mathbb{E}_u \Bigg[ \Big( \big|R_{u_+ u_+}^{(u)} \big|^{(\alpha-s)/2} - C_1 \omega^{s-\alpha} \Big) \cdot \one_{|R_{u_+ u_+}^{(u)}| \le 1/\varsigma} \\
		& \qquad \qquad \qquad \qquad \qquad \qquad \qquad \qquad \qquad \qquad  \times \displaystyle\sum_{w \in \mathbb{D}_{\ell-k-1} (u_+)} \big| R_{u_+ w}^{(u)} \big|^s \cdot \one_{\mathscr{D}(u_+, w)} \Bigg],	
	\end{flalign*} 
	
	\noindent for any child $u_+$ of $u$. Inserting this into \eqref{expectationwt0}; using the fact that $\big( R_{u_+ u_+}^{(u)}, R_{u_+ w}^{(u)} \big)$ for $w \in \mathbb{D}_{\ell-k-1} (u_+)$ has the same law as $(R_{00}, R_{0w})$ for $w \in \mathbb{V} (\ell-k-1)$; and taking $\omega$ sufficiently small so that $\varsigma^{(s-\alpha)/2} > 2 C_1 \omega^{s-\alpha}$ yields
		\begin{flalign}
			\label{sumlk} 
			\begin{aligned}
			\mathbb{E} \Bigg[ \displaystyle\sum_{u \in \mathbb{V} (k)} \displaystyle\sum_{w \in \mathbb{D}_{\ell-k} (u)} |R_{0w}|^s \cdot \one_{\mathscr{D} (0, u)} \Bigg] & \ge \displaystyle\frac{1}{2} \cdot \varsigma^{(s-\alpha)/2} \cdot \mathbb{E} \Bigg[ \displaystyle\sum_{u \in \mathbb{V} (k)} \displaystyle\frac{\big| R_{0u_-}^{(u)} \big|^s \cdot |T_{u_- u}|^s}{\Big( \big| z + T_{u_- u}^2 R_{u_- u_-}^{(u)} \big| + 1 \Big)^s} \Bigg]    \\
			& \qquad \times \mathbb{E} \Bigg[  \displaystyle\sum_{w \in \mathbb{V} (\ell - k-1)} |R_{0w}|^s \cdot \one_{\mathscr{D} (0, w)} \cdot \one_{|R_{00}| \le 1/\varsigma} \Bigg].
			\end{aligned} 
	\end{flalign} 
	
	Next, by \eqref{rv4}, for sufficiently small $\varsigma$ we have 
	\begin{flalign}
		\label{sumwvlk1}
		\mathbb{E} \Bigg[ \displaystyle\sum_{w \in \mathbb{V} (\ell-k-1)} |R_{0w}|^s \cdot \one_{\mathscr{D} (0, w)} \cdot \one_{|R_{00}| \le 1/\varsigma} \Bigg] \ge \displaystyle\frac{1}{2} \cdot \mathbb{E} \Bigg[ \displaystyle\sum_{w \in \mathbb{V} (\ell-k-1)} |R_{0w}|^s \cdot \one_{\mathscr{D} (0, w)} \Bigg].
	\end{flalign}

	\noindent Additionally, by \Cref{m2}, there exists a constant $c_1 > 0$ such that 
	\begin{flalign}
		\label{sumuvk1} 
		\mathbb{E} \Bigg[ \displaystyle\sum_{u \in \mathbb{V} (k)} \displaystyle\frac{\big| R_{0u_-}^{(u)} \big|^s \cdot |T_{u_- u}|^s}{\Big( \big| z + T_{u_- u}^2 R_{u_- u_-}^{(u)} \big| + 1 \Big)^s} \Bigg] \ge c_1 \cdot \mathbb{E} \Bigg[ \displaystyle\sum_{u \in \mathbb{V} (k-1)} |R_{0u}|^s \Bigg].
	\end{flalign}

	\noindent Combining \eqref{sumlk}, \eqref{sumwvlk1}, and \eqref{sumuvk1} then yields the lemma.	 		
	\end{proof}
		
	\begin{proof}[Proof of \Cref{ry0estimate3}]
		
		The proof of this lemma will be similar to that of \Cref{ry0estimate2}. Indeed, following \eqref{expectationwt0}, we obtain
		\begin{flalign}
		\label{expectationwt02}
		\begin{aligned}
		 \mathbb{E} \Bigg[ & \displaystyle\sum_{u \in \mathbb{V} (k)} \displaystyle\sum_{w \in \mathbb{D}_{\ell-k} (u)} |R_{0w}|^s \cdot \one_{|T_{uu_+}| < \omega} \cdot \one_{\mathscr{D} (u_+, w)} \Bigg] \\
		& = \mathbb{E} \Bigg[ \displaystyle\sum_{u \in \mathbb{V} (k)} \big| R_{0u_-}^{(u)} \big|^s \cdot |T_{u_- u}|^s \cdot \mathbb{E}_u \bigg[ \displaystyle\sum_{v \in \mathbb{D} (u)} \displaystyle\frac{|T_{uv}|^s \cdot \one_{|T_{uv}| < \omega}}{\big| z + T_{u_- u}^2 R_{u_- u_-}^{(u)} + T_{uv}^2 R_{vv}^{(u)} + K_{u, v} \big|^s} \\
			& \qquad \qquad \qquad \qquad \qquad \qquad \qquad \qquad \qquad \times \displaystyle\sum_{w \in \mathbb{D}_{\ell - k-1} (v)} \big| R_{vw}^{(u)} \big|^s \cdot \one_{\mathscr{D} (v, w)} \bigg] \Bigg], 
			\end{aligned} 
	\end{flalign}
	
	\noindent where $K_{u, v}$ is as in \eqref{kuv1}. Using the fact that $K_{u, v}$ is independent from $T_{u_- u}$, $R_{u_- u_-}^{(u)}$, $T_{uv}$, $R_{vv}^{(u)}$, $R_{vw}^{(u)}$, and $\mathscr{D} (v, w)$, we obtain from (the $\chi = s$ case of) the second part of \Cref{expectationsum1} that there exists a constant $C_1 > 1$ such that 
	\begin{flalign*}
		 \mathbb{E}_u \Bigg[ \displaystyle\sum_{v \in \mathbb{D} (u)} & \displaystyle\frac{|T_{uv}|^s \cdot \one_{|T_{uv}| < \omega}}{\big| z + T_{u_- u}^2 R_{u_- u_-}^{(u)} + T_{uv}^2 R_{vv}^{(u)} + K_{u, v} \big|^s} \cdot \displaystyle\sum_{w \in \mathbb{D}_{\ell - k-1} (v)} \big| R_{vw}^{(u)} \big|^s \cdot \one_{\mathscr{D} (v, w)} \Bigg] \\
			& \le C_1 \omega^{s-\alpha} \cdot \Big( \big| z + T_{u_- u}^2 R_{u_- u_-}^{(u)} \big| + 1 \Big)^{-s} \cdot \mathbb{E} \Bigg[ \displaystyle\sum_{w \in \mathbb{V} (\ell - k-1)} | R_{0w}|^s \cdot \one_{\mathscr{D} (0, w)} \Bigg], 
	\end{flalign*}	
	
	\noindent where we have used the fact that $R_{vw}^{(u)}$ for $w \in \mathbb{D}_{\ell-k-1} (u)$ has the same law as $R_{0w}$ for $w \in \mathbb{V} (\ell-k-1)$. Inserting this into \eqref{expectationwt02} yields
	\begin{flalign}
	\label{sumuw1}
	\begin{aligned} 
		\mathbb{E} \Bigg[ \displaystyle\sum_{u \in \mathbb{V} (k)} \displaystyle\sum_{w \in \mathbb{D}_{\ell-k} (u)} & |R_{0w}|^s \cdot \one_{|T_{uu_+}| < \omega} \cdot \one_{\mathscr{D} (u, w)}\Bigg] \\
		& < C_1 \omega^{s-\alpha} \cdot \mathbb{E} \Bigg[ \displaystyle\sum_{u \in \mathbb{V} (k)} \big| R_{0u_-}^{(u)} \big|^s \cdot |T_{u_- u}|^s \cdot \Big( \big| z + T_{u_- u}^2 R_{u_- u_-}^{(u)} \big| + 1 \Big)^{-s} \Bigg] \\
		& \qquad \times  \mathbb{E} \Bigg[ \displaystyle\sum_{w \in \mathbb{V} (\ell-k-1)} |R_{0w}|^s \cdot \one_{\mathscr{D} (0, w)} \Bigg].
	\end{aligned} 
	\end{flalign}
	
	Applying \Cref{m2} yields a constant $C_2 > 1$ such that
	\begin{flalign*}
	 \mathbb{E} \Bigg[ \displaystyle\sum_{u \in \mathbb{V} (k)} \big| R_{0u_-}^{(u)} \big|^s \cdot |T_{u_- u}|^s \cdot \Big( \big| z + T_{u_- u}^2 R_{u_- u_-}^{(u)} \big| + 1 \Big)^{-s} \Bigg] \le C_2 \cdot \mathbb{E} \Bigg[ \displaystyle\sum_{u \in \mathbb{V} (k-1)} | R_{0u}|^s \Bigg],
	\end{flalign*}
	
	\noindent which together with \eqref{sumuw1} yields
	\begin{flalign*}
		\mathbb{E} \Bigg[ \displaystyle\sum_{u \in \mathbb{V} (k)} & \displaystyle\sum_{w \in \mathbb{D}_{\ell-k} (u)} |R_{0w}|^s \cdot \one_{|T_{uu_+}| < \omega} \cdot \one_{\mathscr{D} (u, w)}\Bigg] \\
		& < C_1 C_2 \omega^{s-\alpha} \cdot \mathbb{E} \Bigg[ \displaystyle\sum_{u \in \mathbb{V} (k-1)} | R_{0u}|^s \Bigg] \cdot \mathbb{E} \Bigg[ \displaystyle\sum_{w \in \mathbb{V} (\ell-k-1)} |R_{0w}|^s \cdot \one_{\mathscr{D} (0, w)} \Bigg].
	\end{flalign*}
	
	\noindent Now the lemma follows from taking $\omega$ sufficiently small so that $C_1 C_2 \omega^{s-\alpha} < \delta$.
	\end{proof}

	\section{Expected Number of Large Resolvent Entries}
	
	\label{EstimateN}

	In this section we estimate the expectation of the number of vertices $v \in \mathbb{V} (L)$ for which $|R_{0v}|$ is above some given threshold $e^{t L}$. We begin in \Cref{EstimateNL} by using the restricted fractional moment estimate \Cref{expectationr0vsd} to bound expectations of counts for the number of large resolvent entries. We then in \Cref{EstimateNLkappa} provide a result (used in \Cref{ProbabilityNmu0} below) lower bounding the number of moderate resolvent entries, of order $e^{\kappa L}$ for some possibly negative $\kappa$ bounded from below. Throughout this section, we recall the notation and conventions from \Cref{EventR}.

	\subsection{Estimates for the Expected Number of Large $R_{vw}$}
	
	\label{EstimateNL}
	
	In this section we provide lower bounds on the expected number of large off-diagonal resolvent entries. We begin with the following definition, which provides notation for the set and number of vertices for which the off-diagonal resolvent entry $|R_{vw}|$ is bounded below by some threshold $e^{tL}$ (and on which the event $\mathscr{B} (v, w)$ from \Cref{br0vw} holds), where $L$ denotes the distance between $v$ and $w$. 
	
	\begin{definition} 
		
		\label{ns} 
		
		Fix a vertex $v \in \mathbb{V}$; an integer $\ell \ge 1$; and real numbers $t \in \mathbb{R}$, $\omega \in (0, 1)$, and $\Omega > 1$. Define the vertex set $\mathcal{S}_{\ell} (t) = \mathcal{S}_{\ell} (t; v) = \mathcal{S}_{\ell} (t; v; \omega; \Omega) = \mathcal{S}_{\ell} (t; v; \omega; \Omega; z)$ and integer $\mathcal{N}_{\ell} (t) = \mathcal{N}_{\ell} (t; v) = \mathcal{N}_{\ell} (t; v; \omega; \Omega) = \mathcal{N}_{\ell} (t; v; \omega; \Omega; z)$ by 
		\begin{flalign*}
			\mathcal{S}_{\ell} (t) = \Big\{ w \in \mathbb{D}_{\ell} (v) : \big| R_{vw}^{(v_-, w_+)} \big| \cdot \one_{\mathscr{B} (v, w; \omega; \Omega)} \ge e^{t \ell} \Big\}; \qquad \mathcal{N}_{\ell} (t) = \big| \mathcal{S}_{\ell} (t) \big|.
		\end{flalign*} 
	\end{definition} 
	
	Before proceeding, it will be useful to have the following tail estimate on $\mathcal{N}_{\ell} (t)$; it is essentially a consequence of \Cref{treenumber} below, which is a corresponding tail estimate for the number of vertices in a Galton--Watson tree.
	
	\begin{lem} 
	
	\label{nlestimate} 
	
	For any real numbers $\omega \in (0, 1)$, there exists a constant $C = C (\omega) > 1$ such that the following holds. For any integer $\ell \ge  1$, vertex $v \in \mathbb{V}$, and real numbers $t \in \mathbb{R}$ and $A, \Omega \ge 1$, we have 
	\begin{flalign*} 
		\mathbb{P} \big[ \mathcal{N}_{\ell} (t) \ge A \cdot C^{\ell} \big] \le 3e^{-A/2}.
	\end{flalign*} 
	\end{lem} 	

	\begin{proof}

	Let $\mathcal{V} (k) = \big\{ w \in \mathbb{D}_k (v) : \text{$\mathscr{D} (v, w)$ holds} \big\}$, and set $\mathcal{V} = \bigcup_{k=0}^{\infty} \mathcal{V} (k)$. Then, $\mathcal{V} (\ell) \subseteq \mathcal{S}_{\ell} (t)$, as $\mathscr{B} (v, w) \subseteq \mathscr{D} (v, w)$ (by \eqref{br0vw}); hence, $\big| \mathcal{V} (\ell) \big| \ge \mathcal{N}_{\ell} (t)$. Moreover, $\mathcal{V}$ is a Galton--Watson tree (see \Cref{EstimateTreeVertex}) with parameter $\lambda = \alpha \cdot \int_{\omega}^{\infty} x^{-\alpha-1} dx = \omega^{-\alpha}$. Thus, it follows from \Cref{treenumber} that 
	\begin{flalign*} 
		\mathbb{P} \Big[ \big| \mathcal{V} (\ell) \big| \ge A \cdot (2\lambda)^{\ell} \Big] \le 3 e^{-A/2},
	\end{flalign*}
	
	\noindent from which the lemma follows from taking $C = 2 \lambda$ and the fact that $\mathcal{N}_{\ell} (t) \le \mathcal{V} (\ell)$.
	\end{proof} 	
	
	We next state the following proposition that provide lower bounds on $\mathbb{E} \big[ \mathcal{N}_L (t) \big]$.
	
	\begin{prop}
		
		\label{nestimate} 
		
		For any real number $\delta \in (0, 1)$, there exist constants $C = C (\alpha ; s) > 1$ (independent of $\delta$), $L_0 = L_0 (\alpha ; s ; \delta) > 1$, $\omega = \omega (\alpha ; s ; \delta) \in (0, 1)$; and $\Omega = \Omega (\alpha ; s ;  \delta) > 1$ such that the following holds. Fix an integer $L \ge L_0$, and a vertex $v \in \mathbb{V}$; set $\ell = \ell (v)$. Then, there exists a real number $t_0 = t_0 (\alpha; s; z; \ell; L) \in [-C,0]$ such that 
		\begin{flalign}
			\label{expectationm}
			L^{-1} \log \mathbb{E} \big[ \mathcal{N}_L (t_0; v; \omega; \Omega; z) \big] \ge \varphi (s; z) - s t_0 - \delta.
		\end{flalign}

	\end{prop}
	
	\begin{rem} 
	
	\label{t0} 
	
	The real number $t_0$ from \Cref{nestimate} can in principle depend (fairly discontinuously) on the parameters $(z, \ell, L)$, but the fact that $t_0 \in [-C, 0]$ indicates that it is bounded independently of them.  
	
	\end{rem}

	To establish \Cref{nestimate}, we require the following tail bound, stating that the contribution to the truncated expectation considered in \Cref{expectationr0vsd} does not arise if $\big| R_{vw}^{(v_-, w_+)} \big|$ is either too large or small.
	
	\begin{lem} 
		
		\label{sumlargesmallr} 
		
		Adopting the notation of \Cref{nestimate}, we have  
		\begin{flalign}
			\label{wslcc} 
			\begin{aligned} 
			\mathbb{E} \Bigg[ \displaystyle\sum_{w \in \mathcal{S}_L (C)} \big| R_{vw}^{(v_-, w_+)} \big|^s \Bigg] & \le e^{-L} \cdot \exp \big( L \cdot \varphi (s; z) \big); \\
			\mathbb{E} \Bigg[ \displaystyle\sum_{w \notin \mathcal{S}_L (-C)} \big| R_{vw}^{(v_-, w_+)} \big|^s \Bigg] & \le e^{-L} \cdot \exp \big( L \cdot \varphi (s; z) \big).
			\end{aligned}
		\end{flalign}
	\end{lem} 

	\begin{proof} 
		
		The first and second bounds in \eqref{wslcc} will follow from comparing the $s$-th moments considered in their left sides to an $s'$-th moment for $s' > s$ or an $s''$-th moment for $s'' < s$, respectively. Let us begin by showing the first bound in \eqref{wslcc}. Fix $s' = \frac{s+1}{2} > s$; by \eqref{limitr0j2} and the second part of \Cref{ryestimate}, there exists a constant $C_1 > 0$ such that 
		\begin{align*} 
			\mathbb{E} \Bigg[ \displaystyle\sum_{w \in \mathcal{S}_L (C)} \big| R_{vw}^{(v_-, w_+)} \big|^{s'} \Bigg] &\le \E \left[\sum_{w \in \mathbb D_L (v)} |R_{vw}^{(v_-, w_+)}|^{s'} \cdot \one_{\mathscr G_0 (w)}\right] \\
			&\le C_1 \cdot \E\left[ \sum_{u \in \mathbb D_{L-1} (v)} | R_{vu}^{(v_-)}|^{s'} \right]\\ 
			& \le C_1 \cdot \exp \big( (L-1) \cdot \varphi (s'; z) \big) \\
			&\le C_1 \cdot \exp \big( L \cdot \varphi (s'; z) \big) \le C_1 \cdot \exp \big( L \cdot \varphi (s; z) \big),
		\end{align*} 
	
		\noindent where in the last line we used  the fact that $\varphi$ is nonincreasing in $s$ (by the second part of \Cref{limitr0j}). We also used \Cref{estimatemoment1} to bound $\phi(s;z) < C_1$ in order to replace $L-1$ by $L$ in the exponent. Since $\big| R_{vw}^{(v_-, w_+)} \big| \ge e^{CL}$ for each $w \in \mathcal{S}_L (C)$, it follows that 
		\begin{flalign*}
			e^{C (s' - s) L} \cdot \mathbb{E} \Bigg[ \displaystyle\sum_{w \in \mathcal{S}_L (C)} \big| R_{vw}^{(v_-, w_+)} \big|^s \Bigg] \le C_1 \cdot \exp \big( L \cdot \varphi (s; z) \big),
		\end{flalign*} 
	
		\noindent which yields the first statement of \eqref{wslcc}, after selecting $C = C(s, \varepsilon, B) > \frac{2}{s' - s}$. 
	
		Recall from the second part of \Cref{limitr0j} that $\varphi (s; z)$ is (weakly) convex in $s$; in particular, it is differentiable almost everywhere. Moreover, the fourth part of \Cref{limitr0j} implies that $\lim_{s \rightarrow \alpha} \varphi (s; z) = \infty$. Thus, there exists $s'' \in \big( \alpha, \frac{s+\alpha}{2} \big)$ sufficiently close to $\alpha$ such that the following two properties hold. First, $\varphi (s; z)$ is differentiable at $s''$; set $t = \partial_s \varphi (s''; z) \le 0$, where the nonpositivity of $t$ follows from the second part of \Cref{limitr0j}). Second, $\varphi (s''; z) + t (s'' - s) \le \varphi (s; z) - 2$.
		
		Then, again by \eqref{limitr0j2} we have 
		\begin{flalign*} 
			\mathbb{E} \Bigg[ \displaystyle\sum_{w \notin \mathcal{S}_L (t)} \big| R_{vw}^{(v_-, w_+)} \big|^{s''} \cdot \one_{\mathscr{B} (v, w)}  \Bigg] \le C_1 \cdot \exp \big( L \cdot \varphi (s''; z) \big).
		\end{flalign*} 
		
		\noindent Together with the bound $\big| R_{vw}^{(v_-, w_+)} \big|^s \le e^{(s - s'') tL} \cdot \big| R_{vw}^{(v_-, w_+)} \big|^{s''}$ (as $t \le 0$), we deduce
		\begin{flalign*} 
			\mathbb{E} \Bigg[ \displaystyle\sum_{w \notin \mathcal{S}_L (t)} \big| R_{vw}^{(v_-, w_+)} \big|^s \cdot \one_{\mathscr{B} (v, w)} \Bigg] & \le C_1 \cdot \exp \Big( L \cdot \big( \varphi (s''; z) + t (s-s'')  \big) \Big) \\
			& \le C_1 e^{-2L} \cdot \exp \big( L \cdot \varphi (s; z) \big),
		\end{flalign*} 
		
		\noindent where in the last bound we used the fact that $\varphi (s''; z) + t (s'' - s) \le \varphi (s; z) - 2$. Taking $L$ sufficiently large (and $C > |t|$), this yields the second bound in \eqref{wslcc}. 	
	\end{proof} 

	Now we can establish \Cref{nestimate}. 
	
	\begin{proof}[Proof of \Cref{nestimate}]
		
		Observe that $\big| R_{vw}^{(v_-, w_+)} \big| \cdot \one_{\mathscr{B} (v, w)} \le e^{s (t + \delta/2) L}$ for any $w \in \mathcal{S} (t) \setminus \mathcal{S} \big( t + \frac{\delta}{2} \big)$. Thus, by \Cref{expectationr0vsd} and \Cref{sumlargesmallr}, there exist constants $C_1 > 1$, $c_1 > 0$, $\omega = \omega (\delta) \in (0, 1)$, and $\Omega = \Omega (\delta) > 1$ such that   
		\begin{flalign*}
			\displaystyle\sum_{t \in \mathcal{A}} e^{(t+\delta/2) L} \cdot \mathbb{E} \bigg[ \mathcal{N}_L (t) - \mathcal{N}_L \Big( t + \displaystyle\frac{\delta}{2} \Big) \bigg] & \ge \displaystyle\sum_{t \in \mathcal{A}} \mathbb{E} \Bigg[ \displaystyle\sum_{\substack{v \in \mathcal{S} (t + \delta / 2) \\ v \notin \mathcal{S} (t)}} \big| R_{vw}^{(v_-, w_+)} \big|^s \cdot \one_{\mathscr{B} (v, w; \omega; \Omega)} \Bigg] \\
			& \ge \mathbb{E} \Bigg[ \displaystyle\sum_{v \in \mathcal{S} (-C) \setminus \mathcal{S} (C)} \big| R_{vw}^{(v_-, w_+)} \big|^s \cdot \one_{\mathscr{B} (v, w; \omega; \Omega)} \Bigg] \\
			& \ge c_1 \cdot \exp \Bigg( L \cdot \bigg( \varphi (s; z)  - \displaystyle\frac{\delta}{4} \bigg) \Bigg),
		\end{flalign*} 		
	
		\noindent where 
		\begin{flalign*} 
			\mathcal{A} = \bigg\{ -C, \displaystyle\frac{\delta}{2} - C, \delta - C, \ldots , \displaystyle\frac{D \delta}{2} - C \bigg\}, \qquad \text{and} \qquad D = \bigg\lceil \displaystyle\frac{4C}{\delta} \bigg\rceil.
		\end{flalign*} 
	
		\noindent Hence, there exists some $t_0 = t_0 (\alpha; s; z; \ell; L) \in \mathcal{A}$  such that 
		\begin{flalign*} 
			\mathbb{E} \big[ \mathcal{N}_L (t_0) \big] \ge \displaystyle\frac{c_1}{D} \cdot \exp \Bigg( L \cdot \bigg( \varphi (s; z) - s t_0 - \displaystyle\frac{3 \delta}{4} \bigg) \Bigg).
		\end{flalign*} 
	
		\noindent (Observe that it only depends on $v$ through $\ell$ since the random variables $R_{vw}^{(v_-, w_+)} \cdot \one_{\mathscr{B} (v, w; \omega; \Omega)}$ are identically distributed over $v \in \mathbb{V} (\ell)$.) Taking $L$ sufficiently large, this yields Equation \eqref{expectationm} and the proposition up to the condition that $t_0 \leq 0$ which we now check.
		
To this end, we first prove that, up to increasing our constants $C$ and $L_0$, we have that $t_0 \leq 4 \delta / (1- s)$. Indeed, if this is not the case, there exists $s < s' < 1$ such that $t_0 > 2 \delta / (s'-s)$. Then, if $L \geq C / \delta$ where $C$ is the constant of Theorem \ref{limitr0j2}(1), we would have
$$
\varphi(s) \ge \varphi(s') 
\ge \Phi_L (s'; z) - C L^{-1}
\ge s' t_0 + L^{-1} \log \mathbb{E}[\mathcal N_L (t_0)]  - C L^{-1}.$$
Here, the first inequality follows from Theorem \ref{limitr0j2}(1), the second from Theorem \ref{limitr0j2}(2), and the third follows from only considering $v \in S_L (t_0)$ in the sum defining $\Phi(s'; z)$.   Next, using Equation \eqref{expectationm}, we would get 
$$
\varphi(s) \ge (s' - s) t_0 + \varphi (s) -  \delta  - C /L  > \varphi (s).
$$
where the last inequality comes from $t_0 > 2(s'-s)^{-1} \delta$ and $L \geq C / \delta$.

We  thus have proved that $t_0 \in [-C,4\delta / (1-s)]$. If $t_0 \leq 0$, we are done. If $0 < t_0 \leq 4\delta / (1-s)$, we get from \eqref{expectationm} 
$$
			L^{-1} \log \mathbb{E} \big[ \mathcal{N}_L (t_0; v; \omega; \Omega; z) \big] \ge L^{-1} \log \mathbb{E} \big[ \mathcal{N}_L (t_0; v; \omega; \Omega; z) \big] \geq   \varphi (s; z) - s t_0 - \delta \geq \varphi (s; z) - 5 \delta / (1- s).
$$
Hence, at the cost of changing $\delta$ into $(1-s)\delta / 5$, this concludes the proof.
	\end{proof}

	We next have the following corollary, which lower bounds (though by a quantity that is exponentially small) the probability that $\mathcal{N}_L (t_0)$ is large. In \Cref{ProbabilityNmu0} below (see \Cref{mprobability}), we will explain a way of ``amplifying'' the below corollary to instead provide a high-probability bound that $\mathcal{N}_L (t_0)$ is large.
	
	\begin{cor} 
		
		\label{mudeltak} 
		
		Adopt the notation of \Cref{nestimate}. There is a constant $\mu_0 = \mu_0 (\alpha; s; z; \ell; L) \in [-C, C]$ so that 
		\begin{flalign*} 
			\mathbb{P} \big[ L^{-1} \log \mathcal{N}_L (t_0; v; \omega; \Omega; z) \ge \varphi (s; z) - st_0 + \mu_0 \big] \ge C^{-1} \cdot e^{-(\mu_0 + \delta) L}.
		\end{flalign*}
	\end{cor} 
	
	As in \Cref{t0}, we mention that the real number $\mu_0$ from \Cref{mudeltak} can in principle depend quite discontinuously on $(z, \ell, L)$, but that is bounded independently of them. The proof of \Cref{mudeltak} will follow from \Cref{nestimate}, together with the tail estimate \Cref{nlestimate} for $\mathcal{N}_L$.

	\begin{proof}[Proof of \Cref{mudeltak}]
	
		Let $C_1 > 1$ be some constant to be determined later, and assume to the contrary that 
		\begin{flalign}
		\label{nlc}
			\mathbb{P} \big[ L^{-1} \log \mathcal{N}_L (t_0; v; \omega; \Omega; z) \ge \varphi (s; z) - st_0 + \mu \big] \le C_1 \cdot e^{-(\mu + \delta) L}, \quad \text{for each $\mu \in [-2C_1, 2C_1]$},
		\end{flalign}
		
		\noindent which we will show contradicts \Cref{nestimate}. To that end, abbreviate $\mathcal{N} = \mathcal{N}_L (t_0; v; \omega; \Omega; z)$ and denote $\psi = \varphi (s; z) - st_0$. Further define the set 
		\begin{flalign*} 
		\mathcal{M} = \bigg\{ \mu = \displaystyle\frac{k \delta}{2} : k \in \mathbb{Z}, \psi + \mu \ge 0 \bigg\}.
		\end{flalign*} 
		
		\noindent Then, observe that 
		\begin{flalign*}
			\mathbb{E} [ \mathcal{N}] \le \displaystyle\sum_{\mu \in \mathcal{M}} \mathbb{E} \big[ \mathcal{N} \cdot \one_{L^{-1} \log \mathcal{N} \in [\psi + \mu - \delta/2, \psi + \mu]} \big] \le \displaystyle\sum_{\mu \in \mathcal{M}} e^{(\psi + \mu) L} \cdot \mathbb{P} \bigg[ L^{-1} \log \mathcal{N} \ge \psi + \mu - \displaystyle\frac{\delta}{2} \bigg].
		\end{flalign*} 
		
		Applying \eqref{nlc}, we obtain for sufficiently large $L$ that
		\begin{flalign}
		\label{psimu0} 
		\begin{aligned}
			\displaystyle\sum_{\mu \in \mathcal{M}} e^{(\psi + \mu) L} \cdot \mathbb{P} \bigg[ L^{-1} \log \mathcal{N} \ge \psi + \mu - \displaystyle\frac{\delta}{2} \bigg] \cdot \one_{|\mu| \le C_1} & \le \displaystyle\sum_{\mu \in \mathcal{M}} e^{(\psi + \mu) L} \cdot e^{- (\mu + \delta/2) L} \cdot \one_{|\mu| \le C_1} \\
			& \le 4C_1 \delta^{-1} \cdot e^{(\psi - \delta/2) L} \le e^{(\psi - \delta/3) L}.
			\end{aligned}
		\end{flalign}
		
		\noindent Additionally, applying \Cref{nlestimate} (with the $A$ there equal to $e^{\mu L/2}$ here) we obtain for $C_1$ sufficiently large (so that $\psi + \frac{\mu}{2} - 1 \ge C$ for $\mu > C_1$, where $C$ is from \Cref{nlestimate}) sufficiently large that
		\begin{flalign}
		\label{psimu1} 
		\begin{aligned} 
			\displaystyle\sum_{\mu \in \mathcal{M}} e^{(\psi + \mu) L} \cdot \mathbb{P} \bigg[ L^{-1} \log \mathcal{N} \ge \psi + \mu - \displaystyle\frac{\delta}{2} \bigg] \cdot \one_{\mu > C_1} & \le 3 \displaystyle\sum_{\mu \in \mathcal{M}} e^{(\psi + \mu)L} \cdot \exp \bigg( -\displaystyle\frac{e^{\mu L}}{2} \bigg) \cdot \one_{\mu > C_1} \\
			& \le e^{(\psi - \delta/3) L},
		\end{aligned} 
		\end{flalign}
		
		\noindent where the last inequality again follows from taking $C_1$ and $L$ sufficiently large.
		
		Combining \eqref{psimu0} and \eqref{psimu1} (with using the fact that $\mathcal{M}$ cannot contain any $\mu < -C_1$ for $C_1$ sufficiently large) gives 
		\begin{flalign*}
		\mathbb{E} [\mathcal{N}] \le e^{(\psi - \delta/4) L},
		\end{flalign*}		
		
		\noindent for sufficiently large $L$, which contradicts \Cref{nestimate} and establishes the corollary. 		
	\end{proof}

	\subsection{Lower Bound on Smaller Resolvent Entries}
	
	\label{EstimateNLkappa} 
	
	In this section we provide a lower bound on the number of resolvent entries that are not too small. We begin with the following definition that provides a minor modification of the quantity $\mathcal{N}_{\ell} (t)$ from \Cref{ns} (essentially omitting the event $\mathscr{D} (v, w)$ from \Cref{br0vw} that constrains the tree weights $T_{vw}$).
	
	\begin{definition} 
		
	\label{ml} 
	
	Fix a complex number $z \in \mathbb{H}$; an integer $\ell \ge 1$; and real numbers $\kappa \in \mathbb{R}_{> 0}$ and $\Omega > 1$.  Define the vertex set $\mathcal{S}_{\ell}' (\kappa) = \mathcal{S}_{\ell}' (\kappa; \Omega) = \mathcal{S}_{\ell}' (\kappa; \Omega; z)$ and integer $\mathcal{N}_{\ell}' (\kappa) = \mathcal{N}_{\ell}' (\kappa; \Omega; z)$ by 
	\begin{flalign*} 
		\mathcal{S}_{\ell}' (\kappa) = \Big\{ v \in \mathbb{V}(\ell) : \big| R_{0v}^{(v_+)} \big| \cdot \one_{\mathscr{G}_0 (v)} \ge \kappa^{\ell}, \big| R_{vv}^{(v_+)} \big| \le \Omega \Big\}; \qquad \mathcal{N}_{\ell}' (\kappa) = \big| \mathcal{S}_{\ell}' (\kappa) \big|. 
	\end{flalign*}

	\end{definition}  
	
	In this section we establish the following lemma, which states that $\mathcal{N}_{\ell}' (\kappa)$ can be made arbitrarily large by taking $\kappa$ sufficiently small. It will be used in the proof of \Cref{mprobability} below. 
	
	\begin{lem} 
		
		\label{vertexk} 
		
		For any real numbers $\Delta > 1$ and $\varepsilon \in (0, 1)$, there exist constants $\Omega = \Omega (\Delta, \eps) > 1$ and $\kappa = \kappa (\Delta, \eps) > 0$ such that, for any integer $\ell \ge 2$, we have
		\begin{flalign*}
			\mathbb{P} \big[ \mathcal{N}_{\ell}' (\kappa) \ge \Delta^{\ell} \big] \ge 1 - \eps.
		\end{flalign*} 
	\end{lem} 
	
	To establish this lemma, let $\varpi = \varpi (\Delta, \eps) \in (0, 2^{-2/\alpha})$ and $\Theta = \Theta (\Delta, \eps, \varpi) > 1$ be real numbers to be fixed later. For any vertex $v \in \mathbb{V}$, define the vertex set $\mathcal{A}_v = \mathcal{A}_v (\varpi) \subset \mathbb{V}$ by
		\begin{flalign}
			\label{av0} 
			& \mathcal{A}_v = \mathcal{A}_v (\varpi) = \big\{ w \in \mathbb{D} (v) : |T_{vw}| \in [\varpi, 1] \big\},
		\end{flalign} 
	
		\noindent the set of children of $v$ whose weight to $v$ is at least $\varpi$. Further inductively define the vertex sets $\mathcal{U} (k) \subset \mathbb{V}$ for each integer $k \ge 0$, by setting 
		\begin{flalign}
			\label{u0k} 
			\mathcal{U} (0) = \{ v \}; \qquad \mathcal{U} (k+1) = \bigcup_{w \in \mathcal{U}(k)} \mathcal{A}_w; \qquad \mathcal{U} = \bigcup_{k = 0}^{\infty} \mathcal{U} (k),
		\end{flalign} 
	
		\noindent so that in particular all edge weights connecting vertices in $\mathcal{U}$ are (in magnitude) at least $\varpi$. For any vertex $v \in \mathcal{U}$, define the complex numbers $K_v$ and $Q_v$ by
		\begin{flalign*} 
			K_v = \displaystyle\sum_{w \in \mathbb{D} (v) \setminus \mathcal{A}_v} T_{vw}^2 R_{ww}^{(v)}; \qquad Q_v = z + T_{v_- v}^2 R_{v_- v_-}^{(v)} + \displaystyle\sum_{\substack{w \in \mathcal{A}_v \\ w \ne v_+}} T_{vw}^2 R_{ww}^{(v)}.
		\end{flalign*} 
	
		\noindent Then define events bounding $R_{v_- v_-}^{(v)}$, $K_v$, and $Q_v$, given by
		\begin{flalign}
			\label{p0} 
			\begin{aligned} 
			& \mathscr{P}_0 (v) = \big\{ \Theta^{-1} \le \big| R_{v_- v_-}^{(v)} \big| \le \varpi^{-1} \big\}; \qquad \mathscr{P}_1 (v) = \big\{ |K_v| + |Q_v| \le \Theta \big\}; \\
			& \mathscr{P}_2 (v) = \big\{ |K_v + Q_v| \ge \varpi \big\}; \qquad \qquad \quad  \mathscr{P} (v) = \mathscr{P}_0 (v) \cap \mathscr{P}_1 (v) \cap \mathscr{P}_2 (v).
			\end{aligned} 
		\end{flalign} 
	
		Recall that $\mathbb{P}_v$ denotes the conditional probability with respect to the subtree $\mathbb{T}_- (v)$. Then observe, for sufficiently small $\varpi$ and sufficiently large $\Theta$, we have
		\begin{flalign}
			\label{p02} 
			\mathscr{P}_0 (v) \subseteq \mathscr{P}_2 (v_-); \qquad \mathbb{P}_v \big[ \mathscr{P}_1 (v) \cap \mathscr{P}_2 (v) \big| \mathscr{P}_0 (v) \big] \ge 1 - \displaystyle\frac{\eps}{16},
		\end{flalign}
	
		\noindent where the first statement holds since $R_{v_- v_-}^{(v)} = (K_{v_-} + Q_{v_-})^{-1}$ (by \eqref{qvv}) and the second holds from \eqref{zuv} (and \Cref{tuvalpha} and \Cref{ar}). We then define the vertex sets $\mathcal{V} (k)$, for $k \ge 0$, and $\mathcal{V}_0 (\ell)$ by
		\begin{flalign}
			\label{vk} 
			\mathcal{V} (k) = \bigg\{ w \in \mathcal{U} (k) : \text{$\bigcap_{j = 0}^k \mathscr{P} \big( w_-^{(j)} \big)$ holds} \bigg\}; \qquad \mathcal{V}_0 (\ell) = \big\{ w \in \mathcal{V} (\ell) : \text{$\mathscr{G}_0 (w)$ holds} \big\},
		\end{flalign}
		
		\noindent where the $w_-^{(j)}$ are defined inductively, by setting $w_-^{(0)} = w$ and $w_-^{(j+1)} = \big( w_-^{(j)} \big)_-$ for each $j \ge 0$ (in words, $w^{(j)}$ is the $j$-th ancestor of $w$). 
		
	The following lemma indicates that, to lower bound $\mathcal{N}_{\ell}'$, it suffices to lower bound $\big| \mathcal{V}_0 (\ell) \big|$.

	\begin{lem} 
		
		\label{sv} 
		
		Under the above notation, we have $ \mathcal{V}_0 (\ell)\subseteq \mathcal{S}_{\ell}' (\varpi \Theta^{-2}; \varpi^{-1}; z)$. 
	\end{lem} 

	\begin{proof} 
	
		Observe for any $v \in \mathcal{V}_0 (\ell)$ that 
		\begin{flalign} 
			\label{r0vomega} 
			\big| R_{0v}^{(v_+)} \big| = \big| R_{00}^{(0+)} \big| \cdot \displaystyle\prod_{0 \preceq u \prec v} | T_{uu_+}| \cdot \big| R_{uu}^{(u_+)} \big| \ge \Theta^{-\ell + 1} \varpi^{\ell} \ge (\Theta^{-2} \varpi)^{\ell},
		\end{flalign} 
	
		\noindent where the first statement holds by \Cref{rproduct}; the second holds since $|T_{uu_+}| \ge \varpi$ (by \eqref{av0}, \eqref{u0k}, and \eqref{vk}) and $\big| R_{uu}^{u_+} \big| \le \Theta^{-1}$ (by the event $\mathscr{P}_0$ from \eqref{p0}, and \eqref{vk}); and the third holds since $\Theta > 1$. Further observe (from the event $\mathscr{P}_0$ in \eqref{p0}, and \eqref{vk}) that for any $v \in \mathcal{V}_0 (\ell)$ we have
		\begin{flalign*} 
			\big| R_{vv}^{(v_+)} \big| \le \Theta, \quad \text{and $\mathscr{G}_0 (v)$ holds}.
		\end{flalign*} 
	
		\noindent This, together with \eqref{r0vomega} implies the lemma. 
	\end{proof}

	Before lower bounding $\big| \mathcal{V}_0 (\ell) \big|$, we lower bound $\big| \mathcal{V} (m) \big|$. 
	
	\begin{lem} 
		
		\label{vestimate}
		For  real numbers $\Delta > 1$ and $\varepsilon \in (0, 1)$. There exist constants $\varpi_0(\Delta, \eps) > 0$ 
		and $\Theta_0(\Delta, \eps, \varpi)$ such that for
		$\varpi \in (0, \varpi_0)$ and $\Theta > \Theta_0$,
		we have for any integer $m \ge 1$ that
		\begin{flalign*} 
			\mathbb{P} \Big[ \big| \mathcal{V} (m) \big| \ge \Delta^m \Big] \ge 1 - \displaystyle\frac{\eps}{2}.
		\end{flalign*} 
	
	\end{lem} 

	\begin{proof} 
		
		Recall that $\mathbb{T} (k)$ denotes the part of $\mathbb{T}$ above its $k$-th level $\mathbb{V} (k)$, and that $\mathbb{P}_k$ denotes the probability conditional on $\mathbb{T} (k)$. We claim that there exists a constant $c > 0$ such that, for sufficiently small $\varpi$ and large $\Theta$ we have  
		\begin{flalign}
			\label{v0vk}
			\begin{aligned}
			& \mathbb{P} \big[ \mathcal{V} (0) = \{ 0 \} \big] \ge 1 - \displaystyle\frac{\eps}{8}; \qquad \mathbb{P} \bigg[ \big| \mathcal{V} (1) \big| \ge \displaystyle\frac{1}{4 \varpi^{1/\alpha}} \bigg] \ge 1 - \displaystyle\frac{\eps}{4}; \\
			&  \mathbb{P}_k \bigg[ \big| \mathcal{V} (k+1) \big| \ge \displaystyle\frac{1}{4 \varpi^{\alpha}} \cdot \big| \mathcal{V} (k) \big| \bigg] \ge 1 - \exp \big( - c \cdot |\mathcal{V} (k)| \big).
			\end{aligned} 
		\end{flalign}
	
		\noindent This would imply the lemma, as from \eqref{v0vk} we would deduce that    
		\begin{flalign*} 
			\mathbb{P} \Big[ \big| & \mathbb{V} (m) \big| \ge 4^{-m} \varpi^{-\alpha m} \Big] \\ 
			& \ge \mathbb{P} \bigg[ \big| \mathcal{V} (1) \big| \ge \displaystyle\frac{1}{4 \varpi^{\alpha}} \bigg] \cdot \displaystyle\prod_{k = 0}^{m - 1} \mathbb{P}_k \bigg[ \big| \mathcal{V} (k+1) \big| \ge \displaystyle\frac{1}{4 \varpi^{\alpha}} \cdot \mathcal{V} (k) \bigg| \big| \mathcal{V} (k) \big| \ge 4^{-k} \varpi^{-\alpha k} \bigg] \\
			& \ge \bigg( 1 - \displaystyle\frac{\eps}{4} \bigg) \cdot \displaystyle\prod_{k = 0}^{m - 1} \Bigg( 1 - \exp \bigg( -c \cdot \Big( \displaystyle\frac{1}{4 \varpi^{\alpha}} \Big)^k \bigg) \Bigg) \ge 1 - \displaystyle\frac{\eps}{2},
		\end{flalign*} 
	
		\noindent where the last inequality holds by taking $\varpi$ sufficiently small. The lemma then follows from having $4^{-m} \varpi^{-\alpha m} \ge \Delta^m$ by again taking $\varpi$ sufficiently small.
		
		Thus, it suffices to establish \eqref{v0vk}; its first bound follows from the second statement of \eqref{p02}. To establish its second and third, for any vertex $v \in \mathcal{V}$ let $n_v = \big| \mathcal{V} \cap \mathbb{D} (v) \big|$ denote the number of children of $v$ in $\mathcal{V}$. Observe for any vertex $v \in \mathcal{V} (k)$, the event $\mathscr{P}_2 (v)$ holds by \eqref{vk}. Hence, for any child $w \in \mathcal{A}_v$ (recall \eqref{av0}), $\mathscr{P}_0 (w)$ holds by the first statement of \eqref{p02}. Thus, for any $v \in \mathcal{V} (k)$, any child $w \in \mathbb{D} (v)$ of it is in $\mathcal{V} (k+1)$ if and only if $w \in \mathcal{A}_v$ and $\mathscr{P}_1 (w) \cap \mathscr{P}_2 (w)$ both hold. 
		
		Now, we claim (conditional, as above, on $\mathbb{T} (k)$) that 
		\begin{flalign}
			\label{nvomega} 
			\mathbb{P}_k \bigg[ n_v \ge \frac{1}{4} \cdot \varpi^{-\alpha} \bigg] \ge 1 - \frac{\eps}{8}, \qquad \text{for any $v \in \mathcal{V} (k)$}.
		\end{flalign} 
	
		\noindent This would imply the second bound of \eqref{v0vk} (through a union bound with the first bound there). It would also imply the third, as it would follow from a Chernoff estimate that there exists a constant $c_2 > 0$ such that the following holds for any integer $k \ge 1$. With probability at least $1 - e^{-c_2 |\mathcal{V}(k)|}$, the probability that there exist at least $\frac{3}{4} \cdot \big| \mathcal{V} (k) \big|$ vertices $v \in \mathcal{V} (k)$ for which $n_v \ge \frac{1}{4} \cdot \varpi^{-\alpha}$. Summing over all such vertices $v$, it would follow on this event that $\big| \mathcal{V} (k+1) \big| \ge \frac{1}{4} \cdot \varpi^{-\alpha} \cdot \big| \mathcal{V} (k) \big|$, hence showing the third bound of \eqref{v0vk}.
		
		To show \eqref{nvomega}, first observe that $|\mathcal{A}_v|$ (recall \eqref{av0}) is a Poisson random variable with parameter $\alpha \cdot \int_{\varpi}^1 x^{-\alpha-1} dx = \varpi^{-\alpha} - 1 \ge \varpi^{-\alpha/2}$. Thus, for sufficiently small $\varpi$ we have 
		\begin{flalign}
			\label{1av} 
			\mathbb{P} \bigg[ |\mathcal{A}_v| \ge \frac{1}{2} \cdot \varpi^{-\alpha/2} \bigg] \ge 1 - \frac{\eps}{16}.
		\end{flalign} 
	
		\noindent Additionally, by the second statement of \eqref{p02} we have that with probability at least $1 - \frac{\eps}{16}$, at least $\frac{1}{2}$ of the vertices in $w \in \mathcal{A}_v$ satisfy $\mathscr{P}_1 (w) \cap \mathscr{P}_2 (w)$, namely,
		\begin{flalign}
			\label{2av} 
			\mathbb{P} \Bigg[ \bigg| \Big\{ w \in \mathcal{A}_v : \text{$\mathscr{P}_1 (w) \cap \mathscr{P}_2 (w)$ holds} \Big\} \bigg| \ge \displaystyle\frac{1}{2} \cdot |\mathcal{A}_v| \Bigg] \ge 1 - \displaystyle\frac{\eps}{16}.
		\end{flalign}
	
		\noindent Together, \eqref{1av} and \eqref{2av} (with the previously mentioned fact that a vertex $w \in \mathscr{A}_v$ satisfying $\mathscr{P}_1 (w) \cap \mathscr{P}_2 (w)$ is in $\mathcal{V} (k + 1)$) imply \eqref{nvomega} and thus the lemma.
	\end{proof}

	Now we can establish \Cref{vertexk}. 
	
	\begin{proof}[Proof of \Cref{vertexk}]
		
		We adopt the notation from the proof of \Cref{vestimate} and set the $m$ there equal to $\ell - 1$ here. First observe that there exists a constant $c_1 > 0$ such that 
		\begin{flalign}
			\label{g0estimate} 
			\mathbb{P} \big[ \mathscr{G}_0 (v) \big] \ge c_1, \qquad \text{for any $v \in \mathbb{V}$}.
		\end{flalign}
	
		\noindent Now, define the event $\mathscr{U}$ on which $\big| \mathcal{V} (\ell - 1) \big| \ge \Delta^{\ell - 1}$; by \Cref{vestimate}, we have $\mathbb{P} [\mathscr{U}] \ge 1 - \frac{\eps}{2}$. Additionally, by \eqref{nvomega}, we have $\mathbb{P} \big[ n_v \ge \frac{1}{4} \cdot \varpi^{-\alpha} \big] \ge 1 - \frac{\eps}{8}$ for any $v \in \mathcal{V} (\ell - 1)$. We may $\eps < c_1$, so together with \eqref{g0estimate} this implies from a union bound that 
		\begin{flalign*} 
			\mathbb{P} \Bigg[ \big\{ n_v \ge \displaystyle\frac{1}{4 \varpi^{\alpha}} \bigg\} \cap \mathscr{G}_0 (v) \Bigg] \ge \displaystyle\frac{c_1}{2}.
		\end{flalign*}  
		
		Now we may proceed as in the proof of \Cref{vestimate}. Specifically, by a Chernoff estimate, we deduce the existence of a constant $c_2 > 0$ such that the following holds with probability at least $1 - e^{-c_2 |\mathcal{V} (\ell-1)|}$. There exist at least $\frac{c_1}{4} \cdot \big| \mathcal{V} (\ell - 1) \big|$ vertices $v \in \mathcal{V} (k-1)$ such that $n_v \ge \frac{1}{4} \cdot \varpi^{-\alpha}$. Combining this with the event $\mathscr{U}$ and taking $\Delta$ sufficiently large so that $e^{-c_2 \Delta} \le \frac{\eps}{2}$, it follows that 
		\begin{flalign*} 
			\mathbb{P} \bigg[ \big| \mathcal{V} (k) \big| \ge \displaystyle\frac{1}{4} \cdot \varpi^{-1/\alpha} \cdot \Delta^{\ell - 1} \Bigg] \ge \mathbb{P} [\mathscr{U}] - \exp (-c_2 \Delta^{\ell-1}) \ge 1 - \eps,
		\end{flalign*} 
	
		\noindent which gives the lemma.
	\end{proof}

	\section{Delocalization}
	
	\label{EstimateR}

	In this section we establish the delocalization result \Cref{rimaginary0}. We begin in \Cref{ProbabilityNmu0} by first lower bounding the quantity $\mathcal{N}_L (t)$ from \Cref{ns} (counting large off-diagonal resolvent entries) with high probability. Then, in \Cref{RLarge} we use this estimate to show \Cref{rimaginary0}, assuming \Cref{estimater} below, which we will establish in \Cref{ProofEstimateR}. 
	
	\subsection{Lower Bound on $\mathcal{N}_L$} 
	
	\label{ProbabilityNmu0}
	
	In this section we establish the following result that lower bounds $\mathcal{N}_L (t)$ with high probability; it maybe viewed as a variant of \Cref{mudeltak} with $\mu_0 = 0$ (and an improvement to the $\mu_0 = 0$ probability there, increasing $e^{-\delta L}$ to $1 - o(1)$). 
	
	\begin{prop} 
		
		\label{mprobability} 
		
		For any real numbers $\varsigma, \delta \in (0, 1/2)$, there exist constants $C > 1$ (independent of $\varsigma$ and $\delta$), $\omega = \omega (\varsigma, \delta) \in (0, 1)$, $\Omega  = \Omega (\varsigma, \delta) > 1$, and $L_0 = L_0 (\varsigma, \delta) > 1$ such that the following holds. Let $E, \eta \in \mathbb{R}$ be such that $\varepsilon \le |E| \le B$ and $\eta \in (0, 1)$, and set $z = E + \mathrm{i} \eta$; also fix an integer $L \ge L_0$. There exists a real number $t_0 = t_0 (\alpha, s, z, \omega, \Omega, L) \in [-C, 0]$ such that if $\varphi(s;z) \geq \delta$ we have
		\begin{flalign*}
			\mathbb{P} \big[ L^{-1} \log \mathcal{N}_L (t_0 - \delta; 0; \omega; \Omega; z) \ge  (1- \delta) \varphi (s; z) - st_0  - \delta \big] \ge 1 - \varsigma.
		\end{flalign*}
		
	\end{prop}
	
	Variants of \Cref{mprobability} were shown in \cite{aizenman2013resonant} through a second moment method. However, we are unaware how to implement this route in our setting, since the tree $\mathbb{T}$ is infinite (and any finite truncation of it that preserves independence would involve a random number of leaves). Instead, to establish \Cref{mprobability}, we ``amplify'' the low probability result \Cref{mudeltak} as follows. 
	
	First, we identify a subset of vertices $\mathcal{Z}_0 \subset \mathbb{V} (M)$ (for some $M \approx \delta^2 L$) on which $\big| R_{0v}^{(v_+)} \big|$ is likely not too small. The events on which $\mathcal{N}_M (t_0; v; \Omega; z) \ge e^{(\varphi (s; z) - st + \mu) M}$ are independent over $v \in \mathcal{Z}_0$. So, from \Cref{mudeltak} (denoting the $\mu_0$ there by $\mu$ here), if $|\mathcal{Z}_0| \gg e^{\mu M}$ then one will likely find many (about $e^{-\mu M} \cdot e^{(\varphi (s; z) - st + \mu) M} \cdot |\mathcal{Z}_0| = e^{(\varphi (s; z) - st) M} \cdot |\mathcal{Z}_0|$) vertices $v$ for which $\mathcal{N}_M (t_0; v; \Omega; z) \ge e^{(\varphi (s; z) - st + \mu) M}$. Considering the vertices in $\mathcal{S}_M (t_0; v; \Omega; z)$ yields a set $\mathcal{Z}_1$ of vertices $w$ with $\ell (w) \approx 2M$ for which $\big| R_{vw}^{(v_-, w_+)} \big| \ge e^{tM}$ and we likely have $|\mathcal{Z}_1| \ge e^{(\varphi (s; z) - st + \mu) M} \cdot |\mathcal{Z}_0|$. Repeating this procedure $\frac{L}{M}$ times, and using \Cref{rproduct}, will then yield with high probability about $e^{(\varphi (s; z) - st) L}$ vertices $v \in \mathbb{V} (L)$ for which $\big| R_{0v}^{(v_-)} \big| \ge e^{tL}$. 
		
	 Now let us implement this in more detail. We let $M = \lfloor \delta^2 L \rfloor$. For each integer $j \ge 0$, we set $M_j = (M+2) j$ and for notational simplicity assume that there exists an integer $\Theta \in [1, 2 \delta^{-2}]$ for which $L = M_{\Theta + 1} - 2 = M \cdot \Theta + M + 2 \Theta$.\footnote{We use this assumption in \Cref{zidefinition} when define the sets $\mathcal Z_k$ there to be $M$ levels apart. To treat general a $L$, one would consider a set defined analogously to the sets $\mathcal F_k$ specifically for vertices at level $L$, and use this set in all remaining arguments in this section.} For each integer $j \ge 1$, let $t_j$ denote the constant $t_0 (\alpha; s; z; M_j; M)$ from \Cref{nestimate} and $\mu_j$ denote the constant $\mu_0 (\alpha; s; z; M_j; M)$ from \Cref{mudeltak}. We further let $C_0 > 1$ denote a constant larger than the $C$ from \Cref{nestimate} and \Cref{mudeltak}. 
	
	We begin with the following two definitions for certain vertex sets on which diagonal resolvent entries are either not too small or more specifically bounded below by the threshold $e^{t_j M}$.

	\begin{definition} 
		
		\label{x0} 
		For any constant $c > 0$, define the vertex set $\mathcal{X}_0 (c) \subset \mathbb{V} (M)$ by setting 
		\begin{flalign*}
			\mathcal{X}_0 (c) = \big\{ v \in \mathbb{V} (M) : \big| R_{0v}^{(v_+)} \big| \cdot \one_{\mathscr{B} (0, v)} \ge c^M \big\}.
		\end{flalign*} 
		
		\noindent By \Cref{vertexk}, we may select $\kappa \in (0, 1)$ sufficiently small so that 
		\begin{flalign}
			\label{x0nu} 
			\mathbb{P} \Big[ \big| \mathcal{X}_0 (\kappa) \big| \ge e^{5 (C_0 + 1) M} \Big] \ge 1 - \displaystyle\frac{\varsigma}{2},
		\end{flalign}
	where $\varsigma\in(0,1)$ was chosen in \Cref{mprobability}.	
	\end{definition}

	\begin{definition} 
		
		\label{xvw}
		
		Recall the notation from \Cref{br0vw}, and fix an integer $j \ge 0$. For any vertex $v \in \mathbb{V} (M_j)$, define the vertex set $\mathcal{X}_1 (v) \subset \mathbb{V}$ by 
		\begin{flalign*}
			\mathcal{X}_1 (v) = \big\{ w \succ v : w \in \mathbb{D} (M_{j+1}), \big| R_{vw}^{(v_-, w_+)} \big| \cdot \one_{\mathscr{B} (v, w)} \ge e^{t_j M} \big\}.
		\end{flalign*}
		
	\end{definition} 
	
	Next we require the following events.
	
	\begin{definition} 
		
		\label{eventg1} 

	Recalling the notation of \Cref{eventg}, abbreviate $w = \mathfrak{c} (v)$ and set
		\begin{flalign*}
			\mathscr{G}_1 (v) = \mathscr{G}_0 (v) \cap \mathscr{G}_0 (w); \qquad \mathscr{G} (v) = \mathscr{G}_1 (v) \cap \Bigg\{ \bigg| \displaystyle\sum_{\substack{u \in \mathbb{D} (w) \\ u \ne \mathfrak{c} (w)}} R_{uu}^{(w)} T_{wu}^2 \bigg| < 5 \Bigg\}.
		\end{flalign*}
		
	\end{definition} 
	
	\noindent By the definition \eqref{t} of the $\{ T_{vw} \}$, and \Cref{ar}, there exists a constant $c_1 = c_1 (\alpha) > 0$ with
	\begin{flalign}
		\label{c1e} 
		\mathbb{P} \big[ \mathscr{G}_0 (v) \big] \ge c_1; \quad \mathbb{P} \big[ \mathscr{G}_1 (v) \big| \mathscr{G}_0 (v) \big] \ge c_1; \quad \mathbb{P} \big[ \mathscr{G} (v) \big| \mathscr{G}_1 (v) \big] \ge c_1, \qquad \text{so that $\mathbb{P} \big[ \mathscr{G} (v) \big] \ge c_1^3$}.
	\end{flalign} 
	
	We next define the vertex sets $\mathcal{Z}_k$ briefly mentioned above.
	
	\begin{definition}
		
		\label{zidefinition}
		
		For each integer $k \ge 0$, define the vertex set $\mathcal{Z}_k$ inductively as follows. Set 
		\begin{flalign*} 
			\mathcal{Z}_0 = \big\{ v \in \mathcal{X}_0 (\kappa) : \text{$\mathscr{G} (v)$ holds} \big\}.
		\end{flalign*} 
		
		\noindent For $k \ge 1$, define $\mathcal{Z}_k \subset \mathbb{V}$ to be the subset of vertices $v \in \mathbb{V} ( M_k - 2)$ satisfying the following four properties, where here we let $v^{(i)}$ denote the $i$-th parent of $v$, that is, $v^{(j)} = (v^{(j-1)})_-$ for each $j \ge 1$, with $v^{(0)} = v$. First, the event $\mathscr{G} (v)$ holds. Second, we have $v^{(M+2)} \in \mathcal{Z}_{k-1}$. Third, we have $v^{(M)} = \mathfrak{c} \big( \mathfrak{c} (v^{(M+2)}) \big)$. Fourth, we have $v \in \mathcal{X}_1 (v^{(M)})$.
		
		For $k \ge 1$, we then define the events
		\begin{flalign*}
			\mathscr{F}_0 = \big\{ |\mathcal{Z}_0| \ge e^{4 (C_0 + 1) M} \big\}; \quad \mathscr{F}_k  = \Bigg\{ |\mathcal{Z}_k| \ge \exp \bigg( kM \cdot \Big( \varphi (s; z) - \displaystyle\frac{s}{k} \displaystyle\sum_{j=1}^k t_j \Big) + M (2C_0 + 1 - 2 k \delta^2) \bigg) \Bigg\}.
		\end{flalign*} 
	\end{definition}

	We then have the following two lemmas. The first lower bounds $\mathbb{P} [\mathscr{F}_k]$, and the second lower bounds $|R_{0v}^{(v_+)}|$ for $v \in \mathcal{Z}_k$. We recall that $\kappa$ was chosen in \eqref{x0nu}.
	
	\begin{lem} 
		
		\label{estimatekf}
		
If $\varphi(s;z) \geq \delta$ and $\delta \in (0,1/2)$,	there exists a constant $c > 0$ (independent of $\delta$) such that we have $\mathbb{P} [\mathscr{F}_0] \ge 1 - \frac{3\varsigma}{4}$ and $\mathbb{P} [ \mathscr{F}_k] \ge \mathbb{P}[\mathscr{F}_{k-1}] - c^{-1} e^{-c \delta^2 M}$. 
	\end{lem}

	\begin{lem}
		
		\label{estimatenlf} 
		
		There is a constant $C > 1$ such that for any integer $k \ge 0$ and vertex $v \in \mathcal{Z}_k$ we have 
		\begin{flalign*}
			\big| R_{0v}^{(v_+)} \big| \ge \exp \Bigg( \bigg( \displaystyle\sum_{j=1}^k t_j + C (\log \kappa - \delta^4 k - C_0 - 1) \bigg) \cdot M \Bigg).
		\end{flalign*}
		
	\end{lem} 
	
	Given \Cref{estimatekf} and \Cref{estimatenlf}, we can quickly establish \Cref{mprobability}.
	
	\begin{proof}[Proof of \Cref{mprobability}]
		
		Set $M_k' = M_k - 2$. By \Cref{estimatenlf} (and the fact that $\mathcal{Z}_k \subset \mathbb{V} (M_k')$), there exists constants $C_1 > 1$ and $C_2 > 1$ such that on $\mathscr{F}_k$ we have
		\begin{flalign*} 
		  M_k'^{-1} \log \mathcal{N}_{M_k'} &\Bigg( \displaystyle\frac{1}{k} \displaystyle\sum_{j = 1}^k t_j + \frac{C_1}{k} (\log \kappa - \delta^4 k - C_0 - 1) \bigg) \\
			&\ge M_k'^{-1} \log \mathcal{N}_{M_k'}  \Bigg( \displaystyle\frac{M}{M_k'} \cdot \bigg( \displaystyle\sum_{j = 1}^k t_j + C_2 (\log \kappa - \delta^4 k - C_0 - 1) \Bigg)\Bigg) \\
			& \ge M_k'^{-1} \log |\mathcal{Z}_k| \ge  \Bigg( \varphi (s; z) - \displaystyle\frac{s}{k} \displaystyle\sum_{j=1}^k t_j - 2 \delta^2 \Bigg) \cdot \displaystyle\frac{kM}{M_k'}.
		\end{flalign*} 
		
		\noindent Taking $k = \Theta \le 2 \delta^{-2}$; applying \Cref{estimatekf}; and using the facts that $L = M_{\Theta}'$, $\mathcal{N}_L(x)$ is decreasing in $x$, and $-2 \delta^2 > -\delta$, we deduce the existence of constants $t_0 = \Theta^{-1} \sum_{j=1}^{\Theta} t_j$ and $c_1 > 0$ such that 
		\begin{flalign*}
			\mathbb{P} \big[ L^{-1} \log \mathcal{N}_L (t_0 - \delta) \ge ( 1- \delta) \varphi (s;z) - st_0 - \delta \big] \ge \mathbb{P} [ \mathscr{F}_{\Theta} \big] \ge 1 - \displaystyle\frac{3 \varsigma}{4} - c_1^{-1} \Theta \cdot e^{-c_1 \delta^2 M} \ge 1 - \varsigma,
		\end{flalign*} 
		
		\noindent for sufficiently large $L$. The proposition follows.
	\end{proof}

	Now let us establish \Cref{estimatekf} and \Cref{estimatenlf}.
	
	\begin{proof}[Proof of \Cref{estimatekf}]

		\noindent We induct on $k$; let us first verify the case $k = 0$. By \eqref{x0nu}, \eqref{c1e}, the independence of $\big\{ \mathscr{G} (v) \big\}$ over $v \in \mathbb{V} (M)$, and a Chernoff bound, we deduce the existence of a constant $c_2 \in (0, \kappa)$ with 
		\begin{flalign*}
			\mathbb{P} [\mathscr{F}_0] \ge 1 - \displaystyle\frac{\varsigma}{2} - e^{-c_2 M} \ge 1 - \displaystyle\frac{3 \varsigma}{4}.
		\end{flalign*}
		
		\noindent This establishes the lemma when $k = 0$. To establish it when $k \ge 1$, first observe that since $t_j \geq 0$ and $\varphi(s;z) \geq \delta \geq 2 \delta^2$, we have if $\mathcal F_{k-1}$ holds that $|\mathcal Z_{k-1} |\geq\exp( M (2 C_0 + 1))$. Hence, from \eqref{c1e} and a Chernoff bound, there exists a new constant $c_2 > 0$ such that
		\begin{flalign*}
			\mathbb{P} \bigg[  \Big| \big\{ v \in \mathcal{Z}_{k-1} : \text{$\mathscr{G} (v)$ holds} \big\} \Big| \ge \displaystyle\frac{c_1^3}{2} \cdot |\mathcal{Z}_{k-1}| \cdot \one_{\mathscr{F}_{k-1}}  \bigg] \ge 1 - e^{- c_2 M}.
		\end{flalign*}
		
		\noindent  Additionally, observe from \Cref{mudeltak} that for any $u \in \mathbb{V} (M_k)$ that 
		\begin{flalign}
			\label{upsi} 
			\mathbb{P} \big[ \mathscr{H} (u) \big] \ge 1 - e^{-(\delta^2 + \mu_k) M}, \quad \text{where} \quad \mathscr{H} (u) = \bigg\{ \big| \mathcal{X}_1 (u) \big| \ge \exp \Big( \big( \varphi (s; z) - st_k + \mu_k \big) \cdot M \Big) \bigg\}.
		\end{flalign}
		
		\noindent In particular, applying \eqref{upsi} for any $u = \mathfrak{c} \big( \mathfrak{c} (v) \big)$, where $v \in \mathcal{Z}_{k - 1}$ is such that $\mathscr{G} (v)$ holds, we deduce from a Chernoff bound that there exists a constant $c_3 > 0$ such that
		\begin{flalign}
			\label{zk1u} 
			\mathbb{P} \Bigg[ \bigg| \Big\{ u \in \mathcal{Z}_{k-1} : \text{$\mathscr{G}(u) \cap \mathscr{H} \big( \mathfrak{c} ( \mathfrak{c} (u)) \big)$ holds} \Big\} \bigg| \ge e^{-(2 \delta^2 + \mu_k) M} \cdot |\mathcal{Z}_{k-1}| \cdot \one_{\mathscr{F}_{k-1}} \Bigg] \ge 1 - e^{-c_3 \delta^2 M}.
		\end{flalign} 
		
		\noindent By \Cref{zidefinition}, any $v \in \mathcal{X}_1 \big( \mathfrak{c} (\mathfrak{c} (u)) \big)$ for which the event $\{ u \in \mathcal{Z}_{k - 1} \} \cap \mathscr{G} (u)$ holds satisfies $v \in \mathcal{Z}_k$. Thus, \eqref{zk1u} together with the bound 
		\begin{flalign*} 
			e^{-(2 \delta^2 + \mu_k) M} \cdot |\mathcal{Z}_{k-1}| \cdot \one_{\mathscr{F}_{k-1}} & \cdot \exp \Big( \big( \varphi (s; z) - st_k + \mu_k \big) \cdot M \Big) \\
			& \qquad \quad \ge \exp \Bigg( kM \cdot \bigg( \varphi (s; z) - \displaystyle\frac{s}{k} \displaystyle\sum_{j=1}^k t_j \bigg) - M (2C_0 +1 - 2 k \delta^2) \Bigg) \cdot \one_{\mathscr{F}_{k-1}},
		\end{flalign*} 
		
		\noindent yields the lemma.
	\end{proof}

	\begin{proof}[Proof of \Cref{estimatenlf}]
		
		We induct on $k$; the case $k = 0$ follows from \Cref{x0}, as it indicates that $\big| R_{0v}^{(v_+)} \big| \ge \kappa^M$, for any $v \in \mathcal{Z}_0$. To establish it for $k \ge 1$, fix $v \in \mathcal{Z}_k$, and let $w = v^{(M+1)}$ and $u = v^{(M)} = \mathfrak{c} (w)$ (where we recall the notation above \Cref{zidefinition}). We  have $u_- = w$. Then, \Cref{rproduct} implies 
		\begin{flalign}
			\label{rz1} 
			\big| R_{0v}^{(v_+)} \big| = \big| R_{0w_-}^{(w)} \big| \cdot |T_{w_- w}| \cdot \big| R_{ww}^{(v_+)} \big| \cdot |T_{wu}| \cdot \big| R_{uv}^{(u_-, v_+)} \big|.
		\end{flalign} 
		
		\noindent Since $v \in \mathcal{Z}_k$ and $w_- = v^{(M+2)}$, we have $w_- \in \mathcal{Z}_{k-1}$ and that $\mathscr{G} (w_-)$ holds. The former implies 
		\begin{flalign}
			\label{rz2} 
			\big| R_{0w_-}^{(w)} \big| \ge \exp \Bigg( \bigg( \displaystyle\sum_{j=1}^{k-1} t_j + C ( \log \kappa - (k-1) \delta^4 - C_0 - 1) \bigg) \cdot M \Bigg),
		\end{flalign} 
		
		\noindent and the latter (as $\mathscr{G}_0 (w_-) \cap \mathscr{G}_0 (w) \subset \mathscr{G} (w_-)$, since $w = \mathfrak{c} (w_-)$) gives 
		\begin{flalign}
			\label{rz3} 
			|T_{w_- w}| \ge 1; \qquad |T_{wu}| \ge 1.
		\end{flalign}
		
		\noindent Additionally, from the fact that $v \in \mathcal{X}_1 (u)$ (as $v \in \mathcal{Z}_k$), we obtain
		\begin{flalign}
			\label{rz4}
			\big| R_{uv}^{(u_-, v_+)} \big| \cdot \one_{\mathscr{B} (u, v; \varsigma)} \ge e^{t_0 M}. 
		\end{flalign}  
		
		Next, we have 
		\begin{flalign}
			\label{rz5}
			\big| R_{ww}^{(v_+)} \big| = \Bigg| z + T_{w_- w}^2 R_{w_- w_-}^{(w)} + T_{wu}^2 R_{uu}^{(u_-, v_+)} + \displaystyle\sum_{\substack{u' \in \mathbb{D} (w) \\ u' \ne \mathfrak{c} (w)}} R_{u'u'}^{(w)} T_{u' w}^2 \Bigg|^{-1} \ge \big| B + 8 \Omega + 6 \big|^{-1},
		\end{flalign}
		
		\noindent where we have used the bounds  
		\begin{flalign*} 
			& |z| \le |E| + \eta \le B + 1; \qquad |T_{w_- w}| \le 2; \qquad |T_{wu}| \le 2, \\
			&  \big| R_{w_- w_-}^{(w)} \big| \le \Omega; \qquad \big| R_{uu}^{(u_-, v_+)} \big| \le \Omega; \qquad \Bigg| \displaystyle\sum_{\substack{u' \in \mathbb{D} (w) \\ u' \ne \mathfrak{c} (w)}} R_{u'u'}^{(w)} T_{u' w}^2 \Bigg| \le 5.
		\end{flalign*} 
		
		\noindent Here, the second and third bounds follow from the fact that $\mathscr{G}_0 (w_-) \cap \mathscr{G}_0 (w)$ holds; the fourth and fifth follow from the fact that $\mathscr{B} (0, w_-; \varsigma) \cap \mathscr{B} (u, v; \varsigma)$ holds; and the sixth follows from the fact that $\mathscr{G} (w_-)$ holds. Thus, the lemma follows from combining \eqref{rz1}, \eqref{rz2}, \eqref{rz3}, \eqref{rz4}, and \eqref{rz5}, and then taking $M$ (equivalently, $L$) sufficiently large.
	\end{proof}

	\subsection{Delocalization Through Large Resolvent Entries} 
	
	\label{RLarge}
	
	In this section we establish \Cref{rimaginary0}, which is a delocalization result for $\mathbb{T}$, stating that $\Imaginary R_{\star} (E + \mathrm{i} \eta)$ remains bounded below if $\varphi (1; E) > 0$. Given \Cref{mprobability}, its proof will follow the ideas of \cite{aizenman2013resonant}. In particular, we begin with the following definition from \cite{aizenman2013resonant}, which provides notation for the quantiles of $\Imaginary R_{\star} (z)$ (recall \Cref{r}).
	
	\begin{definition}[{\cite[Definition 4.2]{aizenman2013resonant}}]
		
		\label{axi} 
		
		For any complex number $z = E + \mathrm{i} \eta \in \mathbb{H}$ and positive real number $a > 0$, let $\xi (a; z)$ be the largest value of $\xi \ge 0$ so that
		\begin{flalign*}
			\mathbb{P} \big[ \Imaginary R_{\star} (z) \ge \xi \big] \ge a.
		\end{flalign*} 
		
	\end{definition} 
	
	\begin{rem} 
		
		\label{axi0}
		
		Observe that $\xi (a; z)$ is bounded for any fixed $z \in \mathbb{H}$ and $a \in \mathbb{R}_{> 0}$, due to the deterministic estimate $\Imaginary R_{00} (z) \le  \big| R_{00} (z) \big| \le \eta^{-1}$ that holds by the second part of \Cref{q12}.
		
	\end{rem}
	
	Next, we require the following events on which $|R_{0v}|$ is large and on which a vertex $v$ has some child $w$ on which $|T_{vw}|^2 \cdot \Imaginary R_{ww}^{(v)}$ is not too small.
	
	\begin{definition} 
		
		\label{eventa} 
		
		In what follows, for any real number $\delta > 0$, integer $L \ge 0$, and vertices $v \in \mathbb{V} (L)$ and $w \in \mathbb{V} (L + 1)$ with $v \prec w$, we define the events 
		\begin{flalign}
			\label{av} 
			\mathscr{A}_1 (\delta; v)= \mathscr{A}_1 (\delta; v; z) = \big\{ |R_{0v}(z)| \ge e^{\delta L} \big\}
			\end{flalign}
			and 
			\begin{flalign}
			  \mathscr{A}_2 (v; w)=  \mathscr{A}_2 (v; w ;z) = \Bigg\{ |T_{vw}|^2 \Imaginary R_{ww}^{(v)}(z) \ge \xi \bigg( \displaystyle\frac{1}{2}; z \bigg) \Bigg\}; \qquad \mathscr{A}_2 (v)= \mathscr{A}_2 (v;z) = \bigcup_{v \prec w} \mathscr{A}_2 (v; w).
		\end{flalign}
		
	\end{definition}

	The following proposition is similar to \cite[Theorem 4.6]{aizenman2013resonant}, though the probability bound in \eqref{r0jestimate} is $\frac{3}{5}$ (which could be replaced with any number less than and bounded away from $1$, without substantially affecting the proof), instead of an unspecified constant $c$ as in \cite{aizenman2013resonant}. Its proof, which makes use of \Cref{mprobability}, is given in \Cref{ProofEstimateR} below. 
	
	\begin{prop} 
		
		\label{estimater}
		
		Let $s \in (\alpha, 1)$, $\delta \in (0, 1/8)$, $\varepsilon> 0$, $B > 1$, $\nu \in (8 \delta, 1)$ be real numbers. Let $E \in \mathbb{R}$ be a real number such that $\varepsilon \le |E| \le B$, and let $\{ \eta_j\}_{j=1}^\infty$ be a decreasing sequence of positive reals such that $\lim_{j} \eta_j = 0$. Assume
		\begin{flalign}
			\label{estimatespectrum} 
			\displaystyle\lim_{s \rightarrow 1} \bigg( \displaystyle\liminf_{j \rightarrow \infty} \varphi_s (E + \mathrm{i} \eta_j) \bigg) > \nu \text{ and }  \displaystyle\liminf_{j \rightarrow \infty} \mathbb{P} \big[ \Imaginary R_{\star} (E + \mathrm{i} \eta_j) > u \big] = 0, \quad \text{for any $u > 0$}.
		\end{flalign}
		
		\noindent Then, for any sufficiently large positive integer $L \ge L_0 (\nu) > 1$, there exists $J(L)$ such that for all $j \ge J$, we have		\begin{flalign}
			\label{r0jestimate} 
			\mathbb{P} \Bigg[  \bigcup_{v \in \mathbb{V} (L)} \big( \mathscr{A}_1 (\delta; v; E + \iu \eta_j) \cap \mathscr{A}_2 (v;E + \iu \eta_j) \big) \Bigg] \ge \displaystyle\frac{3}{5}.
		\end{flalign}
		
	\end{prop} 
	
	Given \Cref{estimater}, we can prove  \Cref{rimaginary0}.
	
	
		
%
%
%
	\begin{proof} [Proof of \Cref{rimaginary0}]

	We may assume that $E \ne 0$, for otherwise the result follows from \cite[Lemma 4.3(b)]{bordenave2011spectrum}. We first prove that there exists $c >0$ such that
	\begin{flalign}
			\label{dedeXX} 
	\liminf_{j \rightarrow \infty} \mathbb{P} \big[ \Imaginary R_{\star} (E + \mathrm{i} \eta_j) > c \big]  > c.
	\end{flalign}
	Assume to the contrary that $\liminf_{j \rightarrow \infty} \mathbb{P} \big[ \Imaginary R_{\star} (E + \mathrm{i} \eta_j) > c \big]  = 0$ for every $c > 0$.  By assumption, there $\nu >0$ such that 
\be
		\displaystyle\lim_{s \rightarrow 1} \bigg( \displaystyle\liminf_{j \rightarrow \infty} \varphi_s (E + \mathrm{i} \eta_j) \bigg) > \nu.
\ee 
\noindent Also, by \eqref{r00sum}, we have 
		\begin{flalign}
			\label{r00vimaginary} 
			\Imaginary R_{00} \ge \displaystyle\sum_{v \in \mathbb{V} (k)} |R_{0v}|^2 \displaystyle\sum_{v \prec u} |T_{vu}|^2 \Imaginary R_{uu}^{(v)}. 
		\end{flalign}
		
		\noindent By \Cref{estimater}, there exist constants $\delta = \delta (\alpha, \nu, E) > 0$ and $J, L \in \Zplus$ such that
		\begin{flalign}
			\label{probabilitya} 
			\mathbb{P} \big[ \mathscr{A} (\delta) \big] \ge \displaystyle\frac{3}{5}, \qquad \text{where} \quad \mathscr{A} (\delta) = \bigcup_{v \in \mathbb{V} (L)} \mathscr{A}_1 (\delta; v; E + \iu \eta_j) \cap \mathscr{A}_2 (v; E + \iu \eta_j)
		\end{flalign}
		for $j \ge J$.  Inserting this into \eqref{r00vimaginary}, and using the definitions \eqref{av} of the events $\mathscr{A}_1 (\delta; v)$ and $\mathscr{A}_2 (v)$, we find
		\begin{flalign*}
			\one_{\mathscr{A}(\delta)} \cdot \Imaginary R_{00}(E + \iu \eta_j) \ge \one_{\mathscr{A}(\delta)} \cdot e^{\delta L} \cdot \xi \bigg( \displaystyle\frac{1}{2}; E + \iu \eta_j \bigg).
		\end{flalign*}
		
		\noindent This, together with \eqref{probabilitya} and \Cref{axi}, implies that for all $j \ge J$ we have 
		\begin{flalign*}
			e^{\delta L} \cdot \xi \bigg( \displaystyle\frac{1}{2} ; E + \mathrm{i} \eta_j \bigg) \le   \xi \bigg( \displaystyle\frac{3}{5}, E + \mathrm{i} \eta_j \bigg) \le  \xi \bigg( \displaystyle\frac{1}{2}; E + \iu \eta_j \bigg),
		\end{flalign*}
		
		\noindent where in the last inequality we used the fact that $\xi (a; z)$ is nonincreasing in $a$ (which follows from \Cref{axi}). For $L$ large enough, this is a contradiction and thus establishes \eqref{dedeXX}.

We may now check that $ 
		\liminf_{\eta  \rightarrow 0} \mathbb{P} \big[ \Imaginary R_{\star} (E + \mathrm{i} \eta) > c \big] > c$. Set $z_j  = E + \mathrm{i} \eta_j$. By \eqref{dedeXX}, there exists a constant $c > 0$ such that $\mathbb{E} [\big( \Imaginary R_{\star} (z_{j}) \big)^{\alpha/2} ] > c$ for sufficiently large $j$. Recalling the notation from \eqref{kappatheta1} and \Cref{rrealimaginary}, it follows that $\sigma \big(\vartheta_0 (z_{j}) \big) \ge c^{2/\alpha}$, and so for any $\theta \in (0, 1)$ there exists a real number $\delta = \delta (\theta) > 0$ such that $\mathbb{P} \big[ \vartheta (z_{j}) > \delta \big] > \theta$ for all sufficiently large $j$. In particular, any limit point as $j$ tends to $\infty$ of $\vartheta (z_{j})$ is almost surely positive, which by \eqref{kappatheta1} (and \eqref{q00delta1}) implies the same for $\Imaginary R_{\star} (z_{\eta})$. 
		\end{proof} 

		 	\begin{rem} 
 	
 	\label{rnu} 
 	
 	Although not explicitly stated in \Cref{rimaginary0}, it is quickly verified that its proof shows the following more uniform variant of it. For any real numbers $ \nu, \theta \in (0, 1)$, there exists a constant $\delta = \delta (\varepsilon, B, \nu, \theta) > 0$ such that, if\footnote{Observe in this statement that the uniformity of the lower bound on $\delta$ is lost as $E$ tends to $0$; this seems to be an artifact of our proof method. Indeed, at $E = 0$, \cite[Lemma 4.3(b)]{bordenave2011spectrum} explicitly identifies the law of $R_{\star} (\mathrm{i} \eta)$ for any $\eta > 0$ and observes it to be almost surely positive as $\eta$ tends to $0$.} $\varepsilon \le |E| \le B$ and 
	\be
	\displaystyle\lim_{s \rightarrow 1} \bigg( \displaystyle\liminf_{j \rightarrow \infty} \varphi_s (E + \mathrm{i} \eta_j) \bigg) > \nu,
	\ee 
	then $\mathbb{P} \big[ \Imaginary R_{\star} (E + \mathrm{i} \eta_j) > \delta \big] \ge \theta$ for large enough $j$. 
 	
 	\end{rem} 
	
	\subsection{Proof of \Cref{estimater}}
	
	\label{ProofEstimateR}
	
	In this section we establish \Cref{estimater}. Let $\delta \in (0, 1/8)$ be a parameter, as in the statement of that proposition. Perhaps the ``most unlikely aspect'' of the event described in \eqref{r0jestimate} is the exhibition of some vertex $v \in \mathbb{V} (L)$ for which $\mathscr{A}_1 (\delta; v)$ holds, that is, for which $|R_{0v}| \ge e^{\delta L}$. The heuristic underlying the existence of such a vertex is simlar to the one described in \cite[Section 4.3]{aizenman2013resonant} under the name of \emph{resonant delocalization}. 
	
	To explain it, first observe that \Cref{mprobability} with $s$ close to $1$ yields about $e^{ (\varphi(s; z) - t_0)L}$ vertices $w \in \mathbb{V} (L-1)$ for which 
	\begin{flalign} 
		\label{r0w} 
		\big|R_{0w}^{(w_+)} \big| \ge e^{t_0 L}.
	\end{flalign} 

	\noindent Since $t_0$ could be negative, this does not immediately give an off-diagonal resolvent entry of size $e^{\delta L}$. So, we use the identity (from \Cref{rproduct}) 
	\begin{flalign} 
		\label{r0vw} 
		|R_{0v}| = \big| R_{0w}^{(w_+)} \big| \cdot |T_{wv}| \cdot |R_{vv}|
	\end{flalign}  

	\noindent for such $w$, where $v = w_+$. We will take $v$ so that $|T_{vw}| \ge 1$ (such a child $v$ of $w$ exists with positive probability), and so $|R_{0v}| \ge e^{t_0 L} \cdot |R_{vv}|$. 
	
	Next, since $\lim_{j \rightarrow \infty} \Imaginary R_{\star} (E + \mathrm{i} \eta_j) = 0 $, \Cref{rrealimaginary} shows that $R_{\star}$ becomes a real random variable whose density is bounded below; in particular, $\mathbb{P} \big[ |R_{vv}| \ge e^{(\delta - t_0) L} \big] \sim e^{(t_0 - \delta) L}$.  Since there are about $e^{(\varphi (s; z) - t_0) L}$ vertices $w$ satisfying \eqref{r0w}, and since $e^{(t_0  - \delta)L} \cdot e^{(\varphi (s; z) - t_0) L} = e^{(\varphi (s; z) - \delta)L} \gg 1$ (as $\varphi (s; z) \ge \nu > \delta$), there is likely at least one such vertex $v$ satisfying $|R_{vv}| \ge e^{(\delta - t_0) L}$. Combined with \eqref{r0w} and \eqref{r0vw}, this yields a vertex $v \in \mathbb{V}(L)$ for which $|R_{0v}| \ge e^{\delta L}$. 
	
	Now let us implement this in detail. Let $C_0 > 1$ denote a sufficiently large constant (so that in particular $C_0 > 4C$, where $C$ is from \Cref{mprobability}). Let $L \in \Zplus$ be a parameter. By \eqref{estimatespectrum}, we may fix $s \in (\alpha, 1)$ sufficiently close to $1$ and $J(L)\in \Zplus$ so that
	\begin{flalign}
		\label{r0}  
		\varphi_s (E + \mathrm{i} \eta_j) \ge \nu; \qquad \mathbb{P} \big[ \Imaginary R_{\star} (E + \mathrm{i} \eta_j) > e^{-2C_0 L} \big] < e^{-C_0 L}; \qquad \eta_j < e^{-2C_0 L}
	\end{flalign} 
	for all $j\in \Zplus$ such that $j \ge J$. Set $z = E + \mathrm{i} \eta$, $z_j = E + \iu \eta_j$, and for each integer $j \ge 1$ let $R_k = R_k (z)$ denote mutually independent random variables with law $R_{\star} (z)$. Letting $\boldsymbol{\zeta} = (\zeta_1, \zeta_2, \ldots )$ denote a Poisson point process on $\mathbb{R}_{> 0}$ with intensity $\alpha  x^{-\alpha/2-1} dx$, it follows from \eqref{r0}, \Cref{sumaxi}, \Cref{rrealimaginary}, and \Cref{ar} that 
	\begin{flalign}
		\label{zetajrj} 
		\mathbb{P} \Bigg[ \eta_j + \displaystyle\sum_{k = 1}^{\infty} \zeta_k \Imaginary R_k(z_j) > e^{-C_0 L} \Bigg] < e^{-C_0 L}
	\end{flalign} 
	for sufficiently large $j$.

	Next, recall the real number $t_0 = t_0 (\alpha, s, z, \omega, \Omega, L-1) \in \big[ -\frac{C_0}{4}, 0 \big]$ from \Cref{mprobability}, and the vertex set $\mathcal{S}_L (t) = \mathcal{S}_L (t; 0; \omega; \Omega;z)$ and its size $\mathcal{N}_L (t) = \mathcal{N}_L (t; 0; \omega; \Omega;z)$ from \Cref{ns}. Here we take $L-1$ instead of $L$ in the statement of \Cref{mprobability}, and restrict to $L \ge L_0$, where $L_0$ is given by the same proposition.
	By \Cref{mprobability} and the first bound of \eqref{r0} (and by taking $s \ge 1 - C_0^{-1} \delta$, so that $\big| (s - 1) t_0 \big| \le \frac{\delta}{2}$), we may fix $\omega = \omega (\delta) \in (0, 1)$ and $\Omega = \Omega (\delta) > 1$ such that 
	\begin{flalign} 
		\label{9nl110} 
		\mathbb{P} \big[ L^{-1} \log \mathcal{N}_{L-1} (t_0 - \delta) \ge \nu - t_0 - \delta \big] \ge \frac{9}{10}.
	\end{flalign} 

	We next require the following two vertex sets. 
	
	\begin{definition} 
		
	\label{0r1} 
	
	Further define the vertex set 
	\begin{flalign*}
		 \mathcal{R}_1 = \mathcal{R}_1(z) = \Big\{ w \in \mathcal{S}_{L-1} (t_0 - \delta) : \text{$\mathscr{G}_0 \big( \mathfrak{c} (w) \big)$ holds} \Big\},
	\end{flalign*} 

	\noindent where we recall the event $\mathscr{G}_0$ from \Cref{eventg}; the fact that any $w \in \mathcal{S}_L (t_0 - \delta)$ satisfies $\mathscr{G}_0 (w)$ (by \Cref{br0vw} and \Cref{ns}); and the child $\mathfrak{c} = \mathfrak{c} (w)$ of $w$ from \Cref{eventg}. 
	
	\end{definition} 

	\begin{definition} 
	
	\label{r02} 
	Additionally define the vertex set $\mathcal{R}_2 = \mathcal{R}_2(z)$ consisting of those vertices $v \in \mathbb{V} (L)$ such that the following five statements hold. 
	
	\begin{enumerate} 
		\item There exists a vertex $w \in \mathcal{R}_1$ for which $v = \mathfrak{c} (w)$. 
		\item The event $\mathscr{G}_0 (v)$ holds. 
		\item Denoting $\mathfrak{c} = \mathfrak{c} (v)$, we have the bounds $\Imaginary R_{\mathfrak{c} \mathfrak{c}}^{(v)}  \ge \xi \big( \frac{1}{2}; z \big)$ and $\big| R_{\mathfrak{c} \mathfrak{c}}^{(v)} \big| \le C_0$.
		\item We have 
		\begin{flalign*} 
			\Imaginary K_v \le e^{-C_0 L}; \quad \Imaginary Q_v \le e^{-C_0 L}; \quad \Real K_v \in [- \Real Q_v - e^{(2\delta + t_0 - \nu)L}, e^{(2\delta + t_0 - \nu)L} - \Real Q_v \big], 
		\end{flalign*}
		
		\noindent where
		\begin{flalign}
			\label{kvqv} 
			 K_v =  \displaystyle\sum_{\substack{u \in \mathbb{D} (v) \\ u \ne \mathfrak{c}}} T_{vu}^2 R_{uu}^{(v)}; \qquad Q_v = z + T_{w v}^2 R_{ww}^{(v)} + T_{v \mathfrak{c}}^2 R_{\mathfrak{c} \mathfrak{c}}^{(v)}.
	\end{flalign} 

	\end{enumerate} 

	\noindent We refer to the second, third, and fourth events above as $\mathscr{E}_1 (w)$, $\mathscr{E}_2 (w)$, and $\mathscr{E}_3 (w)$, respectively (indexing these events by $w = v_-$ instead of $v$); we also set $\mathscr{E} (w) = \mathscr{E}_1 (w) \cap \mathscr{E}_2 (w) \cap \mathscr{E}_3 (w)$. These events all depend on the choice of $z \in \bbH$; however, we omit this dependence from the notation.

	\end{definition} 

	The following lemma indicates that the event required in \eqref{r0jestimate} holds if $\mathcal{R}_2$ is nonempty. 
	
	\begin{lem} 
		
		\label{va} 
		
		For any $v \in \mathcal{R}_2$, the event $\mathscr{A}_1 (\delta; v) \cap \mathscr{A}_2 \big( v; \mathfrak{c} (v) \big)$ holds.
		
	\end{lem} 

	\begin{proof} 
		
		Abbreviating $\mathfrak{c} = \mathfrak{c}(v)$, we have $|T_{v\mathfrak{c}}| \ge 1$ (since $\mathscr{G}_0 (v)$ holds) and $\Imaginary R_{\mathfrak{c} \mathfrak{c}}^{(v)} \ge \Xi \big( \frac{1}{2}; z \big)$. Together, these imply that $\mathscr{A}_2 \big( v; \mathfrak{c} (w) \big)$ holds, so it remains to verify that $\mathscr{A}_1 (\delta, v)$ does as well. 
		
		To that end, recalling that $v = \mathfrak{c} (w)$ for some $w \in \mathcal{R}_1$, we have 
		\begin{flalign*} 
			|R_{0v}| = \big| R_{ww}^{(v)} \big| \cdot |T_{wv}| \cdot |R_{vv}| \le e^{(t_0 - \delta) (L-1)} \cdot \Bigg| z + T_{wv}^2 R_{ww}^{(v)} + T_{v\mathfrak{c}}^2 R_{\mathfrak{c}\mathfrak{c}}^{(v)} + \displaystyle\sum_{\substack{u \in \mathbb{D}(v) \\ u \ne \mathfrak{c}}} T_{vu}^2 R_{uu}^{(v)} \Bigg|^{-1},
		\end{flalign*} 
	
		\noindent where the first equality follows from \Cref{rproduct}, and the second follows from the fact that $\big| R_{ww}^{(v)} \big| \ge e^{(t_0 - \delta) (L-1)}$ (since $w \in \mathcal{S}_{L-1} (t_0 - \delta)$); the fact that $|T_{vw}| \ge 1$ (since $\mathscr{G}_0 (w)$ holds); and the Schur complement identity \eqref{qvv} for $R_{vv}$. Combining this with the fourth statement of \Cref{r02} (and taking $C_0$ sufficiently large), it follows that $|R_{0v}| \ge \frac{1}{2} \cdot e^{-t_0} \cdot e^{(\nu - 3 \delta)L} \ge e^{\delta L}$, where the last statement holds for sufficiently large $L$, since $\nu \ge 8 \delta$ and $|t_0| \le C_0$. Thus, $\mathscr{A} (\delta; v)$ holds, confirming the lemma.  
	\end{proof} 
		
	Now we can establish \Cref{estimater}.
	
	\begin{proof}[Proof of \Cref{estimater}] 
		
		By \Cref{va}, it suffices to show that $\mathcal{R}_2$ is not empty. We begin by considering an arbitrary point $z = E + \iu \eta$ for $\eta \in (0,1)$.
		
		 Let us first lower bound $|\mathcal{R}_1|$. To that end, recall from \eqref{c1e} that there exists a constant $c_1 > 0$ such that $\mathbb{P} \big[ \mathscr{G}_0 (w) \big] \ge c_1$. Moreover, the events $\big\{ \mathscr{G}_0 (w) \big\}$ over $w \in \mathcal{S}_{L-1} (t_0)$ are mutually independent. Together with \eqref{9nl110}, a Chernoff estimate, and a union bound, this yields a constant $c_2 = c_2 (\nu) > 0$ such that for sufficiently large $L$ we have
		\begin{flalign}
			\label{r1estimate} 
			\mathbb{P} \bigg[ |\mathcal{R}_1| \ge \displaystyle\frac{c_1}{2} \cdot e^{(\nu - t_0 - \delta) L} \bigg] \ge \displaystyle\frac{9}{10} - e^{-c_2 L} \ge \displaystyle\frac{4}{5}.	
		\end{flalign}

		Next, let us use \eqref{r1estimate} to lower bound the proabability that $\mathcal{R}_2$ is nonempty, to which end we lower bound $\mathbb{P} \big[ \mathscr{E} (w) \big]$ from \Cref{r02}. To do this, we condition on $\mathbb{T}_- (w)$ and again use the fact from \eqref{c1e} that $\mathbb{P} \big[ \mathscr{G}_0 (v) \big] \ge c_1$. Next, we condition on those $\{ T_{vu} \}$ for which $|T_{vu}| \in [1, 2]$, and on the event that $\mathscr{G}_0 (v)$ holds. Under this conditioning, we have for sufficiently large $C_0$ that 
		\begin{flalign}
			\label{e1e2} 
			\mathbb{P} \Bigg[ \Imaginary R_{\mathfrak{c} \mathfrak{c}}^{(v)} \ge \xi \bigg( \displaystyle\frac{1}{2}; z \bigg) \Bigg] \ge \displaystyle\frac{1}{2}, \quad \text{and} \quad \mathbb{P} \Big[ \big| R_{\mathfrak{c} \mathfrak{c}}^{(v)} \big| \le C_0 \Big] \ge \displaystyle\frac{3}{4}, \quad \text{so} \quad \mathbb{P} \big[ \mathscr{E}_2 (w) \big| \mathscr{E}_1 (w) \big] \ge \displaystyle\frac{1}{4}.
		\end{flalign}
	
		\noindent Indeed, since $R_{\mathfrak{c} \mathfrak{c}}^{(v)}$ from the $T_{vu}$, the first bound follows from \Cref{axi} for $\xi$ and the second follows from \eqref{q00delta2}; then the third follows from a union bound. 
		
		Next, we us additionally condition on $R_{\mathfrak{c} \mathfrak{c}}^{(v)}$, and on the event that $\mathscr{E}_1 (w) \cap \mathscr{E}_2 (w)$ holds. Then the quantity $Q_v$ from \eqref{kvqv} becomes deterministic and satisfies $|Q_v| \le |B| + 4 \Omega + 4 C_0 + 1$, since $|z| \le |E| + |\eta| \le B + 1$; since $|T_{vw}|, |T_{wu}| \le 2$ (since $\mathscr{G}_0 (w) \cap \mathscr{G}_0 (v)$ holds); since $\big| R_{vv}^{(w)} \big| \le \Omega$ (since $\mathscr{B} (0, w)$ holds, as $w \in \mathcal{S}_{L-1} (t_0 - \delta)$); and since $\big| R_{\mathfrak{c} \mathfrak{c}}^{(v)} \big| \le C_0$ (since $\mathscr{E}_2 (w)$ holds). 
		
		We now specialize to $\eta = \eta_j$ for some $j\in \Zplus$. By \eqref{zetajrj}, \eqref{zuv} (with \Cref{tuvalpha}, \Cref{q0estimate1}, and \Cref{sigmabetaestimate}), and a union bound it follows that there exists a constant $c_3 > 0$ such that 
		\begin{flalign} \label{b27}
			\mathbb{P} \big[ \mathscr{E}_3 (w) \big| \mathscr{E}_1 (w) \cap \mathscr{E}_2 (w) \big] \ge c_3 \cdot e^{(2\delta + t_0 - \nu) L} - 2 e^{-C_0 L} \ge \displaystyle\frac{c_3}{2} \cdot e^{(2 \delta + t_0 - \nu) L}
		\end{flalign}
	if $j$ is sufficiently large. Together with \eqref{e1e2} and the previously mentioned bound $\mathbb{P} \big[\mathscr{E}_1 (w) \big] = \mathbb{P} \big[ \mathscr{G}_0 (w) \big] \ge c_1$, this gives 
		\begin{flalign*} 
			\mathbb{P} \big[ \mathscr{E} (w) \big] \ge \displaystyle\frac{c_1 c_3}{8} \cdot e^{(2 \delta + t_0 - \nu)L} 
		\end{flalign*}
		for $z = E + \iu \eta_j$ when $j$ is sufficiently large.
	
		Now observe that, conditional on the subtree $\mathbb{T}_- (L)$ above $\mathbb{V} (L)$, the events $\big\{ \mathscr{E} (w) \big\}$ are mutually independent over $w \in \mathcal{R}_1$. Hence, since by \Cref{r02} the event $\mathcal{R}_2$ consists of those $\mathfrak{c} (w)$ for which $\mathscr{E} (w)$ holds, it follows from \eqref{r1estimate}, a Chernoff estimate, and a union bound that there exists a constant $c_4 = c_4 (\nu) > 0$ such that
		\begin{flalign*}
			\mathbb{P} \big[ |\mathcal{R}_2| \ge 1] \ge \mathbb{P} \bigg[ |\mathcal{R}_2| \ge \displaystyle\frac{c_1 c_3}{8} \cdot e^{(2 \delta + t_0 - \nu) L} \cdot e^{(\nu - t_0 - \delta) L} \bigg] \ge \displaystyle\frac{4}{5} - e^{-c_4 L} \ge \displaystyle\frac{3}{5},
		\end{flalign*} 
		
		\noindent for sufficiently large $L$. We choose $L_0$ so that \eqref{r1estimate}, \eqref{b27}, and the previous inequality all hold for $L \ge L_0$. Together with \Cref{va}, this implies the proposition.
	\end{proof}

	\newpage
	
	\chapter{Explicit Formula for the Mobility Edge} 

\label{EdgeExplicit} 

\section{Transfer Operator}\label{s:transfer}

This section is devoted to the proof of \Cref{l:bootstrap}. We begin in \Cref{s:productexpansion} by deriving a useful product expansion for $\Phi_L (s; z)$ (recall \Cref{moment1}). We take its limit as $\Imaginary z$ tends to $0$ in \Cref{EtaSmall}, which will imply that $\Phi_L$ may be calculated by iterating a certain integral operator $T$. In \Cref{s:transferoperator}, we estimate the operator norm of powers of $T$ in terms of the largest positive eigenvalue of $T$ (which will be evaluated in \Cref{s:pfeigenvector} explicitly). Finally, in \Cref{s:provebootstrap}, we prove \Cref{l:bootstrap}.

\subsection{Product Expansion}\label{s:productexpansion}

Throughout this section, we fix a complex number $z = E + \mathrm{i} \eta \in \mathbb{H}$ and an integer $L \ge 1$; we frequently abbreviate $R_{vw} = R_{vw} (z)$ for $v, w \in \mathbb{V}$ and $R_{\star} = R_{\star} (z)$ (from \Cref{gr}). Recalling the random variables $\varkappa_0 = \varkappa_0 (z)$ and $\vartheta_0 = \vartheta_0 (z)$ from \eqref{kappatheta1}, we let $V = V(z)$ denote a complex random variable with law 
\begin{flalign*}
	V \sim \varkappa_0 + \mathrm{i} \vartheta_0 \sim -z - R_{\star}^{-1}.
\end{flalign*}

\noindent For each integer $j \in \unn{1}{L}$, let $V_j = V_j (z)$ denote a complex random variable with law $V$, with the $\{ V_j \}$ mutually independent. 

Next, we fix some infinite path $\mathfrak{p} = (v_0, v_1, \ldots ) \in \mathbb{V}$ with $v_0 = 0$, as well as a sequence of real numbers $\bm{t} = (t_1, t_2, \ldots , t_L)$. The following definition recursively introduces two sets of random variables $\{ R_i(z ; L; \bm{t}; \mathfrak{p} )\}_{i=0}^L$ and $\{ S_i(z ; L; \bm{t}; \mathfrak{p}) \}_{i=0}^L$. We often abbreviate $R_i(z ; L) =  R_i(z ; L; \bm{t}) = R_i (z; L; \bm{t}; \mathfrak{p})$, and $S_i(z;L) = S_i(z ; L; \bm{t}) = S_i (z; L; \bm{t}; \mathfrak{p})$.

\begin{definition} 
	
\label{srz} 

For $i \in \unn{0}{L}$, we define the random variables $S_i (z; L)$ recursively, by setting
\begin{flalign} 
	\label{sirecursion}
S_L(z;L) = -  \big( R_{v_L v_L}^{(v_{L-1})} \big)^{-1}  -z, \quad \text{and} \quad S_i(z;L) = V_{i+1}(z) - \frac{t_{i+1}^2}{z + S_{i+1}(z;L)}, \quad \text{for $i \in \unn{0}{L-1}$}.
\end{flalign}

\noindent Moreover, for $i \in\unn{0}{L}$, we define 
\be\label{ridef}
R_i(z;L) = - \big(z + S_i(z;L) \big)^{-1}.
\ee

\end{definition} 

 Observe in particular that
\begin{flalign*}
R_L (z; L) = R_{v_L v_L}^{(v_{L-1})}, \quad \text{and} \quad R_i(z;L) = - \big(z + V_{i+1}(z) + t^2_{i+1} R_{i+1}(z ; L) \big)^{-1}, \quad \text{for $i \in \unn{0}{L-1}$},
\end{flalign*}

\noindent and that the latter equality is analogous in form to the Schur complement identity \eqref{qvv}, with one term (being $t_{i+1}^2 R_{i+1} (z; L)$ here) separated from the sum on the right side (with the remaining part of the sum being $V_{i+1} (z)$ here).

We next require a truncation of the fractional moment sum $\Phi_L$ (from \Cref{moment1}). For any real numbers $s \in (\alpha, 1)$ and $\omega > 0$, define 
\be
\Phi_L^\circ(\omega) 
= \Phi_L^\circ(s;z;\omega)=
\mathbb{E} \Bigg[ \displaystyle\sum_{v \in \mathbb{V}_{L}} \big| R_{0v} \big|^s \cdot \one_{\mathscr{D} (0, v; \omega)} \Bigg],
\ee

\noindent recalling the event $\mathscr{D} (0, v; \omega)$ from \Cref{br0vw}. Observe in particular that $\Phi^\circ_L(s;z;0) = \Phi_L(s;z)$.

The following lemma provides a product expansion for $\Phi^\circ_L (\omega)$.

\begin{lem}\label{l:expansionz}

For any real numbers $s \in (\alpha,1)$ and $\omega > 0$, we have  
\be\label{productN}
\Phi^\circ_{L}(s;z;\omega)
=
\int_{\mathbb{R}^L}
\E \Bigg[
\big| R_L (z; L; \bm{t}) \big|^s \cdot 
\prod_{j=1}^{L}
|t_j|^s \big| R_{j-1}(z; L; \bm{t}) \big|^s \Bigg] \cdot
\prod_{j=1}^{L}
\alpha t_j^{-\alpha -1 } \one_{t_j \ge \omega} \, dt_j
.
\ee
\end{lem}

\begin{proof}

For each $i \in \unn{1}{L}$, set $\bm{t}^{(i)} = (t_1, t_2, \ldots , t_i)$. Set $\Phi_{L, 0}^{\circ} = \Phi_L^{\circ}$ and, for any $i \in \unn{1}{L}$, set 
\begin{flalign}
	\label{lsumi}
	\begin{aligned}
\Phi^\circ_{L,i}(s;z)
=
\int_{\mathbb{R}^i} &
\E \Bigg[
\prod_{j=1}^{i}
|t_j|^s \Big| R_{j-1} \big(z; i; \bm{t}^{(i)} \big) \Big|^s 
 \cdot
\sum_{ v\in \mathbb{D}_{L-i} (v_i)} 
\big| R_{v_i v}^{(v_{i-1})} \big|^s \cdot \one_{\mathscr{D}(v_i,v;\omega)}
\Bigg] \\
& \qquad \times 
\prod_{j=1}^{i}
\alpha t_j^{-\alpha -1 } \one_{t_j \ge \omega} \, dt_j.
\end{aligned}
\end{flalign}

\noindent We will show that for every $i \in \unn{0}{ L-1}$, if $\Phi^\circ_{L, i }= \Phi^\circ_{L} $, then $\Phi^\circ_{L, i +1} = \Phi^\circ_{L} $.
Since the desired conclusion \eqref{productN} is $\Phi^\circ_{L, L} = \Phi^\circ_{L}$ (as $R_L (z; L) = R_{v_L v_L}^{(v_{L-1})}$), this will prove the lemma by induction.

By the product expansion \Cref{rproduct}, for any $w \in \mathbb{D} (v_i)$ and $v \in \mathbb{D}_{L-i-1} (w) \subseteq \mathbb{V}(L)$, we have 
\be\label{tildeproduct}
\big| R^{(v_{i-1})}_{v_i v } \big| = \big| R^{(v_{i-1})}_{v_i v_i} \big| \cdot |T_{v_i w }| \cdot \big| R^{(v_i)}_{wv} \big|.
\ee

\noindent Further, we have by the Schur complement identity \eqref{qvv} that 
\be\label{schur}
R^{(v_{i-1})}_{v_i v_i} = -\big( z + T^2_{v_i w }R^{(v_i)}_{ww} + K_{w} \big)^{-1},\quad \text{where}\quad
K_{w} = \sum_{\substack{u \in \mathbb{D} (v_i) \\ u \neq w}} T^2_{v_i u} R^{(v_i)}_{uu}.
\ee

\noindent Observe that $K_w$ has the same law as $V$, as (by \Cref{sumaxi} and \eqref{kappa0theta0sum}) both are complex $\frac{\alpha}{2}$-stable laws with the same parameters. Inserting these into \eqref{lsumi} gives
\begin{flalign}
	\label{momenti1}
	\begin{aligned} 
\Phi^\circ_{L,i}(s;z) & = \int_{\mathbb{R}^i}
\E \Bigg[ \bigg( \prod_{j=1}^{i}
|t_j|^s \big|R_{j-1}(z; i) \big|^s \bigg)
\cdot 
\displaystyle\sum_{w \in \mathbb{D} (v_i)} \bigg|
\frac{T_{v_i w}}{z + T^2_{v_i w }R^{(v_i)}_{ww} + K_{w}}
\bigg|^s \\
& \qquad \qquad \times 
\sum_{\substack{ v\in \bbV(L)\\ w \preceq v}} 
\big| R_{w v}^{(v_i)} \big|^s \cdot \one_{\mathscr{D} (w, v; \omega)}
\Bigg]  \cdot 
\prod_{j=1}^{i}
\alpha t_j^{-\alpha -1 }\one_{t_j \ge \omega}\, dt_j.
\end{aligned} 
\end{flalign}

\noindent Recall from \Cref{srz} and \eqref{schur} that $R_{i-1}(z;i) = - \big(z + S_{i-1}(z;i) \big)^{-1}$ and
\bex S_{i-1}(z;i) = V_i - t_i^2 \cdot R_i (z; i) = V_{i} + t_i^2 \cdot R_{v_i v_i}^{(v_{i-1})}= V_i + t_i^2 \cdot \big(z + T^2_{v_i w }R^{(v_i)}_{ww} + K_{w} \big)^{-1}.
\eex

\noindent Further set
\begin{flalign*}
	& \tilde S_{i-1}(z;i; t_{i+1}) = 
	\tilde S_{i-1}(z; i ; \bm{t})
	= V_i + t_i^2 \cdot \big( z + t_{i+1}^2 R^{(v_i)}_{v_{i+1},v_{i+1}} + V_{i+1} \big)^{-1}; \\
& \tilde R_{i-1}(z;i; t_{i+1}) = 
\tilde R_{i-1}(z; i; \bm{t})  = -  \big( z + \tilde S_{i-1}(z;i; t_{i+1}) \big)^{-1}.
\end{flalign*}
and define $\tilde R_j(z;i; t_{i+1}) = \tilde R_j(z;i;  \bm{t})$ and $\tilde S_j(z;i;t_{i+1}) = \tilde S_j(z;i; \bm{t}) $ for $j \le i-2$ recursively using the analogues of \eqref{sirecursion} and \eqref{ridef} (replacing $S$ and $R$ there with $\widetilde{S}$ and $\widetilde{R}$ here, respectively). Observe, since $\big(K_w, R_{ww}^{(v_i)} \big)$ has the same law as $\big( V_{i+1}, R_{v_{i+1}, v_{i+1}}^{(v_i)} \big)$, we have that $\widetilde{S}_{i-1} (z; i; t_{i+1})$ has the same law as $S_{i-1} (z; i)$ if $t_{i+1} = T_{v_i w}$. Thus, for $t_{i+1} = T_{v_i w}$, we have that $\big( \widetilde{R}_{i-1} (z; i; t_{i+1}), \widetilde{S}_{i-1} (z; i; t_{i+1}) \big)$ has the same law as $\big( R_{i-1} (z; i), S_{i-1} (z; i) \big)$. It follows for $t_{i+1} = T_{v_i w}$ that the random variables $\big( \widetilde{R}_j (z; i; t_{i+1}), \widetilde{S}_j (z; i; t_{i+1}) \big)_j$ have the same law as $\big( R_j (z; i), S_j (z; i) \big)_j$, as they satisfy the same recursion. Hence, \eqref{momenti1} implies
\begin{flalign}
	\label{momenti2}
	\begin{aligned} 
		\Phi^\circ_{L,i}(s;z) & = \int_{\mathbb{R}^i}
		\E \Bigg[ \bigg( \prod_{j=1}^{i}
		|t_j|^s \big| \widetilde{R}_{j-1}(z; i) \big|^s \bigg)
		\cdot 
		\displaystyle\sum_{w \in \mathbb{D} (v_i)} \bigg|
		\frac{T_{v_i w}}{z + T^2_{v_i w }R^{(v_i)}_{ww} + K_{w}}
		\bigg|^s \\
		& \qquad \qquad \times 
		\sum_{\substack{ v\in \bbV(L)\\ w \preceq v}} 
		\big| R_{w v}^{(v_i)} \big|^s \cdot \one_{\mathscr{D} (w, v; \omega)}
		\Bigg]  \cdot 
		\prod_{j=1}^{i}
		\alpha t_j^{-\alpha -1 }\one_{t_j \ge \omega}\, dt_j.
	\end{aligned} 
\end{flalign}

We evaluate the right side of \eqref{momenti2} through \eqref{fxi}, to which end we must introduce a suitable a function $f \colon \mathbb{R}_{> 0} \times \mathcal N_{\mathbb{R}_{> 0}}^{\#*} \rightarrow \mathbb{R}_{\ge 0}$ to which \Cref{fidentityxi} would apply. Suppose that $\nu \in  \mathcal N_{\mathbb{R}_{> 0}}^{\#*}$ is supported on a finite interval $(0,B)$ for some $B > 0$. Then we can write $\nu = \sum_{k=1}^\infty \delta_{\nu_k}$, where the sequence $\nu_1, \nu_2, \dots$ is decreasing. Suppose also that the sum 
\be\label{Knu}
K(\nu) = \sum_{k=1}^\infty \nu^2_k \cdot R_{k}(z),
\ee
converges, where  $\{R_{k}(z)\}_{k=1}^\infty$ is a collection of independent, identically distributed random variables with law $R_\star(z)$.
We set
\begin{flalign*}
& \tilde S_{i-1}(z;i; t_{i+1};\nu) = \tilde S_{i-1}(z; i ; \bm{t};\nu)= V_i + t_i^2 \cdot \big(z + t_{i+1}^2 R^{(v_i)}_{v_{i+1},v_{i+1}} + K(\nu) \big)^{-1}; \\
& \tilde R_{i-1}(z;i; t_{i+1}, \nu) = \tilde R_{i-1}(z; i; \bm{t};\nu)  = - \big( z + \tilde S_{i-1}(z;i; t_{i+1};\nu)\big)^{-1},
\end{flalign*}

\noindent and define $\tilde R_{j}(z;i; t_{i+1};\nu) = \tilde R_{j}(z;i; \bm{t};\nu)$ and $\tilde S_{j}(z;i;t_{i+1};\nu) = \tilde S_{j}(z;i; \bm{t};\nu) $ for $j \le i-2$ recursively using the analogues of \eqref{sirecursion} and \eqref{ridef} (replacing the $S$ and $R$ there with $\widetilde{S}$ and $\widetilde{R}$ here, respectively). Observe that, if $\nu$ is sampled according to a Poisson point process with intensity measure $\alpha x^{-\alpha-1} dx$, then \Cref{tuvalpha} implies that $\big( \widetilde{S}_j (z; i; t_{i+1}; nu), \widetilde{R}_j (z; i; t_{i+1}; \nu) \big)_j$ has the same law as $\big( \widetilde{S}_j (z; i; t_{i+1}), \widetilde{R}_j (z; i; t_{i+1}) \big)_j$. Now, if \eqref{Knu} diverges set $f = 0$, and when it converges define 
$f \colon \mathbb{R}_{> 0} \times \mathcal N_{\mathbb{R}_{> 0}}^{\#*} \rightarrow \mathbb{R}_{\ge 0}$  by
\begin{flalign*}
f(x, \nu) & =
\int_{\mathbb{R}^i}
\E \Bigg[
\bigg(
\prod_{j=1}^{i} |t_j|^s \big| \tilde R_{j-1}(z; i; x; \nu) \big|^s 
\bigg) \cdot
\bigg|
\frac{x}{z + x^2 R^{(v_i)}_{v_{i+1},v_{i+1}} + K(\nu)}
\bigg|^s \\
& \qquad\qquad \times \sum_{ v\in \mathbb{D}_{L-i-1} (v_{i+1})}
\big| R_{v_{i+1}, v}^{(v_i)} \big|^s \cdot \one_{\mathscr{D} (v_{i+1}, v; \omega)}
\Bigg] \cdot \prod_{j=1}^{i}
\alpha t_j^{-\alpha -1 }\one_{t_j \ge \omega} \, dt_j.
\end{flalign*}

We can then apply \eqref{fxi} to the process $\Xi = \big\{ |T_{v_i u}| \big\}_{u \in \mathbb{D}( v_i)}$ to obtain from \eqref{momenti2} that
\begin{flalign} 
	\label{cutsum12}
	\begin{aligned} 
\Phi^\circ_{L,i}(s;z) & = \mathbb{E} \Bigg[ \displaystyle\sum_{w \in \mathbb{D} (v_i)} f (T_{v_i w}, \Xi \setminus T_{v_i w}) \Bigg] \\
& = \int_{\mathbb{R}^{i+1}}
\E \Bigg[
\bigg(
\prod_{j=1}^{i}
|t_j|^s \big|\tilde R_{j-1}(z; i; t_{i+1}) \big|^s 
\bigg) \cdot 
\bigg|
\frac{t_{i+1}}{z + t_{i+1}^2 R^{(v_i)}_{v_{i+1},v_{i+1}} + V_{i+1}(z)}
\bigg|^s \\
& \qquad \qquad \qquad \times \sum_{ v\in \mathbb{D}_{L-i-1} (v_{i+1})} \left| R_{v_{i+1}, v}^{(v_i)}\right|^s \cdot \one_{\mathscr{D} (v_{i+1}, v; \omega)}
\Bigg] \cdot \alpha t_j^{-\alpha-1} \one_{t_{i+1} \ge \omega} dt_{i+1} \\
& \qquad \qquad \qquad \times \Bigg( \prod_{j=1}^{i} \alpha t_j^{-\alpha -1 } \one_{t_j \ge \omega} dt_j \Bigg),
\end{aligned}
\end{flalign}

\noindent where in the second term the expectation is with respect to a Poisson point process with intensity $\alpha t^{-\alpha-1} dt$ by \Cref{tuvalpha} (and we used the fact that we can exchange the sum and the integrals because all terms are positive). 

Recalling that $V_{i+1}$ has the same law as $K_w$ (by \Cref{tuvalpha}), we find $\big( \widetilde{S}_j (z; i; t_{i+1}), \widetilde{R} (z; i; t_{i+1}) \big)_j$ has the same law as $\big( S_j (z; i+1), R_j (z; i+1) \big)_j$. Thus, the right side of \eqref{cutsum12} is $\Phi^\circ_{L,i+1}(s ; z)$, by \eqref{momenti1}. Hence, \eqref{cutsum12} is the claim $\Phi^\circ_{L,i}(s ; z) = \Phi^\circ_{L,i+1}(s ; z)$. Using the induction hypothesis that $\Phi_{L,i}(s ; z)  = \Phi^\circ_L(s;z)$, this shows the desired equality $\Phi^\circ_{L,i+1}(s ; z) = \Phi^\circ_L(s;z)$.
\end{proof}

\subsection{The Small $\eta$ Limit of the Product Expansion} 

\label{EtaSmall} 

In this section we analyze the limit of $\Phi_L (s; z)$ (from \Cref{moment1}) as $\Imaginary z$ tends to $0$. To that end, we adopt the notation of \Cref{s:productexpansion}; we further fix a complex number $u_L \in \overline{\mathbb{H}}$ and a sequence of complex numbers $\bm{r} = (r_1, r_2, \ldots , r_{L-1}) \in \overline{\mathbb{H}}^{L-1}$. We define quantities $u_i \in \overline{\mathbb{H}}$ recursively for $i\in \unn{0}{L-1}$ by
\be\label{ui}
u_i = r_i - \frac{t_{i+1}^2}{z+u_{i+1}}.
\ee

\noindent Observe that $(u_i, r_i)$ serve as the analogs of $\big( S_i (z; L), V_{i+1} \big)$ in \Cref{srz}. Also observe that $u_i$ is a function of the parameters $\{t_{i+1},\dots, t_L \} \cup \{ r_i, \dots, r_{L-1} \} \cup \{ u_L \}$ only. We additionally set
\bex
F(z; \bm{t}, \bm{r}, u_L)= \frac{1}{z +  u_0} \cdot \prod_{i=1}^L
  \frac{t_i}{z + u_i} .
\eex

\noindent Also let $p_z\colon \bbC \rightarrow \bbR$ denote the density of the complex random variable $\varkappa_0(z) + \iu \vartheta_0(z)$, and let $\mu_z$ denote the measure on $\bbC$ given by $p_z(y) \, dy$.

Under this notation, \Cref{srz} implies that, if $u_L$ is sampled under the measure $\mu_z$ (which has law $-z-R_{\star}^{-1}$, or equivalently of $-z - \big( R_{v_L v_L}^{(v_{L-1})} \big)^{-1}$) then $u_L$ has the same law as $S_L (z; L)$. If the $r_i$ are further sampled independently under $\mu_z$, then $r_i$ has the same law as $V_{i+1}$, and so \eqref{ui} implies that $(\bm{r}, u_L)$ has the same law as $\big( S_1 (z; L), S_2 (z; L), \ldots, S_L (z; L) \big)$. Together with \eqref{productN} and the definition of $R_i (z; L)$ from \Cref{srz}, this implies that
\begin{multline}\label{productN3}
\Phi^\circ_{L}(\omega)
=
\int_{\mathbb{R}^L}
\E \Bigg[
\big| R_L (z; L; \bm{t}) \big|^s \cdot \prod_{j=1}^{L}
|t_j|^s \big|R_{j-1}(z; L) \big|^s \Bigg] \cdot
\prod_{j=1}^{L}
\alpha t_j^{-\alpha -1 }\one_{t_j \ge \omega} \, dt_j \\
=
\int_{\mathbb{C}^L \times \mathbb{R}^L}
\big|F(z; \bm{t}, \bm{r}, u_L)\big|^s
d\mu_z(u_L)  \prod_{i=0}^{L-1}  d\mu_z(r_i ) 
\prod_{j=1}^{L}
\alpha t_j^{-\alpha -1 }\one_{t_j \ge \omega} \, dt_j.
\end{multline}

We next introduce an analog of the right side of \eqref{productN3}, assuming as $\eta$ tends to $0$ the random variable $R_{\star} (E + \mathrm{i} \eta)$ becomes real. To that end, recall from \Cref{pkappae} that $p_E(x) \colon \bbR \rightarrow \bbR$ denote the density of the real random variable $\varkappa_\loc(E)$. 

\begin{definition} 
	
	\label{amoment2} 
	
	For any real numbers $s \in (\alpha, 1)$; $E \in \bbR$; and $\omega > 0$, define 
\begin{flalign}\label{productN2}
\Phi_L^{\loc} (s;E;\omega) =\int_{\mathbb{R}^{2L}}
\big|F(E; \bm{t}, \bm{r}, u_L)\big|^s 
\cdot p_E (u_L) d u_L  \cdot \prod_{i=0}^{L-1}  p_E (r_i) dr_i \cdot
\prod_{j=1}^{L}
\alpha t_j^{-\alpha -1 }\one_{t_j \ge \omega}\, dt_j,
\end{flalign}
and set $\Phi_L^{\loc} (s;E) = \Phi_L^{\loc} (s;E;0)$.

\end{definition}

The next lemma considers the behavior of $ \Phi_L(s;z;\omega)$ as $\eta$ tends to zero. 

\begin{lem}\label{l:PhiEA}
Fix an integer $L \ge 1$ and real numbers $s \in (\alpha,1)$; $E \in \bbR$; and $\omega>0$.
Suppose there exists a decreasing sequence of positive real numbers $(\eta_1, \eta_2, \ldots ) \subset (0, 1)$ tending to $0$ and an almost surely real random variable $R_\star (E)$, such that $\lim_{k \rightarrow \infty} R_\star(E + \iu \eta_k) = R_\star(E)$ weakly. Then,
\bex
\lim_{k \rightarrow \infty} \Phi^\circ_L(s;E + \iu \eta_k; \omega ) =\Phi_L^{\loc} (s ; E ;\omega ).
\eex
\end{lem}
\begin{proof}
	
	We may assume that $E \ne 0$, for otherwise such a sequence $(\eta_1, \eta_2, \ldots )$ does not exist by \cite[Lemma 4.3(b)]{bordenave2011spectrum}. Throughout this proof, we let $\mu_E$ denote the probability measure on $\mathbb{C}$, which is supported on $\bbR$ and whose restriction to $\bbR$ is the measure $p_E(x)\, dx$. 
	
	By \Cref{l:boundarybasics}, we have $R_\star(E) = R_\loc(E)$ in law. Setting $z_k = E + \mathrm{i} \eta_k$, this implies $\lim_{k \rightarrow \infty} \mu_{z_k} = \mu_E$, which in turn implies the weak convergence
\begin{align}\label{replace!}
\lim_{k \rightarrow \infty} &d\mu_z(u_L) \cdot \prod_{i=0}^{L-1}  d\mu_z(r_i ) \cdot
\prod_{j=1}^{L}
\alpha t_j^{-\alpha -1 }\one_{t_j \ge \omega} \, dt_j = d\mu_E(u_L) \cdot \prod_{i=0}^{L-1}  d\mu_E(r_i ) \cdot
\prod_{j=1}^{L}
\alpha t_j^{-\alpha -1 }\one_{t_j \ge \omega}\, dt_j.
\end{align}

We observe that, up to a normalization (dependent on $\omega$), the measure $\alpha t_j^{-\alpha -1 } \cdot \one_{t_j \ge \omega}\, dt_j$ prescribes the law of a real random variable. Then (up to this normalization), \eqref{productN3} may be viewed as the expectation of the random variable $|F(z)|^s$, where the law of $F(z)$ is induced by the $L$ mutually independent random variables $(\bm{r}, u_L)$, each with law $\mu_z$, and the $L$ mutually random variables $\bm{t}$, each with law proportional to $\alpha t_j^{-\alpha -1 } \cdot \one_{t_j \ge \omega}\, dt_j$. Further observe that the sequence $\{\Phi_{L}(s_0,z_k)\}_{k=1}^\infty$ is uniformly bounded for any fixed $s_0 \in (s, 1)$ by \Cref{limitr0j}. Hence, we have $ \E \big[ |F(E + \iu \eta)|^{s_0} \big] < C$ for some constant $C = C(L) > 1$ independent of $\eta \in (0,1)$. 

It follows that the sequence of random variables $\{ |F(z_k)|^s\}_{k=1}^\infty$ is uniformly integrable, which justifies the exchange of limits (using \eqref{replace!} for the first equality),
\begin{align*}
\lim_{k \rightarrow \infty} & \Phi^\circ_{L} (s;z_k;\omega) \\
&= \int_{\mathbb{C}^L \times \mathbb{R}^L} \lim_{k \rightarrow \infty}
\big|F(z_k; \bm{t}, r_1, \dots , r_{L-1}, u_L)\big|^s
d\mu_z(u_L)  \prod_{i=0}^{L-1}  d\mu_z(r_i )
\prod_{j=1}^{L}
\alpha t_j^{-\alpha -1 }\one_{t_j \ge \omega}\, dt_j
\\
&= \int_{\mathbb{C}^L \times \mathbb{R}^L}
\big|F(E; \bm{t}, r_1, \dots , r_{L-1}, u_L)\big|^s
d\mu_E(u_L) \prod_{i=0}^{L-1}  d\mu_E(r_i )
\prod_{j=1}^{L}
\alpha t_j^{-\alpha -1 }\one_{t_j \ge \omega}\, dt_j.
\end{align*}
The lemma follows after recalling the definition \eqref{productN2} and that of $\mu_E$ in terms of $p_E$.
\end{proof}

The next lemma removes the cutoff $\omega$ from the previous lemma, and computes the boundary value of $\Phi_L(s;z)$ under the same hypotheses as in \Cref{l:PhiEA}.
\begin{lem}\label{l:PhiE}
	
	Adopting the notation and assumptions of \Cref{l:PhiEA}, we have 
\bex
\lim_{k \rightarrow \infty} \Phi_L(s;E + \iu \eta_k ) =\Phi_L^{\loc} (s ; E  ).
\eex
\end{lem}
\begin{proof}

Set $z_k = E + \iu \eta_k$; as in the proof of \Cref{l:PhiEA}, we may assume that $E \ne 0$. For any $\omega > 0$, we have $\Phi_L(s ; z_k) \ge \Phi^\circ_L(s ; z_k ;\omega )$, so we conclude from \Cref{l:PhiEA} that
\bex
\liminf_{k \rightarrow \infty}
\Phi_L(s ; z_k) \ge \liminf_{k \rightarrow \infty} \Phi_L^{\circ} (s; z_k; \omega) =  \Phi_L^{\loc} (s ; E ;\omega ).
\eex

\noindent By the monotone convergence theorem, we make take $\omega$ to zero to conclude that
\be\label{liminf}
\liminf_{k \rightarrow \infty}
\Phi_L(s ; z_k) \ge  \Phi_L^{\loc} (s ; E  ).
\ee
Next, by \Cref{ry0estimate}, there exists for any $\delta >0$ a constant $\omega = \omega(\delta) > 0$ such that
\bex
\Phi^\circ (s; z_k ; \omega) \ge ( 1 - \delta) \cdot \Phi_L(s, z_k).
\eex
Holding $\delta$ fixed, taking $k$ to infinity, and using \Cref{l:PhiEA} gives 
\be\label{limsup}
 \Phi_L^{\loc} (s, E) \ge \Phi_L^{\loc} (s ; E ;\omega) = \displaystyle\lim_{k\rightarrow \infty} \Phi_L^{\circ} (s; z_k; \omega) \ge (1 - \delta) \cdot 
 \limsup_{k\rightarrow\infty} \Phi_L(s, z_k).
\ee
Taking $\delta$ to zero and combining \eqref{liminf} and \eqref{limsup} completes the proof.
\end{proof}

We next show that $\Phi_L^{\loc} (s; E)$ can be computed by iterating a certain integral operator.
We  define a sequence of functions $\{g_i(x;E)\}_{i=0}^L$ in the following way. We set $g_L(x ; E) = p_E(x)$, and for $i \in \unn{0}{L-1}$ we set
\be\label{gidef}
g_i(x;E) = \int_{\mathbb{R}^2} g_{i+1}(y,E)
\left| \frac{t}{E + y} \right|^s \cdot p_E\left( x + \frac{t^2}{E + y } \right) \, dy  \cdot \alpha t^{-\alpha - 1 } \one_{t > 0}\, dt .
\ee
\begin{lem}\label{l:Erecursion}
Fix an integer $L \ge 1$ and real numbers $s \in (\alpha,1)$ and $E \in \bbR$. We have 
\begin{flalign}\label{lsbig}
\Phi_L^{\loc} (s; E) = \int_{-\infty}^{\infty} \left| \frac{1 }{E +x } \right|^s  g_0(x ; E) \, dx.
\end{flalign}
\end{lem}
\begin{proof}
Iterating \eqref{gidef}, we obtain that the right side of \eqref{lsbig} equals
\be
\int_{\mathbb{R}^{2L}}
\left| \frac{1}{E + y_0} \prod_{i=1}^L
  \frac{t_i}{E + y_i}  \right|^s \cdot 
 p_E(y_L) \, dy_L \cdot 
  \prod_{i=0}^{L-1}  p_E\left( y_i + \frac{t_{i+1}^2}{E + y_{i+1}} \right)\, dy_i \cdot  \prod_{j=1}^L \alpha t_j^{-\alpha - 1} \one_{t_j > 0}\, dt_j.
\ee
We now make the change of variables 
$r_i  = y_i + \frac{t_{i+1}^2}{E + y_{i+1}}$ for $i\in\unn{0}{L-1}$ to obtain that the previous line equals
\be
\int_{\mathbb{R}^{2L}}
\left| \frac{1}{E + y_0} \prod_{i=1}^L
  \frac{t_i}{E + y_i} \right|^s
 \cdot p_E(y_L) \, dy_L
 \cdot \prod_{i=0}^{L-1}  p_E(r_i) \, dr_i \cdot \prod_{j=1}^L  \alpha t_j^{-\alpha - 1}\one_{t_j > 0}\, dt_j,
\ee
where we consider the $y_i$ for $i \in \unn{0}{ L -1}$ as functions of $(r_0, \dots , r_{L-1}, y_L)$. Further, these $y_i$ satisfy the relations \eqref{ui}, so (setting $y_L = u_L$) the previous integral is equivalent to
\be
\int_{\mathbb{R}^{2L}}
\big| F(z; \bm{t}, \bm{r}, u_L)\big|^s \cdot 
 p_E(u_L) \, du_L \cdot 
  \prod_{i=0}^{L-1}  p_E(r_i) \, dr_i \cdot \prod_{j=1}^L \alpha t_j^{-\alpha - 1} \one_{t_j > 0}\, dt_j,
\ee

\noindent which by the definition of $\Phi_L^{\loc} (s; E)$ (from \Cref{amoment2}) yields the lemma.
\end{proof}

\subsection{Integral Operator}\label{s:transferoperator}
Given $\alpha \in (0,1)$ and $s \in (\alpha, 1)$,
let $\| \cdot \|=\| \cdot \|_{\alpha,s}$ be the norm on functions $f\colon \bbR \rightarrow \bbR$ defined by
\begin{flalign} 
	\label{xf}
\| f \| = \Big\| f(x) \cdot \big(1 + |x|^{(\alpha-s)/2 + 1} \big) \Big\|_\infty,
\end{flalign}
where $\| \cdot \|_\infty$ is the $L^\infty$ norm. Let $\mathcal X$ be the Banach space of functions $f \colon \bbR \rightarrow \bbR$ such that $\|f \| < \infty$ (more precisely, we consider equivalence classes of functions that differ only on sets of measure zero). The following definition introduces a transfer operator that will be used to analyze $\Phi_L$; in the below, we recall the density $p_E$ from \Cref{pkappae}. 

\begin{definition} 
	\label{toperator} 
We define the operator $T= T_{s,\alpha,E}$ on functions $f \in \mathcal X$ by 
\bex
Tf(x) 
=
\frac{\alpha}{2} \int_{\bbR^2}
f(y)
\left| \frac{1}{E + y} \right|^s
 |h|^{s-\alpha -1 }  p_E\left(x  + h^2 (E+y)^{-1}\right)
\, d h \, dy.
\eex
\end{definition}

\begin{rem}
	\label{tlgl} 
	
Using this notation, we may write \eqref{gidef} as 
\be\label{Erecursionrewrite}
g_i(x ; E) = \big(T g_{i+1} ( \,\cdot \,; E) \big)(x).
\ee

\noindent Through $L$ applications of \eqref{Erecursionrewrite}, it follows that $g_0 = T^L g_L = T^L p_E$.   
\end{rem}

Before analyzing $T$, we state the following useful integral bound.
\begin{lem}\label{l:Tintegral}
Fix $\alpha \in (0,1)$ and $s \in (\alpha/2 , 1/2)$. 
There exists a constant $c(s,\alpha) > 0$ such that, for all $x \in \bbR$,
\bex
\int_{0}^\infty \frac{ r^{(s- \alpha)/2 -1} dr}{1 + | x -r |^{ 1 + \alpha/2} }
\le  
c^{-1} \min(1, |x|^{(s - \alpha)/2 -1}).
\eex
Additionally, if $x \ge 0$, we have 
\be\label{Tintegrallower}
c \min(1,x^{(s-\alpha)/2 -1})
\le 
\int_{0}^\infty \frac{ r^{(s-\alpha)/2 -1} dr}{1 + | x -r |^{ 1 + \alpha/2} },
\ee
and if $x \le 0$, 
\be\label{Tintegrallowerweak}
c \min(1,|x|^{s/2 - \alpha -1})
\le 
\int_{0}^\infty \frac{ r^{(s-\alpha)/2 -1} dr}{1 + | x -r |^{ 1 + \alpha/2} },
\ee
\end{lem}
\begin{proof}
First, suppose that $x>2$. 
We write
\begin{align*}
\int_{0}^\infty \frac{ r^{(s-\alpha)/2 -1} dr}{1 + | x -r |^{ 1 + \alpha/2} }
&=
\int_{0}^{x-1} \frac{ r^{(s-\alpha)/2 -1} dr}{1 + | x -r |^{ 1 + \alpha/2} } +\int_{x-1}^{x+1} \frac{ r^{(s-\alpha)/2 -1} dr}{1 + | x -r |^{ 1 + \alpha/2} } +\int_{x+1}^{\infty} \frac{ r^{(s-\alpha)/2 -1} dr}{1 + | x -r |^{ 1 + \alpha/2} }.
\end{align*}
For the first integral, we bound
\begin{flalign*}
\int_{0}^{x-1} \frac{ r^{(s-\alpha)/2 -1} dr}{1 + | x -r |^{ 1 + \alpha/2} }
\le
\int_{0}^{x-1} \frac{ r^{(s-\alpha)/2 -1} dr}{ | x -r |^{ 1 + \alpha/2} }.
\end{flalign*}
We now substitute $r = tx$ to obtain
\bex
\int_{0}^{x-1} \frac{ r^{(s-\alpha)/2 -1} dr}{ | x -r |^{ 1 + \alpha/2} }
=
x^{s/2 - \alpha -1 } \int_{0}^{1- x^{-1}} \frac{ t^{(s-\alpha)/2 -1} dt}{ | 1 -t|^{ 1 + \alpha/2} } .
\eex
The latter integral is straightforwardly bounded by splitting it into two pieces at $t=1/2$. We have
\bex
 \int_{0}^{1/2 } \frac{ t^{(s-\alpha)/2 -1} dt}{ | 1 -t|^{ 1 + \alpha/2} }
 \le C, 
\eex
and direct integration gives
\bex
 \int_{1/2}^{1- x^{-1}} \frac{ t^{(s-\alpha)/2 -1} dt}{ | 1 - t|^{ 1 + \alpha/2} } 
\le  C \int_{1/2}^{1- x^{-1}} \frac{  dt}{ | 1 - t|^{ 1 + \alpha/2} } 
\le
 C ( x^{\alpha/2} + 1). 
\eex

\noindent Summing these two estimates, we conclude that 
\begin{flalign}
	\label{integralr}
	\displaystyle\int_0^{x-1} \displaystyle\frac{r^{(s-\alpha)/2 - 1} dr}{1 + |x-r|^{1+\alpha/2}} \le C x^{s/2 - \alpha -1 } \cdot x^{\alpha/2} = C x^{(s-\alpha)/2 -1}.
\end{flalign}
when $x > 1$.
Similar reasoning gives a matching lower bound of $c x^{(s-\alpha)/2 -1}$ for the integral on the left side of \eqref{integralr} if $x > 2$. 

We next bound
\begin{flalign}
	\label{integralr1}
\int_{x-1}^{x+1} \frac{ r^{(s-\alpha)/2 -1} dr}{1 + | x -r |^{ 1 + \alpha/2} }
\le \int_{x-1}^{x+1}  r^{(s-\alpha)/2 -1} dr \le C x^{(s-\alpha)/2 -1},
\end{flalign}
and a matching lower bound follows similarly. Finally, again changing variables $r = tx$,
\begin{align}
	\label{integralr2} 
	\begin{aligned} 
\int_{x+1}^{\infty} \frac{ r^{(s-\alpha)/2 -1} dr}{1 + | x -r |^{ 1 + \alpha/2} }
& \le x^{s/2 -\alpha -1 } \int_{1 + x^{-1}}^\infty
\frac{ t^{(s-\alpha)/2 -1} dt}{ | 1 -t|^{ 1 + \alpha/2} }\\
&\le 
C x^{s/2 -\alpha -1 } 
\left(
\int_{1 + x^{-1}}^2
\frac{  dt}{ | 1 -t|^{ 1 + \alpha/2} }
+
\int_{2}^\infty
\frac{ t^{(s-\alpha)/2 -1} dt}{ | 1 -t|^{ 1 + \alpha/2} }
\right) \le C x^{(s-\alpha)/2 -1},
\end{aligned} 
\end{align}

\noindent and a matching lower bound follows similarly. Summing \eqref{integralr}, \eqref{integralr1}, and \eqref{integralr2} proves both the upper and lower bounds in the lemma for $x>2$.

When $x\in [-2,2]$, the continuity and positivity of
\bex
x \mapsto \int_{0}^\infty \frac{ r^{(s-\alpha)/2 -1} dr}{1 + | x -r |^{ 1 + \alpha/2} },
\eex

\noindent imply that it is bounded above and below by a uniform constant on the compact interval $x \in [-2, 2]$; this verifies the upper and lower bounds in the lemma in this case.

For $x< - 2$, we have
\bex
\int_{0}^\infty \frac{ r^{(s-\alpha)/2 -1} dr}{1 + | x -r |^{ 1 + \alpha/2} }
\le 
\int_{0}^\infty \frac{ r^{(s-\alpha)/2 -1} dr}{1 + | |x| -r |^{ 1 + \alpha/2} },
\eex
which proves the upper bound, since the $x>2$ case was shown above. For the lower bound, we use the change of variables $r = tx$ (and the fact that $|x - r| \ge 1$ for $r \ge 0 > -2 \ge x$) to deduce that 
\bex
\int_{0}^\infty \frac{ r^{(s-\alpha)/2 -1} dr}{1 + | x -r |^{ 1 + \alpha/2} } \ge c \displaystyle\int_0^{\infty} \displaystyle\frac{r^{(s-\alpha)/2 - 1} dr}{1 + |x-r|^{1+\alpha/2}} 
= c |x|^{s/2 -\alpha -1 } \int_{0}^\infty
\frac{ t^{(s-\alpha)/2 -1} dt}{ | 1  + t|^{ 1 + \alpha/2} }
\ge c 
|x|^{s/2 -\alpha -1 },
\eex

\noindent which establishes the lemma.
\end{proof}

We now use the previous lemma to estimate the norm of $T$. We also provide a lower bound on $T^2f$ for nonnegative $f \in \mathcal X$ that are nonzero on a set of positive measure.
\begin{lem}\label{l:Tbounds}
For any $s \in (\alpha, 1)$, there exists a constant $C = C(s) > 1$ such that, for all $f \in \mathcal X$,
\bex
 \| Tf \| \le C \| f\|.
\eex
\item
Further, if $f(x_0) \ge 0$ for all $x\in \bbR$ and $\int_{-\infty}^{\infty} \big| f(x) \big| dx > 0$, then there exists a constant $c_f = c_f (s) > 0$ such that 
\be\label{Tlower}
T^2 f(x) \ge \frac{c_f}{1 + x^{1 + (\alpha-s)/2} }.
\ee
\end{lem}

\begin{proof}
Using the fact (from \Cref{ar}) that $|p_E(x)| \le C \cdot \min \big\{ 1, |x|^{-1-\alpha/2} \big\}$, we have
\bex
\big|Tf(x)\big| \le C \int_{\bbR^2}
f(y)
\left| \frac{1}{E + y} \right|^s
 |h|^{s-\alpha -1 }  \frac{1}{ 1 + \big|x  + h^2 (E+y)^{-1} \big|^{1 + \alpha/2} }
\, d h \, dy.
\eex
Set $t  = h^2$, so that $ t^{-1/2} \, dt = 2 \, dh$. Then it follows that 
\bex
 \big| Tf (x) \big| \le \int_{-\infty}^\infty \int_{0}^\infty
f(y)
\left| \frac{1}{E + y} \right|^s
 |t|^{(s-\alpha)/2 -1 }  \frac{1}{ 1 + \big| x  + t (E + y)^{-1} \big|^{1 + \alpha/2} }
\, d t  \, dy.
\eex

\noindent Setting $r = t |E + y|^{-1}$ (implying that $|E + y|\, dr = dt$), we find that 
\begin{flalign}
	\label{tfx0} 
\big| Tf (x) \big| \le C \int_{-\infty}^\infty \int_{0}^\infty
f(y)
|E+ y|^{- (s+\alpha)/2 }   \frac{|r|^{(s-\alpha)/2 -1 }}{ 1 + \left|x\sgn(E-y)  + r \right|^{1 + \alpha/2} }
\, d r \, dy.
\end{flalign}

\noindent Since we also have (from \Cref{ar}) that $\big| p_E (x) \big| \ge c \cdot \min \big\{ 1, |x|^{-1-\alpha/2} \big\}$, the same reasoning yields
\begin{flalign}
	\label{tfx2}
	\big| Tf(x) \big| \ge c \displaystyle\int_{-\infty}^{\infty} \displaystyle\int_0^{\infty} f(y) |E + y|^{-(s+\alpha)/2} \displaystyle\frac{|r|^{(s-\alpha)/2-1}}{1 + \big| x \sgn (E-y) + r \big|^{1+\alpha/2}} dr dy.
\end{flalign} 

Applying \Cref{l:Tintegral} in \eqref{tfx0}, it follows that   
\begin{flalign}
	\label{tfx1} 
\big| Tf(x) \big| \le \frac{C}{1 + |x|^{1 + (\alpha-s)/2}} \int_{-\infty}^\infty 
f(y)
|E+y|^{- (s+\alpha)/2 } 
\, dy.
\end{flalign}
Using the definition of the norm $\| f \|$, we have
\bex
 \int_{-\infty}^{\infty}
f(y)
 |E+y|^{-(s+\alpha)/2}   \, dy \le 
C \|f \| \int_{-\infty}^{\infty}
 \frac{1}{1 + |y|^{1 + (\alpha-s)/2}} 
 |E+y|^{-(s+\alpha)/2}   \, dy
 \le C \cdot \| f\|,
\eex

\noindent where to deduce the last inequality we used the facts that $\frac{s+\alpha}{2} \in (0, 1)$ (to bound the integral around $y = -E$) and $\frac{\alpha-s}{2} + 1 + \frac{s + \alpha}{2} > 1$ (to bound the integral around $y = \infty$). This, together with \eqref{tfx1}, concludes the proof of the upper bound in the lemma.

Now let us establish the lower bound of the lemma; in what follows, we assume that $E \ge 0$, as the proof when $E \le 0$ is entirely analogous. Inserting \eqref{Tintegrallower} and \eqref{Tintegrallowerweak} in \eqref{tfx2}, we deduce 
\begin{flalign*}
\big| Tf(x) \big| \ge\frac{c}{1 + x^{1 + \alpha -s/2}} \int_{-\infty}^\infty
f(y)
|E-y|^{- (s+\alpha)/2 }  \, dy
= 
\frac{c_f}{1 + x^{1 + \alpha -s/2}}, \qquad \text{for any $x \in \mathbb{R}$},
\end{flalign*}
where $c_f = c_f (s) > 0$ is a constant depending on the function $f$. 
This implies that $Tf (x) \ge c \cdot c_f$, uniformly for $x \in [-2E, 2E]$. Inserting this in \eqref{tfx2}, and using \eqref{Tintegrallower}, we obtain
\begin{flalign*} 
	T^2 f(x) & \ge c \displaystyle\int_{\infty}^{\infty} \displaystyle\int_0^{\infty} f(y) |E+y|^{-(s+\alpha)/2} \displaystyle\frac{|r|^{(s-\alpha)/2 - 1}}{1 + \big| x \sgn (E-y) + r \big|^{1 + \alpha/2}} dr dy \\
	& \ge \displaystyle\frac{c}{1 + x^{1 + (\alpha-s)/2}} \Bigg( \one_{x \ge 0} \cdot \displaystyle\int_{E}^{2E} f(y) |E - y|^{-(s+\alpha)/2} dy \\
	& \qquad \qquad \qquad \qquad \qquad + \one_{x < 0} \cdot \displaystyle\int_{-2E}^E f(y) |E-y|^{-(s+\alpha)/2} dy \Bigg) \ge \displaystyle\frac{ c \cdot c_f}{1 + x^{1 + (\alpha-s)/2}},
\end{flalign*} 

\noindent where to obtain the second bound we observe that $\big| x \sgn (E - y) + r \big| = |r - x|$ either if $y > E$ and $x > 0$ or if $y < E$ and $x < 0$. This establishes the lower bound on $\| T^2 f \|$.
\end{proof}

The following lemma shows that $T$ is compact.
\begin{lem}\label{l:compact}
For all $s\in (\alpha,1)$, the operator $T\colon \mathcal X \rightarrow \mathcal X$ is compact. 
\end{lem}
\begin{proof}

Changing integration variables $u = h^2 (E + y)^{-1}$ in the definition \Cref{toperator} of $T$, we deduce
\begin{align*}
Tf(x) 
&=
\frac{\alpha}{2} \int_{\bbR} \int_0^\infty
f(y) |E+y|^{-(s+\alpha)/2}
 \cdot |u|^{(s-\alpha)/2 -1 }  p_E (x + u)
\, d u \, dy\\
& = 
\frac{\alpha}{2} \int_{-E}^\infty \int_0^\infty
f(y)
\left|E + y \right|^{-(s+\alpha)/2} 
 \cdot u^{(s-\alpha)/2 -1 }  p_E\left(x  + u\right)
\, d u \, dy \\ & \qquad +
\frac{\alpha}{2} \int_{-\infty}^{-E} \int_{-\infty}^0
f(y)
\left|E + y \right|^{-(s+\alpha)/2} 
 \cdot |u|^{(s-\alpha)/2 -1 }  p_E\left(x  + u\right)
\, d u \, dy.
\end{align*}
We observe that this separates the $u$ and $y$ variables, so that we have 
\be\label{compactnessrepresentation}
Tf(x)  = I_1 \cdot F_1(x) + I_2 \cdot F_2(x).
\ee

\noindent Here, $I_1 = I_1 (f)$ and $I_2 = I_2 (f)$ are constants given by 
\begin{flalign*}
I_1 = \frac{\alpha}{2} \int_{-E}^\infty 
f(y)
\left|E + y \right|^{-(s+\alpha) / 2} dy; \qquad I_2 = \displaystyle\frac{\alpha}{2} \displaystyle\int_{-\infty}^{-E} f(y) |E + y|^{-(s+\alpha)/2} dy,
\end{flalign*} 

\noindent whose convergence is guaranteed by the fact that $f \in \mathcal{X}$, with 
\begin{flalign}
	\label{i1i2cf}
	|I_1| + |I_2| \le C \cdot \| f\|.
\end{flalign}

\noindent Moreover, $F_1$ and $F_2$ are functions of $x$ given by
\begin{flalign*}
F_1 (x) = \displaystyle\int_0^{\infty} u^{(s-\alpha)/2 - 1} p_E (x + u) du; \qquad F_2 (x) = \displaystyle\int_{-\infty}^0 |u|^{(s-\alpha)/2 - 1} p_E (x + u) du.
\end{flalign*}

\noindent We have $F_1, F_2 \in \mathcal{X}$ with $\| F_1 \| + \| F_2 \| \le C$, since $\big| p_E (x) \big| \le C \cdot \min \{ 1, |x|^{-\alpha/2-1} \}$ by \Cref{ar}.

Now, we claim that the image of the unit ball $\mathcal B_1 \subset \mathcal X$ under $T$ is relatively compact, which implies the conclusion of the lemma. Let $\{f_n\}_{n=1}^\infty$ be an infinite sequence of functions with $f_n \in \mathcal B_1$ for all $n\in \bbN$. It suffices to exhibit a convergent subsequence of $\{T f_n\}_{n=1}^\infty $. 
We observe that, since each $f_j \in \mathcal B_1$, then by \eqref{i1i2cf} there exists a subsequence $\{g_n\}_{n=1}^\infty$ of $\{f_n\}_{n=1}^\infty$ such that the sequences $\{I_1(g_n)\}_{n=1}^\infty$ and $\{I_2(g_n)\}_{n=1}^\infty$ both converge. The representation \eqref{compactnessrepresentation} then shows that $\{T g_n \}_{n=1}^\infty$ converges in $\mathcal X$, establishing the lemma.
\end{proof}
We have the following proposition stating that the Perron--Frobenius eigenvalue of $T$ (corresponding to a nonnegative eigenvector)  is $\lambda (E,s, \alpha)$. Its proof is similar to the one given in \cite{tarquini2016level} and will be given in \Cref{s:pfeigenvector} below. 

	\begin{prop} 
	
		\label{l:pfeigenvector}
		
		Fix $s \in (\alpha, 1)$ and recall the notation of \Cref{lambdaEsalpha}. There exists a unique positive function $f_E = f_{E, s} \in \mathcal{X}$ such that 
		\begin{flalign*}
			Tf_E (x) = \lambda (E, s, \alpha) \cdot f_E (x) \qquad \text{for all $x \in \mathbb{R}$}.
		\end{flalign*}
	\end{prop}

We next state a consequence of the Krein--Rutman theorem, which will be useful for studying $T$. We first recall some functional analytic terminology. Let $X$ be a Banach space. A subset $K \subset X$ is called a \emph{cone} if it is a closed, convex set such that $c \cdot K \subseteq K$ for all $c > 0$ and $K \cap (-K) = \{ 0 \}$. If the cone $K$ has nonempty interior $K^0$, then it is called a \emph{solid cone}. An operator $S : X \rightarrow X$ is said to be \emph{positive} if $Sx \in K$ for every $x \in K$. It is further said to be \emph{strongly positive} if $Sx \in K^0$ for each $x \in S \setminus \{ 0 \}$. Finally, let $r(S)$ denote the spectral radius of $S$.

\begin{thm}[{\cite[Theorem 19.3]{deimling2010nonlinear}}]\label{t:kr}
Let $X$ be a Banach space, $K \subset X$ be a solid cone, and $S : X \rightarrow X$ be a compact, positive, linear operator with spectral radius $r = r(T)$. Then there exists an eigenvector $v \in K \setminus \{ 0 \}$ of $S$ with eigenvalue $r$, that is, $Sv = r \cdot v$. Moreover, if $S$ is strongly positive, then $r$ is a simple eigenvalue of $S$, and $S$ has no other positive eigenvalue. 
\end{thm}

With these preparations, we can prove the following corollaries. 

\begin{lem}
	
	\label{t2positivity}
	
	Let. $\mathcal{K} = \{ f \in \mathcal{X}: f \ge 0 \}\subset \mathcal{X}$ denote the cone of nonnegative functions in $\mathcal{X}$. Then, $T$ is positive with respect to $\mathcal{K}$, and $T^2$ is strongly positive.
	\end{lem} 
	
	\begin{proof}
 Consider the function $g \in \mathcal{K}$ defined by
\bex
g(x) = \frac{1}{1 + |x|^{(\alpha-s)/2 +1} }.
\eex
The open ball $\big\{ f \in \mathcal X : \| g -f \| < \frac{1}{2} \big\}$ is contained in $\mathcal{K}$, so the interior $\mathcal{K}^0$ of $\mathcal{K}$ is nonempty, meaning that $\mathcal{K}$ is a solid cone. 
Given $f \in \mathcal K$, it is immediate from the definition of $T$ that $Tf \in \mathcal K$, so $T$ is positive. Further, given $f \in \mathcal{K}$, \eqref{Tlower} implies that $T^2f \in \mathcal{K}^0$, so $T^2$ is strongly positive. 
	\end{proof}

\begin{cor}\label{l:pfcompute}
We have
\bex
\lim_{n\rightarrow \infty} \| T^n \|^{1/n} = \lambda (E, s, \alpha) .
\eex
where $\| T^n \|$ denotes the norm of the operator $T^n \colon \mathcal X \rightarrow \mathcal X$. 
\end{cor}
\begin{proof}
The operator $T^2$ is compact by \Cref{l:compact}, and strongly positive by \Cref{t2positivity}.
Then $T^2$ satisfies the hypothesis of \Cref{t:kr}, and we conclude that $r(T^2) = \lambda(E,s, \alpha)^2$,
where $\lambda(E,s, \alpha)$ is the eigenvalue associated to the positive eigenfunction $f_E$ given by \Cref{l:pfeigenvector}. Since Gelfand's spectral radius formula implies that
\bex
\lim_{n\rightarrow \infty} \| T^{2n} \|^{1/n} = r(T^2) = \lambda(E, s, \alpha)^2,
\eex

\noindent this shows that
\bex
r(T) = \lim_{n\rightarrow \infty} \| T^{n} \|^{1/n} = \lambda(E, s, \alpha)  = \lambda(E, s, \alpha),
\eex 

\noindent establishing the corollary.
\end{proof}

\subsection{Proof of \Cref{l:bootstrap}}\label{s:provebootstrap}


We are now ready to complete the proof of \Cref{l:bootstrap}.

\begin{proof}[Proof of \Cref{l:bootstrap}]
We first define the linear operator $S = S_{E,\alpha,s}$ on $\mathcal X$ by 
\bex
S(f) = \int_{-\infty}^\infty \left| \frac{1 }{E +y  } \right|^s  f(y) \, dy.
\eex
By the definition of the norm \eqref{xf} on $\mathcal X$, we have
\bex
\big|f(y) \big|
\le \| f \| \cdot \big(1 + |y|^{(\alpha-s) /2 + 1}\big)^{-1}.
\eex
It follows that
\begin{align}\label{Sfupper}
\big| S(f) \big| 
&\le \| f \| \int_{-\infty}^\infty \frac{1}{|E + y|^s } 
\frac{1}{1 + |y|^{(\alpha-s)/2 + 1}} \,dy =  C \cdot \| f \|.
\end{align}


We now claim that the sequence $R_{\star} (E + \iu \eta_j)$ converges in law to $R_{\loc}(E)$. It suffices to show that any weak limit point $R(E)$ of the sequence $R_{\star} (E + \iu \eta_j)$ equals $R_{\loc}(E)$. To that end, fix a subsequence $\{ \eta'_j\}_{j=1}^\infty$ of $\{ \eta_j\}_{j=1}^\infty$ such that $R(E + \iu \eta'_j)$ converges to $R(E)$. Then the hypothesis $\lim_{j\rightarrow\infty} \eta_j =0$ and \Cref{l:boundarybasics} together imply that $R(E) = R_{\loc}(E)$, showing that $R_{\loc} (E) = \lim_{j \rightarrow \infty} R_{\star} (E + \mathrm{i} \eta_j)$. Using \Cref{l:PhiE}, we conclude that
$\lim_{j \rightarrow \infty} \Phi_L(s;E + \iu \eta_j ) = 
\Phi_L^{\loc} (s ; E )$ and therefore that 
\be\label{phiint1}
\lim_{j \rightarrow \infty} L^{-1} \log \Phi_L(s;E + \iu \eta_j ) = 
L^{-1} \log \Phi_L^{\loc} (s ; E ) .
\ee

 We now examine the limit $\lim_{L \rightarrow \infty} L^{-1} \log \Phi_L^{\loc} (s ; E) = \lim_{L \rightarrow \infty} (2L)^{-1} \log \Phi_{2L}^{\loc}  (s; E)$. Let $g(x) = p_E (x)$, which satisfies $c \big( 1 + |x|^{(\alpha-s)/2 + 1} \big)^{-1} \le g(x) \le C \big( 1 + |x|^{(\alpha-s)/2 + 1} \big)^{-1}$, for any $x \in \mathbb{R}$, by \Cref{ar}. From \Cref{tlgl}, \Cref{l:Erecursion}, and \eqref{productN2}, we have 
\begin{flalign}
	\label{smomentt} 
\Phi_{2L}^{\loc} (s ; E) = S \left( T^{2L} \left( g \right)  \right).
\end{flalign}

\noindent Further, using \eqref{Sfupper}, we obtain
\bex
\Phi_{2L}^{\loc} (s ; E)  \le C \cdot \left \|  T^{2L} \left( g \right)  \right\| \le C \cdot r(T^{2L}) \cdot \| g\| \le C \cdot r(T^2)^L \cdot \| g \|,
\eex

\noindent where $r(T^2)$ and $r(T^{2L})$ denote the spectral radii of $T$ and $T^{2L}$, respectively. Then,
\begin{align}\label{loglambdaupper}
\lim_{L \rightarrow \infty} (2L)^{-1} \log \Phi_{2L}^{\loc} (s ; E) &\le 
\lim_{L \rightarrow \infty} (2L)^{-1}
\left( \log C   + L \log r(T^2)
+ \log \| g\| \right) \le \log \lambda(E, s, \alpha),
\end{align} 
where we used \Cref{l:pfcompute} in the last line.

For the lower bound, we note that \eqref{Tlower} and the fact that $f_E \in \mathcal{X}$ (by the first statement of \Cref{l:pfeigenvector}) impy that $T^2 g (x) \ge c \cdot f_E(x)$, for each $x \in \mathbb{R}$. Then, by \eqref{smomentt} and \Cref{l:pfeigenvector}, 
\bex
\Phi_L(s ; E) = S (T^L  g)\ge c \cdot
S ( T^{L-2} f_E) \ge 
c \cdot \lambda(E,s,\alpha)^{L-2} \cdot S(f_E) \ge c \cdot \lambda (E, s, \alpha)^{L-2},
\eex
and we conclude 
\bex
\lim_{L \rightarrow \infty} (2L)^{-1} \log \Phi_{2L}^{\loc} (s ; E)\ge
\log \lambda(E, s, \alpha).
\eex
Combining the previous equation with \eqref{loglambdaupper} gives
\be\label{phiLlimit}
\displaystyle\lim_{L \rightarrow \infty} L^{-1} \log \Phi_L^{\loc} (s; E) = \lim_{L \rightarrow \infty} (2L)^{-1} \log \Phi_{2L}^{\loc} (s ; E) 
= \log \lambda(E,s, \alpha).
\ee
From \eqref{limitr0j2}, we have
\be\label{finiteLbdd}
\left| L^{-1} \log \Phi_L(s ; E +\iu \eta_j ) 
- \phi(s; E + \iu \eta_j) 
\right| \le \frac{C}{L}
\ee
for all $k \in \Zplus$. Then using \eqref{phiint1} and \eqref{finiteLbdd}, we have
\bex
L^{-1} \log \Phi_L^{\loc} (s ; E ) - \frac{C}{L} = 
\lim_{j \rightarrow \infty} L^{-1} \log \Phi_L(s ; E +\iu \eta_j ) - \frac{C}{L}
\le \liminf_{j \rightarrow \infty} \phi(s; E + \iu \eta_j),
\eex
and 
\bex
\limsup_{j \rightarrow \infty} \phi(s; E + \iu \eta_j)
\le
\lim_{j \rightarrow \infty} L^{-1} \log \Phi_L(s ; E +\iu \eta_j ) + \frac{C}{L} 
= 
\log \Phi_L^{\loc} (s ; E ) + \frac{C}{L}.
\eex
Taking the limit as $L$ goes to infinity and using \eqref{phiLlimit}, we conclude that
\bex
\lim_{k \rightarrow \infty} \phi(s; E + \iu \eta_j) = \log \lambda(E,s,\alpha),
\eex

\noindent verifying the lemma.
\end{proof}

\section{The Perron--Frobenius Eigenvalue of $T$}\label{s:pfeigenvector}
	
	\label{TEigenvalue}
	
	In this section we establish \Cref{l:pfeigenvector}. First observe, by the Krein--Rutman theorem (\Cref{t:kr}) and the strong positivity of $T^2$ (from \Cref{t2positivity}), that $T$ admits an eigenfunction $f = f_E = f_{E, s} \in \mathcal{X}$ with a nonnegative eigenvalue $\lambda_0 = \lambda_0 (s, E) > 0$; further $\lambda_0$ is the unique positive eigenvalue, and it admits no other eigenvectors. In \Cref{Transformf}, we analyze the Fourier transform $\hat f$, and in \Cref{Lambda0Lambda} we use this analysis to complete the proof of \Cref{l:pfeigenvector}. Finally, in \Cref{LambdaOther}, we provide an alternative expression for $\lambda(E,s,\alpha)$, which will be used in \Cref{s:alphanear1} and \Cref{s:alphanear1uniqueness} below.
	
	We note that the definition \eqref{toperator} of $T$ gives, for any $x \in \mathbb{R}$, that
	\begin{flalign}
		\label{lambda0f}
		\lambda_0 \cdot f (x) = Tf (x) = \displaystyle\frac{\alpha}{2} \displaystyle\int_{\mathbb{R}^2} f (y) \bigg| \displaystyle\frac{1}{E+y} \bigg|^s |h|^{s-\alpha-1} p_E \bigg( x + \displaystyle\frac{h^2}{E+y} \bigg) \, dh\, dy.
	\end{flalign}

	\noindent Observe since $f_E \in \mathcal{X}$ that $f_E \in L^2 (\mathbb{R})$ (by the definition of the norm \eqref{xf} on $\mathcal{X}$, together with the fact that $\frac{\alpha-s}{2} + 1 > \frac{1}{2}$). In particular, the Fourier transform of $f_E$ satisfies $\widehat{f}_E \in L^2 (\mathbb{R})$.

	\subsection{Analysis of $\widehat{f}$}
	
	\label{Transformf}
	
	We recall our convention for the Fourier transform from \Cref{s:fouriertransform}. We then have the following lemma, whose statement and proof are similar to those of \cite[Equation (15)]{tarquini2016level}. In what follows, we will repeatedly use the fact that, for any $r \in (0, 1)$ and $\xi \in \mathbb{R}$, we have  
	\begin{flalign}
		\label{xrexi}
		\displaystyle\int_0^{\infty} e^{\mathrm{i} \xi x} x^{-r} dx = \Gamma (1 - r) \cdot \exp \bigg( \displaystyle\frac{\pi \mathrm{i}}{2} (1-r) \sgn (\xi) \bigg) \cdot |\xi|^{r-1}.
	\end{flalign}

	\begin{lem}
		\label{flambda1}
		
		For any $\xi \in \mathbb{R}$, we have  
		\begin{flalign}
			\label{lambda0fxi} 
			\begin{aligned}
			\lambda_0 \cdot \widehat{f} (\xi) & = \displaystyle\frac{\alpha}{2} \cdot \Gamma \bigg( \displaystyle\frac{s-\alpha}{2} \bigg) \cdot |\xi|^{(\alpha-s)/2} \cdot \widehat{p}_E (\xi) \\
			& \qquad \times \displaystyle\int_{-\infty}^{\infty} \displaystyle\frac{f(y)}{|E+y|^{(s+\alpha)/2}} \exp \bigg( \displaystyle\frac{\pi \mathrm{i} (\alpha-s)}{4} \cdot \sgn \big( \xi (E + y) \big) \bigg) \, dy.
			\end{aligned}
		\end{flalign}
	\end{lem}

	\begin{proof}
		
	It will be useful to begin by implementing a cutoff in the integral (over $h$) on the right side of \eqref{lambda0f} since, as is, it does not define a function in $L^1 (\mathbb{R})$. So, for any integer $N \ge 1$, define the function $g_N \in \mathcal{X}$ through 
	\begin{flalign*}
		\lambda_0 \cdot g_N (x) = \displaystyle\frac{\alpha}{2} \displaystyle\int_{-\infty}^{\infty} f(y) \bigg| \displaystyle\frac{1}{E+y}\bigg|^s \displaystyle\int_{-N}^N |h|^{s-\alpha} p_E \bigg( x + \displaystyle\frac{h^2}{E+y} \bigg) \, dh \, dy,
	\end{flalign*} 
	 
	 \noindent observing by \eqref{lambda0f} that $\lim_{N \rightarrow \infty} g_N = f$. Taking the Fourier transform yields
	\begin{flalign*}
		\lambda_0 \cdot \widehat{g}_N (\xi) = \displaystyle\frac{\alpha}{2} \displaystyle\int_{-\infty}^{\infty} e^{\mathrm{i} x \xi} \displaystyle\int_{-\infty}^{\infty} f (y) \bigg| \displaystyle\frac{1}{E+y} \bigg|^s \displaystyle\int_{-N}^{N} |h|^{s-\alpha-1} p_E \bigg( x + \displaystyle\frac{h^2}{E+y} \bigg) \, dh\, dy\, dx,
	\end{flalign*} 

	\noindent where the integral absolutely converges since $f \in \mathcal{X}$ and $p_E (w) \le C \big( |x| + 1 \big)^{\alpha/2 + 1}$ for any $w \in \mathbb{R}$ (by \Cref{ar}). Changing variables $t = |E+y|^{-1} \cdot h^2$, and then replacing $x$ by $x - t \sgn (E+y)$ yields 
	\begin{flalign}
		\label{lambda0gn} 
		\begin{aligned} 	 
		\lambda_0 \cdot \widehat{g}_N (x) & = \displaystyle\frac{\alpha}{2} \displaystyle\int_{-\infty}^{\infty} \displaystyle\int_{-\infty}^{\infty} \displaystyle\int_0^{N^2} e^{\mathrm{i} x \xi} f(y) \bigg| \displaystyle\frac{1}{E+y} \bigg|^{(\alpha+s)/2} t^{(s-\alpha)/2-1} p_E \big( x + t \sgn (E+y) \big) \, dt \, dy \, dx \\
		& = \displaystyle\frac{\alpha}{2} \displaystyle\int_{-\infty}^{\infty} e^{\mathrm{i} x \xi} p_E (x) dx \cdot \displaystyle\int_{-\infty}^{\infty} f(y) \bigg| \displaystyle\frac{1}{E+y} \bigg|^{(\alpha+s)/2} \displaystyle\int_0^{N^2} e^{-\mathrm{i} t \xi \sgn (E+y)} t^{(s-\alpha)/2-1} \, dy \, dt \\
		& = \displaystyle\frac{\alpha}{2} \cdot \widehat{p}_E (\xi) \displaystyle\int_{-\infty}^{\infty} f(y) \bigg| \displaystyle\frac{1}{E+y} \bigg|^{(\alpha+s)/2} \displaystyle\int_0^{N^2} e^{-\mathrm{i} t \xi \sgn (E+y)} t^{(s-\alpha)/2-1} \, dy \, dt.
		\end{aligned} 
	\end{flalign}

	\noindent By \eqref{xrexi}, we have 
	\begin{flalign*}
		\displaystyle\lim_{N \rightarrow \infty} \displaystyle\int_0^{N^2} e^{\mathrm{i} t \xi \sgn (E+y)} t^{(s-\alpha)/2-1} dt = \Gamma \bigg( \displaystyle\frac{s-\alpha}{2} \bigg) \cdot \exp \bigg( \displaystyle\frac{\pi \mathrm{i} (s-\alpha)}{4} \cdot \sgn \big( \xi (E + y) \big) \bigg) \cdot |\xi|^{(\alpha-s)/2}.
	\end{flalign*} 

	\noindent Inserting this into \eqref{lambda0gn} and using the fact that $\lim_{N \rightarrow \infty} g_N = f_N$ then yields the lemma (observing that the resulting integral converges absolutely, since $f \in \mathcal{X}$).	
	\end{proof}

	The next lemma provides a different form for $\widehat{f} (\xi)$.
	
	\begin{lem} 
		
		\label{fxie1} 
		
		For any $\xi \in \mathbb{R}$, we have 
		\begin{flalign*}
			\lambda_0 \cdot \widehat{f}_E (\xi) = \displaystyle\frac{1}{\pi} \cdot K_{\alpha, s} \cdot |\xi|^{(\alpha-s)/2} \cdot \widehat{p}_E (\xi) \displaystyle\int_{-\infty}^{\infty} e^{\mathrm{i} E \zeta} |\zeta|^{(\alpha+s)/2 - 1} \cdot \big( t_{\alpha} \cdot \mathbbm{1}_{\xi \zeta > 0} + t_s \cdot \mathbbm{1}_{\xi \zeta < 0} \big) \widehat{f} (\zeta)\, d \zeta.
		\end{flalign*}
	\end{lem} 
	
	\begin{proof} 
	First observe from \Cref{flambda1}, the explicit form \eqref{xtsigma} of $\widehat{p}_E$; and the convergence of the integral in \eqref{lambda0fxi} (by the fact that $f \in \mathcal{X}$) that 
	\begin{flalign}
		\label{fxiintegralc} 
		\big| \widehat{f} (\xi) \big| \le C \exp \big(-c |\xi|^{\alpha/2} \big), \qquad \text{and so $\widehat{f} \in L^1 (\mathbb{R})$}.
	\end{flalign}

	\noindent Next, by \eqref{lambda0fxi} and the fact that $f \in \mathcal{X}$, we obtain 
	\begin{flalign}
		\label{lambda0fxi2} 
		\begin{aligned}
		\lambda_0 \cdot \widehat{f} (\xi) & = \displaystyle\frac{\alpha}{2} \cdot \Gamma \bigg( \displaystyle\frac{s-\alpha}{2} \bigg) \cdot |\xi|^{(\alpha-s)/2} \cdot \widehat{p}_E (\xi) \\
		& \qquad \times \displaystyle\lim_{N \rightarrow \infty} \displaystyle\int_{-N}^N \displaystyle\frac{f(y)}{|E+y|^{(\alpha+s)/2}} \exp \bigg( \displaystyle\frac{\pi \mathrm{i} (\alpha-s)}{4} \cdot \sgn \big( \xi (E+y) \big)\bigg) \, dy.
		\end{aligned}
	\end{flalign}

	\noindent Applying the Fourier inversion identity $f(y) = \frac{1}{2 \pi} \int_{-\infty}^{\infty} e^{-\mathrm{i} y \zeta} \widehat{f} (\zeta)\, d \zeta$, we obtain
	\begin{flalign}
		\label{flambda0xi2n}
		\begin{aligned}
		\lambda_0 \cdot \widehat{f} (\xi) & = \displaystyle\frac{\alpha}{4 \pi} \cdot \Gamma \bigg( \displaystyle\frac{s-\alpha}{2} \bigg) \cdot |\xi|^{(\alpha-s)/2} \cdot \widehat{p}_E (\xi) \\
		& \qquad \times \displaystyle\lim_{N \rightarrow \infty} \displaystyle\int_{-\infty}^{\infty} e^{-\mathrm{i} y \zeta }\displaystyle\int_{-N}^N \displaystyle\frac{\widehat{f} (\zeta)}{|E+y|^{(\alpha+s)/2}} \exp \bigg( \displaystyle\frac{\pi \mathrm{i} (\alpha-s)}{4} \cdot \sgn \big( \xi (E + y) \big) \bigg) \, dy\, d \zeta \\
		& = \displaystyle\frac{\alpha}{4\pi} \cdot \Gamma \bigg( \displaystyle\frac{s-\alpha}{2} \bigg) \cdot |\xi|^{(\alpha-s)/2} \cdot \widehat{p}_E (\xi) \\
		& \qquad \times \displaystyle\lim_{N\rightarrow \infty} \displaystyle\int_{-\infty}^{\infty} e^{\mathrm{i} E \zeta} \widehat{f} (\zeta) \displaystyle\int_{-N-E}^{N-E} e^{-\mathrm{i} y \zeta} |y|^{-(\alpha+s)/2} \exp \bigg( \displaystyle\frac{\pi \mathrm{i} (\alpha-s)}{4} \cdot \sgn (\xi y) \bigg) \, dy \, d \zeta,
		\end{aligned}
	\end{flalign}

	\noindent where in the last equality we changed variables, setting the original $y$ now to $y - E$. Next, observe from \eqref{xrexi} that 
	\begin{flalign*}
		\displaystyle\lim_{N \rightarrow \infty} & \displaystyle\int_{-N-E}^0 e^{-\mathrm{i} y \zeta} |y|^{-(\alpha+s)/2} \exp \bigg( \displaystyle\frac{\pi \mathrm{i} (\alpha-s)}{4} \sgn (\xi y) \bigg) dy \\
		&  = \Gamma \bigg( 1 - \displaystyle\frac{\alpha+s}{2} \bigg) \cdot \exp \bigg( \displaystyle\frac{\pi \mathrm{i}}{2} \Big( 1 - \displaystyle\frac{\alpha+s}{2} \Big) \cdot \sgn (\zeta) \bigg)  \exp \bigg( \displaystyle\frac{\pi \mathrm{i} (s-\alpha)}{4} \cdot \sgn (\xi) \bigg) \cdot |\zeta|^{(\alpha+s)/2 - 1}; \\
		\displaystyle\lim_{N \rightarrow \infty} & \displaystyle\int_0^{N-E} e^{-\mathrm{i} y \zeta} |y|^{-(\alpha+s)/2} \exp \bigg( \displaystyle\frac{\pi \mathrm{i} (\alpha-s)}{4} \sgn (\xi y) \bigg) dy \\
		&  = \Gamma \bigg( \displaystyle\frac{\alpha+s}{2} - 1 \bigg) \cdot \exp \bigg( \displaystyle\frac{\pi \mathrm{i}}{2} \Big( \displaystyle\frac{\alpha+s}{2} - 1 \Big) \cdot \sgn (\zeta) \bigg)  \exp \bigg( \displaystyle\frac{\pi \mathrm{i} (\alpha - s)}{4} \cdot \sgn (\xi) \bigg) \cdot |\zeta|^{(\alpha+s)/2 - 1}.
	\end{flalign*} 

	\noindent Inserting these into \eqref{flambda0xi2n} (and using \eqref{fxiintegralc}), we find 
	\begin{flalign}
		\label{flambda03}
		\begin{aligned}
		\lambda_0 \cdot \widehat{f} (\xi) & = \displaystyle\frac{\alpha}{2\pi} \cdot \Gamma \bigg( \displaystyle\frac{s-\alpha}{2} \bigg) \cdot \Gamma \bigg( 1 - \displaystyle\frac{\alpha+s}{2} \bigg) \cdot |\xi|^{(\alpha-s)/2} \cdot \widehat{p}_E (\xi) \\
		& \qquad \times \displaystyle\int_{-\infty}^{\infty} e^{\mathrm{i} E \zeta} |\zeta|^{(\alpha+s)/2 - 1} \cdot \big( t_{\alpha} \cdot \mathbbm{1}_{\xi \zeta > 0} + t_s \cdot \mathbbm{1}_{\xi \zeta < 0} \big) \cdot \widehat{f}(\zeta) \, d \zeta,
		\end{aligned} 
	\end{flalign} 

	\noindent where we used the facts that 
	\begin{flalign*} 
		\exp \bigg( & \displaystyle\frac{\pi \mathrm{i}}{2} \Big( 1 - \displaystyle\frac{\alpha+s}{2} \Big) \cdot \sgn (\zeta) + \displaystyle\frac{\pi \mathrm{i} (s-\alpha)}{4} \cdot \sgn (\xi) \bigg) \\ 
		& \qquad + \exp \bigg( \displaystyle\frac{\pi \mathrm{i}}{2} \Big( \displaystyle\frac{\alpha+s}{2} - 1 \Big) \cdot \sgn (\zeta) + \displaystyle\frac{\pi \mathrm{i} (\alpha-s)}{4} \cdot \sgn (\xi) \bigg) \\
		& \qquad \qquad \qquad \qquad = 2 \bigg( \sin \Big( \displaystyle\frac{\pi \alpha}{2} \Big) \cdot \mathbbm{1}_{\xi \zeta > 0} + \sin \Big( \displaystyle\frac{\pi s}{2} \Big) \cdot \mathbbm{1}_{\xi \zeta < 0} \bigg),
	\end{flalign*}

	\noindent and recalled the definition of $t_r$ from \eqref{tlrk}. The lemma then follows from combining \eqref{flambda03} with the definition \eqref{tlrk} of $K_{\alpha, s}$. 	 
	 \end{proof}

 	\subsection{Solution for $\lambda_0$} 
	
	\label{Lambda0Lambda} 

	In this section we solve for the eigenvalue $\lambda_0$ of $T$ with eigenvector $f$, continuing to follow \cite{tarquini2016level}. To that end, we denote
	\begin{flalign}
		\label{llii}
		\begin{aligned}
		& I_- = I_- (E) = \displaystyle\frac{1}{\pi} \displaystyle\int_{-\infty}^0 e^{\mathrm{i} E \xi} |\xi|^{(\alpha+s)/2 - 1} \widehat{f} (\xi)\, d \xi; \qquad I_+ = I_+ (E) = \displaystyle\frac{1}{\pi} \displaystyle\int_0^{\infty} e^{\mathrm{i} E \xi} |\xi|^{(\alpha+s)/2-1} \widehat{f} (\xi) \, d \xi; \\
		& \ell_- = \ell_- (E) = \displaystyle\frac{1}{\pi} \displaystyle\int_{-\infty}^0 e^{\mathrm{i} E \xi} |\xi|^{\alpha-1} \widehat{p}_E (\xi)\, d\xi; \qquad \qquad \ell_+ = \ell_+(E) = \displaystyle\frac{1}{\pi} \displaystyle\int_0^{\infty} e^{\mathrm{i} E \xi} |\xi|^{\alpha-1} \widehat{p}_E (\xi) \, d \xi.
		\end{aligned} 
	\end{flalign}

	The following corollary, which follows from \Cref{fxie1}, indicates that $(I_+, I_-)$ satisfy a certain linear relation.
	
	\begin{cor}
		
		\label{isumi}

	For any $E \in \mathbb{R}$, we have 
	\begin{flalign*}
		\lambda_0 \cdot I_+ = K_{\alpha, s} \ell_+ (t_{\alpha} I_+ + t_s I_-); \qquad \lambda_0 \cdot I_- = K_{\alpha, s} \ell_- (t_s I_+ + t_{\alpha} I_-).
	\end{flalign*}
	
	\end{cor}
	
	\begin{proof}
		
		By \Cref{fxie1} and \eqref{llii}, we have
		\begin{flalign*}
			 \lambda_0 \cdot |\xi|^{(\alpha+s)/2 - 1} \widehat{f} (\xi) \cdot \mathbbm{1}_{\xi > 0} & = \displaystyle\frac{1}{\pi} \cdot K_{\alpha, s}  \cdot |\xi|^{\alpha-1} \widehat{p}_E (\xi) \cdot (t_{\alpha} I_+ + t_s I_-) \cdot \mathbbm{1}_{\xi > 0}; \\
			 \lambda_0 \cdot |\xi|^{(\alpha+s)/2 - 1} \widehat{f} (\xi) \cdot \mathbbm{1}_{\xi < 0} & = \displaystyle\frac{1}{\pi} \cdot K_{\alpha, s} \cdot |\xi|^{\alpha-1} \widehat{p}_E (\xi) \cdot (t_s I_+ + t_{\alpha} I_-) \cdot \mathbbm{1}_{\xi < 0}.
		\end{flalign*}
		
		\noindent Integrating over $\xi$ and using \eqref{llii} again then yields the lemma.
	\end{proof}

	The following lemma, which makes use of the nonnegativity of $f$, indicates that $(I_+, I_-) \ne (0, 0)$ (meaning that the linear relation between $(I_+, I_-)$ given by \Cref{isumi} is nondegenerate). 
	
	\begin{lem}
		
		\label{iireal0} 
		
		For any $E \in \mathbb{R}$, we have that $(I_+, I_-) \ne (0, 0)$.
	\end{lem}

	\begin{proof}
		
		It suffices to show that $\Real I_+ \ne 0$. To that end, observe that 
		\begin{flalign}
			\label{isum2}
			\pi \cdot I_+ = \mathcal{F} \big( \mathbbm{1}_{\xi > 0} \cdot |\xi|^{(\alpha+s)/2 - 1} \cdot \widehat{f} (\xi) \big) (E) = (2 \pi)^{-1} \cdot \big( \mathcal{F} \big( \mathbbm{1}_{\xi > 0} \cdot |\xi|^{(\alpha+s)/2 - 1} \big) \ast \mathcal{F}^2 f \big) (E).
		\end{flalign}
	
		\noindent Defining $\widetilde{f} \in \mathcal{X}$ by $\widetilde{f}(x) = f(-x)$, observe by Fourier inversion and \eqref{xrexi} that 
		\begin{flalign*}
			\mathcal{F}^2 f = 2\pi \cdot \widetilde{f}; \qquad \mathcal{F} \big( \mathbbm{1}_{\xi > 0} \cdot |\xi|^{(\alpha+s)/2-1} \big) = \Gamma \bigg( \displaystyle\frac{\alpha+s}{2} \bigg) \cdot \exp \bigg( \displaystyle\frac{\pi \mathrm{i}}{4} (s + \alpha) \sgn (x) \bigg) \cdot |x|^{-(\alpha+s)/2}.
		 \end{flalign*} 
	 
	 	\noindent Inserting this, together with the facts that $\widetilde{f} (x) = f(-x) \in \mathbb{R}_{\ge 0}$ and
	 	\begin{flalign*}
	 		\Real \exp \bigg( \displaystyle\frac{\pi \mathrm{i}}{4} (\alpha+s) \sgn(x) \bigg) = \cos \bigg( \displaystyle\frac{\pi}{4} (\alpha+s) \bigg) \in (0, 1),
	 	\end{flalign*}
 	
 		\noindent into \eqref{isum2}, it follows that 
	 	\begin{flalign*}
	 		\pi \cdot \Real I_+ = \Gamma \bigg( \displaystyle\frac{\alpha+s}{2} \bigg) \cos \bigg( \displaystyle\frac{\pi}{4} (\alpha+s) \bigg) \cdot \big( \widetilde{f} \ast |x|^{-(\alpha+s)/2} \big) (E) > 0,
	 	\end{flalign*}
 	
 		\noindent where in the last inequality we further used the fact that $\widetilde{f}(x) = f(x) \in \mathbb{R}_{\ge 0}$ for all $x \in \mathbb{R}$, and that $\widetilde{f} \ne 0$. In particular, this shows that $\Real I_+ \ne 0$, which establishes the lemma.
	\end{proof}
	
	Given \Cref{isumi} and \Cref{iireal0}, we can quickly establish \Cref{l:pfeigenvector}
	
	\begin{proof}[Proof of \Cref{l:pfeigenvector}]
	
	Observe that \Cref{isumi} implies the equation 
	\begin{flalign*}
		\boldsymbol{A} \cdot \boldsymbol{v} = \boldsymbol{0}, \qquad \text{where} \qquad \boldsymbol{A} = \left[ \begin{array}{cc} t_{\alpha} K_{s, \alpha} \ell_+ - \lambda_0 & t_s K_{s,\alpha} \ell_+ \\ t_s K_{s, \alpha} \ell_- & t_{\alpha} K_{s, \alpha} \ell_-  - \lambda_0 \end{array}\right], \qquad \text{and} \qquad \boldsymbol{v} =  \left[ \begin{array}{cc} I_+ \\ I_- \end{array} \right].
	\end{flalign*}

	\noindent By \Cref{iireal0}, this implies that $\det \boldsymbol{A} = 0$, meaning that $\lambda_0$ is the positive (as $\lambda_0$ is the Perron--Frobenius eigenvalue of $T$) solution to
	\begin{flalign*}
		\lambda_0^2 - t_{\alpha} K_{s, \alpha} (\ell_+ + \ell_-)\cdot \lambda_0 + (t_{\alpha}^2 - t_s^2) K_{s,\alpha}^2 \ell_+ \ell_- = 0.
	\end{flalign*}
	
	\noindent This, together with \Cref{lambdaEsalpha} and the facts that $\ell_+ = \ell (E)$ and that $\ell_+ = \overline{\ell_-}$ (since $\widehat{p}_E (x)$ is the conjugate of $\widehat{p}_E (-x)$ for each $x \in \mathbb{R}$, by \eqref{xtsigma}) yields the proposition.	
	\end{proof} 

	\subsection{Alternative Expression for $\lambda (E, s, \alpha)$}
	
	\label{LambdaOther} 
	
	In this section we provide a different expression for the eigenvalue $\lambda (E, s, \alpha)$ from \Cref{lambdaEsalpha}, in terms of the random variable $R_{\loc}$ from \Cref{locre}. 
	
	\begin{lem} 
	
	\label{realimaginaryl}
	
	For every $E \in \mathbb{R}$, we have 
	\begin{flalign*}
		\Real \ell (E) & = \pi^{-1} \cdot \Gamma (\alpha) \cdot \cos \Big( \displaystyle\frac{\pi \alpha}{2} \Big) \cdot \mathbb{E} \Big[ |R_{\loc} (E)|^{\alpha} \Big]; \\
		\Imaginary \ell (E) & =  \pi^{-1} \cdot \Gamma (\alpha) \cdot \sin \Big( \displaystyle\frac{\pi \alpha}{2} \Big) \cdot \mathbb{E} \Big[ \big| R_{\loc} (E) \big|^{\alpha} \cdot \sgn \big(-  R_{\loc} (E) \big) \Big].
	\end{flalign*}
	\end{lem} 

	\begin{proof} 
		
	By \eqref{tlrk}, \eqref{xrexi}, and Fourier inversion, we have 
	\begin{flalign*}
		 \ell (E) = \pi^{-1} \cdot \mathcal{F} \big( \mathbbm{1}_{\xi>0} \cdot |\xi|^{\alpha-1} \cdot \widehat{p}_E \big) (E) & = \displaystyle\frac{1}{2\pi^2} \cdot \Big( \mathcal{F} (\mathbbm{1}_{\xi > 0} \cdot |\xi|^{\alpha-1} \big) \ast \mathcal{F}^2 (p_E) \Big) (E) \\
		 & = \pi^{-1} \cdot \Gamma (\alpha) \cdot \Bigg( \bigg( \exp \Big( \displaystyle\frac{\pi \mathrm{i} \alpha}{2} \cdot \sgn (x) \Big) \cdot |x|^{-\alpha} \bigg) \ast \widetilde{p}_E \Bigg) (E),
	\end{flalign*}

	\noindent where we have defined $\widetilde{p}_E : \mathbb{R} \rightarrow \mathbb{R}$ by setting $\widetilde{p}_E (x) = p_E (-x)$ for all $x \in \mathbb{R}$. Thus,
	\begin{flalign*}
		\Real \ell (E) & = \pi^{-1} \cdot \Gamma (\alpha) \cos \Big( \displaystyle\frac{\pi \alpha}{2} \Big) \cdot \big( |x|^{-\alpha} \ast \widetilde{p}_E \big) (E) \\
		& = \pi^{-1} \cdot \Gamma (\alpha) \cdot \cos \Big( \displaystyle\frac{\pi \alpha}{2} \Big) \cdot\displaystyle\int_{-\infty}^{\infty} p_E (x-E) \cdot |x|^{-\alpha} dx \\
		& = \pi^{-1} \cdot \Gamma (\alpha) \cdot \cos \Big( \displaystyle\frac{\pi \alpha}{2} \Big) \cdot \displaystyle\int_{-\infty}^{\infty} p_E (x) \cdot |x+E|^{-\alpha} dx \\
		& = \pi^{-1} \cdot \Gamma (\alpha) \cdot \cos \Big( \displaystyle\frac{\pi \alpha}{2} \Big) \cdot \mathbb{E} \Big[ \big| R_{\loc} (E) \big|^{\alpha} \Big],
	\end{flalign*} 

	\noindent where in the third equality we changed variables (mapping $x$ to $x + E$) and in the fourth we recalled (from \Cref{pkappae}) the definition $p_E$ as the denstiy of $\varkappa_{\loc} (E)$ and (from \Cref{locre}) that $R_{\loc} (E) = - \big(\varkappa_{\loc} (E) + E \big)^{-1}$. This establishes the first statement of the lemma; the proof of the latter is entirely analogous and is thus omitted.
	\end{proof}

	\begin{lem} 
		
		\label{2lambdaEsalpha}
		
		For any real numbers $s \in (\alpha, 1)$ and $E \in \mathbb{R}$, we have
		\begin{flalign*}
			\lambda (E, s, \alpha) & =  \pi^{-1} \cdot K_{\alpha, s} \cdot \Gamma (\alpha) \cdot \bigg(  t_{\alpha} \sqrt{1 - t_{\alpha}^2} \cdot \mathbb{E} \Big[ \big| R_{\loc} (E) \big|^{\alpha} \Big] \\
			& \qquad + \sqrt{t_s^2 (1 - t_{\alpha}^2) \cdot \mathbb{E} \Big[ \big| R_{\loc} (E) \big|^{\alpha} \Big]^2 + t_{\alpha}^2 (t_s^2 - t_{\alpha}^2) \cdot \mathbb{E}  \Big[ \big| R_{\loc} (E) \big|^{\alpha} \cdot \sgn \big( - R_{\loc} (E) \big) \Big]^2} \bigg).
		\end{flalign*}
	\end{lem} 

	\begin{proof} 
		
	Denoting $\lambda = \lambda (E, s, \alpha)$, we have from \Cref{lambdaEsalpha} that
	\begin{flalign*}
		\lambda = K_{\alpha, s} \Big( t_{\alpha} \Real \ell (E) + \sqrt{t_s^2 \cdot \big( \Real \ell (E) \big)^2 + (t_s^2 - t_{\alpha}^2) \cdot \big( \Imaginary  \ell (E) \big)^2 } \Big). 
	\end{flalign*}
	
	\noindent This, together with \eqref{realimaginaryl}, yields the lemma. 	
	\end{proof}

\section{Continuity of the Lyapunov Exponent}\label{s:lyapunovcontinuity}

In this section we show \Cref{l:phicontinuity}, which essentially states that $\varphi(s; E) = \limsup_{\eta \rightarrow 0} \varphi (s; E + \mathrm{i} \eta)$ is continuous in $E$ on the region where $\varphi (s; E) < 0$. To that end, it suffices to show two statements. First, $\varphi (s; E + \mathrm{i} \eta)$ is uniformly continuous in $\eta > 0$ for any fixed $E = E_0$ with $\varphi (s; E_0) < 0$. Second, $\varphi (s; E + \mathrm{i} \eta)$ is uniformly continuous in $E$ if $\eta = \eta_0$ uniformly bounded away from $0$. Let us first provide a heuristic for the former, which is more involved.

Recalling the definition \eqref{szl} of $\phi$, it suffices to understand the continuity of $\Phi_L(s; E + \mathrm{i} \eta)$ in $\eta$, which was defined in \eqref{sumsz1} as the sum of the $s$-th powers of off-diagonal resolvent entries. For the purposes of this heuristic, we take $L = 1$ and indicate the continuity of $\Phi_{1} (s; z) = \mathbb{E} \big[ \sum_{v \sim 0} | R_{0v} (z)|^s \big]$. Fixing $\eta_1 > \eta_2 > 0$ and letting $(z_1, z_2) = (E + \mathrm{i} \eta_1, E + \mathrm{i} \eta_2)$, we have since $s < 1$ that 
\begin{align}\label{thediff}
\left|
\displaystyle\sum_{v \in \mathbb{V}  (L)} \big| R_{0v} (z_1) \big|^s
-
\displaystyle\sum_{v \in \mathbb{V}  (L)} \big| R_{0v} (z_1) \big|^s
\right| 
&\le 
\displaystyle\sum_{v \in \mathbb{V}  (L)} \big| R_{0v} (z_1) 
- R_{0v} (z_1) \big|^s.
\end{align}

\noindent In view of the product expansion \Cref{rproduct}, we have that $R_{0v} (z) = -R_{00} (z) T_{0v} R_{vv}^{(0)} (z)$. Thus, to show that $R_{0v} (z_1) \approx R_{0v} (z_2)$, it plausibly suffices to show that $R_{00} (z_1) \approx R_{00} (z_2)$ with high probability (which implies $R_{vv}^{(0)} (z_1) \approx R_{vv}^{(0)} (z_2)$ by symmetry). To that end, we apply the resolvent identity $\bm{A}^{-1} - \bm{B}^{-1} = \bm{A}^{-1} (\bm{B} - \bm{A}) \bm{B}^{-1}$ to write
\begin{flalign}
	\label{r00z1r00z2} 
	\begin{aligned}
\big| R_{00} (z_1) - R_{00} (z_2) \big| & = |z_2-z_1| \displaystyle\sum_{w \in \mathbb{V} } \big| R_{0w} (z_1) R_{w0} (z_2) \big| \\
& \le |z_1 - z_2| \Bigg( \displaystyle\sum_{w \in \mathbb{V}} \big| R_{0w} (z_1) \big|^2 \Bigg)^{1/2} \Bigg( \displaystyle\sum_{w \in \mathbb{V}} \big| R_{0w} (z_2) \big|^2 \Bigg)^{1/2}.
\end{aligned} 
\end{flalign} 

\noindent Applying the Ward identity \eqref{sumrvweta} would give the bound
\begin{align}
	\label{r002} 
\big| R_{00}(z_1) -  R_{00}(z_2) \big| \le |z_1 - z_2| (\eta_1\eta_2) ^{-1/2} 
\big( \Im R_{00}(z_1) \big)^{1/2}
\big( \Im R_{00}(z_2) \big)^{1/2}.
\end{align}

\noindent Assuming for $i \in \{ 1, 2 \}$ that $\Imaginary R_{00} (z_i) \le C$ (that is, it is bounded), this would suggest that $\big| R_{00} (z_1) - R_{00} (z_2) \big| \le C$, which does not quite give a continuity bound; we thus require a minor improvement on the right side of \eqref{r002} (by a factor of $\eta_1^c$, for example). 

Hence, we instead use the assumption $\varphi (s; z_1) < 0$ to deduce that, with high probability, 
\begin{flalign}
	\label{r0wz1} 
	\displaystyle\sum_{w \in \mathbb{V}} \big| R_{0w} (z_1) \big|^2 \le \Bigg( \displaystyle\sum_{w \in \mathbb{V}} \big| R_{0w} (z_1) \big|^s \Bigg)^{2/s} = \Bigg( \displaystyle\sum_{k=0}^{\infty} \displaystyle\sum_{w \in \mathbb{V} (k)} \big| R_{0w} (z_1) \big|^s \Bigg)^{2/s} \le  C.
\end{flalign}

\noindent Inserting this into \eqref{r00z1r00z2} and again applying the Ward identity \eqref{sumrvweta} gives 
\begin{flalign*}
	\big| R_{00} (z_1) - R_{00} (z_2) \big| \le C |z_1 - z_2| \eta_2^{-1/2} \big( \Imaginary R_{00} (z_2) \big)^{1/2},
\end{flalign*}

\noindent which upon again assuming that $\Imaginary R_{00} (z_2) \le C$ gives the required improvement on the right side of \eqref{r002}. Indeed, it implies that $\big| R_{00} (z_1) - R_{00} (z_2) \big| \le C \eta_2^{-1/2} |\eta_1 - \eta_2| \le C \eta_1^{1/4}$ if $\eta_2 \ge \eta_1^{3/2}$, and so by repeatedly applying this bound on the sequence of (for example) $\eta_j = e^{-(3/2)^j}$ tending to $0$ yields the uniform continuity of $R_{00} (E + \mathrm{i} \eta)$ in $\eta$. 

The argument outlined above is not entirely valid in (at least) two ways. First, it assumes that $R_{00} (z_1) \approx R_{00} (z_2)$ implies $R_{0v} (z_1) \approx R_{0v} (z_2)$, but by the product expansion \Cref{rproduct} one should only expect this to be true if $T_{0v}$ is not too large. We cirumvent this issue by first truncating all large $T_{0v}$ and show that omitting them from the fractional moment sums yields a negligible error; we will in fact do the same for all small $T_{0v}$ in order to reduce the number of entries in $\mathbb{V}(L)$, which is summed over in the definition \eqref{sumsz1} of $\Phi_L$. Second, the above heuristic conflates the high probability estimate \eqref{r0wz1} with an expectation bound. We confront this issue by slightly varying the moment parameter $s$ in the above bounds, namely, by comparing $s$-th moments to the $s(1+\kappa)$-th ones and showing that they are not too distant. These are done in \Cref{s:continuitypreliminary}, with \Cref{l:phicutoff} and \Cref{momentkappa} addressing the former and latter issues, respectively. In \Cref{Estimate1} we then use these bounds to show the continuity of $\varphi (s; E + \mathrm{i} \eta)$ in $E$ (\Cref{l:goingsideways}) and $\eta$ (\Cref{l:goingdown} and \Cref{l:goingup}), which are used in \Cref{s:continuityproof} to complete the proof of \Cref{l:phicontinuity}.

Throughout this section, we fix $\alpha \in (0,1)$, and constants $\eps > 0$ and $B > 1$. All complex numbers $z \in \bbH$ discussed here will satisfy $\eps \le | \Re z | \le B$ and $\Im z \in (0,1)$ (including those labeled $z_1$ and $z_2$). All constants below depend on $\alpha$, $\eps$, and $B$, even when not stated explicitly.

\subsection{Preliminary Estimates}
\label{s:continuitypreliminary}

We begin with a fractional moment bound for the resolvent.
\begin{lem}\label{l:qsexpectation}
Fix $s \in (\alpha,1)$ and $\delta \in (0,1)$. There exists a constant $C=C(\delta,s) > 1$  such that the following holds for any $z = E + \iu \eta \in \bbH$ such that $\eps \le |E| \le B$ and $\eta \in (0,1)$. If $\varphi (s; z) < -\delta$, then
\bex
\E \left[ \big( \Im R_{00}( z ) \big)^{s/2} \right] \le C \eta^{s/2}.
\eex
\end{lem}
\begin{proof}
		
		We abbreviate $R_{0v} = R_{0v} (z)$ for any vertex $v \in \mathbb{V}$. By the Ward identity \eqref{sumrvweta}, we have
		
		\bex
		\Im R_{00}( z ) = \eta \sum_{v \in \mathbb{V}} |R_{0v}|^2,
		\eex
		so it suffices to bound the sum in the previous equation.
		We have
		\begin{flalign*}
	\Bigg( \displaystyle\sum_{v \in \mathbb{V}} |R_{0v}|^2 \Bigg)^{s/2} = \Bigg( \displaystyle\sum_{L = 0}^{\infty} \displaystyle\sum_{v \in \mathbb{V} (L)} |R_{0v}|^2 \Bigg)^{s/2} \le \displaystyle\sum_{L = 0}^{\infty} \displaystyle\sum_{v \in \mathbb{V} (L)} |R_{0v}|^{s},
		\end{flalign*}
		
		\noindent where the last bound follows from $s/2 < 1$. Taking expectations, applying \Cref{moment1} (and \eqref{limitr0j2}), and using the fact that $\varphi (s; z) < -\delta$, we deduce the existence of a constant $C > 1$ such that
		\begin{flalign*}
	 \mathbb{E}  \Bigg[ \displaystyle\sum_{L = 0}^{\infty} \displaystyle\sum_{v \in \mathbb{V} (L)} |R_{0j}|^{s} \Bigg] =   \displaystyle\sum_{L = 0}^{\infty} \Phi_L (s; z) \le C  \displaystyle\sum_{L = 0}^{\infty} e^{-\delta L} \le  \delta^{-1} C ,
		\end{flalign*}
which completes the proof.
\end{proof}

The previous lemma quickly implies \Cref{p:imvanish}.

\begin{proof}[Proof of \Cref{p:imvanish}]
	We only prove the first part, since the second is similar.
	By \Cref{l:qsexpectation}, we have $\mathbb{E} \big[ (\Imaginary R_{\star} (E + \iu\eta_j))^s \big] \le C \eta_j^s$ for all $j\in\Zplus$. Then $(\Imaginary R_{\star} (E + \iu \eta_j))^s$ converges to $0$ in expectation as $j$ tends to infinity, which implies it converges to zero in probability.
	\end{proof}

For the next lemma, we recall the definition of the event $\mathscr{D} (v, w ; \omega)$ from  \Cref{br0vw}, and for brevity set $\mathscr{D}(v) = \mathscr{D}(v; \omega) = \mathscr{D} (0, v; \omega)$. We also make the following definition.

	\begin{definition}
		
		\label{br0vw2} 
		
		Fix $z \in \mathbb{H}$; let $v, w \in \mathbb{V}$ be vertices with $v \prec w$; and let $\omega \in (0,1)$ and $\Omega > 1$ be real numbers. For any vertex $u \in \mathbb{V}$ with $v \preceq u \prec w$, define the events $\mathscr{T} (u; v, w) = \mathscr{T} (u; v, w; \omega, \Omega) = \mathscr{T} (u; v, w; \omega, \Omega; z)$ and $\mathscr{T} (v, w) = \mathscr{T} (v, w; \omega, \Omega) = \mathscr{T} (v, w; \omega, \Omega; z)$ by 
		\begin{flalign*}
			\mathscr{T} (u; v, w) = \big\{  \omega \le  |T_{uu_+}| \le \Omega \big\}; \qquad \mathscr{T} (v, w) = \bigcap_{v \preceq u \prec w} \mathscr{T} (u; v, w).
				\end{flalign*}	
We also set $\mathscr{T}(v) = \mathscr{T}(v; \omega, \Omega) = \mathscr{T} (0, v; \omega, \Omega)$.	
\end{definition}

The next lemma shows the vertices $v\in \bbV$ such that $\one_{\mathscr{D}(v)} = 0$ may be neglected in the sum defining $\Phi_L$. 

\begin{lem}\label{l:phicutofflower}
Fix $s \in (\alpha,1)$, $ \theta \in (0,1/2)$, and $L \in \Zplus$.  There exist constants $C(s) > 1$  and $c( s) > 0$ 
such that for any $z = E + \iu \eta \in \bbH$ such that $\eps \le |E| \le B$ and $\eta \in (0,1)$, and any $\omega \in [0, c \theta^C]$, we have
\bex
\left|
\E 
\left[
\sum_{v \in \bbV(L) } 
\left|  R_{0v}(z) \right|^s
\right]
-
\E 
\left[
\sum_{v \in \bbV(L) } 
\left|  R_{0v}(z) \right|^s \one_{\mathscr{D}(v)}
\right]
\right|
 \le \theta L C^{1+L}  .
\eex
\end{lem} 
\begin{proof}
By \Cref{ry0estimate} and the second part of \Cref{estimatemoment1},
\begin{align*}
&\E 
\left[
\sum_{v \in \bbV(L) } 
\left|  R_{0v}(z) \right|^s
\right]
-
\E 
\left[
\sum_{v \in \bbV(L) } 
\left|  R_{0v}(z) \right|^s \one_{\mathscr{D}(v)}
\right]\\
&
\le 
\big(1 - (1 -\theta)^L )\cdot \E 
\left[
\sum_{v \in \bbV(L) } 
\left|  R_{0v}(z) \right|^s
\right]\le \theta L C^{1+ L},
\end{align*}
which completes the proof. In the last inequality, we used $\theta < 1/2$ and Taylor expansion.
\end{proof}
We next state some useful moment bounds on the sum defining $\Phi_L$, with estimate the $1+\kappa$ moment of this sum for any $\kappa >0$. 
\begin{lem}
	
\label{momentkappa}  
Fix $s\in (\alpha,1)$, $ \omega \in (0,1 )$, $L \in \Zplus$, and $  \Omega\ge 1$.  Fix $\kappa,\chi> 0$ such that 
\be\label{chisconditions}
 s(1+2\kappa)/(1-2\kappa) < 1, \quad \chi > (s-\alpha)/2, \quad \chi (1+2\kappa) / (1 - 2\kappa) < 1.\ee
 There exist constants $C(\chi,s) > 1$  and $c( \chi, s) > 0$ such that the following holds.
Fix $E_1, E_2 \in \bbR$ such that $|E_1|, |E_2|\in [\eps, B]$, and fix $\eta_1, \eta_2 \in (0,1)$.  Set $z_i = E_i + \iu \eta_i$ for $i=1,2$. 
Then we have 
\begin{align}\label{fracmom1}
 \E\left[
\left(\sum_{w\in \mathbb{V}(1)} \sum_{v \in \bbD_{L-1}(w) }
\left|R_{00}(z_1)\right|^\chi \left|
 T_{0w}
R^{(0)}_{w v}(z_1)
\right|^s \one_{\mathscr{D}(v)} \right)^{1+\kappa}
\right]\le C^{1+L}  \omega^{-2  \kappa L }
\end{align}
and
\begin{align}\label{fracmom2}
 \E\left[
\left(
\sum_{w\in \mathbb{V}(1)}
\sum_{v \in \bbD_{L-1}(w) }
\left|R_{00}(z_2) \right|^\chi \left|
 T_{0w }
R^{(0)}_{w v}(z_1)
\right|^s \one_{\mathscr{T}(v)} \right)^{1+\kappa}
\right]\le \Omega \cdot C^{1+L}  \omega^{-2 \kappa L }.
\end{align}
\end{lem}
\begin{proof}
We begin with the proof of  \eqref{fracmom2}.
We let $V$ denote the random variable equal to the number of nonzero $\one_{\mathscr{T}(v)}$ at level $L$ of the Poisson weighted infinite tree:
\bex
V = \left| \{ v \in \bbV(L) : \one_{\mathscr{T}(v)} =1  \}\right|.
\eex
Using H\"older's inequality with conjugate exponents 
\bex
p^{-1} = (1-\kappa)(1 + \kappa)^{-1}, \qquad q^{-1} = 2\kappa(1 + \kappa)^{-1},
\eex
 we get
\begin{align*}
&\left(\sum_{w\in \mathbb{V}(1)} \sum_{v \in \bbD_{L-1}(w) }
\left|R_{00}(z_2)\right|^\chi \left| 
 T_{0w}
R^{(0)}_{w v}(z_1)
\right|^s \one_{\mathscr{T}(v)} \right)^{1+\kappa}\\
&\le V^{2\kappa}
\left(\sum_{w\in \mathbb{V}(1)} \sum_{v \in \bbD_{L-1}(w) }
\left|R_{00}(z_2) \right|^{\chi(1+\kappa)/(1-\kappa) }\left|
 T_{0w}
R^{(0)}_{w v}(z_1)
\right|^{s(1+\kappa)/(1-\kappa)} \one_{\mathscr{T}(v)} \right)^{1-\kappa}.\end{align*}
Further using H\"older's inequality with conjugate exponents 
\bex
p^{-1} =  (1 - \kappa) , \qquad q^{-1} = \kappa,
\eex
we have
\begin{align}
&\E\left[ V^{2\kappa}
\left(\sum_{w\in \mathbb{V}(1)} \sum_{v \in \bbD_{L-1}(w) }
\left|R_{00}(z_2) \right|^{\chi(1+\kappa)/(1-\kappa) }\left|
 T_{0w}
R^{(0)}_{w v}(z_1)
\right|^{s(1+\kappa)/(1-\kappa)} \one_{\mathscr{T}(v)} \right)^{1-\kappa} \right]\notag \\
&\le \E \left[ V^2 \right]^{\kappa}\notag \\
&\quad \times 
\E\left[
\sum_{w\in \mathbb{V}(1)} \sum_{v \in \bbD_{L-1}(w) }
\left|R_{00}(z_2) \right|^{\chi(1+\kappa)/(1-\kappa) }\left|
 T_{0w}
R^{(0)}_{w v}(z_1)
\right|^{s(1+\kappa)/(1-\kappa)} \one_{\mathscr{T}(v)}  \right]^{1-\kappa}.\label{banana3}
\end{align}
We observe that $V$ is equal to the number of leaves at level $L$ of a Galton--Watson tree with parameter 
\bex
\lambda = \int_{\omega}^\infty \alpha x^{-\alpha - 1}\, dx = \omega^{-\alpha}.
\eex
By \Cref{treenumber}, we have for any $t \ge  0$  that
\begin{flalign*}
			  \mathbb{P} [ V \ge  t] \le 3 \exp(- t \cdot 2^{-L-1} \omega^{\alpha L}).
		\end{flalign*}
Then 
\begin{align}
\E [ V^2 ] 
\le 2 \cdot 3  \int_0^\infty t \exp(- t \cdot 2^{-L-1} \omega^{\alpha L})\, dt 
\le 24 \cdot 2^{2L} \omega^{-2 \alpha L}.\label{banana1}
\end{align}	
Further, by using the Schur complement formula \eqref{qvv}, we may write 
\begin{align}\label{cherry}
& \sum_{w\in \mathbb{V}(1)} \sum_{v \in \bbD_{L-1}(w) }
\left|R_{00}(z_2) \right|^{\chi(1+\kappa)/(1-\kappa) }\left|
 T_{0w}
R^{(0)}_{w v}(z_1)
\right|^{s(1+\kappa)/(1-\kappa)} \one_{\mathscr{T}(v)}\\
& =  \sum_{w\in \mathbb{V}(1)} \frac{|T_{0w}|^{s(1+\kappa)/(1-\kappa)} \one_{\omega \le |T_{0w}| \le  \Omega } S_w}{| z_2 + T_{0w}^2 R^{(0)}_{w w} + K_w |^{\chi(1+\kappa)/(1-\kappa)} }
\le \sum_{w\in \mathbb{V}(1)} \frac{|T_{0w}|^{s(1+\kappa)/(1-\kappa)} \one_{|T_{0w}| \le \Omega } S_w}{| z_2 + T_{0w}^2 R^{(0)}_{w w} + K_w |^{\chi(1+\kappa)/(1-\kappa)} }\notag,
\end{align}
where we set 
\bex
S_w = \sum_{v \in \bbD_{L-1}(w) } \left| R^{(0)}_{w v}(z_1)
\right|^{s(1+\kappa)/(1-\kappa)}\one_{\mathscr{T}(w, v)}, \qquad
K_w = \sum_{\substack{ u \in \bbV(1)\\ u \neq w }} 
T_{0u}^2 R^{(0)}_{uu}(z_2).
\eex
We observe that the sequence $\{ (K_w, R^{(0)}_{ww},S_w,T_{0w}) \}_{w \in \bbV(1)}$ satisfies \Cref{sqk}. 
Then by \eqref{cherry}, the second part of \Cref{expectationsum1} and the second part of \Cref{estimatemoment1},
\begin{align}\label{banana2}
&\E\left[
\sum_{w\in \mathbb{V}(1)} \sum_{v \in \bbD_{L-1}(w) }
\left|R_{00}(z_2) \right|^{\chi(1+\kappa)/(1-\kappa) }\left|
 T_{0w}
R^{(0)}_{w v}(z_1)
\right|^{s(1+\kappa)/(1-\kappa)} \one_{\mathscr{T}(v)}  \right] \\
&\le 
C \Omega^{s(1+\kappa)/(1-\kappa) - \alpha} \cdot 
\E\left[
\sum_{v \in \bbV(L-1) } \left| R_{0 v}(z_1)
\right|^{s(1+\kappa)/(1-\kappa)}\right]
\le 
\Omega C^{L+1}\notag
\end{align}
for some $C> 1$. 
Inserting \eqref{banana1} and \eqref{banana2} into \eqref{banana3} concludes the proof of \eqref{fracmom2}. 

The proof of \eqref{fracmom1} is similar to \eqref{fracmom2}, so we omit it. However, let us explain why the right side of \eqref{fracmom1} differs from \eqref{fracmom2}, and how this changes the proof. We note that because \eqref{fracmom1} involves only resolvents evaluated at $z_1$, and no resolvent evaluated at $z_2$, the first part of \Cref{expectationsum1} may be used in place of the second part of \Cref{expectationsum1} above (in \eqref{banana2}) when proving \eqref{fracmom1}. This improves the right side of \eqref{banana2} by a factor of $\Omega$, which yields the same improvement of \eqref{fracmom1} over \eqref{fracmom2}.
\end{proof}

\begin{lem}\label{holdershort}
Retain the notation and assumptions of the previous theorem. For any event $\mathscr A$, we have
\bex
\E\left[
\one_{\mathscr A}\sum_{w\in \mathbb{V}(1)} \sum_{v \in \bbD_{L-1}(w) }
\left|R_{00}(z_2) \right|^\chi
\left|
 T_{0w}
R^{(0)}_{w v}(z_1)
\right|^s \one_{\mathscr{T}(v)} 
\right]
\le \P(\mathscr A)^{\kappa /( 1 + \kappa)} 
 \cdot \Omega \cdot C^{1+L} \omega^{- 2 \kappa L}.
\eex
\end{lem}
\begin{proof}
By H\"older's inequality and \eqref{fracmom2},
\begin{align*}
&\E\left[
\one_{\mathscr A}\sum_{w\in \mathbb{V}(1)} \sum_{v \in \bbD_{L-1}(w) }
\left|R_{00}(z_2) \right|^\chi
\left|
 T_{0w}
R^{(0)}_{w v}(z_1)
\right|^s \one_{\mathscr{T}(v)} 
\right]\\
&\le 
\P(\mathscr A)^{\kappa /( 1 + \kappa)} 
\E\left[
\left(\sum_{w\in \mathbb{V}(1)} \sum_{v \in \bbD_{L-1}(w) }
\left|R_{00}(z_2) \right|^\chi
\left|
 T_{0w}
R^{(0)}_{w v}(z_1)
\right|^s \one_{\mathscr{T}(v)}\right)^{1+\kappa} 
\right]^{1/(1+\kappa)}\\
&\le 
\P(\mathscr A)^{\kappa / ( 1 + \kappa)} 
 \cdot \Omega \cdot C^{1+L} \omega^{- 2 \kappa L}.
\end{align*}
This completes the proof.
\end{proof}

We now show that vertices $v\in \bbV$ such that $\one_{\mathscr{T}(v)} = 0$ may be neglected in the sum defining $\Phi_L$, further improving \Cref{l:phicutofflower}.
\begin{lem}\label{l:phicutoff}
Fix $s \in (\alpha,1)$, $ \omega \in (0,1)$, $L \in \Zplus$, and $\Omega\ge 1$. There exist constants $C=C(s) > 1$ and $c(s) >0$ such that for any $z = E + \iu \eta \in \bbH$ such that $\eps \le |E| \le B$ and $\eta \in (0,1)$, we have
\bex
\left|
\E 
\left[
\sum_{v \in \bbV(L) } 
\left|  R_{0v}(z) \right|^s \one_{\mathscr{D}(v)}
\right]
-
\E 
\left[
\sum_{v \in \bbV(L) } 
\left|  R_{0v}(z) \right|^s \one_{\mathscr{T}(v)}
\right]
\right|
 \le L C^{L+1} \Omega^{- c} \omega^{- L }.
\eex
\end{lem} 
\begin{proof} 

For this proof only, we define $\tilde T_{vw}$ for $v,w \in \bbV$ by 
$\tilde T_{vw} = \one_{ |T_{vw}| > \Omega }  T_{vw}$. 
It suffices to show that for each of the $L$ terms 
\begin{align}\label{DeltaDef}
\Delta_k = 
\E 
\left[
\sum_{w \in \bbV(k) } 
\sum_{u \in \bbD(w) } 
\sum_{v \in \bbD(L - k - 1) } 
\left| R_{0w}(z)  \tilde T_{wu }  R^{(w)}_{uv} (z)  
\right|^s \one_{\mathscr{D}(v)}
\right]
\end{align}
defined for $k \in \unn{0}{L-1}$, we have 
\be\label{delt}
\Delta_k \le 
C^{L+1} \Omega^{-c} \omega^{- L}.\ee
This is because \eqref{delt} implies
\begin{align}\label{pepper}
&\left|
\E 
\left[
\sum_{v \in \bbV(L) } 
\left|  R_{0v}(z) \right|^s \one_{\mathscr{D}(v)}
\right]
-
\E 
\left[
\sum_{v \in \bbV(L) } 
\left|  R_{0v}(z) \right|^s \one_{\mathscr{T}(v)}
\right]
\right| \\ &\le 
\E 
\left[
\sum_{v \in \bbV(L) } 
\left|  R_{0v}(z)\one_{\mathscr{D}(v)}
 -  R_{0v}(z)\one_{\mathscr{T}(v)}
 \right|^s
\right] \le \sum_{k=0}^{L-1} \Delta_k \le L C^{L+1} \Omega^{- c} \omega^{- L }\notag
\end{align}
after using the elementary inequality $\big| |x|^s  - |y|^s \big| \le | x -y|^s$, which is the desired conclusion. Hence, we now turn to proving \eqref{delt}.

We first essentially reduce to the case when $k = 0$; this will closely follow the proof of \Cref{chilchil1}. Specifically, using \Cref{rproduct} to expand the $R_{0w}(z)$ term appearing in \eqref{DeltaDef} and the Schur complement formula \eqref{qvv}, we obtain
\begin{align}\label{pineapple}
& 
\E 
\left[
\sum_{w \in \bbV(k) } 
\sum_{u \in \bbD(w) } 
\sum_{v \in \bbD(L - k - 1) } 
\left| R_{0w}(z)  \tilde T_{wu }  R^{(w)}_{uv} (z)  
\right|^s \one_{\mathscr{D}(v)}
\right]\\
&= \E 
\left[
\sum_{r \in \bbV(1)}
\sum_{w \in \bbD_{k-1}(r) } 
\sum_{u \in \bbD(w) } 
\sum_{v \in \bbD(L - k - 1) } 
\left| R_{00}(z) T_{0r} R^{(0)}_{rw}  \tilde T_{wu }  R^{(w)}_{uv} (z)  
\right|^s \one_{\mathscr{D}(v)}
\right]\label{pineapple3} \\
&= 
\E 
\left[
\sum_{r \in \bbV(1)}
\sum_{w \in \bbD_{k-1}(r) } 
\sum_{u \in \bbD(w) } 
\sum_{v \in \bbD(L - k - 1) } 
 \frac{|T_{0r}|^s S_r }{| z + T^2_{0r} R^{(0)}_{rr} + K_r|^s}  
\right],\label{pineapple4}
\end{align}
where in the last equality we set
\bex
S_r = \sum_{w \in \bbD_{k-1}(r) } 
\sum_{u \in \bbD(w) } 
\sum_{v \in \bbD(L - k - 1) }  \left| R^{(0)}_{rw}  \tilde T_{wu }  R^{(w)}_{uv} (z)  
\right|^s \one_{\mathscr{D}(r,v)},\qquad
K_r = \sum_{\substack{u \in \bbV(1)\\ u \neq r}}
T_{0u}^2 R^{(0)}_{uu}.
\eex
We observe that the sequence $\{ (K_r, R^{(0)}_{rr},S_r,T_{0r}) \}_{r \in \bbV(1)}$ satisfies \Cref{sqk}. This permits use of the first part of \Cref{expectationsum1} in \eqref{pineapple4}, which yields 
\begin{align}\notag
&\E 
\left[
\sum_{r \in \bbV(1)}
\sum_{w \in \bbD_{k-1}(r) } 
\sum_{u \in \bbD(w) } 
\sum_{v \in \bbD(L - k - 1) } 
 \frac{|T_{0r}|^s S_r }{| z + T^2_{0r} R^{(0)}_{rr} + K_r|^s}  
\right]\\
&\le \label{rterm}
C \cdot \E\left[ 
\sum_{w \in \bbD_{k-1}(r) } 
\sum_{u \in \bbD(w) } 
\sum_{v \in \bbD(L - k - 1) } \big|R^{(0)}_{rr}(z)\big|^{(\alpha -s)/2} \left| R^{(0)}_{rw}  \tilde T_{wu }  R^{(w)}_{uv} (z)  
\right|^s \one_{\mathscr{D}(r,v)} \right]\\
& \le  C \cdot \E 
\left[
\sum_{w \in \bbV(k-1) } 
\sum_{u \in \bbD(w) } 
\sum_{v \in \bbD(L - k - 1) } 
 \big|R_{00}(z)\big|^{(\alpha -s)/2} \left| R_{0w}(z)  \tilde T_{wu }  R^{(w)}_{uv} (z)  
\right|^s \one_{\mathscr{D}(v)}
\right],\label{pineapple2}
\end{align}
where in \eqref{rterm}, $r \in \bbV(1)$ is an arbitrarily selected vertex (and the choice of $r$ does change the expectation as all $S_r$ are identically distributed).
We now observe that the term \eqref{pineapple2} is essentially the same as \eqref{pineapple}, except with the sum $w\in \bbV(k)$ replaced by $w \in \bbV(k-1)$, and an additional factor of $\big|R_{00}(z)\big|^{(\alpha -s)/2}$.
Therefore, considering the sum \eqref{pineapple2}, we can repeat the computations from \eqref{pineapple} through \eqref{pineapple2} (with the change that \eqref{pineapple3} has a factor of $\big|R_{00}(z)\big|^{(\alpha + s)/2}$ in place of the $\big|R_{00}(z)\big|^{s}$ there, which must be carried through the calculation). Iterating these computations $k-1$ times and putting the resulting bound on \eqref{pineapple} in \eqref{DeltaDef} gives
\begin{align}
\Delta_k 
&\le 
C^k \E 
\left[
\sum_{w \in \bbV(1)}
\sum_{v \in \bbD_{L -k -1}(w) } 
\left| R_{00}(z)\right|^{(\alpha-s)/2} \left| R_{00}(z) \tilde T_{0w}  R^{(0)}_{wv} (z)
\right|^s  \one_{\mathscr{D}(v)}
\right]\notag \\
&=
C^k  
\E 
\left[
\sum_{w \in \bbV(1)}
\sum_{v \in \bbD_{L -k -1}(w) } 
\left| R_{00}(z)\right|^{(\alpha+s)/2} \left| \tilde T_{0w}  R^{(0)}_{wv} (z)
\right|^s  \one_{\mathscr{D}(v)}
\right]\label{thesum}.
\end{align}
Using the Schur complement formula \eqref{qvv} to expand $R_{00}(z)$, we may write the previous line as 
\be\label{campbellsetup}
\Delta_k \le C^k \sum_{w \in \bbV(1)} \frac{|T_{0w}|^s \one_{|T_{j}| > \Omega } S_w}{| z + T_{0w}^2 R^{(0)}_{w w} + K_w |^{(\alpha + s)/2} },
\ee
where 
\bex
S_w = \sum_{v\in \bbD_{L-k-1}(w)} \left| R^{(0)}_{w v}\right|^s \one_{\mathscr{D}(w,v)},
\qquad K_w = \sum_{\substack{ u \in \bbV(1)\\ u \neq w }} 
T_{0u}^2 R^{(0)}_{uu}.
\eex
Then \Cref{fidentityxi} shows that
\begin{align}\label{jalapeno1}
& \sum_{w \in \bbV(1)} \frac{|T_{0w}|^s \one_{|T_{j}| > \Omega } S_w}{| z + T_{0w}^2 R^{(0)}_{w w} + K_w |^{(\alpha + s)/2} } = \int_\Omega^\infty 
\E \left[ 
\frac{t^s S }{\big  | z + t^2 R   + K  \big|^{(\alpha+s)/2}}
\right]
\cdot \alpha  t^{-\alpha -1 } \,dt \\
&\le
\left(  \int^\infty_\Omega \alpha t^{-\alpha -1 } \,dt \right)^{\kappa/(1+\kappa)}
\left(
 \int_0^\infty 
\E \left[ 
\frac{t^{s(1+\kappa)} S^{(1+\kappa)} }{\big ( | z + t^2 R|  +1  \big)^{(\alpha+s)(1+\kappa)/2}}
\right]
\cdot \alpha t^{-\alpha -1 } \,dt
\right)^{(1+\kappa)^{-1}}
,\notag
\end{align}
where we used H\"older's inequality in the last line, and set $\kappa = (1/2)(1 - s)$ so that $s (1 + \kappa ) < 1$. Here the random vector $(K,S,R)$ is defined to have the same joint law as any $(K_w, S_w, R^{(0)}_{ww})$.

We have
\be\label{jalapeno2}
\alpha \int^\infty_\Omega t^{-\alpha -1 } \,dt
= \Omega^{-\alpha}.
\ee
Also, by the first part of \Cref{expectationsum1},
\begin{align}
&\int_0^\infty 
\E \left[ 
\frac{t^{s(1+\kappa)} S^{1+\kappa} }{\big ( | z + t^2 R|  +1  \big)^{(\alpha+s)(1+\kappa)/2}}
\right]
\cdot \alpha t^{-\alpha -1 } \,dt\notag \\
&\le C \cdot \E \left[ | R |^{(\alpha-s)/2} S^{1+\kappa} \right]\notag \\ &= C \cdot
\E 
\left[
\left| R_{00}(z)\right|^{(\alpha - s)/2}
\left(
\sum_{v\in \bbV(L-k-1)} \left| R_{0 v}(z)\right|^s \one_{\mathscr{D}(v)}
\right)^{1+\kappa}
\right]\notag\\
&= C \cdot
\E 
\left[
\left(
\sum_{w\in \bbV(1)}
\sum_{v\in \bbD_{L-k-2}(w)}
\left| R_{00}(z)\right|^{(\alpha - s)(1+\kappa)^{-1}/2}
\left| R_{00}(z) T_{0w}  R^{(0)}_{wv} (z) 
\right|^s\right)^{1+\kappa}
\one_{\mathscr{D}(v)}
\right]\notag\\
&\le C^{L-k + 1} \omega^{-2 \kappa L}.\label{jalapeno3}
\end{align}
where the last inequality follows from \eqref{fracmom1}.

Set $c(s) = \kappa /(1 + \kappa)$. Using \eqref{campbellsetup}, \eqref{jalapeno1}, \eqref{jalapeno2}, and \eqref{jalapeno3}, we obtain $\Delta _k \le  C^{L+1} \Omega^{- c} \omega^{- 2\kappa L}$, proving \eqref{delt}. This completes the proof after recalling \eqref{pepper} and using $\kappa < 1/2$ to show that $\omega^{-2 \kappa L} < \omega^{-L}$.
\end{proof}

 \subsection{Continuity Estimates}
 
 \label{Estimate1}
 
 We now prove a continuity estimate for the truncated version of $\Phi_L$ considered in \Cref{l:phicutoff}, involving only the vertices $v \in \bbV$ such that $\one_{\mathscr{T}(v)} \neq 0$. It will be our primary technical tool for the remainder of this section.
 \label{s:continuityestimates}
\begin{lem}\label{l:cutoffclose}
Fix $s \in (\alpha,1)$, $ \delta,  \omega \in (0,1)$, and $ \Omega \ge 1$. There exist constants $C(\delta, s) > 1$ and $c(s) > 0$ 
such that the following holds. For  $i=1,2$, fix $z_i = E_i + \iu \eta_i \in \bbH$ such that $\eps \le |E_i| \le B$ and $\eta_i \in (0,1)$. Suppose that $\phi(s;z_2) < - \delta$. Then
\bex
\E
\left[
\sum_{v \in \bbV(L)} \left|   
 R_{0v}(z_1) -  R_{0v}(z_2) 
\right|^s \one_{\mathscr{T}(v)}
\right] 
\le L \Omega^2 C^{L+1}|z_1 - z_2|^s (\eta_1 \eta_2)^{-s/2}
\cdot \eta_2^{cs } \omega^{-2 L }.
\eex
\end{lem}
\begin{proof}
For $k \in \unn{0}{L-2}$, we define
\begin{align}\label{PsiDef}
&\Psi_{L, k} (s; \omega, \Omega; z_1, z_2)\notag \\
&= 
\E\left[ 
\sum_{w \in \bbV(k+1)}
\sum_{v \in \bbD_{L-k-1}(w)}
\left|
 R_{0 w_-}(z_2)  T_{w_- w}  R^{(w_-)}_{w v }(z_1)
-  R_{0 w}(z_2)  T_{w w_+}  R^{(w)}_{w_+ v } (z_1)
\right|^s\one_{\mathscr{T}(v)}
\right].
\end{align}
We also set 
\bex
\Psi_{L, -1} (s; \omega, \Omega; z_1, z_2)=
\E\left[ 
\sum_{w \in \bbV(1)}
\sum_{w \in \bbD_{L-1} (w)}
\left|
 R_{0 v}(z_1)
- R_{0  0 }(z_2)  T_{0 w}  R^{(0)}_{w v}(z_1 )
\right|^s\one_{\mathscr{T}(v)}
\right]
\eex
and
\bex
\Psi_{L, L-1} (s; \omega, \Omega; z_1, z_2)=
\E\left[ 
\sum_{w \in \bbV(L-1)}
\sum_{v \in \bbD(w)}
\left|
 R_{0  w }(z_2)  T_{w v}  R^{(w)}_{vv}(z_1)
- 
 R_{0 v}(z_2)
\right|^s\one_{\mathscr{T}(v)}
\right].
\eex
Then using \Cref{rproduct}, we  have
\be\label{thephisum}
\E
\left[
\sum_{v \in \bbV(L)} \left|   
 R_{0v}(z_1) -  R_{0v}(z_2) 
\right|^s \one_{\mathscr{T}(v)}
\right] \le 
\sum_{j = -1}^{L-1} \Psi_{L, j} (s; \omega, \Omega ; z_1, z_2),
\ee
so to complete the proof, it suffices to bound the sum on the right side of \eqref{thephisum}.

Considering the terms in the sum defining \eqref{PsiDef}, we have using \Cref{rproduct} that 
\begin{align*}
&  R_{0 w_-}(z_2)  T_{w_- w}  R^{(w_-)}_{w v }(z_1)
-  R_{0 w}(z_2)  T_{w w_+}  R^{(w)}_{w_+ v } (z_1)\\
& =  R_{0 w_-}(z_2)  T_{w_- w}  R^{(w_-)}_{w w}(z_1) 
 T_{w w_+}
 R^{(w)}_{w_+ v} (z_1)\\
&\quad -  R_{0 w_-}(z_2)  T_{w_- w}  R^{(w_-)}_{w w}(z_2) 
 T_{w w_+}
 R^{(w)}_{w_+ v} (z_1)\\
&=  R_{0 w_-}(z_2)  T_{w_- w} \big(R^{(w_-)}_{w w}(z_1) -   R^{(w_-)}_{w w}(z_2) \big)
T_{w w_+}
 R^{(w)}_{w_+ v} (z_1).
\end{align*}
Then using the previous line in the definition of $\Psi_{L, k}$ in \eqref{PsiDef} gives
\begin{align*}
&\Psi_{L, k} (s; \omega, \Omega; z_1, z_2)\notag\\
&= 
\E\left[
\sum_{w \in \bbV(k+1)}
\sum_{v \in \bbD_{L-k-1}(w)}
\left|
R_{0 w_-}(z_2)  T_{w_- w} \big(R^{(w_-)}_{w w}(z_1) -   R^{(w_-)}_{w w}(z_2) \big)
T_{w w_+}
 R^{(w)}_{w_+ v} (z_1)
\right|^s \one_{\mathscr{T}(v)}
\right].
\end{align*}
Using the product expansion from \Cref{rproduct}, the Schur complement formula \eqref{qvv} together with the first part of \Cref{expectationsum1} a total of $k$ times, and the second part of \Cref{expectationsum1} once, we obtain that 
(reasoning as in the calculations leading to \eqref{thesum})
\begin{align}
\Psi_{L, k}& (s; \omega, \Omega; z_1, z_2)\notag \\
&\le 
\Omega C^{k+1} \E\left[
\sum_{w \in \bbV(k+1)}
\sum_{v \in \bbD_{L-k-1}(w)}
 \left|\big(R^{(w_-)}_{ww}(z_1) -   R^{(w_-)}_{ww}(z_2) \big)
T_{ww_+}
 R^{(0)}_{w_+ v} (z_1)
\right|^s \one_{\mathscr{T}(v)}
\right]\notag\\
&\le 
\Omega C^{k+1} \E\left[
\sum_{w \in \bbV(1)}
\sum_{v \in \bbD_{L-k-2}(w)}
 \left|\big(R_{00}(z_1) -   R_{00}(z_2) \big)
T_{0w}
 R^{(0)}_{w v} (z_1)
\right|^s \one_{\mathscr{T}(v)}
\right]\label{phiprev}
\end{align}
for some $C > 1$. 
Using the resolvent identity $\bm{A}^{-1} - \bm{B}^{-1} = \bm{A}^{-1} (\bm{B} - \bm{A}) \bm{B}^{-1}$, we write 
\bex
R_{00}(z_1) -  R_{00}(z_2)
= (z_2 - z_1) \sum_{v \in \bbV} R_{0v}(z_1) R_{v0}(z_2).
\eex
 H\"older's inequality and the Ward identity \eqref{sumrvweta} together give the bound
\begin{align}
\left| R_{00}(z_1) -  R_{00}(z_2) \right|
&\le |z_1 - z_2| 
\left(\sum_{v \in \bbV} |R_{0v}(z_1)|^2 \right)^{1/2}
\left(\sum_{v \in \bbV} |R_{0v}(z_2)|^2 \right)^{1/2}\notag
\\
& = |z_1 - z_2| (\eta_1\eta_2) ^{-1/2} 
\big( \Im R_{00}(z_1) \big)^{1/2}
\big( \Im R_{00}(z_2) \big)^{1/2}.\label{holder2}
\end{align}
Inserting \eqref{holder2} into \eqref{phiprev} gives
\begin{align}\notag
&\Psi_{L, k} (s; \omega, \Omega ; z_1, z_2)\le 
\Omega C^{k+1}|z_1 - z_2|^s (\eta_1 \eta_2)^{-s/2}\\
&\times  \E\left[
\sum_{w \in \bbV(1)}
\sum_{v \in \bbD_{L-k-2}(w)}
\big( \Im R_{00}(z_1) \big)^{s/2}
\big( \Im R_{00}(z_2) \big)^{s/2}
\left|
 T_{0w}
R^{(0)}_{w v}(z_1)
\right|^s \one_{\mathscr{T}(v)} \label{intPhi}
\right].
\end{align}

Let $\kappa > 0$ be a parameter, to be fixed later. By Markov's inequality, the assumption $\phi(s;z_2) < - \delta$, and \Cref{l:qsexpectation}, we have  
\be\label{simplemarkov}
\P\left( \big( \Im R_{00}(z_2) \big)^{s/2} > \eta_2^\kappa  \right) \le C \eta_2^{s/2 - \kappa}.
\ee
Define the event 
\bex
\mathscr A = \left\{ \big( \Im R_{00}(z_2) \big)^{s/2} < \eta_2^\kappa \right\}.
\eex
Then by the definition of $\mathscr A$, we have 
\begin{align}\label{ongoodevent}
&\E\left[\one_{\mathscr A}
\sum_{w \in \bbV(1)}
\sum_{v \in \bbD_{L-k-2}(w)}
\big( \Im R_{00}(z_1) \big)^{s/2}
\big( \Im R_{00}(z_2) \big)^{s/2}
\left|
 T_{0w}
R^{(0)}_{w v}(z_1)
\right|^s
\right]
\\
&\le \eta^{\kappa}_2\cdot \E\left[
\sum_{w \in \bbV(1)}
\sum_{v \in \bbD_{L-k-2}(w)}
\big| R_{00}(z_1) \big|^{s/2}
\left|
 T_{0w}
R^{(0)}_{w v}(z_1)
\right|^s
\right]
\le 
\eta^{\kappa}_2
C^{L-k},\notag
\end{align}
where the last inequality follows from iterating \Cref{chilchil1} with $\chi$ there equal to $s/2$, and using \Cref{expectationqchi} (as in the proof of \Cref{estimatemoment1}).

 We next note that the elementary inequality $ab \le a^2 + b^2$ gives 
\bex
\big( \Im R_{00}(z_1) \big)^{s/2} \big( \Im R_{00}(z_2) \big)^{s/2}  \le\big( \Im R_{00}(z_1) \big)^{s}
 +
\big( \Im R_{00}(z_2) \big)^{s} .
\eex
Fix $\kappa>0$ so that \eqref{chisconditions} holds, with $\chi$ there equal to $s/2$. Then the previous line and \Cref{holdershort} give 
\begin{align}
& \E\left[
\one_{\mathscr A^c}
\sum_{w \in \bbV(1)}
\sum_{v \in \bbD_{L-k-2}(w)}
\big( \Im R_{00}(z_1) \big)^{s/2}
\big( \Im R_{00}(z_2) \big)^{s/2}
\left|
 T_{0w}
R^{(0)}_{w v}(z_1)
\right|^s \one_{\mathscr{T}(v)} 
\right]\notag \\
&\le  \E\left[
\one_{\mathscr A^c} \sum_{w \in \bbV(1)}
\sum_{v \in \bbD_{L-k-2}(w)}
\big( \Im R_{00}(z_1) \big)^{s}
\left|
 T_{0w}
R^{(0)}_{w v}(z_1)
\right|^s \one_{\mathscr{T}(v)} 
\right]\notag \\
&+
 \E\left[
\one_{\mathscr A^c}\sum_{w \in \bbV(1)}
\sum_{v \in \bbD_{L-k-2}(w)}
\big( \Im R_{00}(z_2) \big)^{s}
\left|
 T_{0w}
R^{(0)}_{w v}(z_1)
\right|^s \one_{\mathscr{T}(v)} 
\right]\notag \\
&\le \P(\mathscr A^c)^{c_1}  \cdot \Omega \cdot C^{L-k}  \omega^{-2  L }
= \eta_2^{c_1s }
\cdot \Omega \cdot C^{L-k}  \omega^{-2  L }
\label{onbadevent}
\end{align}
for some $c_1(s) > 0$, where we used \eqref{simplemarkov} in the last line. The claimed bound follows after inserting \eqref{ongoodevent} and \eqref{onbadevent} into \eqref{intPhi}, and using \eqref{thephisum}.
\end{proof}

The next lemma states a continuity estimate for $\phi(s; E + \iu \eta)$ as $\eta$ is a fixed and $E$ varies.  
\begin{lem}\label{l:goingsideways}
Fix $s \in (\alpha,1)$ and $ \delta, \eta \in (0,1)$. Then for every $\mathfrak b >0$, there exists $\mathfrak a (\delta, \eta, s) > 0$ such that following holds. For every $E_1, E_2 \in \bbR$ such that $|E_1|, |E_2| \in [\eps , B]$, $|E_1 - E_2| < \mathfrak a$, and $\phi(s;E_2 + \iu \eta) < - \delta$,
we have
\bex
\big|
\phi(s;E_1 + \iu \eta ) - \phi(s;E_2 + \iu \eta)
\big| < \mathfrak b.
\eex
\end{lem}
\begin{proof}
Set $z_1 = E_1 + \iu \eta$ and $z_2 = E_2 + \iu \eta$. We may assume that $|E_1 - E_2| < 1/10$. By the first part of \Cref{limitr0j}, there exists $C_1 > 1$ such that for any $L \in \Zplus$,
\be\label{combineme2}
\left|
\big( 
\phi_L(s;z_1) - \phi_L(s;z_2) 
 \big)
 -
 \big( 
\phi(s;z_1) - \phi(s;z_2) 
 \big)
\right|
\le \frac{C_1}{L}.
\ee
We set $L = \lceil 2 C_1 \mathfrak b^{-1} \rceil$.\footnote{In what follows, we drop the ceiling functions when doing computations with $L$, since it makes no essential difference.} Then it remains to bound the difference $\big| 
\phi_L(s;z_1) - \phi_L(s;z_2) 
 \big|$.

Recalling the definition \eqref{sumsz1}, we have
\begin{align}\label{mango1}
\phi_L(s;z_1)
-
\phi_L(s;z_2) 
=
\log \left(
\frac{\Phi_L(s; z_1)}{\Phi_L(s; z_2)}
\right).
\end{align}
We write
\begin{align}\notag
\left|\log \left(
\frac{\Phi_L(s; z_1)}{\Phi_L(s; z_2)}
\right)\right|
&=
\left|\log \left(
\frac{ \Phi_L(s; z_2) + \big(\Phi_L(s; z_1) - \Phi_L(s; z_2)  \big)}{\Phi_L(s; z_2)}
\right)\right|\\
&\le 
 \frac{ 2 \left|\E
\left[
\sum_{v \in \bbV(L)} \left|   
 R_{0v}(z_1) 
\right|^s
\right]
-
\E
\left[
\sum_{v \in \bbV(L)} \left|   
 R_{0v}(z_2) 
\right|^s
\right]\right|}{\Phi_L(s; z_2)},\label{mango2}
\end{align}
where the inequality is valid under the assumption that 
\be
\left|\E
\left[
\sum_{v \in \bbV(L)} \left|   
 R_{0v}(z_1) 
\right|^s
\right]
-
\E
\left[
\sum_{v \in \bbV(L)} \left|   
 R_{0v}(z_2) 
\right|^s
\right]\right|\le \frac{1}{2}\Phi_L(s; z_2).\label{mango3}
\ee
From the second part of \Cref{estimatemoment1}, there exists $c_1 \in (0,1)$ such that
\be\notag
\Phi_L(s; z_2) \ge c_1^{L+1}.
\ee

Let $\theta \in(0,1/2)$ and $\Omega > 1$ be real parameters, and let $\omega = c_2 \theta^{C_2}$, where $C_2>1$ and $c_2>0$ are the constants given by \Cref{l:phicutofflower}. 
Using \Cref{l:phicutofflower}, \Cref{l:phicutoff}, \Cref{l:cutoffclose}, and the  elementary inequality $\big| |x|^s  - |y|^s \big| \le | x -y|^s$, there exist constants $C_3>1$ and $c_3>0$ such that
\begin{align}
&\left|\E
\left[
\sum_{v \in \bbV(L)} \left|   
 R_{0v}(z_1) 
\right|^s
\right]
-
\E
\left[
\sum_{v \in \bbV(L)} \left|   
 R_{0v}(z_2) 
\right|^s
\right]\right|
\\
&\le \theta L C_3^{1+L} + L C_3^{1+L} \Omega^{-c_3} \omega^{-C_3L} + L C_3^{L+1} \Omega^2 |E_1 - E_2 |^s \eta^{s(c_3-1)}  \omega^{-2L}.
\label{crudebound}
\end{align}
 We set
\be\label{blueberry1}
\theta = \frac{1}{16} \cdot L^{-1} C_3^{-1 - L} c_1^{L+1} \mathfrak b
\ee
and fix $\Omega = \Omega(\theta, \mathfrak b)>1$ large enough so that 
\be\label{blueberry2}
L C_3^{1+L} \Omega^{-c_3} \omega^{-C_3L}
\le \frac{1}{16} \cdot c_1^{L+1} \mathfrak b.
\ee
Finally, we fix $\mathfrak a(\theta,\Omega, \eta, \mathfrak b) > 0$ such that
\be\label{blueberry3}
L C_3^{L+1} \Omega^2 |E_1 - E_2 |^s \eta^{s(c_3-1)}  \omega^{-2L} 
\le \frac{1}{16} \cdot c_1^{L+1} \mathfrak b
\ee
when $|E_1 - E_2 | < \mathfrak{a}$.
With the choices \eqref{blueberry1}, \eqref{blueberry2}, and \eqref{blueberry3}, equation \eqref{crudebound} gives 
\begin{align}\notag
\left|
\E
\left[
\sum_{v \in \bbV(L)} \left|   
 R_{0v}(z_1) 
\right|^s
\right]
-
\E
\left[
\sum_{v \in \bbV(L)} \left|   
 R_{0v}(z_2) 
\right|^s
\right]
\right|
\le \frac{3}{16} \cdot c_1^{L+1} \mathfrak{b}.
\end{align}
Combining the previous line with \eqref{mango1}, \eqref{mango2}, and \eqref{mango3},
we get
\be\notag
\big| \phi_L(s;z_1)
-
\phi_L(s;z_2) \big| < \frac{\mathfrak{b}}{2}.
\ee
Combining the previous line with \eqref{combineme2} and the choice of $L$ below \eqref{combineme2} completes the proof.
\end{proof}
The next lemma provides a continuity estimate for $\phi(s, E + \iu \eta)$ when $E$ is fixed and $\eta$ varies. Given the initial estimate $\phi(s; E + \iu \eta_0) < - \delta$ for some $\delta > 0$ and sufficiently small $\eta_0$, the lemma states that $\limsup_{\eta \rightarrow 0}
\phi(s; E + \iu \eta)$ is negative. In other words, the negativity of $\phi(s, E + \iu \eta_0)$ propagates toward the real axis. 
\begin{lem}\label{l:goingdown}
Fix $s \in (\alpha,1)$ and $ \delta \in (0,1)$. Then for every $\mathfrak b \in (0, \delta)$, there exists $\mathfrak a (\delta, \eps,  B,s) \in (0,1)$ such that following holds. For every $E\in \bbR$ such that $|E| \in [\eps, B]$, and every $\eta_0 \in (0, \mathfrak a)$ such that ${\phi(s;E + \iu \eta_0)} < - \delta$,
we have
\bex
\limsup_{\eta \rightarrow 0}
\phi(s; E + \iu \eta) < - \mathfrak b.
\eex
\end{lem}
\begin{proof}
From the second part of \Cref{estimatemoment1}, there exists $c_1 \in (0,1)$ such that
\be\label{upcondition}
\Phi_L(s; E + \iu \eta) \ge c_1^{L+1}.
\ee
for all $E\in \bbR$ such that $|E| \in [\eps, B]$ and $\eta \in (0,1)$. Let $C_2>1$ and $c_2 > 0$ be the two constants given by \Cref{l:cutoffclose}. 
We begin by considering two points $z_1 = E + \iu \eta_1$ and $z_2 = E + \iu \eta_2$ with $\eta_1, \eta_2 \in (0,1)$ such that $\eta_1 \in [\eta_2^{(1 - c_2s/7)^{-1}}, \eta_2]$, so that $\eta_1 \le \eta_2$. We also suppose that $\phi(s;E + \iu \eta_2) < - \delta/2$. 
We make the choice of parameters
\be\label{choices}
\omega = \exp\left( -  ( -\log \eta_1)^{1/4}  \right),
\quad
\Omega = \exp\left( (- \log \eta_1)^{1/2} \right),
\quad
L = ( - \log \eta_1)^{1/5}.
\ee
We also set $\theta = (c_2^{-1} \omega)^{C_2^{-1}}$.

Using \Cref{l:phicutofflower}, \Cref{l:phicutoff}, \Cref{l:cutoffclose}, and the  elementary inequality $\big| |x|^s  - |y|^s \big| \le | x -y|^s$, we find
\begin{align}\notag
&\left|\E
\left[
\sum_{v \in \bbV(L)} \left|   
 R_{0v}(z_1) 
\right|^s
\right]
-
\E
\left[
\sum_{v \in \bbV(L)} \left|   
 R_{0v}(z_2) 
\right|^s
\right]\right|
\\
&\le \theta L C_3^{1+L} + L C_3^{1+L} \Omega^{-c_3} \omega^{-C_3L} + L C_3^{L+1} \Omega^2 |\eta_1 - \eta_2 |^s (\eta_1 \eta_2)^{-s/2}
\cdot \eta_2^{c_2s } \omega^{-2 L },
\label{crudebound12}
\end{align}
where  $C_3 >1$ and $c_3 >0$ are constants.
Then by \eqref{choices} and \eqref{crudebound2}, there exists $\eta_0(s, c_1, c_2, C_2, c_3, C_3)>0$ such that, if $\eta_2 \in (0 , \eta_0)$, then
\be\label{crudebound2}
\left|\E
\left[
\sum_{v \in \bbV(L)} \left|   
 R_{0v}(z_1) 
\right|^s
\right]
-
\E
\left[
\sum_{v \in \bbV(L)} \left|   
 R_{0v}(z_2) 
\right|^s
\right]\right|
\le \frac{1}{4} \cdot c_1^{L+1}. 
\ee
Using \eqref{upcondition}, and \eqref{combineme2}, \eqref{mango1}, and \eqref{mango2}, we find that there exists $C_4 > 1$ such that 
\be \label{phigoingdown}
\big| \phi(s;z_1) - \phi(s;z_2) \big|
\le \frac{C_4}{L} + \frac{1}{4} \cdot c_1^{L+1}
< C_4 (  - \log \eta_2)^{-1/5},
\ee
where we increased the value of $C_4$ in the second inequality.

Now consider a decreasing sequence $\{\nu_k\}_{k=1}^\infty$ of reals such that $\nu_1 < \eta_0$, $\phi(s; E + \iu \nu_1) < -\delta$, and $\nu_{k+1} = \nu_k^{(1 - c_2s/7)^{-1}}$ for all $k \in \Zplus$. Set $w_k = E + \iu \nu_k$. With $\mathfrak{b}$ chosen as in the statement of this lemma, we will show by induction that there exists $\mathfrak a(\mathfrak b)>0$ such that $\phi(s; w_k) < - \mathfrak b$ for all $k\in \Zplus$ if $\nu_1 < \mathfrak a$.  We may suppose that $\mathfrak b > \delta /2$. 
For brevity, we set $c_5 = 1 - c_2s/7$.

We now claim that $\phi(s; w_j) < - \mathfrak b$ for all $j \in \Zplus$; we will prove this claim by induction. 
For the induction hypothesis at step $n \in \Zplus$, suppose that $\phi(s; w_j) < - \mathfrak b$ holds for all $j \le n$.
We will prove the same estimate holds for $j = n+1$. 

The bound \eqref{phigoingdown} gives, for all $k\le n$, that
\begin{align*}
\left|
\phi(s;w_k) - \phi(s;w_{k+1}) 
\right|  &\le C_4 ( - \log \nu_k)^{-1/5}\\
 &\le C_4 \left( - \log\left( \nu_1^{c_5^{-k}} \right)\right)^{-1/5}
\le C_4 c_5^{k/5} ( - \log \nu_1)^{-1/5}.
\end{align*}
This previous line implies that
\begin{align}\notag
\left|
\phi(s;w_1) - \phi(s;w_{n+1})
\right|
&\le 
\sum_{k=1}^{n} 
\left|
\phi(s;w_k) - \phi(s;w_{k+1})
\right| \\
&\le C_4 ( - \log \nu_1)^{-1} \sum_{k=1}^\infty c_5^k\le C_6 ( - \log \nu_1)^{-1/5}\label{returnto}
\end{align}
for some $C_6(C_4, c_5) >1$. We choose $\mathfrak a$ so that 
\bex
C_6 ( - \log \mathfrak a)^{-1} < \delta - \mathfrak b.
\eex
Then \eqref{returnto} and the assumption that $\nu_1 < \mathfrak{a}$ implies that 
\be\label{returnto2}
\phi(s;w_{n+1})
\le \phi(s;w_1) + C_6 ( - \log \nu_1)^{-1/5} <  - \delta + (\delta - \mathfrak b) < -\mathfrak b.
\ee
This completes the induction step, and shows that $\phi(s; w_j) < - \mathfrak b$ for all $j \in \Zplus$.

We now claim that for any that any $\tilde w = E + \iu \tilde \nu$ with $\tilde \nu \in (0, \mathfrak a)$, we have $
\phi(s; \tilde w) < - \mathfrak b$. To see this, observe that there is a unique index $k \in \Zplus$ such that $\nu_k > \tilde \nu \ge \nu_{k+1}$. Consider the sequence $\{\tilde w_j\}_{j=1}^{k+1}$ defined by $\tilde w_j = w_j$ for $j \neq k+1$, and $\tilde w_{k+1} = E+ \iu \tilde \nu$. Then the same induction argument that gave \eqref{returnto2} also gives 
 \be\label{returnto3}
\phi(s; \tilde w_{k+1})
 < -\mathfrak b.
\ee
This completes the proof.
\end{proof}
The following lemma is in some sense the reverse of the previous one. It shows that if 
\bex\limsup_{\eta \rightarrow 0}
{\phi(s; E + \iu \eta)} < 0,\eex
 then $\phi(s, E + \iu \eta_0)<0$ if $\eta_0$ is chosen sufficiently small, in a way that is independent of the energy $E$ (assuming $|E| \in [\eps, B]$). In other words, the negativity of $\limsup_{\eta \rightarrow 0}\phi(s; E + \iu \eta) $ propagates away from the real axis in a uniform way.
\begin{lem}\label{l:goingup}
Fix $s \in (\alpha,1)$ and $ \delta \in (0,1)$. Then for every $\mathfrak b \in (0, \delta)$, there exists $\mathfrak a (\delta, s) \in (0,1)$ such that following holds. For every $E\in \bbR$ such that $|E| \in [\eps, B]$ and 
\be\label{limsuphypo}
\limsup_{\eta \rightarrow 0}
\phi(s; E + \iu \eta) < - \delta,
\ee
we have
\bex
\sup_{\eta \in (0, \mathfrak{a}]} \varphi (s; E + \mathrm{i} \eta) < -\mathfrak{b}.
\eex
\end{lem}
\begin{proof}
From the second part of \Cref{estimatemoment1}, there exists $c_1 \in (0,1)$ such that
\bex
\Phi_L(s; E + \iu \eta) \ge c_1^{L+1}.
\eex
for all $E\in \bbR$ such that $|E| \in [\eps, B]$ and $\eta \in (0,1)$. Let $C_2>1$ and $c_2 > 0$ be the two constants given by \Cref{l:cutoffclose}. 
We begin by considering two points $z_1 = E + \iu \eta_1$ and $z_2 = E + \iu \eta_2$ with $\eta_1, \eta_2 \in (0,1)$ such that $\eta_1 \in [\eta_2, \eta_2^{(1 - c_2s/7)}]$, so that $\eta_1 \le \eta_2$. We also suppose that $\phi(s;E + \iu \eta_2) < - \delta/2$. 
Then repeating the calculations in \eqref{crudebound}, \eqref{crudebound2}, and \eqref{phigoingdown} (with the same parameter choices \eqref{choices}) shows that there exists $C_3 > 0$ such that 
\be \label{phigoingup}
\big| \phi(s;z_1) - \phi(s;z_2) \big|
< C_3 (  - \log \eta_2)^{-1/5}
\ee
for $\eta < C_3^{-1}$.

 We define a sequence $\{\tilde \nu_k \}_{k=1}^\infty$ in the following way. We let $\tilde \nu_1$ be any positive real number such that 
\be\label{initialcondition}
\phi(s; E + \iu x) < - \delta + \frac{\delta - \mathfrak b}{2}.
\ee
for all $x \in (0, \tilde \nu_1]$. 
The existence of such a $\nu_1$ is guaranteed by the assumption \eqref{limsuphypo}. 
For brevity, we set $c_5 = 1 - c_2 s/7$.
We define $\tilde \nu_k$ for $k \ge 1$ recursively by $\tilde \nu_{k+1} = \tilde \nu_k^{c_5}$. 
Let $\mathfrak a$ be a parameter to be determined later, and let $M \in \Zplus$ be the smallest integer such that $\nu_M \ge \mathfrak a$. We define the sequence $\{\nu_k \}_{k=1}^M$ by $\nu_k = \tilde \nu_k$ for $k < M$ and $\nu_M = \mathfrak a$, and set $w_k = E + \nu_k$.

We now show by induction that $\phi(s;w_k) < - \mathfrak b$ for all $k \in \unn{1}{M}$.  The base case $k=1$ holds by the definition $\nu_1 = \tilde \nu_1$. For the induction step, suppose that $\phi(s;E + \nu_k) < - \mathfrak b$ holds for all $k\le n$. We may suppose that $\mathfrak b > \delta /2$. The bound \eqref{phigoingup} gives,  for all $k\le n$, that
\begin{align*}
\left|
\phi(s;w_k) - \phi(s;w_{k+1}) 
\right|  &\le C_3 ( - \log \nu_k)^{-1/5}\\
 &\le C_3 \left( - \log\left( \nu_1^{c_5^k} \right)\right)^{-1/5}
\le C_3 c_5^{-k/5} ( - \log \nu_1)^{-1/5}.
\end{align*}
This previous line implies that
\begin{align}\notag
\left|
\phi(s;w_1) - \phi(s;w_{n+1})
\right|
&\le 
\sum_{k=1}^{n} 
\left|
\phi(s;w_k) - \phi(s;w_{k+1})
\right| \\
&\le C_3 ( - \log \nu_1)^{-1/5} \sum_{k=1}^M c_5^{-k/5} 
\le  C_3 c_5^{-(M+1)/5} ( - \log \nu_1)^{-1/5}.\label{apple}
\end{align}
Since $\nu_{M-1} < \mathfrak a$, we have
\bex
\nu_1^{c_5^{-M+1}} < \mathfrak a.
\eex
Rearranging the previous line gives 
\bex
c_5^{-(M+ 1)/5}\le  c_5^{-2/5} \left( \frac{ - \log \nu_1}{ - \log \mathfrak a} \right)^{1/5}.
\eex
Inserting this equation in \eqref{apple} gives 
\begin{align*}
\left|
\phi(s;w_1) - \phi(s;w_{n+1})
\right|
\le C_3  c_5^{-2/5}
 ( - \log \mathfrak a)^{-1/5}.
\end{align*}
 Choosing $\mathfrak a(C_3, c_5, \delta, \mathfrak b)> 0 $ sufficiently small in the previous inequality gives 
\be\label{achoice1}
\left|
\phi(s;w_1) - \phi(s;w_{n+1})
\right| \le \frac{\delta - \mathfrak b}{2}.
\ee
Combining the previous line with \eqref{initialcondition} gives
\bex
\phi(s;w_{n+1}) < - \mathfrak b,
\eex
as desired. This completes the induction step.
We conclude that
\be\label{inductionconclude}
\phi(s; E + \iu \mathfrak a ) 
= \phi(s; E + \iu \nu_M )  < - \mathfrak b.
\ee

We now show that $\phi(s; E + \iu \nu ) < - \mathfrak{b}$ for any $\nu \in (0, \mathfrak{a})$. Recall that $\tilde \nu_1$ was defined through \eqref{initialcondition}; if $\nu \le \tilde \nu_1$, then we are done by this inequality. Otherwise, let $m\in \Zplus$ be the unique index such that 
$\nu_m < \nu \le \nu_{m+1}$, and define the sequence $\{\hat \nu_k \}_{k=1}^{m+1}$ by $\hat \nu_k = \nu_k$ for $k \le m$, and $\hat \nu_{m+1} = \nu$. Then the same argument that gave \eqref{inductionconclude} applied to the sequence $\{\hat \nu_k \}_{k=1}^{m+1}$ (with the index $M$ replaced by $m+1$) and our choice of $\mathfrak {a}$ before \eqref{achoice1} together give
\bex
\phi(s; E + \iu \hat \nu_{m+1}) = \phi(s; E + \iu \nu) < - \mathfrak{b}.
\eex
This completes the proof.
\end{proof}
\subsection{Proof of \Cref{l:phicontinuity}}\label{s:continuityproof}
We can now establish \Cref{l:phicontinuity} by combining \Cref{l:goingsideways}, \Cref{l:goingdown}, and \Cref{l:goingup}.
\begin{proof}[Proof of \Cref{l:phicontinuity}]
By \Cref{l:goingdown} and \Cref{l:goingup}, there exists $\delta_1(\kappa, \omega) \in (0, 1) $ such that the following two claims hold for all $E\in I$. 
First, if
\bex
\limsup_{\eta\rightarrow 0}
\phi(s; E + \iu \eta) < -\kappa.
\eex
then for all $\eta \in (0, \delta_1]$,
\be \label{down}
\phi(s; E + \iu \eta) < -\kappa + \frac{\omega}{3}.
\ee
Second, if $\eta \in (0, \delta_1]$ and
\bex
\phi(s; E + \iu \eta) < -\kappa + \frac{2 \omega}{3},
\eex
then
\be\label{up}
\limsup_{\eta\rightarrow 0}
\phi(s; E + \iu \eta) < -\kappa + \eps.
\ee
Next, by \Cref{l:goingsideways}, there exists $\delta_2(\delta_1, \kappa, \omega)>0$ such that, if 
\bex
\phi(s; E_1 + \iu \delta_1/2) < -\kappa + \frac{\omega}{3}
\eex
then
\be\label{sideways}
\big|
\phi(s; E_2 + \iu \delta_1/2) - \phi(s; E_1 + \iu \delta_1/2)
\big|
< \omega/3
\ee
for all $E_1, E_2 \in I$ such that $|E_1 - E_2| < \delta_2$. 

Under the assumption that $|E_1 - E_2| < \delta_2$, using \eqref{down}, \eqref{up}, and \eqref{sideways} with the choice $\eta = \delta_1/2$ gives 
\bex
\limsup_{\eta\rightarrow 0}
\phi(s; E_2 + \iu \eta) < -\kappa + \eps,
\eex
as desired.
\end{proof}

\newpage

\chapter{The Mobility Edge for Large and Small $\alpha$}

\label{Scaling} 
\section{Scaling Near One}\label{s:alphanear1}

The goal of this section is to prove \Cref{l:crudeasymptotic}, which states that any solution in $E$ to $\lambda(E,\alpha)=1$ (recall \eqref{lambdaEalpha}) must scale as $(1-\alpha)^{-1}$, as $\alpha$ tends to $1$. In \Cref{s:aEbEpreliminaries} we state some preliminary estimates on $\alpha$-stable laws. In \Cref{s:aEbEbounds} we estimate certain functionals of $a(E)$ and $b(E)$ (recall \eqref{opaque}), and in \Cref{s:asymptoticconclusion} we use these bounds to prove \Cref{l:crudeasymptotic}. Throughout this section, constants $C > 1$ and $c > 0$ will be independent of $\alpha$.

\subsection{Preliminaries}\label{s:aEbEpreliminaries}

We require the following tail bounds for nonnegative stable laws.
\begin{lem}\label{l:stabletailbounds}
There exists $C>1$ such that the following holds for all $\alpha \in [1/2, 1]$. Let $g_\alpha(x)$ be the density of a nonnegative $\alpha/2$-stable law. Then 
\be\label{stabletailbounds}
g_\alpha(x)  \le C \min(1 , x^{-1 -\alpha/2} ), \qquad
\big| g_\alpha'(x) \big| \le C \min(1 , x^{-2 -\alpha/2} ).
\ee
Further, for $x > C$, we have
\be\label{gprimenegative}
g_\alpha'(x) \le  - C^{-1}  x^{-2 - \alpha/2},
\ee
and
\be\label{gprimenegative2}
\big( x g_\alpha(x) \big)' \le - C^{-1} x^{-1 - \alpha/2}.
\ee
\end{lem}
\begin{proof}
The uniform bounds 
\bex
g_\alpha(x)  \le C, \qquad
\big| g_\alpha'(x) \big| \le C .
\eex
are easily obtained from the Fourier inversion formula, using the explicit representation of $\hat g_\alpha(k)$ given by \eqref{xtsigma} and recalling that the Fourier transform of $g'(x)$ is $\iu k \hat g(k)$.

We recall the following series expansion from \cite[(4)]{pollard1946representation} (see also \cite[(7)]{penson2010exact}). The representation
\bex
g_\alpha(x) = \frac{1}{\pi} \sum_{j=1}^\infty 
\frac{(-1)^{j+1}}{j! x^{1+\alpha j/2}} \Gamma( 1 + \alpha j/2) \sin(\pi \alpha j/2)
\eex
is valid for all $\alpha \in (0,1)$ and $x>0$. For $x>2$, it gives
\bex
\left| g_\alpha(x)\right|
\le \sum_{j=1}^\infty 
\frac{1}{j! x^{1+\alpha j/2}} \Gamma( 1 + j)
\le C x^{-1 - \alpha/2},
\eex
for some $C>1$. 
Differentiating term by term for $x>2$ gives
\bex
g'_\alpha(x) = \frac{1}{\pi x^{2+\alpha/2 }} \sum_{j=1}^\infty 
\frac{(-1)^{j}(1+\alpha j/2)}{j! x^{\alpha (j-1)/2}} \Gamma( 1 + \alpha j/2) \sin(\pi \alpha j/2). 
\eex
A similar bound gives $\big| g_\alpha'(x) \big| \le C  x^{-2 -\alpha/2}$. We next write the summation as
\bex
- (1 + \alpha/2)\Gamma(1+\alpha/2) \sin(\pi \alpha/2 ) + 
\sum_{j=2}^\infty 
\frac{(-1)^{j}(1+\alpha j/2)}{j! x^{\alpha (j-1)/2}} \Gamma( 1 + \alpha j/2) \sin(\pi \alpha j/2),
\eex
and bound, using $\alpha \ge 1/2$,
\bex
\left|
\sum_{j=2}^\infty 
\frac{(-1)^{j}(1+\alpha j/2)}{j! x^{\alpha (j-1)/2}} \Gamma( 1 + \alpha j/2) \sin(\pi \alpha j/2)
\right|
\le
\sum_{j=2}^\infty 
\frac{(1+ j/2)}{j! x^{(j-1)/4}} \Gamma( 1 +  j/2).
\eex
This bound is smaller than 
\bex
\frac{1}{2}(1 + 1/4)\Gamma(1+1/4) \sin(\pi/4 ) \le \frac{1}{2} (1 + \alpha/2)\Gamma(1+\alpha/2) \sin(\pi \alpha/2 )
\eex
for $x > C$ if $C>1$ is chosen large enough and $\alpha \in [1/2,1]$. This proves \eqref{gprimenegative}.
Next, we have
\bex
x g_\alpha(x) = \frac{1}{\pi} \sum_{j=1}^\infty 
\frac{(-1)^{j+1}}{j! x^{\alpha j/2}} \Gamma( 1 + \alpha j/2) \sin(\pi \alpha j/2),
\eex
and differentiating term by term for $x > 2$ gives
\bex
\big(x g_\alpha(x)\big)' = \frac{1}{\pi} \sum_{j=1}^\infty 
\frac{\alpha j}{2} \frac{(-1)^{j}}{j! x^{1+\alpha j/2}} \Gamma( 1 + \alpha j/2) \sin(\pi \alpha j/2).
\eex
We write this as 
\be\label{prv}
- \frac{\alpha}{2\pi} \frac{1}{x^{1 + \alpha/2}} \Gamma(1 + \alpha/2) \sin(\pi \alpha/2)
+ 
\frac{1}{\pi x^{1 + \alpha/2}} \sum_{j=2}^\infty 
\frac{\alpha j}{2} \frac{(-1)^{j}}{j! x^{\alpha (j-1)/2}} \Gamma( 1 + \alpha j/2) \sin(\pi \alpha j/2).
\ee
For every $\eps > 0$, there exists $C(\eps)>1$ such that 
\bex
\left|\sum_{j=2}^\infty 
\frac{\alpha j}{2} \frac{(-1)^{j}}{j! x^{\alpha (j-1)/2}} \Gamma( 1 + \alpha j/2) \sin(\pi \alpha j/2)\right|\le \eps
\eex
for $x > C$ uniformly for $\alpha \in [1/2, 1]$. Together with \eqref{prv}, this proves the final bound \eqref{gprimenegative2}.
\end{proof}

\subsection{Bounds on $a(E)$ and $b(E)$}\label{s:aEbEbounds}
In this section, we develop large $E$ asymptotics for $a(E)$ and $b(E)$ (recall \eqref{opaque}). We begin with a definition.

\begin{definition}
Fix $\gamma \in (0,1)$.
We
define the functions $F_\gamma\colon\bbR\times \bbR_+ \times \bbR_+ \rightarrow \bbR$ and 
$G_\gamma\colon{\bbR \times \bbR_+ \times \bbR_+} \rightarrow \bbR$ by
\begin{flalign}\label{FandG}
F_\gamma (E,x,y) &= \E \left[ \left( E + x^{2/\alpha} S - y^{2/\alpha} T   \right)_-^{-\gamma}  \right]; \qquad G_\gamma(E,x,y) = \E \left[ \left( E + x^{2/\alpha} S - y^{2/\alpha} T   \right)_+^{-\gamma}  \right],
\end{flalign} 
where $S$ and $T$ are independent, nonnegative $\alpha/2$-stable random variables.
\end{definition}
\begin{rem}\label{r:abfixedpointeqn}
We observe that the fixed point equations \eqref{fixedpoint} can be written as 
\bex
a(E) = F_{\alpha/2}(E , a , b), \qquad b(E) = G_{\alpha/2}(E, a, b).
\eex
\end{rem}

\begin{lem}\label{l:abasymptotic}
Suppose $a(E)$ and $b(E)$ solve \eqref{fixedpoint}.
Then there exists $c>0$ such that for all $\alpha \in (1-c, 1)$, $\gamma\in(0,1)$, and $E > c^{-1}$, 
\be\label{Fgammabound}
\left| F_{\gamma}\big(E , a(E) , b(E)\big) \right|  \le  \frac{c^{-1} E^{-\alpha - \gamma}}{1 - \gamma} ,\qquad \left| G_\gamma\big(E, a(E) ,b(E)\big) - E^{ - \gamma} \right| \le  \frac{c^{-1}  E^{-\alpha - \gamma}}{1-\gamma}.
\ee
\end{lem}
\begin{proof}

Throughout this proof, we abbreviate $a = a(E)$ and $b = b(E)$. Set $Y = a^{2/\alpha} S - b^{2/\alpha} T$, where $S$ and $T$ are independent, nonnegative $\alpha/2$-stable random variables, and let $g_\alpha(x)$ be the density of a nonnegative $\alpha/2$-stable law. Then the density of $a^{2/\alpha}S$ is $a^{-2/\alpha} g_\alpha(x a^{-2/\alpha})$, and the density of $-b^{2/\alpha}T$ is $b^{-2/\alpha} g_\alpha( - x b^{-2/\alpha})$. The density $f(y)$ of $Y$ is given by their convolution. For $y > 0$, we bound
\begin{align}
	\begin{aligned}
f(y) &= \int_{\bbR} a^{-2/\alpha} g_\alpha(x a^{-2/\alpha}) \cdot b^{-2/\alpha} g_\alpha\big( - (y-x) b^{-2/\alpha}\big)\, dx  \\ 
&= \int_{y}^\infty a^{-2/\alpha} g_\alpha(x a^{-2/\alpha}) \cdot b^{-2/\alpha} g_\alpha\big( (x-y) b^{-2/\alpha}\big)\, dx\\ & \le  \sup_{x >y} a^{-2/\alpha} g_\alpha(x a^{-2/\alpha}) \le C a  y^{-1 - \alpha/2 },
\end{aligned} 
\label{yuppertail}
\end{align}

\noindent for some constant $C>1$, where for the last inequality we used \Cref{l:stabletailbounds}.
Similarly, 
\be\label{ylowertail}
f(y) \le C b |y|^{-1 - \alpha/2 }, \qquad \text{for $y < 0$}.
\ee

We write the first equation in \eqref{FandG} as
\begin{align}
F_{\gamma} (E , a , b)  &=   \int_{-\infty}^{ -E } |E + y|^{-\gamma}  f(y) \, dy =
\int_{- 3E/2 }^{-E} 
|E + y|^{-\gamma}  f(y) \, dy
 +
\int_{-\infty}^{ - 3E/2 }
|E + y|^{-\gamma}  f(y) \, dy.\label{Ffirstpiece}
\end{align}
Using \eqref{ylowertail} and substituting $y = Ex$ in the first term of \eqref{Ffirstpiece}, we get 
\begin{align*}
\int_{- 3E/2  }^{-E} 
|E + y|^{-\gamma}  f(y) \, dy 
&\le C b E^{- 1 - \alpha/2} \int_{- 3E/2 }^{-E}  |E + y|^{-\gamma} \, dy \\
&\le Cb E^{- 1 - \alpha/2} \int_{-3/2}^{-1} E^{-\gamma} | 1 + x|^{-\gamma} E \, dx \le C(1- \gamma)^{-1}b  E^{-\alpha/2 - \gamma}.
\end{align*}
For the second, 
\begin{align*}
\int_{-\infty}^{ - 3E/2  }
|E + y|^{-\gamma}  f(y) \, dy &= 
\int_{-\infty}^{ - 3/2  }
|E + Ex |^{-\gamma}  f(Ex) \,  E \, dx \\
&= E^{1 - \gamma}
\int_{-\infty}^{ - 3/2  }
|1 + x |^{-\gamma}  f(Ex)  \, dx \\
&\le C  b E^{1 - \gamma}
\int_{-\infty}^{ - 3/2  }
|1 + x |^{-\gamma}  E^{-1 - \alpha/2} |x|^{-1 -\alpha/2}  \, dx \le C b E^{-\alpha/2 - \gamma}.
\end{align*}

\noindent Summing these estimates, we obtain
\be\label{Fgammaestimate}
\big| F_{\gamma} (E , a , b)  \big| \le  C(1- \gamma)^{-1}b  E^{-\alpha/2 - \gamma}.
\ee

Next, we write the second equation in \eqref{FandG} as
\bex
G_{\gamma}(E , a , b) =   \int_{-E}^{\infty}
|E + y|^{-\gamma}  f(y) \, dy.
\eex
We break the interval of integration into the three intervals $[-E, -E/2]$, $[-E/2, E/2]$, and $[E/2, \infty]$. For the first using \eqref{ylowertail} and substituting $y =Ex$, we find
\begin{align*}
\int_{-E  }^{ - E/2 }
|E + y|^{-\gamma}  f(y) \, dy
&\le Cb E^{-1  - \alpha/2} \int_{-E}^{-E/2} |E + y |^{-\gamma}  \, dy\\
&= C b E^{- 1 - \alpha/2} \int_{-1}^{-1/2} E^{-\gamma} | 1 + x|^{-\gamma} E \, dx \le Cb ( 1- \gamma)^{-1} E^{-\alpha/2 - \gamma}.
\end{align*}
Also, using \eqref{yuppertail} and substituting $y =Ex$, we find
\begin{align*}
\int_{E /2}^{ \infty }
|E + y|^{-\gamma}  f(y) \, dy
&\le C a \int_{E/2}^{\infty} |E + y |^{-\gamma} y^{-1 - \alpha/2} \, dy\\
&= C a E^{- 1 - \alpha/2} \int_{1/2}^{\infty} E^{-\gamma} | 1 + x|^{-\gamma} x^{-1 - \alpha/2} E \, dx \le Ca  E^{-\alpha/2 - \gamma}.
\end{align*}

It remains to evaluate 
\bex
\int_{-E/2  }^{ E/2  }
|E + y|^{-\gamma}  f(y) \, dy
=E^{-\gamma}\int_{-E/2  }^{ E/2  }
|1 + y/E|^{-\gamma}  f(y) \, dy.
\eex
We have the Taylor series expansion
\bex
( 1 + x)^{-\gamma} = 1 +  \gamma \cdot O \big( |x| \big)
\eex
for $|x| < 1/2$, where the implicit constant is uniform in $\gamma\in (0,1)$.
Inserting this expansion into the previous integral, the first term yields
\bex
E^{-\gamma}\int_{-E/2  }^{ E/2  } f(y) \, dy = E^{-\gamma} \left ( 1  + (a+b) \cdot O(E^{-\alpha/2}) \right),
\eex
where we used \eqref{yuppertail} and \eqref{ylowertail}.
The error term is bounded by
\begin{align}
C \gamma E^{-1 - \gamma}\int_{-E/2  }^{ E/2  }   |y| f(y) \, dy
& \le 
C \gamma E^{-1 - \gamma}
\left(
\int_{-1 }^{ 1 }   |y| f(y) \, dy +   a \int_{1 }^{ E/2  }  y^{-\alpha/2} dy + b \int_{-E/2 }^{-1  }  |y|^{-\alpha/2} \, dy
\right)\notag
\\
&\le 
C \gamma E^{-1 - \gamma}
\left(
1  +  (a+b) E^{1 - \alpha/2}
\right). 
\end{align}

\noindent Summing these estimates, we obtain
\begin{align}\label{Ggammaestimate}
	\begin{aligned} 
\big| G_{\gamma} (E , a , b) - E^{-\gamma} \big| &\le  
C(a+b) ( 1- \gamma)^{-1} E^{-\alpha/2 - \gamma}\\
& \qquad +  C E^{-\gamma}  (a+b) E^{-\alpha/2} + C \gamma E^{-1 - \gamma}
\left(
1  +  (a+b) E^{1 - \alpha/2}
\right).
\end{aligned} 
\end{align}
With $\gamma = \alpha/2$, we obtain from the assumption that $a = a(E)$ and $b = b(E)$ satisfy \eqref{fixedpoint}, \eqref{Fgammaestimate}, and \eqref{Ggammaestimate} that 
\be\label{Fgammaestimate2}
a \le  Cb  E^{-\alpha},
\ee
and 
\begin{align}\label{Ggammaestimate2}
| b - E^{-\alpha/2} | &\le  
C(a+b)  E^{-\alpha}+  E^{-\alpha/2} \cdot C (a+b) E^{-\alpha/2} + C  E^{-1 - \alpha/2}
\left(
1  +  (a+b) E^{1 - \alpha/2}
\right).
\end{align}
After substituting \eqref{Fgammaestimate2} into \eqref{Ggammaestimate2}, we see that there exists $C > 0$ such that $b < C E^{-\alpha/2}$ (otherwise \eqref{Ggammaestimate2} is a contradiction). 
Putting this bound into \eqref{Fgammaestimate} yields the first claim in \eqref{Fgammabound}. We also note that $b < C E^{-\alpha/2}$ and \eqref{Fgammaestimate2} together yield
\be\label{aplusb}
|a + b| \le C E^{-\alpha/2}.
\ee
The second claim of \eqref{Fgammabound} then follows from inserting \eqref{aplusb} into \eqref{Ggammaestimate}. This completes the proof.
\end{proof}

\begin{lem}\label{l:ablower}
Fix a constant $A > 1$. Then there exists $c(A) > 0$ such that for all $|E| < A$, we have $a(E) > c$ or $b(E) > c$ for all $\alpha \in (1/2,1)$.
\end{lem}
\begin{proof} 
By the symmetry of $a(E)$ and $b(E)$, it suffices to consider the case $E>0$. 
Throughout this proof, we abbreviate $a = a(E)$ and $b = b(E)$, and set $Y = a^{2/\alpha} S - b^{2/\alpha} T$, where $S$ and $T$ are independent, nonnegative $\alpha/2$-stable random variables. 
For $t >0$, we have 
\bex
\P[ Y > t] \le \P[ a^{2/\alpha} S  > t] \le C_0 (t a^{-2/\alpha})^{-\alpha/2  } = C_0 t^{-\alpha/2 } a 
\eex
for some $C_0>1$, by \eqref{stabletailbounds}. 
Similarly, 
\bex
\P[ Y < - t]  \le C_0 t^{-\alpha/2 } b.
\eex
We conclude that
\be\label{garlic}
\P\big[ Y \in ( -t, t) \big]  \ge 1 -  C_0 t^{-\alpha/2} (a  + b).
\ee
Together with \eqref{FandG} and \eqref{r:abfixedpointeqn}, the previous line yields \begin{align*}
b(E) = \E \left[ \left( E + Y   \right)_+^{-\alpha/2}  \right]
&\ge 
\E \left[ \left( E + Y   \right)_+^{-\alpha/2} \one_{\{Y \in ( -t, t) \}} \right] \\ &\ge  (E+t)^{-\alpha/2} \big( 1 -  C_0 t^{-\alpha/2} (a  + b)\big).
\end{align*}
Setting $t=1$ and using $E <A$, the previous line gives
\be\label{e:prev1}
(A+1)^{\alpha/2} b + C_0  (a+ b ) \ge 1.
\ee
This shows the theorem, after choosing $c$ depending on $C_0$ and $A$ (since we reach a contradiction if the quantity $\max(a,b)$ tends to zero).
\end{proof}

The proof of the following lemma is similar to the proof of \Cref{l:boundarybasics}, so we omit it.
\begin{lem}\label{l:boundarybasics2}
Fix $E \in \bbR$. 
Then
\bex
F_{\alpha}\big(E, a(E), b(E)\big) = \E\left[ \big( R_\loc(E) \big)^{\alpha}_+   \right], \qquad G_{\alpha}\big(E, a(E), b(E)\big) = \E\left[ \big( R_\loc(E) \big)^{\alpha}_-  \right].
\eex
\end{lem}
The next lemma provides a useful lower bound on $\E\left[ \big( R_\loc(E) \big)^{\alpha}   \right]$.
\begin{lem}\label{l:Rlower}
Fix a constant $A > 1$. Then there exists $c(A) > 0$ such that for all $|E| < A$, we have $\E\left[ \big( R_\loc(E) \big)^{\alpha}   \right] \ge c$ for all $\alpha \in (1/2,1)$.
\end{lem}
\begin{proof}
We abbreviate $a = a(E)$ and $b = b(E)$, and set $Y = a^{2/\alpha} S - b^{2/\alpha} T$, where $S$ and $T$ are independent, nonnegative $\alpha/2$-stable random variables. 
By \Cref{l:boundarybasics}, \Cref{l:boundarybasics2}, and H\"older's inequality,
\begin{align*}
b(E) = \E \left[ \left( E + Y   \right)_+^{-\alpha/2}  \right]
\le \E \left[ \left( E + Y   \right)_+^{-\alpha}  \right]^{1/2}
=\E\left[ \big( R_\loc(E) \big)^{\alpha}_-  \right]^{1/2}.
\end{align*}
Similarly, 
\begin{align*}
a(E)\le \E\left[ \big( R_\loc(E) \big)^{\alpha}_+ \right]^{1/2}.
\end{align*}
The previous two lines and \Cref{l:ablower} together yield 
\begin{align*}
c \le a(E)^{2} + b(E)^{2} \le \E\left[ \big( R_\loc(E) \big)^{\alpha} \right].
\end{align*}
for $|E| < A$. This completes the proof.
\end{proof}

\subsection{Conclusion}\label{s:asymptoticconclusion}

We require the below lemma, which follows from routine Taylor series approximations.
\begin{lem}\label{l:mobeqasymptotics}
As $\alpha$ tends to $1$, we have
\bex
K_\alpha =  \frac{2 \alpha}{ (1-\alpha)^2} + O\big((1 - \alpha)^{-1}  \big)
, \qquad t_1^2 - t_{\alpha}^2 =  \frac{\pi^2}{4} (\alpha - 1)^{2} + O\big( (1-\alpha)^3  \big).
\eex
\end{lem}
\begin{proof}
The function $\Gamma(1/2 - \alpha/2)$ has has pole at $\alpha = 1$ of order $1$, with the expansion 
\bex
\Gamma(1/2 - \alpha/2) = \frac{2}{1-\alpha} + O(1).
\eex
Using the definition of $K_\alpha$ in \eqref{tlrk}, we find
\be\label{Kalpha}
K_\alpha = \frac{\alpha}{2} \cdot \Gamma(1/2 - \alpha/2)^2 = \frac{2 \alpha}{ (1-\alpha)^2} + O\big((1 - \alpha)^{-1}  \big).
\ee
This proves the first claim.

Next, we note that
\bex
t_\alpha^2 - t_{1}^2 = (t_\alpha - t_1)(t_\alpha + t_1)
\eex
has a double root at $\alpha = 1$. Using $t_\alpha = \sin(\alpha \pi/2)$ and considering a power series around $\alpha = 1$, we find 
\bex
t_\alpha^2 - t_{1}^2 =  -\frac{\pi^2}{4} (\alpha - 1)^{2} + O\big( (1-\alpha)^3  \big),
\eex
which completes the proof.
\end{proof}
We are now ready to prove the main result of this section. 
\begin{lem}\label{l:crudeasymptotic}
There exist $c, c_0>0$ such that for all $\alpha \in (1-c ,1)$ the following holds.
Any $E_{\mob} \in \bbR$ such that $\lambda(E_{\mob}, \alpha) = 1$ satisfies
\bex
c_0 ( 1- \alpha)^{-1} \le  E_{\mob} \le c_0^{-1} (1- \alpha)^{-1}.
\eex
Further, we have $\lambda(E, \alpha) > 1$ for $E \in \big(0, c_0 (1 - \alpha)^{-1} \big)$ and $\lambda(E, \alpha) < 1$ for
$E \in \big( c_0^{-1} (1 - \alpha)^{-1}, \infty \big)$
\end{lem}

\begin{proof}
First, note that by \Cref{2lambdaEsalpha} and the second part of \Cref{l:lambdalemma}, we have (using $t_1 = 1$) that
		\begin{flalign}
		\lambda(E,\alpha) &= \lim_{s\rightarrow 1} 	\lambda (E, s, \alpha)\notag \\ & =  \pi^{-1} \cdot K_{\alpha} \cdot \Gamma (\alpha) \cdot \bigg(  t_{\alpha} \sqrt{1 - t_{\alpha}^2} \cdot \mathbb{E} \Big[ \big| R_{\loc} (E) \big|^{\alpha} \Big] \label{kiwi} \\
			& \qquad + \sqrt{ (1 - t_{\alpha}^2) \cdot \mathbb{E} \Big[ \big| R_{\loc} (E) \big|^{\alpha} \Big]^2 + t_{\alpha}^2 (1 - t_{\alpha}^2) \cdot \mathbb{E}  \Big[ \big| R_{\loc} (E) \big|^{\alpha} \cdot \sgn \big( - R_{\loc} (E) \big) \Big]^2} \bigg).\notag
		\end{flalign}
For brevity, in this proof  we write $F_\alpha = F_\alpha(E) = F_\alpha(E,a(E), b(E))$ and $G_\alpha = G_\alpha(E) = G_{\alpha} (E,a(E), b(E))$. Recalling \Cref{l:boundarybasics2}, we have
\bex
F_{\alpha} = \E\left[ \big( R_\loc(E) \big)^{\alpha}_+   \right], \qquad G_{\alpha} = \E\left[ \big( R_\loc(E) \big)^{\alpha}_-  \right].
\eex
Putting the previous line into \eqref{kiwi} gives
		\begin{flalign}\notag
		\lambda(E,\alpha) & =  \pi^{-1} \cdot K_{\alpha} \cdot \Gamma (\alpha) \cdot \bigg(  t_{\alpha} \sqrt{1 - t_{\alpha}^2} \cdot (F_\alpha + G_\alpha)  \\
			& \qquad + \sqrt{ (1 - t_{\alpha}^2) \cdot (F_\alpha + G_\alpha)^2 + t_{\alpha}^2 (1 - t_{\alpha}^2) \cdot (- F_\alpha + G_\alpha)^2} \bigg)\notag\\
			& =  \pi^{-1} \cdot K_{\alpha} \cdot \Gamma (\alpha) \cdot \bigg(  t_{\alpha} \sqrt{1 - t_{\alpha}^2} \cdot (F_\alpha + G_\alpha) \label{kiwi3} \\
			& \qquad + \sqrt{1 - t^2_\alpha} \cdot \sqrt{ 2 t_\alpha^2 (F_\alpha^2 + G_\alpha^2) + (1 - t_{\alpha}^2) ( F_\alpha + G_\alpha)^2} \bigg)\notag
		\end{flalign}

Let $c_1>0$ be the constant given by  \Cref{l:abasymptotic} (so that \eqref{Fgammabound} holds if $E > c_1^{-1}$). By \Cref{l:Rlower}, there exists a constant $c_2>0$ such that 
\be F_\alpha + G_\alpha = \E\left[ \big( R_\loc(E) \big)^{\alpha}   \right] \ge c_2\label{kiwi2}\ee for $|E| < c_1^{-1}$. Using  \eqref{kiwi3}, \eqref{kiwi2}, and \Cref{l:mobeqasymptotics}, we find that there exists a constant $c_3 >0$ such that $\lambda(E,\alpha) > 2$ for $|E| < c_1^{-1}$ and $\alpha \in (1- c_3, 1)$. For the rest of the proof, we suppose that $|E| > c_1^{-1}$ and $\alpha \in (1- c_3, 1)$, so that we can apply the conclusions of \Cref{l:abasymptotic}. 

From \eqref{Fgammabound}, we have  
\be\label{somebounds}
0\le F_\alpha \le C (1-\alpha) E^{-2\alpha} ,
\qquad |G_\alpha - E^{-\alpha} | \le C (1-\alpha)^{-1} E^{-2\alpha}.
\ee
The conclusion of the lemma will now follow from combining the previous equation with \eqref{kiwi3}. We give only the details of the proof of the upper bound on $E_{\mob}$, since the lower bound is similar. Inserting \eqref{somebounds} into \eqref{kiwi} and using the second estimate in \Cref{l:mobeqasymptotics} gives
\begin{align}
\lambda(E,\alpha)&\ge 
\pi^{-1} \cdot K_{\alpha} \cdot \Gamma (\alpha) \cdot  t_{\alpha} \sqrt{1 - t_{\alpha}^2} \cdot \big(E^{-\alpha} -C (1-\alpha)^{-1} E^{-2\alpha} \big). \label{kiwi4}
\end{align}
From \Cref{l:mobeqasymptotics}, we have 
\bex
 K_{\alpha} \cdot \Gamma (\alpha) \cdot  t_{\alpha} \sqrt{1 - t_{\alpha}^2} = \frac{\pi \alpha}{(1-\alpha)}  + O(1)
\eex
as $\alpha$ tends to $1$. Then from the previous equation and \eqref{kiwi4}, there exists a constant $c>0$ such that 
\be
\lambda(E,\alpha) \ge c (1-\alpha)^{-1} E^{-\alpha} - C (1-\alpha)^{-2} E^{-2\alpha}.
\ee
The previous equation shows that there exists $C_0>1$ such that $\lambda(E,\alpha) > 0$ if $E> C_0 (1-\alpha)^{-1/\alpha}$. This proves the claimed upper bound on $E_{\mob}$ and completes the proof after noticing that \bex2 (1-\alpha)^{-1/\alpha} >(1-\alpha)^{-1}\eex for $\alpha$ sufficiently close to $1$ (and decreasing $c_3$ if necessary).\end{proof}

\section{Uniqueness Near One}
\label{s:alphanear1uniqueness}

In this section we collect some derivative bounds, which we then use to prove the uniqueness of the mobility edge for $\alpha$ near $1$. In \Cref{s:dxy} and \Cref{DerivativesG} we consider the derivatives of $F_\gamma$ and $G_\gamma$ in $x$ and $y$, and in \Cref{s:dE} we consider the derivative in $E$. \Cref{s:dAB} and \Cref{s:dtildeAB} consider derivatives of certain fractional moments of the resolvent $R_\star$, and \Cref{s:proveuniqueness} concludes with the proof of \Cref{t:main2}. Throughout this section, constants $c > 0$ and $C > 1$ will be independent of $\alpha$.

\subsection{Derivatives in $x$ and $y$ of $F$}\label{s:dxy}

We begin with the following integral estimate that will be used to bound derivatives of $F_{\gamma}$ and $G_{\gamma}$. 

\begin{lem}
There exists a constant $C>1$ such that the following holds for all $\alpha \in (0,1)$. 
Let $g(s)$ denote the density of a nonnegative $\alpha/2$-stable law.
For $x, E>0$ and $\gamma \in (0,1)$, we have 
\be\label{already!}
 \int_0^\infty s (E + x^{2/\alpha} s)^{-\gamma-1}g(s)\, ds
\le Cx^{-2/\alpha + 1}E^{-\alpha/2 - \gamma}.
\ee
\end{lem}
\begin{proof}

We have
\bex
(E + x^{2/\alpha} s)^{-\gamma-1}
\le \max(E, x^{2/\alpha} s)^{-\gamma-1}
\eex
We split the integral on the left side of \eqref{already!} at $s =  x^{-2/\alpha} E$ and use the bound $s g(s) \le C s^{-\alpha/2}$ from \Cref{l:stabletailbounds} to obtain
\begin{align*}
 \int_0^{x^{-2/\alpha} E} s (E + x^{2/\alpha} s)^{-\gamma-1}g(s)\, ds 
&\le C E^{-1 - \gamma} \int_0^{ x^{-2/\alpha} E} s^{-\alpha/2} \,ds\\
&\le 
C E^{-1 - \gamma} \cdot x^{ - 2/\alpha  + 1}  E^{ 1 - \alpha/2} = C x^{-2/\alpha + 1} E^{-\alpha/2 - \gamma}.
\end{align*}
The second piece is 
\begin{align*}
 \int_{x^{-2/\alpha} E}^\infty s (E + x^{2/\alpha} s)^{-\gamma-1}g(s)\, ds &\le
\int_{x^{-2/\alpha} E}^\infty s (x^{2/\alpha} s)^{-\gamma-1}g(s)\, ds
\\ & \le x^{-(2/\alpha)(\gamma + 1)}\int_{x^{-2/\alpha} E}^\infty s^{-\gamma }g(s)\, ds\\
& \le C x^{-(2/\alpha)(\gamma + 1)}
\int_{x^{-2/\alpha} E}^\infty s^{-1  - \alpha/2 -\gamma} \, ds\\
& \le C x^{-(2/\alpha)(\gamma + 1)} (x^{-2/\alpha} E)^{-\alpha/2 - \gamma} \le C x^{-2/\alpha + 1} E^{-\alpha/2 - \gamma} .
\end{align*}
This completes the proof.
\end{proof}

The following two lemmas then bound the derivatives of $F$ with respect to $x$ and $y$.

\begin{lem}\label{l:Fx}
There exists a constant $c \in (0, 1)$ such that for all $\alpha \in (1-c, 1)$, $\gamma \in (0,1)$ and $E > c^{-1}$, we have 
\bex
|\partial_x F_\gamma (E, x,y) | \le \frac{  y E^{-1 - \alpha }}{c(1- \gamma)}, \qquad \text{for all $x \le 1$ and $y \le 1$}.
\eex
 
\end{lem}
\begin{proof}
By the definition of $F_\gamma(E,x,y)$, we have
\bex
F_\gamma(E, x,y) =  \int \int_{x^{2/\alpha}s - y^{2/\alpha} t < -E }
 |E  + x^{2/\alpha}  s - y^{2/\alpha}  t|^{-\gamma}  g(t) g(s) \, dt\, ds,
\eex
where $g$ is the density of the one-sided $\alpha/2$-stable law. We now rewrite the limits of integration. For the outer integral, we must have $s >0$ for $g(s)$ to be nonzero. In the interior integral, the range of $t$ on which the integrand is supported is 
\begin{flalign*} 
	\{ t \in \mathbb{R}: -y^{2/\alpha} t < - E - x^{2/\alpha} s \} = \big\{ t \in \mathbb{R}: t > y^{-2/\alpha} ( E +  x^{2/\alpha} s ) \big\}.
\end{flalign*} 

\noindent Hence,
\be\label{Flimitsrewrite}
F_\gamma(E, x,y) =  \int_{0}^\infty g(s)
\int_{y^{-2/\alpha} ( E +  x^{2/\alpha} s )}^\infty
 (- E  -  x^{2/\alpha}  s + y^{2/\alpha}  t)^{-\gamma}   g(t) \, dt\, ds.
\ee

\noindent Changing variables $u = y^{2/\alpha} t - E - x^{2/\alpha} s$, we obtain
\be\label{Fu}
F_\gamma(E,x,y) = y^{-2/\alpha}  \int_0^\infty
 \int_{ 0 }^\infty
 u^{-\gamma} 
g\big(y^{-2/\alpha} (u +E  + x^{2/\alpha} s )\big) g(s) \, du \, ds.
\ee

We differentiate \eqref{Fu} in $x$ to obtain
\bex
\partial_x F_\gamma (E, x,y) = (2/ \alpha) y^{-4/\alpha} x^{2/\alpha-1 }\int_0^\infty s  \int_{ 0 }^\infty
u^{-\gamma}
g'(y^{-2/\alpha} (u +E  + x^{2/\alpha} s ))   g(s)  \, du \,ds.
\eex
Since $y \le 1$ and $E > c^{-1} \ge 1$, we have $y^{-2/\alpha} (u +E  + x^{2/\alpha} s ) \ge 1$. Then using  the estimate $|g'(w)| \le C w^{-\alpha/2 -2 }$ (which holds by \Cref{l:stabletailbounds}), the interior integral is bounded in absolute value by
\begin{align*}
	C y^{-4/\alpha} x^{2/\alpha-1 } & \int_{ 0 }^\infty
u^{-\gamma}
\big(y^{-2/\alpha} (u +E  + x^{2/\alpha} s )\big)^{-\alpha/2 -2}    \, du
\\ &= C y x^{2/\alpha-1 }  \int_{ 0 }^\infty
u^{-\gamma}
(u +E  + x^{2/\alpha} s )^{-\alpha/2 -2}    \, du.
\end{align*}
Therefore, it suffices to bound
\be\label{Fusplit}
 y x^{2/\alpha-1 }\int_0^\infty s g(s)  \int_{ 0 }^\infty
u^{-\gamma}
(u +E  + x^{2/\alpha} s )^{-\alpha/2 -2}    \, du \, ds
\ee
to complete the proof.

We now split the interior integral in \eqref{Fusplit} at $u=1$. For $u \in (0,1)$, we can bound the integrand using
\bex
(u +E  + x^{2/\alpha} s )^{-\alpha/2 -2} \le
(E  + x^{2/\alpha} s ) ^{-\alpha/2 - 2}  \le 
 \max (E, x^{2/\alpha} s)^{-\alpha/2 - 2},
\eex
which to leads to 
\begin{align}\label{parts1}
y x^{2/\alpha-1 } & \int_0^\infty s g(s)  \int_{ 0 }^1
u^{-\gamma}
(u +E  + x^{2/\alpha} s )^{-\alpha/2 -2}    \, du \, ds\\
&\le C ( 1- \gamma)^{-1}y x^{2/\alpha-1 } 
\int_0^\infty   \max(E, x^{2/\alpha} s)^{-\alpha/2 - 2} s g(s) \, ds\notag
\end{align}
after integrating in $u$. 
We split the integral in $s$ at $s = x^{-2/\alpha} E$. The contribution from $s <  x^{-2/\alpha} E$ is at most 
\be\label{slowerregime}
C ( 1- \gamma)^{-1} y x^{2/\alpha-1 }  E^{-\alpha/2 - 2 } \int_0^{x^{-2/\alpha} E}  s g(s) ds.
\ee

\noindent Again applying \Cref{l:stabletailbounds},
\bex
\int_0^{x^{-2/\alpha} E}  s g(s) ds \le C + C \int_1^{x^{-2/\alpha} E}  
s^{-\alpha/2} ds \le C( 1 + x^{1 - 2/\alpha} E^{-\alpha/2 + 1}).
\eex
The contribution from \eqref{slowerregime} is then bounded by
\begin{align}
C(1-\gamma)^{-1} y x^{2/\alpha-1 }  E^{-\alpha/2 - 2 }  (1 + x^{1 - 2/\alpha} E^{-\alpha/2 + 1}   )&\le C(1-\gamma)^{-1} y (x^{2/\alpha-1 }  E^{-\alpha/2 - 2 } 
+ E^{-1 - \alpha})\notag \\
&\le C(1-\gamma)^{-1} y E^{-1 - \alpha},\label{lemon1}
\end{align}
where we used the assumptions that $x\le 1$ and $E > c^{-1} \ge 1$. The contribution from $s >  x^{-2/\alpha} E$ in \eqref{parts1} is at most
\begin{align}
C(1-\gamma)^{-1} y x^{2/\alpha-1 } \int_{x^{-2/\alpha} E}^\infty (x^{2/\alpha} s )^{-\alpha/2 -2} s^{-\alpha/2} ds\notag
&=
C(1-\gamma)^{-1} y x^{-2/\alpha-2 }  \int_{x^{-2/\alpha} E}^\infty 
s^{-\alpha -2} ds\notag \\
& = C(1-\gamma)^{-1} y x^{-2/\alpha-2 }  (x^{-2/\alpha} E )^{-\alpha -1}\notag \\
&= C(1-\gamma)^{-1} y  E^{-\alpha -1}.\label{lemon2}
\end{align}

We next consider the $u\ge 1$ part of \eqref{Fusplit}. This can be bounded by
\begin{align*}
&C y x^{2/\alpha-1 }\int_0^\infty s g(s)  \int_{1}^\infty
(u +E  + x^{2/\alpha} s )^{-\alpha/2 -2}    \, du \, ds
\le
C y x^{2/\alpha-1 }\int_0^\infty s g(s) 
( E  + x^{2/\alpha} s )^{-\alpha/2 -1}  \, ds.
\end{align*}
Using \eqref{already!}, we bound this by $Cy E^{-\alpha}$. Combining the estimates \eqref{lemon1} and \eqref{lemon2} (for the $u\le 1$ part of \eqref{Fusplit}) and the previous line (for the $u\ge 1$ part) bounds \eqref{Fusplit} and completes the proof.
\end{proof}

\begin{lem}\label{l:Fy}
	
There exists a constant $c \in (0, 1)$ such that for all $\alpha \in (1-c, 1)$, $\gamma\in(0,1)$, and $E > c^{-1}$, we have 
\bex
 |\partial_y  F_\gamma(E, x,y)|  \le 
c^{-1} (1 - \gamma)^{-1} E^{-\alpha/2 - \gamma} \eex
and 
\be\label{Fyispositive}
\partial_y  F_\gamma(E, x,y) \ge 0 
\ee
for $y \le 1$. 
\end{lem}

\begin{proof}
Recall from \eqref{FandG} that
\bex
 F_\gamma(E, x,y)  =  \int \int_{x^{2/\alpha}s - y^{2/\alpha} t < -E }
 (E  + x^{2/\alpha}  s - y^{2/\alpha}  t)^{-\gamma}  g(s) g(t) \, ds\, dt .
\eex
Setting $v = y^{2/\alpha}  t$ gives
\bex
F_\gamma(E, x,y)  =  \int \int_{x^{2/\alpha} s - v < -E }
 |E  + x^{2/\alpha} s - v |^{-\gamma}  g(s) g(y^{-2/\alpha } v) y^{-2/\alpha} \, ds \, dv.
\eex

\noindent Since
\bex
\frac{d}{dy} \big( g(y^{-2/\alpha } v) y^{-2/\alpha} \big) = - \frac{2}{\alpha}g'( y^{-2/\alpha} v) \cdot y^{-1 - 4/\alpha} v  - \frac{2}{\alpha} y^{-1 -2/\alpha} \cdot g(y^{-2/\alpha } v),
\eex

\noindent we have 
\begin{flalign*}
	\partial_y F_{\gamma} (E, x, y) = - \displaystyle\frac{2}{\alpha} & \displaystyle\int \displaystyle\int_{x^{2/\alpha} s -v < -E} |E + x^{2/\alpha} s- v|^{-\gamma} g(s) \\
	& \qquad \qquad \times y^{-1-2/\alpha}  \big( g' (y^{-2/\alpha} v) v y^{-2/\alpha} + g (y^{-2/\alpha} v) \big) ds dv.
\end{flalign*}

\noindent 

\noindent Recalling $v = y^{2/\alpha}  t$, we find 
\bex
 \partial_y F_{\gamma} (E, x, y) = - \frac{2}{\alpha y } \int \int_{x^{2/\alpha}s - y^{2/\alpha} t < -E }
  |E  + x^{2/\alpha}  s - y^{2/\alpha}  t|^{-\gamma}  g(s) 
 [g( t )t]'
 \, ds \, dt.
\eex
after adding both terms. Rewriting the limits of integration as in \eqref{Flimitsrewrite} gives
\be\label{Flimitsrewrite2}
 \partial_y  F_\gamma(E,x,y) = - \frac{2}{\alpha y } \int_{0}^\infty
\int_{y^{-2/\alpha} ( E +  x^{2/\alpha} s )}^\infty
  |E  + x^{2/\alpha}  s - y^{2/\alpha}  t|^{-\gamma}  (tg(t))' 
g(s)
 \, ds \, dt.
\ee
We observe that $y^{-2/\alpha}(E + x^{2/\alpha} s ) \ge E$ for $y \le 1$ and $x \ge 0$. We also note that there exists $C>1$ such that $\big(t g(t)\big)' < 0$ for $t > C$, where $C$ is uniform in $\alpha \in (1/2, 1)$, by \Cref{l:stabletailbounds}. Hence the previous expression is positive, showing \eqref{Fyispositive}.

Using \Cref{l:stabletailbounds} to bound $\big|\big( t g(t) \big)' \big| \le g(t) + t \big| g'(t) \big| \le Ct^{-1-\alpha/2}$, we find from \eqref{Flimitsrewrite2} that
\bex
\big| \partial_y F_{\gamma} (E, x, y) \big| \le Cy^{-1} \int_{0}^\infty
\int_{y^{-2/\alpha} ( E +  x^{2/\alpha} s )}^\infty
  |E  + x^{2/\alpha}  s - y^{2/\alpha}  t|^{-\gamma} t^{-1 - \alpha/2} 
g(s)
 \, ds \, dt.
\eex

\noindent Changing variables $w =y^{2/\alpha}(E + x^{2/\alpha}s )^{-1} t$, it follows that 
\begin{align}
\big| \partial_y F_{\gamma} (E, x, y) \big| \le C y^{-1} &\int_{0}^\infty
\int_{1}^\infty
  \big|E  + x^{2/\alpha}  s - (E  + x^{2/\alpha}  s) w \big|^{-\gamma}  (w y^{-2/\alpha} \big(E + x^{2/\alpha} s) \big) ^{-1 - \alpha/2} \notag \\ 
  & \qquad \qquad \times y^{-2/\alpha} (E + x^{2/\alpha} s) g(s) ds dw \notag \\
 &= C  \int_{0}^\infty
(E  + x^{2/\alpha}  s)^{-\alpha/2 - \gamma} g(s)
\int_{1}^\infty
 w^{-1 - \alpha/2}  (w -1) ^{-\gamma}  
\, dw  \, ds .\label{129}
\end{align}
We note that
\be\label{wintegral!}
\int_{1}^\infty
 w^{-1 - \alpha/2}  (w -1) ^{-\gamma}  
\, dw  \le C ( 1 - \gamma)^{-1}.
\ee
Then we use the bound $(E  + x^{2/\alpha}  s)^{-\alpha/2 - \gamma} \le E^{-\alpha/2 - \gamma}$, \eqref{wintegral!}, and the fact that $g$ is a probability density
in \eqref{129} to obtain the desired bound $\big| \partial_y F_{\gamma} (E, x, y) \big| \le C (1 - \gamma)^{-1} E^{-\alpha/2 - \gamma}$. 
\end{proof}

\subsection{Derivatives in $x$ and $y$ of $G$} 

\label{DerivativesG} 

In this section we establish the below two lemmas that bound the derivatives of $G$ with respect to $x$ and $y$, which are parallel to those of \Cref{l:Fx} and \Cref{l:Fy}.
\begin{lem}\label{l:Gx}
There exists a constant $c \in (0, 1)$ such that for all $\alpha \in (1-c, 1)$ and $E > c^{-1}$, we have 
\begin{align}\label{gxfinal}
|\partial_x G_\gamma(E,x,y) | &\le 
C (1 - \gamma)^{-1} E^{-\alpha/2 - \gamma}  
+ C ( 1- \gamma)^{-2}  E^{-\alpha/2 -2 \gamma} y\\
&+ C E^{-1 - \alpha/2 - \gamma} y^{-2/\alpha}
+ C E^{-\gamma} y.
\end{align}
\end{lem}
\begin{proof}
From the definition \eqref{FandG}, we have
\bex
G_\gamma(E,x,y) =  \int \int_{x^{2/\alpha}s - y^{2/\alpha} t > -E }
 (E  + x^{2/\alpha}  s - y^{2/\alpha}  t)^{-\gamma}  g(s) g(t) \, dt \, ds 
\eex
where $g$ is the density of the nonnegative $\alpha/2$-stable law.

We now rewrite the limits of integration. The integrand of the outer integral is supported on $\{ s >0 \}$. For the interior integral, the integrand is supported on $\big\{ t \in \mathbb{R}_{> 0} : t < y^{-2/\alpha} ( E +  x^{2/\alpha} s ) \big\}$. This shows that
\be\label{Gintegralx}
G_\gamma(E,x,y)  = \int_{0}^\infty
\int_0^{y^{-2/\alpha} ( E +  x^{2/\alpha} s )}
 (E  + x^{2/\alpha}  s - y^{2/\alpha}  t)^{-\gamma}  g(s) g(t) \, dt\, ds.
\ee
In the interior integral, we set $ v = (E + x^{2/\alpha} s)^{-1} y^{2/\alpha} t$ to obtain
\begin{flalign}
G_{\gamma} (E, x, y) &=  y^{-2/\alpha}\int_0^\infty g(s)  (E + x^{2/\alpha} s)^{1-\gamma} \int_0^1\label{diffme2}
 (1  -   v )^{-\gamma} 
 g\big(y^{-2/\alpha} (E + x^{2/\alpha} s) v\big) \, dv \,ds.
\end{flalign}

We differentiate \eqref{diffme2} in $x$ and obtain two terms. The first is 
\be\label{diffresult1}
\displaystyle\frac{2}{\alpha} (1 - \gamma) y^{-2/\alpha} x^{2/\alpha -1}\int_0^\infty g(s)   (E + x^{2/\alpha} s)^{-\gamma} s \int_0^1
 (1  -   v )^{-\gamma} 
 g\big(y^{-2/\alpha} (E + x^{2/\alpha} s) v\big) \, dv\,ds,
\ee
and the second is
\be\label{put}
\frac{2}{\alpha}\cdot 
y^{-4/\alpha} x^{2/\alpha-1}\int_0^\infty g(s) (E + x^{2/\alpha} s)^{1-\gamma}s  \int_0^1
(1  -   v )^{-\gamma}   v g'\big(y^{-2/\alpha} (E + x^{2/\alpha} s) v\big) \, dv.
\ee

We begin by analyzing \eqref{diffresult1}.
We split the integral in $v$ at $v = (E (E + x^{2/\alpha} s ))^{-1}$.
For $v < (E (E + x^{2/\alpha} s ))^{-1}$, we obtain
\begin{align*}
\displaystyle\frac{2}{\alpha} (1 - \gamma) y^{-2/\alpha}& x^{2/\alpha -1} \int_0^\infty g(s)   (E + x^{2/\alpha} s)^{-\gamma} s \int_0^{(E (E + x^{2/\alpha} s ))^{-1}}
 (1  -   v )^{-\gamma} 
 g(y^{-2/\alpha} (E + x^{2/\alpha} s) v) \, dv\,ds
\\ \le& C
y^{-2/\alpha} x^{2/\alpha -1}\int_0^\infty g(s)   (E + x^{2/\alpha} s)^{-\gamma} s \big( E (E + x^{2/\alpha} s ) \big)^{-1} ds
\\ \le & C E^{-1}   y^{-2/\alpha}
x^{2/\alpha -1}\int_0^\infty g(s)   (E + x^{2/\alpha} s)^{-1-\gamma} s \,ds,
\end{align*}

\noindent where in the second bound we used the fact that $g(w) \le C$ for $w = y^{-2/\alpha}(E + x^{2/\alpha} s)v$ (from \eqref{stabletailbounds} and  that $E > c^{-1} > 2$ (by imposing that $c < \frac{1}{2}$). To bound the integral in $s$, we use \eqref{already!} to find 
\begin{align}
 E^{-1} &  y^{-2/\alpha}
x^{2/\alpha -1} \int_0^\infty g(s) (E + x^{2/\alpha} s)^{-\alpha/2-1} s\, ds \notag
\\
 &\le  C  E^{-1}   y^{-2/\alpha}
x^{2/\alpha -1} x^{-2/\alpha + 1}E^{-\alpha/2 - \gamma} \le 
 C  E^{-1-\alpha/2 - \gamma} y^{-2/\alpha}.\label{2ndtrm}
\end{align}
This yields the third term in the claimed bound \eqref{gxfinal}.

We now consider the case $v > (E (E + x^{2/\alpha} s ))^{-1}$ in \eqref{diffresult1}, which gives
\begin{align*}
&(1 - \gamma)y^{-2/\alpha} x^{2/\alpha -1}\int_0^\infty g(s) (E + x^{2/\alpha} s)^{-\gamma}s \int_{(E (E + x^{2/\alpha} s ))^{-1}}^1
 (1  -   v )^{-\gamma} 
 g(y^{-2/\alpha}(E + x^{2/\alpha} s) v) \, dv\, ds \notag 
\\ &\le 
C(1 - \gamma) y^{-2/\alpha} x^{2/\alpha -1}\int_0^\infty g(s) (E + x^{2/\alpha} s)^{-\gamma}s  \int_{(E (E + x^{2/\alpha} s ))^{-1}}^1
 (1  -   v )^{-\gamma} 
 (y^{-2/\alpha} (E + x^{2/\alpha} s) v)^{-1 - \alpha/2} \, dv\, ds \notag
\\ 
&\le 
 C (1 - \gamma) y x^{2/\alpha -1}\int_0^\infty g(s) (E + x^{2/\alpha} s)^{-1 -\alpha/2 - \gamma }s \int_{(E (E + x^{2/\alpha} s ))^{-1}}^1
 (1  -   v )^{-\gamma} 
v^{-1 - \alpha/2} \, dv\, ds
\\& \le 
 Cy  x^{2/\alpha -1}\int_0^\infty g(s) (E + x^{2/\alpha} s)^{1 -\alpha/2 - \gamma}s
 (E (E + x^{2/\alpha} s ))^{\alpha/2} \, ds
\\ &\le Cy  x^{2/\alpha -1} E^{\alpha/2} 
  \int_0^\infty  g(s) (E + x^{2/\alpha} s)^{-1 -\gamma}s \\
 & \le C y E^{ - \gamma} .
\end{align*}
In the last inequality, we used \eqref{already!}. This gives the last term in \eqref{gxfinal}. 

We now consider \eqref{put}. We split the integral where the argument of $g'$ is $1$, which leads to the condition
\bex
1 = y^{-2/\alpha} (E + x^{2/\alpha} s) v.
\eex
This gives
\bex
v = y^{2/\alpha} (E + x^{2/\alpha} s)^{-1}.
\eex
Using $|g'| \le C$ from \Cref{l:stabletailbounds}, we find
\begin{align}\label{vcontrib1}
\int_0^{y^{2/\alpha} (E + x^{2/\alpha} s)^{-1}}
(1  -   v )^{-\gamma}   v g'\big(y^{-2/\alpha} (E + x^{2/\alpha} s) v\big) \, dv
&\le  \int_0^{y^{2/\alpha} (E + x^{2/\alpha} s)^{-1}}
v \, dv\\ &= C y^{4/\alpha} (E + x^{2/\alpha} s)^{-2} \notag.
\end{align}
Putting this in \eqref{put} and using \eqref{already!} shows that the $v  < y^{2/\alpha} (E + x^{2/\alpha} s)^{-1}$ contribution from  
\eqref{vcontrib1} to \eqref{put} is bounded by
\bex
x^{2/\alpha -1} \int_0^\infty s (E + x^{2/\alpha} s)^{-1 - \alpha/2} \, ds \le C E^{-\alpha/2 - \gamma}.
\eex

We are left with the contribution from $v > y^{2/\alpha} (E + x^{2/\alpha} s)^{-1}$, which is 
\bex
y^{-4/\alpha} x^{2/\alpha-1} 
\int_0^\infty g(s) (E + x^{2/\alpha} s)^{1-\gamma}s 
 \int_{y^{2/\alpha} (E + x^{2/\alpha} s)^{-1}}^1
(1  -   v )^{-\gamma}   v g'\big(y^{-2/\alpha} (E + x^{2/\alpha} s) v\big) \, dv.
\eex
On this interval we use $|g'(x) |\le x^{-2 - \alpha/2}$ from \Cref{l:stabletailbounds}. After taking absolute value, we get the bound
\begin{align}\notag
&y^{-4/\alpha} x^{2/\alpha-1} \int_0^\infty g(s) (E + x^{2/\alpha} s)^{1-\gamma} s \int_{y^{2/\alpha} (E + x^{2/\alpha} s)^{-1}}^1
(1  -   v )^{-\gamma}   v \big(y^{-2/\alpha} (E + x^{2/\alpha} s) v\big)^{-2 - \alpha/2} \, dv\,ds\\
&= y x^{2/\alpha-1} \int_0^\infty g(s) (E + x^{2/\alpha} s)^{-1  - \alpha/2 - \gamma} s \int_{y^{2/\alpha} (E + x^{2/\alpha} s)^{-1}}^1
(1  -   v )^{-\gamma}   v^{-1-\alpha/2}  \, dv\, ds.\label{vpieces1}
\end{align}

 We first consider the piece of the integral in $v$ with $v \in (1/2 ,1)$.
Using 
\bex
\int_{1/2}^1
(1  -   v )^{-\gamma}   v^{-1-\alpha/2}  \, dv \le C( 1- \gamma)^{-1},
\eex
we obtain
\begin{align*}
&y x^{2/\alpha-1} \int_0^\infty g(s) (E + x^{2/\alpha} s)^{-1  - \alpha/2 - \gamma} s \int_{1/2}^1
(1  -   v )^{-\gamma}   v^{-1-\alpha/2}  \, dv\,ds \\
&\le C E^{-\gamma}( 1- \gamma)^{-1} y x^{2/\alpha -1} \int_0^\infty s (E + x^{2/\alpha} s)^{-1 - \alpha/2}g(s)\, ds\\
&\le C (1 - \gamma)^{-2} y E^{-\alpha/2 - 2 \gamma},
\end{align*}
where the last inequality follows from \eqref{already!}.
This corresponds to the second term in \eqref{gxfinal}.

The contribution from $v < 1/2$ in \eqref{vpieces1} is bounded by
\begin{align*}
& C y x^{2/\alpha-1} \int_0^\infty g(s) (E + x^{2/\alpha} s)^{-1  - \alpha/2 - \gamma}s \int_{y^{2/\alpha} (E + x^{2/\alpha} s)^{-1}}^{1/2}
  v^{-1-\alpha/2}  \, dv \, ds\\
&\le C y x^{2/\alpha-1} \int_0^\infty g(s) (E + x^{2/\alpha} s)^{-1  - \alpha/2 - \gamma} s 
\big( y^{2/\alpha} (E + x^{2/\alpha} s)^{-1} \big) ^{-\alpha/2}\, ds\\
& = x^{2/\alpha-1}\int_0^\infty g(s) (E + x^{2/\alpha} s)^{-1  - \gamma}s \, ds\\
&\le C( 1 - \gamma)^{-1} E^{-\alpha/2 - \gamma}.
\end{align*}
This corresponds to the first term in \eqref{gxfinal}, and completes the proof.
\end{proof}
\begin{lem}\label{l:Gy}
There exists $c \in \big( 0, \frac{1}{2} \big)$ such that for all $\alpha \in (1-c, 1)$, $\gamma\in(0,1)$, $x \le 1$, and $E > c^{-1}$, we have 
\be\label{Gy1}
|\partial_y G(E,x,y)|  \le c^{-1} \left( (1- \gamma)^{-1}E^{-\alpha/2 - \gamma} + E^{-\gamma} \right)
\ee
and
\bex
\partial_y G(E,x,y) \ge c (1 - \gamma)^{-1} E^{-\alpha/2 - \gamma}  + O(E^{-\gamma} ).
\eex
\end{lem}
\begin{proof}
Analogously to \eqref{Flimitsrewrite2}, we have
\be\label{firsttt}
 \partial_y G_\gamma(E,x,y) = - \frac{2}{\alpha y }  \int_{0}^\infty
\int_0^{y^{-2/\alpha} ( E +  x^{2/\alpha} s )}
  (E  + x^{2/\alpha}  s - y^{2/\alpha}  t)^{-\gamma}  \big(tg(t)\big)' 
g(s)
 \, dt \, ds.
\ee
We split the integral in $t$ into two parts at $t= y^{-2/\alpha}$.
We claim that the contribution from $t< y^{-2/\alpha}$ in \eqref{firsttt} satisfies
\be\label{firstt}
\Bigg| \frac{2}{\alpha y }  \int_{0}^\infty
\int_0^{y^{-2/\alpha}}
  (E  + x^{2/\alpha}  s - y^{2/\alpha}  t)^{-\gamma}  \big(tg(t)\big)' 
g(s)
 \, ds \, dt \Bigg| \le C E^{-\gamma}.
\ee
We integrate by parts with respect to $t$ in the integral on the left side; then the boundary term at $t=0$ vanishes, and the boundary term at $t=y^{-2/\alpha}$ is in absolute value, after using \Cref{l:stabletailbounds} to bound $g(y^{-2/\alpha}) \le Cy^{2/\alpha + 1}$, at most 
\begin{align*}
Cy^{-1} \int_{0}^\infty
(E  + x^{2/\alpha}  s - 1)^{-\gamma}  y^{-2/\alpha} g(y^{-2/\alpha}) g(s)
\, ds
&\le  C \int_{0}^\infty
(E  + x^{2/\alpha}  s - 1)^{-\gamma} g(s)
\, ds\\
&\le  C \int_{0}^\infty
(E- 1 )^{-\gamma} g(s)
\, ds \le C E^{-\gamma},
\end{align*}

\noindent where in the last two bounds we used the facts that $E - 1 \ge \frac{E}{2}$ (as $E > c^{-1} \ge 2$) and that $g$ is a probability density function. From \Cref{l:stabletailbounds}, we have $tg(t) \le C t^{-\alpha/2}$. Then the main term from integrating by parts satisfies the absolute value bound
\begin{align*}
\displaystyle\frac{4 \gamma}{\alpha^2} \Bigg|  y^{2/\alpha-1} & \int_{0}^\infty 
\int_0^{y^{-2/\alpha}}
  (E  + x^{2/\alpha}  s - y^{2/\alpha}  t)^{-\gamma-1}  tg(t)
g(s)
 \, ds \, dt \Bigg|\\
&\le C y^{2/\alpha-1} \int_{0}^\infty
\int_0^{y^{-2/\alpha}}
  (E  + x^{2/\alpha}  s - y^{2/\alpha}  t)^{-\gamma-1}  t^{-\alpha/2}
g(s)
 \, dt \, ds \\
& \le C y^{2/\alpha-1} y^{-2/\alpha}
 \int_{0}^\infty ( E + x^{2/\alpha}s) 
 \int_0^{ ( E + x^{2/\alpha}s)^{-1}}
 (E  + x^{2/\alpha}  s - ( E + x^{2/\alpha}s) w)^{-\gamma-1}\\ 
 & \qquad \qquad \qquad \qquad \qquad \times  (y^{-2/\alpha} ( E + x^{2/\alpha}s) w)^{-\alpha/2}
 g(s)
 \, ds \, dw,
 \end{align*}

\noindent where in the second bound we changed variables $t = y^{-2/\alpha} ( E + x^{2/\alpha}s) w$. Hence,   
\begin{align*}
\displaystyle\frac{4 \gamma}{\alpha^2} \Bigg|  & y^{2/\alpha-1} \int_{0}^\infty 
\int_0^{y^{-2/\alpha}}
(E  + x^{2/\alpha}  s - y^{2/\alpha}  t)^{-\gamma-1}  tg(t)
g(s)
\, ds \, dt \Bigg| \\ 
& \le 
C \int_{0}^\infty ( E + x^{2/\alpha}s)^{-\alpha/2 - \gamma}
\int_0^{( E + x^{2/\alpha}s)^{-1}}
  (1-  w)^{-\gamma-1}  w^{-\alpha/2}
g(s)
 \, dw \, ds
\\ & \le
C E^{-\alpha/2 - \gamma}  
\int_0^{ E^{-1} }
  (1-  w)^{-\alpha/2-1}  w^{-\alpha/2}
 \, dw  \le 
C E^{-\alpha/2 - \gamma}   \int_0^{E^{-1}}
   w^{-\alpha/2}
 \, dw \le  C E^{-1 - \gamma},
\end{align*}

\noindent where in the second bound we used the fact that $g$ is a probability density. This confirms \eqref{firstt}.

The contribution from $t> y^{-2/\alpha}$ in \eqref{firsttt} is 
\be\label{2t}
 - \frac{2}{\alpha y }  \int_{0}^\infty
\int_{y^{-2/\alpha}}^{y^{-2/\alpha} ( E +  x^{2/\alpha} s )}
  (E  + x^{2/\alpha}  s - y^{2/\alpha}  t)^{-\gamma}  \big(tg(t)\big)' 
g(s)
 \, ds \, dt.
\ee
We bound this in absolute value using \Cref{l:stabletailbounds} by
\bex
 \frac{C}{ y }  \int_{0}^\infty
\int_{y^{-2/\alpha}}^{y^{-2/\alpha} ( E +  x^{2/\alpha} s )}
  (E  + x^{2/\alpha}  s - y^{2/\alpha}  t)^{-\gamma}  t^{-1-\alpha/2}
g(s)
\, dt
 \, ds.
\eex

Setting $w =y^{2/\alpha} (E + x^{2/\alpha} s)^{-1} t$, the above integral is bounded by 
\begin{align}
 C \int_0^\infty &
(E + x^{2/\alpha} s)^{-\alpha/2 - \gamma}
\int_{(E + x^{2/\alpha} s)^{-1}}^{  1 }
  (1-  v)^{-\gamma}  v^{-1 - \alpha/2}
g(s)
\, dv
 \, ds \notag
\\
&\le
C\left((1 - \gamma)^{-1} \int_{0}^\infty
(E + x^{2/\alpha} s)^{  -\alpha/2 - \gamma}
g(s)
 \, ds 
 +
\int_{0}^\infty
(E + x^{2/\alpha} s)^{  - \gamma}
g(s)
 \, ds \right)\label{intsplit}
\\
&\le
(1- \gamma)^{-1}E^{-\alpha/2 - \gamma} + E^{-\gamma}.\notag
\end{align}

\noindent To deduce the first bound of \eqref{intsplit}, we split the integral in $v$ at the point $v=1/2$ and bounded each piece separately; to deduce the second, we used the fact that $g$ is a probability density function. Together with \eqref{firsttt}, this completes the proof of \eqref{Gy1}. 

Next, we will find a lower bound on \eqref{2t}. Observe for $t \ge C$ that by \Cref{l:stabletailbounds} we have $\big( t g(t) \big)' \le - c t^{-1 - \alpha/2 }$. Then \eqref{2t} is lower bounded by 
\bex
  \frac{2c}{\alpha y }  \int_{0}^\infty
\int_{y^{-2/\alpha}}^{y^{-2/\alpha} ( E +  x^{2/\alpha} s )}
  (E  + x^{2/\alpha}  s - y^{2/\alpha}  t)^{-\gamma} 
 t^{-1 - \alpha/2}
g(s)
 \, ds \, dt.
\eex
for some $c > 0$. Again substituting $w =y^{2/\alpha} (E + x^{2/\alpha} s)^{-1} t$, the previous line becomes
\be\label{lowerboundme}
\frac{2c}{\alpha} \int_0^\infty
(E + x^{2/\alpha} s)^{-\alpha/2 - \gamma}
\int_{(E + x^{2/\alpha} s)^{-1}}^{  1 }
  (1-  v)^{-\gamma}  v^{-1 - \alpha/2}
g(s)
\, dv
 \, ds.
\ee

\noindent Since $E > c^{-1} \ge 2$, we have the lower bound
\begin{align*}
\int_{(E + x^{2/\alpha} s)^{-1}}^{  1 }
  (1-  v)^{-\gamma}  v^{-1 - \alpha/2}
\, dv
&\ge 
\int_{1/2}^{  1 }
  (1-  v)^{-\gamma}  v^{-1 - \alpha/2}
\, dv \ge c (1 - \gamma)^{-1} .
\end{align*}
Then we see that \eqref{lowerboundme} is lower bounded by
\begin{align*}
\frac{c}{\alpha(1- \gamma)} \int_0^\infty
(E + x^{2/\alpha} s)^{-\alpha/2 - \gamma}
g(s)
 \, ds
 &\ge 
\frac{c}{\alpha(1- \gamma)} \int_0^1
(E + x^{2/\alpha} s)^{-\alpha/2 - \gamma}
 \, ds \ge 
  \frac{c}{\alpha(1- \gamma)} E^{ - \alpha/2 - \gamma},
\end{align*}

\noindent where in the last bounds we used the fact that $g(s)$ is uniformly bounded below on the compact interval $[0, 1]$ (as $g(x) > 0$ there), and that $E + x^{2/\alpha} s \le E + 1 \le 2E$ for $x \le 1$ and $s \le 1$. Together with \eqref{firsttt}, this completes the proof of the lemma.
\end{proof}

\subsection{Derivatives in $E$}\label{s:dE}

In this section we establish the following three lemmas. The first bounds the derivatives in $E$ of the function $F_{\gamma}$, and the second and third bound those in $E$ of $G_{\gamma}$.

\begin{lem}\label{l:FE}
There exists a constant $c \in (0, 1)$ such that for all $\alpha \in (1-c, 1)$, $\gamma \in (0,1)$, $E > c^{-1}$, and $x, y \le 1$, we have 
\be\label{fefirst}
| \partial_E F_\gamma (E,x,y)| \le  c^{-1} ( 1- \gamma)^{-1} y E^{-1 - \alpha/2 - \gamma}.\ee
and
\be\label{fesecond}
\partial_E F_\gamma (E,x,y) \le - c ( 1- \gamma)^{-1} y E^{-1 - \alpha/2 - \gamma}
+ c^{-1} y E^{-1 - \alpha/2 - \gamma}.
\ee
\end{lem}
\begin{proof}
We recall from \eqref{FandG} that
\bex
F_\gamma(E, x,y) = 
 \int_{0}^\infty
\int_{y^{-2/\alpha} ( E +  x^{2/\alpha} s )}^\infty
 (- E  -  x^{2/\alpha}  s + y^{2/\alpha}  t)^{-\gamma}  g(s) g(t) \, dt\, ds.
\eex
We set $w = y^{2/\alpha} (E + x^{2/\alpha} s)^{-1} t$. Then 
\begin{align*}
F(E, x,y) &= 
y^{-2/\alpha} \int_{0}^\infty
(E + x^{2/\alpha} s)^{1 - \gamma}
\int_1^\infty
 (w- 1 )^{-\gamma}  g(s) g\big( w y^{-2/\alpha} (E+ x^{2/\alpha} s) \big) \, dw\, ds.
\end{align*}
When we differentiate in $E$, there are two contributions to $\partial_E F(E,x,y)$:
\be\label{FE1}
(1 - \gamma) y^{-2/\alpha} \int_{0}^\infty
(E + x^{2/\alpha} s)^{ - \gamma}
\int_1^\infty
 (w- 1 )^{-\gamma}  g(s) g\big( w y^{-2/\alpha} (E+ x^{2/\alpha} s) \big) \, dw\, ds,
\ee
and 
\be\label{FE2}
y^{-2/\alpha} \int_{0}^\infty
(E + x^{2/\alpha} s)^{1 - \gamma}
\int_1^\infty
 (w- 1 )^{-\gamma}  g(s) g'\big( w y^{-2/\alpha} (E+ x^{2/\alpha} s) \big)  w y^{-2/\alpha} \, dw\, ds.
\ee

We first bound \eqref{FE1}. Using the facts that $ w y^{-2/\alpha} (E+ x^{2/\alpha} s) > 1$ for $w \ge 1$ (which holds by the assumptions that $y\le 1$ and $E > c^{-1} \ge 1$) and that $g(x) \le C x^{-1 - \alpha/2}$ for $x \ge 1$ (by \Cref{l:stabletailbounds}), we see that \eqref{FE1} is bounded in absolute value by 
\begin{align}
	\label{estimatege0} 
	\begin{aligned}
&C(1-\gamma) y^{-2/\alpha} \int_{0}^\infty
(E + x^{2/\alpha} s)^{ - \gamma}
\int_1^\infty
 (w- 1 )^{-\gamma}  g(s) \big( w y^{-2/\alpha} (E+ x^{2/\alpha} s) \big)^{-1 - \alpha/2} \, dw\, ds\\
&\le C(1-\gamma) y  \int_{0}^\infty
(E + x^{2/\alpha} s)^{ -1 - \alpha/2 - \gamma} 
g(s) 
\int_1^\infty
 (w- 1 )^{-\gamma}    w^{-1 - \alpha/2}    \, dw\, ds\\
&\le Cy  \int_{0}^\infty
(E + x^{2/\alpha} s)^{ -1 - \alpha/2 - \gamma} 
g(s)\, ds\le  C y E^{-1 - \alpha/2 - \gamma},
\end{aligned}
\end{align}

\noindent where in the last estimate we used the facts that $g$ is a probability density function and that $(E + x^{2/\alpha} s)^{-1-\alpha/2-\gamma} \le E^{-1-\alpha/2-\gamma}$. 

Next, we bound \eqref{FE2} in absolute value using the fact (from \Cref{l:stabletailbounds}) that $|g'(x)| \le C x^{-2 - \alpha/2}$ for $x \ge 1$. This yields that \eqref{FE2} is bounded by
\begin{align}
	\label{estimatege2}
	\begin{aligned} 
C y^{-4 /\alpha} \int_{0}^\infty &(E + x^{2/\alpha} s)^{1 - \gamma}
\int_1^\infty
 (w- 1 )^{-\gamma}  w g(s) ( w y^{-2/\alpha} (E+ x^{2/\alpha} s) )^{-2 - \alpha/2}  \, dw\, ds\\
& = 
C y \int_{0}^\infty
(E + x^{2/\alpha} s)^{-1 - \alpha/2 - \gamma} g(s) 
\int_1^\infty
 (w- 1 )^{-\gamma}  w^{-1 - \alpha/2}  \, dw\, ds\\
 & \le  C (1 - \gamma)^{-1} y\int_{0}^\infty
(E + x^{2/\alpha} s)^{-1 -\alpha/2 -  \gamma} g(s)  \,ds \le C(1 - \gamma)^{-1} y E^{-1 - \alpha/2 - \gamma},
\end{aligned} 
\end{align}

\noindent where the last estimates again follows from the that fact $g$ is a probability density function and that $(E + x^{2/\alpha} s)^{-1-\alpha/2-\gamma} \le E^{-1-\alpha/2-\gamma}$.
Summing \eqref{estimatege0} and \eqref{estimatege2} verifies \eqref{fefirst}. 

To prove \eqref{fesecond}, we first recall (from \Cref{l:stabletailbounds}) that $g'(x) \le  -c  x^{-2 - \alpha/2}$ for $x \ge C$.
We then see that \eqref{FE2} is negative and upper bounded by 
\begin{align*}
-c & y^{-4 /\alpha} \int_{0}^\infty
(E + x^{2/\alpha} s)^{1 - \gamma}
\int_1^\infty
 (w- 1 )^{-\gamma}  w g(s) ( w y^{-2/\alpha} (E+ x^{2/\alpha} s) )^{-2 - \alpha/2}  \, dw\, ds\\
 &= 
- c y \int_{0}^\infty
(E + x^{2/\alpha} s)^{-1 - \alpha/2 - \gamma} g(s) 
\int_1^\infty
 (w- 1 )^{-\gamma}  w^{-1 - \alpha/2}  \, dw\, ds\\
 &=
- c(1- \gamma)^{-1} y \int_{0}^\infty
(E + x^{2/\alpha} s)^{-1 - \alpha/2 - \gamma} g(s) 
\, ds \le - c (1 - \gamma)^{-1} y E^{-1 - \alpha/2 - \gamma},
\end{align*}

\noindent following the same reasoning as in \eqref{estimatege2}. Summing with our estimate \eqref{estimatege0} for \eqref{FE1} completes the proof.
\end{proof}

\begin{lem}\label{l:GE}
There exists a constant $c \in \big(0, \frac{1}{2} \big)$ such that for all $\alpha \in (1-c, 1)$, $\gamma\in(0,1)$, $E > c^{-1}$, and $y \le 1$, we have 
\bex
\left| \partial_E G(E,x,y)  \right| \le c^{-1} E^{-1 - \gamma}\big(1 + (1-\gamma)^{-1} y E^{-\alpha/2}\big) .
\eex
\end{lem}
\begin{proof}
From \eqref{diffme2}, we have
\bex
G(E,x,y) = 
y^{-2/\alpha}\int_0^\infty  (E + x^{2/\alpha} s)^{1-\gamma} \int_0^1
 (1  -   v )^{-\gamma} 
 g\big(y^{-2/\alpha} (E + x^{2/\alpha} s) v\big) g(s) \, dv \, ds.
\eex
There are two contributions from differentiating in $E$. The first is 
\be\label{GE1}
( 1- \gamma) y^{-2/\alpha}\int_0^\infty  (E + x^{2/\alpha} s)^{-\gamma} \int_0^1
 (1  -   v )^{-\gamma} 
 g\big(y^{-2/\alpha} (E + x^{2/\alpha} s) v\big)g(s) \, dv\, ds.
\ee
The second is 
\be\label{GE2}
y^{-4 /\alpha}\int_0^\infty  (E + x^{2/\alpha} s)^{1-\gamma} \int_0^1
 (1  -   v )^{-\gamma} 
 g'\big(y^{-2/\alpha} (E + x^{2/\alpha} s) v\big) v g(s)  \, dv\, ds. 
\ee
We begin by bounding \eqref{GE1} in absolute value. We divide the interval of integration at $v = (E + x^{2/\alpha} s)^{-1} y^{2/\alpha}$. This results in two integrals. The contribution from $v < (E + x^{2/\alpha} s)^{-1} y^{2/\alpha}$ is 
\begin{flalign}
	\label{Gprev1}
	\begin{aligned}
( 1- \gamma) & y^{-2/\alpha} \int_0^\infty  (E + x^{2/\alpha} s)^{-\gamma} \int_0^{(E + x^{2/\alpha} s)^{-1} y^{2/\alpha}}
 (1  -   v )^{-\gamma} g\big(y^{-2/\alpha} (E + x^{2/\alpha} s) v\big) g(s) \, dv\, ds\\
 &\le C ( 1- \gamma) y^{-2/\alpha}\int_0^\infty  (E + x^{2/\alpha} s)^{-\gamma} \int_0^{(E + x^{2/\alpha} s)^{-1} y^{2/\alpha}}
 (1  -   v )^{-\gamma} g(s)  \, dv\, ds \\
 &\le C( 1- \gamma) \int_0^\infty  (E + x^{2/\alpha} s)^{-1 -\gamma} g(s) \, ds \le C( 1- \gamma) E^{ -1 - \gamma}. 
\end{aligned}
\end{flalign}

\noindent To deduce the first inequality, we used \Cref{l:stabletailbounds} to bound $g \big( y^{-2/\alpha} v (E + x^{2/\alpha})\big) \le C$; to deduce the second, we used the fact that $y^{2/\alpha} (E+x^{2/\alpha} s)^{-1} \le E^{-1} \le \frac{1}{2}$ (as $y \le 1$ and $E \ge c^{-1} \ge 2$); and to deduce the third we used the facts that $(E + x^{2/\alpha} s)^{-1-\gamma} \le E^{-1-\gamma}$ and that $g$ is a probability density function.

By similar reasoning, the contribution from $v > (E + x^{2/\alpha} s)^{-1} y^{2/\alpha}$ in \eqref{GE1} is 
\begin{align*}
( 1- &\gamma)y^{-2/\alpha}\int_0^\infty  (E + x^{2/\alpha} s)^{-\gamma} \int_{(E + x^{2/\alpha} s)^{-1} y^{2/\alpha}}^1
 (1  -   v )^{-\gamma}  g\big(y^{-2/\alpha} (E + x^{2/\alpha} s) v\big) g(s) \, dv\, ds \\
&\le
C( 1- \gamma) y^{-2/\alpha}\int_0^\infty  (E + x^{2/\alpha} s)^{-\gamma} g(s) \\
& \qquad \qquad \qquad \qquad \times \int_{(E + x^{2/\alpha} s)^{-1} y^{2/\alpha}}^1
 (1  -   v )^{-\gamma}  \big(y^{-2/\alpha} (E + x^{2/\alpha} s) v\big)^{-1 - \alpha/2}  \, dv\, ds \\
&\le
C( 1- \gamma) y\int_0^\infty  (E + x^{2/\alpha} s)^{-1 -\alpha/2 - \gamma} g(s) \int_{(E + x^{2/\alpha} s)^{-1} y^{2/\alpha}}^1
 (1  -   v )^{-\gamma}  v^{-1 - \alpha/2}  \, dv\, ds,
\end{align*}

\noindent where in the first inequality we used \Cref{l:stabletailbounds} to now bound $g (w) \le C w^{-1-\alpha/2}$ at $w = y^{-2/\alpha} v (E + x^{2/\alpha} s)$.  Restricting the integral to $v > 1/2$ gives
\begin{align}
	\label{integraly2} 
	\begin{aligned}
C ( 1- \gamma) & y\int_0^\infty  (E + x^{2/\alpha} s)^{-1 -\alpha/2 - \gamma} g(s) \int_{1/2}^1
 (1  -   v )^{-\gamma}  v^{-1 - \alpha/2}  \, dv\, ds\\
 &\le 
C( 1- \gamma) y\int_0^\infty  E ^{-1 -\alpha/2 - \gamma} g(s) (1 - \gamma)^{-1} \, ds =Cy E^{-1 - \alpha/2 - \gamma}.
\end{aligned} 
\end{align}

\noindent Integrating over the complementary region $v < 1/2$ yields
\begin{align}
	\label{integraly3} 
	\begin{aligned}
	C (1 & - \gamma) y\int_0^\infty  (E + x^{2/\alpha} s)^{-1 -\alpha/2 - \gamma} g(s) \int_{(E + x^{2/\alpha} s)^{-1} y^{2/\alpha}}^{1/2}
 (1  -   v )^{-\gamma}  v^{-1 - \alpha/2}  \, dv\, ds\\
&\le
C y\int_0^\infty  (E + x^{2/\alpha} s)^{-1 -\alpha/2 - \gamma} g(s) \int_{(E + x^{2/\alpha} s)^{-1} y^{2/\alpha}}^{1/2}v^{-1 - \alpha/2}  \, dv\, ds\\
&\le
C y\int_0^\infty  (E + x^{2/\alpha} s)^{-1 -\alpha/2 - \gamma} g(s)
\big( y^{2/\alpha} (E + x^{2/\alpha} s)^{-1} \big)^{-\alpha/2} ds \le
C \int_0^\infty  E^{-1  - \gamma} g(s)  ds = C  E^{ - 1 - \gamma}.
\end{aligned} 
\end{align} 

\noindent Summing \eqref{Gprev1}, \eqref{integraly2}, and \eqref{integraly3} yields 
\begin{flalign*}
	\Bigg| (1-\gamma) y^{-2/\alpha} \displaystyle\int_0^{\infty} (E + x^{2/\alpha} s)^{-\gamma} \displaystyle\int_0^1 & (1 - v)^{-\gamma} g \big( y^{-2/\alpha} (E + x^{2/\alpha} s) v \big) g(s) dv ds \Bigg| \\
	& \le C  E^{-1-\gamma} + Cy E^{-1-\alpha/2-\gamma},	
\end{flalign*}
which gives a bound for \eqref{GE1}.

We next examine \eqref{GE2} by again splitting the interval of integration at $v = (E + x^{2/\alpha} s)^{-1} y^{2/\alpha}$.
The first piece is 
\be\label{gefirst}
y^{-4 /\alpha}\int_0^\infty  (E + x^{2/\alpha} s)^{1-\gamma} \int_0^{(E + x^{2/\alpha} s)^{-1} y^{2/\alpha}}
 (1  -   v )^{-\gamma} 
 g'\big(y^{-2/\alpha} (E + x^{2/\alpha} s) v\big) v g(s)  \, dv\, ds. 
\ee
Using $| g'(x) | \le C$ from \Cref{l:stabletailbounds}, and $(E + x^{2/\alpha} s)^{-1} y^{2/\alpha} \le 1/2$, we bound this in absolute value by
\begin{align*}
&C y^{-4 /\alpha}\int_0^\infty  (E + x^{2/\alpha} s)^{1-\gamma} 
g(s) \int_0^{(E + x^{2/\alpha} s)^{-1} y^{2/\alpha}}
 (1  -   v )^{-\gamma}   v  \, dv\, ds\\
&\le C y^{-4 /\alpha}\int_0^\infty  (E + x^{2/\alpha} s)^{1-\gamma}
g(s) \int_0^{ (E + x^{2/\alpha} s)^{-1} y^{2/\alpha}} v  \, dv\,  ds\\
&\le  C y^{-4 /\alpha}\int_0^\infty  (E + x^{2/\alpha} s)^{1-\gamma}
g(s) ( (E + x^{2/\alpha} s)^{-1} y^{2/\alpha})^2  \, ds\\
&\le  C \int_0^\infty  (E + x^{2/\alpha} s)^{-1-\gamma}
g(s)   \, ds\\
&\le C E^{-1 - \gamma}.
\end{align*}

The second piece of \eqref{GE2} is 
\begin{align*}
&y^{-4 /\alpha}\int_0^\infty  (E + x^{2/\alpha} s)^{1-\gamma} \int_{(E + x^{2/\alpha} s)^{-1} y^{2/\alpha}}^1
 (1  -   v )^{-\gamma} 
 g'\big(y^{-2/\alpha} (E + x^{2/\alpha} s) v\big) v g(s)  \, dv\, ds\\
&\le Cy^{-4 /\alpha}\int_0^\infty  (E + x^{2/\alpha} s)^{1-\gamma} \int_{(E + x^{2/\alpha} s)^{-1} y^{2/\alpha}}^1
 (1  -   v )^{-\gamma} 
 \big(y^{-2/\alpha} (E + x^{2/\alpha} s) v\big)^{-2 - \alpha/2} v g(s)  \, dv\, ds\notag \\
&\le
Cy \int_0^\infty  (E + x^{2/\alpha} s)^{- 1-\alpha/2 - \gamma}g(s) \int_{(E + x^{2/\alpha} s)^{-1} y^{2/\alpha}}^1
 (1  -   v )^{-\gamma}  v^{-1 - \alpha/2}  \, dv\, ds.\notag
\end{align*}
The contribution from $v \in [1/2 ,1 ]$ is bounded by $C y (1- \gamma)^{-1}E^{-1- \alpha/2 - \gamma}$, after integrating in $s$. Further, for $v < 1/2$, we have 
\begin{align*}
&y \int_0^\infty  (E + x^{2/\alpha} s)^{- 1-\alpha/2 - \gamma}g(s) \int_{(E + x^{2/\alpha} s)^{-1} y^{2/\alpha}}^{1/2}
 (1  -   v )^{-\gamma}  v^{-1 - \alpha/2}  \, dv\, ds\\
&\le 
C y \int_0^\infty  (E + x^{2/\alpha} s)^{- 1-\alpha/2 - \gamma}g(s) \int_{(E + x^{2/\alpha} s)^{-1} y^{2/\alpha}}^{1/2} v^{-1 - \alpha/2}  \, dv\, ds\\
&\le 
C y \int_0^\infty  (E + x^{2/\alpha} s)^{- 1-\alpha/2 - \gamma}g(s) \big((E + x^{2/\alpha} s)^{-1} y^{2/\alpha} \big)^{-\alpha/2} \, ds\\
&\le 
C \int_0^\infty  (E + x^{2/\alpha} s)^{- 1-\gamma }g(s)\, ds\\
&\le C E^{-1 - \gamma}.
\end{align*}
This finishes the bound for \eqref{GE2} and completes the proof of the upper bound.
\end{proof}

\begin{lem}\label{l:GE2}
There exists $c>0$ such that for all $\alpha \in (1-c, 1)$, $E > c^{-1}$, $\gamma\in(0,1)$, and $x, y \le 1/2$, we have 
\bex
 \partial_E G_\gamma(E,x,y) \le - c \gamma  E^{-1-\gamma}  +  c^{-1} E^{-1 - \gamma - \alpha/4} .\eex
\end{lem}
\begin{proof}
From \eqref{diffme2}, we have
\bex
G(E,x,y) = 
y^{-2/\alpha}\int_0^\infty  (E + x^{2/\alpha} s)^{1-\gamma} \int_0^1
 (1  -   v )^{-\gamma} 
 g\big(y^{-2/\alpha} (E + x^{2/\alpha} s) v\big) g(s) \, dv \, ds.
\eex
There are two contributions from differentiating in $E$. The first is 
\be\label{GE11}
( 1- \gamma) y^{-2/\alpha}\int_0^\infty  (E + x^{2/\alpha} s)^{-\gamma} \int_0^1
 (1  -   v )^{-\gamma} 
 g\big(y^{-2/\alpha} (E + x^{2/\alpha} s) v\big)g(s) \, dv\, ds.
\ee
The second is 
\be\label{GE21}
y^{-4 /\alpha}\int_0^\infty  (E + x^{2/\alpha} s)^{1-\gamma} \int_0^1
 (1  -   v )^{-\gamma} 
 g'\big(y^{-2/\alpha} (E + x^{2/\alpha} s) v\big) v g(s)  \, dv\, ds. 
\ee
Making the change of variables $t = y^{-2/\alpha} (E + x^{2/\alpha} s ) v$ and summing \eqref{GE11} and \eqref{GE21}, we obtain
\begin{align}
\partial_E G(E,x,y) =& 
( 1- \gamma) \int_0^\infty  (E + x^{2/\alpha} s)^{-1-\gamma} \int_0^{y^{-2/\alpha} (E + x^{2/\alpha} s )}
\big(1  -   y^{2/\alpha} (E + x^{2/\alpha} s )^{-1} t \big)^{-\gamma} 
 g(t)g(s) \, dt\, ds\notag
\\
&+ 
\int_0^\infty  (E + x^{2/\alpha} s)^{-1-\gamma} \int_0^{y^{-2/\alpha} (E + x^{2/\alpha} s )}
 \big(1  -   y^{2/\alpha} (E + x^{2/\alpha} s )^{-1}t \big)^{-\gamma} 
 g'(t) t g(s)  \, dt\, ds\notag \\
=&
 - \gamma \int_0^\infty  (E + x^{2/\alpha} s)^{-1-\gamma} \int_0^{y^{-2/\alpha} (E + x^{2/\alpha} s )}
\big(1  -   y^{2/\alpha} (E + x^{2/\alpha} s )^{-1} t \big)^{-\gamma} 
 g(t)g(s) \, dt\, ds \label{iron1} \\
 & + \int_0^\infty  (E + x^{2/\alpha} s)^{-1-\gamma} \int_0^{y^{-2/\alpha} (E + x^{2/\alpha} s )}
\big(1  -   y^{2/\alpha} (E + x^{2/\alpha} s )^{-1} t \big)^{-\gamma} 
 \big(g(t)t \big)' g(s) \, dt\, ds.\label{iron2}
\end{align}
Considering \eqref{iron1}, we have from \Cref{l:stabletailbounds} that
\begin{align}
- &\gamma \int_0^\infty  (E + x^{2/\alpha} s)^{-1-\gamma} \int_0^{y^{-2/\alpha} (E + x^{2/\alpha} s )}
\big(1  -   y^{2/\alpha} (E + x^{2/\alpha} s )^{-1} t \big)^{-\gamma}g(t)g(s) \, dt\, ds\notag \\
&\le 
- \gamma \int_0^\infty  (E + x^{2/\alpha} s)^{-1-\gamma} \int_0^{y^{-2/\alpha} (E + x^{2/\alpha} s )}
g(t)g(s) \, dt\, ds\notag \\
& - c \gamma  E^{-1-\gamma}  \int_0^{y^{-2/\alpha} (E + x^{2/\alpha} s )}
g(t) \, dt \le  - \frac{c \gamma  }{2} E^{-1-\gamma},\label{mercury4}
\end{align} 
if $E >C_0$, for some constants $C_0 > 1$ and $c>0$.

For the term \eqref{iron2}, we divide the integral in $t$ into the sum of integrals of $[0,M]$ and $\big[M, y^{-2/\alpha} (E + x^{2/\alpha} s )\big]$, where $M$ is a parameter to be chosen later. Suppose that $M \in (C_1, E)$, where $C_1$ is the constant given by \Cref{l:stabletailbounds}. For the integral over $[0,M]$, we set
\bex
u = \big(1  -   y^{2/\alpha} (E + x^{2/\alpha} s )^{-1} t \big)^{-\gamma}, \qquad w = g(t) t,
\eex
and use integration by parts in the form 
\bex
\int u\, dw =  uw - \int w\, du. 
\eex
Using the previous line in \eqref{iron2}, and noting that the boundary term at $0$ vanishes, we get 
\begin{align}\label{mercury3}
\int_0^{M}&
\big(1  -   y^{2/\alpha} (E + x^{2/\alpha} s )^{-1} t \big)^{-\gamma} 
 \big(g(t)t \big)'  \, dt\\
 =& \big(1  -   y^{2/\alpha} (E + x^{2/\alpha} s )^{-1} M \big)^{-\gamma} \cdot g(M) M\notag \\
 & -  \gamma \int_0^{M} \big(1  -   y^{2/\alpha} (E + x^{2/\alpha} s )^{-1} t \big)^{-1-\gamma} y^{2/\alpha} (E + x^{2/\alpha} s )^{-1}    g(t) t \, dt.\notag\\
 \le& \big(1  -   y^{2/\alpha} (E + x^{2/\alpha} s )^{-1} M \big)^{-\gamma} \cdot g(M) M.\notag
\end{align}
Under the assumptions on $M$ and $y$, and using \Cref{l:stabletailbounds}, there exists $C>0$ such that
\be\label{mercury1}
\big(1  -   y^{2/\alpha} (E + x^{2/\alpha} s )^{-1} M \big)^{-\gamma} \cdot g(M) M \le  C M^{-\alpha/2}.
\ee
Then inserting \eqref{mercury1} into \eqref{mercury3}, we see that 
\begin{align}\notag
 &\int_0^\infty  (E + x^{2/\alpha} s)^{-1-\gamma} \int_0^{M}
\big(1  -   y^{2/\alpha} (E + x^{2/\alpha} s )^{-1} t \big)^{-\gamma} 
 \big(g(t)t \big)' g(s) \, dt\, ds\\
 &\le C M^{-\alpha/2} \int_0^\infty  (E + x^{2/\alpha} s)^{-1-\gamma} g(s)\, ds 
\le  C M^{-\alpha/2} E^{-1 - \gamma}.\label{mercury5}
\end{align}
Next, we consider the integral in $v$ over $\big[M, y^{-2/\alpha} (E + x^{2/\alpha} s )\big]$ in \eqref{iron2}. Again using \Cref{l:stabletailbounds} to bound $\big( g(t) t \big)'$, and setting $t = y^{-2/\alpha} (E + x^{2/\alpha} s ) v$, we obtain 
\begin{align*}
&\left|\int_0^\infty  (E + x^{2/\alpha} s)^{-1-\gamma}\int_M^{y^{-2/\alpha} (E + x^{2/\alpha} s )}
\big(1  -   y^{2/\alpha} (E + x^{2/\alpha} s )^{-1} t \big)^{-\gamma} 
 \big(g(t)t \big)' g(s) \, dt\, ds\right|\\
& \le  
C \int_0^\infty  (E + x^{2/\alpha} s)^{-1-\gamma}\int_M^{y^{-2/\alpha} (E + x^{2/\alpha} s )}
\big(1  -   y^{2/\alpha} (E + x^{2/\alpha} s )^{-1} t \big)^{-\gamma} 
 t^{-1 - \alpha/2}  g(s) \, dt\, ds\\
 &= 
C y \int_0^\infty  (E + x^{2/\alpha} s)^{-1-\gamma - \alpha/2}\int_{My^{2/\alpha} (E + x^{2/\alpha} s )^{-1}}^{1}
(1  -   v  )^{-\gamma} 
 v^{-1 - \alpha/2}  g(s) \, dv\, ds\\
 &\le 
C y \int_0^\infty  (E + x^{2/\alpha} s)^{-1-\gamma - \alpha/2} \big(M^{-\alpha/2} y^{-1} (E + x^{2/\alpha} s )^{\alpha/2}   \big)g(s)\, ds
\le C M^{-\alpha/2} E^{-1 - \gamma}.
\end{align*}
We complete the proof by taking $M= E^{1/2}$, and combining the previous line with \eqref{mercury4} and \eqref{mercury5}.
\end{proof}

\subsection{Derivatives of $a(E)$ and $b(E)$}\label{s:dAB}

\begin{lem}\label{uniqueness}
There exists a constant $c \in (0, 1)$ such that, for all $\alpha \in (1-c, 1)$ and $E > c^{-1}$, the  function $E \mapsto \big(a(E), b(E)\big)$ is differentiable.
\end{lem}

\begin{proof}
Consider the function $H(E,x, y)$ defined by 
\bex
H(E,x,y) = \left( F_{\alpha/2}(E,x,y) - x, G_{\alpha/2}(E,x,y) - y  \right).
\eex
By \Cref{r:abfixedpointeqn}, we have $H(E, a, b) = 0$ for all $E \in \bbR$. 
The Jacobian matrix of $H$ is given by
\begin{align*}
&\begin{bmatrix}
\partial_E (F(E,x,y) - x)  & \partial_x (F(E,x,y) - x) & \partial_y (F(E,x,y) - x) \\
\partial_E (G(E,x,y) - y) &\partial_x (G(E,x,y) - y)  & \partial_y (G(E,x,y) - y) 
\end{bmatrix}\\
&=\begin{bmatrix}
\partial_E F(E,x,y)   & \partial_x F(E,x,y) - 1 & \partial_y F(E,x,y)  \\
\partial_E G(E,x,y) &\partial_x G(E,x,y)  & \partial_y G(E,x,y) - 1
\end{bmatrix}.
\end{align*}
Using \Cref{l:abasymptotic}, \Cref{l:Fx}, \Cref{l:Fy}, \Cref{l:Gx}, and \Cref{l:Gy}, there exists $c>0$ such that the right $2\times 2$ submatrix evaluated at $(E, x, y) = (E, a(E), b(E))$ can be written as 
\bex
\begin{bmatrix}
O(E^{-1/9}) - 1 & O(E^{-1/9})  \\
O(E^{-1/9}) & O(E^{-1/9}) - 1
\end{bmatrix}
\eex
for $\alpha \in (1-c, 1)$.
Hence this matrix is invertible for $E \ge c^{-1}$, after possibly decreasing $c$. The conclusion now follows from the implicit function theorem.
\end{proof}
\begin{lem}\label{l:abE}
There exists $c>0$ such that for all $\alpha \in (1-c, 1)$ and $E > c^{-1}$, 
\be\label{Ederiv1}
\big|a'(E)\big| \le c^{-1} E^{- 1 - 3\alpha/2},\qquad \big|b'(E)\big| \le c^{-1} E^{- 1 - \alpha/2}.
\ee
\end{lem}
\begin{proof}
Using $a(E) = F_{\alpha/2}\big(E, a(E), b(E)\big)$, we have
\begin{align*}
a'(E) &= \partial_E F_{\alpha/2} \big(E,a(E), b(E)\big) 
\\ &+  \partial_x F_{\alpha/2}\big(E, a(E), b(E)\big) a'(E) 
\\ &+ \partial_y F_{\alpha/2}\big(E, a(E), b(E)\big) b'(E).
\end{align*}
Using \Cref{l:abasymptotic}, \Cref{l:Fx}, \Cref{l:Fy}, and \Cref{l:FE}, the previous line yields
\be\label{peach1}
\left(1 + O(E^{- 1 - 3\alpha/2}) \right) a'(E)
+ O( E^{-\alpha}) b'(E) 
= O(E^{- 1 - 3\alpha/2}).
\ee
Similarly expanding $b'(E) = \partial_E \Big(G_{\alpha/2} \big(E,a(E), b(E)\big) \Big) $ using the chain rule and applying \Cref{l:abasymptotic}, \Cref{l:Gx}, \Cref{l:Gy}, and \Cref{l:GE}, we obtain
\be\label{peach2}
O(E^{-\alpha}) a'(E) + \left(1  + O(E^{-\alpha/2})   \right) b'(E) = O (E^{-1 - \alpha/2}).
\ee
The equations \eqref{peach1} and \eqref{peach2} can be written as the system
\bex
\begin{bmatrix}
1 + O(E^{- 1 - 3\alpha/2})  & O( E^{-\alpha})\\
O(E^{-\alpha})  & 1  + O(E^{-\alpha/2})
\end{bmatrix}
\begin{bmatrix}
a'(E) \\
b'(E) 
\end{bmatrix}
=
\begin{bmatrix}
O(E^{- 1 - 3\alpha/2}) \\
O (E^{-1 - \alpha/2})
\end{bmatrix}.
\eex
We invert this system to obtain
\begin{align*}
\begin{bmatrix}
a'(E) \\
b'(E) 
\end{bmatrix}
&=
\frac{1}{ 1 + O(E^{-\alpha/2})} 
\begin{bmatrix}
1  + O(E^{-\alpha/2})  & O( E^{-\alpha})\\
O(E^{-\alpha})  & 1 + O(E^{- 1 - 3\alpha/2}) 
\end{bmatrix}
\begin{bmatrix}
O(E^{- 1 - 3\alpha/2}) \\
O (E^{-1 - \alpha/2})
\end{bmatrix}\\
&=\begin{bmatrix}
O(E^{- 1 - 3\alpha/2}) \\
O (E^{-1 - \alpha/2})
\end{bmatrix}.
\end{align*}
This shows \eqref{Ederiv1}.
\end{proof}
We next show that $b'(E)$ is negative for sufficiently large $E$. 
\begin{lem}\label{l:bprimenegative}
There exists $c>0$ such that for $\alpha \in (1-c, 1)$,
the following holds. Then for every $D>1$, there exists $C(D) >1$ such that 
\bex
 b'(E) \le 0.
\eex
for all $E$ such that $D^{-1}(1 - \alpha)^{-1} \le  E \le D(1-\alpha)^{-1}$  and $E \ge C(D)$.
\end{lem}
\begin{proof}
Differentiating $b(E) = G_{\alpha/2}\big(E, a(E), b(E)\big)$, where the equality follows from \Cref{r:abfixedpointeqn}, we have
\begin{align}
b'(E)  &= \partial_E G_{\alpha/2} \big(E,a(E), b(E)\big) \label{dassum}
\\ &+  \partial_x G_{\alpha/2}\big(E, a(E), b(E)\big) a'(E)\notag\\
&+ \partial_y G_{\alpha/2}\big(E, a(E), b(E)\big) b'(E).\notag
\end{align}
Using \Cref{l:GE2}, \eqref{Fgammabound} to bound $b(E)= G_{\alpha/2}\big(E, a(E), b(E)\big)$, and the assumption on $E$, we get 
\be\label{ledt}
\partial_E G_{\alpha/2} \big(E,a(E), b(E)\big)
\le - c E^{ -1 - \alpha/2}.
\ee
Further using 
\eqref{Fgammabound},  \Cref{l:Gx}, and \Cref{l:Gy}, we obtain 
\begin{align*}
\left| \partial_x G_{\alpha/2}\big(E, a(E), b(E)\big) a'(E)\right|
\le C E^{-\alpha} \cdot E^{-1 - 3\alpha/2},
\end{align*}
and 
\bex
\left|\partial_y G_{\alpha/2}\big(E, a(E), b(E)\big) b'(E)\right|
\le  E^{-\alpha/2} \cdot E^{-1 - \alpha/2}.
\eex
Combining these bounds completes the proof after choosing $C(D)$ sufficiently large and using $E > C(D)$, since \eqref{ledt} represents the leading order term in the sum \eqref{dassum} for $\alpha > 3/4$. 
\end{proof}

\subsection{Derivatives of Fractional Powers of the Resolvent}\label{s:dtildeAB}
We begin by defining functions $\tilde a (E)$ and $\tilde b (E)$. 
\begin{definition}
Let $S$ and $T$ be nonnegative $\alpha/2$-stable random variables.
We define
\begin{align}\label{abtildedef}
\tilde a(E) &= \E \left[ \left( E + a(E)^{2/\alpha} S - b(E)^{2/\alpha} T   \right)_-^{-\alpha}  \right],\\
\tilde b(E) &= \E \left[ \left( E + a(E)^{2/\alpha} S - b(E)^{2/\alpha} T   \right)_+^{-\alpha}  \right].\notag
\end{align}
\end{definition}
We remark that the definitions of $F_\gamma$ and $G_\gamma$ in \eqref{FandG} imply that
\be\label{tildedefs}
\tilde a(E) = F_{\alpha}(E, a(E), b(E)), \qquad \tilde b(E) = B_{\alpha}(E, a(E), b(E)).
\ee

\noindent We now derive bounds on the derivatives of $\tilde a(E)$ and $\tilde b(E)$. 

\begin{lem}\label{l:tildeaE}
There exists $c>0$ such that for all $\alpha \in (1-c, 1)$,
the following holds. 
For every $D>1$, there exists $c_1(D) >0$ such that 
\bex
\tilde a'(E) \le - c_1 E^{-2\alpha}.
\eex
for all $E$ such that $D^{-1}(1 - \alpha)^{-1} \le  E \le D(1-\alpha)^{-1}$   and $E \ge c_1^{-1}$.
\end{lem}
\begin{proof}
We have 
 $\tilde a(E) = F_{\alpha}(E, a(E), b(E))
$ from \eqref{tildedefs}. Then from the chain rule, we obtain
\begin{align}
\tilde a'(E) &= \partial_E F_{\alpha} (E,a(E), b(E)) \label{onion1}
\\ &+  \partial_x F_{\alpha}(E, a(E), b(E)) a'(E) \notag
\\ &+ \partial_y F_{\alpha}(E, a(E), b(E)) b'(E).\notag
\end{align}
Using  \eqref{fesecond}, $D (1 - \alpha)^{-1} \ge  E $, and \eqref{Fgammabound} to bound $b(E)= G_{\alpha/2}\big(E, a(E), b(E)\big)$, we have
\begin{align}
 \partial_E F_\gamma (E,x,y) &\le  -c ( 1- \alpha)^{-1} b(E) E^{-1 -3 \alpha/2} + C b(E) E^{-1 - 3 \alpha/2} \notag \\
&\le - c E^{-2 \alpha} + C E^{-1 - 2\alpha}\label{onion2}
\end{align}
for some constants  $c(D) > 0$ and $C(D)>1$.
Further, using \Cref{l:Fx} and \Cref{l:abE}, we have
\be\label{onion3}
|\partial_x F_{\alpha}(E, a(E), b(E)) a'(E) | 
\le C (1 - \alpha)^{-1} E^{-1 - 3\alpha/2}  \cdot E^{-1 - 3\alpha/2}.
\ee
Finally, we observe that
\be\label{onion4}
\partial_y F_{\alpha}(E, a(E), b(E)) b'(E) \le 0
\ee
as a consequence of \eqref{Fyispositive} and \Cref{l:bprimenegative}.
Combining \eqref{onion1}, \eqref{onion2}, \eqref{onion3}, and \eqref{onion4} completes the proof.
\end{proof}

\begin{lem}\label{l:tildebE}
There exists $c>0$ such that for all $\alpha \in (1-c, 1)$,
the following holds. 
For every $D>1$, there exists $c_1(D) >0$ such that 
\bex
\tilde b'(E) \le - c_1 E^{-3/2 - \alpha/2}.
\eex
for all $E$ such that $D^{-1}(1 - \alpha)^{-1} \le  E \le D(1-\alpha)^{-1}$  and $E \ge c_1^{-1}$.
\end{lem}
\begin{proof}
Using \eqref{tildedefs}, we have
$
\tilde b(E) = G_{\alpha}(E, a(E), b(E))$. 
Then the chain rule gives
\begin{align}
\tilde b'(E) &= \partial_E G_{\alpha} (E,a(E), b(E)) \label{starfruit3}
\\ &+  \partial_x G_{\alpha}(E, a(E), b(E)) a'(E) \notag
\\ &+ \partial_y G_{\alpha}(E, a(E), b(E)) b'(E).\notag
\end{align}
Using the assumption on $E$ and 
\eqref{Fgammabound} to bound $b(E)= G_{\alpha/2}\big(E, a(E), b(E)\big)$, and \Cref{l:GE2}, we have
\begin{align}
 \partial_E G\big(E, a(E),b(E)\big) &\le   - c E^{-1  - \alpha} +  C   E^{-1 - \alpha - \alpha/4 }, \label{starfruit1}
\end{align}
and from \Cref{l:Gx} and  \Cref{l:abE}, we have
\begin{align}
|\partial_x G_{\alpha}(E, a(E), b(E)) a'(E) | 
&\le C  \left( E^{1 - 3\alpha/2} + E^{2 - 3\alpha} + E^{-3\alpha/2} \right)   \cdot E^{-1 - 3\alpha/2}\le CE^{-3 \alpha}.\label{starfruit2}
\end{align}
Finally, from \Cref{l:Gy}, we have 
\bex
\partial_y G_{\alpha}(E,a(E), b(E)) \ge c ( 1 - \alpha)^{-1} E^{-\alpha -1/2 } + O(E^{ -1/2 - \alpha/2}) \ge 0,
\eex
and from \Cref{l:bprimenegative} we have $b'(E) \le 0$, which implies $\partial_y G_{\alpha}(E, a(E), b(E)) b'(E) \le 0$. Combining this inequality with \eqref{starfruit3}, \eqref{starfruit1}, and \eqref{starfruit2} completes the proof.
\end{proof}

\subsection{Uniqueness of the Mobility Edge}\label{s:proveuniqueness}

We are now ready to prove the first part of \Cref{t:main2}.

\begin{proof}[Proof of \Cref{t:main2}(1)]

By \Cref{l:crudeasymptotic}, there exists $c>0$ such that  any solution to $\lambda(E,\alpha) = 1$ satisfies
\be\label{therange}
c ( 1- \alpha)^{-1} \le  E \le c^{-1} (1- \alpha)^{-1}.
\ee
Further, we have $\lambda(E,\alpha) > 1$ for $0 < E < c(1 - \alpha)^{-1}$ and $\lambda(E,\alpha) < 1$  for $E > c^{-1} (1-\alpha)^{-1}$, so there is at least one $E$ in the range \eqref{therange} such that $\lambda(E,\alpha) = 1$ (using the continuity of $E \mapsto \lambda(E,\alpha)$ provided by the third part of \Cref{l:lambdalemma}). Therefore, it suffices to show that $E\mapsto \lambda(E,\alpha)$ is strictly decreasing in \eqref{therange}, since this implies there is exactly one solution to $\lambda(E,\alpha) = 1$.

For brevity, in this proof  we write $F_\alpha = F_\alpha(E) = F_\alpha(E,a(E), b(E))$ and $G_\alpha = G_\alpha(E) = G_{\alpha} (E,a(E), b(E))$. We recall from \eqref{kiwi3} that
		\begin{flalign}\label{qkiwi}
		\lambda(E,\alpha) 
			& =  \pi^{-1} \cdot K_{\alpha} \cdot \Gamma (\alpha) \cdot \bigg(  t_{\alpha} \sqrt{1 - t_{\alpha}^2} \cdot (F_\alpha + G_\alpha)  \\
			& \qquad + \sqrt{1 - t^2_\alpha} \cdot \sqrt{ 2 t_\alpha^2 (F_\alpha + G_\alpha)^2 + (1 - t_{\alpha}^2) ( F_\alpha^2 + G_\alpha^2)} \bigg).
		\end{flalign}
		
Recalling that $\tilde a(E) = F_\alpha(E)$ and $\tilde b(E) = G_\alpha(E)$ (see \eqref{tildedefs}), and using \Cref{l:tildeaE} and \Cref{l:tildebE}, we find that $F_\alpha(E)$ and $G_\alpha(E)$ are decreasing for $E$ in the region \eqref{therange}. Hence $(F_\alpha + G_\alpha)^2$ and $F_\alpha^2 + G_\alpha^2$ are also decreasing in this range, and we conclude from \eqref{qkiwi} that $\lambda(E,\alpha)$ is decreasing for such $E$ (since all coefficients in \eqref{qkiwi} are positive). We noted previously that this claim is enough to establish the theorem, so the proof is complete.
\end{proof}

	\section{Scaling Near Zero} 
	
	\label{Alpha0Scaling}
	
	In this section we analyze the how any solution to the equation $\lambda (E, \alpha) = 1$ (recall \Cref{lambdaEalpha}) scales when $\alpha$ is small. Throughout, we assume that $\alpha \in \big( 0, \frac{1}{20} \big)$, even when not explicitly stated. We recall from \eqref{FandG} that, for any real numbers $\gamma \in (0, 1)$ and $E, x, y > 0$, we define
	\begin{flalign}
		\label{fggamma} 
		F_{\gamma} (E, x, y) = \mathbb{E} \Big[ (E + x^{2/\alpha} S - y^{2/\alpha} T)_-^{\gamma} \Big]; \qquad G_{\gamma} (E, x, y) = \mathbb{E} \Big[ (E+x^{2/\alpha} S - y^{2/\alpha} T)_+^{-\gamma} \Big],
	\end{flalign}
	
	\noindent where $S$ and $T$ are positive $\frac{\alpha}{2}$-stable laws. We further recall from \eqref{r:abfixedpointeqn} that, under this notation, $(a, b) = \big( a(E), b(E) \big)$ (recall \eqref{opaque}) solves the system
	\begin{flalign}
		\label{fba} 
		a = F_{\alpha/2} (E, a, b); \qquad b = F_{\alpha/2} (E, a, b).
	\end{flalign}
	
	\noindent It will be convenient to parameterize $E = u^{2/\alpha}$ and view $a$, $b$, $F$, and $G$ as functions of $u$; we will do this throughout, often abbreviating $a = a(u^{2/\alpha})$ and $b = b(u^{2/\alpha})$ without comment. We begin in \Cref{H0Halpha} by explaining how, for $\alpha$ small, $\frac{\alpha}{2} \log S$ behaves as a Gumbel random variable \cite{SDCS}. We then provide some initial estimates on $a$ and $b$ in \Cref{InitialabAlpha0}. In \Cref{FFGG} we state an estimate for the error in replacing $(F, G)$ with more explicit quantities (in terms of Gumbel random variables), which is proven in \Cref{Replaceh0halpha20} and \Cref{FFGGalpha0Estimate0}. We then establish the scaling for the mobility edge in \Cref{Alpha0E0Scale}. Throughout this section, constants $c > 0$ and $C > 1$ will be independent of $\alpha$.
	
	\subsection{Approximation by Gumbel Random Variables}
	
	\label{H0Halpha}
	
	In this section, we quantify the sense in which the logarithm of a positive $\frac{\alpha}{2}$-stable law can be approximated by a Gumbel random variable. To that end, we begin with the following definition, where $\Upsilon_{\alpha/2}$ below measures this error. 
	
	\begin{definition}
		
		\label{halpha}
		
		Fix $\alpha \in (0, 1)$. Let $S$ be a positive $\frac{\alpha}{2}$-stable law; let $W_{\alpha/2} = \frac{\alpha}{2} \log S$; and let $h_{\alpha/2} (x)$ denote the probability density function of $W_{\alpha/2}$. Further let $W_0$ denote a Gumbel random variable, and let $h_0 (x) = e^{-x - e^{-x}}$ denote its probability density function. Define $\Upsilon_{\alpha/2} (x) : \mathbb{R} \rightarrow \mathbb{R}$ so that $h_{\alpha/2} (x) = h_0 (x) + \alpha^2 \cdot \Upsilon_{\alpha/2} (x)$.
	\end{definition}
	
	We begin with the following lemma providing the characterstic functions of $W_{\alpha/2}$ and $W_0$.
	
	\begin{lem}
		
		\label{eitw0walpha2}
		
		For any $\alpha \in (0, 1)$ and $t \in \mathbb{R}$, we have 
		\begin{flalign*}
			\mathbb{E} [e^{\mathrm{i} t W_{\alpha/2}}] = \displaystyle\frac{\Gamma (1-\mathrm{i} t)}{\Gamma \big( 1 - \frac{\mathrm{i} t \alpha}{2} \big)} \cdot \Gamma \Big(1 - \displaystyle\frac{\alpha}{2} \Big)^{\mathrm{i} t}; \qquad \mathbb{E} [ e^{\mathrm{i} t W_0}] = \Gamma (1 - \mathrm{i}t).
		\end{flalign*}
	\end{lem}
	
	\begin{proof}

		The expression for $\mathbb{E} [ e^{\mathrm{i} t W_0}]$ follows from direct integration, so we omit the proof. For the characteristic function of $W_{\alpha/2}$, we recall the identity
		$$
		x^{-k} = \frac{1}{\Gamma(k)}\int_0^\infty e^{- t x} t^{k-1} dt,
		$$
		which holds for all $x >0$ and $k >0$,  
		Using \eqref{ytsigma}, we deduce that for real $k > 0$,
		\begin{eqnarray*}
			\mathbb{E} [e^{- k W_{\alpha/2} }] = \mathbb{E} [S^{-k\alpha/2}] &= &\frac{1}{\Gamma(k\alpha/2)}\int_0^\infty \mathbb{E} [e^{- t S} t^{k\alpha/2 - 1}]\, dt \\
			& = &  \frac{1}{\Gamma(k\alpha/2)}\int_0^\infty e^{- \Gamma(1-\alpha/2) t^{\alpha/2}} t^{k\alpha/2-1} dt  \\
			& = & \frac{2}{\alpha\Gamma(k\alpha/2)\Gamma(1-\alpha/2)^k}\int_0^\infty e^{- t} t^{k -1} dt \\
			& = &\frac{2\cdot \Gamma(k)}{\alpha\Gamma(k\alpha/2)\Gamma(1-\alpha/2)^k}.
		\end{eqnarray*}
		Recall that $z \Gamma(z) = \Gamma(z+1)$. Then the previous display gives
		\begin{equation}\label{eq:negmomS}
			\mathbb{E} [e^{-k W_{\alpha/2}}] = \mathbb{E} [S^{-k\alpha/2}] = \frac{\Gamma(1+k)}{\Gamma(1+k\alpha/2)\Gamma(1-\alpha/2)^k}.
		\end{equation}
		For all real $t$, \eqref{eq:negmomS} implies by analytic continuation that
		$$
		\mathbb{E} \left[ e^{\iu t W_{\alpha/2}} \right]    
		= 
		\frac{\Gamma (1-\mathrm{i} t) \cdot \Gamma(1-\alpha/2)^{\iu t}}{\Gamma(1-\iu t \alpha/2)},
		$$
		which completes the proof.
	\end{proof}

	Lemma \ref{eitw0walpha2} indicates that, as $\alpha$ tends to $0$, the characteristic function of $W_{\alpha/2}$ converges to that of $W_0$; in particular, $W_{\alpha/2}$ converges weakly to a Gumbel distribution.\footnote{In fact, it suggests that $W_{\alpha/2}$ can be written exactly (for any $\alpha \in (0, 1)$) as the sum of a Gumbel distribution and another independent random variable; this is indeed shown to be the case as \cite[Corollary 4.1]{SDCS}.} The following lemma quantifies this convergence by bounding (derivatives of) $\Upsilon_{\alpha/2}$.	
	
	\begin{lem}
		
		\label{w0walpha} 
		
		For any positive integers $p, q \ge 0$, there exists a constant $C = C(p, q) > 1$ such that
		\begin{flalign}
			\label{h0halpha21}
			\displaystyle\sup_{x \in \mathbb{R}} \big| \partial_x^q \Upsilon_{\alpha/2} (x) \big| \le C; \qquad \displaystyle\int_{-\infty}^{\infty} |x|^p \cdot \big|\partial_x^q \Upsilon_{\alpha/2} (x) \big| dx < C.
		\end{flalign}
		
	\end{lem} 
	
	\begin{proof}

		\noindent Recall for any $x \in \mathbb{R}$ that  
		\begin{flalign*}
			\big|\Gamma (1 - \mathrm{i} x) \big|^2 = \displaystyle\frac{\pi x}{\sinh(\pi x)},
		\end{flalign*}
		
		\noindent which implies that 
		\begin{flalign}
			\label{estimate1it0}
			\Bigg| \displaystyle\frac{\Gamma (1 - \mathrm{i} t)}{\Gamma \big( 1 - \frac{\mathrm{i} t \alpha}{2} \big)} \Bigg| = \bigg( \displaystyle\frac{2}{\alpha} \cdot \displaystyle\frac{\sinh \big (\frac{\pi \alpha t}{2}\big)}{\sinh (\pi t)} \bigg)^{1/2}.
		\end{flalign}
		
		\noindent These, together with a Taylor expansion, yields
		\begin{flalign*}
			\Bigg| \Gamma \bigg(1 - \frac{\alpha}{2} \bigg)^{\mathrm{i}t} - \Gamma \bigg( 1 - \frac{\mathrm{i} t \alpha}{2} \bigg)  \Bigg| \le C \alpha^2 \big( |t| + |t|^2 \big),
		\end{flalign*}
		
		\noindent uniformly for $t \in \mathbb{R}$. Thus, \Cref{eitw0walpha2} and Fourier inversion gives
		\begin{flalign*}
			\alpha^2 \cdot \displaystyle\sup_{x \in \mathbb{R}} \big| \partial_x^q \Upsilon_{\alpha/2} (x) \big| = \displaystyle\sup_{x \in \mathbb{R}} \big| \partial_x^q h_{\alpha/2} (x) - \partial_x^q h_0 (x) \big| \le \displaystyle\int_{-\infty}^{\infty} t^q \big| \mathbb{E} [ e^{\mathrm{i} t W_0}] - \mathbb{E} [ e^{\mathrm{i} t W_{\alpha/2}}] \big| dt \le C \alpha^2,
		\end{flalign*} 
		
		\noindent which verifies the first bound in \eqref{h0halpha21}.  
		
		To establish the latter, denoting the digamma function $\psi_k = \partial_x^k \big( \log \Gamma(x) \big)$; recall from Section 1.1 of \cite[Section 1.1]{TSF} that for any integer $k \in [0, p+q]$ we have the estimates
		\begin{flalign*}
			\big| \psi_0 (1 + \mathrm{i} x) - \log (1 + \mathrm{i} x) \big| \le \displaystyle\frac{C}{|x| + 1}; \qquad \bigg| \psi_k (1 + \mathrm{i} x) + (-1)^k \displaystyle\frac{k!}{(1 + \mathrm{i} x)^k} \bigg| \le \displaystyle\frac{C}{\big( |x| + 1 \big)^k}, 
		\end{flalign*}
		
		\noindent uniformly in $x \in \mathbb{R}$. Together with \eqref{estimate1it0}, these yield 
		\begin{flalign}
			\label{estimate1it1} 
			\Bigg| \partial_t^k \bigg( \displaystyle\frac{\Gamma (1 - \mathrm{i} t)}{\Gamma \big( 1 - \frac{\mathrm{i}t\alpha}{2} \big)} \bigg) \Bigg| \le C \log \big( |t| + 3 \big)^k \cdot \bigg( \displaystyle\frac{2}{\alpha} \cdot \displaystyle\frac{\sinh \big( \frac{\pi \alpha t}{2} \big)}{\sinh (\pi t)} \bigg)^{1/2}.
		\end{flalign}
		
		\noindent From a Taylor expansion (and again \eqref{estimate1it0}), we also have 
		\begin{flalign}
			\label{estimate1it2}
			\Bigg| \partial_t^k \bigg( \Gamma \Big( 1 - \displaystyle\frac{\alpha}{2} \Big)^{\mathrm{i} t} - \Gamma \Big( 1 - \displaystyle\frac{\mathrm{i} t \alpha}{2} \Big) \bigg) \Bigg| \le C \alpha^2 \big( |t|^k + |t| \big).
		\end{flalign}
		
		\noindent Together \Cref{eitw0walpha2}, Fourier inversion, \eqref{estimate1it1}, and \eqref{estimate1it2} yield
		\begin{flalign*}
			\alpha^2 \displaystyle\int_{-\infty}^{\infty} |x|^p \cdot \big| \partial_x^q \Upsilon_{\alpha/2} (x) \big| dx & = \displaystyle\int_{-\infty}^{\infty} |t|^p \cdot \big| \partial_x^q h_{\alpha/2} (x) - \partial_x^q h_0 (x) \big| \\
			& \le C \displaystyle\sum_{0 \le k+l \le p+q} \displaystyle\int_{-\infty}^{\infty} |t|^l \cdot \Big| \partial_t^k \big( \mathbb{E} [e^{\mathrm{i} t W_0}] - \mathbb{E} [e^{\mathrm{i} t W_{\alpha/2}}] \big) \Big| dt \le C \alpha^2,
		\end{flalign*}
		
		\noindent from which we deduce the second statement of \eqref{h0halpha21}.
	\end{proof}
	
	Recall that the total variational distance between two real random variables $(X, Y)$, with probability density functions $(f, g)$ respectively, is defined by
	\begin{flalign*}
		d_{\TV} (X, Y) = \displaystyle\int_{-\infty}^{\infty} \big( f(x) - g(x) \big)_+ dx.
	\end{flalign*}

	We then have the following two corollaries. The first bounds the total variation distance between $W_{\alpha/2}$ and $W_0$; the second bounds integrals of various densities (and their derivatives and differences) against a certain fractional moment.
	
	\begin{cor}
		
		\label{w0walphavariation} 
		
		There exists a constant $C > 1$ such that $d_{\TV} (W_{\alpha/2}, W_0) \le C \alpha^2$.
	\end{cor} 
	
	\begin{proof} 
		
		This follows from the fact that 
		\begin{flalign*}
			d_{\TV} (W_{\alpha/2}, W_0) = \displaystyle\frac{1}{2} \displaystyle\int_{-\infty}^{\infty} \big| h_{\alpha/2} (x) - h_0 (x) \big| dx = \displaystyle\frac{\alpha^2}{2} \displaystyle\int_{-\infty}^{\infty} \big| \Upsilon_{\alpha/2} (x) \big| dx < C \alpha^2,
		\end{flalign*}
		
		\noindent where in the last bound we used the second estimate in \Cref{h0halpha21}.
	\end{proof} 
	
	\begin{cor}
		
		\label{integralk} 
		
		There exists a constant $C > 1$ such that the following holds for any $\gamma \in \big[ \frac{\alpha}{2}, \alpha \big]$. For any index $f \in \{ h_{\alpha/2}, h_0, \Upsilon_{\alpha/2} \}$, we have 
		\begin{flalign*} 
			\displaystyle\sup_{K \in \mathbb{R}} \displaystyle\int_{-\infty}^{\infty} (e^{2t/\alpha} - 1)^{-\gamma} \big| f (t+K) \big| dt < C; \qquad \displaystyle\sup_{K \in \mathbb{R}} \displaystyle\int_{-\infty}^{\infty} (e^{2t/\alpha} - 1)^{-\gamma} \big| f' (t+K) \big| \le C.
		\end{flalign*} 
	\end{cor}
	
	\begin{proof}
		
		We only establish the first bound, since the proof of the latter is very similar. To that end, observe from \Cref{w0walpha} that 
		\begin{flalign*} 
			\sup_{x \in \mathbb{R}} \big| f (x) \big| \le C; \qquad \displaystyle\sup_{x \in \mathbb{R}} \big| f' (x) \big| \le C,
		\end{flalign*} 
		
		\noindent for $f = \Upsilon_{\alpha/2}$. Since the same bounds hold for $f = h_0$, they also hold for $f = h_{\alpha/2}$ (upon replacing $C$ with $2C$ if necessary). Hence,
		\begin{flalign*}
			\displaystyle\sup_{K \in \mathbb{R}} \displaystyle\int_{-\infty}^{\infty} (e^{2t/\alpha} - 1)^{-\gamma} f (t + K) dt = C \displaystyle\int_{-1}^1 (e^{2t/\alpha} - 1)^{-\gamma} dt + \displaystyle\sup_{K \in \mathbb{R}} \displaystyle\int_{|t| \ge 1} f (t+K) dt \le C,
		\end{flalign*}
		
		\noindent  where in the last bound we again used \Cref{w0walpha}; this establishes the lemma.	
	\end{proof} 
	
	We conclude with the following two tail bounds on $h_{\alpha/2} (x)$ and $h_0 (x)$ for $x$ negative.
	
	\begin{lem} 
		
		\label{halpha2h0xsmall}
		
		There exists a constant $C > 1$ such that, for any $x \ge 0$, we have 
		\begin{flalign*}
			\displaystyle\int_{-\infty}^{-x} \big( h_{\alpha/2} (y) + h_0 (y) \big) dy \le 4 e^{-e^x}.
		\end{flalign*}
	\end{lem} 
	
	\begin{proof}
		
		We have $\int_{-\infty}^{-x} h_0 (y) dy = e^{- e^x}$ by \Cref{halpha} and integration. So, it suffices to show that 
		\begin{flalign} 
			\label{walpha2x} 
			\mathbb{P} [W_{\alpha/2} \le -x] = \int_{-\infty}^{-x} h_{\alpha/2} (y) dy \le 3 e^{-e^x},
		\end{flalign} 
		
		\noindent which follows from the bound 
		\begin{flalign*}
			\mathbb{P} [W_{\alpha/2} \le -x] = \mathbb{P} [S \le e^{-2x/\alpha}] \le 3 \cdot \mathbb{E} \big[ \exp (-e^{2x/\alpha} S) \big] = 3 \cdot \exp \big( - \Gamma (1 - \alpha) \cdot e^x \big) \le 3 e^{-e^x}.
		\end{flalign*}  
		
		\noindent Here, to deduce the first statement we applied \eqref{halpha}; to deduce the second we applied a Markov estimate; to deduce the third we applied \eqref{ytsigma}; and to deduce the fourth we used the fact that $\Gamma (1 - \alpha) \ge 1$. 	
	\end{proof}
	
	\begin{cor}
		
		\label{exponentialhh} 
		
		There exists a constant $C > 1$ such that, for any $\kappa \in [0, 2]$, we have 
		\begin{flalign*}
			\displaystyle\int_{-\infty}^{\infty} e^{-\kappa s} \Upsilon_{\alpha/2} (s) ds \le C |\log \alpha|^2.
		\end{flalign*}
	\end{cor} 
	
	\begin{proof}
		By \Cref{w0walpha} and \Cref{halpha2h0xsmall}, we have for any real number $A > 0$ that
		\begin{flalign*}
			\displaystyle\int_{-\infty}^{\infty} e^{-\kappa s} \Upsilon_{\alpha/2} (s) ds & \le \alpha^{-2} \displaystyle\int_{-\infty}^{-A} e^{-\kappa s} \big( h_{\alpha/2} (s) + h_0 (s) \big) ds + \displaystyle\int_{-A}^0 e^{-\kappa s} \Upsilon_{\alpha/2} (s) ds + C \\
			& \le C \alpha^{-2} \displaystyle\int_{-\infty}^{-A} e^{-\kappa s -e^{-s}} ds + C e^{\kappa A} = C \displaystyle\int_{e^A}^{\infty} v^{\kappa} e^{-v} dv + C e^{\kappa A}
		\end{flalign*} 
		
		\noindent where in the last statement we changed variables $v = e^{-s}$. Since for $B = e^A$ we have 
		\begin{flalign*}
			\displaystyle\int_B^{\infty} v^{\kappa} e^{-v} dv = e^{-B/2} \displaystyle\int_B^{\infty} v^{\kappa} e^{-v/2} dv = 2^{\kappa + 1} e^{-B/2} \cdot \Gamma (\kappa) \le C e^{-B/2},
		\end{flalign*}
		
		\noindent it follows that for $e^A = 10 |\log \alpha|$ that
		\begin{flalign*}
			\displaystyle\int_{-\infty}^{\infty} e^{-\kappa s} \Upsilon_{\alpha/2} (s) ds \le \alpha^{-2} C e^{-e^A/2} + C e^{\kappa A} \le C \alpha^3 + C |\log \alpha|^2,
		\end{flalign*}
		
		\noindent which verifies the lemma.
	\end{proof}

	\subsection{Initial Bounds on $a$ and $b$}
	
	\label{InitialabAlpha0} 
	
	In this section we show that $a(E_0)$ and $b(E_0)$ (throughout this section, we recall their definitions from \eqref{opaque}) are bounded above and below if $\lambda (E_0, \alpha) = 1$. We begin with the following proposition, approximating the $\alpha$-th moment of $R_{\star} (E_0)$ if $\lambda (E_0, \alpha) = 1$.
	\begin{prop}
		
		\label{lambdaalpha0} 
		
		There exists a constant $C > 1$ such that the following holds. Fix a real number $E_0 > 0$ such that $\lambda (E_0, \alpha) = 1$. Denoting $c_{\star} = 4 \log 2 + \pi$, we have
		\begin{flalign*}
			\bigg| \mathbb{E} \Big[ \big| R_{\star} (E) \big|^{\alpha} \Big] - (2 - c_{\star} \alpha) \bigg| \le C \alpha^2.
		\end{flalign*}
	\end{prop}
	
	\begin{proof}
		
		Throughout this proof, we abbreviate $R_{\star} = R_{\star} (E_0)$ and set
		\begin{flalign}
			\label{alpharab}
			A = \mathbb{E}\big[ |R|^{\alpha} \big]; \qquad B = \mathbb{E}\big[ |R|^{\alpha} \sgn (-R) \big].
		\end{flalign}
		
		\noindent  By \Cref{2lambdaEsalpha}, we have 
		\begin{flalign*}
			\pi^{-1} \cdot K_{\alpha} \cdot \Gamma (\alpha) \cdot (1 - t_{\alpha}^2)^{1/2} \cdot \big( t_{\alpha} A + \sqrt{A^2 + t_{\alpha}^2 B^2} \big) = \lambda(E_0, \alpha) = 1.
		\end{flalign*}
		
		\noindent Using the definitions \eqref{tlrk} of $K_{\alpha}$ and $t_{\alpha}$, and applying the Taylor expansion for $t_{\alpha} = \frac{\pi \alpha}{2} + O(\alpha^3)$, it follows that
		\begin{flalign*}
			\displaystyle\frac{\alpha}{2 \pi} \cdot \Gamma \bigg( \displaystyle\frac{1-\alpha}{2} \bigg)^2 \cdot \Gamma (\alpha) \cdot \big( 1 + O(\alpha^2) \big) \cdot \bigg( \Big( \displaystyle\frac{\pi \alpha}{2} + O (\alpha^3) \Big) \cdot A + \sqrt{A^2 + \big(\alpha^2 + O(\alpha^4) \big) \cdot B^2} \bigg) = 1.
		\end{flalign*}
		
		\noindent Further using the fact that $\alpha \Gamma (\alpha) = \Gamma (1 + \alpha)$; applying Taylor expansions 
		\begin{flalign*} 
			\Gamma (1 + \alpha) = 1 + \alpha \Gamma' (1) + O(\alpha^2); \qquad \Gamma \bigg( \frac{1-\alpha}{2} \bigg) = \Gamma \bigg( \frac{1}{2} \bigg) - \frac{\alpha}{2} \cdot \Gamma' \bigg( \frac{1}{2} \bigg) + O(\alpha^2);
		\end{flalign*} 
		
		\noindent and using the fact that $\Gamma \big( \frac{1}{2} \big) = \pi^{1/2}$, we deduce
		\begin{flalign*}
			\pi^{-1} \cdot \big(1 + \alpha \cdot \Gamma' (1) \big) \cdot \bigg( \pi^{1/2} - \frac{\alpha}{2} \cdot \Gamma' (1/2) \bigg)^2 \cdot \bigg( \displaystyle\frac{\pi \alpha}{2} \cdot A + \sqrt{A^2 + \alpha^2 B^2} \bigg) = 2 + O (\alpha^2),
		\end{flalign*}
		
		\noindent and thus 
		\begin{flalign*}
			\displaystyle\frac{\pi \alpha}{2} \cdot A + \sqrt{A^2 + \alpha^2 B^2} = 2 +  2 \pi^{-1/2} \alpha \cdot \Gamma' \Big( \displaystyle\frac{1}{2} \Big) - 2 \alpha \cdot \Gamma' (1) + O(\alpha^2) = 2 - 4 \alpha \log 2 + O(\alpha^2).
		\end{flalign*}
		
		\noindent Since $|B| \le |A|$ from the definition \eqref{alpharab} of $A$ and $B$, it follows that 
		\begin{flalign*}
			\bigg( 1 + \displaystyle\frac{\pi \alpha}{2} \bigg) \cdot A = 2 - 4 \alpha \log 2 + O(\alpha^2),
		\end{flalign*}
		
		\noindent which yields the proposition. 
	\end{proof} 
	
	The next lemma states that $E_0 \le 1$ if $\lambda (E_0, \alpha) = 1$ (and $\alpha$ is sufficiently small).
	
	\begin{lem}
		
		\label{lambdaalpha01} 
		
		There exists a constant $c>0$ such that the following holds for $\alpha \in (0,c)$. Any positive real number $E_0 > 0$ such that $\lambda (E_0, \alpha) = 1$ satisfies $E_0 \le 1$. 
	\end{lem}
	
	\begin{proof}
		
		Suppose for the sake of contradiction that there exists some $E_0 > 1$ satisfying $\lambda (E_0, \alpha) = 1$. By the elementary inequality $x \le 1 + x^2$, we have by \eqref{fba} and \Cref{lambdaalpha0} that, for $\alpha \in (0,c)$,
		\begin{flalign} 
			\label{abap}
			a + b = \mathbb{E} \Big[ \big| R_{\star} (E_0) \big|^{\alpha/2} \Big]
			\le \mathbb{E} \Big[ \big| R_{\star} (E_0) \big|^{\alpha} \Big] + 1
			\le 4.
		\end{flalign} 
		
		\noindent Further, recalling the variables $S$ and $T$ from \eqref{fggamma}, and setting $Y = a^{2/\alpha} S - b^{2/\alpha} T$, we have 
		\begin{flalign}
			\label{cow}
			\mathbb{E} \big[ |E_0 + Y|^{-\alpha} \big] = \mathbb{E} \Big[ \big| R_{\star} (E_0) \big|^{\alpha} \Big] \ge \displaystyle\frac{7}{4}
		\end{flalign}
		
		\noindent where in the last inequality we applied \eqref{lambdaalpha0}. Using H\"older's inequality, we write
		\begin{align}
			\mathbb{E} \big[ |E_0 + Y|^{-\alpha} \big]\notag
			&= \mathbb{E} \big[ \mathbbm{1}_{Y \notin [E_0 - \alpha, E_0 + \alpha]} \cdot |E_0 + Y|^{-\alpha} \big] + \mathbb{E} \big[ \mathbbm{1}_{Y \in [E_0 - \alpha, E_0 + \alpha]} \cdot |E_0 + Y|^{-\alpha} \big]\\
			&
			\le \alpha^{-\alpha/2} + \mathbb{E} \big[ \mathbbm{1}_{Y \in [E_0 - \alpha, E_0 + \alpha]} \cdot |E_0 + Y|^{-\alpha} \big].
			\label{goat}
		\end{align}
		
		\noindent Letting $g$ denote the probability density function of $y$, we have $\sup_{|s| \ge 1/2} f(s) \le C$ by \eqref{abap}. Hence,
		\begin{flalign*}
			\mathbb{E} \big[ \mathbbm{1}_{Y \in [E_0 - \alpha, E_0 + \alpha]} \cdot |E_0 + Y|^{-\alpha} \big] \le C \alpha,
		\end{flalign*}
		
		\noindent With \eqref{goat}, this contradicts \eqref{cow} after taking $\alpha$ sufficiently small, establishing the lemma.
	\end{proof}

	We further require the following lemma bounding inverse moments of stable laws, which will be useful for obtaining an upper bound on $a + b$.

	\begin{lem}
		
		\label{avestimatek}
		
		There exists a constant $C > 1$ such that the following holds for any $k \in \mathbb{R}$ such that $0< k \le \frac{1}{\alpha} \le 20$.  For any real numbers $A, v \ge 0$, we have
		\begin{flalign}\label{fjt}
			\mathbb{E} \big[ |A^{2/\alpha} - v^{2/\alpha} S|^{-k\alpha/2} \big] + \mathbb{E} \big[ | A^{2/\alpha} + v^{2/\alpha} S|^{-k\alpha/2} \big] \le C \cdot \max ( A, v )^{-k}.
		\end{flalign} 
	\end{lem}
	
	\begin{proof}
		We bound each term on the left side of \eqref{fjt} separately. The claimed bound on the second term follows directly from \eqref{eq:negmomS} after bounding 
		\begin{flalign}
			| A^{2/\alpha} + v^{2/\alpha} S|^{-k\alpha/2} \le \max\big(A^{2/\alpha},
			|v^{2/\alpha} S| \big)^{-k\alpha/2}.
		\end{flalign}
		We now turn to bounding the first term on the left side of \eqref{fjt}. We can assume without loss of generality that $v = 1$ and $A>0$. We write
		$$
		\mathbb{E} \big[ | A^{2/\alpha} -  S |^{-k\alpha/2}\big] = \int_0 ^\infty \mathbb{P}\big( | A^{2/\alpha} -  S | \leq t^{-2/(k\alpha)}  \big)\,  dt. 
		$$
		We decompose this integral over $(0,2^{\alpha k/2}A^{-k})$ and $(2^{\alpha k/2}A^{-k},\infty)$. For the first term, we find
		\begin{eqnarray}
			\int_{0}^{2^{\alpha k/2} A^{-k} } \mathbb{P}\big( | A^{2/\alpha} -  S | \leq t^{-2/(k\alpha)}  \big)\, dt 
			& \leq &\int_{0}^{2^{\alpha k/2} A^{-k}} \mathbb{P}\big( S \leq A^{2/\alpha} + t^{-2/(k\alpha)}  \big)\, dt\notag \\
			&  \leq & \int_{0}^{\infty} \mathbb{P}\big( S \leq (2 +1) t^{-2/(k\alpha)}  \big)\, dt \notag \\
			& = & 3^{k \alpha/2} \cdot \mathbb{E} [S^{-k\alpha/2}] \le C,\label{minpr}
		\end{eqnarray}
		where in the last inequality we used \eqref{eq:negmomS} and the assumed bound on $k$.
		We also have the trivial bound $\mathbb{P}\big( | A^{2/\alpha} -  S | \leq t^{-2/(k\alpha)}  \big)  \leq 1$. We deduce using this estimate, the assumed bound on $k$, and \eqref{minpr} that 
		$$
		\int_{0}^{2^{\alpha k/2}A^{-k}} \mathbb{P}\big( | A^{2/\alpha} -  S | \leq t^{-2/(k\alpha)}  \big)\, dt \leq \min (C , 2^{\alpha k/2} A^{-k}) \leq C \max(A,1)^{-k}. 
		$$
		
		Recall that we denote the density of $\frac{\alpha}{2} \log S$ by $h_{\alpha/2}$ (see \Cref{halpha}). 
		For the integral over $(2^{\alpha k/2}A^{-k},\infty)$, we write
		\begin{eqnarray}
			\int_{2^{\alpha k/2} A^{-k}}^\infty \mathbb{P}\big( | A^{2/\alpha} -  S | \leq t^{-2/(k\alpha)}  \big) dt 
			& = & \int_{2^{\alpha k/2} A^{-k}}^\infty  \int_{\frac{\alpha}{2} \log ( A^{2/\alpha} - t^{-2/(k\alpha)})}^{\frac{\alpha}{2} \log ( A^{2/\alpha} + t^{-2/(k\alpha)})} h_{\alpha/2} (x) \, dx \, dt \nonumber \\
			& \leq & \int_{2^{\alpha k/2} A^{-k}}^\infty  \int_{\log A - \delta }^{\log A + \delta }h_{\alpha/2} (x) \, dx \, dt, \label{eq:2ndint}
		\end{eqnarray}
		where we set $\delta =  2\alpha (At^{1/k})^{-2/\alpha}$ and we have used that $\log (1 + x) \leq x$  and $\log (1- x)\geq x /(1-x_0)$ if $0 \leq x \leq x_0 < 1$ for $x = (At^{1/k})^{-2/\alpha}$ and $x_0  = 1/2$. 
		By \Cref{w0walpha}, we have $h_{\alpha/2}(x) \leq C$ for all $\alpha \in  (0, \frac{1}{20})$. We deduce that 
		$$
		\int_{2^{\alpha k/2} A^{-k}}^\infty  \int_{\log A - \delta }^{\log A + \delta }h_{\alpha/2} (x) \, dx \, dt \leq  C \alpha \int_{2^{\alpha k/2} A^{-k}}^\infty   (At^{1/k})^{-2/\alpha} \, dt  = C \cdot \frac{2^{ \alpha k}  \alpha }{\frac{2}{\alpha k} - 1}\cdot  \frac{1}{A^k} \leq \frac{C}{A^k}.
		$$
		
		The proof of the lemma will thus be complete if we prove that the integral \eqref{eq:2ndint} is also bounded uniformly in $A$ and $\alpha$ by a constant. 
		For this, we need an upper bound on $h_{\alpha/2}(x)$ that depends on $x$ but remains uniform in $\alpha$. To that end, note that by Markov's inequality and \eqref{eq:negmomS},  we have for all $r,x >0 $ and $0 < \alpha < 1/2$ that 
		\begin{flalign}\label{gcdf}
			\mathbb{P}(W_{\alpha/2} \leq  x) = \mathbb{P}( e^{-r W_{\alpha/2}} \geq e^{-rx}) \leq e^{rx }\cdot \mathbb{E} [e^{-k W_{\alpha/2}}] \leq C_r e^{rx},
		\end{flalign}
		where $C_r>1$ is a constant that depends only on $r$, and we used \Cref{eitw0walpha2} in the last inequality. 
		By \Cref{w0walpha}, there exists $C>1$ such that for any $x,y>0$, 
		\begin{flalign}\label{lss}
			h_{\alpha/2} (x) \leq h_{\alpha/2} (y) + C |x-y|.
		\end{flalign}
		Then by \eqref{gcdf}, we have for any $\theta >0$ that
		$$
		C_r e^{rx} \geq \int_{-\infty}^{x} h_{\alpha/2}(y) \, dy \geq   \int_{x - \theta}^{x} h_{\alpha/2}(y) \, dy \geq \theta h_{\alpha/2}(x) - \frac{c \theta^2}{2},
		$$ 
		where we used \eqref{lss} in the last inequality. 
		Choosing $\theta = e^{rx/2}$, we deduce that for any $r >0$, there exists a constant $C'_r>1$ such that 
		\begin{flalign}
			h_{\alpha/2}(x) \leq C'_r e^{rx}. 
		\end{flalign}
		Putting this last estimate in  \eqref{eq:2ndint} and setting $r=k$, we find
		\begin{align}
			\int_{2^{\alpha k/2}A^{-k}}^\infty \mathbb{P}\big( | A^{2/\alpha} -  S | \leq t^{-2/(k\alpha)}  \big)\,  dt \notag
			&\le  C'_k \int_{2^{\alpha k/2}A^{-k}}^\infty  \int_{\log A - \delta }^{\log A + \delta }  e^{kx} \, dx \, dt \\ &= \frac{C'_k A^k }{k} \notag \int_{2^{\alpha k/2}A^{-k}} ^\infty (e^{k\delta} - e^{-k\delta}) \, dt\\
			&\le  \frac{C'_k}{k}  \frac{2^{\alpha k/2} e \alpha }{\frac{2}{\alpha k} - 1},\label{lastintegral}
		\end{align}
		where we increased the value of $C'_k$ in the last line.
		In the last inequality we used the definition of $\delta$ to see that 
		\begin{flalign}k \delta  = 2k \alpha (At^{1/k})^{-2/\alpha} \leq 2^{1-k/2} k \alpha \leq 1, \text{ which implies } e^{k\delta} - e^{-k\delta} \leq 2 e \delta k.
		\end{flalign}
		%
		Using \eqref{lastintegral}, we conclude that the integral \eqref{eq:2ndint} is uniformly bounded in $A$ for all $\alpha \in (0, \frac{1}{20})$ and $k\le \alpha^{-1}$. This completes the proof of the lemma. 
	\end{proof} 
	
	We can now establish the following lemma bounding $a(E_0)$ and $b(E_0)$ from above and below.
	
	\begin{lem}
		\label{abestimatealpha0} 
		
		There exist constants $c > 0$ and $C > 1$ such that, if $\alpha \in (0, c)$ and $E_0 \le 1$, then  
		\begin{flalign*} 
			C^{-1} \le a(E_0) \le C; \qquad C^{-1} \le b(E_0) \le C.
		\end{flalign*} 
		
	\end{lem}

	\begin{proof} 
		
		Throughout this proof, we abbreviate $a = a(E_0)$ and $b = b (E_0)$, and we also set $E_0 = u^{2/\alpha}$. We first show the upper bounds on $a$ and $b$. To that end, observe from \eqref{fggamma} that 
		\begin{flalign}
			\label{abu} 
			a + b = \mathbb{E} \big[ |u^{2/\alpha} + a^{2/\alpha} S - b^{2/\alpha} T|^{-\alpha/2} \big].
		\end{flalign}
	
		\noindent Moreover, from \eqref{avestimatek}, we have that 
		\begin{flalign*} 
			\mathbb{E} \big[ |u^{2/\alpha} + x^{2/\alpha} S - y^{2/\alpha} T|^{-\alpha/2} \big]	\le C \min \{ x^{-1}, y^{-1} \}, \qquad \text{for any $x, y \ge 0$},
		\end{flalign*}
	
		\noindent which together with \eqref{abu} implies that $a + b \le C \min \{ a^{-1}, b^{-1} \}$, meaning that $a + b \le C$. This verifies the upper bounds on $a$ and $b$. 
		
		Hence, it remains to show the lower bounds. We fix a constant $C_+ \geq 1$ such that $\sigma = a + b \leq C_+$ for all $0 < \alpha < 1/20$. 
		Assume first that $a \geq b$, and set $\Omega_1 = \{ T \leq S \}$. Using symmetry, we deduce that $\mathbb{P}(\Omega_1) = 1/2$.
		We have that 
		\begin{align*}
			b &= 
			\mathbb{E} \Big[ (E+a^{2/\alpha} S - b^{2/\alpha} T)_+^{-\alpha/2} \Big]
			\\ &\ge 
			\mathbb{E} \Big[ \mathbbm{1}_{\Omega_1} \cdot (E+a^{2/\alpha} S - b^{2/\alpha} T)^{-\alpha/2} \Big]
			\ge \frac{1}{2}\cdot \mathbb{E} \Big[  (E+a^{2/\alpha} S)^{-\alpha/2} \Big]
			\ge c a^{-1}  \ge c C_{+}^{-1}.
		\end{align*}
		for some $c>0$. This completes the analysis of the case $a \ge b$.
		
		Next, suppose that $ b \geq a$ and $b \geq u$. We consider the event 
		\begin{flalign*} 
			\Omega_2 = \{ 2 \leq T \leq 4,  S  \leq 1\}.
		\end{flalign*} 
		By Corollary \ref{w0walphavariation}, there exists $c>0$ such that
		$P(\Omega_2) > c$. Then
		\begin{align*}
			a &= 
			\mathbb{E} \Big[ (E+a^{2/\alpha} S - b^{2/\alpha} T)_-^{-\alpha/2} \Big]\\
			&\ge 
			\mathbb{E} \Big[\mathbbm{1}_{\Omega_2} \cdot ( u^{2/\alpha} +a^{2/\alpha} S - b^{2/\alpha} T)_-^{-\alpha/2} \Big]\\
			&=
			\mathbb{E} \Big[\mathbbm{1}_{\Omega_2} \cdot ( b^{2/\alpha} T - u^{2/\alpha} - a^{2/\alpha} S )^{-\alpha/2} \Big]\\
			&\ge \mathbb{E} \Big[\mathbbm{1}_{\Omega_2} \cdot ( b^{2/\alpha} T - u^{2/\alpha} )^{-\alpha/2} \Big]
			\ge c ( 4 b^{2/\alpha} - u^{2/\alpha})^{-\alpha/2} \ge c b^{-1} \ge c C_{+}^{-1}. 
		\end{align*}
		This completes the analysis of the case where $b\ge a$ and $b \ge u$.

		Finally, suppose that $ u \geq b \geq a$. We first consider the event $\Omega_3 = \{ T  \leq 1\}$. 
		By Corollary \ref{w0walphavariation}, there exists $c>0$ such that
		$P(\Omega_3 ) > c$. 
		We deduce that 
		\begin{align}
			b &= 
			\mathbb{E} \Big[ (E+a^{2/\alpha} S - b^{2/\alpha} T)_+^{-\alpha/2} \Big]
			\ge 
			\mathbb{E} \Big[\mathbbm{1}_{\Omega_3} \cdot (E+a^{2/\alpha} S - b^{2/\alpha} T)^{-\alpha/2} \Big] \ge c u^{-1} \ge c,\label{lowerbest}
		\end{align}
		where we used $E=u^{2/\alpha}$ and $u\in[0,1]$ in the last inequality.
		Next,  we consider the event 
		\begin{flalign*}
			\Omega_4 = \{ S \leq 1/2, 2 u^{2/\alpha} \leq b^{2/\alpha} T   \}.
		\end{flalign*}
		By Corollary \ref{w0walphavariation}, \eqref{lowerbest}, and the assumption that $u\le 1$, there exists $c>0$ such that
		$P(\Omega_4 ) > c$.
		Then have 
		\begin{align*}
			a = 
			\mathbb{E} \Big[ (E+a^{2/\alpha} S - b^{2/\alpha} T)_-^{-\alpha/2} \Big] & \ge 
			\mathbb{E} \Big[\mathbbm{1}_{\Omega_4} \cdot ( u^{2/\alpha} +a^{2/\alpha} S - b^{2/\alpha} T)_-^{-\alpha/2} \Big] \\
			& \ge \mathbb{E} [ \mathbbm{1}_{\Omega_4} \cdot u^{-1}] \ge c u^{-1} \ge c.
		\end{align*}
		This completes the analysis of the case $ u \geq b \geq a$, and therefore completes proof.
	\end{proof}

	\subsection{Replacement by Gumbel Random Variables} 
	
	\label{FFGG} 
	
	In this section we approximate the functions $F_{\gamma}$ and $G_{\gamma}$ from \eqref{fggamma} by explicit quantities that can be viewed as functionals of Gumbel random variables. In particular, observe from \eqref{fggamma} and \Cref{halpha} that
	\begin{flalign}
		\label{fggammaintegralalpha0} 
		\begin{aligned}
			F_{\gamma} (u^{2/\alpha}, x, y) & = \displaystyle\int_{-\infty}^{\infty} \displaystyle\int_{(\alpha/2) \log (y^{-2/\alpha} u^{2/\alpha} + x^{2/\alpha} y^{-2/\alpha} e^{2s/\alpha})}^{\infty} |u^{2/\alpha} + x^{2/\alpha} e^{2s/\alpha} - y^{2/\alpha} e^{2t/\alpha}|^{-\gamma} \\ 
			& \qquad \qquad \qquad \qquad \qquad \qquad \qquad \qquad \qquad \times h_{\alpha/2} (s) h_{\alpha/2} (t) dt ds \\
			G_{\gamma} (u^{2/\alpha}, x, y) & = \displaystyle\int_{-\infty}^{\infty} \displaystyle\int_{-\infty}^{(\alpha/2) \log (y^{-2/\alpha} u^{2/\alpha} + x^{2/\alpha} y^{-2/\alpha} e^{2s/\alpha})} |u^{2/\alpha} + x^{2/\alpha} e^{2s/\alpha} - y^{2/\alpha} e^{2t/\alpha}|^{-\gamma} \\ 
			& \qquad \qquad \qquad \qquad \qquad \qquad \qquad \qquad \qquad \times h_{\alpha/2} (s) h_{\alpha/2} (t) dt ds. 
		\end{aligned} 
	\end{flalign}
	
	The following definition essentially formally replaces $\alpha$ in \eqref{fggammaintegralalpha0} with $0$.

	\begin{definition}
		
		\label{fuxyguxy2}
		
		For any real numbers $\gamma \in (0, 1)$ and $x, y, u \ge 0$, set 
		\begin{flalign*}
			\widetilde{F}_{\gamma} (u, x, y) = y^{-2\gamma / \alpha} \displaystyle\int_{-\infty}^{\infty} \displaystyle\int_{\log \max \{ u/y, x e^s/y \}}^{\infty}  e^{-2\gamma t/\alpha} h_0 (s) h_0 (t) ds dt; \\
			\widetilde{G}_{\gamma} (u, x, y) = \displaystyle\int_{-\infty}^{\infty} \displaystyle\int_{-\infty}^{\log \max \{ u/y, x e^s / y \}} \max \{ u, x e^s \}^{-2 \gamma / \alpha} h_0 (s) h_0 (t) ds dt. 
		\end{flalign*}
		
	\end{definition} 
	
	We establish the following proposition in \Cref{FFGGalpha0Estimate0} below, indicating that one may replace $(F, G)$ with $(\widetilde{F}, \widetilde{G})$ (the latter of which are explicit by \Cref{fgu}).

	\begin{prop}
		
		\label{gammaffgg0}
		
		There exists a constant $C > 1$ such that, for any real number $\gamma \in \big[ \frac{\alpha}{2}, \alpha \big]$ and index $\mathfrak{Y} \in \{ F, G \}$, we have 
		\begin{flalign*}
			\big| \mathfrak{Y}_{\gamma} (u^{2/\alpha}, x, y) - \widetilde{\mathfrak{Y}}_{\gamma} (u, x, y) \big| \le \displaystyle\frac{C \alpha^2 |\log \alpha|^2}{x^{2\gamma/\alpha}}.
		\end{flalign*} 
	\end{prop}
	
	The following two lemmas explicitly evaluate $\widetilde{F}_{\alpha/2}$, $\widetilde{G}_{\alpha/2}$, and $\widetilde{F}_{\alpha} + \widetilde{G}_{\alpha}$.

	\begin{lem} 
		
		\label{fgu} 
		
		For any real numbers $u, x, y \ge 0$, we have
		\begin{flalign*}
			\widetilde{F}_{\alpha/2} (u, x, y) & = \displaystyle\frac{y}{(x+y)^2} - e^{-(x+y)/u} \bigg( \displaystyle\frac{y}{(x+y)^2} + \displaystyle\frac{y}{(x+y) u} \bigg); \\
			\widetilde{G}_{\alpha/2} (u, x, y) & = \displaystyle\frac{x}{(x+y)^2} + e^{-(x+y)/u} \bigg( \displaystyle\frac{y}{u(x+y)} - \displaystyle\frac{x}{(x+y)^2} \bigg).
		\end{flalign*}
		
	\end{lem} 
	
	\begin{proof}
		
		Changing variables from $(s, t)$ to $(e^{-s}, e^{-t})$ in \Cref{fuxyguxy2} yields
		\begin{flalign*}
			\widetilde{F}_{\alpha/2} (u, x, y) &= \displaystyle\frac{1}{y} \displaystyle\int_0^{\infty} \displaystyle\int_0^{\min \{ y/u, ys/x \}} t e^{-s-t} ds dt \\
			& = \displaystyle\frac{1}{y} \displaystyle\int_{x/u}^{\infty} \displaystyle\int_0^{y/u} t e^{-s-t} ds dt + \displaystyle\frac{1}{y} \displaystyle\int_0^{x/u} \displaystyle\int_0^{ys/x} t e^{-s-t} ds dt \\
			& = \displaystyle\frac{1}{y} \displaystyle\int_{x/u}^{\infty} e^{-s} \big( 1 - (yu^{-1} + 1) e^{-y/u} \big) ds + \displaystyle\frac{1}{y} \displaystyle\int_0^{x/u} e^{-s} \big( 1 - (yx^{-1} s + 1) e^{-ys/x} \big) ds \\
			& = \displaystyle\frac{e^{-x/u}}{y} \big( 1 - (yu^{-1} + 1) e^{-y/u} \big) + \displaystyle\frac{1}{y} (1 - e^{-x/u}) \\
			& \qquad - \displaystyle\frac{x}{y(x+y)} \bigg( 1 - e^{-(x+y)/u} + \displaystyle\frac{y}{x+y} - \displaystyle\frac{y(u+x+y)}{u(x+y)} e^{-(x+y)/u} \bigg),
		\end{flalign*}
		
		\noindent which gives the first statement of the lemma.	Again changing variables from $(s, t)$ to $(e^{-s}, e^{-t})$ in \Cref{fuxyguxy2} gives
		\begin{flalign*}
			\widetilde{G}_{\alpha/2} (u, x, y) & = \displaystyle\int_0^{\infty} \displaystyle\int_{\min \{ y/u, ys/x \}}^{\infty} \max \{ u, s^{-1} x \}^{-1} e^{-s-t} ds dt \\
			& = \displaystyle\int_0^{x/u} \displaystyle\int_{ys/x}^{\infty} sx^{-1} e^{-s-t} ds dt + \displaystyle\int_{x/u}^{\infty} \displaystyle\int_{y/u}^{\infty} u^{-1} e^{-s-t} ds dt \\
			& = \displaystyle\frac{1}{x} \displaystyle\int_0^{x/u} se^{-s-ys/x} ds + \displaystyle\frac{e^{-y/u}}{u} \displaystyle\int_{x/u}^{\infty} e^{-s} ds \\
			& = \displaystyle\frac{x}{(x+y)^2} \bigg( 1 - \displaystyle\frac{x+y+u}{u} e^{-(x+y)/u} \bigg) + \displaystyle\frac{1}{u} e^{-(x+y)/u},
		\end{flalign*}
		
		\noindent which gives the second statement of the lemma.
	\end{proof} 
	
	\begin{lem}
		
		\label{falphagalpha0} 
		
		For any real numbers $u, x, y \ge 0$, we have 
		\begin{flalign*}
			\widetilde{F}_{\alpha} (u, x, y) + \widetilde{G}_{\alpha} (u, x, y) = \displaystyle\frac{2}{(x+y)^2} - 2 \bigg( \displaystyle\frac{1}{u(x+y)} + \displaystyle\frac{1}{(x+y)^2} \bigg) e^{-(x+y)/u}.
		\end{flalign*}
	\end{lem}
	
	\begin{proof} 
		Changing variables from $(s, t)$ to $(e^{-s}, e^{-t})$ in \Cref{fuxyguxy2}, we obtain  
		\begin{flalign*}
			\widetilde{F}_{\alpha} (u, x, y) + \widetilde{G}_{\alpha} (u, x, y) & = \displaystyle\int_{-\infty}^{\infty} \displaystyle\int_{-\infty}^{\infty} \max \{ u, xe^s, ye^t \}^{-2} h_0 (s) h_0 (t) ds dt \\
			& = \displaystyle\int_0^{\infty} \displaystyle\int_0^{\infty} \max \bigg\{ u, \displaystyle\frac{x}{s}, \displaystyle\frac{y}{t} \bigg\}^{-2}  e^{-s-t} ds dt = \mathbb{E} \Big[ \min \{ u^{-1}, x^{-1} S', y^{-1} T' \}^2 \Big],
		\end{flalign*} 
		
		\noindent where $(S', T')$ are independent exponential random variables, with probability density function $e^{-x} dx$. Letting $V = \min \{ x^{-1} S', y^{-1} T' \}$, we find that $V$ is a rescaled exponential random variable, with cummulative density function $\mathbb{P} [V \ge t] = \mathbb{P} [S' \ge xt] \cdot \mathbb{P} [T' \ge yt] = e^{-(x+y) t}$. Hence,
		\begin{flalign*}
			\widetilde{F}_{\alpha} (u, x, y) + \widetilde{G}_{\alpha} (u, x, y) & = \mathbb{E} \big[ \min \{ u^{-1}, V \}^2 \big] \\
			& = \displaystyle\int_0^{1/u} (x+y) v^2 e^{-(x+y) v} dv + u^{-2} \displaystyle\int_{1/u}^{\infty} (x+y) e^{-(x+y) v} dv \\
			& = \displaystyle\frac{1}{(x+y)^2} \displaystyle\int_0^{(x+y)/u} v^2 e^{-v} dv + u^{-2} e^{-(x+y)/u} \\
			& = \displaystyle\frac{1}{(x+y)^2} \bigg( 2 - \bigg( \displaystyle\frac{(x+y)^2}{u^2} + \displaystyle\frac{2(x+y)}{u} + 2 \Big) e^{-(x+y)/u} \bigg) \\
			& \qquad + u^{-2} e^{-(x+y)/u},
		\end{flalign*} 
		
		\noindent from which we deduce the lemma.
	\end{proof}

	\subsection{Replacement of $h_{\alpha/2}$ With $h_0$}
	
	\label{Replaceh0halpha20} 
	
	The following definition replaces the densities $h_{\alpha/2}$ appearing in \eqref{fggammaintegralalpha0} with $h_0$.   
	
	\begin{definition} 
		
		\label{hhfg} 
		
		For any real numbers $\gamma \in \big[ \frac{\alpha}{2}, \alpha \big]$ and $x, y, u \ge 0$ satisfying \eqref{xyualpha0}, set
		\begin{flalign*}
			\breve{F}_{\gamma} (u^{2/\alpha}, x, y) & = \displaystyle\int_{-\infty}^{\infty} \displaystyle\int_{(\alpha/2) \log (y^{-2/\alpha} u^{2/\alpha} + x^{2/\alpha} y^{-2/\alpha} e^{2s/\alpha})}^{\infty}  |u^{2/\alpha} + x^{2/\alpha} e^{2s/\alpha} - y^{2/\alpha} e^{2t/\alpha}|^{-\gamma} \\ 
			& \qquad \qquad \qquad \qquad \qquad \qquad \qquad \qquad \qquad \times h_0 (s) h_0 (t) dt ds \\
			\breve{G}_{\gamma} (u^{2/\alpha}, x, y) & = \displaystyle\int_{-\infty}^{\infty} \displaystyle\int_{-\infty}^{(\alpha/2) \log (y^{-2/\alpha} u^{2/\alpha} + x^{2/\alpha} y^{-2/\alpha} e^{2s/\alpha})} |u^{2/\alpha} + x^{2/\alpha} e^{2s/\alpha} - y^{2/\alpha} e^{2t/\alpha}|^{-\gamma} \\ 
			& \qquad \qquad \qquad \qquad \qquad \qquad \qquad \qquad \qquad \times h_0 (s) h_0 (t) dt ds.
		\end{flalign*}
		
	\end{definition}

	In this section we establish the below lemma showing that $(F_{\gamma}, G_{\gamma}) \approx (\breve{F}_{\gamma}, \breve{G}_{\gamma})$. 
	
	\begin{lem} 
		
		\label{ffgg0} 
		
		There exists a constant $C > 1$ such that, for real number $\gamma \in \big[ \frac{\alpha}{2}, \alpha \big]$ and any index $\mathfrak{Y} \in \{ F, G \}$, we have 
		\begin{flalign*}
			\big| \mathfrak{Y}_{\gamma} (u^{2/\alpha}, x, y) - \breve{\mathfrak{Y}}_{\gamma}  (u^{2/\alpha}, x, y) \big| \le  \displaystyle\frac{C \alpha^2 |\log \alpha|^2}{x^{2\gamma / \alpha}}.
		\end{flalign*}
	\end{lem} 
	
	\begin{proof}
		
		We only analyze the case $\mathfrak{Y} = F$, as the proof is entirely analogous if $\mathfrak{Y} = G$. By changing variables from $(s, t)$ to $\big( s - \log x, t - \log y + \frac{\alpha}{2} \log (u^{2/\alpha} + e^{2s/\alpha}) \big)$ in \eqref{fggammaintegralalpha0}, we obtain
		\begin{flalign}
			\label{fgammaintegral2alpha00} 
			\begin{aligned}
				F_{\gamma} (u^{2/\alpha}, x, y) & =  \displaystyle\int_0^{\infty} \displaystyle\int_{-\infty}^{\infty} (u^{2/\alpha} + e^{2s/\alpha})^{-\gamma} (e^{2t/\alpha} - 1)^{-\gamma} h_{\alpha/2} (s - \log x) \\
				& \qquad \qquad \quad \times h_{\alpha/2} \bigg( t - \log y + \displaystyle\frac{\alpha}{2} \log (u^{2/\alpha} + e^{2s/\alpha}) \bigg) ds dt.
			\end{aligned}
		\end{flalign}
		
		\noindent and, by similar reasoning, 
		\begin{flalign*}
			\breve{F}_{\gamma} (u^{2/\alpha}, x, y) & = \displaystyle\int_0^{\infty} \displaystyle\int_{-\infty}^{\infty} (u^{2/\alpha} + e^{2s/\alpha})^{-\gamma} (e^{2t/\alpha} - 1)^{-\gamma} h_0 (s - \log x) \\
			& \qquad \qquad \qquad \quad  \times h_0 \bigg(t - \log y + \displaystyle\frac{\alpha}{2} \log (u^{2/\alpha} + e^{2s/\alpha}) \bigg) ds dt.
		\end{flalign*}
		
		\noindent Subtracting; using the bound $(u^{2/\alpha} + e^{2s/\alpha})^{-\gamma} \le e^{-2s\gamma / \alpha}$; and recalling $\Upsilon_{\alpha/2}$ from \Cref{halpha} yields
		\begin{flalign*}
			\big| &  F_{\gamma} (u^{2/\alpha}, x, y) - \breve{F}_{\gamma} (u^{2/\alpha}, x, y) \big| \\
			&   \le \alpha^2  \displaystyle\int_{-\infty}^{\infty} \displaystyle\int_{-\infty}^{\infty} e^{-2s\gamma / \alpha} (e^{2t/\alpha} - 1)^{-\gamma} h_0 (s - \log x) \Bigg| \Upsilon_{\alpha/2} \bigg(t - \log y + \displaystyle\frac{\alpha}{2} \log (u^{2/\alpha} + e^{2s/\alpha}) \bigg) \Bigg| ds dt \\
			& \quad   + \alpha^2  \displaystyle\int_{-\infty}^{\infty} \displaystyle\int_{-\infty}^{\infty} e^{-2s\gamma / \alpha} (e^{2t/\alpha} - 1)^{-\gamma} \big| \Upsilon_{\alpha/2} (s-\log x) \big| h_0 \bigg(t - \log y + \displaystyle\frac{\alpha}{2} \log (u^{2/\alpha} + e^{2s/\alpha}) \bigg) ds dt \\
			&   \quad + \alpha^4 \displaystyle\int_{-\infty}^{\infty} \displaystyle\int_{-\infty}^{\infty} e^{-2s \gamma / \alpha} (e^{2t/\alpha} - 1)^{-\gamma} \\
			& \qquad \qquad \qquad \times \Bigg| \Upsilon_{\alpha/2} (s - \log x) \Upsilon_{\alpha/2} \bigg(t - \log y + \displaystyle\frac{\alpha}{2} \log (u^{2/\alpha} + e^{2s/\alpha}) \bigg) \Bigg| ds dt.
		\end{flalign*} 
		
		\noindent Using \Cref{integralk} to integrate in $t$, it follows that 
		\begin{flalign*}
			\big|  F_{\gamma} (u^{2/\alpha}, x, y) - \breve{F}_{\gamma} (u^{2/\alpha}, x, y) \big| & \le C \alpha^2 \displaystyle\int_{-\infty}^{\infty} e^{-2s \gamma / \alpha} \Big( \big| h_0 (s - \log x) \big| + \big| \Upsilon_{\alpha/2} (s - \log x) \big| \Big) ds \\
			& \le \displaystyle\frac{C \alpha^2}{x^{2 \gamma / \alpha}} \displaystyle\int_{-\infty}^{\infty} \displaystyle\frac{C \alpha^2 |\log \alpha|^2}{x^{2\gamma/\alpha}},
		\end{flalign*}
		
		\noindent where in the second bound we changed variables from $s$ to $s - \log x$, and in the third we used \eqref{exponentialhh}. This establishes the lemma.
	\end{proof}

	\subsection{Replacement of $(\breve{F}, \breve{G})$ With $(\widetilde{F}, \widetilde{G})$}
	
	\label{FFGGalpha0Estimate0}

	In this section we establish \Cref{gammaffgg0}. To that end, we first require the following definition. 
	
	\begin{definition} 
		
		\label{psi} 
		
		For any real numbers $u, s \ge 0$, set  
		\begin{flalign}
			\label{psius} 
			\psi_{\alpha/2} (u, s) = (u^{2/\alpha} + e^{2s/\alpha})^{-\alpha/2}; \qquad \psi_0 (u, s) = \max \{ u, e^s \}^{-1}.
		\end{flalign}
		
	\end{definition} 
	
	\noindent We then have the following two lemmas approximating powers of sums of exponentials by exponentials. 
	
	\begin{lem} 
		
		\label{talpha0} 
		
		There exists a constant $C > 1$ such that, for any $\gamma \in \big[ \frac{\alpha}{2}, \alpha \big]$, we have
		\begin{flalign}
			\label{talphae} 
			\begin{aligned} 
				\big| (e^{2t/\alpha} - 1)^{-\gamma} - e^{-2t\gamma / \alpha} \big| & \le  Ct^{-\gamma} \cdot \mathbbm{1}_{t \le \alpha^3} + C \alpha |\log \alpha| \cdot \mathbbm{1}_{t \le \alpha} + C \alpha e^{-2t / \alpha} \cdot \mathbbm{1}_{t > \alpha},  \qquad \quad \text{if $t \ge 0$}; \\
				\big| (1 - e^{2t/\alpha})^{-\gamma} - 1 \big| & \le C |t|^{-\gamma} \cdot \mathbbm{1}_{t \ge -\alpha^3} + C \alpha |\log \alpha| \cdot \mathbbm{1}_{t \ge -\alpha} +  C \alpha e^{2t/\alpha} \cdot \mathbbm{1}_{t \le -\alpha}, \qquad     \text{if $t \le 0$}.
			\end{aligned}
		\end{flalign}
		
	\end{lem} 
	
	\begin{proof}
		
		Observe that the second statement of the lemma implies the first, since 
		\begin{flalign*} 
			\big| (e^{2t/\alpha} - 1)^{-\gamma} - e^{-2t\gamma/\alpha} \big| = e^{-2t\gamma / \alpha} \cdot \big| (1 - e^{-2t/\alpha})^{-\gamma} - 1 \big| \le \big| (1 - e^{-2t/\alpha})^{-\gamma} - 1 \big|,
		\end{flalign*}
		
		\noindent so it suffices to establish the second. To that end, first assume that $t \le -\alpha$. Then, a Taylor expansion implies that $(1 - e^{2t/\alpha})^{-\gamma} = 1 - \gamma e^{2t/\alpha} + O(\gamma^2)$, which since $\gamma \le \alpha$ implies the second estimate in \eqref{talphae}. Next, assume that $-\alpha \le t \le -\alpha^3$. Then, a Taylor expansion yields 
		\begin{flalign*} 
			(1 - e^{2t/\alpha})^{-\gamma} = \bigg( \frac{2|t|}{\alpha} + O \Big( \frac{t^2}{\alpha^2} \Big) \bigg)^{-\gamma} = 1 + O \big( \alpha |\log \alpha| \big),
		\end{flalign*} 
		
		\noindent from which we again deduce the second estimate in \eqref{talphae}. If instead $t \ge -\alpha^3$, then a Taylor expansion gives $(1 - e^{2t/\alpha})^{-\gamma} = O (t^{-\gamma} \cdot \alpha^{-\gamma}) = O (t^{-\gamma})$, which establishes \eqref{talphae} in this last case.
	\end{proof}

	\begin{lem} 
		
		\label{exponentialsu} 
		
		There exists a constant $C > 1$ such that the following holds for any $s \ge 0$. Denoting 
		\begin{flalign}
			\label{su0}  
			\varpi = \varpi (s, u) = \frac{e^s}{u} \cdot \mathbbm{1}_{u \ge e^s} + \frac{u}{e^s} \cdot \mathbbm{1}_{u < e^s} \le 1,
		\end{flalign}
		
		\noindent we have for any real number $\gamma \in \big[ \frac{\alpha}{2}, \alpha \big]$ that 
		\begin{flalign}
			\label{psialpha2psi0} 
			\begin{aligned} 
				& \big| \log \psi_{\alpha/2} (u,s)  - \log \psi_0 (u, s) \big|\le C \alpha \varpi^{2/\alpha}; \\
				& \big| \psi_{\alpha/2} (u, s)^{2\gamma/\alpha} - \psi_0 (u, s)^{2\gamma / \alpha} \big| \le C\alpha \varpi^{2/\alpha} \cdot \psi_0 (u, s)^{2 \gamma / \alpha}; \\
				& \big| \partial_u \psi_{\alpha/2} (u, s)^{2\gamma / \alpha} - \partial_u \psi_0 (u, s)^{2\gamma/\alpha} \big| \le C u^{-2\gamma / \alpha - 1} \varpi^{2/\alpha}. 
			\end{aligned} 
		\end{flalign}
		
	\end{lem} 
	
	\begin{proof} 
		
		Since $\varpi \le 1$, the first statement follows from the second (at $\gamma = \frac{\alpha}{2}$). To verify the second, we estimate  
		\begin{flalign*}
			\psi_{\alpha/2} (u, s)^{2\gamma/\alpha} - \psi_0 (u, s)^{2 \gamma / \alpha} = \psi_0 (u,s)^{2 \gamma / \alpha} \cdot \big| (1 + \varpi^{2/\alpha})^{-\gamma} - 1 \big| \le C \alpha \varpi^{2/\alpha} \cdot \psi_0 (u, s)^{2 \gamma / \alpha}.
		\end{flalign*}
		
		\noindent So, it remains to verify the third, to which end observe that 
		\begin{flalign*}
			\partial_u \psi_{\alpha/2} (u, s)^{2 \gamma / \alpha} &  = -\displaystyle\frac{2 \gamma}{\alpha} u^{-2 \gamma /\alpha-1} \cdot (1 + u^{-2/\alpha} e^{2s/\alpha})^{-\gamma-1}; \\
			\partial_u \psi_0 (u, s)^{2 \gamma / \alpha} & = - \displaystyle\frac{2\gamma}{\alpha} u^{-2\gamma / \alpha - 1} \cdot \mathbbm{1}_{u > e^s}.
		\end{flalign*} 
		
		\noindent This, together with the bounds $\gamma \le \alpha$ and 
		\begin{flalign*}
			\big| (1 + u^{-2/\alpha} e^{2s/\alpha})^{-\gamma - 1} - \mathbbm{1}_{u > e^s} \big| \le C \varpi^{2/\alpha}.
		\end{flalign*}
		
		\noindent yields the third statement of the lemma.
	\end{proof}
	
	Now we can establish \Cref{gammaffgg0}. 
	
	\begin{proof}[Proof of \Cref{gammaffgg0}]
		In view of \Cref{ffgg0}, it suffices to show that 
		\begin{flalign}
			\label{hzk00}
			\big|  \breve{\mathfrak{Y}}_{\gamma} (u^{2/\alpha}, x, y) - \widetilde{\mathfrak{Y}}_{\gamma} (u, x, y) \big| \le \displaystyle\frac{C \alpha^2 |\log \alpha|^3}{x^{2 \gamma / \alpha}}.
		\end{flalign}
		
		\noindent As in the proof of \Cref{ffgg0}, we assume $\mathfrak{Y} = F$, as the proof when $\mathfrak{Y} = G$ is entirely analogous (using the second bound in \eqref{talphae} instead of the first below). Following \eqref{fgammaintegral2alpha00} and using \eqref{psius}, we find 
		\begin{flalign}
			\label{fgammafgamma0} 
			\begin{aligned}
				\breve{F}_{\gamma} (u^{2/\alpha}, x, y) & = \displaystyle\int_0^{\infty} \displaystyle\int_{-\infty}^{\infty} (e^{2t/\alpha} - 1)^{-\gamma} \psi_{\alpha/2} (u, s)^{2 \gamma / \alpha} h_0 (s - \log x) \\
				& \qquad \qquad \qquad \times h_0 \big( t - \log y + \log \psi_{\alpha/2} (u,s) \big) ds dt; \\
				\widetilde{F}_{\gamma} (u, x, y) & = \displaystyle\int_0^{\infty} \displaystyle\int_{-\infty}^{\infty} e^{-2t\gamma / \alpha} \psi_0 (u, s)^{2\gamma / \alpha} h_0 (s - \log x) h_0 \big( t - \log y - \log \psi_0 (u, s) \big) ds dt.
			\end{aligned} 
		\end{flalign}
		
		\noindent Further setting
		\begin{flalign*}
			\widetilde{F}_{\gamma, 1} (u, x, y) & = \displaystyle\int_0^{\infty} \displaystyle\int_{-\infty}^{\infty} e^{-2t \gamma/\alpha} \psi_{\alpha/2} (u, s)^{2\gamma / \alpha} h_0 (s - \log x) h_0 \big( t - \log y - \log \psi_{\alpha/2} (u, s) \big) ds dt; \\
			\widetilde{F}_{\gamma, 2} (u, x, y) & = \displaystyle\int_0^{\infty} \displaystyle\int_{-\infty}^{\infty} e^{-2t \gamma / \alpha} \psi_0 (u, s)^{2\gamma / \alpha} h_0 (s - \log x) h_0 \big( t - \log y - \log \psi_{\alpha/2} (u, s) \big) ds dt,
		\end{flalign*}
		
		\noindent we have 
		\begin{flalign}
			\label{ff1f20}
			\begin{aligned}
				\big|  \breve{F}_{\gamma} (u^{2/\alpha}, x, y) -  \widetilde{F}_{\gamma} (u, x, y) \big| & \le \big|  \breve{F}_{\gamma} (u^{2/\alpha}, x, y) -  \widetilde{F}_{\gamma, 1} (u, x, y) \big|  + \big|  \widetilde{F}_{\gamma, 1} (u, x, y) -  \widetilde{F}_{\gamma, 2} (u, x, y) \big| \\
				& \qquad + \big|  \widetilde{F}_{\gamma, 2} (u, x, y) - \widetilde{F}_{\gamma} (u, x, y) \big|. 
			\end{aligned}
		\end{flalign} 
		
		To bound the first term on the right side of \eqref{ff1f20}, observe since $\psi_{\alpha/2} (u, s) \le e^{-s}$ that 
		\begin{flalign}
			\label{fgammafgamma10} 
			\begin{aligned} 
				\big|  \breve{F}_{\gamma} (u^{2/\alpha}, x, y) -  \widetilde{F}_{\gamma, 1} (u, x, y) \big| & \le \displaystyle\int_0^{\infty} \displaystyle\int_{-\infty}^{\infty} e^{-2s\gamma / \alpha} h_0 (s - \log x) \cdot \big| (e^{2t/\alpha} - 1)^{-\gamma} - e^{-2t\gamma / \alpha} \big| \\
				& \qquad \qquad \qquad \qquad \times  h_0 \big( t - \log y + \log \psi_{\alpha/2} (u, s) \big) ds dt.
			\end{aligned} 
		\end{flalign}
		
		\noindent By \Cref{talpha0} and the fact that $\sup_{w \in \mathbb{R}} h_0 (w) \le C$, we have  
		\begin{flalign*}
			\begin{aligned} 
				\displaystyle\int_0^{\infty} & \big| (e^{2t/\alpha} - 1)^{-\gamma} - e^{-2t\gamma / \alpha} \big| \cdot h_0 \big( t - \log y + \log \psi_{\alpha/2} (u, s) \big) dt \\
				& \le C \displaystyle\int_0^{\infty} \big( |t|^{-\gamma} \cdot \mathbbm{1}_{t \le \alpha^3} + \alpha |\log \alpha| \cdot \mathbbm{1}_{t \le \alpha} + \alpha e^{-2t / \alpha} \cdot \mathbbm{1}_{t > \alpha} \big) dt \le C \alpha^2 |\log \alpha|,
			\end{aligned}
		\end{flalign*}
		
		\noindent which, together with \eqref{fgammafgamma10} and the change of variables (sending $s$ to $s + \log x$), implies 
		\begin{flalign}
			\label{ffgamma10}
			\big|  \breve{F}_{\gamma} (u^{2/\alpha}, x, y) - \widetilde{F}_{\gamma, 1} (u, x, y) \big| \le \displaystyle\int_{-\infty}^{\infty} e^{-2s\gamma / \alpha} h_0 (s - \log x) ds \le  \displaystyle\frac{C \alpha^2}{x^{2 \gamma / \alpha}}.
		\end{flalign}
		
		To bound the second term in \eqref{ff1f20}, observe from \Cref{integralk} that 
		\begin{flalign*} 
			\big|  \widetilde{F}_{\gamma, 1} (u, x, y) -  \widetilde{F}_{\gamma, 2} (u, x, y) \big| &  \le \displaystyle\int_0^{\infty} \displaystyle\int_{-\infty}^{\infty} e^{-2t \gamma / \alpha}  \cdot h_0 \big(t - \log y + \log \psi_{\alpha/2} (u, s) \big) \\
			& \qquad \qquad \quad \times h_0 (s - \log x) \cdot \big| \psi_{\alpha/2} (u, s)^{2 \gamma / \alpha} - \psi_0 (u, s)^{2\gamma / \alpha} \big| ds dt \\
			& \qquad \le C \displaystyle\int_{-\infty}^{\infty} h_0 (s-\log x) \cdot \big| \psi_{\alpha/2} (u, s)^{2\gamma / \alpha} - \psi_0 (u, s)^{2 \gamma / \alpha} \big| ds dt.
		\end{flalign*} 
		
		\noindent This; \Cref{exponentialsu}; the bound $\psi_0 (u, s) \le e^{-s}$; the fact that $\varpi^{2/\alpha} \le C \alpha^2$ (recall \eqref{su0}) when $|s - \log u| \ge C \alpha |\log \alpha|$ (as then $\varpi \le 1 - C \alpha |\log \alpha| + O \big( \alpha^2 |\log \alpha| \big)$); the estimate $\sup_{w \in \mathbb{R}} h_0 (w) \le C$ then together yield
		\begin{flalign}
			\label{fgamma1gamma0} 
			\begin{aligned}
				\big| &  \widetilde{F}_{\gamma, 1} (u, x, y) -  \widetilde{F}_{\gamma, 2} (u, x, y) \big| \\
				& \le C \alpha \displaystyle\int_{|s - \log u| \le C \alpha |\log \alpha|} \psi_0 (u, s)^{2\gamma / \alpha}  h_0 (s - \log x) ds + C \alpha^2 \displaystyle\int_{-\infty}^{\infty} \psi_0 (u, s)^{2\gamma / \alpha} h_0 (s - \log x) ds \\
				& \le C \alpha \displaystyle\int_{|s - \log u| \le C \alpha |\log \alpha|} e^{-2 s \gamma / \alpha} h_0 (s - \log x) ds + C \alpha^2 \displaystyle\int_{-\infty}^{\infty} e^{-2 s \gamma / \alpha} h_0 (s - \log x) ds \\
				& \le \displaystyle\frac{C \alpha}{x^{2 \gamma / \alpha}} \displaystyle\int_{|s - \log u + \log x| \le C \alpha |\log \alpha|} e^{-2s \gamma / \alpha} h_0 (s) ds + \displaystyle\frac{C \alpha^2}{x^{2 \gamma / \alpha}} \displaystyle\int_{-\infty}^{\infty} e^{-2s \gamma / \alpha} h_0 (s) ds  \le \displaystyle\frac{C \alpha^2 |\log \alpha|}{x^{2 \gamma / \alpha}}.
			\end{aligned}
		\end{flalign}
		
		To bound the third term in \eqref{ff1f2}, we first use the fact that $h_0$ is $1$-Lipschitz to find that 
		\begin{flalign*}
			\big|  \widetilde{F}_{\gamma, 2} & (u, x, y) - \widetilde{F}_{\gamma} (u, x, y) \big| \\ 
			& \le \displaystyle\int_0^{\infty} \displaystyle\int_{-\infty}^{\infty} e^{-2t \gamma / \alpha} \cdot \psi_0 (u, s)^{2 \gamma / \alpha} \cdot h_0 (s - \log x) | \cdot \big| \log \psi_{\alpha/2} (u, s) - \log \psi_0 (u, s) \big| ds dt \\ 
			& = \displaystyle\frac{\alpha}{2 \gamma} \displaystyle\int_{-\infty}^{\infty} \psi_0 (u, s)^{2\gamma / \alpha} \cdot h_0 (s - \log x) \cdot \big| \log \psi_{\alpha/2} (u, s) - \log \psi_0 (u, s) \big|ds.
		\end{flalign*}
		
		\noindent By \Cref{exponentialsu}, the bound $\varpi^{2/\alpha} \le C \alpha^2$ (recall \eqref{su0}) when $|s - \log u| \ge C \alpha |\log \alpha|$ (by the above), and the fact that $\psi_0 (u, s) \le e^{-s}$, it follows analogously to in \eqref{fgamma1gamma0} that 
		\begin{flalign}
			\label{fgamma2gamma0} 
			\big|  \widetilde{F}_{\gamma, 2} (u, x, y) - \widetilde{F}_{\gamma} (u, x, y) \big| \le \displaystyle\frac{C \alpha}{\gamma} \displaystyle\int_{-\infty}^{\infty} e^{-2s \gamma / \alpha} h_0 (s - \log x) \cdot \varpi^{2/\alpha} ds \le \displaystyle\frac{C \alpha^2 |\log \alpha|}{x^{2 \gamma / \alpha}}.
		\end{flalign}
		
		\noindent The proposition then follows from combining \eqref{ff1f20}, \eqref{ffgamma10}, \eqref{fgamma1gamma0}, and \eqref{fgamma2gamma0}.
	\end{proof} 
	
	\subsection{Scaling of the Mobility Edge}
	
	\label{Alpha0E0Scale}
	
	In this section we establish the scaling for a mobility edge $E_0$ if $\alpha$ is small; in what follows, we recall $a(E)$ and $b(E)$ from \Cref{opaque}.

	\begin{thm}
		
		\label{e0alpha0u} 
		
		There exist constants $c > 0$ and $C > 1$ such that the following holds for $\alpha \in (0, c)$. Letting $E_0 > 0$ be some real number such that $\lambda (E_0, \alpha) = 1$, we have
		\begin{flalign*} 
			\bigg( \displaystyle\frac{1}{|\log \alpha|} - \displaystyle\frac{C \log |\log \alpha|}{|\log \alpha|^2} \bigg)^{2/\alpha} \le E_0 \le \bigg( \displaystyle\frac{1}{|\log \alpha|} + \displaystyle\frac{C \log |\log \alpha|}{|\log \alpha|^2} \bigg)^{2/\alpha}.
		\end{flalign*}
		
	\end{thm} 
	
	\begin{proof}[Proof of \Cref{e0alpha0u}]
		
		By \eqref{fba}, \Cref{gammaffgg0}, \Cref{fgu}, \Cref{lambdaalpha01}, and \Cref{abestimatealpha0}, we have 
		\begin{flalign}
			\label{abalpha0} 
			\begin{aligned}
				a & = \displaystyle\frac{b}{(a+b)^2} - e^{-(a+b)/u} \bigg( \displaystyle\frac{b}{(a+b)^2} + \displaystyle\frac{b}{u(a+b)} \bigg) + O \big( \alpha^2 |\log \alpha|^2 \big); \\ 
				b & = \displaystyle\frac{a}{(a+b)^2} + e^{-(a+b)/u} \bigg( \displaystyle\frac{b}{u(a+b)} - \displaystyle\frac{a}{(a+b)^2} \bigg) + O \big( \alpha^2 |\log \alpha|^2 \big),
			\end{aligned}  
		\end{flalign} 
		
		\noindent where we have set $a = a(E_0)$, $b = b(E_0)$, and $E_0 = u^{2/\alpha}$.  Summing these two equations, and setting $d = d(E_0) = a + b$, we find
		\begin{flalign}
			\label{d1alpha0}
			d = d^{-1} ( 1 - e^{-d/u}) + O \big( \alpha^2 |\log \alpha|^2 \big).
		\end{flalign}
		
		\noindent Moreover, applying \Cref{lambdaalpha0}, \eqref{fggamma}, \Cref{gammaffgg0}, and \Cref{falphagalpha0} yields
		\begin{flalign}
			\label{d2alpha0}
			\begin{aligned}
				2 - c_{\star} \alpha + O (\alpha^2) = \mathbb{E} \Big[ \big| R(E_0) \big|^{\alpha} \Big] & = F_{\alpha} (E_0, a, b) + G_{\alpha} (E_0, a, b) \\
				& = 2d^{-2} - 2 e^{-d/u} ( u^{-1} d^{-1} + d^{-2}) + O \big( \alpha^2 |\log \alpha|^2 \big).
			\end{aligned} 
		\end{flalign}
		
		\noindent where $c_{\star} = 4 \log 2 + \pi$. Inserting \eqref{d1alpha0} into \eqref{d2alpha0}, and using the fact that $c \le d \le C$ (by \Cref{abestimatealpha0}) yields 
		\begin{flalign}
			\label{d3alpha0} 
			e^{-d/u} u^{-1} =  c_{\star} \alpha d + O \big( \alpha^2 |\log \alpha|^2 \big).
		\end{flalign}
		
		\noindent Again since $c \le d \le C$ (by \Cref{abestimatealpha0}), this implies that $c |\log \alpha|^{-1} \le u \le C |\log \alpha|^{-1}$. Inserting this bound into \eqref{d1alpha0}, it follows that $|d-1| \le C \alpha^c$, which by \eqref{d3alpha0} yields 
		\begin{flalign*}
			\displaystyle\frac{1}{|\log \alpha|} - \displaystyle\frac{C \log |\log \alpha|}{|\log \alpha|^2} \le u \le \displaystyle\frac{1}{|\log \alpha|} + \displaystyle\frac{C \log |\log \alpha|}{|\log \alpha|^2},
		\end{flalign*}
		
		\noindent which establishes the theorem.
	\end{proof}

	We further include the following lemma, which will be useful in \Cref{E0Unique} below.

	\begin{lem}
		
		\label{abdalpha0} 
		
		There exist constants $c > 0$ and $C > 1$ such that the following holds for $\alpha \in (0, c)$. Let $E_0 > 0$ be a real number; abbreviate $a = a(E_0)$ and $b = b(E_0)$; and set $u = E_0^{\alpha/2}$. If 
		\begin{flalign}
			\label{ualpha0} 
			\displaystyle\frac{9}{10} \cdot |\log \alpha|^{-1} \le u \le \displaystyle\frac{11}{10} \cdot |\log \alpha|^{-1},
		\end{flalign}
		
		\noindent then 
		\begin{flalign*} 
			\bigg| a - \displaystyle\frac{1}{2}  \bigg| \le C \alpha^{8/9}; \qquad \bigg|b - \displaystyle\frac{1}{2} \bigg| \le C \alpha^{8/9}.
		\end{flalign*} 
	\end{lem}
	
	\begin{proof}
		
		Throughout this proof, we once again set $d = a + b$. Applying \eqref{fba}, \Cref{gammaffgg0}, \Cref{fgu}, \eqref{ualpha0}, and \eqref{abestimatealpha0} yields \eqref{abalpha0}. This yields \eqref{d1alpha0}, which with \Cref{abestimatealpha0} (and \eqref{ualpha0}), gives $|d-1| \le C \alpha^c$; inserting this again into \eqref{d1alpha0} and applying \eqref{ualpha0}, we obtain $|a + b - 1| \le C \alpha^{8/9}$. Applying these bounds in \eqref{abalpha0} yields
		\begin{flalign*}
			a & = b \big( 1 + O (\alpha^{8/9}) \big) + O \big( \alpha^{8/9} \big); \\
			b \big( 1 - O ( \alpha^{8/9}) \big) & = a \big( 1 + O(\alpha^{8/9})  \big) + O \big( \alpha^{8/9} \big).
		\end{flalign*}
		
		\noindent This, together with the fact that $|a + b - 1| \le C \alpha^{8/9}$, implies the lemma.
	\end{proof}

	\section{Uniqueness Near Zero} 
	
	\label{E0Unique}
	
	In this section we establish the second part of \Cref{t:main2}, showing for $\alpha$ small that there exists a unique real number $E_{\mob} > 0$ such that $\lambda (E_{\mob}, \alpha) = 1$. As in \Cref{Alpha0Scaling}, this will proceed by approximating the functions $F$ and $G$ from \eqref{fggamma} by more explicit quantities, though now we must also approximate their derivatives. Throughout this section, we adopt the notation and conventions described at the beginning of \Cref{Alpha0Scaling}. We further assume (which will be justified by \Cref{e0alpha0u} and \Cref{abdalpha0}) in the below that 
	\begin{flalign}
		\label{xyualpha0} 
		\gamma \in \bigg[ \frac{\alpha}{2}, \alpha \bigg]; \qquad x, y \in \bigg[ \displaystyle\frac{1}{4}, \displaystyle\frac{3}{4} \bigg]; \qquad \displaystyle\frac{99}{100} \cdot |\log \alpha|^{-1} \le u \le \displaystyle\frac{101}{100} \cdot |\log \alpha|^{-1},
	\end{flalign}
	
	\noindent which will be in effect even when not stated explicitly. In \Cref{Replaceh0halpha2} we estimate the error in replacing the derivatives of $(F, G)$ with those of $(\breve{F}, \breve{G})$ (recall \Cref{hhfg}); in \Cref{FFGGalpha0Estimate} we estimate the error in replacing the derivatives of $(\breve{F}, \breve{G})$ by $(\widetilde{F}, \widetilde{G})$ (recall \Cref{fuxyguxy2}). We then bound the derivatives of $a$ and $b$ (recall \Cref{opaque}) in \Cref{EstimateAlpha0ab} and establish the second part of \Cref{t:main2} in \Cref{Alpha0Unique}. Throughout this section, constants $c > 0$ and $C > 1$ will be independent of $\alpha$.

	\subsection{Replacement of $(F, G)$ With $(\breve{F}, \breve{G})$}
	
	\label{Replaceh0halpha2} 
	
	In this section we establish the below lemma stating that the first derivatives of $(F_{\gamma}, G_{\gamma})$ can be approximated by those of $(\breve{F}_{\gamma}, \breve{G}_{\gamma})$, where we recall the definition of the latter from \Cref{hhfg}. 
	
	\begin{lem} 
		
		\label{ffgg} 
		
		There exists a constant $C > 1$ such that, for any indices $\mathfrak{z} \in \{ u, x, y \}$ and $\mathfrak{Y} \in \{ F, G \}$, we have 
		\begin{flalign*}
			\big| \partial_{\mathfrak{z}} \ \mathfrak{Y}_{\gamma} (u^{2/\alpha}, x, y) - \partial_{\mathfrak{z}} \breve{\mathfrak{Y}}_{\gamma}  (u^{2/\alpha}, x, y) \big| \le \displaystyle\frac{C \alpha^2}{\mathfrak{z} u^2}.
		\end{flalign*}
	\end{lem} 
	
	\begin{proof}
		
		We only analyze the case $\mathfrak{Y} = F$, as the proof is entirely analogous if $\mathfrak{Y} = G$. Let us first address the case when $\mathfrak{z} \in \{ x, y \}$; these situations are very similar, so we only consider $\mathfrak{z} = x$. By \eqref{fgammaintegral2alpha00}, we have
		\begin{flalign*}
			\partial_x F_{\gamma} (u^{2/\alpha}, x, y) & = - \displaystyle\frac{1}{x} \displaystyle\int_0^{\infty} \displaystyle\int_{-\infty}^{\infty} (u^{2/\alpha} + e^{2s/\alpha})^{-\gamma} (e^{2t/\alpha} - 1)^{-\gamma} h_{\alpha/2}' (s - \log x) \\
			& \qquad \qquad \qquad \quad  \times h_{\alpha/2} \bigg(t - \log y + \displaystyle\frac{\alpha}{2} \log (u^{2/\alpha} + e^{2s/\alpha}) \bigg) dt ds,
		\end{flalign*} 
		
		\noindent and, by similar reasoning, 
		\begin{flalign*}
			\partial_x \breve{F}_{\gamma} (u^{2/\alpha}, x, y) & = - \displaystyle\frac{1}{x} \displaystyle\int_0^{\infty} \displaystyle\int_{-\infty}^{\infty} (u^{2/\alpha} + e^{2s/\alpha})^{-\gamma} (e^{2t/\alpha} - 1)^{-\gamma} h_0' (s - \log x) \\
			& \qquad \qquad \qquad \quad  \times h_0 \bigg(t - \log y + \displaystyle\frac{\alpha}{2} \log (u^{2/\alpha} + e^{2s/\alpha}) \bigg) ds dt.
		\end{flalign*}
		
		\noindent Subtracting; using the bound $(u^{2/\alpha} + e^{2s/\alpha})^{-\gamma} \le u^{-2\gamma / \alpha} \le u^{-2}$ (as $\gamma \le \alpha$); and recalling $\Upsilon_{\alpha/2}$ from \Cref{halpha} yields
		\begin{flalign*}
			\big| & \partial_x F_{\gamma} (u^{2/\alpha}, x, y) - \partial_x \breve{F}_{\gamma} (u^{2/\alpha}, x, y) \big| \\
			&   \le \displaystyle\frac{\alpha^2}{xu^2}  \displaystyle\int_{-\infty}^{\infty} \displaystyle\int_{-\infty}^{\infty} (e^{2t/\alpha} - 1)^{-\gamma} \Bigg| h_0' (s - \log x) \Upsilon_{\alpha/2} \bigg(t - \log y + \displaystyle\frac{\alpha}{2} \log (u^{2/\alpha} + e^{2s/\alpha}) \bigg) \Bigg| ds dt \\
			& \quad   + \displaystyle\frac{\alpha^2}{xu^2}  \displaystyle\int_{-\infty}^{\infty} \displaystyle\int_{-\infty}^{\infty} (e^{2t/\alpha} - 1)^{-\gamma} \big| \Upsilon_{\alpha/2}' (s-\log x) \big| h_0 \bigg(t - \log y + \displaystyle\frac{\alpha}{2} \log (u^{2/\alpha} + e^{2s/\alpha}) \bigg) ds dt \\
			&   \quad + \displaystyle\frac{\alpha^4}{xu^2} \displaystyle\int_{-\infty}^{\infty} \displaystyle\int_{-\infty}^{\infty} (e^{2t/\alpha} - 1)^{-\gamma} \Bigg| \Upsilon_{\alpha/2}' (s - \log x) \Upsilon_{\alpha/2} \bigg(t - \log y + \displaystyle\frac{\alpha}{2} \log (u^{2/\alpha} + e^{2s/\alpha}) \bigg) \Bigg| ds dt.
		\end{flalign*} 
		
		\noindent Using \Cref{integralk} to integrate in $t$, it follows that 
		\begin{flalign*}
			\big| \partial_x & F_{\gamma} (u^{2/\alpha}, x, y) - \partial_x \breve{F}_{\gamma} (u^{2/\alpha}, x, y) \big| \\
			& \le \displaystyle\frac{C \alpha^2}{xu^2} \displaystyle\int_{-\infty}^{\infty} \Big( \big| h_0' (s - \log x) \big| + \big| \Upsilon_{\alpha/2} (s - \log x) \big| + \big| \Upsilon_{\alpha/2}' (s - \log x) \big| \Big) ds \le \displaystyle\frac{C \alpha^2}{xu^2},
		\end{flalign*}
		
		\noindent where in the last bound we applied \Cref{w0walpha}. This establishes the $z = x$ case of the lemma. 
		
		Next, we address the case $\mathfrak{z} = u$. Differentiating \eqref{fgammaintegral2alpha00} in $u$, we obtain
		\begin{flalign*}
			\partial_u F_{\gamma} (u^{2/\alpha}, x, y) & =  \displaystyle\int_0^{\infty} \displaystyle\int_{-\infty}^{\infty} u^{2/\alpha-1} (u^{2/\alpha} + e^{2s/\alpha})^{-\gamma-1} (e^{2t/\alpha} - 1)^{-\gamma} h_{\alpha/2} (s - \log x) \\
			& \qquad \qquad \quad \times \Bigg( h_{\alpha/2}' \bigg( t - \log y + \displaystyle\frac{\alpha}{2} \log (u^{2/\alpha} + e^{2s/\alpha}) \bigg) \\
			& \qquad \qquad \qquad \qquad - \displaystyle\frac{2\gamma}{\alpha} \cdot h_{\alpha/2} \bigg( t - \log y + \displaystyle\frac{\alpha}{2} \log (u^{2/\alpha} + e^{2s/\alpha}) \bigg) ds dt.
		\end{flalign*}
		
		\noindent By similar reasoning,  
		\begin{flalign*}
			\partial_u \breve{F}_{\gamma} (u^{2/\alpha}, x, y) & = \displaystyle\int_0^{\infty} \displaystyle\int_{-\infty}^{\infty} u^{2/\alpha - 1} (u^{2/\alpha} + e^{2s/\alpha})^{-\gamma-1} (e^{2t/\alpha} - 1)^{-\gamma} h_0 (s - \log x)  \\
			& \qquad \qquad \quad \times \Bigg( h_0' \bigg( t - \log y + \displaystyle\frac{\alpha}{2} \log (u^{2/\alpha} + e^{2s/\alpha}) \bigg) \\
			& \qquad \qquad \qquad \qquad - \displaystyle\frac{2\gamma}{\alpha} \cdot h_0 \bigg( t - \log y + \displaystyle\frac{\alpha}{2} \log (u^{2/\alpha} + e^{2s/\alpha}) \bigg) \Bigg) ds dt.
		\end{flalign*}
		
		\noindent Subtracting, and using the bound
		\begin{flalign*}	
			u^{2/\alpha-1} (u^{2/\alpha} + e^{2s/\alpha})^{-\gamma-1} \le u^{-2\gamma/\alpha-1} \le u^{-3},
		\end{flalign*} 
		
		\noindent  it follows that for $\gamma \le \alpha$ that
		\begin{flalign*}
			\big| & \partial_u F_{\gamma} (u^{2/\alpha}, x, y) - \partial_u \breve{F}_{\gamma} (u^{2/\alpha}, x, y) \big| \\
			& \le \displaystyle\frac{\alpha^2}{u^3} \displaystyle\int_0^{\infty} \displaystyle\int_{-\infty}^{\infty} (e^{2t/\alpha}-1)^{-\gamma} \bigg| \Upsilon_{\alpha/2} (s-\log x) h_{\alpha/2}' \Big(t - \log y + \displaystyle\frac{\alpha}{2} \log (u^{2/\alpha} + e^{2s/\alpha}) \Big) \bigg| ds dt \\
			& \quad + \displaystyle\frac{\alpha^2}{u^3} \displaystyle\int_0^{\infty} \displaystyle\int_{-\infty}^{\infty} (e^{2t/\alpha}-1)^{-\gamma} h_0 (s-\log x) \bigg| \Upsilon_{\alpha/2}' \Big(t - \log y + \displaystyle\frac{\alpha}{2} \log (u^{2/\alpha} + e^{2s/\alpha}) \Big) \bigg| ds dt \\	
			& \quad + \displaystyle\frac{2 \alpha^2}{u^3} \displaystyle\int_0^{\infty} \displaystyle\int_{-\infty}^{\infty} (e^{2t/\alpha}-1)^{-\gamma} \big| \Upsilon_{\alpha/2} (s - \log x) \big| h_0  \Big( t - \log y + \displaystyle\frac{\alpha}{2} \log (u^{2/\alpha} + e^{2s/\alpha}) \Big) ds dt \\	
			& \quad + \displaystyle\frac{2 \alpha^2}{u^3} \displaystyle\int_0^{\infty} \displaystyle\int_{-\infty}^{\infty} (e^{2t/\alpha}-1)^{-\gamma} h_0 (s-\log x) \bigg| \Upsilon_{\alpha/2} \Big(t - \log y + \displaystyle\frac{\alpha}{2} \log (u^{2/\alpha} + e^{2s/\alpha}) \Big) \bigg| ds dt \\
			& \quad + \displaystyle\frac{\alpha^4}{u^3} \displaystyle\int_0^{\infty} \displaystyle\int_{-\infty}^{\infty} (e^{2t/\alpha}-1)^{-\gamma} \bigg| \Upsilon_{\alpha/2} (s-\log x) \Upsilon_{\alpha/2}' \Big(t - \log y + \displaystyle\frac{\alpha}{2} \log (u^{2/\alpha} + e^{2s/\alpha}) \Big) \bigg| ds dt \\
			& \quad +  \displaystyle\frac{2 \alpha^4}{u^3} \displaystyle\int_0^{\infty} \displaystyle\int_{-\infty}^{\infty} (e^{2t/\alpha}-1)^{-\gamma} \bigg| \Upsilon_{\alpha/2} (s-\log x) \Upsilon_{\alpha/2} \Big(t - \log y + \displaystyle\frac{\alpha}{2} \log (u^{2/\alpha} + e^{2s/\alpha}) \Big) \bigg| ds dt.
		\end{flalign*}
		
		\noindent Again using \Cref{integralk} to first integrate with respect to $t$, and then \Cref{w0walpha} to integrate with respect to $s$, we obtain
		\begin{flalign*}
			\big| \partial_u F_{\gamma} (u^{2/\alpha}, x, y) - \partial_u \breve{F}_{\gamma} (u^{2/\alpha}, x, y) \big| & \le \displaystyle\frac{C \alpha^2}{u^3} \displaystyle\int_{-\infty}^{\infty} \Big( h_0 (s - \log x) + \big| \Upsilon_{\alpha/2} (s - \log x) \big| \Big) ds \le \displaystyle\frac{C \alpha^2}{u^3}.
		\end{flalign*}
		
		\noindent thus establishing the lemma.
	\end{proof}

	\subsection{Replacement of $(\breve{F}, \breve{G})$ With $(\widetilde{F}, \widetilde{G})$}
	
	\label{FFGGalpha0Estimate} 
	
	In this section we establish the following two propositions, indicating that one may replace derivatives of $(F, G)$ with those of $(\widetilde{F}, \widetilde{G})$ (whose definitions we recall from \Cref{fuxyguxy2}).

	\begin{prop}
		
		\label{gammaffgg}
		
		There exists a constant $C > 1$ such that, for any indices $\mathfrak{z} \in \{ x, y \}$ and $\mathfrak{Y} \in \{ F, G \}$, we have 
		\begin{flalign*}
			\big| \partial_{\mathfrak{z}}^k \mathfrak{Y}_{\gamma} (u^{2/\alpha}, x, y) - \partial_{\mathfrak{z}}^k \widetilde{\mathfrak{Y}}_{\gamma} (u, x, y) \big| \le  \displaystyle\frac{C \alpha^2 |\log \alpha|}{\mathfrak{z} u^2}.
		\end{flalign*} 
	\end{prop}
	
	\begin{prop}
		
		\label{gammaffggu} 
		
		There exists a constant $C > 1$ such that, for any index $\mathfrak{Y} \in \{ F, G \}$, we have 
		\begin{flalign*}
			\big| \partial_u \mathfrak{Y}_{\gamma} (u^{2/\alpha},x, y) - \partial_u \widetilde{\mathfrak{Y}}_{\gamma} (u^{2/\alpha}, x, y) \big| \le C \alpha^{7/6}.
		\end{flalign*}
	\end{prop}

	\begin{proof}[Proof of \Cref{gammaffgg}]
		In view of \Cref{ffgg}, it suffices to show that 
		\begin{flalign}
			\label{hzk0} 
			\big| \partial_{\mathfrak{z}} \breve{\mathfrak{Y}}_{\gamma} (u^{2/\alpha}, x, y) - \partial_{\mathfrak{z}} \widetilde{\mathfrak{Y}}_{\gamma} (u, x, y) \big| \le \displaystyle\frac{C \alpha^2 |\log \alpha|}{\mathfrak{z} u^2}.
		\end{flalign}
		
		\noindent We again  assume $\mathfrak{Y} = F)$ and only address the case $\mathfrak{z} = x$, as the proofs in the remaining situations are entirely analogous (for the proof when $H = G$, one uses the second bound in \eqref{talphae} instead of the first below). Following \eqref{fgammaintegral2alpha00} and using \eqref{psius}, we find 
		\begin{flalign}
			\label{fgammafgamma} 
			\begin{aligned}
				\breve{F}_{\gamma} (u^{2/\alpha}, x, y) & = \displaystyle\int_0^{\infty} \displaystyle\int_{-\infty}^{\infty} (e^{2t/\alpha} - 1)^{-\gamma} \psi_{\alpha/2} (u, s)^{2 \gamma / \alpha} h_0 (s - \log x) \\
				& \qquad \qquad \qquad \times h_0 \big( t - \log y + \log \psi_{\alpha/2} (u,s) \big) ds dt; \\
				\widetilde{F}_{\gamma} (u, x, y) & = \displaystyle\int_0^{\infty} \displaystyle\int_{-\infty}^{\infty} e^{-2t\gamma / \alpha} \psi_0 (u, s)^{2\gamma / \alpha} h_0 (s - \log x) h_0 \big( t - \log y - \log \psi_0 (u, s) \big) ds dt.
			\end{aligned} 
		\end{flalign}
		
		\noindent Further setting
		\begin{flalign*}
			\widetilde{F}_{\gamma, 1} (u, x, y) & = \displaystyle\int_0^{\infty} \displaystyle\int_{-\infty}^{\infty} e^{-2t \gamma/\alpha} \psi_{\alpha/2} (u, s)^{2\gamma / \alpha} h_0 (s - \log x) h_0 \big( t - \log y - \log \psi_{\alpha/2} (u, s) \big) ds dt; \\
			\widetilde{F}_{\gamma, 2} (u, x, y) & = \displaystyle\int_0^{\infty} \displaystyle\int_{-\infty}^{\infty} e^{-2t \gamma / \alpha} \psi_0 (u, s)^{2\gamma / \alpha} h_0 (s - \log x) h_0 \big( t - \log y - \log \psi_{\alpha/2} (u, s) \big) ds dt,
		\end{flalign*}
		
		\noindent we have 
		\begin{flalign}
			\label{ff1f2}
			\begin{aligned}
				\big| \partial_x \breve{F}_{\gamma} (u^{2/\alpha}, x, y) - \partial_x \widetilde{F}_{\gamma} (u, x, y) \big| & \le \big| \partial_x \breve{F}_{\gamma} (u^{2/\alpha}, x, y) - \partial_x \widetilde{F}_{\gamma, 1} (u, x, y) \big| \\
				& \qquad  + \big| \partial_x \widetilde{F}_{\gamma, 1} (u, x, y) - \partial_x \widetilde{F}_{\gamma, 2} (u, x, y) \big| \\
				& \qquad + \big| \partial_x \widetilde{F}_{\gamma, 2} (u, x, y) - \partial_x \widetilde{F}_{\gamma} (u, x, y) \big|. 
			\end{aligned}
		\end{flalign} 
		
		To bound the first term on the right side of \eqref{ff1f2}, observe from the bounds $\psi_{\alpha/2} (u, s) \le u^{-1}$ and $\gamma \le \alpha$ that 
		\begin{flalign}
			\label{fgammafgamma1} 
			\begin{aligned} 
				\big| \partial_x \breve{F}_{\gamma} (u^{2/\alpha}, x, y) - \partial_x \widetilde{F}_{\gamma, 1} (u, x, y) \big| & \le \displaystyle\frac{1}{xu^2} \displaystyle\int_0^{\infty} \displaystyle\int_{-\infty}^{\infty} \big| h_0' (s - \log x) \big| \cdot \big| (e^{2t/\alpha} - 1)^{-\gamma} - e^{-2t\gamma / \alpha} \big| \\
				& \qquad \qquad \qquad \qquad \times  h_0 \big( t - \log y + \log \psi_{\alpha/2} (u, s) \big) ds dt.
			\end{aligned} 
		\end{flalign}
		
		\noindent By \Cref{talpha0} and the fact that $\sup_{w \in \mathbb{R}} h_0 (w) \le C$, we have  
		\begin{flalign}
			\label{t2alpha1t2alpha} 
			\begin{aligned} 
				\displaystyle\int_0^{\infty} & \big| (e^{2t/\alpha} - 1)^{-\gamma} - e^{-2t\gamma / \alpha} \big| \cdot h_0 \big( t - \log y + \log \psi_{\alpha/2} (u, s) \big) dt \\
				& \le C \displaystyle\int_0^{\infty} \big( |t|^{-\gamma} \cdot \mathbbm{1}_{t \le \alpha^3} + \alpha |\log \alpha| \cdot \mathbbm{1}_{t \le \alpha} + \alpha e^{-2t / \alpha} \cdot \mathbbm{1}_{t > \alpha} \big) dt \le C \alpha^2 |\log \alpha|,
			\end{aligned}
		\end{flalign}
		
		\noindent which together with \Cref{integralk} and \eqref{fgammafgamma1} shows that 
		\begin{flalign}
			\label{ffgamma1}
			\big| \partial_x \breve{F}_{\gamma} (u^{2/\alpha}, x, y) - \partial_x \widetilde{F}_{\gamma, 1} (u, x, y) \big| \le \displaystyle\frac{C \alpha^2 |\log \alpha|}{x u^2}.
		\end{flalign}
		
		To bound the second term in \eqref{ff1f2}, observe from \Cref{integralk} that 
		\begin{flalign*} 
			& \big| \partial_x \widetilde{F}_{\gamma, 1} (u, x, y) - \partial_x \widetilde{F}_{\gamma, 2} (u, x, y) \big| \\
			& \qquad  \le \displaystyle\frac{1}{x} \displaystyle\int_0^{\infty} \displaystyle\int_{-\infty}^{\infty} e^{-2t \gamma / \alpha}  \cdot h_0 \big(t - \log y + \log \psi_{\alpha/2} (u, s) \big) \\
			& \qquad \qquad \qquad \qquad \times \big| h_0' (s - \log x) \big| \cdot \big| \psi_{\alpha/2} (u, s)^{2 \gamma / \alpha} - \psi_0 (u, s)^{2\gamma / \alpha} \big| ds dt \\
			& \qquad \le \displaystyle\frac{C}{x} \displaystyle\int_{-\infty}^{\infty} \big| h_0' (s-\log x) \big| \cdot \big| \psi_{\alpha/2} (u, s)^{2\gamma / \alpha} - \psi_0 (u, s)^{2 \gamma / \alpha} \big| ds dt.
		\end{flalign*} 
		
		\noindent This; \Cref{exponentialsu}; the bound $\psi_0 (u, s) \le u^{-1}$; the facts that $\gamma \le \alpha$ and $\varpi^{2/\alpha} \le C \alpha^2$ (recall \eqref{su0}) when $|s - \log u| \ge C \alpha |\log \alpha|$ (as then $\varpi \le 1 - C \alpha |\log \alpha| + O \big(\alpha^2 |\log \alpha| \big)$); and the estimate $\sup_{w \in \mathbb{R}} \big| h_0' (w) \big| \le C$ then together yield
		\begin{flalign}
			\label{fgamma1gamma} 
			\begin{aligned}
				\big| \partial_x \widetilde{F}_{\gamma, 1} (u, x, y) - \partial_x \widetilde{F}_{\gamma, 2} (u, x, y) \big| & \le \displaystyle\frac{C \alpha}{x} \displaystyle\int_{|s - \log u| \le C \alpha |\log \alpha|} \psi_0 (u, s)^{2\gamma / \alpha} \big| h_0' (s - \log x) \big| ds \\
				& \qquad + \displaystyle\frac{C \alpha^2}{x} \displaystyle\int_{-\infty}^{\infty} \psi_0 (u, s)^{2\gamma / \alpha} \big| h_0' (s - \log x) \big| ds \le \displaystyle\frac{C \alpha^2 |\log \alpha|}{x u^2}.
			\end{aligned}
		\end{flalign}
		
		To bound the third term in \eqref{ff1f2}, we first use the fact that $h_0$ is $1$-Lipschitz to find that 
		\begin{flalign*}
			\big| \partial_x \widetilde{F}_{\gamma, 2} & (u, x, y) - \partial_x \widetilde{F}_{\gamma} (u, x, y) \big| \\ 
			& \le \displaystyle\frac{1}{x} \displaystyle\int_0^{\infty} \displaystyle\int_{-\infty}^{\infty} e^{-2t \gamma / \alpha} \cdot \psi_0 (u, s)^{2 \gamma / \alpha} \cdot \big| h_0' (s - \log x) \big| \\ 
			& \qquad \qquad \qquad \times \big| \log \psi_{\alpha/2} (u, s) - \log \psi_0 (u, s) \big| ds dt \\ 
			& = \displaystyle\frac{\alpha}{2 \gamma x} \displaystyle\int_{-\infty}^{\infty} \psi_0 (u, s)^{2\gamma / \alpha} \cdot \big| h_0' (s - \log x) \big| \cdot \big| \log \psi_{\alpha/2} (u, s)^{2 \gamma / \alpha} - \log \psi_0 (u, s)^{2 \gamma / \alpha} \big|ds.
		\end{flalign*}
		
		\noindent By \Cref{exponentialsu}, the bound $\varpi^{2/\alpha} \le C \alpha^2$ (recall \eqref{su0}) when $|s - \log u| \ge C \alpha |\log \alpha|$ (by the above), and the fact that $\psi_0 (u, s) \le u^{-1}$, it follows analogously to in \eqref{fgamma1gamma} that 
		\begin{flalign}
			\label{fgamma2gamma} 
			\big| \partial_x \widetilde{F}_{\gamma, 2} (u, x, y) - \partial_x \widetilde{F}_{\gamma} (u, x, y) \big| \le \displaystyle\frac{C \alpha}{u^{2 \gamma / \alpha} x} \displaystyle\int_{-\infty}^{\infty} \big| h_0' (s - \log x) \big| \varpi^{2/\alpha} ds \le \displaystyle\frac{C \alpha^2 |\log \alpha|}{u^2 x}.
		\end{flalign}
		
		\noindent The proposition then follows from combining \eqref{ff1f2}, \eqref{ffgamma1}, \eqref{fgamma1gamma}, and \eqref{fgamma2gamma}.
	\end{proof} 
	
	\begin{proof}[Proof of \Cref{gammaffggu}]
		
		The proof of this proposition is similar to that of \Cref{gammaffgg}; the differences can be attributed to the fact that the third bound in \eqref{psialpha2psi0} lacks a power of $\alpha$ when compared to the first and second parts. To this address this, we will use the decay of $h_0$ and the assumption \eqref{xyualpha0}. First observe by \Cref{ffgg} (and \eqref{xyualpha0}) that it suffices to show for any index $\mathfrak{Y} \in \{ F, G \}$ that 
		\begin{flalign*} 
			\big| \partial_u \breve{\mathfrak{Y}}_{\gamma} (u^{2/\alpha}, x, y) - \partial_u \widetilde{\mathfrak{Y}} (u, x, y) \big| \le C \alpha^{7/6}.
		\end{flalign*} 
		
		\noindent We will only consider the case $\mathfrak{Y} = F$, as the case $\mathfrak{Y} = G$ is entirely analogous. By \eqref{fgammafgamma}, we have
		\begin{flalign*}
			\partial_u \breve{F}_{\gamma} (u^{2/\alpha}, x, y) & = \displaystyle\int_0^{\infty} \displaystyle\int_{-\infty}^{\infty} (e^{2t/\alpha} - 1)^{-\gamma} \cdot h_0 (s - \log x) \cdot \psi_{\alpha/2} (u)^{2 \gamma / \alpha - 1} \cdot  \partial_u \psi_{\alpha/2} (u) \\
			& \qquad \times \bigg( h_0' \big(t - \log y + \log \psi_{\alpha/2} (u, s) \big) + \displaystyle\frac{2 \gamma}{\alpha}  \cdot h_0 \big( t - \log y + \log \psi_{\alpha/2} (u, s) \big)  \bigg) ds dt;\\
			\partial_u \widetilde{F}_{\gamma} (u, x, y) & = \displaystyle\int_0^{\infty} \displaystyle\int_{-\infty}^{\infty} e^{-2t \gamma/\alpha} \cdot h_0 (s - \log x) \cdot \psi_0 (u)^{2 \gamma / \alpha - 1} \cdot \partial_u \psi_0 (u) \\
			& \qquad \quad \times \bigg( h_0' \big( t - \log y + \log \psi_0 (u, s) \big) + \displaystyle\frac{2\gamma}{\alpha} \cdot h_0 \big( t - \log y + \log \psi_0 (u, s) \big) \bigg) ds dt.		
		\end{flalign*}
		
		\noindent Further define
		\begin{flalign*}
			\mathfrak{X}_1 & = \displaystyle\int_0^{\infty} \displaystyle\int_{-\infty}^{\infty} e^{-2t \gamma /\alpha} \cdot h_0 (s - \log x) \cdot \psi_{\alpha/2} (u, s)^{2 \gamma / \alpha - 1} \cdot  \partial_u \psi_{\alpha/2} (u, s) \\
			& \qquad \times \bigg( h_0' \big(t - \log y + \log \psi_{\alpha/2} (u, s) \big) + \displaystyle\frac{2 \gamma}{\alpha}  \cdot h_0 \big( t - \log y + \log \psi_{\alpha/2} (u, s) \big)  \bigg) ds dt; \\
			\mathfrak{X}_2 & = \displaystyle\int_0^{\infty} \displaystyle\int_{-\infty}^{\infty} e^{-2 t \gamma/\alpha} \cdot h_0 (s - \log x) \cdot \psi_0 (u, s)^{2 \gamma / \alpha - 1} \cdot  \partial_u \psi_{\alpha/2} (u, s) \\
			& \qquad \times \bigg( h_0' \big(t - \log y + \log \psi_{\alpha/2} (u, s) \big) + \displaystyle\frac{2 \gamma}{\alpha}  \cdot h_0 \big( t - \log y + \log \psi_{\alpha/2} (u, s) \big)  \bigg) ds dt;\\
			\mathfrak{X}_3 & = \displaystyle\int_0^{\infty} \displaystyle\int_{-\infty}^{\infty} e^{-2 t \gamma/\alpha} \cdot h_0 (s - \log x) \cdot \psi_0 (u, s)^{2 \gamma / \alpha - 1} \cdot  \partial_u \psi_{\alpha/2} (u, s) \\
			& \qquad \times \bigg( h_0' \big(t - \log y + \log \psi_0 (u, s) \big) + \displaystyle\frac{2 \gamma}{\alpha}  \cdot h_0 \big( t - \log y + \log \psi_0 (u, s) \big)  \bigg) ds dt,
		\end{flalign*}
		
		\noindent and we have 
		\begin{flalign}
			\label{fx1x2x3f}
			\begin{aligned}
				\big| \partial_u \breve{F}_{\gamma} (u^{2/\alpha}, x, y) - \partial_u \widetilde{F}_{\gamma} (u, x, y) \big| & \le \big| \partial_u \breve{F}_{\gamma} (u^{2/\alpha}, x, y) - \mathfrak{X}_1 \big| + |\mathfrak{X}_1 - \mathfrak{X}_2| + |\mathfrak{X}_2 - \mathfrak{X}_3|  \\
				& \qquad +  \big| \mathfrak{X}_3 - \partial_u \widetilde{F}_{\gamma} (u, x, y) \big|.
			\end{aligned}
		\end{flalign}
		
		We will show that 
		\begin{flalign}
			\label{fx1x2x3}
			\begin{aligned} 
				\big| \partial_u \breve{F}_{\gamma} (u, x, y) - \mathfrak{X}_1 \big| \le C u^{-5} \alpha^2; & \qquad |\mathfrak{X}_1 - \mathfrak{X}_2|  \le C u^{-5} \alpha^2; \qquad |\mathfrak{X}_2 - \mathfrak{X}_3| \le C u^{-5} \alpha^2; \\
				&  \big| \mathfrak{X}_3 - \widetilde{F}_{\gamma} (u, x, y) \big| \le C \alpha^{7/6}.
			\end{aligned} 
		\end{flalign}
		
		\noindent The proofs of the first, second, and third bounds in \eqref{fx1x2x3} are similar to those of \eqref{ffgamma1}, \eqref{fgamma1gamma}, and \eqref{fgamma2gamma}. So, let us only detail the verification of the first (and fourth) bound(s) in \eqref{fx1x2x3}. To that end, observe that 
		\begin{flalign}
			\label{fugammaxyx1}
			\begin{aligned} 
				\big| \partial_u \breve{F}_{\gamma} (u, x, y) - \mathfrak{X}_1 \big| & \le \displaystyle\int_0^{\infty} \displaystyle\int_{-\infty}^{\infty} \big| (e^{2t / \alpha} - 1)^{-\gamma} - e^{-2t\gamma / \alpha} \big| \cdot h_0 (s - \log x) \psi_{\alpha/2} (u)^{2\gamma / \alpha- 1}  \big| \partial_u \psi_{\alpha/2} (u) \big| \\
				& \qquad \times \bigg( \Big| h_0' \big( t - \log y + \psi_{\alpha/2} (u, s) \big) \Big| + \displaystyle\frac{2 \gamma}{\alpha} \cdot h_0 \big( t - \log y + \psi_{\alpha/2} (u, s) \big) \bigg) ds dt \\
				& \le \displaystyle\frac{C}{u^4} \displaystyle\int_0^{\infty} \displaystyle\int_{-\infty}^{\infty} \big| ( e^{2t/\alpha} - 1)^{-\gamma} - e^{-2t \gamma / \alpha} \big| \cdot h_0 (s - \log x) \\
				&\qquad \qquad \quad \times \bigg( \Big| h_0' \big( t - \log y + \psi_{\alpha/2} (u, s) \big) \Big| + h_0 \big( t - \log y + \psi_{\alpha/2} (u, s) \big) \bigg) ds dt,
			\end{aligned}	
		\end{flalign}
		
		\noindent where in the second inequality we used the facts that 
		\begin{flalign}
			\label{psialpha2} 
			\big| \psi_{\alpha/2} (u) \big| + \big| \psi_0 (u) \big| \le 2u^{-1}; \qquad \big| \partial_u \psi_{\alpha/2} (u) \big| + \big| \partial_u \psi_0 (u) \big| \le \frac{4 \gamma}{\alpha} \cdot u^{-2\gamma - 1} \le 4 u^{-3}.
		\end{flalign}
		
		\noindent As in \eqref{t2alpha1t2alpha}, we have since $\sup_{w \in \mathbb{R}} \big( \big| h_0 (w) \big| + \big| h_0' (w) \big| \big) \le C$ that 
		\begin{flalign*}
			\displaystyle\sup_{K \in \mathbb{R}} \displaystyle\int_0^{\infty} & \big| (e^{2t/\alpha} - 1)^{-\gamma} - e^{-2t\gamma / \alpha} \big| \cdot \Big( h_0 (t - K) + \big| h_0' (t - K) \big| \Big) dt \\
			& \le C \displaystyle\int_0^{\infty} \big( t^{-\gamma} \cdot \mathbbm{1}_{t \le \alpha^3} + \alpha |\log \alpha| \cdot \mathbbm{1}_{t \le \alpha} + \alpha e^{-2t/\alpha} \cdot \mathbbm{1}_{t > \alpha} \big) dt  \le  C \alpha^2 |\log \alpha|,
		\end{flalign*} 
		
		\noindent which together with \eqref{fugammaxyx1}, \Cref{integralk}, and \eqref{xyualpha0} implies the first bound in \eqref{fx1x2x3}. As mentioned previously, the proofs of the second and third are very similar to those of \eqref{fgamma1gamma} and \eqref{fgamma2gamma}, respectively, and are thus omitted.
		
		To establish the fourth, observe that 
		\begin{flalign*}
			\big| \mathfrak{X}_3 - \partial_u \widetilde{F}_{\gamma} (u, x, y) \big| & \le 2 \displaystyle\int_0^{\infty} \displaystyle\int_{-\infty}^{\infty} e^{-2t\gamma / \alpha} \cdot h_0 (s - \log x) \cdot \psi_0 (u, s)^{2 \gamma / \alpha} \cdot \big| \partial_u \psi_{\alpha/2} (u, s) - \partial_u \psi_0 (u, s) \big| \\
			& \qquad \times \bigg( \Big|  h_0' \big( t - \log y + \log \psi_0 (u, s) \big) \Big| + h_0 \big( t - \log y + \log \psi_0 (u, s) \big) \bigg) ds dt.
		\end{flalign*}
		
		\noindent Using \Cref{integralk} to bound the integral with respect to $t$, it follows that 
		\begin{flalign*}
			\big| \mathfrak{X}_3 - \partial_u \widetilde{F}_{\gamma} (u, x, y) \big| & \le C \displaystyle\int_{-\infty}^{\infty} h_0 (s - \log x) \cdot \psi_0 (u, s)^{2\gamma / \alpha} \cdot \big| \partial_u \psi_{\alpha/2} (u, s) - \partial_u \psi_0 (u, s) \big| ds dt \\
			& \le C u^{-2} \displaystyle\int_{-\infty}^{\infty} e^{-s - e^{-s}/4} \cdot \big| \partial_u \psi_{\alpha/2} (u, s) - \partial_u \psi_0 (u, s) \big| ds,
		\end{flalign*}
		
		\noindent where in the last inequality we used \eqref{psialpha2}; the fact that $\gamma \le \alpha$; the explicit form for $h_0 (w) = e^{-w - e^{-w}}$; and the fact that $x \ge \frac{1}{4}$ (from \eqref{xyualpha0}). Applying \Cref{exponentialsu} and the fact that $\varpi^{2/\alpha} \le \alpha^2$ (recall \eqref{su0}) if $|s - \log u| > C \alpha |\log \alpha|$ (as then $\varpi \le 1 - C \alpha |\log \alpha| + O \big( \alpha^2 |\log \alpha| \big)$), it follows for sufficiently small $\alpha$ (the bound holds by compactness for $\alpha$ bounded away from $0$) that 
		\begin{flalign*}
			\big| \mathfrak{X}_3 - \partial_u \widetilde{F}_{\gamma} (u, s) \big| & \le C u^{-2 \gamma / \alpha - 3} \bigg( \displaystyle\int_{|s - \log u| \le C \alpha |\log \alpha|} e^{-s - e^{-s} / 4} ds + \alpha^2 \displaystyle\int_{-\infty}^{\infty} e^{-s - e^{-s/4}} ds \bigg) \\
			& \le C u^{-5} \Big( \alpha |\log \alpha| \cdot \displaystyle\max_{|s - \log u| < 1/100} e^{-s - e^{-s}/4} ds +  \alpha^2 \Big) \\
			& \le C u^{-5} \Big( \alpha |\log \alpha| \cdot e^{1/5u -\log u} + \alpha^2 \big) \le C \alpha^{7/6},
		\end{flalign*}
		
		\noindent where in the last two estimates we applied the bound on $u$ from \eqref{xyualpha0}. This verifies \eqref{fx1x2x3}, which together with \eqref{xyualpha0} and \eqref{fx1x2x3f} yields the proposition. 
	\end{proof}

	\subsection{Derivative Estimates for $a$ and $b$}
	
	\label{EstimateAlpha0ab}

	In this section we use \eqref{fba}, \Cref{gammaffgg}, and \Cref{gammaffggu} to establish the following bounds on the derivatives of $a (u^{2/\alpha})$ and $b (u^{2/\alpha})$ (recall \eqref{opaque}). 
	
	\begin{prop}
		
		\label{abderivativeu0}
		
		There exists constants $c > 0$ and $C > 1$ such that for $\alpha \in (0, c)$, we have 
		\begin{flalign*}
			\bigg| \partial_u a (u^{2/\alpha}) + \displaystyle\frac{e^{-1/u}}{4u^3} \bigg| \le C \alpha^{7/6} ; \qquad \bigg| \partial_u b - \displaystyle\frac{e^{-1/u}}{4u^3} (1 - 2u) \bigg| \le C \alpha^{7/6}.
		\end{flalign*}
		
	\end{prop}

	\begin{proof}
		
		Throughout this proof, we abbreviate $a = a(u^{2/\alpha})$ and $b = b(u^{2/\alpha})$; we also set $d = d(u) = a + b$. Then, \Cref{e0alpha0u} and \Cref{abdalpha0} together imply that $(a, b, u)$ satisfy the constraints \eqref{xyualpha0} on $(x, y, u)$. Next, by \eqref{fba}, \Cref{gammaffgg} (with \eqref{xyualpha0}), \Cref{gammaffggu}, and \Cref{fgu}, there exist functions $\mathfrak{f} (u, x, y)$ and $\mathfrak{g} (u, x, y)$ such that 
		\begin{flalign}
			\label{fwgwhw} 
			\displaystyle\sup_{w \in \mathbb{R}^3} \Big( \big| \mathfrak{f} (w) \big| + \big| \mathfrak{g} (w) \big| \Big) \le C; \qquad \displaystyle\sup_{w \in \mathbb{R}^3} \displaystyle\max_{\mathfrak{z} \in \{ u, x, y \}} \Big( \big| \partial_{\mathfrak{z}} \mathfrak{f} (w) \big| + \big| \partial_{\mathfrak{z}} \mathfrak{g} (w) \big|\Big) \le C,
		\end{flalign}
		
		\noindent and 
		\begin{flalign*} 
			a = F_{\alpha/2} (u^{2/\alpha}, a, b) & = \widetilde{F}_{\alpha/2} (u, a, b) + \alpha^{7/6} \cdot \mathfrak{f} (u, a, b) \\
			&  = \displaystyle\frac{b}{(a+b)^2} - e^{-(a+b)/u} \bigg( \displaystyle\frac{b}{(a+b)^2} + \displaystyle\frac{b}{(a+b) u} \bigg) + \alpha^{7/6} \cdot \mathfrak{f} (u, a, b),
		\end{flalign*}
		
		\noindent and 
		\begin{flalign*}
			b = G_{\alpha/2} (u^{2/\alpha}, a, b) & = \widetilde{G}_{\alpha/2} (u, a, b) + \alpha^{7/6} \cdot \mathfrak{g} (u, a, b) \\
			& = \displaystyle\frac{a}{(a+b)^2} + e^{-(a+b)/u} \bigg( \displaystyle\frac{b}{(a+b) u} - \displaystyle\frac{a}{(a+b)^2}\bigg) + \alpha^{7/6} \cdot \mathfrak{g} (u, a, b). 
		\end{flalign*}
		
		Setting $a' = \partial_u a (u^{2/\alpha})$, $b' = \partial_u b (u^{2/\alpha})$, and $d' = d' (u)$, we obtain from differentiating \eqref{a01} and \eqref{b01} that 
		\begin{flalign}
			\label{a01} 
			\begin{aligned} 
				a' & = \displaystyle\frac{b'}{d^2} - \displaystyle\frac{2b (a'+b')}{d^3} - e^{-d / u} \bigg( \displaystyle\frac{d}{u^2} - \displaystyle\frac{d'}{u} \bigg) \bigg( \displaystyle\frac{b}{d^2} + \displaystyle\frac{b}{du} \bigg) - e^{-d/u} \bigg( \displaystyle\frac{b'}{d^2} - \displaystyle\frac{2b d'}{d^3} + \displaystyle\frac{b'}{du} - \displaystyle\frac{bd'}{d^2 u} - \displaystyle\frac{b}{du^2} \bigg)  \\
				& \qquad + \alpha^{7/6} \cdot \big( \partial_u \mathfrak{f} (u, a, b) + a' \cdot \partial_a \mathfrak{f} (u, a, b) + b' \cdot \partial_b \mathfrak{f} (u, a, b) \big),
			\end{aligned} 
		\end{flalign}
		
		\noindent and 
		\begin{flalign}
			\label{b01} 
			\begin{aligned} 
				b' & = \displaystyle\frac{a'}{d^2} - \displaystyle\frac{2a(a'+b')}{d^3} + e^{-d/u} \bigg( \displaystyle\frac{d}{u^2} - \displaystyle\frac{d'}{u} \bigg) \bigg( \displaystyle\frac{b}{du} - \displaystyle\frac{a}{d^2} \bigg) + e^{-d/u} \bigg( \displaystyle\frac{b'}{du} - \displaystyle\frac{bd'}{d^2 u} - \displaystyle\frac{a}{du^2} - \displaystyle\frac{a'}{d^2} + \displaystyle\frac{2ad'}{d^3} \bigg) \\
				& \qquad + \alpha^{7/6} \cdot \big( \partial_u \mathfrak{g} (u, a, b) + a' \cdot \partial_a \mathfrak{g} (u, a, b) + b' \cdot \partial_b \mathfrak{g} (u, a, b) \big).
			\end{aligned} 
		\end{flalign} 
		
		Using the bounds $\big| a - \frac{1}{2} \big| + \big| b - \frac{1}{2} \big| \le C \alpha^2 |\log \alpha|$ from \eqref{abdalpha0}, together with the estimate on $u$ from \eqref{xyualpha0}, it follows from \eqref{a01} that 
		\begin{flalign*}
			a' \cdot \big( 2 + O (u^{-2} e^{-d/u}) + O (\alpha^{7/6}) \big) + b \cdot \big( & O(u^{-2} e^{-d/u}) + O (\alpha^{7/6}) \big) \\
			& = e^{-d/u} bu^{-3} + O(\alpha^{7/6}) = \frac{e^{-d/u}}{2u^3} + O(\alpha^{7/6}),
		\end{flalign*}
		
		\noindent and from \eqref{b01} that 
		\begin{flalign*}
			a' \cdot \big( O(u^{-2} e^{-d/u}) + O (\alpha^{7/6}) & \big) + b \cdot \big( 2 + O(u^{-2} e^{-d/u}) + O (\alpha^{7/6}) \big) \\
			& = e^{-d/u} u^{-3} \big( b - 2au d^{-1}  \big) + O(\alpha^{7/6}) = \displaystyle\frac{e^{-d/u}}{2u^3} ( 1 - 2u) + O(\alpha^{7/6}).
		\end{flalign*}
		
		\noindent The proposition then follows from solving this linear system for $(a', b')$ (and using \eqref{xyualpha0}). 
	\end{proof}
	
	\subsection{Uniqueness of the Mobility Edge}
	
	\label{Alpha0Unique}
	
	In this section we establish the second part of \Cref{t:main2}. Throughout, we abbreviate $a = a(u^{2/\alpha})$ and $b = b(u^{2/\alpha})$ and denote 
	\begin{flalign*} 
		H(u) = F_{\alpha} (u^{2/\alpha}, a, b) + G_{\alpha} (u^{2/\alpha}, a, b); \qquad \widetilde{H} (u) = \widetilde{F}_{\alpha} (u, a, b) + \widetilde{G}_{\alpha} (u, a, b). 
	\end{flalign*} 
	
	We begin with the following lemma that bounds the derivative of $\widetilde{H}$ (which is explicit from \Cref{falphagalpha0}).
	
	\begin{lem} 
		
		\label{derivativeh0}
		
		There exist constants $c > 0$ and $C > 1$ such that, for any $\alpha \in (0, c)$, we have $\partial_u \widetilde{H} (u) \le (Cu^{-2} - cu^{-3}) e^{-1/u}$.
		
	\end{lem} 
	
	\begin{proof}
		
		First observe from \eqref{xyualpha0} and \Cref{abdalpha0} that $(a, b)$ satisfy the constraints \eqref{xyualpha0} on $(x, y)$. Then, denoting $d = a+b$, we have from \Cref{falphagalpha0} that  
		\begin{flalign*}
			\partial_u \widetilde{H} (u) & = 2 \bigg( \displaystyle\frac{1}{ud} + \displaystyle\frac{1}{d^2} \bigg) e^{-d/u} \cdot \partial_u d - \displaystyle\frac{2d}{u^2} \bigg( \displaystyle\frac{1}{ud} + \displaystyle\frac{1}{d^2} \bigg) e^{-d/u} \\
			& \qquad + 2 \bigg( \displaystyle\frac{1}{du^2} + \displaystyle\frac{\partial_u d}{d^2 u} + \displaystyle\frac{2 \partial_u d}{d^3} \bigg) e^{-d/u} - \displaystyle\frac{4}{d^3} \partial_u d.
		\end{flalign*}
		
		\noindent This, together with \Cref{abderivativeu0}, \Cref{abdalpha0}, and \eqref{xyualpha0}, yields the lemma.
	\end{proof}
	
	The next corollary estimates derivatives of $F_{\alpha/2}$, $G_{\alpha/2}$, and $H$. 
	
	\begin{cor}
		
		\label{fguestimate}
		
		There exist constants $c > 0$ and $C > 1$ such that the following holds for any $\alpha \in (0, c)$. Setting $F(u) = F_{\alpha/2} (u^{2/\alpha}, a, b)$ and $G(u) = G_{\alpha/2} (u^{2/\alpha}, a, b)$, we have
		\begin{flalign}
			\label{ugufhu} 
			\partial_u H(u) \le (Cu^{-2} - cu^{-3}) e^{-1/u} + C \alpha^{7/6}; \qquad \big| \partial_u F (u) \big| \le C \alpha^{1/2}; \qquad \big| \partial_u G(u) \big| \le C \alpha^{1/2}.
		\end{flalign}
	\end{cor}
	
	\begin{proof}
		
		As in the proof of \Cref{derivativeh0}, observe from \eqref{xyualpha0} and \Cref{abdalpha0} that $(a, b)$ satisfy the constraints \eqref{xyualpha0} on $(x, y)$. Letting $a' = \partial_u a(u^{2/\alpha})$, $b' = \partial_u b(u^{2/\alpha})$, and $d' = a' + b'$, we find
		\begin{flalign*}
			\big| \partial_u H(u) - \partial_u \widetilde{H}_{\alpha} (u, a, b) \big| & \le \big| \partial_u H_{\alpha} (u, a, b) - \partial_u \widetilde{H}_{\alpha} (u, a, b) \big| + |a'| \cdot \big| \partial_a H_{\alpha} (u, a, b) \big| \\
			& \qquad + |b'| \cdot \big| \partial_b H_{\alpha} (u, a, b) \big| \\
			& \le O(\alpha^{7/6}) + \big( |a'| + |b'| \big) \cdot \Big( \big| \partial_a H_{\alpha} (u, a, b) \big| + \big| \partial_b H_{\alpha} (u, a, b) \big|\Big) \\
			& \le O (\alpha^{7/6}) + O (\alpha^{9/10}) \cdot \Big( \big| \partial_b H_{\alpha} (u, a, b) \big| + \big| \partial_b H_{\alpha} (u, a, b) \big|  \Big)\\
			& \le O(\alpha^{7/6}) + O (\alpha^{9/10}) \cdot \Big( u^{-3} e^{-(a+b)/u} \cdot \big( |a'| + |b'| + 1 \big) \Big) = O(\alpha^{7/6}),
		\end{flalign*}
		
		\noindent where in the first bound we applied \Cref{gammaffggu}; in the second we applied \Cref{abderivativeu0} and \eqref{xyualpha0}; and in the third we applied \Cref{falphagalpha0}, \Cref{abderivativeu0}, and \eqref{xyualpha0}. Together with \Cref{derivativeh0}, this yields the first statement of \eqref{ugufhu}. The proofs of the second and third are very similar (following from \Cref{fgu}, \Cref{gammaffgg}, \Cref{gammaffggu}, and \Cref{abderivativeu0}) and are thus omitted.
	\end{proof}
	
	We can now establish the second part of \Cref{t:main2}.
	
	\begin{proof}[Proof of \Cref{t:main2}(2)]
		
		The scaling \eqref{ealpha0} of $E_{\mob}$ was verified by \Cref{e0alpha0u}, so it suffices to show uniqueness of $E_{\mob}$. By \Cref{e0alpha0u}, It suffices to show that $\lambda (u^{2/\alpha}, \alpha)$ is decreasing in $u$ on the interval 
		\begin{flalign}
			\label{ualphainterval}  
			u \in \bigg[ \frac{99}{100} \cdot |\log \alpha|^{-1}, \frac{101}{100} \cdot |\log \alpha|^{-1} \bigg],
		\end{flalign}
		
		\noindent to which end we must show that 
		\begin{flalign*}
			\mathfrak{L} (u) = \sin \Big( \displaystyle\frac{\pi \alpha}{2} \Big) \cdot (F+G) + \sqrt{(F+G)^2 + \sin \Big( \displaystyle\frac{\pi \alpha}{2} \Big) \cdot (F-G)^2},
		\end{flalign*} 
		
		\noindent is decreasing in $u$; here, we have abbreviated $F = F(u) = F_{\alpha} (u^{2/\alpha}, a, b)$ and $G = G(u) = G_{\alpha} (u^{2/\alpha}, a, b)$ (upon setting $a = a(u^{2/\alpha})$ and $b = b(u^{2/\alpha})$). Observe that 
		\begin{flalign*}
			\mathfrak{L}' (u) & = \sin \Big( \displaystyle\frac{\alpha}{2} \Big) \cdot (F' + G') + (F+G)(F'+G') \cdot  \bigg( (F+G)^2 + \sin^2 \Big( \displaystyle\frac{\pi \alpha}{2} \Big) \cdot (F-G)^2 \bigg)^{-1/2} \\
			& \qquad + \sin^2 \Big( \displaystyle\frac{\pi \alpha}{2} \Big)^2 \cdot (F-G)(F'-G') \bigg( (F+G)^2 + \sin^2 \Big( \displaystyle\frac{\pi \alpha}{2} \Big) \cdot (F-G)^2 \bigg)^{-1/2}.
		\end{flalign*}
		
		\noindent Using the fact that $|F-G| \le |F+G|$ and $F+G = 2 + O(\alpha)$ by \Cref{lambdaalpha0}, it follows that 
		\begin{flalign*}
			\mathfrak{L}' (u) \le F'+G' + C \alpha \big( |F'| + |G'| \big) + C \alpha^2 |F' - G'| \le F'+G' + C \alpha \big( |F'| + |G'| \big).
		\end{flalign*}
		
		\noindent By \Cref{fguestimate}, it follows for sufficiently small $\alpha$ that
		\begin{flalign*}
			\mathfrak{L}' (u) \le (Cu^{-2} - cu^{-3}) u^{-1/\alpha} + C \alpha^{7/6} < 0,
		\end{flalign*}
		
		\noindent where in the last bound we used \eqref{ualphainterval}. This shows that $\mathfrak{L}' (u)$ is decreasing, thereby establishing the theorem. 
	\end{proof}

	\chapter{Appendices}

	\section{Eigenvector Localization and Delocalization Criteria}
		
		\label{a:ELarge}
		
	This appendix is devoted to the proof of \Cref{l:loccriteria}. \Cref{a:localization} and \Cref{a:delocalization} prove the criteria for eigenvector localization and delocalization, respectively, and the proof is concluded in \Cref{a:conclusion}.
	
	 Given $s \in \bbR_+$ and an $N\times N$ symmetric matrix $\boldsymbol{M}$, we recall the quantity $P_I(j)$ from \Cref{jwi} and define $Q_I(s) = Q_I(s; \boldsymbol{M})$ by
	\be
	Q_I(s) = \frac{1}{N} \sum_{j=1}^N P_I(j)^s.
	\ee
Note that under \Cref{jwi}, we have the shorthand $Q_I = Q_I(2)$.

\subsection{Localization}\label{a:localization}

The proofs in this section are similar to that of \cite[Lemma 5.9]{bordenave2013localization} and \cite[Theorem 1.1]{bordenave2017delocalization}, so we only outline them.

	\begin{lem} \label{huw2} 
	 Fix $\alpha, s, \delta \in (0,1)$ and $E \in \mathbb{R}\setminus\{0\}$ such that $\lim_{\eta \rightarrow 0} \Im R_\star(E + \iu \eta) = 0$ in probability. Let $\boldsymbol{H} = \boldsymbol{H}_N$ denote an $N \times N$ L\'{e}vy matrix with parameter $\alpha$, and define the interval $I(\eta) = [E - \eta, E + \eta]$. There exists $\eta_0(\alpha, s, \delta, E) > 0$ such that if $\eta \in(0,\eta_0)$, then
		\begin{flalign*}
			\displaystyle\lim_{N \rightarrow \infty} \mathbb{P} \Bigg[ Q_{I(\eta)} \bigg(\displaystyle\frac{s}{2}; \boldsymbol{H} \bigg) \le \delta \Bigg] = 1.
		\end{flalign*}
	
	\end{lem} 
	
\begin{proof}
Set $z = E + \mathrm{i} \eta$, and denote $\boldsymbol{G} = \boldsymbol{G} (z) = (\boldsymbol{H} - z)^{-1}$, so that 
		 \begin{flalign}
		 	\label{giizu} 
		 	G_{ij} (z) = \displaystyle\sum_{k = 1}^N \displaystyle\frac{u_k (i) u_k (j)}{z - \lambda_k}
		 \end{flalign}
	 
	 	\noindent by the spectral theorem, where $G_{ij} = G_{ij} (z)$ denotes the $(i, j)$-entry of $\boldsymbol{G}$. Then observe for any index $j \in \unn{1}{N}$ that
	\begin{flalign*}
		\displaystyle\frac{|\Lambda_I|}{2 N \eta} \cdot P_I (j) = \displaystyle\frac{1}{2 \eta} \displaystyle\sum_{\lambda_k \in \Lambda_I} \big| u_k (j) \big|^2 \le \displaystyle\sum_{k = 1}^N \displaystyle\frac{\eta \big| u_k (j) \big|^2}{|z - \lambda_k|^2} = \Imaginary G_{jj}.
	\end{flalign*}

	\noindent In the second statement we used the fact that 
	\bex
	|z - \lambda_k| \le \big((\Imaginary z)^2 + |E - \lambda_k|^2 \big)^{1/2} \le 2^{1/2} \eta
	\eex 
	for $\lambda_k \in I = [E - \eta, E + \eta]$, and in the third we used \eqref{giizu}. Thus, 
	\begin{flalign}
		\label{sumwijs}
		Q_I \bigg( \displaystyle\frac{s}{2} \bigg) = \displaystyle\frac{1}{N} \displaystyle\sum_{j = 1}^N P_I (j)^{s/2} \le \displaystyle\frac{1}{N} \bigg( \displaystyle\frac{2N \eta}{|\Lambda_I|} \bigg)^{s/2} \displaystyle\sum_{j = 1}^N (\Imaginary G_{jj})^{s/2}.
	\end{flalign}

	Now we apply a concentration estimate for $N^{-1} \sum_{j = 1}^N (\Imaginary G_{jj})^{s/2}$. Specifically, \cite[Lemma C.4]{bordenave2013localization} implies the existence of a constant $C_1 = C_1 (s) > 1$ such that for $\eta \ge (\log N)^{-1}$ we have 
	\begin{flalign}
		\label{giis} 
	\mathbb{P} \Bigg[ \bigg| \displaystyle\frac{1}{N} \displaystyle\sum_{j = 1}^N (\Imaginary G_{jj})^{s/2} - \mathbb{E} \big[ (\Imaginary G_{11})^{s/2} \big] \bigg| \le \eta \Bigg] \ge 1 - C_1 N^{-100}.
	\end{flalign} 

	\noindent Moreover, it follows from \cite[Theorem 1.1, Theorem 1.6]{arous2008spectrum} that there exists a constant $c_1 = c_1 (E) \in (0, 1)$ such that 
	\begin{flalign}
		\label{lambdaic1}
		\lim_{N \rightarrow \infty} \mathbb{P} \big[ |\Lambda_I| \ge c_1 N \eta \big] = 1.
	\end{flalign}

	\noindent By \eqref{sumwijs}, \eqref{giis}, and \eqref{lambdaic1}, we obtain
	\begin{flalign}
	\label{wijsum} 
	\displaystyle\lim_{N \rightarrow \infty} \mathbb{P} \Bigg[ Q_I \bigg( \displaystyle\frac{s}{2} \bigg) \le 2c_1^{-1} \cdot \mathbb{E} \big[ (\Imaginary G_{11})^{s/2} \big] + \eta \Bigg] \ge 1.
	\end{flalign}

	From \cite[Theorem 2.2, Proposition 2.6]{bordenave2011spectrum}, the random variable $G_{11} = G_{11} (z)$ converges to $R_{\star} = R_{\star} (z)$ in probability when $z$ is fixed. This, together with the deterministic bound $|G_{11}| \le \eta^{-1}$ (recall the second part of \Cref{q12}), yields 
	\begin{flalign}\label{appstarcvg}
		\displaystyle\lim_{N \rightarrow \infty} \mathbb{E} \big[ (\Imaginary G_{ii})^{s/2} \big] = \mathbb{E} \big[ (\Imaginary R_{\star})^{s/2} \big]
	\end{flalign} 
for any $z \in \bbH$. 
Observe that by \Cref{expectationqchi}, the sequence $\big\{ (\Imaginary R_{\star}(E + \iu \eta))^{s/2} \big\}_{\eta \in (0,1)}$ is uniformly integrable, since its $(1+\eps)$-moments are uniformly bounded for some $\eps >0$. 
 Then the hypothesis that $\lim_{\eta \rightarrow 0} \Im R_\star(E + \iu \eta) = 0$ yields a constant $c_2 = c_2 (s, \delta) > 0$ such that for $\eta \in (0, c_2)$ we have $\mathbb{E} \big[ (\Imaginary R_{\star})^{s/2} \big] < \frac{c_1 \delta}{8}$. 
This implies for sufficiently large $N$ that 
	\begin{flalign*}
		\mathbb{E} \big[ (\Imaginary G_{ii})^{s/2} \big] \le \displaystyle\frac{c_1 \delta}{4}. 
	\end{flalign*}

	\noindent Combining this with \eqref{wijsum} and taking $\eta < \frac{\delta}{2}$ (by imposing $\eta_0 < \frac{\delta}{2}$) yields the proposition. 
\end{proof}

\begin{prop}\label{l:localizationmatrix}
Retain the notation and hypotheses of \Cref{huw2}. For every $D>0$, there  exists $\eta_0(\alpha,  D, E) > 0$ such that if $\eta \in(0,\eta_0)$, then

		\begin{flalign*}
			\displaystyle\lim_{N \rightarrow \infty} \mathbb{P} \Big[ Q_{I(\eta)} \left( \boldsymbol{H} \right) \ge D \Big] = 1.
		\end{flalign*}
	
\end{prop}

\begin{proof}
Let $\eps \in (0,1)$ and $p, q \in \bbR_+$ such that $p^{-1} + q^{-1}=1$ be parameters. By H\"older's inequality,
\begin{align}\label{duality1}
N = \sum_{k=1}^N P_I(k) &= 
\frac{1}{N} \sum_{k=1}^N \left( N P_I(k) \right)^\eps 
\left( N P_I(k) \right)^{1-\eps}\\
&\le 
\left(N^{\eps p -1} \sum_{k=1}^N P_I (k)^{\eps p}  \right)^{1/p}
\left(
N^{(1-\eps)q - 1} 
\sum_{k=1}^N P_I (k)^{(1-\eps) q } \right)^{1/q}\notag \\
& = 
\big(N^{\eps p} Q_I(\eps p)  \big)^{1/p}
\left(
N^{(1-\eps)q } 
Q_I\big( (1-\eps) q \big) \right)^{1/q} \notag.
\end{align}
With $s \in(0,1)$ a parameter, we set
\bex
\eps = \frac{s}{4-s}, \qquad p = 2 - \frac{s}{2}.
\eex
These choices imply 
\bex
\eps p  = \frac{s}{2}, \quad (1-\eps) q = 2, \quad \frac{p}{q} = 1 - \frac{s}{2}.\eex
Then taking $p$-th powers in \eqref{duality1} gives
\bex
1 \le \big( Q_I(s/2) \big)
\big( Q_I(2)  \big)^{1 - s/2}.
\eex
Using the conclusion of \Cref{huw2} for an arbitrary value of $s \in (0,1)$, we obtain the desired result.
\end{proof}

\subsection{Delocalization}\label{a:delocalization}

	\begin{prop} 
		
		\label{l:delocalizationmatrix} 
	 Fix $\alpha, s \in (0,1)$ and $E \in \mathbb{R}$  such that $\liminf_{\eta \rightarrow 0} \mathbb{P} \big[ \Imaginary R_{\star} (E + \mathrm{i} \eta) > c \big] > c$ for some $c>0$.  Let $\boldsymbol{H} = \boldsymbol{H}_N$ denote an $N \times N$ L\'{e}vy matrix with parameter $\alpha$, and define the interval $I(\eta) = [E - \eta, E + \eta]$. There exist constants $\eta_0(\alpha, s, E) > 0$ and $C(\alpha,s,E)>1$ such that for all $\eta \in(0,\eta_0)$, we have
		\begin{flalign*}
			\displaystyle\lim_{N \rightarrow \infty} \mathbb{P} \big[ Q_{I(\eta)} ( \boldsymbol{H} ) \le C \big] = 1.
		\end{flalign*}
	
	\end{prop} 
\begin{proof}
We compute
\begin{align*}
P_I(j)
= 
\frac{N}{|\lambda_I|} \sum_{\lambda_i \in \Lambda_I}
\big | u_i(j) \big|^2 
\le
\frac{N}{|\Lambda_I|} \sum_{i=1}^N
\frac{ 2 \eta^2 \big | u_i(j) \big|^2 }
{\eta^2 + (\Lambda_i - E)^2} 
=\frac{2 N \eta }{|\Lambda_I|}\Im G_{jj}(E + \iu \eta).
\end{align*}
Then
\be\label{aboveA9}
Q_I = \displaystyle\frac{1}{N} \displaystyle\sum_{j = 1}^N P_I (j)^2
\le
\frac{ 4 N \eta^2}{|\Lambda_I|^2} \sum_{j=1}^N \big( \Im G_{jj}(z)\big)^2.
\ee
Now  \cite[Lemma C.3]{bordenave2013localization} implies the existence of a constant $C_1 = C_1 (s) > 1$ such that for $\eta \ge (\log N)^{-1}$ we have 
	\begin{flalign*}
	\mathbb{P} \Bigg[ \bigg| \displaystyle\frac{1}{N} \displaystyle\sum_{j = 1}^N (\Imaginary G_{jj})^{2} - \mathbb{E} \big[ (\Imaginary G_{11})^{2} \big] \bigg| \le \eta \Bigg] \ge 1 - C_1 N^{-100}.
	\end{flalign*} 
Together with \eqref{lambdaic1} and \eqref{aboveA9}, we obtain 
	\begin{flalign}
	\label{wijsum2} 
	\displaystyle\lim_{N \rightarrow \infty} \mathbb{P} \left[ Q_I  \le 4c_1^{-2} \cdot \mathbb{E} \big[ (\Imaginary G_{11})^{2} \big] + \eta \right] =  1.
	\end{flalign}
From \cite[Theorem 2.2, Proposition 2.6]{bordenave2011spectrum}, the random variable $G_{11} = G_{11} (z)$ converges to $R_{\star} = R_{\star} (z)$ in probability when $z$ is fixed. Together with the deterministic bound $|G_{11}| \le \eta^{-1}$ (recall the second part of \Cref{q12}), this yields 
	\begin{flalign*}
		\displaystyle\lim_{N \rightarrow \infty} \mathbb{E} \big[ (\Imaginary G_{11})^{2} \big] = \mathbb{E} \big[ (\Imaginary R_{\star})^{2} \big].
	\end{flalign*} 	
	
 By \Cref{rrealimaginary}, the hypothesis  that $\liminf_{\eta \rightarrow 0} \mathbb{P} \big[ \Imaginary R_{\star} (E + \mathrm{i} \eta) > c \big] > c$ implies the existence of a constant $c_2 > 0$ such that $\sigma\big(\vartheta_0(z)\big) > c_2$ for $\eta \in (0, c_2)$. Using \eqref{kappatheta1}, we write
\be
\big( \Im R_\star (z) \big)^2 = \frac{\big( \eta + \vartheta_0(z) \big)^2}{|z + \varkappa_0(z) + \iu \vartheta_0(z)|^4} \le \frac{1}{\vartheta_0(z)^2},
\ee
from which we conclude the existence of a constant $C_2$ such that 
\be\label{IMRUPPER}
\E\left[ \big( \Im R_\star (E + \iu \eta) \big)^2 \right] \le C_2
\ee
for $\eta \in (0, c_2)$. In making the previous assertion, we used $\sigma\big(\vartheta_0(z)\big) > c_2$ and that the density of a one-sided $\alpha/2$-stable decays faster than any polynomial near $0$ (which can be deduced from the explicit form \eqref{xtsigma} of the characteristic function; see also \cite[Lemma B.3]{bordenave2013localization}). Then using \eqref{appstarcvg} and \eqref{IMRUPPER}, we deduce the existence of constants  $C_3, N_0 > 1$ such that
\bex
2c_1^{-1} \cdot \mathbb{E} \big[ (\Imaginary G_{11})^{2} \big] + \eta \le C_3
\eex
if $\eta \in (0, c_2)$ and $N \ge N_0$. Inserting this bound into \eqref{wijsum2} completes the proof.
\end{proof}

\subsection{Conclusion}\label{a:conclusion}

We can now give the proof of \Cref{l:loccriteria}.

\begin{proof}[Proof of \Cref{l:loccriteria}]
This follows immediately from \Cref{l:localizationmatrix} and \Cref{l:delocalizationmatrix}.
\end{proof}

\section{Boundary Values of $y(z)$}\label{b:bvals}

This appendix states some important facts about the boundary values of the quantity $y(z)$. In \Cref{s:boundaryvalues} we show the existence of boundary values of $y(z)$, and in \Cref{b:pfll} we prove \Cref{l:lambdalemma}.

\subsection{Boundary Values}\label{s:boundaryvalues}
Recall our convention for the function $z \mapsto z^a$ from \Cref{s:notation}, and recall from \Cref{s:PWITresults} the definitions $\bbK = \{ z \in \bbC: \Re z > 0\}$, 
\bex
y(z) = \E \left[ \left(- \iu R_\star(z)   \right)^{\alpha/2}  \right],
\qquad
\varphi_{\alpha,z}(x) = \frac{1}{\Gamma(\alpha/2)} \int_0^\infty t^{\alpha/2 -1} e^{\mathrm{i} t z} \exp \left( -\Gamma(1-\alpha/2) t^{\alpha/2} x\right)\, dt,
\eex
and the equation 
\be\label{fpb1}
y(z) = \phi_{\alpha,z}( y(z)).
\ee
In \cite{belinschi2009spectral}, this relation is formulated in a slightly different way. The authors introduce the entire function
\bex
g_\alpha(x) = \int_0^\infty t^{\alpha/2 -1 } e^{-t} \exp( - t^{\alpha/2} x ) \, dt
\eex
and the function $Y(z)\colon \bbH \rightarrow \bbC$ is defined as the unique analytic solution (on $\bbH$) to 
\be\label{fpb2}
z^{\alpha} Y(z) = \frac{\iu^\alpha \Gamma\left( 1 - \frac{\alpha}{2} \right)}{\Gamma\left( \frac{\alpha}{2} \right)} \cdot g_\alpha\big( Y(z) \big).
\ee
Comparing \eqref{fpb1} to \eqref{fpb2}, we deduce the relation
\be\label{darelation}
y(z) = \frac{ z^{\alpha/2}  Y(z)}{\iu^{\alpha/2} \Gamma\left( 1 - \frac{\alpha}{2} \right)}.
\ee

We now show that the boundary values of $y(z)$ exist.

\begin{lem}\label{l:boundaryvalues}
For every $\alpha \in (0,1)$, the function $y\colon \bbH \rightarrow \bbC$ has a continuous extension $y\colon \overline{\bbH} \rightarrow \bbC$.
\end{lem}
\begin{proof}
For points $E \in \bbR$ such that $E\neq 0$, this is a direct consequence of \eqref{darelation} and \cite[Proposition 1.1]{belinschi2009spectral}. It remains to consider the point $E=0$. Define $h_\alpha(x) = 1 - \frac{\alpha}{2} x g_\alpha(x)$. Then by \eqref{fpb2},
\bex
h_\alpha\big( Y(z) \big) = 1 - \frac{\alpha \Gamma \left( \frac{\alpha}{2} \right)}{2 \iu^\alpha \Gamma\left( 1 - \frac{\alpha}{2} \right) } z^\alpha Y(z)^2
= 1 - \frac{\alpha }{2 } \Gamma \left( \frac{\alpha}{2} \right) \Gamma\left( 1 - \frac{\alpha}{2} \right) y(z)^2,
\eex
where the last equality follows from \eqref{darelation}.
In the proof of \cite[Proposition 1.1]{belinschi2009spectral}, it is shown that 
\bex
\lim_{z \rightarrow 0} h_\alpha\big( Y(z) \big)  = 0,
\eex
where the limit is taken over $z \in \bbH$. The completes the proof.
\end{proof}
\begin{rem}
The previous proof further shows that 
\be\label{y0}
y(0) =  \bigg[  \frac{\alpha }{2 } \Gamma \left( \frac{\alpha}{2} \right) \Gamma\left( 1 - \frac{\alpha}{2} \right) \bigg]^{-1/2}
=   \left[  \frac{\alpha }{2 } \frac{\pi}{ \sin \left( \pi \alpha/2 \right)} \right]^{-1/2} = \sqrt{ \frac{2 \sin \left( \pi \alpha/2 \right)}{\pi \alpha } },
\ee
where the second equality uses Euler's reflection formula for the Gamma function.\footnote{We remark that $y(0)$ may alternatively be computed using \cite[Lemma 4.3(b)]{bordenave2011spectrum} and equation (4.9) of that reference.}
\end{rem}

\subsection{Proof of \Cref{l:lambdalemma}}\label{b:pfll}
We recall that the quantity $\lambda(E,\alpha)$ was defined in \Cref{lambdaEalpha}.
\begin{lem}\label{l:lambda>1}
For all $\alpha \in (0,1)$, we have $\lambda(0, \alpha) > 1$.
\end{lem}
\begin{proof}
We have 
\bex
y(0) = \frac{y(0)}{ (\iu)^{\alpha/2} + (-\iu)^{\alpha/2} }
\big(
(\iu)^{\alpha/2} + (-\iu)^{\alpha/2}
\big)
= \frac{y(0)}{
2\cos\left( \alpha \pi /4  \right)
}
\big(
(\iu)^{\alpha/2} + (-\iu)^{\alpha/2}
\big).
\eex
Then from \eqref{opaque}, we have
\bex
a(0) = b(0) = \frac{y(0)}{
2\cos\left( \alpha \pi /4  \right)
},
\eex
so by the definition of $\varkappa_\loc$ and \eqref{sumalpha},
\bex
\varkappa_\loc(0) = a(0)^{2/\alpha} \left(S_1 - S_2  \right)
= a(0)^{2/\alpha} S_0,
\eex
where $S_0$ is a symmetric $\alpha/2$-stable law with scaling parameter $\sigma(S_0)=2^{2/\alpha}$. 

Recalling \eqref{xtsigma}, we have
\bex
\hat p_0(k)
=
\exp
\left(
- \frac{\pi  a(0)}{\sin(\pi \alpha/4) \Gamma(\alpha/2) } 
|k|^{\alpha/2}
\right).
\eex
Recalling the definition of $\ell(E)$ from \eqref{tlrk}, we have
\be\label{ell0}
\ell(0)
= \frac{1}{\pi} \int_0^\infty
k^{\alpha-1} \hat p_0(k)\, dk =
\frac{1}{\pi} \int_0^\infty
k^{\alpha-1} 
\exp
\left(
- \frac{\pi  a(0) }{\sin(\pi \alpha/4) \Gamma(\alpha/2) } 
\cdot k^{\alpha /2}
\right)
\, dk.
\ee
We compute using \eqref{y0} that
\begin{align}
\frac{\pi  a(0) }{\sin(\pi \alpha/4) \Gamma(\alpha/2) } 
&= 
\frac{y(0)}{
2\cos\left( \alpha \pi /4 \right)
}\cdot 
\frac{\pi  }{\sin(\pi \alpha/4) \Gamma(\alpha/2) } \notag
\\ &= 
\left(
\frac{2 \sin \left( \pi \alpha/2 \right)}{\pi \alpha }
\right)^{1/2} 
\frac{1}{
2\cos\left( \alpha \pi /4 \right)
}
\frac{\pi  }{\sin(\pi \alpha/4) \Gamma(\alpha/2) }.\label{ell0compute}
\end{align}
For any parameter $A > 0$, we have
\begin{align*}\int_0^\infty
k^{\alpha-1} 
\exp
\left(
- A  k^{\alpha /2}
\right)
\, dk  = \frac{2}{\alpha} \int_0^\infty m \exp\left( - A m \right)\, dm  = \frac{2}{ \alpha A^2},
\end{align*}
where we used the change of variables $m = k^{\alpha/2}$ in the first equality, so that $dm = (\alpha/2) k^{\alpha/2 -1 } \, dk$.

We conclude from the previous line, \eqref{ell0}, and \eqref{ell0compute} that 
\bex
\ell(0) 
= \frac{2}{\alpha \pi} \left[
\left(
\frac{2 \sin \left( \pi \alpha/2 \right)}{\pi \alpha }
\right)^{1/2} 
\frac{1}{
2\cos\left( \alpha \pi /4 \right)
}
\frac{\pi  }{\sin(\pi \alpha/4) \Gamma(\alpha/2) }
\right]^{-2}.
\eex
We recall that $\lambda(0,\alpha)$ was defined in \eqref{mobilityquadratic} the largest solution of
\be
\label{mobilityquadratic3}
K_\alpha^2 ( t_\alpha^2 - t_{1}^2) \ell(0)^2 - 2 t_\alpha K_\alpha \ell(0) \lambda(0,\alpha)+ \lambda(0,\alpha)^2=0,
\ee
with $
t_\alpha = \sin(\alpha \pi/2)$ and $
K_\alpha = \frac{\alpha}{2} \Gamma(1/2 - \alpha/2)^2$.
The quadratic formula applied to \eqref{mobilityquadratic3} yields 
\bex
\lambda(0,\alpha) = \ell(0) K_\alpha ( t_\alpha + 1).
\eex
We write this as
\begin{align}\label{lambda0exp}
\lambda(0,\alpha) &= 
 \frac{\alpha }{2 \pi^2} \Gamma\left( \frac{1}{2} - \frac{\alpha}{2}\right)^2 \Gamma\left(\frac{\alpha}{2} \right)^{2}
\big( \sin(\pi \alpha/2) + 1  \big) 
\sin(\pi \alpha/2),
\end{align}
where we simplified the previous expression using $2 \sin(x) \cos(x) = \sin(2x)$. Using Euler's reflection formula for the gamma function gives
\bex
\Gamma\left( \frac{\alpha}{2} \right) = \frac{\pi}{ \Gamma\left( 1 - \frac{\alpha}{2} \right)  \sin(\pi \alpha/2).
}
\eex
Putting this identity in \eqref{lambda0exp} gives
\begin{equation}\label{lambda0exp2}
\lambda(0,\alpha) =
 \frac{\alpha }{2 \sin(\pi \alpha/2)} \Gamma\left( \frac{1}{2} - \frac{\alpha}{2}\right)^2 \Gamma\left( 1 - \frac{\alpha}{2}\right)^{-2}
\big( \sin(\pi \alpha/2) + 1  \big).
\end{equation}

It remains to argue that $\lambda(0,\alpha) > 1$.  We also use the following facts, which may be proved through elementary calculus. For $x\in (0,1)$, we have

\bex
\ \sin(\pi x/2) + 1  >  1,\eex
\bex
 \frac{x }{2 \sin(\pi x/2)} >  \frac{1}{\pi},
\eex
\bex
 \Gamma\left( \frac{1}{2} - \frac{x}{2}\right)  \ge \sqrt{\pi} \cdot \Gamma\left( 1 - \frac{x}{2}\right),
\eex
where the last inequality follows from the convexity of $\Gamma$ and the fact that $\Gamma(1/2)/\Gamma(1) = \sqrt{\pi}$. 
Inserting these bounds into \eqref{lambda0exp2} gives $\lambda(0,\alpha) > 1$ for $\alpha \in (0,1)$, as desired.
\end{proof}

We can now prove \Cref{l:lambdalemma}.
\begin{proof}[Proof of \Cref{l:lambdalemma}]
The first and second parts of the theorem is immediate from the definitions of $\lambda(E,s,\alpha)$ and $\lambda(E,\alpha)$, the quadratic formula, and the continuity of the coefficient of \eqref{mobilityquadratic} in $s$. 
For the third part of the theorem, the continuity in $s$ is again clear from the definition of $\lambda(E,s,\alpha)$. For the continuity in $E$ of $\lambda(E,s,\alpha)$, we recall 
from \Cref{2lambdaEsalpha} that 
\begin{flalign*}
			\lambda (E, s, \alpha) & =  \pi^{-1} \cdot K_{\alpha, s} \cdot \Gamma (\alpha) \cdot \bigg(  t_{\alpha} \sqrt{1 - t_{\alpha}^2} \cdot \mathbb{E} \Big[ \big| R_{\loc} (E) \big|^{\alpha} \Big] \\
			& \qquad + \sqrt{t_s^2 (1 - t_{\alpha}^2) \cdot \mathbb{E} \Big[ \big| R_{\loc} (E) \big|^{\alpha} \Big]^2 + t_{\alpha}^2 (t_s^2 - t_{\alpha}^2) \cdot \mathbb{E}  \Big[ \big| R_{\loc} (E) \big|^{\alpha} \cdot \sgn \big( - R_{\loc} (E) \big) \Big]^2} \bigg).
		\end{flalign*}
Then to prove the continuity of $\lambda (E, s, \alpha)$ in $E$, it suffices to establish the continuity of the two quantities 
\be\label{twotwo}
\Big[ \big| R_{\loc} (E) \big|^{\alpha} \Big], \qquad
\Big[ \big| R_{\loc} (E) \big|^{\alpha} \cdot \sgn \big( - R_{\loc} (E) \big) \Big]^2.
\ee
From \Cref{l:boundarybasics2}, we have 
\begin{align}\label{appc1}
\E \big[ |R_\loc(E)|^{\alpha} \big] &=\E \left[ \big(R_\loc(E)\big)_+^{\alpha} \right] 
+ \E \left[ \big(R_\loc(E)\big)_-^{\alpha} \right]  \\
&=
F_{\alpha}\big(E, a(E), b(E)\big)\notag
+ G_{\alpha}\big(E, a(E), b(E)\big),
\end{align}
and 
\begin{align}\label{appc2}
\mathbb{E} \Big[ \big| R_{\loc} (E) \big|^{\alpha} \cdot \sgn \big( - R_{\loc} (E) \big) \Big]
&= 
- E \left[ \big(R_\loc(E)\big)_+^{\alpha} \right] 
+ \E \left[ \big(R_\loc(E)\big)_-^{\alpha} \right]\\
&=  F_{\alpha}\big(E, a(E), b(E)\big) 
+ G_{\alpha}\big(E, a(E), b(E)\big),\notag
\end{align}
where we recall the functions $F_\gamma$ and $G_\gamma$ defined in \eqref{FandG}. Recall also that $a(E)$ and $b(E)$ are continuous; this follows from their definitions in \eqref{opaque} and the continuity of $y(E)$ on $\bbR$ given by \Cref{l:boundaryvalues}. We observe that $F_\alpha$ and $G_\alpha$ are continuous in each of their arguments; in fact, they are differentiable in each argument, as shown in \Cref{l:Fx}, \Cref{l:Gy}, \Cref{l:Gx}, \Cref{l:Gy}, \Cref{l:FE}, and \Cref{l:GE}. We conclude from this, the representations \eqref{appc1} and \eqref{appc2}, and the continuity of $a(E)$ and $b(E)$, that the quantities in \eqref{twotwo} are continuous in $E$, and hence that $\lambda(E,\alpha)$ is continuous in $E$.

The proof of the continuity of $\lambda(E,\alpha)$ in $E$ is similar, and this completes the proof of the third claim. 
The fourth part of the lemma follows directly from \Cref{l:lambda>1}.
\end{proof}

\printbibliography

\end{document}